\documentclass[11pt]{report}
\usepackage{a4wide}

\usepackage{latexsym}
\usepackage{amsmath}
\usepackage{amsfonts}
\usepackage{verbatim}
\usepackage{amssymb}
\usepackage{epsfig}
\usepackage{color}
\usepackage{enumerate}
\usepackage{graphicx}
\usepackage{natbib} 
\usepackage{pst-all}
\usepackage{mathrsfs}
\usepackage{makeidx}
\usepackage{fancyhdr}

\newcommand{\private}[1]{}

\newcommand{\m}{\mbox}

\def\BBox{\rule{2mm}{3mm}}
\def\QED{\hfill$\BBox$}
\newenvironment{proof}
{\begin{rm}\par\smallskip\noindent{\bf Proof.}\quad}{\QED\end{rm}}

\newtheorem{thm}{Theorem}[chapter]
\newtheorem{prop}[thm]{Proposition}
\newtheorem{cor}[thm]{Corollary}
\newtheorem{lem}[thm]{Lemma}

\def\su{\mathop{\it sum}\nolimits}
\def\bb{\mathbf{b}}
\def\cc{\mathbf{c}}

\title{Recognition of generalized network matrices}
\author{PhD thesis \\ 
by \\
Antoine Musitelli\\
\\
Mathematics Institute\\
EPFL, antoine.musitelli@epfl.ch\\
 Lausanne, Switerland}
\date{October 1, 2007}

\definecolor{shadecolor}{rgb}{0.9,0.9,0.9}

\makeindex

\begin{document}
\pagestyle{fancyplain}
\thispagestyle{empty}
\renewcommand{\chaptermark}[1]
{\markboth{#1}{}}
\renewcommand{\sectionmark}[1]
{\markright{\thesection\quad #1}}



\begin{center}
\Large
\textbf{Recognition of Generalized Network Matrices}\\[10mm]
\small
TH\`ESE Num\'ero \ 3938 ( 2007 )\\[5mm]
\small
PR\'ESENT\'EE LE 1 octobre 2007\\[5mm]
\`A LA FACULT\'E DES SCIENCES DE BASE\\[5mm]
Chaire de recherche op\'erationnelle SO\\[5mm]
PROGRAMME DOCTORAL EN MATH\'EMATIQUES\\[5mm]
\normalsize
\textbf{\'ECOLE POLYTECHNIQUE F\'ED\'ERALE DE LAUSANNE}\\[5mm]
\small
POUR L'OBTENTION DU GRADE DE DOCTEUR \`ES SCIENCES\\[5mm]
PAR\\[10mm]
\large
Antoine Musitelli
\\[5mm]
\footnotesize
math\'ematicien dipl\^om\'e de l'Universit\'e de Gen\`eve \\[3mm]
de nationalit\'e suisse et originaire de L'Abbaye (VD)\\[8mm]
\normalsize
accept\'ee sur proposition du jury :\\[5mm]
Prof. T. Mountford, pr\'esident du jury\\
Prof. T. Liebling, Prof. K. Fukuda, directeurs de th\`ese\\
Prof. M. Conforti, rapporteur\\
Prof. D. de Werra, rapporteur\\
Prof. J. Fonlupt, rapporteur
\end{center}
\clearpage
\thispagestyle{empty}
\cleardoublepage
\verb"   "
\newpage

\chapter*{Acknowledgements}

First of all, I would like to thank Professors Komei Fukuda and Thomas Liebling, my supervisors, for their guidance and support. Komei's personal attention was essential for accomplishing this work.

I am grateful for some helpful comments of Professors Michele Conforti, Jean Fonlupt and Dominique de Werra.

I am also indebted to my mother for her encouragement at any time. 

Finally, I should like to acknowledge the financial support of the EPFL, Adonet and the Swiss National Science Foundation Project 200021-105202, "Polytopes, Matroids and Polynomial Systems". 
\chapter*{Abstract}

In this thesis, we deal with binet matrices, an extension of network matrices.
The main result of this thesis is the following. A rational matrix $A$ of size $n\times m$ can be tested for being binet in time $O(n^6 m)$. If $A$ is binet, our algorithm outputs a nonsingular matrix $B$ and a matrix $N$ such that $[B\, N]$ is  the node-edge incidence matrix of a bidirected graph (of full row rank) and $A=B^{-1} N$.

Furthermore, we provide some results about Camion bases.
For a matrix $M$ of size $n\times m'$,
we present a new characterization of
Camion bases of   $M$, whenever $M$ is 
the node-edge incidence matrix of a connected digraph (with one row removed). Then, a general
characterization of Camion bases as well as a recognition 
procedure which runs in $O(n^2m')$ are given. An algorithm which finds a Camion basis is also presented. For totally
unimodular matrices, it is proven to run in time $O((nm)^2)$ 
where $m=m'-n$.

The last result concerns specific network matrices. We give a characterization of nonnegative $\{\epsilon,\rho$\}-noncorelated network matrices, where $\epsilon$ and $\rho$ are two given row indexes. It also results a polynomial recognition algorithm for these matrices.

\paragraph{Keywords: } network matrices, binet matrices and Camion bases.

\chapter*{R\'esum\'e}

Dans cette th\`ese, nous nous penchons sur la notion de matrice binet qui \'etend celle de matrices r\'eseau. Le r\'esultat majeur de la th\`ese est le suivant: on peut tester si une matrice donn\'ee $A$ \`a coefficients rationnels et de taille $n\times m$ est binet en temps $O(n^6 m)$. Si $A$ est binet, notre algorithme fournit une matrice inversible $B$ et une matrice $N$ telles que $[B\, N]$ est la matrice d'incidence sommet-ar\^ete d'un graphe biorient\'e de rang $n$ et $A=B^{-1} N$.

En outre, nous exposons quelques r\'esultats sur les bases de Camion. Pour une matrice $M$ de taille $n\times m'$, nous pr\'esentons une nouvelle caract\'erisation des bases de Camion de $M$, lorsque $M$ est la matrice d'incidence sommet-ar\^ete d'un graphe orient\'e connexe (auquel une ligne a \'et\'e \^ot\'ee). Puis, une caract\'erisation g\'en\'erale des bases de Camion ainsi qu'une proc\'edure de reconnaissance qui s'ex\'ecute en temps $O(n^2m')$ sont donn\'ees. Nous pr\'esentons \'egalement un algorithme permettant de trouver une base de Camion. Pour les matrices totallement unimodulaires, celui-ci n\'ecessite un temps de calcul de l'ordre de $O((nm)^2)$ o\`u $m=m'-n$.

Le dernier r\'esultat int\'eressant concerne certaines matrices r\'eseau sp\'ecifiques. Nous formulons une caract\'erisation des matrices r\'eseau non n\'egatives et non $\{\epsilon,\rho$\}-cor\'el\'ees, o\`u  $\epsilon$ et $\rho$ sont deux indices donn\'es de lignes. Il en r\'esulte un algorithme polynomial de reconnaissance pour cette classe de matrices.

\paragraph{mots-cl\'es: } matrices r\'eseau, matrices binet and bases de Camion.

\clearpage
\thispagestyle{empty}
\cleardoublepage
\verb"   "
\newpage

\tableofcontents

\chapter{Overview}\label{ch:Over}

\begin{flushright}

\emph{Science sans conscience n'est que ruine de l'\^ame.}

Rabelais (1485-1553).

\end{flushright}

\section{Introduction}\label{sec:OverIntro}

Network matrices play a central role in combinatorial optimization. They have been largely used for solving  flow problems and integer programs. Their interest extends also to the field of matroids, graphs embeddable on the plane, etc... 

There exist several polynomial-time algorithms to test if a given matrix is a network matrix. Such an algorithm was designed by Auslander and Trent \cite{AusTrentNetwork-59}, Gould \cite{GouldNetwork-58}, Tutte \cite{Tutte-60,Tutte-65,Tutte-67}, Bixby and Cunningham \cite{BixCun-80} and Bixby and Wagner \cite{BixWagner-88}. 
A famous one is Schrijver's method\label{mycounter3} developed in \cite{ShrAlex} which adapts the matroidal ideas of Bixby and Cunningham \cite{BixCun-80} to matrices. The algorithm works by reducing the problem to a set of smaller problems, which can be handled easily. The smaller problems consist of deciding if a matrix with at most two non-zeros per column is a network matrix or not. The reduction is done by identifying rows of the matrix that correspond to cut-edges of the spanning tree, and then carrying on with the smaller matrices associated with the components. The time complexity of the algorithm developed in \cite{BixCun-80}  is $O(n \alpha)$, where $n$ and $\alpha$ are the number of rows and nonzero elements, respectively, in the input matrix.

Binet matrices, a generalization of network matrices, have been investigated by Kotnyek \cite{KotThesis}. He provided an algorithm to determine the columns of a binet matrix using its graphical representation, and gave some geometrical properties of these matrices, by extending results about totally unimodular matrices (see Section \ref{sec:Ap}). However, one question is left open: 
is it possible to recognize whether a given matrix is binet or not in time polynomial in its size?

In a parallel direction, a matroid called the signed-graphic matroid has been introduced by Zaslavsky at the beginning of the eighties \cite{Zaslavsky-Signed-82}. Any binet matrix is a compact representation matrix of a signed-graphic matroid. Since then, this class of matroids continues to receive much attention in the mathematical literature, but several questions remain open, in particular the problem of determining whether a given matroid is signed-graphic or not.
(See for example \cite{SlilatyProjPlan-03}, \cite{BouchetFlow-83}, \cite{DeVosPhD-00}, \cite{GerardsGraphs-90}, \cite{PaganoPHD},
\cite{SlilatyDecomp-06},
\cite{Slilaty-07},
\cite{SlilatyCoSign-05},
\cite{SlilatyBiasunique-06}, \cite{Zaslavsky-Glos-99}  and \cite{Zaslavsky-Bib-98}.)

In this thesis, we turn to the problem of recognizing binet matrices.
Our central result is the following. \\

\noindent
{\bf Theorem \ref{thmRecBinetMain}}
\emph{A rational matrix of size $n\times m$ can be tested for being binet in time $O(n^6m)$ using the algorithm Binet.}\\
 
\noindent
Given a binet matrix $A$ as input, the algorithm Binet outputs a nonsingular matrix $B$ and a matrix $N$ such that $[B\, N]$ is  the node-edge incidence matrix of a bidirected graph (of full row rank)
and $A=B^{-1} N$.
Theorem \ref{thmRecBinetMain} has different applications that are discussed in Section \ref{sec:Ap}. Section \ref{sec:OverRecBinet} describes in some details the strategy employed for recognizing binet matrices and serves as a tool for reading the Chapters \ref{ch:Rec} to \ref{ch:central}.

One key element of the algorithm Binet consists of finding a Camion basis of the matrix $[I \, A]$. We provide some new results about Camion bases.
For a matrix $M$ of size $n\times m'$,
we present a characterization of
Camion bases of   $M$, whenever $M$ is 
the node-edge incidence matrix of a connected digraph (with one row removed).  
Then, a general
characterization of Camion bases as well as a recognition 
procedure which runs in $O(n^2m')$ are given. There is no known polynomial-time algorithm to find a Camion basis in general.
Fonlupt and  Raco \cite{FonRaco} described a finite procedure to find one  based on the results of Camion, and gave an algorithm which runs in time $O(n^3 \, m^2 )$ for totally
unimodular matrices.  We also present an algorithm which finds a
Camion basis. For totally
unimodular matrices, it is proven to run in time $O((nm)^2)$ where
$m=m'-n$.

The last interesting result concerns specific network matrices. We give a characterization of nonnegative $\{\epsilon,\rho$\}-noncorelated network matrices, where $\epsilon$ and $\rho$ are two given row indexes. It results a polynomial recognition algorithm for this class of matrices.

\section{Notions}\label{sec:IntNot}

We will present here the principal notions used in this thesis. More details are given in Chapters \ref{ch:Int}, \ref{ch:Bidirected} and \ref{ch:Binetdef}. We assume familiarity of the reader with the elements of linear algebra, such as linear (in)dependence, rank, determinant, matrix, non-singular matrix, inverse, Gauss' algorithm for solving a system of linear equations, etc.
A reader not familiar with the content of this section can consult, for example, Schrijver \cite{ShrAlex} and Welsh \cite{WelshMatTheory-76}.

As always, $\mathbb{Z}$, $\mathbb{Q}$, and $\mathbb{R}$ denote the set of integer, rational, and real numbers. $\mathbb{N}$ contains the nonnegative integers. The \emph{projective plane}\index{projective plane}, denoted by $\mathbb{P}^2$, is the sphere (in $\mathbb{R}^3$) where any two antipodal points ($x$ and $-x$) are identified.

Vectors and matrices whose elements are integers are called \emph{integral}\index{integral!vector or matrix}. That is, integral $m$-dimensional vectors are those in $\mathbb{Z}^m$, an integral matrix of size $n\times m$ is in $\mathbb{Z}^{n\times m}$. Similarly, \emph{rational}\index{rational vector or matrix} vectors and matrices have elements from $\mathbb{Q}$. \emph{Half-integral}\index{half-integral!vector or matrix} vectors and matrices have elements that are integer multiples of $\frac{1}{2}$.

A matrix $A$ is of \emph{full row rank}\index{full row rank}, if its rank equals the number of its rows, or equivalently, if its row vectors are linearly independent. If $A$ is of full row rank say $n$,
a \emph{basis}\index{basis} of
$A$ is a non-singular square submatrix of $A$ of size $n$.

A set $P$ of vectors in $\mathbb{R}^m$ is a \emph{polyhedron}\index{polyhedron} if $P=\{x \,| \, Ax \le b \}$ for an $n\times m$ matrix $A$ and $n$-dimensional vector $b$. The vectors of a polyhedron are called its \emph{points}\index{point of a polyhedron}. An \emph{extreme point}\index{extreme point} or \emph{vertex}\index{vertex of a polyhedron} of a polyhedron $P=\{x \,| \, Ax \le b \}$ is a point determined by $m$ linearly independent equations from $Ax =b$. Every extreme point of $P$ can arise as an optimal solution of $\max \{ c^T x \, | \, x \in P\}$ for a suitably chosen $c$. For any $1\le i \le n$, let us denote by $A_{i \bullet}$ the $i$th row of $A$ and let $P^{(i)}$ be the polyhedron obtained from $P$ by removing the $i$th row from the system $Ax \le b$. If $P^{(i)}\neq P$, then $P\cap \{A_{i \bullet}x=b_i\}$ is called a \emph{facet}\index{facet} of $P$.

If $P$ has at least one vertex, then $P$ is called \emph{integral}\index{integral!polyhedron}, if all of its vertices are integral. 
An integral polyhedron $P$ provides integral optimal solutions for $\max \{ c^T x \, | \,  x \in P\}$ for any $c$. Similarly, if $P$ is half-integral, then the optimal solutions are half-integral.

A set $S$ of vectors is \emph{convex}\index{convex} if it satisfies: if $x,y \in S$ and $0\le \lambda \le 1$, then $\lambda x+ (1-\lambda)y \in S$. The \emph{convex hull}\index{convex hull} of a set $S$ of vectors is the smallest convex set containing $S$, and is denoted $conv(S)$.

Let $M\in \mathbf{R}^{n\times m'}$ be a matrix of rank $n$, 
$\mathcal{H}(M)=\{\{x\in \mathbf{R}^n\,:\, c^Tx =0\}\,:
c\m{ is a column of }$\\
$M\}$ and  $B$ a basis of $M$.
$\mathcal{H}(M)$ splits up $\mathbf{R}^n$ into a set $S$ of full dimensional 
 cones (regions). A {\em simplex region} is one  which has exactly $n$ facets. The matrix 
$B$ is called a {\em Camion basis} if the 
corresponding hyperplanes determine the facets of a simplex region in $S$. 
It is known that there always exists a Camion basis. 
After some column permutations, we may write $M=[B\, N]$.  It is possible to 
show that $B$ is a Camion basis if and only if by multiplying some rows and columns of the matrix $B^{-1}N$, the resulting matrix is nonnegative. Geometrically,  
$B^{-1}N\geq 0$ means that the column vectors of $N$ are contained in the
cone generated by $B$. \\

An (\emph{undirected}) \emph{graph}\index{graph} is a pair $G=(V,E)$, where $V$ is a finite set, and $E$ is a family of unordered pairs of elements of $V$. A \emph{directed graph}\index{directed graph} or \emph{digraph}\index{digraph} is a pair $G=(V,E)$, where $V$ is a finite set, and $E$ is a finite family of ordered pairs of elements of $V$. The elements of $V$ are called the \emph{vertices}\index{vertex} or \emph{nodes}\index{node} of $G$, and the elements of $E$ are called the \emph{edges}\index{edge} of $G$.

The \emph{node-edge incidence matrix}\index{node-edge incidence matrix} of a graph $G$ has its rows and columns associated with the nodes and edges of the digraph. The non-zeros in a column associated with edge $e$ stand in the rows that correspond to the endnodes of $e$. If $G$ is directed, then heads get positive signs and tails get negative signs, otherwise all non-zero entries are equal to $1$. We denote by \emph{IMD}\index{IMD} the node-edge incidence matrix of a digraph. See Figure \ref{fig:OverviewDmat}.
The rank of the node-edge incidence matrix $In$ of a connected directed graph $G$ on $n$ nodes is $n-1$. Moreover, by deleting any row, $In$ can be made a full row rank matrix $In'$. The bases of $In'$ correspond to spanning trees of $G$. 
An \emph{RIMD}\index{RIMD}, or restricted IMD, is an IMD with (linearly) redundant rows removed. 

\begin{figure}[h!]
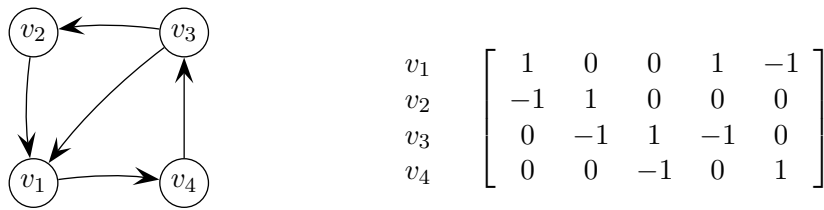

\vspace{0.2cm}
$
\begin{array}{cc}

\psset{xunit=1cm,yunit=1cm,linewidth=0.5pt,radius=3mm,arrowsize=7pt,
labelsep=1.5pt}

\pspicture(-1,1)(5,2)

\cnodeput(1,0){1}{$v_1$}
\cnodeput(1,2){2}{$v_2$}
\cnodeput(3,2){3}{$v_3$}
\cnodeput(3,0){4}{$v_4$}

\ncarc{<-}{1}{2}
\ncarc{<-}{2}{3}
\ncline{<-}{3}{4}

\ncarc{<-}{1}{3}

\ncarc{->}{1}{4}

\endpspicture  & 

\begin{tabular}{rl}

 & $\left. \begin{array}{cccc} \hspace{.3cm} & \hspace{.3cm} & \hspace{.3cm} & \end{array} \right.$ \\

\hspace{0.2cm}$\begin{array}{c} v_1 \\
v_2 \\
v_3 \\
v_4 
\end{array}$ &

$\left[ \begin{array}{ccccc} 1 & 0 & 0 & 1 & -1\\
-1& 1& 0 &0 & 0\\
0&-1&1&-1 & 0\\
0&0&-1&0 & 1
\end{array}\right] 

$ \\

\end{tabular}  
 
\end{array}$
\vspace{.3cm}
\caption{A digraph $G$ (on the left) and the node edge incidence matrix $In$ of $G$.}

\label{fig:OverviewDmat}
\end{figure}

Let $G=(V,E)$ be a digraph and $In$ the $V\times E$-incidence matrix of $G$. Let $In'$ be an RIMD obtained from $In$ and $B$ a basis of $In'$ such that $In'=[B\, N]$ (up to column permutations). The matrix $A=B^{-1}N$ is called a {\em network matrix}\index{network!matrix}.


\begin{figure}[h!]
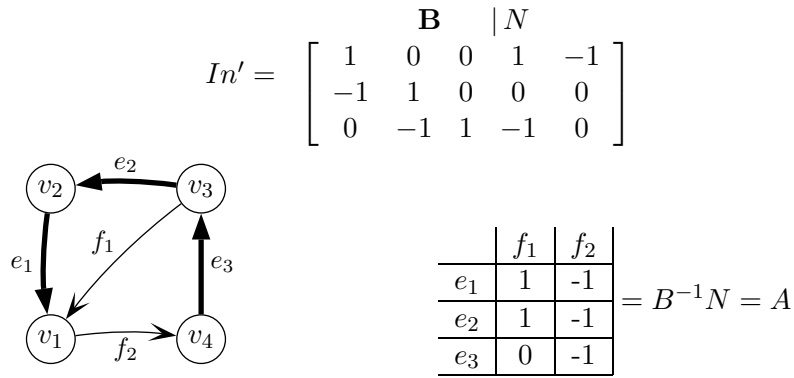


\psset{xunit=1cm,yunit=1cm,linewidth=0.5pt,radius=3mm,arrowsize=7pt,
labelsep=1.5pt}
\pspicture(0,0)(15,5)

\rput(3,2.5){
\psset{xunit=1cm,yunit=1cm,linewidth=0.5pt,radius=3mm,arrowsize=7pt,
labelsep=1.5pt}
\pspicture(0,1)(3,3)

\cnodeput(1,0){1}{$v_1$}
\cnodeput(1,2){2}{$v_2$}
\cnodeput(3,2){3}{$v_3$}
\cnodeput(3,0){4}{$v_4$}

\ncarc[linewidth=2pt,arrowinset=0]{<-}{1}{2}
 \naput{\small{$e_1$}}
\ncarc[linewidth=2pt,arrowinset=0]{<-}{2}{3}
 \naput{\small{$e_2$}}
\ncline[linewidth=2pt,arrowinset=0]{<-}{3}{4}
 \naput{\small{$e_3$}}

\ncarc[arrowinset=.5,arrowlength=1.5]{<-}{1}{3}
 \naput{\small{$f_1$}}

\ncarc[arrowinset=.5,arrowlength=1.5]{->}{1}{4}
\nbput{\small{$f_2$}}

\endpspicture  }

\rput(7.5,4){
$In'= \begin{tabular}{c}

  $\left. \begin{array}{cccc}  \hspace{.2cm}  &{\bf B} \hspace{.3cm} & | \,  N  &\end{array} \right.$ \\

$\left[ \begin{array}{ccccc} 1 & 0 & 0 & 1& -1\\
-1& 1& 0 &0 & 0\\
0&-1&1&-1 & 0
\end{array}\right] 

$ \\

\end{tabular}  $}

\rput(10,1){
$\begin{tabular}{c|c|c|}

& $f_1$ & $f_2$ \\
\hline
$e_1$ & 1& -1\\
\hline
$e_2$ & 1 & -1\\
\hline
$e_3$ & 0 & -1\\
\hline

\end{tabular}=B^{-1} N =A   $}

\endpspicture

\caption{An RIMD  $In'$ (above) obtained from the IMD $In$ described  in Figure \ref{fig:OverviewDmat} by removing the last row.
For the basis $B$ consisting of the three first columns of $In'$, the corresponding network matrix and network representation are given at the bottom.
The edges $e_1,e_2$ and $e_3$ corresponding to the basis $B$ form a spanning tree of $G$. The edges $f_1$ and $f_2$ correspond to the first and second column of $N$, respectively.}
\label{fig:OverviewRIMD}
\end{figure}

By assuming that $G$ is connected, the basis $B$ corresponds to a spanning tree of $G$ (as mentioned above)
whose subgraphs are called \emph{basic}\index{basic}.
The rows and columns of the network matrix are associated with the tree and non-tree edges, respectively. For any non-tree edge $f$, we find the unique cycle (called the fundamental cycle) which contains $f$ and some edges from the tree. The column of the network matrix corresponding to $f$ will contain $\pm 1$ in the rows of the tree edges in its fundamental cycle and $0$ elsewhere. The signs of the non-zeros depend on the directions of the edges. If walking through the tree along the fundamental cycle starting at the tail of $f$, a tree edge lies in the same direction, it gets a positive sign, if it lies in the opposite direction, it gets a negative sign. See Figure \ref{fig:OverviewRIMD}.

Let $A$ be a nonnegative connected network matrix, and suppose that we are given two row indexes, say $\epsilon$ and $\rho$ ($\epsilon \neq \rho$). Let $G(A)$ be a network representation of $A$ and $q$ a basic path in $G(A)$ containing $e_\epsilon$ and $e_\rho$. If $q$ passes through one of the edges $e_\epsilon$ and $e_\rho$ forwardly and through the other  backwardly, then we say that $e_\epsilon$ and $e_\rho$ are \emph{alternating}\index{alternating} in $G(A)$, otherwise \emph{nonalternating}.
If $A$ has a network representation in which $e_\epsilon$ and $e_\rho$ are alternating and another one in which they are nonalternating, then $N$ is said to be \emph{$\{ \epsilon,\rho\}$-noncorelated}, otherwise \emph{$\{ \epsilon,\rho\}$-corelated\index{corelated}}.

A matrix $A$ is \emph{totally unimodular}\index{totally unimodular} if each subdeterminant of $A$ is $0$, $+1$, or $-1$.
In particular, each entry in a totally unimodular matrix is $0$, $+1$, or $-1$. It is known that any network matrix is totally unimodular.\\

A bidirected graph $G=(V,E)$ on node set $V=\{v_1,\ldots,v_n\}$ and with edge set $E=\{e_1,\ldots,e_m\}$ may have four kinds of edges. A \emph{link}\index{link} is an edge with two distinct endnodes;  a \emph{loop}\index{loop} has two identical endnodes; a \emph{half-edge}\index{half-edge} has one endnode, and a \emph{loose edge}\index{loose edge} has no endnode at all.

Every edge is signed with $+$ or $-$ at its endnodes. That is, links or loops can be signed with $+-$, $++$ or $--$; half-edges have only one sign, $+$ or $-$. If an edge is signed with $+$ at an endnode, then this node is an \emph{in-node}\index{in-node} or \emph{head}\index{head} of the edge. An endnode signed with $-$ is called the \emph{out-node}\index{out-node} or \emph{tail}\index{tail} of the edge. One can think of the value $+$ as indicating that the edge is directed into the node, $-$ indicating direction away from the node. All edges, except directed edges, are called \emph{bidirected edges}\index{bidirected!edge}. See Figure \ref{fig:OverviewBidGraph}.

Every edge $e$ is given a \emph{sign}\index{sign}, denoted by $\sigma_e\in \{ +, -\}$. The signing convention we adopt is that the sign of an edge is negative if and only if it is bidirected.

\begin{figure}[h!]
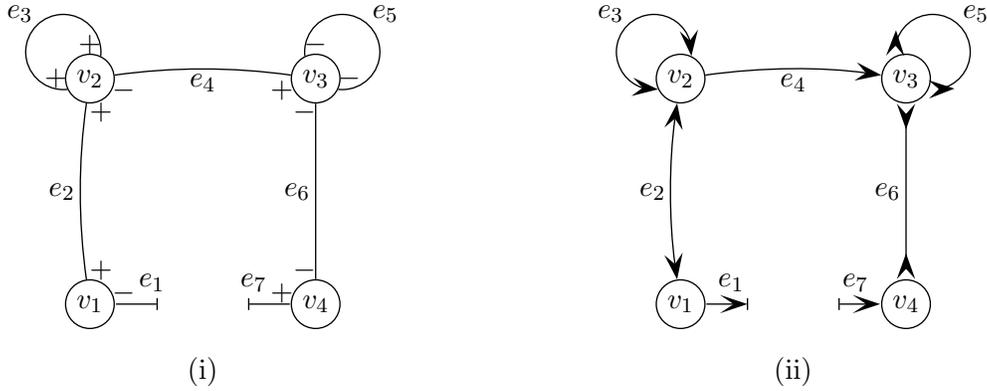

\psset{xunit=1.5cm,yunit=1.5cm,linewidth=0.5pt,radius=3mm,arrowsize=7pt,
labelsep=1.5pt}

\begin{center} 

$\begin{array}{cc}

\pspicture(0,0)(5,3)

\rput(2,-.6){(i)}
\cnodeput(1,0){1}{$v_1$}
\cnodeput(1,2){2}{$v_2$}
\cnodeput(3,2){3}{$v_3$}
\cnodeput(3,0){4}{$v_4$}

\ncarc{-}{1}{2}
\naput{$e_2$}
\rput(1.1,.3){$+$}
\rput(1.1,1.7){$+$}
\ncarc{-}{2}{3}
\nbput{$e_4$}
\rput(1.3,1.9){$-$}
\rput(2.7,1.9){$+$}
\ncline{-}{3}{4}
\nbput{$e_6$}
\rput(2.9,.3){$-$}
\rput(2.9,1.7){$-$}

\psline[arrowinset=.5,arrowlength=1.5]{-|}(1.23,0)(1.6,0)
\rput(1.55,0.19){$e_1$}
\rput(1.3,.1){$-$}

\psline[arrowinset=.5,arrowlength=1.5]{|-}(2.4,0)(2.77,0)
\rput(2.45,0.19){$e_7$}
\rput(2.7,.1){$+$}

\nccircle[angle=45]{-}{2}{0.5cm}
\nbput{$e_3$}
\rput(.7,2){$+$}
\rput(1,2.3){$+$}

\nccircle[angle=-45]{-}{3}{0.5cm}
\nbput{$e_5$}
\rput(3.3,2){$-$}
\rput(3,2.3){$-$}

\endpspicture  &

\pspicture(0,0)(4,3)

\rput(2,-.6){(ii)}
\cnodeput(1,0){1}{$v_1$}
\cnodeput(1,2){2}{$v_2$}
\cnodeput(3,2){3}{$v_3$}
\cnodeput(3,0){4}{$v_4$}

\ncarc{<->}{1}{2}
\naput{$e_2$}
\ncarc{->}{2}{3}
\nbput{$e_4$}
\ncline{>-<}{3}{4}
\nbput{$e_6$}

\psline[arrowinset=.5,arrowlength=1.5]{->|}(1.23,0)(1.6,0)
\rput(1.45,0.19){$e_1$}

\psline[arrowinset=.5,arrowlength=1.5]{|->}(2.4,0)(2.77,0)
\rput(2.55,0.19){$e_7$}

\nccircle[angle=45]{<->}{2}{0.5cm}
\nbput{$e_3$}
\nccircle[angle=-45]{>-<}{3}{0.5cm}
\nbput{$e_5$}

\endpspicture  

\end{array}$

\end{center}
\vspace{.8cm}

\caption{Possible graphical representations of a bidirected graph}\label{fig:overviewGraph}
\label{fig:OverviewBidGraph}
\end{figure}

Let us define the  \emph{(node-edge) incidence matrix }\index{node-edge incidence matrix} $In(G)$ of a bidirected graph $G$ called an IMB\index{IMB}. 
The rows and columns of $In(G)$ are identified with the nodes and edges of $G$, respectively. An entry $(i,j)$ of $In(G)$ is $1$ (resp., $-1$) if $e_j$ is a link or a half-edge entering (resp., leaving) $v_i$, $2$ (resp., $-2$) if $e_j$ is a negative loop entering (resp., leaving) $v_i$, $0$ otherwise. An \emph{RIMB}\index{RIMB}, or restricted IMB, is an IMB with (linearly) redundant rows removed. As an  example, the node-edge incidence matrix of the bidirected graph depicted in Figure \ref{fig:OverviewBidGraph} is:

$$In=\left[ \begin{matrix} -1 & 1 & 0 & 0& 0&0 & 0 \\
0& 1&2 &-1 &0 &0 &0 \\
0&0&0&1&-2&-1&0\\
0&0&0&0&0&-1&1
\end{matrix}\right]$$

\noindent

A matrix $A$ is called a \emph{binet matrix}\index{binet!matrix} if there exist a full row rank incidence matrix $In$  of a bidirected graph $G$ and a basis $B$ of it such that $In=[B\, N]$ (up to column permutations) and $A=B^{-1} N$. 
As any IMD is an IMB, the class of binet matrices contains all network matrices.\\

A \emph{matroid}\index{matroid} $M$ is a finite ground set $S$ and a collection $\phi$ of subsets of $S$  such that (I1)-(I3) are satisfied.

\begin{itemize}

\item[(I1)] $\emptyset \in \phi$.

\item[(I2)] If $X\in \phi$ and $Y\subseteq X$ then $Y\in \phi$.

\item[(I3)] If $X,Y$ are members of $\phi$ with $|X|=|Y|+1$ there exists $x\in X\verb"\" Y$ such that $Y\cup \{x\} \in \phi$.

\end{itemize}

Subsets of $S$ in $\phi$ are called the \emph{independent sets}\index{independent set}, a maximal independent subset in $S$ is a \emph{basis}\index{basis of a matroid} of $M$.  A \emph{circuit}\index{circuit of a matroid} is a minimal dependent subset of $S$. If $B$ is a basis of $M$ and $s\in S\verb"\" B$, then there is a unique circuit in $B\cup \{s\}$, called the \emph{fundamental circuit of $s$ with respect to $B$}\index{fundamental circuit in a matroid}. There are several different but equivalent ways to define a matroid. For example, one can define a matroid on a given ground set through its bases, rank-function, circuits, or fundamental circuits.

A standard example of matroids is when $S$ is a finite set of vectors from a vector space over a field $\mathbb{F}$ and $\phi$ contains the linearly independent subsets of $S$. This kind of matroid is called the \emph{linear matroid}\index{linear matroid}. If for a matroid $M$ there exists a field $\mathbb{F}$ such that $M$ is a linear matroid over $\mathbb{F}$, then $M$ is \emph{representable}\index{representable} over $\mathbb{F}$. The matrix $In$ made up of the vectors in $S$ is called a \emph{standard representation matrix}\index{standard representation matrix} of $M$. 
There is a one-to-one correspondence between linearly independent columns of $In$ and independent sets in $M$, so the linear matroid $M=M(In)$ can be fully given by its representation matrix $In$.

There is another, more compact representation matrix of a linear matroid $M(In)$. To get it, first delete linearly dependent (over $\mathbb{F}$) rows from $In$, if there are any, then choose a basis $B$ of $M$. It corresponds to a basis of $In$, also denoted by $B$, and let $N$ be the remaining matrix in $In$. 
Then the matrix $A=B^{-1} N$ (where the inverse of $B$ is computed over the field $\mathbb{F}$)
is called a \emph{compact representation matrix of $M$ over $\mathbb{F}$}\index{compact representation matrix}.

\section{Applications}\label{sec:Ap}

In combinatorial optimization and geometry, many of the most frequently used algorithms exploit the discrete structure and properties of the matrices involved in the problems: network simplex method, integer programs with totally unimodular constraint matrix, Tardos' method for linear programming,...  Our purpose here is a description of four applications resulting from a recognition procedure for binet matrices. Binet matrices have strong connections with some linear and integer programs that are easily solvable, and a class of polyhedrons with half-integral vertices. Moreover, they are related to relatively unknown matroids and a characterization of the graphs embeddable on the projective plane.\\

Consider the linear program
\begin{eqnarray}
\min  & c^T x \nonumber \\
 s.c. &  Ax = b \label{eqnOverviewLin}\\
      & 0\le x \le \alpha \nonumber
\end{eqnarray}
\noindent
 where $c,\alpha \in \mathbb{R}^m$, $b\in \mathbb{R}^n$, $A \in \mathbb{R}^{n \times m}$ for some $n,m\in \mathbb{N}$, and some entries of $\alpha$ may be infinite. If the constraint matrix in (\ref{eqnOverviewLin}) is an RIMD, then the problem is called a transshipment problem or more simply a network problem in the literature, and the interpretation of (\ref{eqnOverviewLin}) is the following. Variables correspond to flows in edges, while the constraints of the system "$Ax=b$" correspond to supply (if $b_r <0$), demand (if $b_r > 0$) or flow conservation (if $b_r=0$) requirements at the vertices; the inequalities  $0\le x \le \alpha$ are capacity constraints. The objective is to find a minimum cost flow subject to these constraints. The optimum can be found by using the \emph{network simplex method}, a successful cross between the algebra of the simplex method and the combinatorics of the flow algorithms.

If $A$ is an RIMB, then (\ref{eqnOverviewLin}) is a generalization of the network problem and is called a \emph{bidirected LP}. The optimal solution of a bidirected LP can be achieved by general-purpose methods, like the simplex algorithm, or the strongly polynomial algorithm of Tardos \cite{TardosPolAlg-86}. Kotnyek \cite{KotThesis} proved that a bidirected LP can be solved using a 
generalized network simplex method which  is not polynomial in the worst case, but for most of the practical problems it is much more efficient than the simplex method, or the strongly polynomial method of Tardos (see the reference notes in \cite{AhujaMagOrlinNet-93}).

Now assume that we are to solve (\ref{eqnOverviewLin}) with the additional constraint that $x$ is integral. Given any undirected graph $G=(V,E)$, if $A$ is the node-edge incidence matrix of $G$ and $\alpha_i=1$ for all $1\le i \le m$, then (\ref{eqnOverviewLin}) is known as the optimum $b$-matching problem: given a real numerical weight $c_e$ for each edge $e\in E$ and an integer $b_v$ for each node $v\in V$, find in $G$, if there is one, a subgraph $G'$ which has degrees $b_v$ at nodes $v$ and whose edges have maximum weight-sum. 
Whenever $A$ is an RIMB (and $x$ is required to be integral), the problem (\ref{eqnOverviewLin}) is called  a \emph{bidirected IP} and generalizes the $b$-matching problem.
Edmonds proved that bidirected IPs are solvable in polynomial time (see \cite{EdmondsMatching-67,EdmondsMatching2,LawlerOpt-76}).

It is known that the solution set of the system "$Ax=b$" is unchanged under elementary row operations, like multiplying a row by a nonzero constant or adding a row to another one. Using Gauss' algorithm, we may assume that the matrix $A$ in (\ref{eqnOverviewLin}) is in standard form (i.e., the first $n$ columns of $A$ form the identity matrix).  Then a natural question is whether one can convert (\ref{eqnOverviewLin}) to a bidirected LP or IP (when $x$ is required to be integral) using elementary row operations on the system "$Ax =b$". We show in Section \ref{sec:BinetOpt} that this is equivalent to the problem of recognizing binet matrices as stated in the next theorem.\\

{\bf Theorem \ref{thmOptEqui}} \emph{Let $A$ be a real matrix in standard form. Then $A$ can be converted to an RIMB using elementary row operations if and only if $A$ is binet.}\\

Our procedure for recognizing binet matrices provides a way of
transforming any rational matrix $A$ into an RIMB using elementary row operations, whenever $A$ is binet.\\

The second application deals with polyhedral geometry.
Given a rational matrix $A$ of size $n\times m$ and a rational vector $b$ of size $n$, consider $P=\{ x \, :\, x\geq 0, \, Ax \le b \}$. A primary problem of integer programming is to find the \emph{integer hull}\index{integer hull} $P_I=conv\{ P \cap \mathbb{Z}^m\}$  of the polyhedron $P$. Provided that $A$ is integral, it is known that $P=P_I$ for all integral right hand side vectors $b$, if and only if $A$ is totally unimodular. In terms of integer programming, totally unimodular matrices are the integral matrices for which $\max \{ c^T x \, | \, Ax \le b, \, x \geq 0 \}$ has integral optimal solutions for any $c$ and any integral $b$. 

There are situations, however, that take us beyond  total unimodularity. Do we know the non-integral matrices $A$ that ensure integral optimal solutions of $\max \{ c^T x \, | \, Ax \le b, \, x \geq 0 \}$ for any $c$ and any integral $b$? Or what can we say about matrices $A$ that ensure integral optimal solutions for only a special set of right hand side $b$? These questions are not independent. If $A$ is rational, then one can find a nonnegative integer $k$, such that if we multiply every row of $A$ by $k$, we get an integral matrix, $kA$. But then instead of inequalities $Ax \le b$, we have $k Ax \le kb $ and we deal with polyhedra that are required to be integral for  only special $b'$ vectors, namely for those whose elements are integer multiples of $k$. For example, if $k=2$, so the elements of $A$ are halves of integers, then we are to characterize integral matrices $A'$ for which $\{ x\, | \, A'x \le b', \, x \geq 0 \}$ is integral for all even vectors $b'$. Or equivalently, we examine matrices that provide half-integral vertices for polyhedra with integral right hand sides.

A matrix is called \emph{$k$-regular}\index{regular@$k$-regular} ($k\in \mathbb{N}$) if for each of its non-singular square submatrices $\pi$, $k \pi^{-1}$ is integral. $k$-regularity is the property that takes over the role of total unimodularity in the theory of rational matrices that ensure integral vertices for polyhedra with special right hand sides.\\

{\bf Theorem \ref{thmsubclassBkreg}} \emph{(Kotnyek \cite{KotThesis})  A rational matrix $A$ is $k$-regular, if and only if the polyhedron $\{x \, | \, Ax \le kb, \, x \geq 0 \}$ is integral for any integral vector $b$.}\\

Kotnyek \cite{KotThesis} proved that every binet matrix is $2$-regular. Binet matrices seem to form a very big subclass of $2$-regular matrices. It is even conjectured by Appa that a characterization theorem should exist for a decomposition of $2$-regular matrices into binet matrices, the transpose of binet matrices and probably other non-binet matrices. If such a decomposition theorem would exist, then a procedure for recognizing binet matrices would be an essential tool for recognizing $2$-regular matrices.\\

Now we describe an interpretation of binet matrices in the field of matroids. For any bidirected graph $G$, the linear matroid of the node-edge incidence matrix of $G$ is called signed-graphic. Thus any binet matrix based on the bidirected graph $G$ is the compact representation matrix (over $\mathbb{R}$) of a signed-graphic matroid. Furthermore, given a binet matrix $A$, by multiplying the columns of $A$ with nonempty $\pm \frac{1}{2}$-support by $-2$, we shall show that the resulting matrix is a compact representation matrix of a signed-graphic matroid over $GF(3)$ (the field of cardinality $3$).\\

Finally, we state a result discussed in Section \ref{sec:BinetMat}.

\begin{thm}
Let $G$ be a $2$-connected graph and $T$ a spanning tree. Let 
$\overrightarrow{G}$ be a digraph obtained from $G$ by orienting the edges, and
$A$ the network matrix with respect to $\overrightarrow{G}$ and $\overrightarrow{T}\subseteq \overrightarrow{G}$. Then $A^T$
is a binet matrix if and only if $G$ is embeddable on $\mathbb{P}^2$.
\end{thm}

Then an algorithm for recognizing binet matrices yields a way of 
testing if a given graph is embeddable on $\mathbb{P}^2$.
However, in the literature, there exist several methods for doing this that are much more efficient than the one deriving from our recognition algorithm for binet matrices (see \cite{MoharGraphSurface-01}).

\section{Recognizing binet matrices}\label{sec:OverRecBinet}

In this section, we give more technical details of the binet recognition algorithm. Before that, we provide different definitions and notations. This section may help for reading Chapters \ref{ch:Rec} to \ref{ch:central}.

Let $A$ be a rational matrix with row set $R$ and column set $S$. $(A)_{ij}$ or $A_{ij}$ or $a_{ij}$ denotes the element in row $i\in R$ and column $j\in S$. 
For  $R'\subseteq R$ and $S'\subseteq S$,
$A_{R' \bullet}$ (respectively, $A_{\bullet S'}$) denotes the set of rows of $A$ indexed by $R'$ (respectively, columns of $A$ indexed by $S'$). 

For any $j\in S$, denote by $s(A_{\bullet j})=
\{i\,:\, A_{ij}\neq 0\}$ the support of $A_{\bullet j}$ and by    
$s_k(A_{\bullet j})= \{i\, : \, A_{ij}=k\}$ the \emph{$k$-support}\index{support@$k$-support}, for any $k\in \mathbb{R}$. A \emph{$k$-entry}\index{entry@$k$-entry} of $A$ is an entry of $A$ equal to $k$. For a set $I\subseteq \mathbb{R}$, an \emph{$I$-matrix}\index{matrix@$I$-matrix} is a matrix all of whose entries are in $I$. For $R'\subseteq R$,  $f(R')=\{ j \, : \, s(A_{\bullet j}) \cap R' \neq \emptyset \}$
and $\chi_{R'}^{R}$ denotes the \emph{characteristic vector}\index{characteristic vector} or \emph{incidence vector}\index{incidence vector} of the subset $R'$ of $R$, given by $(\chi_{R'}^{R})_i = \left \{ \begin{array}{ll}
1 & \mbox{ if } i \in R' \\
0 & \mbox{ Otherwise }
\end{array} \right.
$ for all $i\in R$. 
$A^{\frac{1}{2}\rightarrow 1}$ denotes the matrix obtained from $A$ by replacing each $\frac{1}{2}$-entry by $1$.

A \emph{walk}\index{walk} in a bidirected graph is a sequence $(v_1,e_1,v_2,\ldots
v_{t-1},e_{t-1},v_t)$ where $v_i$ and $v_{i+1}$ are endnodes of edge $e_i$ ($i=1,\ldots,t-1$), including the case where $v_i=v_{i+1}$ and
$e_i$ is a half-edge. If the walk consists of only links and it does not cross itself, i.e $v_i\neq v_j$ for $1<i<t$, $1\le j\le t$, $i\neq j$, then it is a \emph{path}\index{path}. A closed walk which does not cross itself (except at $v_1=v_t$)
and goes through each edge at most once
is called a \emph{cycle}\index{cycle}. So a loop, a half-edge or a closed path
can make up a cycle. In Figure \ref{fig:OverviewBidGraph}, there are exactly four cycles. The \emph{sign of a cycle}\index{sign of a cycle} is the product of the signs of its edges, so we have a \emph{positive cycle}\index{positive cycle} if the number of negative edges (or bidirected edges) in the cycle is even; otherwise, the cycle is a \emph{negative cycle}\index{negative cycle}. Obviously, a negative loop or a half-edge always makes a negative cycle. A \emph{full}\index{full cycle} cycle in a bidirected graph is a cycle different from a half-edge.

A bidirected graph is \emph{connected}\index{connected!bidirected graph}, if there is a path between any two nodes. 
A \emph{tree}\index{tree} is a connected bidirected graph which does not contain a cycle.
A connected bidirected graph containing exactly one cycle is called a \emph{1-tree}\index{tree@1-tree}, indicative of the fact that a $1$-tree consists of a tree and one additional edge. If the unique cycle in a 1-tree is negative, then we will call it a \emph{negative 1-tree}\index{negative 1-tree}. 

Let $G$ be a bidirected graph and $In(G)$ its node-edge incidence matrix.
Any submatrix $Q$ of the incidence matrix $In(G)$ can be obtained by row and column deletions. By analoguous operations, a bidirected graph $G(Q)$ can be obtained from $G$ (see Section \ref{sec:Incidence} 
for more details).

Now let $B$ be a basis of $In(G)$, $N$ such that $In(G)=[B\, N]$ and $A=B^{-1} N$.
Edges in the subgraph $G(B)$ of $G$ are called \emph{basic}\index{basic!edge} edges. The edges of $G$ that are not in $G(B)$ (i.e., those of $G(N)$) are the \emph{nonbasic} edges. By Lemma \ref{lemBidirectedNonsing}, the graph $G(B)$ consists of negative $1$-tree components. The unique cycle in a basic component is called a \emph{basic}\index{basic!cycle} cycle.  
Basic edges, respectively nonbasic ones, are in one-to-one correspondance with rows of $A$, respectively columns of $A$.

The bidirected graph $G$ with the set of basic and nonbasic edges marked is called
a \emph{binet representation}\index{binet!representation} of $A$ (not unique in general) and is denoted as $G(A)$. By removing the nonbasic edges from $G(A)$, we obtain a subgraph of $G$ called a \emph{basic binet representation} of $A$. There is a similar definition for network matrices.
Let $m=m'-n$. Basic edges of the binet representation will be denoted by $e_1, e_2,\ldots, e_n$ and the nonbasic ones $f_1, f_2,\ldots, f_m$, so that $e_i$ ($1\le i \le n$), respectively $f_j$ ($1\le j \le m$), corresponds to the $i$th row, respectively $j$th column of $A$. For any nonbasic edge $f_j$, the subgraph of $G(A)$ with edge set $\{f_j\} \cup \{ e_i \, : \, i \in s(A_{\bullet j}) \}$ (and without any isolated node) is called the \emph{fundamental circuit of $f_j$}. One can prove that the fundamental circuit of any nonbasic edge $f_j$ falls in one of the following categories (see Figure \ref{fig:OverviewBinet} for an example).

\begin{itemize}

\item[(i)] it is a loose edge (equal to $f_j$), or
\item[(ii)] a positive cycle, or
\item[(iii)] a pair of negative cycles with exactly one common node, or
\item[(iv)] a pair of disjoint negative cycles along with a minimal connecting path.

\end{itemize}

\begin{figure}[h!]
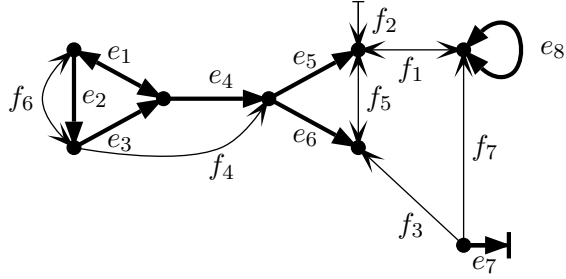


$
\begin{array}{cc}

\begin{tabular}{c|c|c|c|c|c|c|c|}
  & $f_1$ & $f_2$ & $f_3$ & $f_4$ & $f_5$ & $f_6$ & $f_7$  \\
  \hline
$e_1$  &$\frac{1}{2}$&$\frac{1}{2}$&$\frac{1}{2}$&0&1&1&0 \\
\hline
$e_2$ 
&$\frac{1}{2}$&$\frac{1}{2}$&$\frac{1}{2}$&0&1&0&0 \\
\hline
$e_3$  &$\frac{1}{2}$&$\frac{1}{2}$&$\frac{1}{2}$&1&1&-1&0 \\
\hline
$e_4$  &1&1&1&1&2&0&0 \\
\hline
$e_5$  &1&1&0&0&1&0&0 \\
\hline
$e_6$  &0&0&1&0&1&0&0 \\
\hline
$e_7$ &0&0&1&0&0&0&1 \\
\hline  
$e_8$  &$\frac{1}{2}$&0&0&0&0&0&$\frac{1}{2}$ \\
\hline
\end{tabular}

 &

\psset{xunit=1.4cm,yunit=1.3cm,linewidth=0.5pt,radius=0.1mm,arrowsize=7pt,
labelsep=1.5pt,fillcolor=black}

\pspicture(-1.2,1)(5,2.5)

\pscircle[fillstyle=solid](0,1){.1}
\pscircle[fillstyle=solid](0,2){.1}
\pscircle[fillstyle=solid](0.85,1.5){.1}
\pscircle[fillstyle=solid](1.85,1.5){.1}
\pscircle[fillstyle=solid](2.7,1){.1}
\pscircle[fillstyle=solid](2.7,2){.1}
\pscircle[fillstyle=solid](3.7,0){.1}
\pscircle[fillstyle=solid](3.7,2){.1}


\psline[linewidth=1.6pt,arrowinset=0]{<->}(0,2)(0.85,1.5)
\rput(0.44,1.95){$e_1$}

\psline[linewidth=1.6pt,arrowinset=0]{<-}(0,1)(0,2)
\rput(0.2,1.5){$e_2$}

\psline[linewidth=1.6pt,arrowinset=0]{->}(0,1)(0.85,1.5)
\rput(0.44,1.05){$e_3$}



\psline[linewidth=1.6pt,arrowinset=0]{->}(0.85,1.5)(1.85,1.5)
\rput(1.4,1.7){$e_4$}

\psline[linewidth=1.6pt,arrowinset=0]{->}(1.85,1.5)(2.7,2)
\rput(2.2,1.9){$e_5$}

\psline[linewidth=1.6pt,arrowinset=0]{->}(1.85,1.5)(2.7,1)
\rput(2.2,1.1){$e_6$}


\psline[linewidth=1.6pt,arrowinset=0]{-|}(3.7,0)(4.15,0)
\psline[linewidth=1.6pt,arrowinset=0]{->}(3.7,0)(4.15,0)
\rput(3.9,-0.2){$e_7$}

\pscurve[linewidth=1.6pt,arrowinset=0]{<->}(3.7,2)(4.15,2.25)(4.15,1.75)(3.7,2)
\rput(4.55,2){$e_8$}


\psline[arrowinset=.5,arrowlength=1.5]{<->}(2.7,2)(3.7,2)
\rput(3.2,1.8){$f_1$}

\psline[arrowinset=.5,arrowlength=1.5]{-|}(2.7,2)(2.7,2.5)
\psline[arrowinset=.5,arrowlength=1.5]{<-}(2.7,2)(2.7,2.5)
\rput(2.95,2.3){$f_{2}$}

\psline[arrowinset=.5,arrowlength=1.5]{<-}(2.7,1)(3.7,0)
\rput(3.2,.2){$f_3$}

\pscurve[arrowinset=.5,arrowlength=1.5]{->}(0,1)(1.4,1)(1.85,1.5)
\rput(1.4,.8){$f_4$}

\pscurve[arrowinset=.5,arrowlength=1.5]{<->}(0,2)(-.3,1.5)(0,1)
\rput(-0.5,1.5){$f_6$}

\psline[arrowinset=.5,arrowlength=1.5]{<->}(2.7,1)(2.7,2)
\rput(2.9,1.5){$f_5$}

\psline[arrowinset=.5,arrowlength=1.5]{->}(3.7,0)(3.7,2)
\rput(3.9,1){$f_7$}

\endpspicture

\end{array}
$\\
\vspace{.1cm}

\caption{An example of a binet matrix with a binet representation of it.}
\label{fig:OverviewBinet}
\end{figure}

A binet representation of $A$ is called \emph{proper}\index{proper} if
each basic component has exactly one bidirected edge (contained in the basic cycle), this one is entering, and 
there is another edge in the basic cycle entering one endnode of the basic bidirected edge; these two edges and the in-node incident with them are said to be \emph{central}.

A \emph{$\frac{1}{2}$-binet representation}\index{binet@$\frac{1}{2}$-binet!representation} is a proper binet representation in which every basic cycle is a half-edge. A matrix is said to be \emph{$\frac{1}{2}$-binet}\index{binet@$\frac{1}{2}$-binet!matrix} if it has a $\frac{1}{2}$-binet representation.

A \emph{cyclic representation}\index{cyclic!representation} of a matrix is a proper binet representation  of the matrix having exactly one basic cycle, and this one is full. For $R$ a row index set, an \emph{$R$-cyclic representation}\index{cyclic@$R$-cyclic!representation} of a matrix $A$ denotes a cyclic representation such that $R$ is the edge index set of the basic cycle, and this cycle is contained in the fundamental circuit of at least one nonbasic edge $f_j$ (or equivalently $R\subseteq s(A_{\bullet j})$ for at least one column index $j$ of $A$). A \emph{bicyclic representation}\index{bicyclic!representation} of a matrix is a proper binet representation  of the matrix having exactly two basic cycles, and these are full. 
At last, we say that a binet representation is an \emph{$\{\epsilon,\rho\}$-central representation}\index{central@$\{\epsilon,\rho\}$-central!representation}, if it is cyclic, $e_\epsilon$ and $e_\rho$ are edges of the basic cycle incident with one common central node, and one of the edges $e_\epsilon$ and $e_\rho$ is bidirected.
The matrix $A$ is said to be \emph{cyclic} (respectively, \emph{bicyclic},...) if and only if it has a cyclic (respectively, bicyclic,...) representation.\\

Let $A$ be a given rational matrix of size $n\times m$. 
Our purpose now is to determine whether $A$ is binet or not. To achieve this, we first transform $A$ into a nonnegative matrix $A'$ such that $A'$ is binet if and only if $A$ is binet, or outputs that $A$ is not binet (see Chapter \ref{ch:Camion}). We prove that even in the worst case we need at most $O((nm)^2)$ operations to get $A'$. We will show that any
entry of a binet matrix is in $\{0,\pm 1,\pm 2,\pm \frac{1}{2} \}$. Therefore, 
we also check whether each entry of $A'$ belongs to $\{0,1,2,\frac{1}{2}\}$, otherwise $A$ is not binet.

Now we may assume that $A$ is a matrix with $0$-, $1$-, $2$- or $\frac{1}{2}$-entries. In the second step of the algorithm, we carry out a procedure called Decomposition described in Chapter \ref{ch:Rec}. This procedure is performed on the matrix $A$ and the empty row index subset as input. The matrix
$A$ is called \emph{$\frac{1}{2}$-equisupported}\index{equisupported@$\frac{1}{2}$-equisupported} if for any column indexes $j$ and $j'$ such that $s_{\frac{1}{2}}(A_{\bullet j})\neq \emptyset$ and $s_{\frac{1}{2}}(A_{\bullet j'})\neq \emptyset$, we have 
$s_{\frac{1}{2}}(A_{\bullet j})=s_{\frac{1}{2}}(A_{\bullet j'})$.
We mention here a very important theorem.\\

\noindent
{\bf Theorem \ref{thmDecdecomp}} \emph{
The matrix $A$ is binet if and only if one of the following three statements is valid:}

\begin{itemize}

\item[1)] \emph{ $A$  is bicyclic and $\frac{1}{2}$-equisupported, or }

\item[2)] \emph{$A$ is cyclic and without any $\frac{1}{2}$-entry, or}

\item[3)] \emph{the procedure Decomposition with input $A$ and row index subset $Q=\emptyset$ provides a binet representation of $A$. }

\end{itemize}

By Theorem \ref{thmDecdecomp}, if the procedure Decomposition does not provide any binet representation of $A$, then we distinguish two cases: either $A$ has a $\frac{1}{2}$-entry and we test if $A$ is $\frac{1}{2}$-equisupported and bicyclic, or $A$ has no $\frac{1}{2}$-entry and we determine whether $A$ has a cyclic representation. If we fail, we conclude that $A$ is not binet.
Since any cyclic matrix has either an $\{i\}$-cyclic representation for some row index $i$, or an $\{\epsilon,\rho\}$-central representation for some row indexes $\epsilon$ and $\rho$, recognizing whether $A$ is cyclic can be reduced
to the recognition of $R^*$-cyclic ($|R^*|=1$) and $R^*$-central matrices ($|R^*|=2$).\\

The procedure Decomposition takes advantage of a simple property of binet matrices: if $A$ has a binet representation $G(A)$ and the columns $A_{\bullet j}$ and 
$A_{\bullet j'}$ have non-empty non-equal intersecting $\frac{1}{2}$-support for some $1\le j, j' \le m$, then $s_{\frac{1}{2}}(A_{\bullet j}) \cap s_{\frac{1}{2}}(A_{\bullet j'})$ is the edge index set of a basic cycle in $G(A)$. 
Provided that $A$ has at least one $\frac{1}{2}$-entry,
we locate different row index subsets, say $R_1,\ldots, R_\delta$, and related disjoint submatrices $A_1,\ldots,A_\delta$ in $A$. We also compute some matrix
$\tau$ and another one $A(\tau)$ obtained from $A$ by removing all rows and columns intersecting any submatrix $A_i$ ($1\le i \le \delta$) and adding $\tau$. 

We shall prove that 
whenever $A$ has a binet representation $G(A)$, 
the subgraph $G(A_i)$ is an $R_i$-cyclic representation of $A_i$,  for $i=1,\ldots,\delta$; moreover, a $\frac{1}{2}$-binet representation of $A(\tau)$ can be obtained from $G(A)$ as follows: delete in $G(A)$ all edges of the subgraphs $G(A_1),\ldots,G(A_\delta)$ and the remaining isolated nodes, then add a basic half-edge at each left node of any $G(A_i)$ for $1\le i\le \delta$. Using an $R_i$-cyclic representation of $A_i$, for $i=1,\ldots,\delta$, and a $\frac{1}{2}$-binet representation $G(A(\tau))$ of $A(\tau)$, one can compute  a binet representation of $A$. Thus the procedure Decomposition 
uses a subroutine for recognizing nonnegative $R^*$-cyclic matrices and another one for recognizing nonnegative $\frac{1}{2}$-binet matrices.\\

Let us describe some main ideas for the recognition of nonnegative $R^*$-cyclic matrices. See Chapters \ref{ch:multidiD} and \ref{ch:cyc} for more details.
Suppose that we are given the matrix $A$ and a row index subset $R^*$ of $A$. Let $S^*=\{ j \, : \, s(A_{\bullet j}) \cap R^* \neq \emptyset \}$. We decompose the matrix $A_{\overline{R^*}\times \overline{S^*}}$ into a maximum number of blocks and denote by $E_1,\ldots, E_b$ the row index sets of the different blocks. Each set $E_l$ is called a \emph{bonsai}. For all $1\le l\le b$, we compute row index subsets of $E_l$, called \emph{$E_l$-paths}, denoted by $E_l^1,\ldots,E_l^{m(l)}$, as well as a \emph{bonsai} matrix $N_l$. The matrix  $A_{E_l\times \overline{S^*}}$ (without zero columns) and the vectors $\chi_{E_l^k}^{E_l}$ are submatrices of $N_l$. Then we define a digraph denoted by $D$ and a matrix $O(R^*)$ containing the submatrix $A_{R^* \bullet}$. The graphical interpretation of these objects is the following.

Suppose that $A$ has an $R^*$-cyclic representation $G(A)$. Let $1\le l \le b$. The bonsai $E_l$ corresponds to the edge index set of a basic subtree in $G(A)$ (outside the basic cycle) called a \emph{bonsai} and denoted by $B_l$. For any $j\in S^*$, the intersection of the fundamental circuit of the nonbasic edge $f_j$ with some bonsai $B_l$ is either empty, or a directed path called a \emph{$B_l$-path}, or a union of two directed paths, called \emph{$B_l$-paths}, starting at a same node. 
Each $E_l$-path is the edge index set of a $B_l$-path (there is a one-to-one mapping between $B_l$-paths and $E_l$-paths). The bonsai $B_l$ is a basic network representation of the matrix $(N_l)_{E_l \bullet}$. The digraph $D$ contains some information about the feasible connections between the bonsais $B_l$ ($1\le l\le b$).
From $G(A)$ one can derive a basic network representation of $O(R^*)$ as follows: delete all nonbasic edges, contract some bonsais, switch at some nodes and "cut" the basic cycle at some node so that the basic cycle becomes a path, and add some directed edge at each endnode of this path. By computing some feasible spanning forest of $D$ as well as
a network representation of each bonsai matrix $N_l$ ($1 \le l\le b$) and of the matrix $O(R^*)$, if they exist, one can construct an $R^*$-cyclic representation of $A$.\\

The subroutine of the procedure Decomposition for 
recognizing nonnegative $\frac{1}{2}$-binet matrices is presented in Chapter \ref{ch:halfbinet}. The described method is based on the following fact: if $A$ has a $\frac{1}{2}$-binet representation $G(A)$ such that $S$ is the index set of all nonbasic bidirected edges (and these are entering), then $A'= \left[ \begin{array}{c}
A\\
(\chi_S^{\{1,\ldots, m \}})^T
\end{array} \right]$ has a $\{n+1\}$-cyclic representation. So 
our purpose is to compute a column index set $S$ 
such that whenever $A$ is binet $S$
corresponds to the index set of nonbasic bidirected edges in some $\frac{1}{2}$-binet representation of $A$. Let $S_2=\{ j \,: \, s_2(A_{\bullet j}) \neq \emptyset\}$. 
We compute a family $\mathscr{I}$ of pairwise disjoint row index subsets of $A$ and a column index subset $S(\mathscr{I})$ satisfying the following: whenever $A$ has a $\frac{1}{2}$-binet representation $G(A)$,
each element of $\mathscr{I}$ corresponds to the edge index set of a basic (negative) $1$-tree in $G(A)$, and the matrix 
$A'= \left[ \begin{array}{c}
A\\
(\chi_{S(\mathscr{I})\cup S_2}^{\{1,\ldots, m \}})^T
\end{array} \right]$ is $\{n+1\}$-cyclic.
Moreover, given a $\{n+1\}$-cyclic representation of $A'$, it is easy to deduce a $\frac{1}{2}$-binet representation of $A$ (by contracting the basic loop with index $n+1$).
 Thus, the recognition of nonnegative $\frac{1}{2}$-binet matrices is reduced to the recognition of nonnegative $R^*$-cyclic matrices with $|R^*|=1$. \\

Now we turn to the problem of recognizing nonnegative bicyclic matrices. This problem is deeply studied in Chapter \ref{ch:bicyc}. Suppose that we are given the matrix $A$, this one is $\frac{1}{2}$-equisupported and has at least one $\frac{1}{2}$-entry. Let $S_{\frac{1}{2}}=\{j\, :\, s_{\frac{1}{2}}(A_{\bullet j})\neq \emptyset\}$ and $R^*=s_{\frac{1}{2}}(A_{\bullet j})$ for any $j\in S_{\frac{1}{2}}$. We decompose the matrix $A_{\bullet \overline{S_{\frac{1}{2}}}}$ into blocks whose row index sets are denoted by $C_1,\ldots, C_r$ called \emph{cells}. 
Provided that $A$ is bicyclic, one can prove that $R^*$ is the edge index set of both basic cycles in any bicyclic representation of $A$; the difficulty consists in finding a partition of $R^*$ into subsets say $R_1$ and $R_2$ ($R^*=R_1 \uplus R_2$) such that $A$ has a bicyclic representation where $R_1$ and $R_2$ are the edge index sets of the basic cycles. Actually, we will partition the set $\mathcal{K}$ of cells intersecting $R^*$.

Suppose that $A$ has a bicyclic representation $G(A)$. One can prove that any cell $C_k$ ($1\le k \le r$)
is the edge index set of a subtree in $G(A)$. We say that $G(A)$ \emph{induces the bipartition $\Sigma(\mathcal{K})=\{\mathcal{K}_I, \mathcal{K}_{II} \}$}, where $\mathcal{K}_I$ (respectively, $\mathcal{K}_{II}$) is the set of cells corresponding to subtrees in one basic maximal $1$-tree of $G(A)$ (respectively, the other $1$-tree).
 We will prove that
if $A_{C_k \bullet}^{\frac{1}{2} \rightarrow 1}$ is not a network matrix for some $C_k \in \mathcal{K}$, then the corresponding subtree in $G(A)$ contains a central edge. Similarly, when $s_{\frac{1}{2}}(A_{\bullet j}) \nsubseteq R^*$ for some $j \in S_{\frac{1}{2}}$.

These observations restrict the choices of feasible bipartitions and motivate the definition of \emph{bicompatible} bipartition in $\mathcal{K}$. Then, we define an equivalence relation over the set of bicompatible bipartitions in $\mathcal{K}$. We prove that the cardinality of the resulting quotient set $\mathscr{S}$ is bounded by a constant, actually 18. Moreover, we show the following:
provided that $A$ is bicyclic, if a bicompatible bipartition $\Sigma(\mathcal{K})$ is induced by some bicyclic representation of $A$, then for any bicompatible bipartition $\Sigma'(\mathcal{K})$ equivalent to  $\Sigma(\mathcal{K})$, there exists a bicyclic representation of $A$ inducing $\Sigma'(\mathcal{K})$.

At last, for any representant $\Sigma(\mathcal{K})=\{ \mathcal{K}_I, \mathcal{K}_{II} \}$ of an equivalence class in $\mathscr{S}$, we denote by $M_i(\Sigma)$ the matrix $A_{\cup_{C_k \in \mathcal{K}_i} C_k \bullet }$ without zero columns and $R_i^*(\Sigma)=\cup_{C_k \in \mathcal{K}_i} C_k \cap R^*$ for $i=I$ and $II$; we also define a matrix $N(\Sigma)$ containing  every block of  $A_{\overline{R^*} \times S_{\frac{1}{2}}}$ whose row index set does not intersect $R^*$. We will prove that whenever $A$ has a bicyclic representation $G(A)$ inducing a bipartition $\Sigma(\mathcal{K})$, the subgraph $G(M_i(\Sigma))$ is an $R_i^*(\Sigma)$-cyclic representation of $M_i(\Sigma)$ for $i=I$ and $II$. Furthermore, if we consider the union 
of subtrees with edge index set $C_k$ ($C_k \notin \mathcal{K}$) and all fundamental circuits of nonbasic edges with index in $S_{\frac{1}{2}}$, then by deleting the nonbasic edges and replacing each basic cycle by a basic half-edge, we obtain a basic $\frac{1}{2}$-binet representation of $N(\Sigma)$ (with exactly two basic half-edges).
For any bicompatible bipartition $\Sigma(\mathcal{K})$, by  
computing an  $R_i^*(\Sigma)$-cyclic representation of $M_i(\Sigma)$ for $i=I$ and $II$ and some $\frac{1}{2}$-binet representation of $N(\Sigma)$, if they exist, we derive a bicyclic representation of $A$.\\

Finally, we describe the way of recognizing nonnegative $R^*$-central matrices ($|R^*|=2$)
without any $\frac{1}{2}$-entry. See Chapter \ref{ch:central} for more details.
Let $A$ be a $\{0,1,2\}$-matrix and $\epsilon$ and $\rho$ two row indexes. Our method for determining whether $A$ is $\{\epsilon,\rho\}$-central can be viewed as a generalization of Schrijver's method \cite{ShrAlex} for recognizing network matrices. Let $R^*=\{\epsilon,\rho\}$ and $S^*=\{j\, :\, s(A_{\bullet j}) \cap R^* \neq \emptyset \}$.
We decompose the matrix $A_{\overline{R^*} \times \overline{S^*}}$ into (a maximum number of)
blocks whose row index sets are denoted by $E_1,\ldots,E_b$ and called \emph{bonsais}, and compute a digraph $D$ in the same way as for the recognition of $R^*$-cyclic matrices (see Chapter \ref{ch:multidiD}). 

Our first task is to analyze the significative kinds of bonsais.
Let $S_0=\{j\in S^* \, :\, \epsilon,\rho \in s(A_{\bullet j})\}$,
$S_1=\{j\in S^* \, :\, \epsilon \in s(A_{\bullet j}),\, \rho \notin s(A_{\bullet j})\}$ and $S_2=\{j\in S^* \, :\, \epsilon \notin s(A_{\bullet j}), \,\rho \in s(A_{\bullet j})\}$. We say that a bonsai $E_l$ is \emph{shared} if there exist $j_1 \in S_1$ and $j_2 \in S_2$ such that $E_l \cap s(A_{\bullet j_1}) \neq \emptyset$ and $E_l \cap s(A_{\bullet j_2}) \neq \emptyset$.

Suppose that $A$ has an $\{\epsilon,\rho\}$-central representation $G(A)$. Each bonsai $E_l$ corresponds to the edge index set of a subtree denoted by $B_l$ in $G(A)$ and called a \emph{bonsai}. A bonsai $B_l$ ($1\le l \le b$) is said to be shared if and only if $E_l$ is shared.
Let $T$ be the basic maximal $1$-tree in $G(A)$.
A basic subgraph of $G(A)$ is \emph{on the right}\index{on the right}  \emph{of $\{e_\epsilon,e_\rho\}$} if it is contained in the basic connected subtree of  $T\verb"\" \{e_\epsilon,e_\rho\}$ containing the central node incident to $e_\epsilon$ and $e_\rho$, otherwise  \emph{on the left}\index{on the left}  \emph{of $\{e_\epsilon,e_\rho\}$}.
On the other hand, it may happen that two shared bonsais are in the situation described in Figure \ref{figOverviewAssumpA}. (It would be preferable in 
Figure \ref{figOverviewAssumpA} that $v_{l'}=v_{l,1}$ or $v_{l'}=v_{l,2}$.)
 
\begin{figure}[h!]
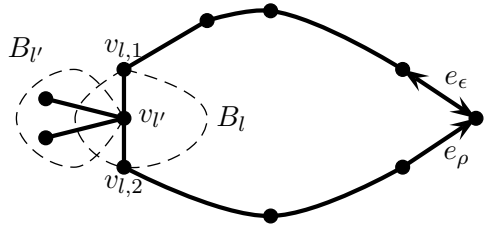


\vspace{0.8cm}

\psset{xunit=1.3cm,yunit=1.3cm,linewidth=0.5pt,radius=0.1mm,arrowsize=7pt,
labelsep=1.5pt,fillcolor=black}

\pspicture(-4,0)(3.5,2.5)

\pscircle[fillstyle=solid](-.8,1.7){.1}
\pscircle[fillstyle=solid](-.8,1.3){.1}

\pscircle[fillstyle=solid](0,1){.1}
\pscircle[fillstyle=solid](0,1.5){.1}
\pscircle[fillstyle=solid](0,2){.1}

\pscircle[fillstyle=solid](0.85,2.5){.1}
\pscircle[fillstyle=solid](1.5,.5){.1}
\pscircle[fillstyle=solid](1.5,2.6){.1}

\pscircle[fillstyle=solid](2.85,1){.1}
\pscircle[fillstyle=solid](2.85,2){.1}
\pscircle[fillstyle=solid](3.6,1.5){.1}

\psline[linewidth=1.6pt,arrowinset=0]{-}(0,1.5)(-.8,1.7)
\psline[linewidth=1.6pt,arrowinset=0]{-}(0,1.5)(-.8,1.3)
\pscurve[linestyle=dashed](0,1.5)(-.5,2)(-1.1,1.5)(-.5,1)(0,1.5)
\rput(-1,2.2){$B_{l'}$}
\rput(.3,1.5){$v_{l'}$}

\rput(0,2.2){$v_{l,1}$}
\rput(0,.8){$v_{l,2}$}

\pscurve[linewidth=1.6pt](0.85,2.5)(1.5,2.6)(2.85,2)
\psline[linewidth=1.6pt,arrowinset=0]{-}(0,2)(0.85,2.5)
\psline[linewidth=1.6pt,arrowinset=0]{-}(0,2)(0,1)
\pscurve[linewidth=1.6pt](0,1)(1.5,.5)(2.85,1)

\psline[linewidth=1.6pt]{<->}(2.85,2)(3.6,1.5)
\rput(3.4,1.9){$e_\epsilon$}

\psline[linewidth=1.6pt]{->}(2.85,1)(3.6,1.5)
\rput(3.4,1.1){$e_\rho$}

\pscurve[linestyle=dashed](0,2)(0.85,1.5)(0,1)(-0.5,1.5)(0,2)
\rput(1.1,1.5){$B_l$}

\endpspicture

\caption{An illustration of two (shared) bonsais $B_l$ and $B_{l'}$ 
in some basic $\{\epsilon,\rho\}$-central representation of $A$. The bonsais $B_l$ and $B_{l'}$ have a common node $v_{l'}$; the intersection of $B_{l}$ with the basic cycle is a path between two nodes say $v_{l,1}$ and $v_{l,2}$, $v_{l'}\neq v_{l,1}$ and 
$v_{l'}\neq v_{l,2}$.}

\label{figOverviewAssumpA}

\end{figure}

In order to bypass this situation, an initialization procedure is performed on $A$, and outputs a matrix $A'$ such that $A$ is $\{\epsilon,\rho\}$-central if and only if $A'$ is $\{1,\rho\}$-central, and
$A'$ satisfies a certain condition called the \emph{assumption $\mathscr{A}$} that prevents the situation described in Figure \ref{figOverviewAssumpA}. Then two cases are distinguished: either $S_0\neq \emptyset$ or $S_0=\emptyset$.

Assume first that $S_0=\emptyset$.
Suppose that $A$ has an $\{\epsilon,\rho\}$-central representation $G(A)$. The main observation is the following: if we look at the succession of bonsais intersecting the edge set of the basic cycle in $G(A)$ (see Figure \ref{figOverviewShared}), we may identify a first sequence of non-shared bonsais, a second of shared bonsais and a third of non-shared  bonsais.
In the sequence of shared bonsais (if not empty), the first and last bonsais, say $B_l$ and $B_{l'}$, are of particular interest: the pair $(E_l,E_{l'})$ or the singleton $(E_l)$ (in case $E_l=E_{l'}$) is called a \emph{left-extreme} set. 
If there is no shared bonsai intersecting the edge set of the basic cycle, then for any shared bonsai $B_l$ on the left of $\{e_\epsilon,e_\rho\}$, the set $(E_l)$ is also called \emph{left-extreme}. For instance, in Figure \ref{figOverviewShared} a (respectively, b), the pair $(E_3,E_5)$ (respectively, $(E_3)$ or $(E_4)$) is left-extreme. 

\begin{figure}[h!]
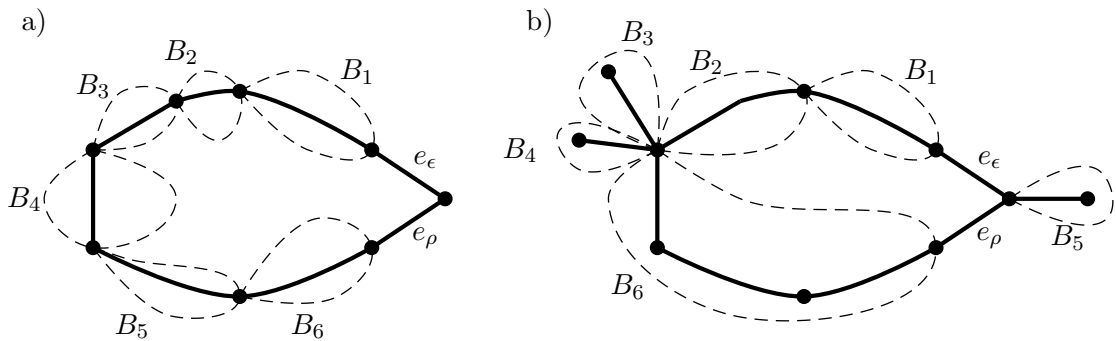


\vspace{1.5cm}

$
\begin{array}{cc}

\psset{xunit=1.3cm,yunit=1.3cm,linewidth=0.5pt,radius=0.1mm,arrowsize=7pt,
labelsep=1.5pt,fillcolor=black}

\pspicture(-1,0)(5.5,2.5)

\pscircle[fillstyle=solid](0,1){.1}
\pscircle[fillstyle=solid](0,2){.1}

\pscircle[fillstyle=solid](0.85,2.5){.1}
\pscircle[fillstyle=solid](1.5,.5){.1}
\pscircle[fillstyle=solid](1.5,2.6){.1}

\pscircle[fillstyle=solid](2.85,1){.1}
\pscircle[fillstyle=solid](2.85,2){.1}
\pscircle[fillstyle=solid](3.6,1.5){.1}

\rput(-.6,3.3){a)}

\pscurve[linewidth=1.6pt](0.85,2.5)(1.5,2.6)(2.85,2)
\psline[linewidth=1.6pt,arrowinset=0]{-}(0,2)(0.85,2.5)
\psline[linewidth=1.6pt,arrowinset=0]{-}(0,2)(0,1)
\pscurve[linewidth=1.6pt](0,1)(1.5,.5)(2.85,1)

\psline[linewidth=1.6pt](2.85,2)(3.6,1.5)
\rput(3.4,1.9){$e_\epsilon$}

\psline[linewidth=1.6pt](2.85,1)(3.6,1.5)
\rput(3.4,1.1){$e_\rho$}

\pscurve[linestyle=dashed](1.5,2.6)(2.2,2)(2.85,2)(2.2,2.8)(1.5,2.6)
\rput(2.7,2.8){$B_1$}

\pscurve[linestyle=dashed](0.85,2.5)(1.1,2.8)(1.5,2.6)(1.3,2.1)(.85,2.5)
\rput(.9,3){$B_2$}

\pscurve[linestyle=dashed](0,2)(0.6,2.1)(0.85,2.5)(0.3,2.6)(0,2)
\rput(0,2.65){$B_3$}

\pscurve[linestyle=dashed](0,2)(0.85,1.5)(0,1)(-0.5,1.5)(0,2)
\rput(-.7,1.5){$B_4$}

\pscurve[linestyle=dashed](0,1)(.8,.3)(1.5,.5)(.3,.9)(0,1)
\rput(.4,.2){$B_5$}

\pscurve[linestyle=dashed](1.5,.5)(2.5,.5)(2.85,1)(2.3,1.3)(1.5,.5)
\rput(2.2,.2){$B_6$}

\endpspicture &

\psset{xunit=1.3cm,yunit=1.3cm,linewidth=0.5pt,radius=0.1mm,arrowsize=7pt,
labelsep=1.5pt,fillcolor=black}

\pspicture(0,0)(3.5,2.5)

\pscircle[fillstyle=solid](-0.8,2.1){.1}
\pscircle[fillstyle=solid](-0.5,2.8){.1}

\pscircle[fillstyle=solid](0,1){.1}
\pscircle[fillstyle=solid](0,2){.1}

\pscircle[fillstyle=solid](1.5,.5){.1}
\pscircle[fillstyle=solid](1.5,2.6){.1}

\pscircle[fillstyle=solid](2.85,1){.1}
\pscircle[fillstyle=solid](2.85,2){.1}
\pscircle[fillstyle=solid](3.6,1.5){.1}

\pscircle[fillstyle=solid](4.4,1.5){.1}

\rput(-1.2,3.3){b)}

\psline[linewidth=1.6pt](-0.8,2.1)(0,2)
\psline[linewidth=1.6pt](-0.5,2.8)(0,2)

\pscurve[linewidth=1.6pt](0.85,2.5)(1.5,2.6)(2.85,2)
\psline[linewidth=1.6pt,arrowinset=0]{-}(0,2)(0.85,2.5)
\psline[linewidth=1.6pt,arrowinset=0]{-}(0,2)(0,1)
\pscurve[linewidth=1.6pt](0,1)(1.5,.5)(2.85,1)

\psline[linewidth=1.6pt](2.85,2)(3.6,1.5)
\rput(3.4,1.9){$e_\epsilon$}

\psline[linewidth=1.6pt](2.85,1)(3.6,1.5)
\rput(3.4,1.1){$e_\rho$}

\psline[linewidth=1.6pt](3.6,1.5)(4.4,1.5)

\pscurve[linestyle=dashed](1.5,2.6)(2.2,2)(2.85,2)(2.2,2.8)(1.5,2.6)
\rput(2.7,2.8){$B_1$}

\pscurve[linestyle=dashed](0,2)(1.3,2.1)(1.5,2.6)
(0.3,2.6)(0,2)
\rput(0.5,2.9){$B_2$}

\pscurve[linestyle=dashed](0,2)(-.2,3)(-.8,2.6)(0,2)
\rput(-0.2,3.2){$B_3$}

\pscurve[linestyle=dashed](0,2)(-1,2.2)(-.8,1.8)(0,2)
\rput(-1.4,2){$B_4$}

\pscurve[linestyle=dashed](3.6,1.5)(4.6,1.3)(4.5,1.8)(3.6,1.5)
\rput(4.2,1.1){$B_5$}

\pscurve[linestyle=dashed](0,2)(0.85,1.5)(2.85,1)
(.8,.3)(-0.5,1.5)(0,2)
\rput(-.3,.6){$B_6$}

\endpspicture 

\end{array}
$

\caption{An illustration of the bonsais intersecting the basic cycle in some basic $\{\epsilon,\rho\}$-central representation of a binet matrix $A$, where $B_3$, $B_4$ and $B_5$ are assumed to be shared. The sequence of shared bonsais intersecting the edge set of the basic cycle is $B_3$, $B_4$ and $B_5$ in case a, and is empty in case b. In case b however, $B_3$ and $B_4$ have exactly one common node with the basic cycle.}

\label{figOverviewShared}

\end{figure}
  
Now consider any shared bonsai $E_l$. By assuming that there exists an $\{\epsilon,\rho\}$-central representation $G(A)$ of $A$ such that $B_l$ is on the left of $\{e_\epsilon,e_\rho\}$, is it possible to compute the indexes of edges belonging to the basic cycle and $B_l$? Surprisingly it is, provided that the left-extreme set is known and of cardinality 2.
This tremendous fact is at the core of our procedure.

The notion of left-extreme set has motivated the definition of left-compatible set (depending only on the matrix $A$)
as we will see.
Provided that $A$ has an $\{\epsilon,\rho\}$-central representation $G(A)$, a left-extreme set of bonsais is necessarily left-compatible. For any left-compatible set $U$ ($|U| \le 2$), we define a $U$-spanning pair $(j_1,j_2)$ of column indexes as follows: $j_1\in S_1$, $j_2 \in S_2$; if $U=(E_u)$, then $s(A_{\bullet j_i}) \cap E_u \neq \emptyset$ for $i=1$ and $2$, and if $U=(E_u,E_{u'})$, then 
$s(A_{\bullet j_1}) \cap E_{u'} \neq \emptyset$ and
$s(A_{\bullet j_2}) \cap E_u \neq \emptyset$.
We also define $V(j_1,j_2)= \{ E_l \,: s(A_{\bullet j_i}) \cap E_l  \neq \emptyset \m{ for } i=1 \m{ or }
2, \m{ or } E_l \m{ is shared} \}$. We will prove that whenever $A$ has an $\{\epsilon,\rho\}$-cyclic representation $G(A)$ and $U$ is left-extreme, the union of the bonsais $B_l$ (with $E_l \in V(j_1,j_2)$) and $\{e_\epsilon,e_\rho\}$ yields
 a basic $1$-tree in $G(A)$. We also define an instance $\Omega(U,j_1,j_2)$ of $2$-SAT, where each variable $x_l$ is associated to a bonsai $E_l \in V(j_1,j_2)$. One necessary condition for a variable to receive a value $0$ or $1$ is that some matrix containing the matrix $A_{E_l \cup \{\epsilon, \rho \} \bullet }$ or $A_{E_l  \bullet }$ is a network matrix.

Given a truth assignment of $\Omega(U,j_1,j_2)$ and a network representation of every bonsai matrix $N_l$ with $E_l \notin V(j_1,j_2)$, we can construct an $\{\epsilon,\rho\}$-central representation $G(A)$ of $A$ such that for any $E_l \in V(j_1,j_2)$, we have 
$x_l=0 \Leftrightarrow B_l$ is on the left of $\{e_\epsilon,e_\rho \}$ in $G(A)$, and for any $E_u \in U $ $x_u=0$.

If we consider a left-compatible set $U$ with only one element say $E_u$, we may have to determine whether the matrix
$A_{E_u \cup \{\epsilon, \rho \} \bullet }$ has a network representation such that $e_\epsilon$ and $e_\rho$ are nonalternating. Thus it is necessary to develop a procedure for recognizing nonnegative $\{\epsilon,\rho\}$-noncorelated network matrices. This is the purpose of Section \ref
{sec:CentralNetcorelated}.

To deal with the case $S_0 \neq \emptyset$, the strategy is very similar to the case $S_0= \emptyset$. We define a notion of $S_0$-straight bonsai. A necessary condition for a bonsai $E_l$ to be $S_0$-straight is that $s(A_{\bullet j}) \cap E_l = 
s(A_{\bullet j'}) \cap E_l \neq \emptyset$ for any $j,j' \in S_0$.

Suppose that $A$ has an $\{\epsilon, \rho\}$-central representation $G(A)$ and $S_0\neq \emptyset$. We prove that if a bonsai $B_l$ intersects the edge set of the basic cycle then $E_l$ is $S_0$-straight. However, it might happen that a bonsai $B_l$ is on the right of $\{e_\epsilon,e_\rho\}$ and $E_l$ is $S_0$-straight. A bonsai $E_u$ is called \emph{right-extreme} if $E_u$ is 
$S_0$-straight, $B_u$ is on the right of $\{ e_\epsilon,e_\rho\}$
and there is no $S_0$-straight bonsai in $G(A)$ "closer" than $B_u$ to the basic cycle. The set of right-extreme bonsais is proved to be of cardinality at most two. This set has motivated the definition of \emph{right-compatibe} set. 

Let $V_0'=\{ E_l \, : \, s(A_{\bullet j}) \cap E_l \neq \emptyset \m{ for some } j\in S_0 \m{ or Ê} E_l \m{ is shared}  \}$.
For any right-compatible set $U$, we define an instance $\Lambda(U)$ of $2$-SAT, where each variable $x_l$ is associated to a bonsai $E_l \in V_0'$. Given  a truth assignment of $\Lambda(U)$ and a network representation of every bonsai matrix $N_l$ with $E_l \notin V_0'$, we can construct an $\{\epsilon,\rho\}$-central representation $G(A)$ of $A$ such that for any $E_l \in V_0'$, we have 
$x_l=0 \Leftrightarrow B_l$ is on the left of $\{e_\epsilon,e_\rho \}$ in $G(A)$.

\section{Contents}\label{sec:Cont}

Chapter \ref {ch:Over} resumes the thesis.  In Section \ref{sec:OverIntro} we introduce the main results of the thesis, and in Section \ref{sec:IntNot} the necessary terms. Some applications
are presented in Section \ref{sec:Ap}. Section \ref{sec:OverRecBinet}
provides some technical details of the most significative part of the thesis: a binet recognition algorithm. This section, Section \ref{sec:Cont}, gives the content of each chapter and section.

In Chapter \ref{ch:Int},  we introduce the main notations and notions. Section \ref{sec:IntbasicMat} is about vectors, matrices and polyhedrons, Section \ref{sec:IntbasicGr} about undirected and directed graphs. Section \ref{sec:IntNet} contains the definition of network matrices and totally unimodular matrices with some known related results.

In Chapter \ref{ch:Bidirected} we define bidirected graphs as a common generalization of undirected and directed graphs. Basic definitions about bidirected graphs are in Section \ref{sec:notions}. In Section \ref{sec:Incidence} we define the node-edge incidence matrix of a bidirected graph and characterize linear independence and dependence in this matrix. We also present some operations on bidirected graphs and the corresponding operations in their node-edge incidence matrix.

We embark on generalizing network matrices in Chapter \ref{ch:Binetdef}, where we define binet matrices. We describe the graphical interpretation of binet matrices in Section \ref{sec:BinetDefi} and some operations on binet matrices in Section \ref{sec:BinetOp}. Section \ref{sec:Binetrep} introduces some particular binet representations. In Section \ref{sec:BinetOpt} we discuss linear and integer programming with the node-edge incidence matrix of a bidirected graph or a binet matrix as constraint matrix. Section \ref{sec:Binetreg} is about
the class of $2$-regular matrices which contains all binet matrices. Section \ref{sec:BinetMat} deals with matroids. To exhibit the connection of binet matrices to existing special classes of matroids, we introduce signed graphs and then the signed-graphic matroid based on these graphs. Some elementary notions about matroids are given in Subsection \ref{sec:BinetMatSub}.

Chapter \ref{ch:Camion} is about Camion bases. We present some characterizations of Camion bases in Sections \ref{sec:di} and \ref{sec:rec} and an algorithm which finds a Camion basis in Section \ref{sec:search}. This Chapter can be read independently of the other ones. 

In Chapter \ref{ch:Rec}, we turn to the problem of recognizing binet matrices by using the results of Chapters \ref{ch:multidiD} to 
\ref{ch:central}. Section \ref{sec:RecBinet} describes the general framework of an algorithm called Binet which determines whether a given matrix is binet in time polynomial in its size. Section \ref{sec:decomppro} describes the second main step of the algorithm Binet, that is a subroutine called Decomposition.

In Chapter \ref{ch:multidiD}, we develop a structure for the recognition of $R^*$-cyclic and $R^*$-central matrices.
Section \ref{sec:mainthm} serves as a tool for the next section. In Section \ref{sec:bonsaimat}, we compute some matrices called bonsai matrices. In
Section \ref{sec:DefDigraphD}, we define a digraph denoted by $D$, and  in Section \ref{sec:FordigaphD}, a feasible spanning forest of $D$.  Section \ref{sec:For} describes an algorithm for computing a feasible spanning forest of $D$, if one exists.

Given a rational matrix $A$ and a row index subset $R^*$ of $A$, one determines whether $A$ has an $R^*$-cyclic representation in Chapter \ref{ch:cyc}. An informal sketch of the recognition procedure is given in Section \ref{sec:infcyc}, and the procedure itself, called RCyclic, with the formal proof of its correctness appears in Section \ref{sec:cyc}.

Similarly, Chapters \ref{ch:halfbinet} and \ref{ch:bicyc}   are about the recognition of $\frac{1}{2}$-binet and bicyclic matrices, respectively. 

In Chapter \ref{ch:central}, given a rational matrix $A$ and two row indexes, say $\epsilon$ and $\rho$, we turn to the problem of recognizing whether $A$ has an $\{\epsilon, \rho\}$-central representation. By letting $S_0=\{j \, :\, \epsilon ,\rho \in s(A_{\bullet j})\}$, a recognition procedure is presented in Section \ref{sec:S0empty} for the case $S_0=\emptyset$, and in Section \ref{sec:S0notempty} for the case $S_0 \neq \emptyset$. Section \ref{sec:PreCentral} provides some definitions and an initialization procedure for the main procedures of Sections \ref{sec:S0empty} and \ref{sec:S0notempty}.  In Section \ref{sec:CentralNetcorelated}, we give a characterization
 of nonnegative $\{\epsilon,\rho\}$-noncorelated matrices and a polynomial recognition procedure for these matrices.

In Chapter \ref{ch:conclude}, we analyze the pertinence of our procedure for recognizing binet matrices and present possible alternatives and future developments.



\chapter{Basic Notions and Definitions}\label{ch:Int}

Here we define the notions used throughout this thesis. The notations we used are standard. We assume familiarity of the reader with the elements of linear algebra, such as linear (in)dependence, rank, determinant, matrix, non-singular matrix, inverse, Gauss' algorithm for solving a system of linear equations, etc.
A reader not familiar with the content of this section can consult, for example, Schrijver \cite{ShrAlex} or Diestel \cite{Diestelgraph}.

As always, $\mathbb{Z}$, $\mathbb{Q}$, and $\mathbb{R}$ denote the set of integer, rational, and real numbers. $\mathbb{N}$ contains the nonnegative integers. The
set of positive real and integer numbers are $\mathbb{R}_+$ and $\mathbb{N}_+$, respectively. The greatest integer smaller than $x\in \mathbb{R}$ is denoted by $\lfloor x\rfloor$. The \emph{projective plane}\index{projective plane}, denoted by $\mathbb{P}^2$, is the sphere (in $\mathbb{R}^3$) where any two antipodal points ($x$ and $-x$) are identified.

For a set $S$, $|S|$ denotes the cardinality of $S$.
Two sets $S$ and $S'$ are called \emph{disjoint}\index{disjoint sets}, if $S\cap S'=\emptyset$.
If $S$ and $S'$ are two disjoint sets, then $S\uplus S'$ denotes the disjoint union of $S$ and $S'$.

A \emph{partial order}\index{partial order} is a binary relation "$\le$" over a set $S$ which is reflexive, antisymmetric, and transitive, i.e., for all $a$, $b$ and $c$ in $S$, we have:

\begin{itemize}

\item $a\le a$ (reflexivity);
\item if $a\le b$ and $b\le a$ then $a=b$ (antisymmetry);
\item if $a\le b$ and $b\le c$ then $a\le c$ (transitivity).

\end{itemize}

A \emph{partially ordered set} (or \emph{poset}\index{poset}) is a set equipped with a partial order relation.

We write $f(x)=O(g(x))$ for real-valued functions $f$ and $g$, if there exists a constant $C$ such that $f(x)\le C g(x)$ for all $x$ in the domain.

If we consider an optimization problem like 

$$ \min\{ \phi(x) \,| \, x\in S\}$$

\noindent 
where $S$ is a set and $\phi: S \rightarrow \mathbb{R}$, then any element $x$ of $S$ is called a \emph{feasible solution}\index{feasible solution} for the minimization problem. A feasible solution attaining the maximum is called an \emph{optimum}\index{optimum} (or \emph{optimal solution}).

\section{Vectors, matrices and polyhedrons}\label{sec:IntbasicMat}

If $x=[x_1,\ldots,x_n]$ and $y=[y_1,\ldots,y_n]$ are row vectors, we write $x\le y$ if $x_i \le y_i$ for $i=1,\ldots,n$.
Similarly for column vectors. If $A$ is a matrix of size $n\times m$, and $b$ is a column vector of size $n$, then the matrix $A$  is called the \emph{constraint matrix}\index{constraint matrix} of the system of linear equations "$Ax=b$".

Vectors and matrices whose elements are integers are called \emph{integral}\index{integral!vector or matrix}. That is, integral $m$-dimensional vectors are those in $\mathbb{Z}^m$, an integral matrix of size $m\times n$ is in $\mathbb{Z}^{m\times n}$. Similarly, \emph{rational}\index{rational vector or matrix} vectors and matrices have elements from $\mathbb{Q}$. \emph{Half-integral}\index{half-integral!vector or matrix} vectors and matrices have elements that are integer multiples of $\frac{1}{2}$.
For  a set $R$ and $R' \subseteq R$,
we denote by $\chi_{R'}^{R}$ the \emph{characteristic vector}\index{characteristic vector} or \emph{incidence vector}\index{incidence vector} of the subset $R'$ of $R$, given by $(\chi_{R'}^{R})_i = \left \{ \begin{array}{ll}
1 & \mbox{ if } i \in R' \\
0 & \mbox{ Otherwise }
\end{array} \right.
$ for all $i\in R$.

Let $A$ be a rational matrix with row set $R$ and column set $S$. $(A)_{ij}$ or $A_{ij}$ or $a_{ij}$ denotes the element in row $i\in R$ and column $j\in S$. 
For  $R'\subseteq R$ and $S'\subseteq S$,
$A_{R'\times S'}$ denotes the submatrix of $A$ whose rows are indexed by $R'$ and columns by $S'$.
$A_{R' \bullet}$ (respectively, $A_{\bullet S'}$) denotes the set of rows of $A$ indexed by $R'$ (respectively, columns of $A$ indexed by $S'$). $A_{R^2}$ is equal to $A_{R \times R}$. For $S' \subseteq S$,
$\overline{S'}=S\verb"\"S'$. $O_{n\times m}$ is a zero $n\times m$ matrix for some integers $n$ and $m$. Sometimes, $I$ denotes a square matrix with ones in the diagonal and zeros elsewhere.
We denote by $kA$ the matrix obtained from $A$ by multiplying all elements by $k$ ($k\in \mathbb{R}$). Thus $A$ is a half-integral matrix, if and only if $2A$ is integral.
$A^T$ is the transpose of $A$.

For any $j\in S$, denote by $s(A_{\bullet j})=
\{i\,:\, A_{ij}\neq 0\}$ the support of $A_{\bullet j}$ and by    
$s_k(A_{\bullet j})= \{i\, : \, A_{ij}=k\}$ the \emph{$k$-support}\index{support@$k$-support}, for any $k\in \mathbb{R}$. A \emph{$k$-entry}\index{entry@$k$-entry} of $A$ is an entry of $A$ equal to $k$. For a set $I\subseteq \mathbb{R}$, an \emph{$I$-matrix}\index{matrix@$I$-matrix} is a matrix all of whose entries are in $I$. For $R'\subseteq R$, we denote $f(R')=\{ j \, : \, s(A_{\bullet j}) \cap R' \neq \emptyset \}$.
$A^{\frac{1}{2}\rightarrow 1}$ denotes the matrix obtained from $A$ by replacing each $\frac{1}{2}$-entry by $1$.
$A$ is called \emph{$\frac{1}{2}$-equisupported}\index{equisupported@$\frac{1}{2}$-equisupported} if for any column indexes $j$ and $j'$ such that $s_{\frac{1}{2}}(A_{\bullet j})\neq \emptyset$ and $s_{\frac{1}{2}}(A_{\bullet j'})\neq \emptyset$, we have 
$s_{\frac{1}{2}}(A_{\bullet j})=s_{\frac{1}{2}}(A_{\bullet j'})$.

If $A=\left[ \begin{tabular}{c|c}
$A_1$ & \\
\hline
       & $A_2$ 
\end{tabular} \right]$ is a matrix that has submatrices $A_1$ and $A_2$ in the upper left and bottom right corners, respectively, and zeros outside them, then this structure of $A$ is its \emph{decomposition}\index{decomposition of a matrix} to components $A_1$ and $A_2$; the matrices $A_1$ and $A_2$ are called \emph{blocks}\index{block} of $A$. A matrix $A$ is said to be \emph{connected}\index{connected!matrix} if and only if $A$ has no zero row or column and $A$ can not be decomposed into two blocks using row and column permutations.

The determinant of a square matrix $A$ is denoted $\det(A)$. The \emph{rank}\index{rank} of a matrix $A$ is denoted $rank(A)$. A matrix is of \emph{full row rank}\index{full row rank}, if its rank equals the number of its rows, or equivalently, if its row vectors are linearly independent. 
A \emph{basis}\index{basis} of a full row rank
$A$ is a non-singular square submatrix of $A$ of size $rank(A)$. The rank of the node-edge incidence matrix $A$ of a connected directed graph $G$ on $n$ nodes is $n-1$. Moreover, by deleting any row, $A$ can be made a full row rank matrix $A'$. The bases of $A'$ correspond to spanning trees of $G$.

A set $S$ of vectors is \emph{convex}\index{convex} if it satisfies: if $x,y \in S$ and $0\le \lambda \le 1$, then $\lambda x+ (1-\lambda)y \in S$. The \emph{convex hull}\index{convex hull} of a set $S$ of vectors is the smallest convex set containing $S$, and is denoted $conv(S)$.

A matrix whose entries are zeros and ones, is said to be an \emph{interval} matrix\index{interval matrix} if its rows can be permuted in such a way that the 1's in each column occur consecutively.

Let $A$ and $A'$ be $n\times m$ matrices. We say that $A$ is \emph{projectively equivalent}\index{projectively equivalent} to $A'$ if there is a nonsingular $n\times n$ matrix $\pi$ such that $A'= \pi A$.

An operation on a matrix $A$ that will be frequently used is \emph{pivoting}\index{pivoting}, that is, the following transformation:

\[
A=\left [\begin{array}{cc} \alpha &
\cc\\ \bb & D \end{array} \right ]  \longrightarrow A'=\left [\begin{array}{cc} \frac{1}{\alpha} &
\frac{\cc}{\alpha}\\ -\frac{\bb}{\alpha} & D- \frac{\bb \cc}{\alpha}\end{array} \right ],
\]

\noindent
where $\alpha$ is a non-zero entry, $\bb$ a
column vector, $\cc$  a row vector, and $D$ a submatrix of $A$.

A set $P$ of vectors in $\mathbb{R}^m$ is a \emph{polyhedron}\index{polyhedron} if $P=\{x \,| \, Ax \le b \}$ for an $n\times m$ matrix $A$ and $n$-dimensional vector $b$. The vectors of a polyhedron are called its \emph{points}\index{point of a polyhedron}. An \emph{extreme point}\index{extreme point} or \emph{vertex}\index{vertex of a polyhedron} of a polyhedron $P=\{x \,| \, Ax \le b \}$ is a point determined by $m$ linearly independent equations from $Ax =b$. Every extreme point of $P$ can arise as an optimal solution of $\max \{ cx \, | \, x \in P\}$ for a suitably chosen $c$.

If $P$ has at least one vertex (in which case $P$ is called \emph{pointed}\index{pointed polyhedron}), then $P$ is called \emph{integral}\index{integral!polyhedron}, if all of its vertices are integral. 
An integral polyhedron $P$ provides integral optimal solutions for $\max \{ cx \, | \,  x \in P\}$ for any $c$. Similarly, if $P$ is half-integral, then the optimal solutions are half-integral.

\section{Graphs and digraphs}\label{sec:IntbasicGr}

An (\emph{undirected}) \emph{graph}\index{graph} is a pair $G=(V,E)$, where $V$ is a finite set, and $E$ is a family of unordered pairs of elements of $V$. The elements of $V$ are called the \emph{vertices}\index{vertex} or \emph{nodes}\index{node} of $G$, and the elements of $E$ are called the \emph{edges}\index{edge} of $G$. For $v,w\in V$, an edge in $E$ connecting $v$ and $w$ is denoted $(v,w)$. The term "family" in the definition of graph means that a pair of vertices may occur several times in $E$.

A \emph{directed graph}\index{directed graph} or \emph{digraph}\index{digraph} is a pair $G=(V,E)$, where $V$ is a finite set, and $E$ is a finite family of ordered pairs of elements of $V$. The elements of $V$ are called the \emph{vertices}\index{vertex} or \emph{nodes}\index{node} of $G$, and the elements of $E$ are called the \emph{edges}\index{edge} or \emph{arcs}\index{arc} of $G$. For $v,w\in V$, an arc in $E$ from $v$ to $w$ is denoted $(v,w)$. The vertices $v$ and $w$ are called the \emph{tail}\index{tail} and the \emph{head}\index{head} of the arc $(v,w)$, respectively. So the difference with undirected graphs is that orientations are given to the pairs. 

For an (undirected or directed) graph $G=(V,E)$, a pair occuring more than once in $E$ is called a \emph{multiple edge}\index{multiple edge}. Graphs without multiple edges are called \emph{simple}. For a graph $G$,
$V(G)$ denotes the vertex set of $G$ and $E(G)$ the edge set. 
Two vertices $v$ and $w$ are \emph{adjacent}\index{adjacent} if they are contained in a same edge. The edge $(v,w)$ is said to be \emph{incident with}\index{incident with} the vertex $v$ and with the vertex $w$, and conversely. The vertices $v$ and $w$ are called the \emph{endnodes}\index{endnode} of the edge $(v,w)$.
The number of edges incident with a vertex $v$ is called the \emph{degree of $v$}\index{degree}. 

For $G=(V,E)$ and $G'=(V',E')$ two graphs, we set $G\cup G':= (V\cup V', E\cup E')$ and $G\cap G':= (V\cap V', E\cap E')$. A graph $G'=(V',E')$ is a \emph{subgraph}\index{subgraph} of $G=(V,E)$ if $V'\subseteq V$ and $E'\subseteq E$. If $E'$ is the family of all edges of $G$ which have both endnodes in $V'$, then $G'$ is said to be \emph{induced by}\index{induced by, subgraph} $V'$. 
For a graph $G$ and a subgraph $G' \subseteq G$, the \emph{edge incidence vector}\index{edge incidence vector} of $G'$ is the vector in $\mathbb{R}^{|E(G)|}$, denoted by $\chi_{G'}$, satisfying

$$(\chi_{G'})_e = \left\{ \begin{array}{ll}
1 & \m{ if } e \in E(G'), \\
0 & \m{ otherwise.}
\end{array}\right.
$$

A simple graph is \emph{complete}\index{complete} if $E$ is the set of all pairs of vertices. Two graphs $G=(V,E)$ and $G'=(V',E')$ are called \emph{isomorphic}\index{isomorphic}, denoted $G\simeq G'$, if there is a bijection $f:V \rightarrow V'$ satisfying, for all $v,w\in V$,

$$ (v,w) \in E \Leftrightarrow (f(v),f(w)) \in E'.$$

\noindent
A graph $G=(V,E)$ is called \emph{bipartite}\index{bipartite} if $V$ can be partitioned into two classes $V_1$ and $V_2$ such that each edge of $G$ contains a vertex in $V_1$ and a vertex in $V_2$. The sets $V_1$ and $V_2$ are called \emph{colour classes}\index{colour classes}.

Let $G=(V,E)$ be an undirected or directed graph.
A \emph{walk}\index{walk}  is a sequence $(v_1,e_1,v_2,\ldots
v_{t-1},$\\
$e_{t-1},v_t)$ where $v_i$ and $v_{i+1}$ are endnodes of edge $e_i$ ($i=1,\ldots,t-1$). The node $v_1$ (resp., $v_t$) is called the \emph{initial node}\index{initial node} (resp., \emph{terminal node}\index{terminal node} or \emph{endnode}\index{endnode}) of the walk.
If the walk does not cross itself, i.e $v_i\neq v_j$ for $1<i<t$, $1\le j\le t$, $i\neq j$, then it is a \emph{path}\index{path}. If $v_1=v_t$, the walk is called \emph{closed}\index{closed walk}. A closed path is called a \emph{cycle}\index{cycle}. In all these subgraphs the directions of the edges are not relevant. When they are, we speak about \emph{directed paths, directed cycle,}\index{directed path}\index{directed cycle}
etc.. The \emph{length}\index{length of a walk} of a walk is its number of edges.

\emph{Removing an edge}\index{removing an edge!from a graph} $e$ from $G$ or \emph{deleting an edge}\index{deleting an edge!in a graph} means removing it and we write $G-\{e\}=(V,E-\{e\})$. 
\emph{Removing a node}\index{removing a node!from a graph} from $G$ or \emph{deleting a node}\index{deleting a node!in a graph} means removing it and all edges incident with it. We write $G-\{v\}$ the graph obtained from $G$ by deleting $v$. For an edge $e=(x,y)$, we denote by $G\verb"/" e$ the graph obtained from $G$ by \emph{contracting}\index{contraction} the edge $e$ into a new vertex $v_e$, which becomes adjacent to all the former neighbours of $x$ and $y$ (if $G$ is directed, we keep the same orientation of the edges).
When we call a graph \emph{minimal}\index{minimal graph} or \emph{maximal}\index{maximal graph} with some property, we are refering to the subgraph relation.

The graph $G$ is \emph{connected}\index{connected!graph} if $G$ is not empty and there is a path between  any two nodes. A \emph{tree}\index{tree} is a connected graph which does  not contain a cycle. A \emph{subtree}\index{subtree} is a connected subgraph of a tree. A subgraph $G'=(V',E')$ of $G=(V,E)$ is a \emph{spanning tree}\index{spanning tree} of $G$ if $V'=V$ and $G'$ is a tree. A \emph{star}\index{star} is a tree with at most one vertex of degree larger than $1$.

A maximal connected subgraph of $G$ is called a \emph{component}\index{component of a graph} of $G$.  A vertex $v$ is called a \emph{cutvertex}\index{cutvertex} if $G-\{v\}$ has one more component than $G$. 
An edge is a \emph{cut-edge}\index{cut-edge}, if it separates two parts of the graph, i.e., after deleting a cut-edge the graph has one more connected component.
$G$ is called $2$-connected if $|G|> 2$ and $G-\{v\}$ is connected for any vertex $v\in G$.

Sometimes, it is convenient to consider one vertex of a tree as special; such a vertex is then called the \emph{root}\index{root} of this tree. A tree with a fixed root is called a \emph{rooted tree}\index{root!rooted tree}. If a rooted tree consists of a set of directed paths entering (respectively, leaving) the root, then it is called an \emph{in-rooted tree}\index{root!in-rooted tree} (respectively, \emph{out-rooted tree}\index{root!out-rooted tree}). 

Suppose that $G=(V,E)$ is a digraph. $G$ is called \emph{strongly connected}\index{strongly connected digraph} if between any two nodes $v$ and $w$ there exists a directed path from $v$ to $w$. A set $V'\subseteq V$ is said to be  \emph{closed in $G$}\index{closed vertex set} if for any arc $(v,v')$ in $G$ with $v\in V'$, we have that $v'\in V'$. A vertex $v$ in $G$ is called a \emph{sink}\index{sink}, if $\{v\}$ is closed in $G$.
If $G$ has no directed cycle, then the \emph{height}\index{height} of a vertex $v \in V$ is the length of the shortest directed path in $G$ from $v$ to a sink vertex, and the set of sink vertices is denoted by $Sink(G)$; moreover, if there is a directed path from a node $v$ to a node $w$, then $w$ is called a \emph{descendant}\index{descendant} of $v$, and $v$ an \emph{ancestor}\index{ancestor} of $w$.

The \emph{node-edge incidence matrix}\index{node-edge incidence matrix} of a graph $G$ has its rows and columns associated with the nodes and edges of the digraph. The non-zeros in a column associated with edge $e$ stand in the rows that correspond to the endnodes of $e$. If $G$ is directed, then heads get positive signs and tails get negative signs, otherwise all non-zero entries are equal to $1$. We denote by \emph{IMD}\index{IMD} the node-edge incidence of a digraph.  An \emph{RIMD}\index{RIMD}, or restricted IMD, is an IMD with (linearly) redundant rows removed. 

An undirected graph $G$ is said to be \emph{embeddable on a surface $S$}\index{embeddable} (for instance $S=\mathbb{R}$) if it is isomorphic to a graph (or pair) $G'=(V',E')$ with the following properties:

\begin{itemize}

\item[(i)] $V', E'\subseteq S$;

\item[(ii)]  different vertices are different points on $S$; 

\item[(iii)] every edge is a curve, that is a continuous injective mapping from the interval $[0,1]$ to $S$, between its endnodes;

\item[(iv)] the interior of an edge contains no vertex and no point of any other edge.

\end{itemize}

\section{Network matrices and totally unimodular matrices}\label{sec:IntNet}

In this Section, we introduce the very well-known class of network matrices and totally unimodular matrices. As an example,
we see that interval matrices are particular network matrices and we show that network matrices constitute a subclass of totally unimodular matrices. 
Moreover, we describe Hoffman's and Kruskal's characterization of totally unimodular matrices, providing the link between total unimodularity and integer linear programming. Then, we present the transshipment problem and the related question of deciding whether a matrix is a network matrix with an application to integer programming. 
We also discuss a deep theorem stating that each totally unimodular matrix arises by certain compositions from network matrices and from certain $5\times 5$ matrices. 
Finally, we provide some definitions that are used in Chapters \ref{ch:multidiD} and \ref{ch:central}.

Let $G=(V,E)$ be a digraph and $In$ the $V\times E$-incidence matrix of $G$. Let $In'$ be an RIMD obtained from $In$, $B$ a basis of $In'$ and suppose that $In'=(B\, N)$. The matrix $A=B^{-1}N$ is called a {\em network matrix}\index{network!matrix}.

Edges in $G$ corresponding to columns of $B$ (resp., $N$) are called \emph{basic}\index{basic!edge} (resp., \emph{nonbasic}) edges. Basic edges, respectively nonbasic ones, are in one-to-one correspondance with rows of $A$, respectively columns of $A$.
The digraph $G$ with the indication of basic and nonbasic edges is called a \emph{network representation}\index{network!representation} of $A$ (not unique in general) and is denoted $G(A)$.

Suppose now that $G$ is connected. It is not difficult to see that the basis $B$ of $In'$ corresponds to a spanning tree $T=(V,E_0)$ of $G$. It is known that for $e\in E_0$ and an edge $f=(u,v)\in E$:

\begin{eqnarray}
a_{e,f} :=  \left \{
\begin{array}{cl}
1 & \m{if the unique } \m{u-v-path in } T \m{ passes through }  e \m{
forwardly},\\
-1 & \m{if the unique } \m{u-v-path in } T \m{ passes through }  e \m{
backwardly},\\
0 & \m{if the unique } \m{u-v-path in } T \m{ does not pass through }  e  .
\end{array}
\right. \label{eqn:Introfund}
\end{eqnarray}

\noindent
The unique closed path going from $u$ to $v$ through $f$ and then from $v$ to $u$ in $T$ is called the {\em fundamental cycle of $f$}\index{fundamental cycle}.

Interval matrices are special network matrices, in which the spanning tree is a directed path. The non-tree edges then can be associated with subpaths. By definition, the non-zero elements in each column of an interval matrix are equal to $1$, and there is a permutation of the rows such that the resulting matrix has consecutive ones in each column. Thus the columns can be considered as characteristic vectors of intervals on a line with finite segments. Hence the name interval matrix.

A matrix $A$ is \emph{totally unimodular}\index{totally unimodular} if each subdeterminant of $A$ is $0$, $+1$, or $-1$.
In particular, each entry in a totally unimodular matrix is $0$, $+1$, or $-1$. A matrix of full row rank is \emph{unimodular}\index{unimodular} if $A$ is integral, and each basis of $A$ has determinant $\pm 1$. It is easy to see that a matrix $A$ is totally unimodular if and only if the matrix $[I \, A]$ is unimodular.

We prove here that every network matrix is totally unimodular.

\begin{thm}\label{thmSubclassNNetTot}
Network matrices are totally unimodular.
\end{thm}

\begin{proof}
Let $A$ be a network matrix. By definition, there exist an
RIMD  $In$ and a basis $B$ of $ In$ such that $ In =[B \, N]$ and $A=B^{-1}N$.  Let $S$ be a square submatrix of $A$ with nonzero determinant. Then $det(S)=det(S')$ for some basis $S'$ of $[I \, A]$. Since the matrix  $[B \,  In]$ is 
an IMD, it is not difficult to prove that any basis of $[B \, In]$ has a determinant equal to $\pm 1$. This implies that $\det(B)=\pm 1$ and $\det(B S')=\pm 1$ (because $BS'$ is a basis of $[B \, In]$). Thus $\det(S)=\pm 1$.
\end{proof}\\

The link between unimodularity or total unimodularity and integer linear programming is given by the following fundamental results (see \cite{ShrAlex} for the proof).

\begin{thm}\label{thmSubclassuni} Let $A$ be an integral matrix
of full row rank. Then $A$ is unimodular if and only if for each integral vector $b$ the polyhedron $\{x \,| \, x \geq 0; Ax = b \}$ is integral.
\end{thm}

\begin{thm}[(Hoffman and Kruskal's theorem)]\label{thmSubclassNHof} Let $A$ be an integral matrix. \\ Then $A$ is totally unimodular if and only if for each integral vector $b$ the polyhedron $\{x \,| \, x \geq 0; Ax \le b \}$ is integral.
\end{thm}

Consider the linear programm 
\begin{eqnarray}
\min  & c^T x \nonumber \\
 s.c. &  Ax = b \label{eqnSubclassNprog}\\
      & x\geq 0 \nonumber
\end{eqnarray}
\noindent
 where $c\in \mathbb{R}^m$, $b\in \mathbb{R}^n$ and $A \in \mathbb{R}^{n \times m}$ for some $n,m\in \mathbb{N}$. Suppose that the constraint matrix $A$ is integral and of full row rank.
Then by Theorem \ref{thmSubclassNHof}, $A$ is unimodular if and only if for any integral right-hand side $b$, there exists an integral optimal solution. 

If the constraint matrix in (\ref{eqnSubclassNprog}) is an RIMD, then the problem is called a transshipment\label{mycounter1} problem or more simply a network problem in the literature and the interpretation of (\ref{eqnSubclassNprog}) is the following. Variables correspond to flows in edges, while the constraints correspond to supply (if $b_r <0$), demand (if $b_r > 0$) or flow conservation (if $b_r=0$) requirements at the vertices; the objective is to find a minimum cost flow subject to these constraints.

Very efficient solution techniques are available for (\ref{eqnSubclassNprog}) such as the Edmonds-Karp scaling method \cite{EdmKarpNetwork-70}, and the network simplex method \cite{CunWH-NetSimMeth,CunNetwork-79,OrlinNetSimpMethod-96}). Second, if $b$ is integral, then any of the standard solution techniques will find an integral solution.

Suppose that the constraint matrix $A$ in (\ref{eqnSubclassNprog}) is not an RIMD and we wish to convert the problem (\ref{eqnSubclassNprog}) to a transshipment problem in order to take advantage of the integrality properties of transshipment problems. It would be nice to have an algorithm for deciding when a given matrix can be converted to an RIMD by using only elementary row operations. This is equivalent to the problem of recognizing network matrices as stated in the next theorem (see \cite{BixCun-80}).

\begin{thm}\label{thmSubclassNEqui}
Let $A$ be a real $n\times m$ matrix in standard form (the first $n$ columns of $A$ form the identity matrix).
Then $A$ is projectively equivalent to an RIMD if and only if $A$ is a network matrix.
\end{thm}

\begin{proof}
We have $A=[I\, A']$ where $A'$ is a submatrix of $A$. Suppose that there exists a nonsingular matrix $\pi$ such that $\pi [I\, A'] = In$ and $In$ is an RIMD. It suffices to prove that $A'$ is a network matrix.
Denote $N=\pi A'$. So $\pi$ is a basis of the RIMD $In=[\pi N]$ and $A'=\pi^{-1} N$.

Conversely, if $A$ is a network matrix, then $A'=B^{-1}N$ where $[B \, N]$ is an RIMD and $B$ a basis of $[B \, N]$, then $B A$ is an RIMD.
\end{proof}\\

There exist several algorithms to test if a given matrix is a network matrix. Such an algorithm was designed by Auslander and Trent \cite{AusTrentNetwork-59}, Gould \cite{GouldNetwork-58}, Tutte \cite{Tutte-60,Tutte-65,Tutte-67}, Bixby and Cunningham \cite{BixCun-80} and Bixby and Wagner \cite{BixWagner-88}. 
A famous one is Schrijver's method\label{mycounter3} developed in \cite{ShrAlex} which adapts the matroidal ideas of Bixby and Cunningham \cite{BixCun-80} to matrices. The algorithm works by reducing the problem to a set of smaller problems, which can be handled easily. The smaller problems consist of deciding if a matrix with at most two non-zeros per column is a network matrix or not. The reduction is done by identifying rows of the matrix that correspond to cut-edges of the spanning tree, and then carrying on with the smaller matrices associated with the components.
We mention here a theorem stated in \cite{BixCun-80} that will be used several times in the following chapters.

\begin{thm}\label{thmSubclassNTutteCunNet}
Let $A$ be a real matrix of size $n\times m$ and $\alpha$ be the number of nonzeros in $A$. Then there exists an algorithm with input $A$ that provides matrices $B$ and $N$ such that $A=B^{-1} N$ and $[B\, N]$ is an RIMD, or proves that $A$ is not a network matrix. The time complexity of the algorithm is $O(n \alpha)$.
\end{thm}

Network matrices form the basis for all totally unimodular matrices. This was shown by a beautiful theorem of Seymour \cite{SeymourRegular-80} (cf \cite{SeymourRegularAppl-85}). Not all totally unimodular matrices are network matrices or their transposes, as is shown by the matrices (cf Bixby \cite{BixKurWag-77}):

\begin{eqnarray}\label{eqnSubclassNmat2}
\left[
\begin{array}{ccccc}
1 &-1& 0&0 &-1 \\
-1 &1 &-1 &0 &0 \\
0 &-1 &1 &-1 &0 \\
0 &0 &-1 &1 &-1 \\
-1 &0 &0 &-1 &1 \\
\end{array}
\right], &\quad 
\left[
\begin{array}{ccccc}
1 &1& 1&1&1 \\
1 &1 &1 &0 &0 \\
1 &0 &1 &1 &0 \\
1 &0 &0 &1 &1 \\
1 &1 &0 &0 &1 \\
\end{array}
\right]
\end{eqnarray}

Seymour showed that each totally unimodular matrix arises, in a certain way, from network matrices and the matrices \ref{eqnSubclassNmat2}.

To describe the characterization theorem, first observe that total unimodularity is preserved under the following operations:

\begin{itemize}

\item[(i)] permuting rows or columns;

\item[(ii)] taking the transpose;

\item[(iii)] signing rows or columns;

\item[(iv)] pivoting;

\item[(v)] adding an all-zero row or column, or adding a row or column with one nonzero, being $\pm 1$;

\item[(vi)] repeating a row or a column;

\end{itemize}

Moreover, total unimodularity is preserved under the following compositions:

\begin{itemize}

\item[(vii)] $A \bigoplus_1 B := \left[\begin{array}{cc}
A & 0 \\
0 & B 
\end{array}\right]$ \quad \quad \quad \quad \quad \quad \quad\quad\quad\quad \quad\quad(1-sum);

\item[(viii)] $\left[\begin{array}{cc}
A & a \\
\end{array}\right] \bigoplus_2 \left[\begin{array}{cc}
b \\
B
\end{array}\right]:= \left[\begin{array}{cc}
A & ab \\
0 & B 
\end{array}\right]$ \quad \quad \quad \quad\quad\quad \quad (2-sum);

\item[(ix)] $\left[\begin{array}{ccc}
A & a & a \\
c & 0 & 1  
\end{array}\right] \bigoplus_3 \left[\begin{array}{ccc}
1 & 0 & b  \\
d & d & B 
\end{array}\right] := \left[\begin{array}{ccc}
A & ab \\
dc & B 
\end{array}\right]$ \quad \quad \,\,(3-sum);

\end{itemize}

(here $A$ and $B$ are matrices, $a$ and $d$ are column vectors, and $b$ and $c$ are row vectors, of appropriate sizes). It is not difficult to see that total unimodularity is maintained under these compositions, e.g. with Ghouila-Houri's characterization (see \cite{ShrAlex}).

\begin{thm}\label{thmSubclassNSey}
[Seymour's decomposition theorem for tot. uni. matrices.] A matrix $A$ is totally unimodular if and only if $A$ arises from network matrices and the matrices \ref{eqnSubclassNmat2} by applying the operations (i), (ii), ... (ix). Here the operations (vii), (viii) and (ix) are applied only if both for $A$ and for $B$ we have: the number of rows plus the number of columns is at least $4$.
\end{thm}

Finally, let $A$ be a nonnegative connected network matrix, and suppose that we are given two row indexes, say $\epsilon$ and $\rho$ ($\epsilon \neq \rho$). Let $G(A)$ be a basic network representation of $A$ and $q$ a path in $G(A)$ containing $e_\epsilon$ and $e_\rho$. If $q$ passes through one of the edges $e_\epsilon$ and $e_\rho$ forwardly and through the other one backwardly, then we say that $e_\epsilon$ and $e_\rho$ are \emph{alternating}\index{alternating} in $G(A)$, otherwise \emph{nonalternating}.
If $A$ has a basic network representation in which $e_\epsilon$ and $e_\rho$ are alternating and an other one in which they are nonalternating, then $N$ is said to be \emph{$\{ \epsilon,\rho\}$-noncorelated}, otherwise \emph{$\{ \epsilon,\rho\}$-corelated\index{corelated}}.

For $R$ a row index subset, an \emph{$R$-network}\index{network@$R$-network!representation} representation denotes a network representation such that $R$ is the index set of one basic edge. A matrix is called an \emph{$R$-network}\index{network@$R$-network!matrix} matrix if it has an $R$-network representation. 
A \emph{basic}\index{basic!subgraph} subgraph of a network representation $G(A)$, is a subgraph consisting of only basic edges.


\chapter{Bidirected graphs}\label{ch:Bidirected}

In this section, we describe bidirected graphs. The notion of bidirected graphs was introduced by Edmonds \cite{EdmondsMatching1} as a common 
generalization of both directed and undirected graphs. 
In a directed graph, if an endnode of an edge is its tail, then the other endnode of the edge must be its head. Undirected graphs can be viewed as graphs in which each edge has two heads. In a bidirected graph,  endnode of an edge can be its head or tail, independently from each other. These graphs serve as a background for introducing binet matrices in Chapter \ref{ch:Binetdef}.

Largely, we adopt the terminology and definitions  given by Appa and Kotnyek \cite{KotThesis}, \cite{Appa-06} and Zaslavsky \cite{Zaslavsky-Glos-99}. Bidirected graphs have appeared in the literature several times. Schrijver \cite{ShrAlex-Dis-91} gave a necessary and sufficient condition for the existence of an integer solution to a linear inequality system $A x \le b$ in which $A$ is the edge-node incidence matrix of a bidirected graph. Gerards and Schrijver \cite{GerSchAlex-Edm-86} characterized bidirected graphs which lead to matrices with strong Chv\'atal rank 1.

\section{Basic notions}\label{sec:notions}

A bidirected graph $G=(V,E)$ on node set $V=\{v_1,\ldots,v_n\}$ and with edge set $E=\{e_1,\ldots,e_m\}$ may have four kinds of edges. A \emph{link}\index{link} is an edge with two distinct endnodes;  a \emph{loop}\index{loop} has tow identical endnodes; a \emph{half-edge}\index{half-edge} has one endnode, and a \emph{loose edge}\index{loose edge} has no endnode at all. 

Every edge is signed with $+$ or $-$ at its endnodes. That is, links or loops can be signed with $+-$, $++$ or $--$; half-edges have only one sign, $+$ or $-$. If an edge is signed with $+$ at an endnode, then this node is an \emph{in-node}\index{in-node} or \emph{head}\index{head} of the edge. An endnode signed with $-$ is called the \emph{out-node}\index{out-node} or \emph{tail}\index{tail} of the edge. One can think of the value $+$ as indicating that the edge is directed into the node, $-$ indicating direction away from the node. 
Figure \ref{fig:bidirectedGraph} shows two possible graphical representations of the same bidirected graph. Heads and tails can be represented with signs as in (i), or arrows as in (ii).
A link or loop with two different signs at its endnodes is said to be a \emph{directed edge}\index{directed edge}.
For simplicity, a directed edge is represented by only one arrow 
instead of two (see $e_4$ for example in Figure \ref{fig:bidirectedGraph}).
All edges, except directed and loose edges, are called \emph{bidirected edges}\index{bidirected!edge}.
If $e_j$ is a link or a loop with endnodes $v_i$ and $v_{i'}$ ($v_i=v_{i'}$ in the case of a loop), then we may denote $e_j$ as $[v_i,v_{i'}]$, $]v_i,v_{i'}[$, $[v_i,v_{i'}[$  or $]v_i,v_{i'}]$ depending on the sign at the endnode; actually a closed (resp., open) bracket at $v_i$ means that one sign at this endnode is $+$ (respectively, $-$). Similarly, a half-edge incident with $v_i$ may be denoted by $[v_i]$ or $]v_i[$.
For instance in Figure \ref{fig:bidirectedGraph}, $e_1=]v_1[$, $e_2=[v_1,v_2]$, $e_3=[v_2,v_2]$ and $e_4=]v_2,v_3]$.

\begin{figure}[h!]
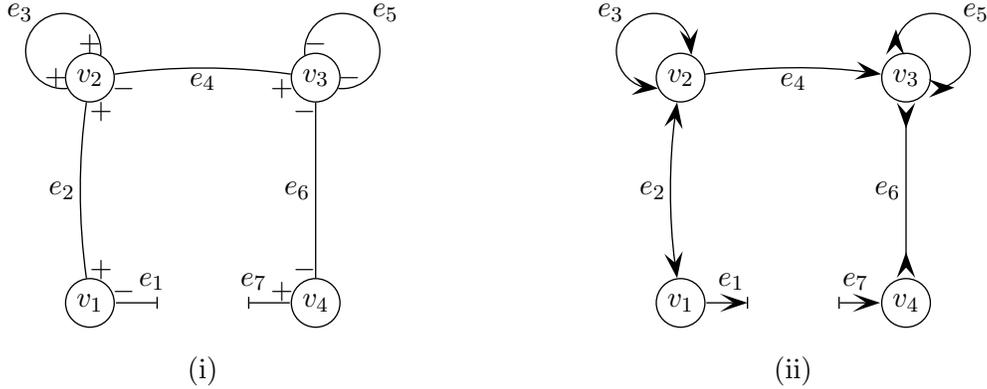

\psset{xunit=1.5cm,yunit=1.5cm,linewidth=0.5pt,radius=3mm,arrowsize=7pt,
labelsep=1.5pt}

\begin{center} 

$\begin{array}{cc}

\pspicture(0,0)(5,3)

\rput(2,-.6){(i)}
\cnodeput(1,0){1}{$v_1$}
\cnodeput(1,2){2}{$v_2$}
\cnodeput(3,2){3}{$v_3$}
\cnodeput(3,0){4}{$v_4$}

\ncarc{-}{1}{2}
\naput{$e_2$}
\rput(1.1,.3){$+$}
\rput(1.1,1.7){$+$}
\ncarc{-}{2}{3}
\nbput{$e_4$}
\rput(1.3,1.9){$-$}
\rput(2.7,1.9){$+$}
\ncline{-}{3}{4}
\nbput{$e_6$}
\rput(2.9,.3){$-$}
\rput(2.9,1.7){$-$}

\psline[arrowinset=.5,arrowlength=1.5]{-|}(1.23,0)(1.6,0)
\rput(1.55,0.19){$e_1$}
\rput(1.3,.1){$-$}

\psline[arrowinset=.5,arrowlength=1.5]{|-}(2.4,0)(2.77,0)
\rput(2.45,0.19){$e_7$}
\rput(2.7,.1){$+$}

\nccircle[angle=45]{-}{2}{0.5cm}
\nbput{$e_3$}
\rput(.7,2){$+$}
\rput(1,2.3){$+$}

\nccircle[angle=-45]{-}{3}{0.5cm}
\nbput{$e_5$}
\rput(3.3,2){$-$}
\rput(3,2.3){$-$}

\endpspicture  &

\pspicture(0,0)(4,3)

\rput(2,-.6){(ii)}
\cnodeput(1,0){1}{$v_1$}
\cnodeput(1,2){2}{$v_2$}
\cnodeput(3,2){3}{$v_3$}
\cnodeput(3,0){4}{$v_4$}

\ncarc{<->}{1}{2}
\naput{$e_2$}
\ncarc{->}{2}{3}
\nbput{$e_4$}
\ncline{>-<}{3}{4}
\nbput{$e_6$}

\psline[arrowinset=.5,arrowlength=1.5]{->|}(1.23,0)(1.6,0)
\rput(1.45,0.19){$e_1$}

\psline[arrowinset=.5,arrowlength=1.5]{|->}(2.4,0)(2.77,0)
\rput(2.55,0.19){$e_7$}

\nccircle[angle=45]{<->}{2}{0.5cm}
\nbput{$e_3$}
\nccircle[angle=-45]{>-<}{3}{0.5cm}
\nbput{$e_5$}

\endpspicture  

\end{array}$

\end{center}
\vspace{.8cm}

\caption{Possible graphical representations of a bidirected graph}\label{fig:bidirectedGraph}
\end{figure}

Every edge $e$ is given a \emph{sign}\index{sign}, denoted by $\sigma_e\in \{ +, -\}$. The signing convention we adopt is that the sign of a link or loop is $-$ times the products of the signs of its endnodes, the sign of a half-edge is always negative, and the sign of a loose edge is positive. Thus, a positive link or loop is a directed edge; as this is just like an ordinary directed graph, a directed graph is the same as a bidirected all-positive signed graph. The case of a positive loop arises when dealing with contraction of graphs, an operation defined later in the section. The normal case we encounter is that of a negative loop (where its endnode is either an in-node or an out-node). In Figure \ref{fig:bidirectedGraph}, the links $e_2$ and $e_6$ and loops $e_3$ and $e_5$ are negative, while the link $e_4$ is positive.

A \emph{walk}\index{walk} in a bidirected graph is a sequence $(v_1,e_1,v_2,\ldots
v_{t-1},e_{t-1},v_t)$ where $v_i$ and $v_{i+1}$ are endnodes of edge $e_i$ ($i=1,\ldots,t-1$), including the case where $v_i=v_{i+1}$ and
$e_i$ is a half-edge. If the walk consists of only links and it does not cross itself, i.e $v_i\neq v_j$ for $1<i<t$, $1\le j\le t$, $i\neq j$, then it is a \emph{path}\index{path}. In the bidirected graph of Figure \ref{fig:bidirectedGraph}, we have the path $(v_1,[v_1,v_2],v_2,]v_2,v_3],v_3,]v_3,v_4[,v_4)$. A closed walk which does not cross itself (except at $v_1=v_t$)
and goes through each edge at most once
is called a \emph{cycle}\index{cycle}. So a loop, a half-edge or a closed path
can make up a cycle. In Figure \ref{fig:bidirectedGraph}, there are exactly four cycles. The \emph{sign of a cycle}\index{sign of a cycle} is the product of the signs of its edges, so we have a \emph{positive cycle}\index{positive cycle} if the number of negative edges (or bidirected edges) in the cycle is even; otherwise, the cycle is a \emph{negative cycle}\index{negative cycle}. Obviously, a negative loop or a half-edge always makes a negative cycle. A \emph{full}\index{full cycle} cycle in a bidirected graph is a cycle different from a half-edge.

A bidirected graph is \emph{connected}\index{connected!bidirected graph}, if there is a path between any two nodes. 
A \emph{tree}\index{tree} is a connected bidirected graph which does not contain a cycle.
A connected bidirected graph containing exactly one cycle is called a \emph{1-tree}\index{tree@1-tree}, indicative of the fact that a $1$-tree consists of a tree and one additional edge. If the unique cycle in a 1-tree is negative, then we will call it a \emph{negative 1-tree}\index{negative 1-tree}. A \emph{path-wheel}\index{path-wheel} in a bidirected graph is a subgraph consisting of a negative  cycle and a path from a node of the cycle to another one not in the cycle.

In a walk $(v_0,e_0,v_1,e_1\ldots,v_{t-1},e_{t-1},v_t)$ of $G$, a node $v_i$ is \emph{consistent}\index{consistent} if the signs at the endnode $v_i$ of edges $e_{i-1}$ and $e_i$ are different.
(This definition applies to $v_0$, with subscripts modulo $t$, if $v_0=v_t$ and $t>0$.) A path is said to be \emph{consistently oriented}\index{consistently oriented} if all nodes, except the first and last ones, are consistent. A \emph{directed}\index{directed path!in a bidirected graph} path is a consistently oriented path with only directed edges. 

\section{Incidence matrix}\label{sec:Incidence}

In this section, we define the  \emph{(node-edge) incidence matrix }\index{node-edge incidence matrix} $In(G)$ of a bidirected graph $G$ and called an IMB\index{IMB}. We describe natural operations on $In(G)$ and the corresponding operations on $G$, and study linear independence and dependence of columns in $In(G)$.
The rows and columns of $In(G)$ are identified with the nodes and edges of $G$, respectively. An entry $(i,j)$ of $In(G)$ is $1$ (resp., $-1$) if $e_j$ is a link or a half-edge entering (resp., leaving) $v_i$, $2$ (resp., $-2$) if $e_j$ is a negative loop entering (resp., leaving) $v_i$, $0$ otherwise. As an  example, the node-edge incidence matrix of the bidirected graph depicted in Figure \ref{fig:bidirectedGraph} is:

$$In=\left[ \begin{matrix} -1 & 1 & 0 & 0& 0&0 & 0 \\
0& 1&2 &-1 &0 &0 &0 \\
0&0&0&1&-2&-1&0\\
0&0&0&0&0&-1&1
\end{matrix}\right]$$

\noindent
An \emph{RIMB}\index{RIMB}, or restricted IMB, is an IMB with (linearly) redundant rows removed.
The relationship of incidence matrices and bidirected graphs is two-way (or bidirectional). Given an $n\times m$ integral matrix $In$ satisfying

\begin{equation}\label{eqnBidirected}
\sum_{i=1}^n |(In)_{ij}| \le 2 \m{ for } j=1,\ldots,m,
\end{equation}

\noindent
one can find a bidirected graph $G(M)$ with $n$ nodes and $m$ edges such that its node-edge incidence matrix is $In$. In other words, property (\ref{eqnBidirected}) characterizes the node-edge incidence matrices of bidirected graphs. 

Operations on $In(G)$ that maintain (\ref{eqnBidirected}) can be translated to operations on bidirected graphs. Such operations are, for example, multiplying a row or column with $-1$, or deleting a row or column. It can be easily verified that the following transformation also maintains (\ref{eqnBidirected}).

\begin{equation}\label{eqnContraction}
In= \left [\begin{array}{cc} \alpha &
\cc \\ \bb & D  \end{array} \right ] \rightarrow
In'=D-\alpha \bb \cc
\end{equation}

\noindent
where $\alpha$ is a non-zero entry, $b$ is a column vector, $c$ is a row vector, and $D$ is a submatrix of $In$.

We now give the graphical equivalents of these operations. They are extensions of standard operations on directed or undirected graphs, such as edge or node deletion, and edge contraction, taking into account the signs of the edges. When multiplying the $j$th column of $In$ with $-1$, we simply change the signs at the endnodes of edge $e_j$. In-nodes of the edge become out-nodes and vice versa, but the sign of the edge is unchanged. We call this operation \emph{reversing the orientation}\index{reversing the orientation} of an edge. Multiplying a row with $-1$ changes the signs at the corresponding endnode of all the incident edges. If the node was an in-node of an edge, it becomes its out-node and vice versa. Consequently, the sign of the incident links or loops change, positive links or loops become negative and vice versa. We call this operation \emph{switching}\index{switching} at a node.

Column deletion easily translates to bidirected graphs, it is equivalent to \emph{deleting an edge}\index{deleting an edge} from the graph. Deletion of a row corresponds to the removal of the associated node together with edge-ends incident to the node. That is, links connected to this node become half-edges while loops and half-edges located at the deleted node become loose edges. All other edges and nodes remain unchanged. we call this operation \emph{deleting a node}\index{deleting a node!in a bidirected graph} or \emph{removing a node}\index{removing a node!from a bidirected graph}.

Finally, the transformation (\ref{eqnContraction}) translates to a \emph{contraction}\index{contraction} in the bidirected graph. If $e_1$ is a negative loop or a half-edge, then $A'=D$, that is, edge $e_1$ and node $v_1$ are deleted from $G$. If $e_1$ is a positive link, then we get the ordinary graph contraction, that is, deletion of the edge and unification of its endnodes. If $e$ is a negative link, we first switch at $v_1$ and then contract the now positive $e_1$. 

Let $Q$ be a submatrix of the incidence matrix $In(G)$. It can be obtained by row and column deletions. By analogous operations, a bidirected graph $G(Q)$ can be obtained from $G$. It is clear that $G(Q)$ does not depend on the order of row and column deletions. This graph can be achieved by edge and node deletions, but strictly speaking, it is not a subgraph of $G$, as it can contain half-edges and loose edges that are not present in the original edge set $E$. However, when it does not create confusion, we will call it a subgraph of $G$.

In the rest of this section, we give a graphical characterization of linear independence and dependence in the node-edge incidence matrix of a bidirected graph. It is easily shown that columns of the matrix corresponding to a tree in the graph are linearly independent. However, contrary to the directed case, in bidirected graphs there are other linearly independent structures, as stated in Lemma \ref{lemBidirectedNonsing}. A subset $C$ of edges in $G$ is called a \emph{circuit}\index{circuit} if the columns of the matrix $In(C)$ form a minimal dependent set.
The propositions listed below are needed to establish Lemma \ref{lemBidirectedNonsing}. We omit their proofs as they can be found in \cite{KotThesis}, or derived from very similar results appearing in, for example, \cite{GerSchAlex-Edm-86,Zaslavsky-Signed-82}.

\begin{prop}\label{propBidirectedmatQ}
Let $Q$ be a square matrix of size $n\geq 2$ whose nonzero elements are $q_{ii}=1$ for $i=1,\ldots,n$; $q_{i+1,i}=-1$ for $i=1,\ldots, n-1$ and $Q_{1 n}=\pm 1$. That is, $Q$ is of the following form 

\begin{equation}\label{equ:Rcycle}
Q=\left [ \begin{array}{ccccc}
1 & & & & \pm 1 \\
-1 & 1 & & & \\
   & -1 & \ddots & & \\
   & & \ddots & 1 & \\
   & & & -1 & 1\\
\end{array} \right ] 
\end{equation}

\noindent
Then $det(Q)=0$ if $q_{1n}=-1$ and $det(Q)=2$ if $q_{1n}=1$.
\end{prop}

\begin{prop}\label{propBidirectedSwitch}
Switchings at nodes do not change the sign of a cycle. 
\end{prop}

It is easy to verify (and available as lemmas 3.1 in \cite{Zaslavsky-Signed-82}) that any tree graph can be transformed to a directed graph by switching. More generally, this leads to:

\begin{prop}
A negative $1$-tree can be transformed to a directed tree together with one extra bidirected edge forming a negative cycle with the tree by switchings. 
\end{prop}

\begin{lem}\label{lemBidirectedNonsing}
A square submatrix $Q$ of the incidence matrix of a bidirected graph $G$ is non-singular if and only if each connected component of $G(Q)$ is a negative $1$-tree.
\end{lem}

This result is derived in \cite{Zaslavsky-Signed-82}, Theorem 5.1, in matroidal terminology. We give here the proof as in \cite{Appa-06} without recourse to matroids.

\begin{proof}
Is is enough to prove the theorem for the case where $G(Q)$ is a connected graph and $Q$ has no zero column.
$Q$ is a square submatrix, so it corresponds to a $1$-tree subgraph. Expanding $\det(Q)$ by rows corresponding to the pendant nodes of $G(Q)$, it is easy to establish that the singularity of $Q$ depends on the singularity of its submatrix $In(C)$ corresponding to the cycle $C$ of the $1$-tree. If $C$ is a loop or a half-edge, then it is obviously negative and $In(C)$ is non-singular. Otherwise, by permutations and row and column multiplications with $-1$ on $In(C)$, or equivalently by switchings and orientation reversals in $C$, $In(C)$ can be brought to a form of (\ref{equ:Rcycle}). It follows then from the previous propositions that $C$ is negative if and only if $In(C)$ is non-singular.
\end{proof}\\

This lemma has the following easy consequences.

\begin{cor}
Let $In$ be a full row rank node-edge incidence matrix of a bidirected graph, and $Q$ a collection of linearly independent columns of $In$. Then each connected component of $G(Q)$ either forms a tree or a negative $1$-tree. Conversely, if every component of a subgraph $G(Q)$ forms a tree or a negative $1$-tree, then the columns of $Q$ are linearly independent.
\end{cor}

\begin{cor}\label{corBidirectedCircuit}
A circuit in a bidirected graph falls in one of the following categories.

\begin{itemize}

\item[(i)] it is a loose edge, or
\item[(ii)] a positive cycle, or
\item[(iii)] a pair of negative cycles with exactly one common node, or
\item[(iv)] a pair of disjoint negative cycles along with a minimal connecting path.

\end{itemize}
\end{cor}

\noindent
The latter two types are called \emph{handcuffs}\index{handcuff}.
For an illustration of Corollary \ref{corBidirectedCircuit}, see Figure \ref{fig:bidirectedCircuit}.

\begin{figure}[h!]
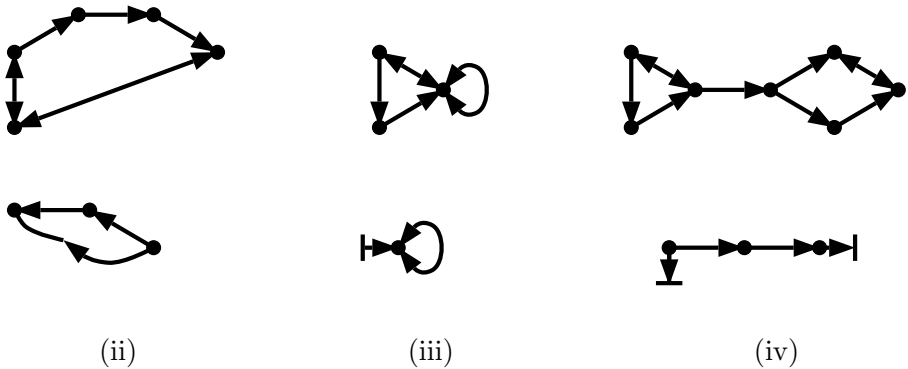


\vspace{.7cm}
$
\begin{array}{ccc}

\psset{xunit=1cm,yunit=1cm,linewidth=0.5pt,radius=0.1mm,arrowsize=7pt,
labelsep=1.5pt,fillcolor=black}

\pspicture(-1,0)(3,1.5)

\pscircle[fillstyle=solid](0,0){.1}
\pscircle[fillstyle=solid](0,1){.1}
\pscircle[fillstyle=solid](0.85,1.5){.1}
\pscircle[fillstyle=solid](1.85,1.5){.1}
\pscircle[fillstyle=solid](2.7,1){.1}

\psline[linewidth=1.6pt,arrowinset=0]{->}(0,1)(0.85,1.5)

\psline[linewidth=1.6pt,arrowinset=0]{<->}(0,0)(0,1)

\psline[linewidth=1.6pt,arrowinset=0]{->}(0.85,1.5)(1.85,1.5)

\psline[linewidth=1.6pt,arrowinset=0]{->}(1.85,1.5)(2.7,1)

\psline[linewidth=1.6pt,arrowinset=0]{<->}(0,0)(2.7,1)


\endpspicture & 

\psset{xunit=1cm,yunit=1cm,linewidth=0.5pt,radius=0.1mm,arrowsize=7pt,
labelsep=1.5pt,fillcolor=black}

\pspicture(-0.5,0)(4,1)

\pscircle[fillstyle=solid](1,0){.1}
\pscircle[fillstyle=solid](1,1){.1}
\pscircle[fillstyle=solid](1.85,0.5){.1}

\psline[linewidth=1.6pt,arrowinset=0]{<->}(1,1)(1.85,0.5)

\psline[linewidth=1.6pt,arrowinset=0]{<-}(1,0)(1,1)

\psline[linewidth=1.6pt,arrowinset=0]{->}(1,0)(1.85,0.5)

\pscurve[linewidth=1.6pt,arrowinset=0]{<->}(1.85,0.5)(2.3,0.8)(2.3,0.2)(1.85,0.5)


\endpspicture

 &

\psset{xunit=1cm,yunit=1cm,linewidth=0.5pt,radius=0.1mm,arrowsize=7pt,
labelsep=1.5pt,fillcolor=black}

\pspicture(0,0)(4,1)

\pscircle[fillstyle=solid](0,0){.1}
\pscircle[fillstyle=solid](0,1){.1}
\pscircle[fillstyle=solid](0.85,0.5){.1}
\pscircle[fillstyle=solid](1.85,0.5){.1}
\pscircle[fillstyle=solid](2.7,0){.1}
\pscircle[fillstyle=solid](2.7,1){.1}
\pscircle[fillstyle=solid](3.55,0.5){.1}

\psline[linewidth=1.6pt,arrowinset=0]{<->}(0,1)(0.85,0.5)

\psline[linewidth=1.6pt,arrowinset=0]{<-}(0,0)(0,1)

\psline[linewidth=1.6pt,arrowinset=0]{->}(0,0)(0.85,0.5)

\psline[linewidth=1.6pt,arrowinset=0]{->}(0.85,0.5)(1.85,0.5)

\psline[linewidth=1.6pt,arrowinset=0]{->}(1.85,0.5)(2.7,1)

\psline[linewidth=1.6pt,arrowinset=0]{->}(1.85,0.5)(2.7,0)

\psline[linewidth=1.6pt,arrowinset=0]{->}(2.7,0)(3.55,0.5)

\psline[linewidth=1.6pt,arrowinset=0]{<->}(2.7,1)(3.55,0.5)


\endpspicture  \\  &  \\ &  \\

\psset{xunit=1cm,yunit=1cm,linewidth=0.5pt,radius=0.1mm,arrowsize=7pt,
labelsep=1.5pt,fillcolor=black}

\pspicture(-1,0)(3,1)

\pscircle[fillstyle=solid](0,1){.1}
\pscircle[fillstyle=solid](1,1){.1}
\pscircle[fillstyle=solid](1.85,0.5){.1}

\psline[linewidth=1.6pt,arrowinset=0]{<-}(0,1)(1,1)

\psline[linewidth=1.6pt,arrowinset=0]{<-}(1,1)(1.85,0.5)

\pscurve[linewidth=1.6pt,arrowinset=0]{-}(0,1)(0.15,0.78)(0.65,0.6)
\pscurve[linewidth=1.6pt,arrowinset=0]{<-}(0.65,0.6)(1.25,0.3)(1.85,0.5)

\put(1,-1){ (ii)
}

\endpspicture

  & 

\psset{xunit=1cm,yunit=1cm,linewidth=0.5pt,radius=0.1mm,arrowsize=7pt,
labelsep=1.5pt,fillcolor=black}

\pspicture(-0.5,0)(3.5,0.5)

\pscircle[fillstyle=solid](1,0.5){.1}

\pscurve[linewidth=1.6pt,arrowinset=0]{<->}(1,0.5)(1.45,0.8)(1.45,0.2)(1,0.5)

\psline[linewidth=1.6pt,arrowinset=0]{|->}(.5,0.5)(1,0.5)

\put(1,-1){ (iii)
}

\endpspicture &

\psset{xunit=1cm,yunit=1cm,linewidth=0.5pt,radius=0.1mm,arrowsize=7pt,
labelsep=1.5pt,fillcolor=black}

\pspicture(0,0)(3,0.5)

\pscircle[fillstyle=solid](0,0.5){.1}
\pscircle[fillstyle=solid](1,0.5){.1}
\pscircle[fillstyle=solid](2,0.5){.1}

\psline[linewidth=1.6pt,arrowinset=0]{->|}(2,0.5)(2.5,0.5)

\psline[linewidth=1.6pt,arrowinset=0]{->}(1,0.5)(2,0.5)

\psline[linewidth=1.6pt,arrowinset=0]{->}(0,0.5)(1,0.5)

\psline[linewidth=1.6pt,arrowinset=0]{->}(0,0.5)(0,0)
\psline[linewidth=1.6pt,arrowinset=0]{-|}(0,0.5)(0,0)

\put(1,-1){ (iv)
}

\endpspicture

\end{array}
$
\vspace{1.2cm}

\caption{Examples of circuits.
}\label{fig:bidirectedCircuit}
\end{figure}

Finally, Using Proposition \ref{propBidirectedmatQ} and Lemma \ref{lemBidirectedNonsing}, it follows that any nonsingular square submatrix of an IMB has a determinant equal to a power of $2$ times $\pm 1$.

\begin{thm}\label{thmBidirected2mod}
Let $In$  be a full row rank node-edge incidence matrix of a bidirected graph. Then, for any nonsingular square submatrix $Q$ of $In$, $det(Q)=\pm 2^r$ for some $r\in \mathbb{N}$.
\end{thm}



\chapter{Binet matrices}\label{ch:Binetdef}

This chapter is devoted to present known fundamental results on binet matrices which will be used frequently in later chapters.
It describes the generalization of network matrices for bidirected graphs. Largely, the results and explanations are taken from \cite{KotThesis,AppaKotBinet} and \cite{Zaslavsky-Bi-06}.

Network matrices can be defined in two equivalent ways. The graphical definition starts with a connected directed graph with a given spanning tree in it. The rows and columns of the network matrix are associated with the tree and non-tree edges, respectively. For any non-tree edge $f$, we find the unique cycle (called the fundamental cycle) which contains $f$ and some edges from the tree. The column of the network matrix corresponding to $f$ will contain $\pm 1$ in the rows of the tree edges in its fundamental cycle and $0$ elsewhere. The signs of the non-zeros depend on the directions of the edges. If walking through the tree along the fundamental cycle starting at the tail of $f$, a tree edge lies in the same direction, it gets a positive sign, if it lies in the opposite direction, it gets a negative sign.

In the algebraic derivation of the network matrices, the incidence matrix $In$ of the directed graph is used. To make it full row rank, an arbitrary row is deleted. Every basis in this full row rank matrix $In'$ corresponds to a spanning tree in the graph. If basis $B$ is associated with the given spanning tree, and we denote the remaining part of $In'$ as $N$, then the network matrix equals $B^{-1} N$.

We apply these methods to bidirected graphs to get the bidirected analogue of network matrices, the \emph{binet matrices}\footnote{the term binet is used as a short form for \emph{bi}direted \emph{net}work, but by coincidence it also matches the name of Jacques Binet (1786-1856) who worked on the foundations of matrix theory and gave the rule of matrix multiplication.}. We define binet matrices in the algebraic way but, in parallel with network matrices, we also provide an algorithm to determine the columns of a binet matrix using its graphical representation. Similarly to network matrices, if the graphical definition is not available, then the analysis of binet matrices would be more cumbersome. That is why we will spend a relatively large part of the chapter on explaining the graphical method of deriving binet matrices. Most of Section \ref{sec:BinetDefi} is devoted to this task.

In Section \ref{sec:BinetOp}, we describe operations on binet matrices.
Section \ref{sec:Binetrep} introduces some particular binet representations. In Section \ref{sec:BinetOpt} we discuss linear and integer programming with an RIMB or a binet matrix as constraint matrix. Section \ref{sec:Binetreg} is about
the class of $2$-regular matrices which contains all binet matrices. Section \ref{sec:BinetMat} deals with matroids. To exhibit the connection of binet matrices to existing special classes of matroids, we introduce signed graphs and matroids in Subsections \ref{sec:BinetMatSign} and \ref{sec:BinetMatSub}, respectively.
Then the signed-graphic matroid based on signed graphs is defined in Subsection \ref{sec:BinetMatGraph}.

\section{Definition and graphical representation}\label{sec:BinetDefi}

We first define binet matrices algebraically and then show that there exists an equivalent graphical definition.

{\bf Definition.}
A matrix $A$ is called a \emph{binet matrix}\index{binet!matrix} if there exist
 a full row rank incidence matrix $In$ of size $n\times m'$ of a bidirected graph $G$ and a basis $B$ of it such that $In=[B\, N]$ and $A=B^{-1} N$. \\

From this definition, since any RIMD is an RIMB, it follows that any network matrix is a binet matrix.

Edges in the subgraph $G(B)$ of $G$ are called \emph{basic}\index{basic!edge} edges. The edges of $G$ that are not in $G(B)$ (i.e., those of $G(N)$) are the \emph{nonbasic} edges. By Lemma \ref{lemBidirectedNonsing}, the graph $G(B)$ consists of negative $1$-tree components. The unique cycle in a basic component is called a \emph{basic}\index{basic!cycle} cycle.  
Basic edges, respectively nonbasic ones, are in one-to-one correspondance with rows of $A$, respectively columns of $A$.

The bidirected graph $G$ with the indication of basic and nonbasic edges is called
a \emph{binet representation}\index{binet!representation} of $A$ (not unique in general) and is denoted $G(A)$. Let $m=m'-n$. Basic edges of the binet representation will be noted $e_1, e_2,\ldots, e_n$
and the nonbasic ones $f_1, f_2,\ldots, f_m$, so that $e_i$ ($1\le i \le n$), respectively $f_j$ ($1\le j \le m$), corresponds to the $i$th row, respectively $j$th column of $A$. Sometimes, we will identify some row  or column of $A$ with its corresponding basic or (respectively) nonbasic edge. See the binet matrix (\ref{equdefiInc}) and Figure \ref{fig1:BinetRep}.

The same binet matrix may arise from different incidence matrices, that is, it may have different binet representations. For example, the two bidirected graphs in Figure \ref{fig1:BinetRep} give two possible binet representations of binet matrix $A$ below, as can be checked by writing out the incidence matrix of the graphs, taking the inverse of the basis $B$ (corresponding to the edges $e_1$, $e_2$ and $e_3$) and multiplying it with $N$ (corresponding to the edges $f_1$, $f_2$ and $f_3$).

\begin{equation}\label{equdefiInc}
A=\begin{tabular}{c|c|c|c|}
& $f_1$ & $f_2$ & $f_3$ \\
\hline
$e_1$ & 1 & 0 & 1 \\
\hline
$e_2$ & 1 & 1 & 0 \\
\hline
$e_3$ & 0 & 1 & 1 \\
\hline
\end{tabular}
\end{equation}

\begin{figure}[ht!]
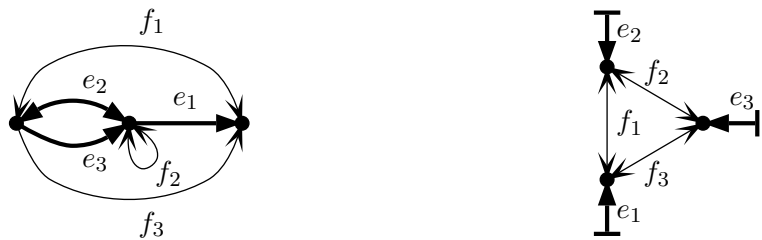


$
\begin{array}{cc}

\psset{xunit=1.5cm,yunit=1.5cm,linewidth=0.5pt,radius=0.1mm,arrowsize=7pt,
labelsep=1.5pt,fillcolor=black}

\pspicture(-1.5,1)(5,2)

\pscircle[fillstyle=solid](0,1.5){.1}
\pscircle[fillstyle=solid](1,1.5){.1}
\pscircle[fillstyle=solid](2,1.5){.1}


\psline[linewidth=1.6pt,arrowinset=0]{->}(1,1.5)(2,1.5)
\rput(1.5,1.7){$e_1$}

\pscurve[linewidth=1.6pt,arrowinset=0]{<->}(0,1.5)(.5,1.7)(1,1.5)
\rput(.7,1.85){$e_2$}

\pscurve[linewidth=1.6pt,arrowinset=0]{->}(0,1.5)(.5,1.3)(1,1.5)
\rput(.7,1.15){$e_3$}

\pscurve[arrowinset=.5,arrowlength=1.5]{<->}(0,1.5)(.25,2)(1.7,2)(2,1.5)
\rput(1.2,2.4){$f_1$}

\pscurve[arrowinset=.5,arrowlength=1.5]{<->}(1,1.5)(1.25,1.2)(1.1,1.1)(1,1.5)
\rput(1.35,1.05){$f_2$}

\pscurve[arrowinset=.5,arrowlength=1.5]{->}(0,1.5)(.3,1)(1.7,1)(2,1.5)
\rput(1.2,.6){$f_3$}

\endpspicture &

\psset{xunit=1.5cm,yunit=1.5cm,linewidth=0.5pt,radius=0.1mm,arrowsize=7pt,
labelsep=1.5pt,fillcolor=black}

\pspicture(0,1)(1.5,3)

\pscircle[fillstyle=solid](0,1){.1}
\pscircle[fillstyle=solid](0,2){.1}
\pscircle[fillstyle=solid](0.85,1.5){.1}


\psline[linewidth=1.6pt,arrowinset=0]{|->}(0,0.5)(0,1)
\rput(.2,0.7){$e_1$}

\psline[linewidth=1.6pt,arrowinset=0]{|->}(0,2.5)(0,2)
\rput(.2,2.3){$e_2$}

\psline[linewidth=1.6pt,arrowinset=0]{<-|}(0.85,1.5)(1.35,1.5)
\rput(1.2,1.7){$e_3$}

\psline[arrowinset=.5,arrowlength=1.5]{<->}(0,1)(0,2)
\rput(0.2,1.5){$f_1$}

\psline[arrowinset=.5,arrowlength=1.5]{<->}(0,2)(0.85,1.5)
\rput(0.44,1.95){$f_2$}

\psline[arrowinset=.5,arrowlength=1.5]{<->}(0,1)(0.85,1.5)
\rput(0.44,1.05){$f_3$}

\endpspicture 

\end{array}
$

\vspace{1cm}

\caption{Different binet representations of binet matrix (\ref{equdefiInc}).} 
\label{fig1:BinetRep}
\end{figure}

Let $1\le j \le m$. Since $A_{\bullet j}=B^{-1} N_{\bullet j }
\Leftrightarrow B A_{\bullet j}=N_{\bullet j}$, the column $A_{\bullet j}$
represents the unique coordinates of vector $N_{\bullet j}$ in the basis given by $B$.
Let $\{B_{\bullet k_1},B_{\bullet k_2},\ldots,B_{\bullet k_t}\}$ be the subset of columns of $B$
where the coordinates are non-zero and  $B'=[N_{\bullet j},B_{\bullet k_1},\ldots,B_{\bullet k_t}]$. 
The columns of the matrix $B'$ form a minimal dependent set in $\mathbf{R}^n$. The subgraph $G(B')$ of $G$ is called the 
\emph{fundamental}\index{fundamental circuit} circuit of $f_j$. 
A \emph{basic}\index{basic!subgraph} subgraph of a binet representation G(A) is a subgraph consisting of only basic edges. The \emph{basic}\index{basic!fundamental circuit} fundamental circuit of $f_j$ is the fundamental circuit of $f_j$ without $f_j$. If the fundamental circuit of $f_j$ is a handcuff of type (iii) (see Corollary \ref{corBidirectedCircuit}), then it is called a \emph{pathcycle}\index{pathcycle}.

By Corollary \ref{corBidirectedCircuit}, we can
describe the different types of fundamental circuits according to $f_j$. Later (see Lemma \ref{lemBinetRepBid}), we will prove that a binet matrix has always a binet representation in which each component of $G(B)$ has exactly one bidirected edge (contained in the basic cycle). So, in what follows, we assume that each component of $G(B)$ has exactly one bidirected edge contained in the basic cycle.
Let $v,v'$ be the end-nodes of $f_j$ ($v=v'$ if $f_j$ is a half-edge or loop). The vertex
$v$ is contained in some maximal negative 1-tree, call it $ T$. 
Let $p$ be the path in 
$ T$ from $v$ to the first node, say $w$, in the basic cycle
($v=w$ if $v$ lies on the basic cycle). In the same way, define 
$ T'$, $p'$, $w'$  and $q'$ with respect to $v'$. Each of the 
paths $p$ and $p'$ is called a \emph{ stem}\index{stem} issued from $f_j$. 
A subpath of a stem is called a \emph{substem}\index{substem}. If $T \neq T'$, then $f_j$ is called a \emph{$2$-edge}\index{edge@$2$-edge}, otherwise a \emph{$1$-edge}\index{edge@$1$-edge}.

\begin{figure}[h!]
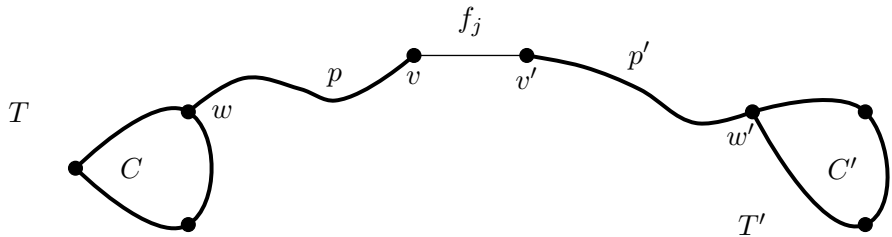

\vspace{.5cm}
\begin{center}

\psset{xunit=1.5cm,yunit=1.5cm,linewidth=0.5pt,radius=0.1mm,arrowsize=7pt,
labelsep=1.5pt,fillcolor=black}

\pspicture(0,.3)(8.5,2)

\pscircle[fillstyle=solid](1,1){.1}
\pscircle[fillstyle=solid](2,0.5){.1}
\pscircle[fillstyle=solid](2,1.5){.1}
\pscircle[fillstyle=solid](4,2){.1}
\pscircle[fillstyle=solid](5,2){.1}
\pscircle[fillstyle=solid](7,1.5){.1}
\pscircle[fillstyle=solid](8,1.5){.1}
\pscircle[fillstyle=solid](8,0.5){.1}





\psline(4,2)(5,2)
\rput(4.5,2.3){$f_j$}

\pscurve[linewidth=1.5pt,arrowinset=0]{-}(1,1)(2,1.5)(2,0.5)(1,1)

\pscurve[linewidth=1.5pt,arrowinset=0]{-}(5,2)(5.5,1.9)(6,1.7)(6.5,1.4)(7,1.5)
\rput(6,2){$p'$}

\pscurve[linewidth=1.5pt,arrowinset=0]{-}(7,1.5)(8,1.5)(8,0.5)(7,1.5)

\pscurve[linewidth=1.5pt,arrowinset=0]{-}(2,1.5)(2.5,1.8)(3,1.7)(3.3,1.6)(4,2)
\rput(3.3,1.8){$p$}

\rput(0.5,1.5){$T$}
\rput(7,0.5){$T'$}
\rput(4,1.8){$v$}
\rput(5,1.8){$v'$}
\rput(2.3,1.5){$w$}
\rput(6.9,1.3){$w'$}
\rput(1.5,1){$C$}
\rput(7.8,1){$C'$}

\endpspicture 
\end{center}

\caption{An illustration of the fundamental circuit of $f_j$ in case $T\neq T'$.} 
\label{figalpha1:2basicComponents}

\end{figure}

Assume $ T\neq  T'$. \label{mycounter2} Let $C$ and $C'$ be the (negative) cycles in $T$ and $T'$, respectively. The fundamental circuit of $f_j$ is the bidirected graph formed by 
$f_j$, $p$, $C$, $p'$ and $C'$. See Figure
\ref{figalpha1:2basicComponents}.

Now assume $ T= T'$. 
Let $C_1$ (respectively, $C_2$) be the path in the basic cycle from $w$ to $w'$ containing the 
bidirected edge (respectively, only directed edges).
If $f_j$ is bidirected, then the
fundamental circuit of $f_j$ is the subgraph of $G$ consisting of $f_j$, $p$, $C_1$ and $p'$. 
If $f_j$ is directed and $w\neq w'$ then the
fundamental circuit of $f_j$ is the subgraph of $G$ consisting of 
$f_j$, $p$, $C_2$ and $p'$.
If $f_j$ is directed and $w= w'$, then the
fundamental circuit of $f_j$ is the bidirected graph whose edge set contains  $f_j$, the edges of $p$ not in $p'$ and the edges of $p'$ not in
$p$. See Figure \ref{figalpha2:fundCircuits} for an illustration.

\begin{figure}[h!]
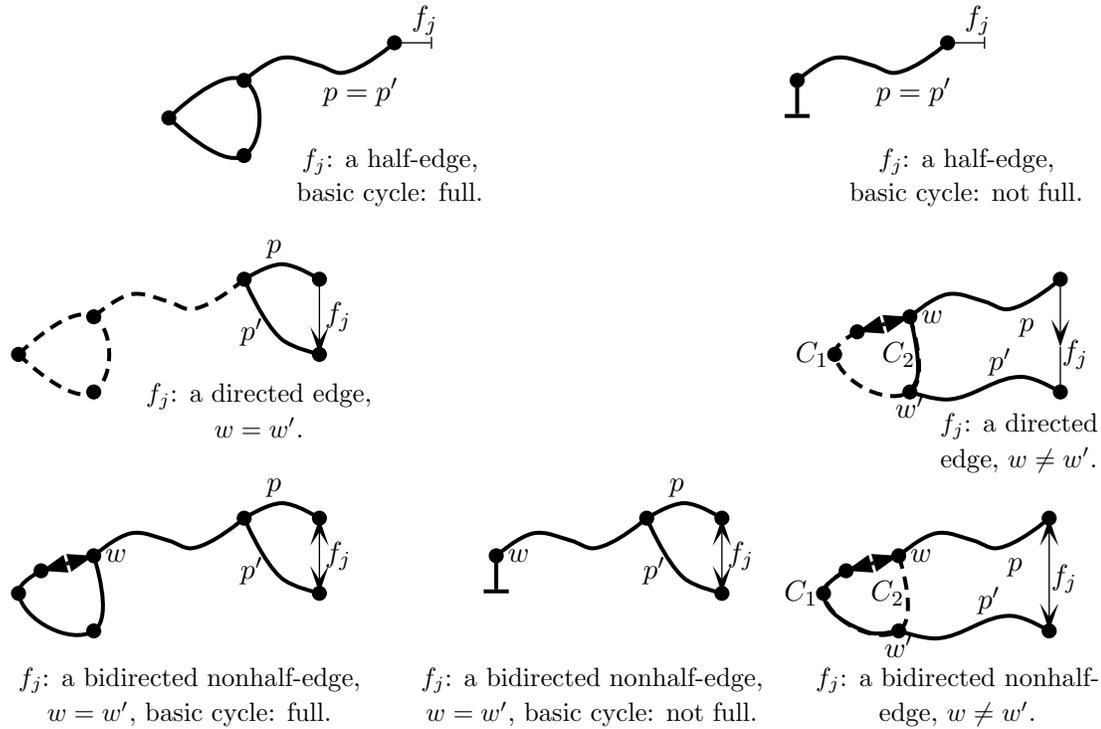


$\begin{array}{cc}

\psset{xunit=1cm,yunit=1cm,linewidth=0.5pt,radius=0.1mm,arrowsize=7pt,
labelsep=1.5pt,fillcolor=black}

\pspicture(0,0)(5,3)

\pscircle[fillstyle=solid](1,1){.1}
\pscircle[fillstyle=solid](2,0.5){.1}
\pscircle[fillstyle=solid](2,1.5){.1}
\pscircle[fillstyle=solid](4,2){.1}

\psline[arrowinset=.5,arrowlength=1.5]{-|}(4,2)(4.5,2)
\rput(4.4,2.3){$f_j$}

\pscurve[linewidth=1.5pt,arrowinset=0]{-}(1,1)(2,1.5)(2,0.5)(1,1)

\pscurve[linewidth=1.5pt,arrowinset=0]{-}(2,1.5)(2.5,1.8)(3,1.7)(3.3,1.6)(4,2)

\rput(3.55,1.35){$p=p'$}


\put(2.7,-0.1){\shortstack{\small{$f_j$: a half-edge,}\\ \small{basic cycle: full.}}}

\endpspicture  &

\psset{xunit=1cm,yunit=1cm,linewidth=0.5pt,radius=0.1mm,arrowsize=7pt,
labelsep=1.5pt,fillcolor=black}

\pspicture(0,0)(8,3)

\pscircle[fillstyle=solid](2,1.5){.1}
\pscircle[fillstyle=solid](4,2){.1}

\psline[arrowinset=.5,arrowlength=1.5]{-|}(4,2)(4.5,2)
\rput(4.4,2.3){$f_j$}

\psline[linewidth=1.5pt,arrowinset=0]{-|}(2,1.5)(2,1)

\pscurve[linewidth=1.5pt,arrowinset=0]{-}(2,1.5)(2.5,1.8)(3,1.7)(3.3,1.6)(4,2)


\rput(3.55,1.35){$p=p'$}

\put(2.7,-0.1){\shortstack{\small{$f_j$: a half-edge,}\\ \small{basic cycle: not full.}}}

\endpspicture \\

\psset{xunit=1cm,yunit=1cm,linewidth=0.5pt,radius=0.1mm,arrowsize=7pt,
labelsep=1.5pt,fillcolor=black}

\pspicture(0,0)(9,3)

\pscircle[fillstyle=solid](1,1){.1}
\pscircle[fillstyle=solid](2,0.5){.1}
\pscircle[fillstyle=solid](2,1.5){.1}
\pscircle[fillstyle=solid](4,2){.1}
\pscircle[fillstyle=solid](5,2){.1}
\pscircle[fillstyle=solid](5,1){.1}

\pscurve[linewidth=1.5pt,arrowinset=0,linestyle=dashed]{-}(1,1)(2,1.5)(2,0.5)(1,1)

\pscurve[linewidth=1.5pt,arrowinset=0]{-}(4,2)(4.5,2.2)(5,2)

\pscurve[linewidth=1.5pt,arrowinset=0]{-}(4,2)(4.5,1.2)(5,1)

\pscurve[linewidth=1.5pt,arrowinset=0,linestyle=dashed]{-}(2,1.5)(2.5,1.8)(3,1.7)(3.3,1.6)(4,2)

\psline[arrowinset=.5,arrowlength=1.5]{->}(5,2)(5,1)
\rput(5.25,1.5){$f_j$}

\rput(4.4,2.4){$p$}\rput(4.1,1.3){$p'$}

\put(2.7,-0.1){\shortstack{\small{$f_j$: a directed edge,}\\ \small{$w=w'$.}}}

\endpspicture  &

\psset{xunit=1cm,yunit=1cm,linewidth=0.5pt,radius=0.1mm,arrowsize=7pt,
labelsep=1.5pt,fillcolor=black}

\pspicture(0,0)(5,3)

\pscircle[fillstyle=solid](1,1){.1}
\pscircle[fillstyle=solid](1.3,1.3){.1}
\pscircle[fillstyle=solid](2,0.5){.1}
\pscircle[fillstyle=solid](2,1.5){.1}
\pscircle[fillstyle=solid](4,2){.1}
\pscircle[fillstyle=solid](4,0.5){.1}

\pscurve[linewidth=1.5pt,arrowinset=0,linestyle=dashed]{-}(2,1.5)(2,0.5)(1,1)(1.3,1.3)

\pscurve[linewidth=1.5pt,arrowinset=0]{-}(2,1.5)(2.1,.7)(2,0.5)

\pscurve[linewidth=1.5pt,arrowinset=0]{-}(2,0.5)(2.5,0.4)(3.5,0.7)(4,0.5)

\psline[linewidth=1.5pt,arrowinset=0]{<->}(1.3,1.3)(2,1.5)

\psline[arrowinset=.5,arrowlength=1.5]{->}(4,2)(4,1.1)\psline(4,1.1)(4,0.5)
\rput(4.2,1){$f_j$}

\pscurve[linewidth=1.5pt,arrowinset=0]{-}(2,1.5)(2.5,1.8)(3,1.7)(3.3,1.6)(4,2)

\rput(3.2,0.9){$p'$}
\rput(3.55,1.35){$p$}
\rput(0.7,1){$C_1$}
\rput(1.85,1){$C_2$}
\rput(2.3,1.5){$w$}
\rput(2,.3){$w'$}

\put(2.4,-0.5){\shortstack{\small{$f_j$: a directed }\\ \small{edge, $w\neq w'$.}}}

\endpspicture 

\end{array}$

$\begin{array}{ccc}

\psset{xunit=1cm,yunit=1cm,linewidth=0.5pt,radius=0.1mm,arrowsize=7pt,
labelsep=1.5pt,fillcolor=black}

\pspicture(0,0)(5,3)

\pscircle[fillstyle=solid](1,1){.1}
\pscircle[fillstyle=solid](1.3,1.3){.1}
\pscircle[fillstyle=solid](2,0.5){.1}
\pscircle[fillstyle=solid](2,1.5){.1}
\pscircle[fillstyle=solid](4,2){.1}
\pscircle[fillstyle=solid](5,2){.1}
\pscircle[fillstyle=solid](5,1){.1}

\pscurve[linewidth=1.5pt,arrowinset=0]{-}(2,1.5)(2,0.5)(1,1)(1.3,1.3)

\psline[linewidth=1.5pt,arrowinset=0]{<->}(1.3,1.3)(2,1.5)

\pscurve[linewidth=1.5pt,arrowinset=0]{-}(4,2)(4.5,2.2)(5,2)

\pscurve[linewidth=1.5pt,arrowinset=0]{-}(4,2)(4.5,1.2)(5,1)

\pscurve[linewidth=1.5pt,arrowinset=0]{-}(2,1.5)(2.5,1.8)(3,1.7)(3.3,1.6)(4,2)

\psline[arrowinset=.5,arrowlength=1.5]{<->}(5,2)(5,1)
\rput(5.25,1.5){$f_j$}


\rput(4.4,2.4){$p$}\rput(4.1,1.3){$p'$}
\rput(2.3,1.5){$w$}

\put(1,-0.7){\shortstack{\small{$f_j$: a bidirected nonhalf-edge,}\\ 
\small{$w=w'$, basic cycle: full.}}}

\endpspicture &

\psset{xunit=1cm,yunit=1cm,linewidth=0.5pt,radius=0.1mm,arrowsize=7pt,
labelsep=1.5pt,fillcolor=black}

\pspicture(0,0)(5,3)

\pscircle[fillstyle=solid](2,1.5){.1}
\pscircle[fillstyle=solid](4,2){.1}
\pscircle[fillstyle=solid](5,2){.1}
\pscircle[fillstyle=solid](5,1){.1}

\psline[linewidth=1.5pt,arrowinset=0]{-|}(2,1.5)(2,1)

\pscurve[linewidth=1.5pt,arrowinset=0]{-}(4,2)(4.5,2.2)(5,2)

\pscurve[linewidth=1.5pt,arrowinset=0]{-}(4,2)(4.5,1.2)(5,1)

\pscurve[linewidth=1.5pt,arrowinset=0]{-}(2,1.5)(2.5,1.8)(3,1.7)(3.3,1.6)(4,2)

\psline[arrowinset=.5,arrowlength=1.5]{<->}(5,2)(5,1)
\rput(5.25,1.5){$f_j$}


\rput(4.4,2.4){$p$}\rput(4.1,1.3){$p'$}
\rput(2.3,1.5){$w$}

\put(1,-0.7){\shortstack{\small{$f_j$: a bidirected nonhalf-edge,}\\ 
\small{$w=w'$, basic cycle: not full.}}}

\endpspicture &

\psset{xunit=1cm,yunit=1cm,linewidth=0.5pt,radius=0.1mm,arrowsize=7pt,
labelsep=1.5pt,fillcolor=black}

\pspicture(0,0)(5,3)

\pscircle[fillstyle=solid](1,1){.1}
\pscircle[fillstyle=solid](1.3,1.3){.1}
\pscircle[fillstyle=solid](2,0.5){.1}
\pscircle[fillstyle=solid](2,1.5){.1}
\pscircle[fillstyle=solid](4,2){.1}
\pscircle[fillstyle=solid](4,0.5){.1}

\pscurve[linewidth=1.5pt,arrowinset=0,linestyle=dashed]{-}(2,1.5)(2,0.5)(1,1)(1.3,1.3)

\pscurve[linewidth=1.5pt,arrowinset=0]{-}(2,0.5)(1.6,0.48)(1,1)(1.3,1.3)

\pscurve[linewidth=1.5pt,arrowinset=0]{-}(2,0.5)(2.5,0.4)(3.5,0.7)(4,0.5)

\psline[linewidth=1.5pt,arrowinset=0]{<->}(1.3,1.3)(2,1.5)

\psline[arrowinset=.5,arrowlength=1.5]{<->}(4,2)(4,0.5)
\rput(4.2,1.2){$f_j$}

\pscurve[linewidth=1.5pt,arrowinset=0]{-}(2,1.5)(2.5,1.8)(3,1.7)(3.3,1.6)(4,2)

\rput(3.2,0.9){$p'$}
\rput(3.55,1.35){$p$}
\rput(0.7,1){$C_1$}
\rput(1.85,1){$C_2$}
\rput(2.3,1.5){$w$}
\rput(2,.3){$w'$}

\put(.9,-0.7){\shortstack{\small{$f_j$: a bidirected nonhalf-}\\ 
\small{edge, $w\neq w'$.}}}

\endpspicture 

\end{array}$
\vspace{1cm}

\caption{An illustration of the different types of fundamental circuit of a nonbasic edge $f_j$ in case $T= T'$. (The fundamental circuit of $f_j$ is formed by $f_j$ and heavy edges.)} 
\label{figalpha2:fundCircuits}

\end{figure}

A \emph{minimal covering closed walk}\index{minimal covering walk} of a circuit $C$ in $G$ is a closed walk of minimal length that covers each edge and is denoted by $w(C)$. Using Corollary \ref{corBidirectedCircuit}, we deduce that the walk $w(C)$ covers each edge of a connecting
path twice (provided that $C$ is a handcuff) and each other edge exactly once.  For any edge $f$ in $C$, the walk $w(C)$ is said to be \emph{consistently oriented according to $f$}\index{consistently oriented!according to a nonbasic edge} if every node in $w(C)$, except the endnodes of $f$, is consistent. It is easy to verify that for any edge $f$ in $C$, one may orient the edges of $C\verb"\" \{f\}$ so that $w(C)$ becomes consistently oriented according to $f$. (One may consult \cite{Zaslavsky-Or-91} for a detailed discussion of how to orient a signed graph. We note that in articles by Zaslavsky the word used for "consistent" is "coherent", while in \cite{Appa-06}, the corresponding term is "incoherent".)

In what follows, we give a graphical method adapted from \cite{Zaslavsky-Bi-06} to obtain a binet matrix from its binet representation . This spares us the need to compute the inverse of the basis, makes handling of binet matrices much easier, and helps in proving properties of binet matrices in later sections.

\begin{tabbing}
\textbf{Procedure\,\,WeightCircuit}\\

\textbf{Input: }\= A bidirected graph $G$, a basis $B$ of the incidence matrix $In(G)$,\\
\> and a non-zero column $N_{\bullet j}$ of $In$ not in $B$.\\
\textbf{Output:} The column $A_{\bullet j}=B^{-1} N_{\bullet j}$.\\ 

1)\verb"  "\= let $T$ be the edge set corresponding to the basis $B$ and $C$ the fundamental circuit of $f_j$\\
\>  in $G$ and $w(C)$ a minimal covering closed walk of $C$;\\
2) \> reorient edges of $C\verb"\" \{f_j\} $ so that $w(C)$ becomes consistently oriented according to $f_j$;\\
3) \> assign weights $+1$ to each singly covered edge (except half-edges), \\
\> $+2$ to each doubly covered edge or half-edge, and $0$ to the other edges in $T$; \\
4) \> negate the values assigned to edges that were reoriented; divide by $2$ if necessary \\
\>  to ensure that $f_j$ has weight $+1$;\\
\> output the edge weights on $T$ that correspond to the different coefficients of $A_{\bullet j}$.
\end{tabbing}

From this procedure, it is easy to deduce the following two lemmas.

\begin{lem}\label{lemdefiWeight1}
For any nonbasic edge $f_j$ and $e_i$ a basic edge in the fundamental circuit of $f_j$, we have 

$
A_{ij}= \left\{
\begin{array}{ll}

 \pm \frac{1}{2} & \m{ if }  f_j  \m{ is a 2-edge or a half-edge, and } e_i  \m{ is in a full basic cycle}; \\
  \pm 2  & \m{ if } f_j  \m{  is a bidirected nonhalf 1-edge, and } e_i \m{  is a half-edge or belongs to } \\   
 & \m{ both stems issued from } f_j; \\
 \pm 1 & otherwise.

\end{array} 
\right.
$


\end{lem}

\begin{lem}\label{lemdefiWeight2}
Let $f_j$ be a nonbasic edge, $C$ its fundamental circuit and $w(C)$ a minimal covering closed walk of $C$. Then $w(C)$ is 
consistently oriented according to $f_j$ if and only if $A_{\bullet j}$ is nonnegative.
\end{lem}

From these lemmas, we can describe the different types of
fundamental circuits of a nonbasic edge $f_j$ with weights on the edges. By assuming that each component of $G(B)$ has exactly one bidirected edge (which is entering), an illustration is given in 
Figure \ref{fig:Apositive} (where $p$, $p'$, $C$, $C'$, $C_1$ and $C_2$ are defined at page \pageref{mycounter2}).

\begin{figure}[h!]

\begin{center}

\psset{xunit=1cm,yunit=1cm,linewidth=0.5pt,radius=0.1mm,arrowsize=7pt,
labelsep=1.5pt,fillcolor=black}

\pspicture(0,0)(9,3)

\pscircle[fillstyle=solid](1,1){.1}
\pscircle[fillstyle=solid](2,0.5){.1}
\pscircle[fillstyle=solid](2,1.5){.1}
\pscircle[fillstyle=solid](1.55,1.6){.1}
\pscircle[fillstyle=solid](4,2){.1}
\pscircle[fillstyle=solid](5,2){.1}
\pscircle[fillstyle=solid](7,1.5){.1}

\psline[arrowinset=.5,arrowlength=1.5]{<->}(4,2)(5,2)
\rput(4.5,2.3){$f_j$}

\psline[linewidth=1.5pt,arrowinset=0]{->}(1,1)(1.4,0.6)

\psline[linewidth=1.5pt,arrowinset=0]{->}(1,1)(1.4,0.6)(2,0.5)

\psline[linewidth=1.5pt,arrowinset=0]{->}(2,0.5)(2.2,0.8)(2,1.5)

\psline[linewidth=1.5pt,arrowinset=0]{<->}(1,1)(1.55,1.6)

\psline[linewidth=1.5pt,arrowinset=0]{->}(1.55,1.6)(2,1.5)

\pscurve[linewidth=1.5pt,arrowinset=0]{-}(5,2)(5.5,1.9)(6,1.7)
\pscurve[linewidth=1.5pt,arrowinset=0]{<-}(6,1.7)(6.5,1.4)(7,1.5)
\rput(6,2){$p'$}

\psline[linewidth=1.5pt,arrowinset=0]{<-|}(7,1.5)(7.5,1.5)

\pscurve[linewidth=1.5pt,arrowinset=0]{->}(2,1.5)(2.5,1.8)(3,1.7)
\pscurve[linewidth=1.5pt,arrowinset=0]{-}(3,1.7)(3.3,1.6)(4,2)
\rput(3.3,1.8){$p$}

\rput(4,1.8){$v$}
\rput(5,1.8){$v'$}
\rput(2.3,1.5){$w$}
\rput(6.9,1.25){$w'$}
\rput(1.35,1.8){$\frac{1}{2}$}\rput(1.4,0.22){$ \frac{1}{2}$}
\rput(2.5,0.8){$ \frac{1}{2}$} \rput(2.5,2){$ 1$}
\rput(3.5,1.3){$ 1$}\rput(5.5,2.2){$ 1$}

\rput(6.2,1.2){$ 1$}\rput(7.5,1.1){$ 1$}
\put(-3.2,2){\shortstack{The case $T\neq T'$:}}

\endpspicture 
\end{center}

$\begin{array}{cc}

\psset{xunit=1cm,yunit=1cm,linewidth=0.5pt,radius=0.1mm,arrowsize=7pt,
labelsep=1.5pt,fillcolor=black}

\pspicture(0,0)(5,3)

\pscircle[fillstyle=solid](1,1){.1}
\pscircle[fillstyle=solid](1.55,1.6){.1}
\pscircle[fillstyle=solid](2,0.5){.1}
\pscircle[fillstyle=solid](2,1.5){.1}
\pscircle[fillstyle=solid](4,2){.1}

\psline[arrowinset=.5,arrowlength=1.5]{<-|}(4,2)(4.5,2)
\rput(4.4,2.3){$f_j$}

\psline[linewidth=1.5pt,arrowinset=0]{->}(1,1)(1.4,0.6)

\psline[linewidth=1.5pt,arrowinset=0]{->}(1,1)(1.4,0.6)(2,0.5)

\psline[linewidth=1.5pt,arrowinset=0]{->}(2,0.5)(2.2,0.8)(2,1.5)

\psline[linewidth=1.5pt,arrowinset=0]{<->}(1,1)(1.55,1.6)

\psline[linewidth=1.5pt,arrowinset=0]{->}(1.55,1.6)(2,1.5)
\pscurve[linewidth=1.5pt,arrowinset=0]{->}(2,1.5)(2.5,1.8)(3,1.7)
\pscurve[linewidth=1.5pt,arrowinset=0]{-}(3,1.7)(3.3,1.6)(4,2)
\rput(3.4,1.9){$ 1$}

\rput(1.35,1.8){$ \frac{1}{2}$}\rput(1.4,0.22){$ \frac{1}{2}$}
\rput(2.5,0.8){$ \frac{1}{2}$} \rput(2.5,2){$ 1$}
\rput(3.55,1.35){$p=p'$}

\put(-2,3){\shortstack{The case $T=T'$:}}

\put(2.7,-0.1){\shortstack{\small{$f_j$: a half-edge,}\\ \small{basic cycle: full.}}}

\endpspicture  &

\psset{xunit=1cm,yunit=1cm,linewidth=0.5pt,radius=0.1mm,arrowsize=7pt,
labelsep=1.5pt,fillcolor=black}

\pspicture(0,0)(8,3)

\pscircle[fillstyle=solid](2,1.5){.1}
\pscircle[fillstyle=solid](4,2){.1}

\psline[arrowinset=.5,arrowlength=1.5]{<-|}(4,2)(4.5,2)
\rput(4.4,2.3){$f_j$}

\psline[linewidth=1.5pt,arrowinset=0]{<-|}(2,1.5)(2,1)

\pscurve[linewidth=1.5pt,arrowinset=0]{->}(2,1.5)(2.5,1.8)(3,1.7)
\pscurve[linewidth=1.5pt,arrowinset=0]{-}(3,1.7)(3.3,1.6)(4,2)

\rput(3.4,1.9){$ 1$}
\rput(1.6,1){$ 1$}
\rput(2.5,2){$ 1$}

\rput(3.55,1.35){$p=p'$}

\put(2.7,-0.1){\shortstack{\small{$f_j$: a half-edge,}\\ \small{basic cycle: not full.}}}

\endpspicture \\

\psset{xunit=1cm,yunit=1cm,linewidth=0.5pt,radius=0.1mm,arrowsize=7pt,
labelsep=1.5pt,fillcolor=black}

\pspicture(0,0)(9,3)

\pscircle[fillstyle=solid](1,1){.1}
\pscircle[fillstyle=solid](2,0.5){.1}
\pscircle[fillstyle=solid](2,1.5){.1}
\pscircle[fillstyle=solid](4,2){.1}
\pscircle[fillstyle=solid](5,2){.1}
\pscircle[fillstyle=solid](5,1){.1}

\pscurve[linewidth=1.5pt,arrowinset=0,linestyle=dashed]{-}(1,1)(2,1.5)(2,0.5)(1,1)

\pscurve[linewidth=1.5pt,arrowinset=0]{-}(4,2)(4.5,2.2)(5,2)
\psline[linewidth=1.5pt,arrowinset=0]{<-}(4.499,2.2)(4.501,2.1995)

\pscurve[linewidth=1.5pt,arrowinset=0]{->}(4,2)(4.5,1.2)(5,1)

\pscurve[linewidth=1.5pt,arrowinset=0,linestyle=dashed]{-}(2,1.5)(2.5,1.8)(3,1.7)(3.3,1.6)(4,2)
\rput(3.4,1.9){$0$}

\psline[arrowinset=.5,arrowlength=1.5]{->}(5,2)(5,1)
\rput(5.25,1.5){$f_j$}

\rput(1.35,1.7){$0$}\rput(1.4,0.22){$0$}
\rput(2.5,0.8){$0$} \rput(2.5,2){$0$}
\rput(4.5,1.9){$ 1$}\rput(4.5,0.9){$ 1$}
\rput(4.4,2.4){$p$}\rput(4.1,1.3){$p'$}

\put(2.7,-0.1){\shortstack{\small{$f_j$: a directed edge,}\\ \small{$w=w'$.}}}

\endpspicture  &

\psset{xunit=1cm,yunit=1cm,linewidth=0.5pt,radius=0.1mm,arrowsize=7pt,
labelsep=1.5pt,fillcolor=black}

\pspicture(0,0)(5,3)

\pscircle[fillstyle=solid](1,1){.1}
\pscircle[fillstyle=solid](1.3,1.3){.1}
\pscircle[fillstyle=solid](2,0.5){.1}
\pscircle[fillstyle=solid](2,1.5){.1}
\pscircle[fillstyle=solid](4,2){.1}
\pscircle[fillstyle=solid](4,0.5){.1}

\pscurve[linewidth=1.5pt,arrowinset=0,linestyle=dashed]{-}(2,1.5)(2,0.5)(1,1)(1.3,1.3)

\pscurve[linewidth=1.5pt,arrowinset=0]{-}(2,1.5)(2.1,.7)(2,0.5)
\psline[linewidth=1.5pt,arrowinset=0]{->}(2,1.5)(2.14,0.8)

\pscurve[linewidth=1.5pt,arrowinset=0]{->}(2,0.5)(2.5,0.4)(3,0.6)
\pscurve[linewidth=1.5pt,arrowinset=0]{-}(3,0.6)(3.5,0.7)(4,0.5)

\psline[linewidth=1.5pt,arrowinset=0]{<->}(1.3,1.3)(2,1.5)

\psline[arrowinset=.5,arrowlength=1.5]{->}(4,2)(4,1.1)\psline(4,1.1)(4,0.5)
\rput(4.2,1){$f_j$}

\pscurve[linewidth=1.5pt,arrowinset=0]{-<}(2,1.5)(2.5,1.8)(3,1.7)
\pscurve[linewidth=1.5pt,arrowinset=0]{-}(2.9,1.74)(3.3,1.6)(4,2)
\rput(3.4,1.9){$ 1$}

\rput(1.35,1.7){$0$}\rput(1.4,0.22){$0$}
\rput(2.5,0.9){$ 1$} \rput(2.5,2){$ 1$}\rput(3.5,0.45){$ 1$}
\rput(3.2,0.9){$p'$}
\rput(3.55,1.35){$p$}
\rput(0.7,1){$C_1$}
\rput(1.83,1){$C_2$}
\rput(2.3,1.5){$w$}
\rput(2,.3){$w'$}

\put(2.4,-0.5){\shortstack{\small{$f_j$: a directed }\\ \small{edge, $w\neq w'$.}}}

\endpspicture 

\end{array}$

$\begin{array}{ccc}

\psset{xunit=1cm,yunit=1cm,linewidth=0.5pt,radius=0.1mm,arrowsize=7pt,
labelsep=1.5pt,fillcolor=black}

\pspicture(0,0)(5,3)

\pscircle[fillstyle=solid](0.8,1){.1}
\pscircle[fillstyle=solid](1.3,1.3){.1}
\pscircle[fillstyle=solid](2,0.5){.1}
\pscircle[fillstyle=solid](2,1.5){.1}
\pscircle[fillstyle=solid](4,2){.1}
\pscircle[fillstyle=solid](5,2){.1}
\pscircle[fillstyle=solid](5,1){.1}

\pscurve[linewidth=1.5pt,arrowinset=0]{-}(2,1.5)(2,0.5)(1.5,0.52)(0.8,1)

\psline[linewidth=1.5pt,arrowinset=0]{->}(0.8,1)(1.5,0.52)

\psline[linewidth=1.5pt,arrowinset=0]{-<}(2,1.5)(2.14,0.8)

\psline[linewidth=1.5pt,arrowinset=0]{<->}(1.3,1.3)(2,1.5)

\psline[linewidth=1.5pt,arrowinset=0]{->}(1.3,1.3)(.8,1)

\pscurve[linewidth=1.5pt,arrowinset=0]{-}(4,2)(4.5,2.2)(5,2)
\psline[linewidth=1.5pt,arrowinset=0]{->}(4.499,2.1995)(4.501,2.2)

\pscurve[linewidth=1.5pt,arrowinset=0]{->}(4,2)(4.5,1.2)(5,1)

\pscurve[linewidth=1.5pt,arrowinset=0]{->}(2,1.5)(2.5,1.8)(3,1.7)
\pscurve[linewidth=1.5pt,arrowinset=0]{-}(3,1.7)(3.3,1.6)(4,2)

\psline[arrowinset=.5,arrowlength=1.5]{<->}(5,2)(5,1)
\rput(5.25,1.5){$f_j$}

\rput(1.35,1.7){$ 1$}\rput(1.4,0.22){$ 1$}
\rput(2.5,0.9){$ 1$}

\rput(2.5,2){$ 2$}\rput(3.4,1.9){$ 2$}
\rput(4.5,1.9){$ 1$}\rput(4.5,0.9){$ 1$}
\rput(4.4,2.4){$p$}\rput(4.1,1.3){$p'$}
\rput(2.3,1.5){$w$}

\put(1,-0.7){\shortstack{\small{$f_j$: a bidirected nonhalf-edge,}\\ 
\small{$w=w'$, basic cycle: full.}}}

\endpspicture &

\psset{xunit=1cm,yunit=1cm,linewidth=0.5pt,radius=0.1mm,arrowsize=7pt,
labelsep=1.5pt,fillcolor=black}

\pspicture(0,0)(5,3)

\pscircle[fillstyle=solid](2,1.5){.1}
\pscircle[fillstyle=solid](4,2){.1}
\pscircle[fillstyle=solid](5,2){.1}
\pscircle[fillstyle=solid](5,1){.1}

\psline[linewidth=1.5pt,arrowinset=0]{<-|}(2,1.5)(2,1)

\pscurve[linewidth=1.5pt,arrowinset=0]{-}(4,2)(4.5,2.2)(5,2)
\psline[linewidth=1.5pt,arrowinset=0]{->}(4.499,2.1995)(4.501,2.2)

\pscurve[linewidth=1.5pt,arrowinset=0]{->}(4,2)(4.5,1.2)(5,1)

\pscurve[linewidth=1.5pt,arrowinset=0]{->}(2,1.5)(2.5,1.8)(3,1.7)
\pscurve[linewidth=1.5pt,arrowinset=0]{-}(3,1.7)(3.3,1.6)(4,2)

\psline[arrowinset=.5,arrowlength=1.5]{<->}(5,2)(5,1)
\rput(5.25,1.5){$f_j$}

\rput(1.6,1){$ 2$}

\rput(2.5,2){$ 2$}\rput(3.4,1.9){$ 2$}
\rput(4.5,1.9){$ 1$}\rput(4.5,0.9){$ 1$}
\rput(4.4,2.4){$p$}\rput(4.1,1.3){$p'$}
\rput(2.3,1.5){$w$}

\put(1,-0.7){\shortstack{\small{$f_j$: a bidirected nonhalf-edge,}\\ 
\small{$w=w'$, basic cycle: not full.}}}

\endpspicture &

\psset{xunit=1cm,yunit=1cm,linewidth=0.5pt,radius=0.1mm,arrowsize=7pt,
labelsep=1.5pt,fillcolor=black}

\pspicture(0,0)(5,3)

\pscircle[fillstyle=solid](0.8,1){.1}
\pscircle[fillstyle=solid](1.3,1.3){.1}
\pscircle[fillstyle=solid](2,0.5){.1}
\pscircle[fillstyle=solid](2,1.5){.1}
\pscircle[fillstyle=solid](4,2){.1}
\pscircle[fillstyle=solid](4,0.5){.1}

\pscurve[linewidth=1.5pt,arrowinset=0,linestyle=dashed]{-}(2,1.5)(2,0.5)(1.5,0.52)(0.8,1)

\psline[linewidth=1.5pt,arrowinset=0]{->}(0.8,1)(1.5,0.52)

\psline[linewidth=1.5pt,arrowinset=0]{<->}(1.3,1.3)(2,1.5)

\psline[linewidth=1.5pt,arrowinset=0]{->}(1.3,1.3)(.8,1)

\pscurve[linewidth=1.5pt,arrowinset=0]{-}(2,0.5)(1.6,0.48)(0.8,1)

\pscurve[linewidth=1.5pt,arrowinset=0]{->}(2,0.5)(2.5,0.4)(3,0.6)
\pscurve[linewidth=1.5pt,arrowinset=0]{-}(3,0.6)(3.5,0.7)(4,0.5)

\psline[linewidth=1.5pt,arrowinset=0]{<->}(1.3,1.3)(2,1.5)

\psline[arrowinset=.5,arrowlength=1.5]{<->}(4,2)(4,0.5)
\rput(4.2,1.2){$f_j$}

\pscurve[linewidth=1.5pt,arrowinset=0]{->}(2,1.5)(2.5,1.8)(3,1.7)
\pscurve[linewidth=1.5pt,arrowinset=0]{-}(3,1.7)(3.3,1.6)(4,2)
\rput(3.4,1.9){$ 1$}

\rput(1.35,1.7){$ 1$}\rput(1.4,0.22){$ 1$}
\rput(2.4,1){$0$} \rput(2.5,2){$ 1$}\rput(3.5,0.45){$ 1$}
\rput(3.2,0.9){$p'$}
\rput(3.55,1.35){$p$}
\rput(0.5,1){$C_1$}
\rput(1.85,1){$C_2$}
\rput(2.3,1.5){$w$}
\rput(2,.3){$w'$}

\put(.9,-0.7){\shortstack{\small{$f_j$: a bidirected nonhalf-}\\ 
\small{edge, $w\neq w'$.}}}

\endpspicture 

\end{array}$
\vspace{1cm}

\caption{An illustration of the different types of fundamental circuit of a nonbasic edge $f_j$ with weights on the edges, when $A$ is nonnegative. Each basic $1$-tree contains exactly one bidirected edge (which is entering). (In case $T\neq T'$, the basic full cycle may be replaced by a half-edge and the half-edge by a full cycle.)
}  
\label{fig:Apositive}

\end{figure} 

\section{Operations on binet matrices}\label{sec:BinetOp}

In this section we give some operations that, when applied to a binet matrix, result in another binet matrix.
The content of the section can be found in \cite{Appa-06,KotThesis}.

Let us start with a graphical operation which does not change the binet matrix. Switching at a node of a binet representation keeps the parity of the number of negative edges in a fundamental circuit (see Proposition \ref{propBidirectedSwitch}), so clearly does not affect the calculations (see the procedure WeightCircuit in Section \ref{sec:BinetDefi}). The matrix operation equivalent to switching is multiplying a row of the node-edge incidence matrix $In$ by $-1$. The effect of this change is easy to detect: a column in the inverse of the basis $B$ and the corresponding row in the nonbasic part $N$ is multiplied by $-1$. So $B^{-1} N$ does not change.

\begin{lem}\label{lemdefiSwi}
Switching at a node of a binet representation of a binet matrix keeps the binet matrix unchanged.
\end{lem}

Permuting rows or columns of a binet matrix obviously results in another binet matrix. Define a {\em
signing} of a vector as multiplying it by $-1$. Signing a row or a column of a binet matrix also results in a binet matrix, as it is equivalent to reversing the orientation of the corresponding basic, or nonbasic, edge.

\begin{lem}\label{lemdefiRowCol}
Let $A$ be a binet matrix. Matrix $A'$ obtained by the following operations from $A$ is also a binet matrix.

\begin{itemize}

\item[a)] deleting a row or a column,
\item[b)] repeating a row or a column,
\item[c)] signing a row or a column,
\item[d)] adding a unit row or a unit column. 

\end{itemize}
\end{lem}

The proof of Lemma \ref{lemdefiRowCol} can be found in \cite{Appa-06}. Part a) of Lemma \ref{lemdefiRowCol} can be rephrased.

\begin{thm}\label{thmBinetsub}
Every submatrix of a binet matrix is binet.
\end{thm}

Using Theorem \ref{thmBinetsub}, we may deduce Lemma \ref{lemdefiBlock} below. This lemma  will enable us to assume that a given matrix $A$ is connected for the recognition problem discussed in Chapter \ref{ch:Rec}.

\begin{lem}\label{lemdefiBlock}
A matrix which may be decomposed into blocks is binet 
if and only if all blocks are binet. 
\end{lem}

Another operation that preserves binetness is pivoting.

\begin{lem}\label{lemBinetPivot}
Let $A$ be a binet matrix. Matrix $A'$ obtained by pivoting on a non-zero element of $A$ is binet.
\end{lem}

\begin{proof}
The proof follows from the fact that pivoting is equivalent to changing the basis.
\end{proof}\\

\begin{lem}\label{lemBinetconnectedPivot}
Let $A$ be a connected matrix. The matrix $A'$ obtained by pivoting on a non-zero element of $A$ is connected.
\end{lem}

\begin{proof}
The proof by contradiction is straightforward.
\end{proof}\\

Switching at a node of the binet representation of a binet matrix keeps the binet matrix unchanged, which provides a method to find a special binet representation of a binet matrix.

\begin{lem}\label{lemBinetRepBid}
There always exists a representation of a binet matrix where each connected component of the basis contains only one bidirected edge.
\end{lem}

Finally, we show that the number of nonequal columns of a binet matrix is bounded by a quadratic function of  the number of rows.

\begin{lem}\label{lemBinetColBound}
Let $A$ be a binet matrix of size $n\times m$ having no two identical columns. Then $m\le 4\left( \begin{array}{c}
n\\
2 
\end{array} \right) + 2n +1 $.
\end{lem}

\begin{proof}
Let $G(A)$ be a binet representation of $A$. We know that $G(A)$ contains $n$ nodes.
So, since no two columns in $A$ are identical, the number of distinct nonbasic directed links is at most $2\left( \begin{array}{c}
n\\
2 
\end{array} \right)$ (which is two times the number of pairs of nodes). Similarly, for nonbasic bidirected links. Moreover,  the number of nonbasic half-edges is bounded by $2n$ and there is at most one loose edge (whose corresponding column is zero). This concludes the proof.
\end{proof}\\

\section{Some binet representations}\label{sec:Binetrep}

In this section, we define some particular binet representations and related results. Given a binet representation of a matrix $A$, by removing the nonbasic edges we obtain a union of basic components which is called a \emph{basic binet representation}\index{basic!binet representation} of $A$. Using $A$ and a basic binet representation of $A$, it is easy to construct a binet representation of $A$.

A binet representation of $A$ is called \emph{proper}\index{proper} if
each basic component has exactly one bidirected edge (contained in the basic cycle), this one is entering, and 
one end-node of the basic bidirected edge is not a consistent node of the basic cycle. A node in the basic cycle which is not consistent is called a \emph{central node}\index{central!node}. A \emph{central} edge\index{central!edge} is an edge in a basic cycle incident with a central node.

\begin{lem}
If the matrix $A$ is binet, then it has a proper binet representation.
\end{lem}

\begin{proof}
Consider a binet representation $G(A)$ of $A$. We may suppose that $G(A)$ has exactly one basic maximal negative $1$-tree, call it $T$, containing a full basic cycle $C$ which is not a loop. By Lemma \ref{lemBinetRepBid}, we may suppose that exactly one edge of $T$  is bidirected and this edge is entering. Denote by $w_1,\ldots,w_\rho$ the vertices of the cycle in clockwise order and assume that $[w_1,w_\rho]\in G(A)$. Consider the longest directed path from $w_1$ to a vertex, say $w_i$, in $C$. If $i=1$ or $\rho$, then one directed edge of the cycle is entering an end-node of the bidirected edge. Otherwise, we have directed edges $]w_{j-1},w_j]$ for $j=2,\ldots,i$ and $]w_{i+1},w_i]$ in $G(A)$. By switching at all nodes of the trees in $G(A)\verb"\"E(C)$ containing some node $w_j$ with $1\le j \le i-1$, we obtain a new binet representation $G'(A)$ of $A$ such that $[w_{i-1},w_i],]w_{i+1},w_i]\in G'(A)$, and $[w_{i-1},w_i]$ is the unique bidirected edge in the maximal basic $1$-tree. 
\end{proof}\\

A \emph{$\frac{1}{2}$-binet representation}\index{binet@$\frac{1}{2}$-binet!representation} is a proper binet representation in which every basic cycle is a half-edge. A matrix is said to be \emph{$\frac{1}{2}$-binet}\index{binet@$\frac{1}{2}$-binet!matrix} if it has a $\frac{1}{2}$-binet representation.

A \emph{cyclic representation}\index{cyclic!representation} of a matrix is a proper binet representation  of the matrix having exactly one basic cycle, and this one is full. A matrix  is said to be \emph{cyclic}\index{cyclic!matrix}, if it has a cyclic representation. 
For $R$ a row index set, an \emph{$R$-cyclic representation}\index{cyclic@$R$-cyclic!representation} of a matrix $A$ denotes a cyclic representation such that $R$ is the edge index set of the basic cycle, and this cycle is contained in the fundamental circuit of at least one nonbasic edge $f_j$ (or equivalently $R\subset s(A_{\bullet j})$ for at least one column index $j$ of $A$). A matrix is said to be \emph{$R$-cyclic}\index{cyclic@$R$-cyclic!matrix}, if it has an $R$-cyclic representation. If a matrix $A$ has an $R$-cyclic representation $G(A)$, up to row permutations we may assume that $w_1,\ldots,w_\rho$ are the vertices of the basic cycle, $e_1=[w_1,w_\rho]$, $e_\rho=]w_{\rho-1},w_\rho] \in G(A)$ and for $i=2,\ldots, \rho$, $e_i$ is the edge incident with $w_{i-1}$ and $w_i$.

A \emph{bicyclic representation}\index{bicyclic!representation} of a matrix is a proper binet representation  of the matrix having exactly two basic cycles, and these are full. A matrix  is said to be \emph{bicyclic}\index{bicyclic!matrix}, if it has a bicyclic representation.

Finally, we say that a binet representation is an \emph{$\{\epsilon,\rho\}$-central representation}\index{central@$\{\epsilon,\rho\}$-central!representation}, if it is cyclic, $e_\epsilon$ and $e_\rho$ are edges of the basic cycle incident with one common central node, and one of the edges $e_\epsilon$ and $e_\rho$ is bidirected.
Let us mention here that if a binet representation is $\{\epsilon,\rho\}$-central and $e_\rho$ is bidirected, then up to switching operations it is possible to transform it into an 
$\{\epsilon,\rho\}$-central representation such that $e_\epsilon$ is bidirected.
We say that a matrix is an \emph{ $\{\epsilon,\rho\}$-central}\index{central@$\{\epsilon,\rho\}$-central!matrix} matrix, if it has an $\{\epsilon,\rho\}$-central representation.
Suppose that a matrix $A$ has an $\{\epsilon,\rho\}$-central representation $G(A)$ and $T$ is the basic maximal $1$-tree in $G(A)$. We denote by $G_1(A)$ the connected component of $T\verb"\" \{e_\epsilon,e_\rho\}$ containing the central node incident to $e_\epsilon$ and $e_\rho$, and by $G_0(A)$ the other connected component. We say that a basic subgraph of $G(A)$ is \emph{on the right}\index{on the right} (respectively, \emph{on the left}\index{on the left}) \emph{of $\{e_\epsilon,e_\rho\}$}, if it is contained in $G_1(A)$ (respectively, $G_0(A)$).

Using Lemma \ref{lemdefiWeight1}, we may deduce the following important lemma.

\begin{lem}\label{lemBinetintegral}
Let $A$ be an integral matrix. Then $A$ is binet if and only if it is either cyclic or $\frac{1}{2}$-binet.
\end{lem}

\section{Linear and integer programming}\label{sec:BinetOpt}

In this section, we deal with linear and integer programming in which the constraint matrix is the node-edge incidence matrix of a bidirected graph (IMB). Let us designate a linear (respectively, integer) programming problem with an IMB constraint matrix a \emph{bidirected LP}\index{bidirected!LP} (respectively, \emph{bidirected IP})\index{bidirected!IP}. It is shown that efficient methods are available to solve bidirected LPs and IPs. Finally, we see that the problem of converting some linear or integer program to a bidirected LP or IP, respectively, using elementary row operations on the constraint matrix  is equivalent to the recognition of binet matrices.

Consider the linear programm 
\begin{eqnarray}
\min  & c^T x \nonumber \\
 s.c. &  Ax = b \label{eqnOptprog}\\
      & 0\le x \le \alpha \nonumber
\end{eqnarray}
\noindent
 where $c,\alpha \in \mathbb{R}^m$, $b\in \mathbb{R}^n$, $A \in \mathbb{R}^{n \times m}$ for some $n,m\in \mathbb{N}$ and some entries of $\alpha$ might be infinite.

If $A$ is an RIMB, then problem (\ref{eqnOptprog}) is a bidirected LP, and (\ref{eqnOptprog}) can be seen as a generalization of the transshipment problem described at page \pageref{mycounter1}. The optimal solution of a bidirected LP can be achieved by general-purpose methods, like the simplex algorithm, or the strongly polynomial algorithm of Tardos \cite{TardosPolAlg-86}. This latter states that for rational linear programming problems in which the elements of the constraint matrix are bounded, there exists an algorithm to solve the problem that uses arithmetic operations whose number is a polynomial function of the dimension of the problem and which act on rationals of size polynomially bounded by the size of the input. More about this ingenious algorithm can be found in \cite{ShrAlex}. Since the elements of an RIMB are between $-2$ and $2$, Tardos' algorithm on bidirected LP has a strongly polynomial worst-case running time. However, despite this attrative theoretical complexity result, the up-to-date implementations of the simplex algorithm usually outperform Tardos' strongly polynomial method on practical instances (see \cite{KotThesis}).

Alternatively, Kotnyek \cite{KotThesis} proved that a bidirected LP can be converted to a generalized network problem.  A \emph{generalized network}\index{generalized network} is defined on a connected digraph $G$. There is a real non-zero multiplier $p_e$ associated with each edge $e=(i,j)$ of $G$. We assume that if a unit flow leaves the tail $i$ of $e$, then $p_e$ units arrive at $j$. $G$ can also contain loops, i.e., edges whose tail and head coincide. It is assumed that the multiplier of a loop cannot be $+1$, as it would mean that the same flow leaves and enters the node on such a loop, making the loop redundant. Trivially, if all multipliers are equal to $1$, then we have the well-known pure network. 

A way of describing a generalized network is with its node-edge incidence matrix. The column of the incidence matrix that corresponds to a non-loop edge $e=(i,j)$ has $-1$ in row $i$ and $p_e$ in row $j$, zeros elsewhere. If $e$ is a loop at node $i$, then its column has only one non-zero, $(p_e-1)$ in row $i$. A \emph{generalized network problem}\index{generalized network!problem} is a linear program in which the constraint matrix is the node-edge incidence matrix of a generalized network. See \cite{AhujaMagOrlinNet-93,MurtyNetwork-92}.

The idea behind the network simplex method is that the main steps of the simplex method (such as calculating the primal and dual solutions corresponding to a basis, or changing the basis) can be executed on the network associated with the constraint matrix. This idea can be followed in generalized networks too, leading to the \emph{generalized network simplex method}\index{generalized network!simplex method}. It shoul be noted that the generalized network simplex method is not polynomial in the worst case, but for most of the practical problems it is much more efficient than the simplex method, or the strongly polynomial method mentioned above (see the reference notes in \cite{AhujaMagOrlinNet-93}).

Now we turn to the integer case, that is, we are facing with the following integer program:
\begin{eqnarray}
&\min  & c^T x \nonumber \\
& s.c. &  Ax = b \label{eqnOptprogI}\\
&      & 0\le x \le \alpha, \nonumber \\
&      & x \m{ integral } \nonumber
\end{eqnarray}
\noindent
where $c,\alpha \in \mathbb{R}^m$, $b\in \mathbb{R}^n$, $A \in \mathbb{R}^{n \times m}$ for some $n,m\in \mathbb{N}$ and some entries of $\alpha$ might be infinite.

If $A$ is an RIMB, then (\ref{eqnOptprogI}) is a bidirected IP. The bidirected IP problem was introduced by Edmonds \cite{EdmondsMatching-67} as a generalization of the following problem. Given any graph $G=(V,E)$ with a real numerical weight $c_e$ for each edge $e\in E$ and an integer $b_v$ for each node $v\in V$, find in $G$, if there is one, a subgraph $G'$ which has degrees $b_v$ at nodes $v$ and whose edges have maximum weight-sum. This is called the \emph{optimum $b$-matching problem}\index{optimum $b$-matching problem} on graph $G$. A \emph{$b$-matching}\index{matching@$b$-matching} $E'$ in $G$ is a subset $E'\subseteq E$ of edges such that $b_v$ edge-ends of edges in $E'$ meet node $v$. Obviously there is a $1-1$ correspondence between the $b$-matchings in a graph $G$ and the $b$-degree subgraphs of $G$ that contain all the nodes of $G$.

If $A$ is an IMB with respect to a graph $G$ (all edges are bidirected entering links) and $\alpha_e=1$ for every edge in $G$, then (\ref{eqnOptprogI}) is simply the optimum $b$-matching problem relative to $G$. Edmonds
showed that any bidirected IP is equivalent
to a $b$-matching problem (see also \cite{LawlerOpt-76}).
Moreover, Edmonds and Johnson \cite{EdmondsMatching2} proved that there exists a polynomial algorithm to solve $b$-matching problems, even on bidirected graphs.

At last, suppose  that $A$ is in standard form (i.e., the first $n$ columns of $A$ constitute an identity matrix) and we wish to 
convert the problem (\ref{eqnOptprog}) (resp., (\ref{eqnOptprogI})) to a bidirected LP (resp., IP) using only elementary row operations on $A$. This is equivalent to the problem of recognizing binet matrices as stated in the next theorem.

\begin{thm}\label{thmOptEqui}
Let $A$ be a real matrix in standard form. Then $A$ is projectively equivalent to an RIMB if and only if $A$ is a binet matrix.
\end{thm}

\begin{proof} The proof is similar to the proof of Theorem \ref{thmSubclassNEqui}.
\end{proof}\\

\noindent
By Theorems \ref{thmOptEqui} and \ref{thmRecBinetMain}, it results that there is an $O(\max(n^6m,n^2 m^3))$ algorithm (the algorithm Binet) for determining whether $A$ is projectively equivalent to an IMB, and thus whether the problem (\ref{eqnOptprog}) (resp., (\ref{eqnOptprogI})) is convertible to a bidirected LP (resp., IP). If such a conversion is possible, the algorithm Binet, in effect, carries it out.

\section{A subclass of $2$-regular matrices}\label{sec:Binetreg}

In this section we introduce the definition of $k$-regular matrices given by Appa and Kotnyek \cite{AppaKotkReg-04,KotThesis}, as a generalization of totally unimodular matrices, and  provide a geometrical interpretation. Then we concentrate on the case $k=2$. It is proved that $2$-regular matrices contain binet matrices. Moreover, we present some known theorems about binet matrices or their transposes with strong Chv\'atal rank $1$. Finally, examples of minimally non-binet $2$-regular matrices are given. The content of this section is mainly taken from \cite{KotThesis}.


Given a rational matrix $A$ of size $n\times m$ and a rational vector $b$ of size $n$, consider $P=\{ x \, :\, x\geq 0, \, Ax \le b \}$. A primary problem of integer programming is to find the \emph{integer hull}\index{integer hull} $P_I=conv\{ P \cap \mathbb{Z}^m\}$  of the polyhedron $P$. A rephrasing of Theorem \ref{thmSubclassNHof} is that in the case of an integral matrix $A$, we have $P=P_I$ for all integral right hand side vectors $b$, if and only if $A$ is totally unimodular. In terms of integer programming, totally unimodular matrices are the integral matrices for which $\max \{ cx \, | \, Ax \le b, \, x \geq 0 \}$ has integral optimal solutions for any $c$ and any integral $b$. 

There are situations, however, that take us beyond  total unimodularity. Theorem \ref{thmSubclassNHof} holds provided that $A$ is integral. What if the matrix in question is not integral? Or what can we say about matrices $A$ that ensure integral optimal solutions for only a special set of right hand side? These questions are not independent. If $A$ is rational, then one can find a nonnegative integer $k$, such that if we multiply every row of $A$ by $k$, we get an integral matrix, $kA$. But then instead of inequalities $Ax \le b$, we have $k Ax \le kb $ and we deal with polyhedra that are required to be integral for  only special $b'$ vectors, namely for those whose elements are integer multiples of $k$. For example, if $k=2$, so the elements of $A$ are halves of integers, then we are to characterize integral matrices $A'$ for which $\{ x\, | \, A'x \le b', \, x \geq 0 \}$ is integral for all {\bf even} vectors $b'$. Or equivalently, we examine matrices that provide half-integral vertices for polyhedra with integral right hand sides.

A matrix is called \emph{$k$-regular}\index{regular@$k$-regular} ($k\in \mathbb{N}$) if for each of its non-singular square submatrices $\pi$, $k \pi^{-1}$ is integral. $k$-regularity is the property that takes over the role of total unimodularity in the theory of rational matrices that ensure integral vertices for polyhedra with special right hand sides. One important theorem states this (see \cite{KotThesis}).

\begin{thm}\label{thmsubclassBkreg}
A rational matrix $A$ is $k$-regular, if and only if the polyhedron $\{x \, | \, Ax \le kb, \, x \geq 0 \}$ is integral for any integral vector $b$.
\end{thm}

The second nice result is the following.

\begin{thm}\label{thmsubclassBbinet}
Every binet matrix is $2$-regular.
\end{thm} 

\begin{proof}
Let $A$ be a binet matrix of size $n\times m$. By Theorem \ref{thmBinetsub}, we only have to prove that for any non-singular $n\times n$ submatrix $\pi$ of $A$ the matrix $2 \pi^{-1}$ is integral. Let $\pi$ be a non-singular $n\times n$ submatrix of $A$. Then, by definiton of a binet matrix, there exists an RIMB $In=[B\,\, B' \,\, N ]$ where $B$ and $B'$ are two disjoint bases of $In$, $A=B^{-1}[ B'\,\, N]$ and $\pi=B^{-1} B'$. Thus $(B')^{-1}B=\pi^{-1}$ is also binet.
By Lemma \ref{lemdefiWeight1}, every entry of $\pi^{-1}$ is half-integral, which completes the proof.
\end{proof}\\

Now let us define the notion of strong Chv\'atal rank $1$.
Given a polyhedron $P$, in several theoretical and practical problems, we have $P\neq P_I$. To tackle these cases and find integer solutions, different methods have been developed. One approach is the cutting plane method, pioneered by Gomory \cite{GomoryIP-58}. Its most basic concept is the \emph{Chv\'atal-Gomory (CG) cut}\index{Chv\'atal-Gomory (CG) cut}, defined as follows. Given a rational $n\times m$ matrix $A$ and a rational vector $b$ of size $n$, a CG cut of the polyhedron $P=\{ x \, :\, Ax \le b \}$ is an inequality of the form $\lambda^T Ax \le \lfloor \lambda^T b \rfloor$ where $\lambda \in \mathbb{R}^n_+$ and $\lambda^T A\in \mathbb{Z}^m$ ($\lfloor . \rfloor$ denotes the lower integer part).

The intersection of $P$ with the half-spaces induced by all possible CG cuts is its \emph{rank-$1$ closure}\index{rank-$1$ closure}, denoted $P_1$. Obviously, $P_I\subseteq P_1 \subseteq P$. It is also known that $P=P_I$ holds if and only if $P=P_1$.
Matrices $A$ for which the integer hull $P_I$ of $P=\{ x \, :\, x\geq 0, \, Ax \le b \}$ is the same as the rank-$1$ closure $P_1$ for any integral $b$ are said to have \emph{Chv\'atal rank $1$}\index{Chv\'atal rank $1$}.

A stronger requirement is to assume that we have lower and upper bounds on $Ax$ and $x$, so the polyhedron is of the form $P=\{ x \, | \, \alpha \le x \le \beta, \, a \le Ax \le b\}$. Matrix $A$ has \emph{strong Chv\'atal rank $1$}\index{strong Chv\'atal rank $1$}, if $P_1=P_I$ for any integral choice for $\alpha$, $\beta$, $a$ and $b$ (including $\infty$).

Edmonds and Johnson \cite{EdmJohnMatEu-73,EdmondsMatching2} showed that if $A$ is the node-edge incidence matrix if a bidirected graph, then it has strong Chv\'atal rank $1$. (That is why matrices with strong Chv\'atal rank $1$ are sometimes said to have the \emph{Edmonds-Johnson property}.) In \cite{GerSchAlex-Edm-86}, Gerards and Schrijver  gave a characterization of matrices that are edge-node incidence matrices of bidirected graphs and have strong Chv\'atal rank $1$. The key matrix in their characterization is $In(K_4)$, the edge-node incidence matrix of $K_4$, the complete graph on four nodes:

$$In(K_4)= \left[ 
\begin{array}{cccc}
1 &1&0&0 \\
1&0&1&0\\
1&0&0&1\\
0&1&1&0 \\
0&1&0&1 \\
0&0&1&1 \\
\end{array}\right] $$

The edge-node incidence matrices of bidirected graphs are exactly the integral matrices $In$ (of size $m \times n$) that satisfy the transposed version of (\ref{eqnBidirected}), namely,

\begin{equation}\label{eqnsubclassBEN}
\sum_{j=1}^n |(In)_{ij}| \le 2 \m{ for } i=1,\ldots,m,
\end{equation}

The characterization appearing in \cite{GerSchAlex-Edm-86} is now the following.

\begin{thm}\label{thmsubclassBGerSch}(Gerards and Schrijver) An integral matrix satisfying 
\ref{eqnsubclassBEN} has strong Chv\'atal rank $1$, if and only if it cannot be transformed to $In(K_4)$ by a series of following operations:

\begin{itemize}
\item[(i)] deleting or permuting rows or columns, or multiplying them by $-1$;
\item[(ii)] replacing matrix $\left[
\begin{array}{cc}
1  & g \\
f & D
\end{array} \right]$ by the matrix $D-fg$.
\end{itemize}

\end{thm}

Here we extend the set of matrices with strong Chv\'atal rank $1$ by stating that integral binet matrices are such \cite{KotThesis}.

\begin{thm}\label{thmsubclassBBinIn}
If $A$ is an integral binet matrix, then it has strong Chv\'atal rank $1$.
\end{thm}

Note that Theorem \ref{thmsubclassBBinIn} cannot be extended to rational binet matrices, as the following example shows.

$$A=B^{-1} N = \left[
\begin{array}{ccc}
\frac{1}{2}  & 1 & 0  \\
0 & 1 & 1 \\
\frac{1}{2} & 0 & 1
\end{array} \right] \m{ with } B= \left[
\begin{array}{ccc}
1  & 0 & -1  \\
1 & -1 & 1 \\
0 & 1 & 0
\end{array} \right], N= \left[
\begin{array}{ccc}
0  & 1 & -1  \\
1 & 0 & 0 \\
0 & 1 & 1
\end{array} \right].$$

Binet matrix $A$ does not have strong Chv\'atal rank $1$, because the non-zero integral solutions of the polyhedron $P= \{ x \, : \, 0\le x \le 1, 0 \le Ax \le 1 \}$ are : $(1,0,0)$, $(0,1,0)$ and $(0,0,1)$, so $x_1 + x_2 + x_3 \le 1$ is a facet of $P_I$. But $(1,\frac{1}{2},\frac{1}{2})\in P$, so $\delta=2$ is the smallest value for which $x_1 + x_2 + x_3 \le \delta$ is valid for $P$.
(In \cite{KotThesis} the notion of half-integral cut is discussed.)

Now let us study some minimally non-binet matrices. The graphical equivalent of the matrix operations in Theorem \ref{thmsubclassBGerSch} can also be given, following \cite{GerSchAlex-Edm-86} and Section \ref{sec:BinetOp}. Deleting a row or a column of an edge-node incidence matrix is equivalent to deleting an edge or a node from the graph. Multiplying a row with $-1$ translates to reversing the direction of an edge, whil multiplying a column with $-1$ corresponds to a switching. Operation (ii) has already appeared as (\ref{eqnContraction}), so it is the same as the contraction of an edge (note that (\ref{eqnContraction}) is symmetric to transposing). Thus, the edge-node incidence matrix of a bidirected graph has strong Chv\'atal rank $1$, if and only if the graph cannot be transformed to $K_4$ by a series of edge and node deletions, edge direction reversals, switching and contractions. Combining Theorems \ref{thmsubclassBGerSch} and \ref{thmsubclassBBinIn} we get:

\begin{thm}
If a bidirected graph can be transformed to $K_4$ by a series of edge and node deletions, edge direction reversals, switching and contractions, then its edge-node incidence matrix is not binet.
\end{thm}

Furthermore, let $In(K_6)$ the edge-node incidence matrix of the complete graph on six nodes. Kotnyek proved that we cannot change the signs of some entries of $In(K_6)$ to make it binet. On the other hand, if the edges of $K_6$ are oriented so that the graph is directed, then the corresponding edge-node incidence matrix is the transpose of a network matrix, so it is totally unimodular. Thus, we have an example of a matrix which is totally unimodular, but not binet. A complete list of minimally non-binet totally unimodular matrices follows from a submitted paper of Hongxun Qin, Daniel C. Slilaty and Xiyngqian Zhou \cite{Slilaty-07}.

There are also minimally non-binet $2$-regular matrices that are not totally unimodular. A trivial example is $[\frac{1}{2}\, 2]$. No binet matrix can have a $\pm 2$ and a $\pm \frac{1}{2}$ in the same row or column, for example because then pivoting on the $\pm \frac{1}{2}$ would result in a $\pm 4$. The matrix 

$$\left[
\begin{array}{cccc}
1  & 1 & 1 & 1  \\
-1 & 1 & 1 & -1 \\
1 & 1 & -1 & -1
\end{array} \right]$$

\noindent 
and its transpose are minimally non-binet, as can be shown by careful analysis of all the cases, but they are $2$-regular and clearly not totally unimodular.

At last, we saw in Section \ref{sec:BinetDefi} that binet matrices extend the class of network matrices, and in Section \ref{sec:IntNet} that every totally unimodular matrix is built up by simple operations from network matrices and two further matrices (\ref{eqnSubclassNmat2}) that are not network.
We mention here that the matrices (\ref{eqnSubclassNmat2}) were proved to be binet by Kotnyek  \cite{KotThesis}. Since the class of $2$-regular matrices extends the binet matrices,
it is conjectured that a characterization theorem should exist for a decomposition of $2$-regular matrices into binet matrices and probably other non-binet matrices.

\section{Matroids related to binet matrices}\label{sec:BinetMat}

In Subsection \ref{sec:BinetMatSign}, we introduce a generalization of undirected graphs, called signed graphs. Signed graphs can be achieved from bidirected graphs by ignoring the signs at the ends of the edges and focusing on only whether an edge is bidirected or directed. Subsection \ref{sec:BinetMatSub} contains an overview of the necessary theory about matroids. In Subsection \ref{sec:BinetMatGraph}, we define the signed-graphic matroid, that is a combinatorial structure associated with a signed graph. We will show that binet matrices are representative matrices of signed-graphic matroids, therefore results about bidirected graphs
have consequences in terms of these matroids. We define also the near-regular matroid as a generalization of regular matroids.

The most important reference about signed graphs is Thomas Zaslavsky's work. His research is the basis of this chapter. Signed graphs were introduced by Harary \cite{HararyFSigned-54}.
Zaslavsky's annotated bibliography of signed and gain graphs \cite{Zaslavsky-Bib-98}, which contains hundreds of references, is an essential tool for anyone interested in the subject. Our notations and results mainly follow the Glossary of Signed and Gain Graphs \cite{Zaslavsky-Glos-99}, by Zaslavsky, and Kotnyek's thesis \cite{KotThesis}.

\subsection{Signed graphs}\label{sec:BinetMatSign}

A \emph{signed graph}\index{signed graph} is a pair $\Sigma=(G,\sigma)$, where $G$ is an undirected graph with possibly loose edges (with no endnodes) and half-edges (having exactly one endnode), and the edges of $G$ are labelled by $+$ or $-$, that is, there is a mapping $\sigma\,: \, E \rightarrow \{+,-\}$ on the edges. It is also required that the label of a loose edge is $+$ and that of a half-edge is $-$.

An edge of a signed graph with two distinct endnodes is called a \emph{link}\index{link}. A \emph{path}\index{path} is a sequence of links $e_1,\ldots, e_k$ where $e_i$ and $e_{i+1}$ ($i=1,\ldots, k-1$) have a common endnode, but none of these nodes is repeated. If $e_1$ and $e_k$ also have a common endnode for $k\geq 3$ (and two common endnodes for $k=2$), then the path is \emph{closed}\index{closed}. A closed path, a loop or a half-edge is called a \emph{cycle}\index{cycle}. A cycle is called \emph{positive}\index{positive cycle} (\emph{negative}\index{negative cycle}), if the product of the signs of its edges is positive (negative).

Clearly, a bidirected graph is a signed graph, bidirected edges are negative, directed and loose edges are positive. Conversely, the edges of a signed graph $\Sigma$ can be \emph{oriented}\index{orient} to get a bidirected graph, denoted $\vec \Sigma$. To do so, we allocate arbitrary signs to the ends of every edge so that positive edges become directed and negative edges become bidirected. More formally, if $u$ and $v$ are (possibly coinciding) endnodes of a link or loop $e$ and the sign of $e$ at $u$ is $s(e,u)$, then the sign of $e$ at $v$ is $s(e,v)=-\sigma (e) s(e,u)$. 

Basic operations on signed graphs, such as deletion, contraction and switching are defined in the same way as for bidirected graphs (see Section \ref{sec:Incidence}). A subgraph of a signed graph achieved by deletions and contractions of edges is sometimes called the \emph{minor}\index{minor} of the graph. Two signed graphs that can be obtained from each other by switchings are called \emph{switching equivalent}\index{switching equivalent}.

Now we derive the node-edge incidence matrix of a signed graph, it is the node-edge incidence matrix of the bidirected graph obtained by an orientation. Different orientations yield different node-edge incidence matrices. That is, the incidence matrix of a signed graph is not unambiguously defined, but any incidence matrix conveys all the information about the signed graph. In fact, two bidirected graphs can be oriented versions of two switching equivalent signed graphs if and only if one can be obtained from the other by switchings and edge direction reversals.

\subsection{Matroids}\label{sec:BinetMatSub}

A \emph{matroid}\index{matroid} $M$ is a finite ground set $S$ and a collection $\phi$ of subsets of $S$  such that (I1)-(I3) are satisfied.

\begin{itemize}

\item[(I1)] $\emptyset \in \phi$.

\item[(I2)] If $X\in \phi$ and $Y\subseteq X$ then $Y\in \phi$.

\item[(I3)] If $X,Y$ are members of $\phi$ with $|X|=|Y|+1$ there exists $x\in X\verb"\" Y$ such that $Y\cup \{x\} \in \phi$.

\end{itemize}

Subsets of $S$ in $\phi$ are called the \emph{independent sets}\index{independent set}, a maximal independent subset in $S$ is a \emph{basis}\index{basis of a matroid} of $M$. The \emph{rank}\index{rank of a matroid} of $M$, which is the cardinality of a basis, is called the rank of the matroid, denoted as $r(M)$. A \emph{circuit}\index{circuit of a matroid} is a minimal dependent subset of $S$. If $B$ is a basis of $M$ and $s\in S\verb"\" B$, then there is a unique circuit in $B\cup \{s\}$, called the \emph{fundamental circuit of $s$ with respect to $B$}\index{fundamental circuit in a matroid}. There are several different but equivalent ways to define a matroid. For example, one can define a matroid on a given ground set through its bases, rank-function, circuits, or fundamental circuits.
A matroid $M$ is called \emph{connected}\index{connected matroid} if for every pair of distinct elements $x$ and $y$ of $S$ there is a circuit of $M$ containing $x$ and $y$. Otherwise, it is disconnected.

An important class of matroids is that of the \emph{uniform matroids}\index{uniform matroid}. If $|S|=m$, and $\phi$ contains all the subsets with at most $r$ elements ($r\le m$), then $(S,\phi)$ is the uniform matroid of rank $r$ on $m$ elements, denoted as $U_m^r$.

An other standard example of matroids is when $S$ is a finite set of vectors from a vectorspace over a field $\mathbb{F}$ and $\phi$ contains the linearly independent subsets of $S$. This kind of matroid is called the \emph{linear matroid}\index{linear matroid}. If for a matroid $M$ there exists a field $\mathbb{F}$ such that $M$ is a linear matroid over $\mathbb{F}$, then $M$ is \emph{representable}\index{representable} over $\mathbb{F}$. Matroids representable over $GF(2)$, the field with two elements, are called \emph{binary}\index{binary}, and matroids representable over $GF(3)$, the field with three elements, are called \emph{ternary}\index{ternary}. A matroid representable over every field is \emph{regular}\index{regular}. For a matroid to be regular, it must obviously be binary and ternary. It turns out that this condition is also sufficient.

\begin{thm}[Tutte \cite{TutteMat-58,Tutte-65}]
A matroid is regular if and only if it is $GF(2)$- and $GF(3)$-representable.
\end{thm}

Linear matroids can be represented by matrices. Let $M$ be a linear matroid on a finite set of vectors over $\mathbb{F}$. The matrix $In$ made up of these vectors is a \emph{standard representation matrix}\index{standard representation matrix} of $M$. There is a one-to-one correspondence between linearly independent columns of $In$ and independent sets in $M$, so the linear matroid $M=M(In)$ can be fully given by its representation matrix $In$. Obviously, deleting linearly dependent (over $\mathbb{F}$) rows from $In$ does not have any effect on the structure of independent columns. Therefore, we can assume that the rows of $A$ are linearly independent.

If $In$ is the node-edge incidence matrix of a (connected) undirected graph $G$, then the corresponding binary matroid $M(In)$ is called \emph{graphic}\index{graphic matroid}. A graphic matroid can be viewed as defined on the edges of the graph, the independent sets are the (edge-sets of) forests, a circuit is a cycle, a basis is a spanning tree.  It is known (see \cite{WelshMatTheory-76}) that for a given connected graph $G$ with no loop and at least three vertices, the graphic matroid $In(G)$ is connected if and only if $G$ is $2$-connected.

There is another, more compact representation matrix of a linear matroid $M(In)$. To get it, first delete linearly dependent (over $\mathbb{F}$) rows from $In$, if there are any, then choose a basis $B$ of $M$. It corresponds to a basis of $In$, also denoted by $B$. By multiplying $In$ by the inverse of $B$ (computed over the fielf $\mathbb{F}$), the submatrix $B$ can be converted to an identity matrix. It is clear that the independence is not affected by this operation, so the transformed matrix also represents $M$. This remains true if the full-row-size identity submatrix is deleted. The remaining matrix, denoted as $A$, is a \emph{compact representation matrix of $M$ over $\mathbb{F}$}\index{compact representation matrix}. The rows (respectively, columns) of $A$ correspond to the vectors in (respectively, out of) the basis $B$. Note that a matroid might have several different compact representation matrices, depending on the selection of basis $B$. A column $s$ of $A$ gives us the fundamental circuit of $s$ with respect to the chosen basis $B$. In fact, the basic vectors whose rows contain non-zeros in column $s$ are the basic elements of the fundamental circuit of $s$. As an example, take the following matrix over $GF(3)$.
 
\begin{eqnarray}\label{eqnBinetU42}
A=\begin{tabular}{c|c|c|}
 & $s_3$ & $s_4$ \\
\hline
$s_1$ & 1 & 1 \\
\hline
$s_2$ & 1 & -1 \\
\hline
\end{tabular}
\end{eqnarray}

The ternary matroid represented by $A$ is the uniform matroid $U_4^2$. Its ground set has four vectors $\{s_1,s_2,s_3,s_4\}$. Subset $B=\{s_1,s_2\}$ is a basis. The fundamental circuit of say $s_3$ is $\{s_1,s_2,s_3\}$.

Graphic matroids provide an other example. Take a connected undirected graph $G$ and its node-edge incidence matrix $In$, which is the standard representation of the graphic matroid based on $G$. First we delete a row to make $In$ a full row rank matrix $In'$. Selecting a basis $B$ of $In'$, which correspond to a spanning tree of $G$, and pivoting on its elements is equivalent to premultiplying $In'$ by the inverse of $B$ (all the operations are done modulo $2$). As a result, we get a matrix $A$ the columns of which give the fundamental cycles of $G$ with respect to $B$, i.e., the unique cycle that contains exactly one non-tree edge. In other words, $A$ is an unsigned (i.e., modulo $2$) network matrix (see Section \ref{sec:IntNet}).

\begin{lem}\label{lemMatroidNetComp}
Any compact representation matrix of a graphic matroid  can be signed with $\{+,-\}$ to obtain a network matrix. Conversely, the binary support of any network matrix is the compact representation matrix of a graphic matroid.
\end{lem}

Basic operations on matroids are dualization, deletion and contraction. If $M=(S,\phi)$ is a matroid, then $M^* =(S, \{S\verb"\" X \, : \, X\in \phi\})$ is also a matroid, called the \emph{dual}\index{dual} of $M$. If $M$ is linear and $A$ is a compact representation matrix of $M$, then $A^T$ is a compact representation matrix of $M^*$. A consequence of this fact is that if $M$ is representable over a field $\mathbb{F}$, then so is the dual of $M$. In matroid terminology, dualization is usually expressed by the 'co' prefix\index{co-matroid}. Thus, if $M^*$ is graphic, then $M$ is \emph{cographic}. We state one of the most  beautiful theorems in combinatorial theory characterizing
the graphs embeddable in the plane by use of their co-graphic matroid (see \cite{WelshMatTheory-76}). 

\begin{thm}\label{thmMatroidBeautiful}
Let $G$ be an undirected graph. The co-graphic matroid $M^*(G)$ is graphic if and only if $G$ is planar.
\end{thm}

The \emph{deletion}\index{deletion} of $X\subseteq S$ from $M$ results in a matroid (denoted as $M\verb"\" X$) on $S\verb"\"X$ the independent sets of which are in $\{Y \subseteq S\verb"\"X \, :\, Y\in \phi \}$. The matroid resulting from the \emph{contraction}\index{contraction} of a set $X\subseteq S$ in $M$ is defined as $M/X = (M^* \verb"\" X)^*$. Deletion and contraction in graphic matroids are naturally expressed by deletion and contraction of edges. A matroid achieved by contractions and deletions in $M$ is called a \emph{minor}\index{minor} of $M$. For any minor $N$ of a linear matroid $M$ one can find a compact representation matrix $A$ of $M$ such that $N$ is represented by a submatrix of $A$. As a corollary, representability over a field is maintained under minor-taking. It is a classical result that binary matroid can be characterized by forbidden uniform minors.

\begin{thm}(Tutte \cite{TutteMat-58}) A matroid is binary if and only if it does not have $U_4^2$ minors.
\end{thm}

Another kind of characterization is when matroids are decomposed to simpler matroids in special ways. The most striking of this kind of results is the decomposition of regular matroid, due to Seymour \cite{SeymourRegular-80}. It claims that the building blocks of a regular matroid are graphic matroids, cographic matroids, or matroids represented by the following compact representation matrix.

$$ 
A(R_{10})=\left[ \begin{array}{ccccc}
1 & 1 & 1 & 1 & 1 \\
1 & 1 & 1 & 0 & 0 \\
1 & 0 & 1 & 1 & 0 \\
1 & 0 & 0 & 1 & 1 \\
1 & 1 & 0 & 0 & 1 \\
\end{array} \right] $$

The special way regular matroids are built up are through $1$-sums, $2$-sums and $3$-sums, which we do not define here. The interested reader can find the definitions in e.g., \cite{TruemperMat-92}. The decomposition theorem of regular matroids goes as follows.

\begin{thm}\label{thmMatroidReg}
Every regular matroid can be produced from graphic and cographic matroids and $R_{10}$ by consecutive $1-$, $2-$, and $3-$sums. Conversely, every matroid produced this way is regular.
\end{thm}

The importance of regular matroids is very much connected to the following theorem.

\begin{thm}\label{thmMatroidTutteUni}
(Tutte \cite{TutteMat-58,Tutte-65})
A matroid is regular if and only if it has a binary compact representation matrix the $1$s of which can be replaced by $\pm 1$ so that the resulting real matrix is totally unimodular.
\end{thm}

Theorem \ref{thmSubclassNSey} in Section \ref{sec:IntNet} is a consequence of Theorems \ref{thmMatroidReg} and \ref{thmMatroidTutteUni}.

\subsection{The signed-graphic matroid}\label{sec:BinetMatGraph}

In Lemma \ref{lemMatroidNetComp} we stated the well-known result that unsigned network matrices are compact representation matrices over $GF(2)$ of graphic matroids. In graphical terms this means that with any direction of the edges of $G$, the undirected graph underlying the graphical matroids, leads to a network matrix defined on the now directed graph $G$. We see an analogous result with binet matrices.

Let $\Sigma$ be a signed graph. The \emph{signed-graphic matroid of $\Sigma$}\index{signed-graphic matroid} is denoted by $M(\Sigma)$. The element set of $M(\Sigma)$ is $E(\Sigma)$ and a circuit of $M(\Sigma)$ falls in one the following categories.

\begin{itemize}

\item[(i)] it is a loose edge, or
\item[(ii)] a positive cycle, or
\item[(iii)] a pair of negative cycles with exactly one common node, or
\item[(iv)] a pair of disjoint negative cycles along with a minimal connecting path.

\end{itemize}

Now take a bidirected graph $G$,  remove the orientation of the edges to get a signed graph $\Sigma$. By Corollary \ref{corBidirectedCircuit}, the linear matroid of the node-edge incidence matrix of $G$ is the signed-graphic matroid of $\Sigma$. Thus, any binet matrix based on the bidirected graph $G$ is the compact representation matrix (over $\mathbb{R}$) of the signed-graphic matroid of $\Sigma$, as it is obtained from the node-edge incidence matrix by $\mathbb{R}$-pivots.

\begin{thm}\label{thmMatroidBinetComp}
If $M(\Sigma)$ is  a signed-graphic matroid based on signed graph $\Sigma$, and $\vec \Sigma$ is obtained from $\Sigma$ by orienting the edges, then any binet representation based on 
$\vec \Sigma$ is a compact representation matrix of $M(\Sigma)$ over $\mathbb{R}$.
\end{thm}

We know from Lemma \ref{lemdefiSwi} that switchings in a bidirected graph do not alter its binet matrices. This phenomenon can be expressed in matroidal terms too. Notably, if two signed graphs are switching equivalent, then their signed-graphic matroids are the same.

As binet matrices generalize network matrices, the class of signed-graphic matroids contains all graphic matroids.
To be even more specific, if all cycles in a signed graph $\Sigma$ are positive, then $M(\Sigma)$ is the graphic matroid of $\Sigma$. At the other end of the scale, 
the signed-graphic matroid of an undirected graph, which is a signed graph where all the edges are negative, is the even-cycle matroid, employed by Doob \cite{DoobMat-73}.

It is shown in \cite{ZaslavskyBias-91} that minors of the signed-graphic matroid of $\Sigma$ correspond to the minors of $\Sigma$. This is equivalent to saying that deletions and contractions of edges in a signed graph correspond to deletions and contractions in its signed-graphic matroid.

\begin{thm}
The class of signed-graphic matroid is closed under minor-taking.
\end{thm}

It is a standard technique of matroid theory to find \emph{minimal violators} for a given property, i.e., matroids that do not have this property but all their minors do. Zaslavsky gave some minimal violators of signed-graphic matroids.

\begin{thm} (Zaslavsky \cite{Zaslavsky-Signed-82})
$U_4^2$ is a signed-graphic matroid. $U_5^2$ is not a signed-graphic matroid. 
\end{thm}

\begin{proof}
The matroid $U_4^2$ is signed-graphic since the compact representation matrix $A$ in (\ref{eqnBinetU42}) is a binet representation of the following bidirected graph, where every pair of edges corresponds to a basis of $U_4^2$.

\begin{figure}[h!]
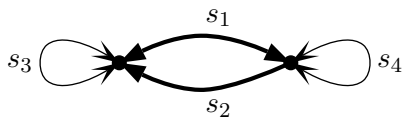

\psset{xunit=1.2cm,yunit=1.2cm,linewidth=0.5pt,radius=0.1mm,arrowsize=7pt,
labelsep=1.5pt,fillcolor=black}

\pspicture(-4,0)(5,2)

\pscircle[fillstyle=solid](1,1){.1}
\pscircle[fillstyle=solid](2.9,1){.1}

\pscurve[linewidth=1.6pt,arrowinset=0]{<->}(1,1)(1.9,1.3)(2.9,1)
\rput(2.1,1.5){$s_{1}$}

\pscurve[linewidth=1.6pt,arrowinset=0]{<-}(1,1)(1.9,.7)(2.9,1)
\rput(2.1,.5){$s_2$}

\pscurve[arrowinset=.5,arrowlength=1.5]{<->}(1,1)(.2,1.2)(.2,.8)(1,1)
\rput(-.1,1){$s_3$}

\pscurve[arrowinset=.5,arrowlength=1.5]{<->}(2.9,1)(3.7,1.2)(3.7,.8)(2.9,1)
\rput(4,1){$s_4$}

\endpspicture 
\caption{A binet representation of the matrix $A$ given in (\ref{eqnBinetU42}).} 
\label{fig:BinetU42}

\end{figure}

To show that $U_5^2$ is not signed-graphic, we eliminate all the cases. First, if  a signed graph $\Sigma$ has the signed-graphic matroid $U_5^2$, then it must have $5$ edges and any two of them form a basis. In other words, any subgraph with at least three edges is not independent, but all subgraphs with two edges are such. This rules out signed graphs on a single node, as they do not have two independent edges. Let us assume that there are no isolated nodes in $\Sigma$. If a signed graph has more than one component, then the union of independent subgraphs of the components is also independent. It follows that $\Sigma$ cannot have more than two components. If it has two components, then at least one of them has at least two edges. Taking two edges forming a basis from this component and one edge from the other component would result in an independent set wih three edges, a contradiction. 

Thus, $\Sigma$ is connected. There cannot be three parallel edges in $\Sigma$ because they would form a circuit, which is impossible in signed graphs by definition. If $\Sigma$ has more than three nodes, then there would be a tree with three edges in it, i.e., an independent set with more than two edges. For a similar reason, there cannot be a half-edge or a loop in $\Sigma$ if it has three nodes. If $\Sigma$ has three nodes, the only possible structure left is a triangle in which two edges are repeated so that the graph contains two cycles of length $2$.  Both cycles of length $2$ must be negative to form a basis, so adding an extra edge to either of them would form an independent set with more than two edges.

We are left with the case when $\Sigma$ has exactly two nodes. It can not have more than two of half-edges and loops because then two of them would be incident to the same node and they would not be independent. But then $\Sigma$ must contain three parallel links, which we already ruled out. We examined all the possible cases and they all led to a  contradiction, so $U_5^2$ is not signed-graphic.
\end{proof}\\

A very recent result due to Daniel Slilaty \cite{SlilatyCoSign-05} and independently by Hongxun Qin and Thomas A. Dowling  gives a technique to find minimal forbidden minors for signed-graphic matroids. They proved that a connected cographic matroid $M^*(G)$ (i.e., a matroid that is not a $1$-sum of two smaller matroids, and its dual is the graphic matroid  based on the undirected graph $G$) is signed-graphic, if and only if $G$ can be embedded into the projective plane (see also \cite{SlilatyCharImb-06}). Thus connected cographic matroids of minimally non-embeddable graphs are all minimal forbidden minors for signed-graphic matroids. This is an extension of  Theorem \ref{thmMatroidBeautiful}.

\begin{thm}\label{thmBinetSliP2}(Slilaty \cite{SlilatyCoSign-05})
A connected and cographic matroid $M^*(G)$ is signed-graphic  if and only if $G$ is a $2$-connected projective-planar graph (save perhaps for some isolated vertices).
\end{thm}

A reformulation  of Theorem \ref{thmBinetSliP2} in terms of matrices is as follows.

\begin{thm}
Let $G$ be a $2$-connected graph and $T$ a spanning tree. Let 
$\overrightarrow{G}$ be a digraph obtained from $G$ by orienting the edges, and
$A$ the network matrix with respect to $\overrightarrow{G}$ and $\overrightarrow{T}\subseteq \overrightarrow{G}$. Then $A^T$
is a binet matrix if and only if $G$ is embeddable in $\mathbb{P}^2$.
\end{thm}

Then an algorithm for recognizing binet matrices yields a way of 
testing if a given graph is embeddable in $\mathbb{P}^2$.
However, in the literature, there exist several methods for doing this that are much more efficient than the one deriving from our recognition algorithm for binet matrices (see \cite{MoharGraphSurface-01} ).

Theorem \ref{thmMatroidBinetComp} is not all that we can say about the representability of signed-graphic matroids. In fact, they are representable over any field where $2$ is not a zero-divisor, which in technical terms is refered to as a field the characterisitic of which is not equal to $2$.

\begin{thm}\label{thmMatroidZasGF2}(Zaslavsky \cite{Zaslavsky-Signed-82}) 
A signed-graphic matroid is representable over any field of characteristic not equal to $2$.
\end{thm}

Following an idea of Michele Conforti, one can prove the following.








\begin{thm}\label{thmMatroidConf}
A matroid is signed-graphic if and only if it has a compact representation matrix $A$ over $GF(3)$ so  that by applying the following operations on $A$, the resulting matrix is binet:

\begin{itemize}

\item[1)] Replacing some $entries$ equal to $1$ or $2$ by $-2$ and $-1$, respectively;

\item[2)] Multiplying  some columns in $A$ by $-\frac{1}{2}$ (working over $\mathbb{R}$).

\end{itemize}

\end{thm}

\begin{proof}
Let $\Sigma$ be a signed graph, and orient the edges of $\Sigma$ to get a bidirected graph $\vec \Sigma$. Let $A$ be a binet matrix based on $\vec \Sigma$ and $A'$ be obtained from $A$ by multiplying the columns with nonempty $\pm \frac{1}{2}$-support by $-2$ ($-2=1\mod 3$). Since a
column of a binet matrix with nonempty $\pm \frac{1}{2}$-support has no $\pm 2$-entry, it follows that $A'$ is a $\{0,\pm 1, \pm 2\}$-matrix. The signed-graphic matroid of $\Sigma$ is represented by the node-edge incidence matrix $In$ of $\vec \Sigma$. By Theorem \ref{thmBidirected2mod}, it follows that a set of column vectors in $In$ are linearly independent over $\mathbb{R}$ if and only if they are linearly independent over $GF(3)$. So, since $A$ is a compact representation matrix of $M(\Sigma)$ over $\mathbb{R}$, $A'$ is a compact representation matrix of $M(\Sigma)$ over $GF(3)$. This proves both parts of the theorem.  
\end{proof}\\

So it would be interesting to determine when $M(\Sigma)$ is representable over fields of characteristic two. It is shown in \cite{WhittleGF3-97} that if a matroid $M$ is representable over $GF(3)$, $\mathbb{Q}$, and a field of characteristic two, then $M$ is representable over all fields except maybe $GF(2)$. From this result and Theorem \ref{thmMatroidZasGF2}, it follows that a signed-graphic matroid can be one of the following three types:

\begin{itemize}

\item[(i)] regular,

\item[(ii)] representable over any field except $GF(2)$, or

\item[(iii)] representable over any field of characteristic other than $2$.

\end{itemize}

From this description of signed-graphic matroids, we deduce that $M(\Sigma)$ is binary if and only if $M(\Sigma)$ is regular. Moreover, provided that $M(\Sigma)$ is not regular, if $M(\Sigma)$ is quaternary
(i.e., representable over $GF(4)$), then it is of type (ii), otherwise of type (iii). See \cite{SlilatyDecomp-06} for some results about binary and quaternary signed-graphic matroids. The complete list of regular excluded minors for the class of signed-graphic matroids is given in \cite{Slilaty-07} by Hongxun Qin, Daniel C. Slilaty and Xiangqian Zhou. So it remains to find when $M(\Sigma)$  is quaternary.

We saw that graphic matroids constitute the main building blocks of the class of regular matroids. Moreover, we know that signed-graphic matroids generalize graphic matroids. A particularly natural generalisation of regular matroids is 
the class of near-regular matroids, that are the matroids representable over all fields except possibly $GF(2)$, while regular matroids are the ones representable over all fields. Whittle conjectures in \cite{WhittleMat-05} that there is a theorem for near-regular matroids similar to Theorem \ref{thmMatroidReg} that uses signed-graphic matroids and co-signed-graphic matroids as the basic terms in the decomposition.



\clearpage
\thispagestyle{empty}
\cleardoublepage
\verb"   "
\newpage

\chapter{Camion bases}\label{ch:Camion}

In what follows, we will consider a matrix as a set of column vectors. 
Let $M\in \mathbf{R}^{n\times m'}$ be a matrix of rank $n$, 
$\mathcal{H}(M)=\{\{x\in \mathbf{R}^n\,:\, c^Tx =0\}\,:
c\in M\}$ and  $B$ a basis of $M$.
$\mathcal{H}(M)$ splits up $\mathbf{R}^n$ into a set $S$ of full dimensional 
 cones (regions). A {\em simplex region} is one  which has exactly $n$ facets. 
$B$ is called a {\em Camion basis} if the 
corresponding hyperplanes determine the facets of a simplex region in $S$. 
It is known that there always exists a Camion basis. 
After some column permutations, we may write $M=[B\, N]$.  If $A$ is a matrix, we write
$A\geq 0$ if each entry of $A$ is nonnegative.
It is possible to 
show that $B$ is a Camion basis if and only if there exists a signing of some
columns of $M$ so that $B^{-1}N\geq 0$. Geometrically,  
$B^{-1}N\geq 0$ means that the column vectors of $N$ are contained in the
cone generated by $B$. 
We define $A=B^{-1}N$ and denote by $m=m'-n$ the number 
of columns of $A$.

Camion bases were first investigated by Camion \cite{Camionmod}, who proved that one always exists. 
The geometric counterpart was studied  by Shannon \cite{s-scah-79},
who gave the best lower bound of the number of simplex regions,
namely twice the number of distinct hyperplanes
in $\mathcal{H}(M)$.  Note that this lower bound does not translate 
immediately to a lower bound for the Camion bases, because one Camion basis 
might be associated with many (and at least two) simplex regions.

There is no known polynomial-time algorithm to find a Camion basis in general.
Fonlupt and  Raco \cite{FonRaco} described a finite procedure to find one  based
on the results of Camion. They also gave an algorithm which runs 
in time $O(n^3 \, m^2 )$ for totally
unimodular matrices.   

In this Chapter, we present a new characterization of
Camion bases, in the case where $M$ is the column set of the 
node-edge incidence matrix
(without one row) of a given connected digraph.  
Then, a general
characterization of Camion bases as well as a recognition 
procedure which runs in $O(n^2m')$ are given. Finally, an algorithm which finds a
Camion basis is presented. For totally
unimodular matrices, it is proven to run in time $O((nm)^2)$.

\section{Camion bases of digraphs}\label{sec:di}

Let $G=(V,E)$ be a connected digraph and $M$ the $V\times E$-incidence matrix of
$G$. Let $\tilde M$ be any submatrix of $M$ obtained by deleting one row. Let $B$ a basis of $\tilde M$, $T=(V,E_0)$ the corresponding spanning tree and 
suppose $\tilde M=[B\, N]$.  We call $T$ 
a {\em Camion tree} if $B$ is a Camion basis. The matrix $A=B^{-1} N$ is a network matrix.

Define an auxiliary graph $H$ associated with $T$
 as follows. The set of vertices is $E_0$ and for
$e,f\in E_0$, $e$ and $f$ are \emph{adjacent} if and only if they share a 
common end-point in $G$ and are in a same fundamental cycle. Then we have:

\begin{prop}\label{cambip}
A spanning tree $T$ is a Camion tree if and only if the auxiliary graph $H$ 
is bipartite.
\end{prop}

\begin{proof}
By observation (\ref{eqn:Introfund}), 
a spanning tree  $T$ is a camion tree if and only if there 
exists an orientation
of the edges of $G$ so that for each non-basic edge $g=uv$ of $G$, all basic
edges of the unique u-v-path in $T$ are forward edges. Such an orientation will
be called a \emph{proper} orientation. 

Suppose that there is a minimal odd cycle $\mathcal{C}$ in $H$. By minimality of
$\mathcal{C}$, the subgraph of $H$ induced by $\mathcal{C}$ 
has no chord. Using the definition of $H$, we deduce that the vertices of 
$\mathcal{C}$ determine a star in $T$. Denote the central vertex of the star by
$v_0$. Since $\mathcal{C}$ is an odd cycle, for each orientation of $G$, there
are two adjacent vertices in $\mathcal{C}$ corresponding to edges in $G$ that
are both either entering $v_0$ or leaving it.
Thus, there is no proper orientation of the edges of $T$.

Now suppose that $H$ is bipartite. We may suppose $H$ is 
 connected (otherwise what follows can
be applied to each connected component of $H$). 
Here is a little procedure that finds a proper orientation of $G$.

\begin{quote}
{\bf Procedure Orientation\_Propagation.} Choose any element $e_0 \in E_0$ and orient it in an arbitrary way in $G$. Let $F=\{e_0\}$. Then
successively choose an element $e\in E_0 \verb"\" F$ adjacent in $H$ to some $f\in F$; 
put $e$ in $F$ and orient it so that $\{e,f\}$ determines a directed path in $G$. 
\end{quote}

\noindent
Denote by $T(F)$ the subgraph of $G$ whose edge set is $F$. 
Let us prove by induction on the cardinality of $F$ that during the above procedure, $F$ 
always satisfies properties a) and b) below:

\begin{enumerate}
\item[a)] $T(F)$ is a tree.
\item[b)] For all $f_1,f_2 \in F$ such that $(f_1,f_2)$ is an edge in $H$, the reoriented edges
$f_1$ and $f_2$ determine a directed path in $G$.
\end{enumerate}

\noindent 
Let $k=|F|$.
If $k=1$, then clearly a) and b) are true.
Now suppose $k\geq 1$ and let $e\in E_0\verb"\" F$, $f\in F$ such that $(e,f)$ is an edge of $H$.
Using the definition of $H$ and the induction hypothesis, we have that 
$T( F\cup\{e\})$ is a tree in $G$. Let $N_F(e)=\{g\in F\,: \, (e,g)\in H\}$.
Since $e$ is a hanging edge of the tree $T(F\cup\{e\})$, from the definition of $H$ we deduce 
that $N_F(e)\cup \{e\}$ determines a star in $G$ with a central node, say $v_0$.

Let $f_1,f_2\in N_F(e)$. As $H$ is bipartite and the subgraph of $H$ induced by $F$ is connected
by construction,
there is a minimal path of even length in $H$ between $f_1$ and $f_2$. Since the path linking
$f_1$ to $f_2$ is of minimal length, all its vertices correspond to edges of the star. 
So $f_1$ and $f_2$ are both
entering into $v_0$ or leaving this node. It follows that the elements of $N_F(e)$ are all 
either entering into $v_0$ or leaving it. Thus we can orient $e$ in such a way that $F\cup \{e\}$
satisfies b).

Finally, some elements of $E_0$ might not be in $F$. But 
such edges are not in any fundamental cycle and their orientation can be arbitrarily 
chosen. The orientation of $F$ given by the above procedure induces a proper 
orientation of $G$.
\end{proof}

Hoffman and Kruskal \cite{HofKrus} and Heller and Hoffman \cite{HellerHof} 
gave a characterization of positive network 
matrices. A directed graph is called {\em alternating} if in each circuit the
edges are oriented alternately forwards and backwards. In Schrijver \cite{ShrAlex}
(p. 278-279), it is shown that a $\{0,1\} $-matrix is a network matrix if and only
if its columns are the incidence vectors of some directed paths in an
alternating digraph. To prove the necessary part of the condition, the
alternating digraph $G'=(E_0,F')$ is defined, where for edges $e,e'\in E_0$,
$(e,e')$ is in $F'$ if and only if the head of $e$ is the same as the tail of
$e'$. In proposition \ref{cambip}, $H$ is simply a subgraph of $G'$ considered 
as an unoriented graph.

\section{A polynomial-time recognition algorithm}\label{sec:rec}

We are going to see a characterization of Camion bases and a procedure 
that recognizes them in polynomial time. Let us remark that when one of 
$m'$ or $n$ is fixed, there is a trivial polynomial algorithm. We will present 
an algorithm to check for a given basis $B$ whether $B$ is a Camion basis in 
time $O(n^2m')$. 

Now suppose that $B$ (a basis), $N$ and $A=B^{-1}N$ are given.
Consider the bipartite digraph $G(A)$, whose vertex set is the index set of the rows
and columns of $A$ and $e=(i,j)$ is an edge of $G(A)$ if and only if 
$a_{ij}\neq 0$. Set a weight function $w$ on the set of edges:
\[w((i,j))=\left \{
\begin{array}{ll}
2 & \m{ if } sign(a_{ij})=+,\\
0 & \m{ if } sign(a_{ij})=-.
\end{array}\right.
\]
If $C$ is a cycle of $G(A)$, set $w(C)=\sum_{e\in C} w(e)$. 
 The weight function on the edges of the graph $G(A')$
will be denoted by $w'$.  
Then, we have the following characterization of Camion bases.

\begin{prop}\label{propweight}
A basis $B$ is a Camion basis if and only if each non-oriented cycle of 
$G(A)$ has a total weight equal to $0 \mod (4)$.
\end{prop}

\begin{lem}\label{zij}
Assume that $A'$ is obtained from $A$
by successive applications of signing operations. If $C$ is a cycle of
$G(A)$, $w'(C)\equiv w(C) \mod (4)$.\\

\begin{proof} We can assume that $A'$ is obtained from $A$ by application of a signing 
operation. Note that $G(A)=G(A')$. We can assume that $C$ is an elementary cycle of
$G(A)=G(A')$. But clearly $w'(C)=w(C)\pm 4$ or $w'(C)=w(C)$ and the result
follows.
\end{proof}\\
\end{lem}

\noindent
{\bf Proof of Proposition \ref{propweight}.}
First we prove the necessity. Let $B$ be a Camion basis. 
There exists a matrix $A'\geq 0$ which can be obtained
from $A$ by successive applications of signing operations. 
If $C$ is a cycle of $G(A')=G(A)$, $w'(C)=2|C|$ and $w'(C)\equiv 0\mod (4)$ since
$C$ is even. The result follows from the previous lemma.

To prove the sufficiency part, we can assume that $G(A)$ is
connected and we use the procedure below. In particular, the procedure
returns a cycle with a total weight equal to $2 \mod (4)$ if $B$ is not a Camion basis.
For a given spanning tree $T$ of $G(A)$ and an edge $e\in G(A)-T$, let $C_e$ denote the unique cycle in $T\cup \{e\}$.

\begin{tabbing}
\textbf{Procedure\,\, IS\_CAMION($A$)}\\
\textbf{Input:} A matrix $A\in \mathbf{R}^{n\times m}$, where 
$A=B^{-1}N$ such that $G(A)$ is connected. \\
\textbf{Output:}Yes (if $B$ is a Camion basis) with a matrix $A'\geq 0$ or No (otherwise) \\
with a certificate, that is a cycle  
in $G(A')$ of total weight equal to $ 2 \mod (4)$, \\
where $A'$ is $A$ after some signing operations.\\
1)\verb"  "\=  
Let $T$ be a subset of edges of $G(A)$ which induce a spanning tree of $G(A)$.\\
\> Assign to an initial node $v$ of $G(A)$ the label $l(v)=+1$, the other nodes are
unlabelled.\\ 
2) \> {\bf while} \= $\exists$ $(i,j)\in T$ with  $i$ (resp. $j$) labelled and  $j$ (resp. $i$) unlabelled {\bf do} \\ 
 \> \> label $j$ (resp. $i$) with the label $+1$
or $-1$ in such a way that 
$a_{ij}\cdot l(i) \cdot l(j)>0$.\\
3) \>Let $A'$ be the matrix obtained by multiplying each row $i$ of $A$ by its label
$l(i)$ \\ 
\> and each column $j$ of $A$ by its label $l(j)$. If $A'\geq 0$, output Yes and $A'$. \\
\> Otherwise, there exists $e=(i,j) \notin T$ such that $a_{ij}'<0$. Output No and $C_e$. \\
\end{tabbing}

\noindent
All the nodes of $G(A)$ are labelled by $+1$ or $-1$ at the end of this procedure.
Let $A'$ be the matrix constructed at step 3). Note that for all
$(i,j)\in T$, $a_{ij}'=a_{ij}\cdot l(i) \cdot l(j)>0$.\\ 
If $B$ is not a
Camion basis, there exists $e=(i,j) \notin T$ such that $a_{ij}'<0$ and 
$w'(e)=0$. We have $w'(C_e)\equiv 2\mod (4)$ since $w'(C_e)=2(|C_e|-1)$ and $C_e$ is even. But by the previous lemma we
have also $w(C_e)\equiv 2 \mod (4)$.\\ 
If $B$ is Camion, then for each edge $e=(i,j)\in G(A)-T$,  $w'(C_e)\equiv 0 \mod (4)$ as proved above and since $w'(C_e)=2(|C_e|-1)+w'(e)$, it follows that  $w'(e)=2$ and $a_{ij}'>0$.
\begin{rm}\par\smallskip \quad  \QED\end{rm}

\begin{cor}
There is a polynomial-time algorithm to check whether a basis $B$ is Camion.
\end{cor}

\begin{proof} 
To check whether a given basis $B$ is a Camion basis, we simply 
apply the procedure IS\_CAMION to the matrix $A=B^{-1} N$.
Since the complexity of calculating $B^{-1}N$ is $O(n^2m')$, computing $A$ dominates the procedure IS\_CAMION. Thus, the complexity of checking a given
basis $B$ is $O(n^2m')$.
\end{proof}

\section{Searching for a Camion basis}\label{sec:search}

Suppose that $B$, $N$ (and $A=B^{-1}N$) are given. 
A switching operation
corresponds to the replacement of a basic vector by a nonbasic one. 
Remark that a pivoting operation is described in general on the standard matrix $M=[I,A]$. For our purpose, it is convenient to use the "condensed" matrix by ignoring the identity matrix part.
Finding a Camion basis is equivalent to transforming  the matrix $A$ into a nonnegative one by
application of the following two operations:

\begin{itemize}
\item pivoting operations.
\item signing operations.
\end{itemize}

Signing a row of $A$ is equivalent to signing an associated basic vector.
A matrix $A'$ obtained from $A$ after some signing and pivoting operations will be
called \emph{equivalent} to $A$. 
Denote by $\epsilon(A)$ the set of matrices that are equivalent to $A$. 
Since the number of bases of $A$ and possibilities 
of signing some columns of $A$ is finite, the cardinality of $\epsilon(A)$ is 
finite. \\
Let us remark that for an
implementation of the algorithm that is described below, it is important to
maintain the correspondence between the rows of $A$ and the basic vectors and
between the columns of $A$ and the nonbasic ones. \\

We will see an algorithm, called Simp, that
runs in $O(\Delta^3(n\, m)^2 )$ if $M$ is an integral matrix, 
where $\Delta$ is the greatest
determinant (in absolute value) of a basis. 
So, for the particular case of 
totally unimodular matrices (where $\Delta=1$), Simp is
faster than the algorithm of Fonlupt and Raco \cite{FonRaco}. 
Moreover, the procedure Simp applied to real matrices is also finite.

For $i\in \{1,\ldots,n\}$ and $j\in \{1,\ldots,m\}$, let us denote by
$A_{i}$ the $i$th row of $A$ and  $A_{\bullet j}$ the $j$th column of $A$.
Moreover, $A^{(i,j)}$ denotes the matrix obtained from $A$ by pivoting on
$a_{i\,j}$. 
For any vector or matrix $V$, the function
$\su(V)$ evaluates the sum of all components of $V$.\\

Let us present the algorithm Simp.
\begin{tabbing} 
\textbf{Procedure:\,\,Simp($A$)}\\
\textbf{Input:} A real matrix $A$ of size $n\times m$.\\
\textbf{Output:} A nonnegative matrix $A'$ (obtained from $A$ by a
sequence of pivots and signings).\\
1)\verb"  " \= {\bf for} \=$i=1,\ldots, n$ {\bf do}\\
            \>  \> {\bf if}  $\su(A_{i})<0$, {\bf then}
 multiply the $i$th row by 
$-1$;\\
2) \> {\bf for} \, $j=1,\ldots ,m$  {\bf do}\\
             \>\> {\bf if} $A_{\bullet j}\le 0$, {\bf then} multiply 
the $j$th  column by $-1$; \\
3)\>  {\bf if}  $A\geq 0$ , {\bf return} $A$;\\
   \> {\bf otherwise}, let $k$, $k'$ and $l$  such that $a_{k\,l}>0$ 
   and  $a_{k'\,l}<0$; {\bf do}  \\
  \>\>  {\bf if} $\su(A_{\bullet l})< 1$, {\bf then} 
    pivot on $a_{k\,l}$;\\
  \>\>  {\bf otherwise}  pivot on $ a_{k'\,l}$;\\
  \>  {\bf go to} step 1;\\
\end{tabbing}

\begin{thm}\label{thmCamionSimp1}
For any real matrix $A$, the procedure Simp returns a nonnegative matrix 
after a finite number of steps. Furthermore if $M$ is integral, then
Simp runs in time $O(\Delta^3(n\, m)^2 )$.
\end{thm}

\begin{proof}
Suppose that $\su(A_{i})<0$ at step 1 or $A_{\bullet j}\le 0$ at step 2
for some $i$ or some $j$. 
Let $A'$ be the matrix obtained from $A$ by multiplying its $i$th row by $-1$ (at step 1) or its 
$j$th column by $-1$ (at step 2). Clearly, we have $\su(A')>\su(A)$.

At the beginning of step 3, $\su(A_{i})\geq 0$ for all $i$ and there is no
column $A_{\bullet j}$ such that $A_{\bullet j}\le 0$.
So, if there exists an entry $a_{k'\,l}<0$, then there exists $k$ such that 
$a_{k\,l}>0$. 
Denote by $\cc$ the $i$th row of $A$ without the $l$th column and
$\bb$ the $l$th column of $A$ without the $i$th row.  
We have that 

\[
A^{(i,l)}-A=\left( \begin{matrix} \frac{1}{\alpha}-\alpha &
\frac{\cc}{\alpha}-\cc\\ -\frac{\bb}{\alpha}-\bb  & -\frac{\bb \cc}{\alpha}
\end{matrix}\right)=\frac{1}{\alpha} \left( \begin{matrix} 1-\alpha\\-\bb
\end{matrix}\right)\cdot  \left( \begin{matrix} 1+\alpha & \cc
\end{matrix}\right).\]

\noindent
(Note that in the middle of this equation, the matrix is written as if $i=l=1$.) Thus,

\[
\su(A^{(i,l)})-\su(A)=\frac{1-\su(A_{\bullet l})}{a_{i\,l}}(\su(A_{i})+1) 
\m{ for } 1\le i\le n.
\]

\noindent
So, if $\su(A_{\bullet l})\neq 1$,  $\su(A)$ will increase  
by pivoting either on $a_{k\,l}$ or $a_{k'\,l}$.
If $\su(A_{\bullet l})=1$, then $\su(A)$ neither increases nor decreases. 
However, since $\sum_{j\neq l} a_{k'\,j}\geq 0$ and $a_{k'\,l}<0$, we deduce 
that $\su(A^{(k',l)}_{k'})=\frac{1}{a_{k'\,l}}(1+\sum_{j\neq l} a_{k'\,j})<0$. 
Therefore, $\su(A)$ will increase at the next step 1.

As $|\epsilon (A)|<\infty $, it follows that Simp generates 
a well-oriented matrix after a finite number of steps.

Now, suppose that $M$ is integral. By Cramer's rule,
for each matrix $A'$ equivalent to $A$, we have $a'_{i\,j}=\pm \frac{\det(B'')}
{\det(B')}$ $\forall i,j$, where $B',B''$ are two bases of $M$.
Since $\det(B'')\le \Delta $ and $|\det(B')|\geq 1 $, we deduce that $a_{ij}'\le
\Delta$ $\forall i,j$. Moreover, if $A'=(B')^{-1} N'$ where $B'$ is a basis 
and $N'$ a submatrix of $M$ after some signing operations of columns of $M$, 
then we may write each entry of $A'$ as a fraction of an
integer over $\det(B')$. 
Thus $\su(A')$, respectively $\su(A)$, 
is a ratio of some integer to $\det(B')$, respectively $\det(B)$.
So, if $A'$ is obtained from $A$ at some step
and $\su(A')-\su(A)>0$, then $\su(A')-\su(A)\geq \frac{1}{\Delta^2}$. Since
$\su(A)$ never decreases, but increases after at most two passages
at step 1 and is between $-\Delta nm$ and $\Delta nm$, the 
number of passages at step 1 is $O(\Delta^3\,n\,m)$.

Since the number of elementary operations at steps 1,2 and 3 is $O(n\,m)$, we conclude that the complexity of Simp is $O(\Delta^3(n\,m)^2)$.
\end{proof}\\

For dealing with matrices having entries in $\{0,\pm 1, \pm \frac{1}{2}, \pm 2\}$, for instance binet matrices, we easily deduce the following procedure.

\begin{tabbing} 
\textbf{Procedure:\,\,Camion($A$)}\\
\textbf{Input:} A real connected matrix $A$ of size $n\times m$
with no two identical columns.\\
\textbf{Output: }\=  Either a connected matrix $A'\in \{0, \frac{1}{2},1,2\}^{n\times m}$ 
such that $A'$ is binet and\\
\>  $m\le 4\left( \begin{array}{c}
n\\
2 
\end{array} \right) + 2n +1 $ if and only if the input matrix $A$ is binet, or determines\\
\>  that $A$ is not binet.\\

1) \verb"  "\= if $m > 4\left( \begin{array}{c}
n\\
2 
\end{array} \right) + 2n +1$, then STOP: output  that $A$ is not binet; \\

2) \> check the elements of $A$; if it has an element other than $0$, $\pm 1$, $\pm \frac{1}{2}$, or $\pm 2$
, then STOP: \\
\> output  that $A$ is not binet; \\
3)  \> {\bf for} \=$i=1,\ldots, n$ {\bf do}\\
            \>  \> {\bf if}  $\su(A_{i})<0$, {\bf then}
 multiply the $i$th row by 
$-1$;\\
4) \> {\bf for} \, $j=1,\ldots ,m$  {\bf do}\\
             \>\> {\bf if} $A_{\bullet j}\le 0$, {\bf then} multiply 
the $j$th  column by $-1$; \\
5)\>  {\bf if}  $A\geq 0$ , {\bf return} $A'=A$;\\
   \> {\bf otherwise}, let $k$, $k'$ and $l$  such that $a_{k\,l}>0$ 
   and  $a_{k'\,l}<0$; {\bf do}  \\
  \>\>  {\bf if} $\su(A_{\bullet l})< 1$, {\bf then} 
    pivot on $a_{k\,l}$;\\
  \>\>  {\bf otherwise}  pivot on $ a_{k'\,l}$;\\
  \>  {\bf go to} step 2;\\
\end{tabbing}

\begin{thm}\label{thmCamionSimp2}
The output of procedure Camion is correct. The running time of Camion is $O((n m)^2)$.
\end{thm}

\begin{proof}
If $A$ is binet, then by Lemma \ref{lemBinetColBound} $m\le 4\left( \begin{array}{c}
n\\
2 
\end{array} \right) + 2n +1 $, and Camion does not stop in step 1.
By Lemmas \ref{lemdefiRowCol}, \ref{lemBinetPivot} and \ref{lemBinetconnectedPivot}, the class of connected binet matrices is closed under pivoting and signing operations. Moreover, by Lemma \ref{lemdefiWeight1}, any entry of a binet matrix is equal to $0$, $\pm \frac{1}{2}$, $\pm 1$, or $\pm 2$.
So, the number of passages through step 5 is clearly bounded by
$3nm+1$, because of step 2 and $sum(A)$ increases either at step 3 or 5 (see the  proof of Theorem \ref{thmCamionSimp1}). This concludes the proof of Theorem \ref{thmCamionSimp2}.
\end{proof}\\


\clearpage
\thispagestyle{empty}
\cleardoublepage
\verb"   "
\newpage

\chapter{Recognizing binet matrices}\label{ch:Rec}

In this chapter, we turn to the problem of recognizing binet matrices. This problem is in the complexity class $\mathcal{N}\mathcal{P}$, because it is easy to verify that a matrix is binet, one only has to give a binet representation. Appa and Kotnyek formulated it as a mixed integer programming (MIP) problem (see \cite{KotThesis}). Unfortunately, their method is not polynomial. We present here a combinatorial polynomial-time algorithm, called Binet. Throughout this chapter, using Lemmas \ref{lemdefiRowCol} and \ref{lemdefiBlock}, we will assume that any input matrix 
for our recognition problem does not have two identical columns and is connected. We shall prove the following main result.

\begin{thm}\label{thmRecBinetMain}
A rational matrix of size $n\times m$ can be tested for being binet in time $O(n^6m)$ using the algorithm Binet.
\end{thm}

For network matrices, the parallel question is answered and discussed in Section \ref{sec:IntNet}.
The approach of Schrijver's method outlined on page \pageref{mycounter3} (see \cite{ShrAlex}) cannot be directly adapted to binet matrices because the basis of a binet graph, a forest of negative $1$-trees is more complex than a tree. Nevertheless, in the case where $A$ is an integral and cyclic matrix, the algorithm Binet can be viewed as a generalization of Schrijver's method.

Section \ref{sec:RecBinet} is devoted to the description of the algorithm Binet. In Section \ref{sec:decomppro}, we provide an important subroutine of the algorithm Binet called Decomposition and a related central theorem. Other subroutines are depicted and analyzed in the next chapters.

\section{The algorithm Binet}\label{sec:RecBinet}

The algorithm Binet can be schematized by a flow chart in Figure \ref{fig:RecBinetSketch}. It takes a rational (connected) matrix $A$ of size $n\times m$ as input and returns a binet representation $G(A)$ of $A$, or determines that none exists. 

Let us describe the algorithm and analyze its correctness.
The first task is to transform the matrix $A$ into a nonnegative one by using pivoting and signing operations, and this is performed by the procedure Camion described in Section \ref{sec:search}. By Theorem \ref{thmCamionSimp2}, Camion returns a connected  matrix $A'$ with entries $0$, $\frac{1}{2}$, $1$ or $2$ and check that $m\le 4\left( \begin{array}{c}
n\\
2 
\end{array} \right) + 2n +1 $, or determines that $A$ is not binet (see Lemma \ref{lemdefiWeight1}); moreover, $A'$ is binet if and only if $A$ is binet.
This first step is a key ingredient, at the center of fundamental properties in the design of the recognition algorithm.

\begin{figure}[t!]
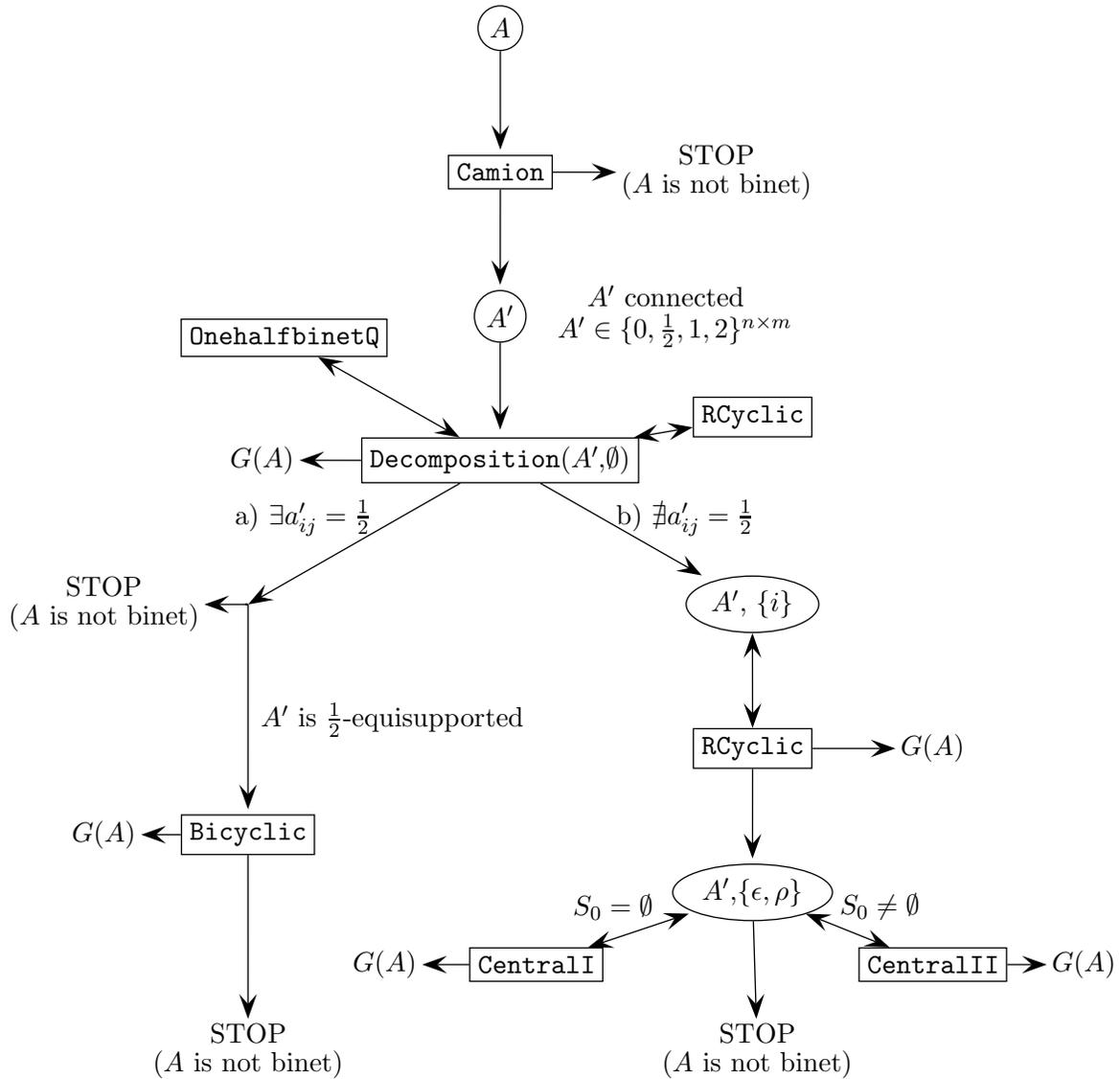

\begin{center}

\psset{xunit=1cm,yunit=1cm,linewidth=0.5pt,radius=0.1mm,arrowsize=7pt,
labelsep=1.5pt,fillcolor=white}

\pspicture(0.4,0)(16,10)

\rput(7.5,10){\rnode{A}{\pscirclebox{$A$}}
}


\rput(7.5,8){\rnode{Simp}{\psframebox{{\tt Camion}}}
}
\rput(10.5,8){\rnode{S1}{\shortstack{ STOP\\
($A$ is not binet)}}
}
\ncline[nodesepB=2pt]{->}{A}{Simp}
\ncline[nodesepB=2pt]{->}{Simp}{S1}

\rput(7.5,6){\rnode{A'}{\pscirclebox{$A'$}}
}
\rput(9.8,6){\shortstack{$A'$ connected\\
\,\,\,\,$A' \in\{0,\frac{1}{2},1,2\}^{n\times m}$}
}
\ncline[nodesepB=2pt]{->}{Simp}{A'}

\rput(7.5,4){\rnode{De}{\psframebox{{\tt Decomposition}($A'$,$\emptyset$)}}
}
\rput(4.2,4){\rnode{G5}{$G(A)$}
}
\rput(4.5,5.7){\rnode{OneQ}{\psframebox{{\tt OnehalfbinetQ}}}
}
\rput(11,4.6){\rnode{RCycDe}{\psframebox{{\tt RCyclic}}}
}
\ncline[nodesepB=3pt]{->}{A'}{De}
\ncline{<->}{De}{OneQ}
\ncline{<->}{De}{RCycDe}
\ncline[nodesepB=2pt]{->}{De}{G5}
\rput(10,3.2){b) $\nexists a'_{ij}=\frac{1}{2}
$}
\rput(4.7,3.2){a) $\exists a'_{ij}=\frac{1}{2}
$}

\rput(4,2){\rnode{cir}{\circle*{.01}}
}
\rput(2,2){\rnode{ST}{\shortstack{ STOP\\
($A$ is not binet)}}
}
\ncline{->}{De}{cir}
\ncline[nodesepB=2pt]{->}{cir}{ST}
\rput(6,.4){$A'$ is $\frac{1}{2}$-equisupported
}

\rput(4,-1.2){\rnode{Bi}{\psframebox{{\tt Bicyclic}}}
}
\rput(2,-1.2){\rnode{G2}{$G(A)$}
}
\ncline[nodesepB=2pt]{->}{cir}{Bi}
\ncline[nodesepB=2pt]{->}{Bi}{G2}
\rput(4,-4.2){\rnode{SC}{\shortstack{ STOP\\
($A$ is not binet)}}
}
\ncline[nodesepB=2pt]{->}{Bi}{SC}

\rput(11,2){\rnode{Ai}{\psovalbox{$A'$, $\{ i\}$}}
}
\ncline[nodesepB=0pt]{->}{De}{Ai}

\rput(11,0){\rnode{RCyc}{\psframebox{{\tt RCyclic}}}
}
\rput(13.5,0){\rnode{G0}{$G(A)$}
}
\rput(12.7,-2.2){$S_0\neq\emptyset$}
\rput(9,-2.2){$S_0= \emptyset$}
\ncline[nodesepB=2pt]{->}{RCyc}{G0}
\ncline{<->}{RCyc}{Ai}

\rput(11,-2){\rnode{Aepsrho}{\psovalbox{$A'$,$\{\epsilon,\rho\}$}}
}
\ncline[nodesepB=2pt]{->}{RCyc}{Aepsrho}

\rput(8,-3){\rnode{CI}{\psframebox{{\tt CentralI}}}
}
\rput(5.9,-3){\rnode{G3}{$G(A)$}
}
\rput(11,-4.2){\rnode{S2}{
\shortstack{ STOP\\
($A$ is not binet)} }}
\ncline[nodesepA=-3pt]{<->}{Aepsrho}{CI}
\ncline[nodesepB=2pt]{->}{CI}{G3}
\ncline[nodesepB=2pt]{->}{Aepsrho}{S2}

\rput(13.5,-3){\rnode{CII}{\psframebox{{\tt CentralII}}}
}
\rput(15.6,-3){\rnode{G4}{$G(A)$}
}

\ncline[nodesepA=-3pt]{<->}{Aepsrho}{CII}
\ncline[nodesepB=2pt]{->}{CII}{G4}

\endpspicture

\end{center}
\vspace{4cm}

\caption{A flow chart of the algorithm Binet.}
\label{fig:RecBinetSketch}

\end{figure}

Then one makes use of
the procedure Decomposition described in Subsection \ref{subsec:decomppro}. This procedure takes as input a matrix $A'$ and a row index subset of it denoted by $Q$; it contains subroutines called RCyclic and OnehalfbinetQ that are described in Sections \ref{sec:cyc} and \ref{sec:recOnehalfbinet}, respectively; and RCyclic is also a subroutine of OnehalfbinetQ.
The procedure Decomposition is applied to matrix $A'$ with $Q=\emptyset$.  If one does  not get a binet representation of $A'$, then by Theorem \ref{thmDecdecomp}, it follows that $A'$ is binet if and only if $A'$ is $\frac{1}{2}$-equisupported and bicyclic, or it is cyclic and without any $\frac{1}{2}$-entry. Thus two cases (a and b) are distinguished.

\begin{enumerate}

\item[a) ]  The matrix $A'$ has a $\frac{1}{2}$-entry.\\
If $A'$ is $\frac{1}{2}$-equisupported, then
the procedure Bicyclic described in Section \ref{sec:recbicyc} is performed on $A'$. Let us mention here that the procedure Bicyclic uses the procedure RCyclic as subroutine. 
In the case where $A'$ is not $\frac{1}{2}$-equisupported or
the procedure Bicyclic does not find any binet representation of $A'$, by Theorem \ref{thmbicyclicpro} we conclude that $A'$ is not binet, and so $A$ is not binet.

\item[b) ] The matrix $A'$ has no $\frac{1}{2}$-entry.\\
So $A'$ is binet if and only if it is cyclic, or equivalently $A'$ is $\{i\}$-cyclic for some row index $i$ or $\{\epsilon,\rho\}$-central for some pair $\{\epsilon,\rho\}$ of row indexes. Thus one computes whether $A'$ has an $\{i\}$-cyclic representation for some row index $i$ using the procedure RCyclic; 
if not, one checks whether $A'$ has an
$\{\epsilon,\rho\}$-central representation for some pair $\{\epsilon,\rho\}$ of row indexes; by letting  $S_0=\{ j \, : \, \epsilon,\rho \in s(A_{\bullet j}) \}$, this is done using the procedure CentralI, if $S_0=\emptyset$, or CentralII otherwise.
By Theorems \ref{thmcyclicproCyc}, \ref{thmcentralCarEmpty} and \ref{thmcentralCarNotEmpty}, one obtains a cyclic representation of $A'$ if and only if such a representation exists.
\end{enumerate}

In case a or b, whenever a binet representation of $A'$ has been found, one easily deduces a binet representation of $A$. Given a binet representation $G(A')$ of $A'$, one simply computes the inverse of all operations performed in step 1, and applies the corresponding operations on $G(A')$, in reverse order.
Mathematically, the algorithm Binet is stated as follows.

\begin{tabbing}
\textbf{Algorithm\,\,Binet(A)}\\

\textbf{Input:} A (connected) matrix $A$ of size $n \times m$.\\
\textbf{Output:} Either a binet representation $G(A)$ of $A$, or determines that none exists.\\

1)\verb"  "\= let $A'=A$, call { \tt Camion($A'$)} of Section \ref{sec:search}  by keeping in memory \\
\>  all pivoting and signing operations; \\

2)  \> call {\tt Decomposition($A'$,$Q=\emptyset$)} of Section \ref{sec:decomppro}; if we obtain a binet representation\\
\> $G(A')$  of $A'$, then go to 6; if $A'$ has no $\frac{1}{2}$-entry, then go to 4;\\

3) \>  if $A'$ is $\frac{1}{2}$-equisupported, then call {\tt Bicyclic($A'$)} of Section \ref{sec:recbicyc}; go to 6;\\

4) \> {\bf for } \= every row index $i$, {\bf do} \\
\> \> call {\tt RCyclic($A'$,$\{i\}$)} of Section \ref{sec:recOnehalfbinet}; \\
\>\> if we obtain an $\{i\}$-cyclic representation $G(A')$ of $A'$, then go to 6;\\
 \> {\bf endfor } \\

5) \> {\bf for } \= every pair $\{\epsilon,\rho\}$ of row indexes, 
{\bf do} \\
\> \> let $S_0=\{ j \, : \, \epsilon,\rho \in s(A_{\bullet j}) \}$;  \\
\> \>{\bf if} \= $S_0= \emptyset$ {\bf then} \\
\> \>\> call {\tt CentralI($A'$,$\{\epsilon,\rho\}$)} of Section \ref{sec:S0empty}; \\
\> \>{\bf otherwise}\\
\> \>\> call {\tt CentralII($A'$,$\{\epsilon,\rho\}$)} of Section \ref{sec:S0notempty}; \\
\> \>{\bf endif}\\
\>  \> if we have an $\{\epsilon,\rho\}$-central representation $G(A')$ of $A'$, then
go to 6; \\
\> {\bf endfor}\\

6) \> if we have a binet representation $G(A')$ of $A'$, then  compute in reverse order \\
\> the inverse of all operations performed in step 1  on $G(A')$ and output the resulting \\
\> binet representation $G(A)$ of $A$, otherwise output that $A$ is not binet; \\

\end{tabbing}

We mention here that the procedures RCyclic, CentralI and  CentralII make use of a digraph, called $D$, and described in chapter \ref{ch:multidiD}. If $A$ has an $R^*$-cyclic, $R^*$-central or $R^*$-network representation $G(A)$, where $R^*$ is a row index subset of $A$, then the basic forest in $G(A)$ with edge index set $\overline{R^*}$ has some nice properties that are recognizable thanks to the digraph $D$. A subroutine of the procedures RCyclic and CentralI called Forest and given in Section \ref{sec:For} computes some feasible spanning forest of $D$.

Using the results of the next chapters, we prove here Theorem \ref{thmRecBinetMain}.\\

\noindent
{\bf Proof of Theorem \ref{thmRecBinetMain}.}
The correctness of the algorithm has been proved earlier. Let us analyze its running time. By Theorem \ref{thmCamionSimp2}, the procedure Camion takes time $O(n^2 m^2)$.
Let $\alpha'$ be the number of nonzero entries of the matrix $A'$  in step 2.
By Theorem  \ref{thmDecdecomp}, the procedure  Decomposition takes time $O(n m^2 \alpha')$. By Theorem \ref{thmbicyclicpro}, the procedure Bicyclic works in time $O(nm \alpha')$.
Moreover, by Theorem \ref{thmcyclicproCyc},
the procedure RCyclic takes time $O(nm \alpha')$, and the number of passages through step 4 does not exceed $n$.
Finally, by Theorems \ref{thmcentralCarEmpty} and \ref{thmcentralCarNotEmpty}, the procedures CentralI and CentralII work in time $O(n^3\alpha')$, and the number of passages through step 5 does not exceed $n^2$. If $m> 4\left( \begin{array}{c}
n\\
2 
\end{array} \right) + 2n +1 $, then in step 1 the subroutine Camion directly outputs that $A$ is not binet. Otherwise, using previous arguments, we deduce that the running time of the algorithm Binet is $O(\max(n^2m^2, n m^2 \alpha', n^2 m \alpha',n^5 \alpha'))$. Hence, since $\alpha'\le nm$, the computational effort of the algorithm Binet
is $O(n^6 m)$.
{\hfill$\BBox{\rule{.3mm}{3mm}}$} \\

\section{The procedure Decomposition}\label{sec:decomppro}

Suppose we are given a connected matrix $A$ in $\{0,\frac{1}{2},
1, 2\}^{n\times  m}$. Let $\alpha$ be the number of nonzero entries in $A$.
The goal of the present section is to decompose the matrix $A$ into smaller pieces by reducing the binet recognition problem to the simpler case of $R$-cyclic matrices and  $\frac{1}{2}$-binet matrices. For that purpose, we describe a procedure called Decomposition which takes the matrix $A$ and a row index subset $Q$ of $A$ as input. Under certain conditions, the procedure computes a binet representation of $A$ such that each basic edge with index in $Q$ is a half-edge; the name of the procedure comes from the fact that it "decomposes" the matrix $A$ into two matrices (if $A$ has a $\frac{1}{2}$-entry), and then works iteratively on one of both matrices. A proof of the following theorem is given.

\begin{thm} \label{thmDecdecomp}
The matrix $A$ is binet if and only if one of the following three statements is valid:

\begin{itemize}

\item[1)] $A$  is bicyclic and $\frac{1}{2}$-equisupported, or

\item[2)] $A$ is cyclic and without any $\frac{1}{2}$-entry, or

\item[3)] the procedure Decomposition with input $A$ and row index subset $Q=\emptyset$ provides a binet representation of $A$. 

\end{itemize}
Moreover, the running time of the procedure Decomposition is $O(nm^2 \alpha)$.
\end{thm}

Before stating the procedure Decomposition and proving Theorem 
\ref{thmDecdecomp}, we provide in Subsection \ref{subsec:infdec} 
some intuitions and graphical ideas on which these are based in an informal way. Then, in Subsection \ref{subsec:decomppro}, the procedure Decomposition and a formal proof of Theorem \ref{thmDecdecomp} are given.

\subsection{An informal sketch of the procedure}\label{subsec:infdec}

Let us consider the binet matrix $A$ given in Figure \ref{fig:exampleA}. Our aim is to construct a binet representation of $A$, for instance the one given in Figure \ref{fig:exampleA}, without knowing that such a representation exists. We describe some steps of the procedure Decomposition applied on $A$.

\begin{figure}[ht!]
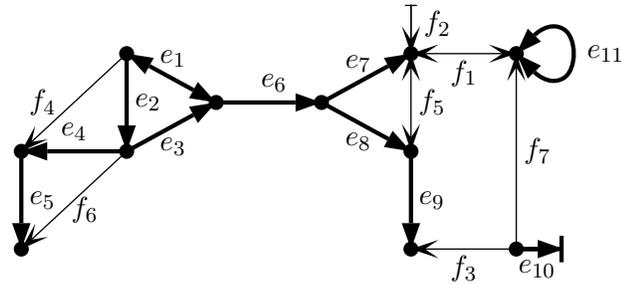


$
\begin{array}{cl}

\begin{tabular}{c|c|c|c|c|c|c|c|}
  & $f_1$ & $f_2$ & $f_3$ & $f_4$ & $f_5$ & $f_6$ & $f_7$  \\
  \hline
$e_1$  &$\frac{1}{2}$&$\frac{1}{2}$&$\frac{1}{2}$&0&1&0&0 \\
\hline
$e_2$  &$\frac{1}{2}$&$\frac{1}{2}$&$\frac{1}{2}$&1&1&0&0 \\
\hline
$e_3$  &$\frac{1}{2}$&$\frac{1}{2}$&$\frac{1}{2}$&0&1&0&0 \\
\hline
$e_4$  &0&0&0&1&0&1&0 \\
\hline
$e_5$  &0&0&0&0&0&1& 0\\
\hline
$e_6$  &1&1&1&0&2&0&0 \\
\hline
$e_7$  &1&1&0&0&1&0&0 \\
\hline
$e_8$  &0&0&1&0&1&0&0 \\
\hline
$e_{9}$  &0&0&1&0&0&0&0 \\
\hline
$e_{10}$ &0&0&1&0&0&0&1 \\
\hline  
$e_{11}$  &$\frac{1}{2}$&0&0&0&0&0&$\frac{1}{2}$ \\
\hline
\end{tabular}  &

\psset{xunit=1.4cm,yunit=1.3cm,linewidth=0.5pt,radius=0.1mm,arrowsize=7pt,
labelsep=1.5pt,fillcolor=black}

\pspicture(-1.5,0)(5,2.5)

\pscircle[fillstyle=solid](-1,0){.1}
\pscircle[fillstyle=solid](-1,1){.1}
\pscircle[fillstyle=solid](0,1){.1}
\pscircle[fillstyle=solid](0,2){.1}
\pscircle[fillstyle=solid](0.85,1.5){.1}
\pscircle[fillstyle=solid](1.85,1.5){.1}
\pscircle[fillstyle=solid](2.7,0){.1}
\pscircle[fillstyle=solid](2.7,1){.1}
\pscircle[fillstyle=solid](2.7,2){.1}
\pscircle[fillstyle=solid](3.7,0){.1}
\pscircle[fillstyle=solid](3.7,2){.1}

\psline[linewidth=1.6pt,arrowinset=0]{<->}(0,2)(0.85,1.5)
\rput(0.44,1.95){$e_1$}

\psline[linewidth=1.6pt,arrowinset=0]{<-}(0,1)(0,2)
\rput(0.2,1.5){$e_2$}

\psline[linewidth=1.6pt,arrowinset=0]{->}(0,1)(0.85,1.5)
\rput(0.44,1.05){$e_3$}

\psline[linewidth=1.6pt,arrowinset=0]{<-}(-1,1)(0,1)
\rput(-0.5,1.2){$e_4$}

\psline[linewidth=1.6pt,arrowinset=0]{<-}(-1,0)(-1,1)
\rput(-0.8,0.5){$e_5$}

\psline[linewidth=1.6pt,arrowinset=0]{->}(0.85,1.5)(1.85,1.5)
\rput(1.4,1.7){$e_6$}

\psline[linewidth=1.6pt,arrowinset=0]{->}(1.85,1.5)(2.7,2)
\rput(2.2,1.9){$e_7$}

\psline[linewidth=1.6pt,arrowinset=0]{->}(1.85,1.5)(2.7,1)
\rput(2.2,1.1){$e_8$}

\psline[linewidth=1.6pt,arrowinset=0]{->}(2.7,1)(2.7,0)
\rput(2.9,0.5){$e_9$}

\psline[linewidth=1.6pt,arrowinset=0]{-|}(3.7,0)(4.15,0)
\psline[linewidth=1.6pt,arrowinset=0]{->}(3.7,0)(4.15,0)
\rput(3.9,-0.2){$e_{10}$}

\pscurve[linewidth=1.6pt,arrowinset=0]{<->}(3.7,2)(4.15,2.25)(4.15,1.75)(3.7,2)
\rput(4.55,2){$e_{11}$}


\psline[arrowinset=.5,arrowlength=1.5]{<->}(2.7,2)(3.7,2)
\rput(3.2,1.8){$f_1$}

\psline[arrowinset=.5,arrowlength=1.5]{-|}(2.7,2)(2.7,2.5)
\psline[arrowinset=.5,arrowlength=1.5]{<-}(2.7,2)(2.7,2.5)
\rput(2.95,2.3){$f_{2}$}

\psline[arrowinset=.5,arrowlength=1.5]{<-}(2.7,0)(3.7,0)
\rput(3.2,-.2){$f_3$}

\psline[arrowinset=.5,arrowlength=1.5]{<-}(-1,1)(0,2)
\rput(-0.8,1.5){$f_4$}

\psline[arrowinset=.5,arrowlength=1.5]{<-}(-1,0)(0,1)
\rput(-0.4,0.4){$f_6$}

\psline[arrowinset=.5,arrowlength=1.5]{<->}(2.7,1)(2.7,2)
\rput(2.9,1.5){$f_5$}

\psline[arrowinset=.5,arrowlength=1.5]{->}(3.7,0)(3.7,2)
\rput(3.9,1){$f_7$}

\endpspicture

\end{array}$
\vspace{.1cm}

\caption{A binet matrix $A$ with a binet representation of $A$.}
\label{fig:exampleA}

\end{figure}

Provided that there exists a binet representation of $A$, one arising question is whether one can locate the edge index set of a basic full cycle. We observe that the $\frac{1}{2}$-support of the first and second column intersect and are not equal. Using this simple property, it will be proved that $s_{\frac{1}{2}}(A_{\bullet 1}) \cap s_{\frac{1}{2}}(A_{\bullet 2})$ is the edge index set of a full basic cycle in any binet representation of $A$, if one exists. Can we derive a general statement? 

Once some row index subset $R$ of $A$, for instance $R=\{1,2,3\}=s_{\frac{1}{2}}(A_{\bullet 1}) \cap s_{\frac{1}{2}}(A_{\bullet 2})$ has been located, we consider any column $A_{\bullet j}$, namely $j=4$ and $j=5$, whose support intersects $R$ and such that $R \nsubseteq s_{\frac{1}{2}}(A_{\bullet j})$. Provided that $A$ has a binet representation $G(A)$, it will be proved that the nonbasic edges $f_4$ and  $f_5$ correspond to $1$-edges in $G(A)$ whose fundamental circuit intersects the full cycle with edge index set $R$; then, since  $s(A_{\bullet 6}) \cap s(A_{\bullet 4}) \neq \emptyset$ and $R \nsubseteq s_{\frac{1}{2}}(A_{\bullet 6})$, 
\begin{figure}[h!]
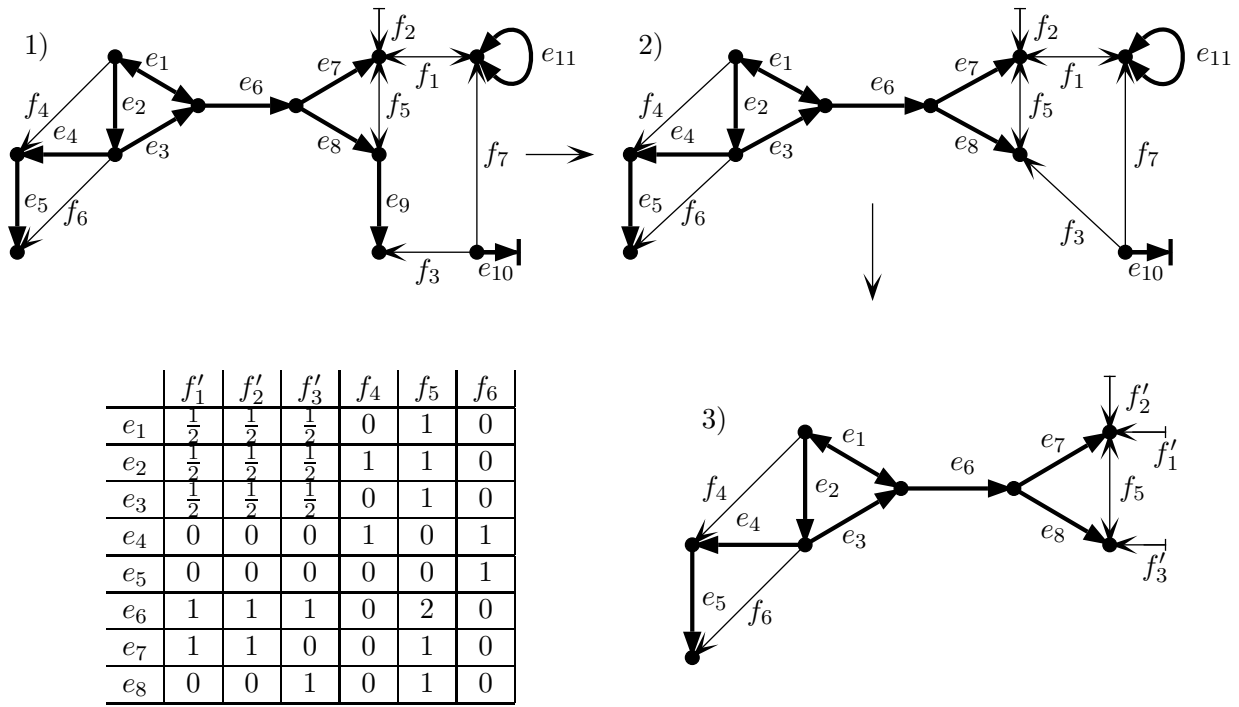


$
\begin{array}{cc}

\psset{xunit=1.3cm,yunit=1.3cm,linewidth=0.5pt,radius=0.1mm,arrowsize=7pt,
labelsep=1.5pt,fillcolor=black}

\pspicture(-1,0)(5,2.5)

\pscircle[fillstyle=solid](-1,0){.1}
\pscircle[fillstyle=solid](-1,1){.1}
\pscircle[fillstyle=solid](0,1){.1}
\pscircle[fillstyle=solid](0,2){.1}
\pscircle[fillstyle=solid](0.85,1.5){.1}
\pscircle[fillstyle=solid](1.85,1.5){.1}
\pscircle[fillstyle=solid](2.7,0){.1}
\pscircle[fillstyle=solid](2.7,1){.1}
\pscircle[fillstyle=solid](2.7,2){.1}
\pscircle[fillstyle=solid](3.7,0){.1}
\pscircle[fillstyle=solid](3.7,2){.1}

\rput(-.8,2.1){1)}

\psline[linewidth=1.6pt,arrowinset=0]{<->}(0,2)(0.85,1.5)
\rput(0.44,1.95){$e_1$}

\psline[linewidth=1.6pt,arrowinset=0]{<-}(0,1)(0,2)
\rput(0.2,1.5){$e_2$}

\psline[linewidth=1.6pt,arrowinset=0]{->}(0,1)(0.85,1.5)
\rput(0.44,1.05){$e_3$}

\psline[linewidth=1.6pt,arrowinset=0]{<-}(-1,1)(0,1)
\rput(-0.5,1.2){$e_4$}

\psline[linewidth=1.6pt,arrowinset=0]{<-}(-1,0)(-1,1)
\rput(-0.8,0.5){$e_5$}

\psline[linewidth=1.6pt,arrowinset=0]{->}(0.85,1.5)(1.85,1.5)
\rput(1.4,1.7){$e_6$}

\psline[linewidth=1.6pt,arrowinset=0]{->}(1.85,1.5)(2.7,2)
\rput(2.2,1.9){$e_7$}

\psline[linewidth=1.6pt,arrowinset=0]{->}(1.85,1.5)(2.7,1)
\rput(2.2,1.1){$e_8$}

\psline[linewidth=1.6pt,arrowinset=0]{->}(2.7,1)(2.7,0)
\rput(2.9,0.5){$e_9$}

\psline[linewidth=1.6pt,arrowinset=0]{-|}(3.7,0)(4.15,0)
\psline[linewidth=1.6pt,arrowinset=0]{->}(3.7,0)(4.15,0)
\rput(3.9,-0.2){$e_{10}$}

\pscurve[linewidth=1.6pt,arrowinset=0]{<->}(3.7,2)(4.15,2.25)(4.15,1.75)(3.7,2)
\rput(4.55,2){$e_{11}$}


\psline[arrowinset=.5,arrowlength=1.5]{<->}(2.7,2)(3.7,2)
\rput(3.2,1.8){$f_1$}

\psline[arrowinset=.5,arrowlength=1.5]{-|}(2.7,2)(2.7,2.5)
\psline[arrowinset=.5,arrowlength=1.5]{<-}(2.7,2)(2.7,2.5)
\rput(2.95,2.3){$f_{2}$}

\psline[arrowinset=.5,arrowlength=1.5]{<-}(2.7,0)(3.7,0)
\rput(3.2,-.2){$f_3$}

\psline[arrowinset=.5,arrowlength=1.5]{<-}(-1,1)(0,2)
\rput(-0.8,1.5){$f_4$}

\psline[arrowinset=.5,arrowlength=1.5]{<-}(-1,0)(0,1)
\rput(-0.4,0.4){$f_6$}

\psline[arrowinset=.5,arrowlength=1.5]{<->}(2.7,1)(2.7,2)
\rput(2.9,1.5){$f_5$}

\psline[arrowinset=.5,arrowlength=1.5]{->}(3.7,0)(3.7,2)
\rput(3.9,1){$f_7$}

\psline[arrowinset=.5,arrowlength=1.5]{->}(4.2,1)(4.9,1)

\endpspicture &

\psset{xunit=1.4cm,yunit=1.3cm,linewidth=0.5pt,radius=0.1mm,arrowsize=7pt,
labelsep=1.5pt,fillcolor=black}

\pspicture(-1,0)(5,2.5)

\pscircle[fillstyle=solid](-1,0){.1}
\pscircle[fillstyle=solid](-1,1){.1}
\pscircle[fillstyle=solid](0,1){.1}
\pscircle[fillstyle=solid](0,2){.1}
\pscircle[fillstyle=solid](0.85,1.5){.1}
\pscircle[fillstyle=solid](1.85,1.5){.1}
\pscircle[fillstyle=solid](2.7,1){.1}
\pscircle[fillstyle=solid](2.7,2){.1}
\pscircle[fillstyle=solid](3.7,0){.1}
\pscircle[fillstyle=solid](3.7,2){.1}

\rput(-.8,2.1){2)}

\psline[linewidth=1.6pt,arrowinset=0]{<->}(0,2)(0.85,1.5)
\rput(0.44,1.95){$e_1$}

\psline[linewidth=1.6pt,arrowinset=0]{<-}(0,1)(0,2)
\rput(0.2,1.5){$e_2$}

\psline[linewidth=1.6pt,arrowinset=0]{->}(0,1)(0.85,1.5)
\rput(0.44,1.05){$e_3$}

\psline[linewidth=1.6pt,arrowinset=0]{<-}(-1,1)(0,1)
\rput(-0.5,1.2){$e_4$}

\psline[linewidth=1.6pt,arrowinset=0]{<-}(-1,0)(-1,1)
\rput(-0.8,0.5){$e_5$}

\psline[linewidth=1.6pt,arrowinset=0]{->}(0.85,1.5)(1.85,1.5)
\rput(1.4,1.7){$e_6$}

\psline[linewidth=1.6pt,arrowinset=0]{->}(1.85,1.5)(2.7,2)
\rput(2.2,1.9){$e_7$}

\psline[linewidth=1.6pt,arrowinset=0]{->}(1.85,1.5)(2.7,1)
\rput(2.2,1.1){$e_8$}


\psline[linewidth=1.6pt,arrowinset=0]{-|}(3.7,0)(4.15,0)
\psline[linewidth=1.6pt,arrowinset=0]{->}(3.7,0)(4.15,0)
\rput(3.9,-0.2){$e_{10}$}

\pscurve[linewidth=1.6pt,arrowinset=0]{<->}(3.7,2)(4.15,2.25)(4.15,1.75)(3.7,2)
\rput(4.55,2){$e_{11}$}


\psline[arrowinset=.5,arrowlength=1.5]{<->}(2.7,2)(3.7,2)
\rput(3.2,1.8){$f_1$}

\psline[arrowinset=.5,arrowlength=1.5]{-|}(2.7,2)(2.7,2.5)
\psline[arrowinset=.5,arrowlength=1.5]{<-}(2.7,2)(2.7,2.5)
\rput(2.95,2.3){$f_{2}$}

\psline[arrowinset=.5,arrowlength=1.5]{<-}(2.7,1)(3.7,0)
\rput(3.2,.2){$f_3$}

\psline[arrowinset=.5,arrowlength=1.5]{<-}(-1,1)(0,2)
\rput(-0.8,1.5){$f_4$}

\psline[arrowinset=.5,arrowlength=1.5]{<-}(-1,0)(0,1)
\rput(-0.4,0.4){$f_6$}

\psline[arrowinset=.5,arrowlength=1.5]{<->}(2.7,1)(2.7,2)
\rput(2.9,1.5){$f_5$}

\psline[arrowinset=.5,arrowlength=1.5]{->}(3.7,0)(3.7,2)
\rput(3.9,1){$f_7$}

\psline[arrowinset=.5,arrowlength=1.5]{->}(1.3,.5)(1.3,-.5)

\endpspicture \\

\begin{tabular}{c|c|c|c|c|c|c|}
  & $f_1'$ & $f_2'$ & $f_3'$ & $f_4$ & $f_5$ & $f_6$  \\
  \hline
$e_1$  &$\frac{1}{2}$&$\frac{1}{2}$&$\frac{1}{2}$&0&1&0\\
\hline
$e_2$  &$\frac{1}{2}$&$\frac{1}{2}$&$\frac{1}{2}$&1&1&0 \\
\hline
$e_3$  &$\frac{1}{2}$&$\frac{1}{2}$&$\frac{1}{2}$&0&1&0 \\
\hline
$e_4$  &0&0&0&1&0&1 \\
\hline
$e_5$  &0&0&0&0&0&1 \\
\hline
$e_6$  &1&1&1&0&2&0\\
\hline
$e_7$  &1&1&0&0&1&0\\
\hline
$e_8$  &0&0&1&0&1&0 \\
\hline
\end{tabular} &

\psset{xunit=1.5cm,yunit=1.5cm,linewidth=0.5pt,radius=0.1mm,arrowsize=7pt,
labelsep=1.5pt,fillcolor=black}

\pspicture(-1.5,1)(4,3.5)

\pscircle[fillstyle=solid](-1,0){.1}
\pscircle[fillstyle=solid](-1,1){.1}
\pscircle[fillstyle=solid](0,1){.1}
\pscircle[fillstyle=solid](0,2){.1}
\pscircle[fillstyle=solid](0.85,1.5){.1}
\pscircle[fillstyle=solid](1.85,1.5){.1}
\pscircle[fillstyle=solid](2.7,1){.1}
\pscircle[fillstyle=solid](2.7,2){.1}

\rput(-.8,2.1){3)}

\psline[linewidth=1.6pt,arrowinset=0]{<->}(0,2)(0.85,1.5)
\rput(0.44,1.95){$e_1$}

\psline[linewidth=1.6pt,arrowinset=0]{<-}(0,1)(0,2)
\rput(0.2,1.5){$e_2$}

\psline[linewidth=1.6pt,arrowinset=0]{->}(0,1)(0.85,1.5)
\rput(0.44,1.05){$e_3$}

\psline[linewidth=1.6pt,arrowinset=0]{<-}(-1,1)(0,1)
\rput(-0.5,1.2){$e_4$}


\psline[linewidth=1.6pt,arrowinset=0]{<-}(-1,0)(-1,1)
\rput(-0.8,0.5){$e_5$}

\psline[linewidth=1.6pt,arrowinset=0]{->}(0.85,1.5)(1.85,1.5)
\rput(1.4,1.7){$e_6$}

\psline[linewidth=1.6pt,arrowinset=0]{->}(1.85,1.5)(2.7,2)
\rput(2.2,1.9){$e_7$}

\psline[linewidth=1.6pt,arrowinset=0]{->}(1.85,1.5)(2.7,1)
\rput(2.2,1.1){$e_8$}

\psline[arrowinset=.5,arrowlength=1.5]{<-}(2.7,2)(3.2,2)
\psline[arrowinset=.5,arrowlength=1.5]{-|}(2.9,2)(3.2,2)
\rput(3.2,1.8){$f_1'$}

\psline[arrowinset=.5,arrowlength=1.5]{-|}(2.7,2)(2.7,2.5)
\psline[arrowinset=.5,arrowlength=1.5]{<-}(2.7,2)(2.7,2.5)
\rput(2.95,2.3){$f_{2}'$}

\psline[arrowinset=.5,arrowlength=1.5]{<-}(2.7,1)(3.2,1)
\psline[arrowinset=.5,arrowlength=1.5]{-|}(3,1)(3.2,1)
\rput(3.1,0.8){$f_3'$}


\psline[arrowinset=.5,arrowlength=1.5]{<-}(-1,1)(0,2)
\rput(-0.8,1.5){$f_4$}

\psline[arrowinset=.5,arrowlength=1.5]{<-}(-1,0)(0,1)
\rput(-0.4,0.4){$f_6$}

\psline[arrowinset=.5,arrowlength=1.5]{<->}(2.7,1)(2.7,2)
\rput(2.9,1.5){$f_5$}


\endpspicture

\end{array}
$\\
\vspace{.1cm}

\caption{the matrix $A(R)$ obtained from $A$ given in Figure \ref{fig:exampleA} by the procedure Decomposition, and a $\{1,2,3\}$-cyclic representation of $A(R)$ on the right of $A(R)$ obtained from $G(A)$ (from 1) to 2) edge $e_9$ is contracted, and from 2) to 3) edges $e_{10}$ and $e_{11}$ are contracted).}
\label{fig:A1}
\end{figure}
one can show that $f_6$ is also a $1$-edge and $R' = R \cup s(A_{\bullet 4}) \cup s(A_{\bullet 5}) \cup s(A_{\bullet 6})$ is 
the edge index set of a $1$-tree in $G(A)$. The procedure Decomposition builds up the matrix $A(R)=A_{R' \times \{1,\ldots,6 \}}$. Given a binet representation $G(A)$ of $A$, if one exists, one can construct an $R$-cyclic representation
$G(A(R))$ of $A(R)$, by contracting all basic edges with index in $\overline{R'}$ and deleting the remaining loose edges
as illustrated in  Figure \ref{fig:A1}.

We may write $A= \left[ \begin{array}{c}
A(R) \,\, O_{8\times 1} \\
A'
\end{array}\right]
$, where $A'$ is a submatrix of $A$. In general, it is not sufficient to have binet representations $G(A(R))$ and $G(A')$ of the matrices $A(R)$ and $A'$, respectively, to compute a binet representation of $A$. The procedure Decomposition constructs a matrix $A(\tau)$ whose $A'$ (without zero columns) 
and $\tau$ are submatrices.

Suppose that $A$ has a binet representation $G(A)$.
Let us look at the graphical interpretation of $A(\tau)$, without giving an explicit definition of this matrix. Delete in $G(A)$ the edges in $G(A(R))$, and the remaining isolated nodes. We are left with a connected bidirected graph containing some nodes of $G(A(R))$. By adding a basic half-edge entering each left node of $G(A(R))$, we obtain a binet representation of the matrix $A(\tau)$. We have $\tau=\left( \begin{array}{ccc}
1 & 1 & 0\\ 0& 0& 1 \end{array} \right)$ and
$A(\tau)$ has $q:=2$ more rows than $A'$ (see Figure \ref{fig:exampleAtau}).

\begin{figure}[ht!]
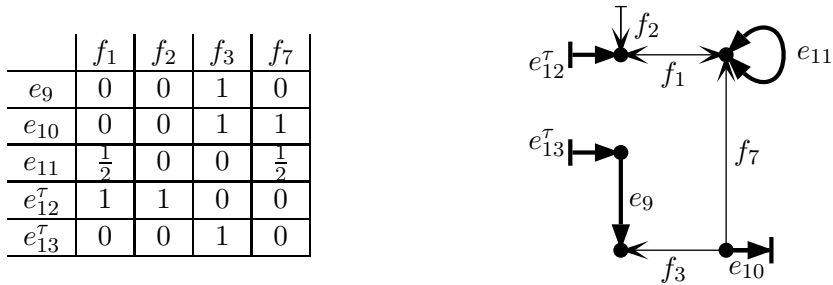


\begin{center}
$
\begin{array}{cl}

\begin{tabular}{c|c|c|c|c|c|c|c|}
  & $f_1$ & $f_2$ & $f_3$ & $f_7$  \\
\hline 
$e_{9}$  &0&0&1&0 \\
\hline
$e_{10}$ &0&0&1&1 \\
\hline  
$e_{11}$  &$\frac{1}{2}$&0&0&$\frac{1}{2}$ \\
\hline
$e_{12}^\tau$  &1&1&0&0 \\
\hline
$e_{13}^\tau$  &0&0&1& 0\\
\hline
\end{tabular} &

\psset{xunit=1.4cm,yunit=1.3cm,linewidth=0.5pt,radius=0.1mm,arrowsize=7pt,
labelsep=1.5pt,fillcolor=black}

\pspicture(0,1)(5,2.5)

\pscircle[fillstyle=solid](2.7,0){.1}
\pscircle[fillstyle=solid](2.7,1){.1}
\pscircle[fillstyle=solid](2.7,2){.1}
\pscircle[fillstyle=solid](3.7,0){.1}
\pscircle[fillstyle=solid](3.7,2){.1}

\psline[linewidth=1.6pt,arrowinset=0]{|->}(2.2,2)(2.7,2)
\rput(2,1.9){$e_{12}^\tau$}

\psline[linewidth=1.6pt,arrowinset=0]{|->}(2.2,1)(2.7,1)
\rput(2,1.1){$e_{13}^\tau$}

\psline[linewidth=1.6pt,arrowinset=0]{->}(2.7,1)(2.7,0)
\rput(2.9,0.5){$e_9$}

\psline[linewidth=1.6pt,arrowinset=0]{-|}(3.7,0)(4.15,0)
\psline[linewidth=1.6pt,arrowinset=0]{->}(3.7,0)(4.15,0)
\rput(3.9,-0.2){$e_{10}$}

\pscurve[linewidth=1.6pt,arrowinset=0]{<->}(3.7,2)(4.15,2.25)(4.15,1.75)(3.7,2)
\rput(4.55,2){$e_{11}$}


\psline[arrowinset=.5,arrowlength=1.5]{<->}(2.7,2)(3.7,2)
\rput(3.2,1.8){$f_1$}

\psline[arrowinset=.5,arrowlength=1.5]{-|}(2.7,2)(2.7,2.5)
\psline[arrowinset=.5,arrowlength=1.5]{<-}(2.7,2)(2.7,2.5)
\rput(2.95,2.3){$f_{2}$}

\psline[arrowinset=.5,arrowlength=1.5]{<-}(2.7,0)(3.7,0)
\rput(3.2,-.2){$f_3$}

\psline[arrowinset=.5,arrowlength=1.5]{->}(3.7,0)(3.7,2)
\rput(3.9,1){$f_7$}

\endpspicture 

\end{array}$
\end{center}
\vspace{.3cm}
\caption{The binet matrix $A(\tau)$ and a binet representation of $A(\tau)$, where $A$ is given in Figure \ref{fig:exampleA}.}
\label{fig:exampleAtau}

\end{figure}

The procedure Decomposition works iteratively by searching for a  binet representation of $A(\tau)$ such that its last $q$ rows correspond to half-edges denoted by $e_{12}^\tau$ and $e_{13}^\tau$.
If binet representations $G(A(R))$ and $G(A(\tau))$ of $A(R)$ and $A(\tau)$, respectively, have been found as in Figures \ref{fig:A1} and \ref{fig:exampleAtau}, it is possible to
construct a binet representation of $A$ as follows. Identify the endnode of $e_{12}^\tau$ (respectively, $e_{13}^\tau$) with the endnode of $f_1'$ and $f_2'$ (respectively, $f_3'$). Then delete the edges $e_{12}^\tau$, $e_{13}^\tau$, $f_1'$, $f_2'$ and $f_3'$. See Figure \ref{fig:exampleARtau-A}.

\begin{figure}[ht!]
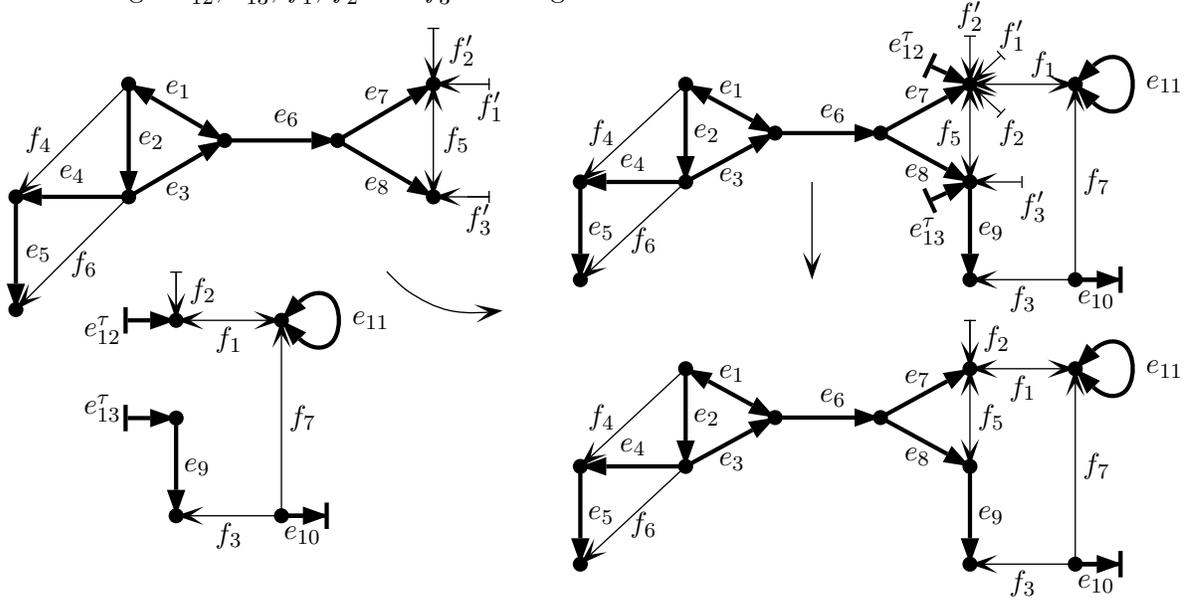


$
\begin{array}{cc}
\begin{array}{c}

\psset{xunit=1.5cm,yunit=1.5cm,linewidth=0.5pt,radius=0.1mm,arrowsize=7pt,
labelsep=1.5pt,fillcolor=black}

\pspicture(-1.2,0)(4,2)

\pscircle[fillstyle=solid](-1,0){.1}
\pscircle[fillstyle=solid](-1,1){.1}
\pscircle[fillstyle=solid](0,1){.1}
\pscircle[fillstyle=solid](0,2){.1}
\pscircle[fillstyle=solid](0.85,1.5){.1}
\pscircle[fillstyle=solid](1.85,1.5){.1}
\pscircle[fillstyle=solid](2.7,1){.1}
\pscircle[fillstyle=solid](2.7,2){.1}

\psline[linewidth=1.6pt,arrowinset=0]{<->}(0,2)(0.85,1.5)
\rput(0.44,1.95){$e_1$}

\psline[linewidth=1.6pt,arrowinset=0]{<-}(0,1)(0,2)
\rput(0.2,1.5){$e_2$}

\psline[linewidth=1.6pt,arrowinset=0]{->}(0,1)(0.85,1.5)
\rput(0.44,1.05){$e_3$}

\psline[linewidth=1.6pt,arrowinset=0]{<-}(-1,1)(0,1)
\rput(-0.5,1.2){$e_4$}


\psline[linewidth=1.6pt,arrowinset=0]{<-}(-1,0)(-1,1)
\rput(-0.8,0.5){$e_5$}

\psline[linewidth=1.6pt,arrowinset=0]{->}(0.85,1.5)(1.85,1.5)
\rput(1.4,1.7){$e_6$}

\psline[linewidth=1.6pt,arrowinset=0]{->}(1.85,1.5)(2.7,2)
\rput(2.2,1.9){$e_7$}

\psline[linewidth=1.6pt,arrowinset=0]{->}(1.85,1.5)(2.7,1)
\rput(2.2,1.1){$e_8$}

\psline[arrowinset=.5,arrowlength=1.5]{<-}(2.7,2)(3.2,2)
\psline[arrowinset=.5,arrowlength=1.5]{-|}(2.9,2)(3.2,2)
\rput(3.2,1.8){$f_1'$}

\psline[arrowinset=.5,arrowlength=1.5]{-|}(2.7,2)(2.7,2.5)
\psline[arrowinset=.5,arrowlength=1.5]{<-}(2.7,2)(2.7,2.5)
\rput(2.95,2.3){$f_{2}'$}

\psline[arrowinset=.5,arrowlength=1.5]{<-}(2.7,1)(3.2,1)
\psline[arrowinset=.5,arrowlength=1.5]{-|}(3,1)(3.2,1)
\rput(3.1,0.8){$f_3'$}


\psline[arrowinset=.5,arrowlength=1.5]{<-}(-1,1)(0,2)
\rput(-0.8,1.5){$f_4$}

\psline[arrowinset=.5,arrowlength=1.5]{<-}(-1,0)(0,1)
\rput(-0.4,0.4){$f_6$}

\psline[arrowinset=.5,arrowlength=1.5]{<->}(2.7,1)(2.7,2)
\rput(2.9,1.5){$f_5$}


\endpspicture \\

\psset{xunit=1.4cm,yunit=1.3cm,linewidth=0.5pt,radius=0.1mm,arrowsize=7pt,
labelsep=1.5pt,fillcolor=black}

\pspicture(1.5,0)(6,2)

\pscircle[fillstyle=solid](2.7,0){.1}
\pscircle[fillstyle=solid](2.7,1){.1}
\pscircle[fillstyle=solid](2.7,2){.1}
\pscircle[fillstyle=solid](3.7,0){.1}
\pscircle[fillstyle=solid](3.7,2){.1}

\psline[linewidth=1.6pt,arrowinset=0]{|->}(2.2,2)(2.7,2)
\rput(2,1.9){$e_{12}^\tau$}

\psline[linewidth=1.6pt,arrowinset=0]{|->}(2.2,1)(2.7,1)
\rput(2,1.1){$e_{13}^\tau$}

\psline[linewidth=1.6pt,arrowinset=0]{->}(2.7,1)(2.7,0)
\rput(2.9,0.5){$e_9$}

\psline[linewidth=1.6pt,arrowinset=0]{-|}(3.7,0)(4.15,0)
\psline[linewidth=1.6pt,arrowinset=0]{->}(3.7,0)(4.15,0)
\rput(3.9,-0.2){$e_{10}$}

\pscurve[linewidth=1.6pt,arrowinset=0]{<->}(3.7,2)(4.15,2.25)(4.15,1.75)(3.7,2)
\rput(4.55,2){$e_{11}$}

\pscurve[arrowinset=.5,arrowlength=1.5]{->}(4.7,2.5)(4.9,2.3)(5.3,2.1)(5.8,2.1)


\psline[arrowinset=.5,arrowlength=1.5]{<->}(2.7,2)(3.7,2)
\rput(3.2,1.8){$f_1$}

\psline[arrowinset=.5,arrowlength=1.5]{-|}(2.7,2)(2.7,2.5)
\psline[arrowinset=.5,arrowlength=1.5]{<-}(2.7,2)(2.7,2.5)
\rput(2.95,2.3){$f_{2}$}

\psline[arrowinset=.5,arrowlength=1.5]{<-}(2.7,0)(3.7,0)
\rput(3.2,-.2){$f_3$}

\psline[arrowinset=.5,arrowlength=1.5]{->}(3.7,0)(3.7,2)
\rput(3.9,1){$f_7$}

\endpspicture 
\end{array}  & 

\begin{array}{c}

\psset{xunit=1.4cm,yunit=1.3cm,linewidth=0.5pt,radius=0.1mm,arrowsize=7pt,
labelsep=1.5pt,fillcolor=black}

\pspicture(-0.5,-.3)(5,2.5)

\pscircle[fillstyle=solid](-1,0){.1}
\pscircle[fillstyle=solid](-1,1){.1}
\pscircle[fillstyle=solid](0,1){.1}
\pscircle[fillstyle=solid](0,2){.1}
\pscircle[fillstyle=solid](0.85,1.5){.1}
\pscircle[fillstyle=solid](1.85,1.5){.1}
\pscircle[fillstyle=solid](2.7,0){.1}
\pscircle[fillstyle=solid](2.7,1){.1}
\pscircle[fillstyle=solid](2.7,2){.1}
\pscircle[fillstyle=solid](3.7,0){.1}
\pscircle[fillstyle=solid](3.7,2){.1}

\psline[linewidth=1.6pt,arrowinset=0]{<->}(0,2)(0.85,1.5)
\rput(0.44,1.95){$e_1$}

\psline[linewidth=1.6pt,arrowinset=0]{<-}(0,1)(0,2)
\rput(0.2,1.5){$e_2$}

\psline[linewidth=1.6pt,arrowinset=0]{->}(0,1)(0.85,1.5)
\rput(0.44,1.05){$e_3$}

\psline[linewidth=1.6pt,arrowinset=0]{<-}(-1,1)(0,1)
\rput(-0.5,1.2){$e_4$}

\psline[linewidth=1.6pt,arrowinset=0]{<-}(-1,0)(-1,1)
\rput(-0.8,0.5){$e_5$}

\psline[linewidth=1.6pt,arrowinset=0]{->}(0.85,1.5)(1.85,1.5)
\rput(1.4,1.7){$e_6$}

\psline[linewidth=1.6pt,arrowinset=0]{->}(1.85,1.5)(2.7,2)
\rput(2.2,1.9){$e_7$}

\psline[linewidth=1.6pt,arrowinset=0]{->}(1.85,1.5)(2.7,1)
\rput(2.2,1.1){$e_8$}

\psline[linewidth=1.6pt,arrowinset=0]{->}(2.7,1)(2.7,0)
\rput(2.9,0.5){$e_9$}

\psline[linewidth=1.6pt,arrowinset=0]{-|}(3.7,0)(4.15,0)
\psline[linewidth=1.6pt,arrowinset=0]{->}(3.7,0)(4.15,0)
\rput(3.9,-0.2){$e_{10}$}

\pscurve[linewidth=1.6pt,arrowinset=0]{<->}(3.7,2)(4.15,2.25)(4.15,1.75)(3.7,2)
\rput(4.55,2){$e_{11}$}


\psline[arrowinset=.5,arrowlength=1.5]{<-|}(2.7,2)(3,1.7)
\rput(3.1,1.5){$f_2$}

\psline[arrowinset=.5,arrowlength=1.5]{<->}(2.7,2)(3.7,2)
\rput(3.4,2.2){$f_1$}

\psline[arrowinset=.5,arrowlength=1.5]{<-|}(2.7,2)(3,2.3)
\rput(3.1,2.5){$f_1'$}

\psline[arrowinset=.5,arrowlength=1.5]{<-|}(2.7,2)(2.7,2.5)
\rput(2.7,2.7){$f_{2}'$}

\psline[linewidth=1.6pt,arrowinset=0]{<-|}(2.7,2)(2.3,2.2)
\rput(2.1,2.4){$e_{12}^\tau$}

\psline[arrowinset=.5,arrowlength=1.5]{<-}(2.7,0)(3.7,0)
\rput(3.2,-.2){$f_3$}

\psline[arrowinset=.5,arrowlength=1.5]{<-}(-1,1)(0,2)
\rput(-0.8,1.5){$f_4$}

\psline[arrowinset=.5,arrowlength=1.5]{<-}(-1,0)(0,1)
\rput(-0.4,0.4){$f_6$}

\psline[linewidth=1.6pt,arrowinset=0]{|->}(2.3,.8)(2.7,1)
\rput(2.3,.5){$e_{13}^\tau$}

\psline[arrowinset=.5,arrowlength=1.5]{<->}(2.7,1)(2.7,2)
\rput(2.5,1.5){$f_5$}

\psline[arrowinset=.5,arrowlength=1.5]{<-|}(2.7,1)(3.2,1)
\rput(3.3,.75){$f_3'$}

\psline[arrowinset=.5,arrowlength=1.5]{->}(3.7,0)(3.7,2)
\rput(3.9,1){$f_7$}

\psline[arrowinset=.5,arrowlength=1.5]{->}(1.2,1)(1.2,0)

\endpspicture \\

\psset{xunit=1.4cm,yunit=1.3cm,linewidth=0.5pt,radius=0.1mm,arrowsize=7pt,
labelsep=1.5pt,fillcolor=black}

\pspicture(-0.5,0)(5,2.5)

\pscircle[fillstyle=solid](-1,0){.1}
\pscircle[fillstyle=solid](-1,1){.1}
\pscircle[fillstyle=solid](0,1){.1}
\pscircle[fillstyle=solid](0,2){.1}
\pscircle[fillstyle=solid](0.85,1.5){.1}
\pscircle[fillstyle=solid](1.85,1.5){.1}
\pscircle[fillstyle=solid](2.7,0){.1}
\pscircle[fillstyle=solid](2.7,1){.1}
\pscircle[fillstyle=solid](2.7,2){.1}
\pscircle[fillstyle=solid](3.7,0){.1}
\pscircle[fillstyle=solid](3.7,2){.1}

\psline[linewidth=1.6pt,arrowinset=0]{<->}(0,2)(0.85,1.5)
\rput(0.44,1.95){$e_1$}

\psline[linewidth=1.6pt,arrowinset=0]{<-}(0,1)(0,2)
\rput(0.2,1.5){$e_2$}

\psline[linewidth=1.6pt,arrowinset=0]{->}(0,1)(0.85,1.5)
\rput(0.44,1.05){$e_3$}

\psline[linewidth=1.6pt,arrowinset=0]{<-}(-1,1)(0,1)
\rput(-0.5,1.2){$e_4$}

\psline[linewidth=1.6pt,arrowinset=0]{<-}(-1,0)(-1,1)
\rput(-0.8,0.5){$e_5$}

\psline[linewidth=1.6pt,arrowinset=0]{->}(0.85,1.5)(1.85,1.5)
\rput(1.4,1.7){$e_6$}

\psline[linewidth=1.6pt,arrowinset=0]{->}(1.85,1.5)(2.7,2)
\rput(2.2,1.9){$e_7$}

\psline[linewidth=1.6pt,arrowinset=0]{->}(1.85,1.5)(2.7,1)
\rput(2.2,1.1){$e_8$}

\psline[linewidth=1.6pt,arrowinset=0]{->}(2.7,1)(2.7,0)
\rput(2.9,0.5){$e_9$}

\psline[linewidth=1.6pt,arrowinset=0]{-|}(3.7,0)(4.15,0)
\psline[linewidth=1.6pt,arrowinset=0]{->}(3.7,0)(4.15,0)
\rput(3.9,-0.2){$e_{10}$}

\pscurve[linewidth=1.6pt,arrowinset=0]{<->}(3.7,2)(4.15,2.25)(4.15,1.75)(3.7,2)
\rput(4.55,2){$e_{11}$}


\psline[arrowinset=.5,arrowlength=1.5]{<->}(2.7,2)(3.7,2)
\rput(3.2,1.8){$f_1$}

\psline[arrowinset=.5,arrowlength=1.5]{-|}(2.7,2)(2.7,2.5)
\psline[arrowinset=.5,arrowlength=1.5]{<-}(2.7,2)(2.7,2.5)
\rput(2.95,2.3){$f_{2}$}

\psline[arrowinset=.5,arrowlength=1.5]{<-}(2.7,0)(3.7,0)
\rput(3.2,-.2){$f_3$}

\psline[arrowinset=.5,arrowlength=1.5]{<-}(-1,1)(0,2)
\rput(-0.8,1.5){$f_4$}

\psline[arrowinset=.5,arrowlength=1.5]{<-}(-1,0)(0,1)
\rput(-0.4,0.4){$f_6$}

\psline[arrowinset=.5,arrowlength=1.5]{<->}(2.7,1)(2.7,2)
\rput(2.9,1.5){$f_5$}

\psline[arrowinset=.5,arrowlength=1.5]{->}(3.7,0)(3.7,2)
\rput(3.9,1){$f_7$}

\endpspicture
\end{array}

\end{array}
$

\vspace{.3cm}

\caption{How to obtain a binet representation of $A$ (at the bottom right) using binet representations of $A(R)$ and $A(\tau)$, where $A$ is given in Figure \ref{fig:exampleA}.}
\label{fig:exampleARtau-A}

\end{figure}

\vspace{.3cm}

\subsection{The procedure Decomposition}\label{subsec:decomppro}

Let us introduce the main definitions and lemmas 
involved in the description of the procedure Decomposition and the proof of Theorem \ref{thmDecdecomp}.
A pair of columns $(A_{\bullet j},A_{\bullet j'})$ with $1 \le j,j' \le m$ such that $\emptyset \neq 
s_{\frac{1}{2}}(A_{\bullet j})\neq s_{\frac{1}{2}}(A_{\bullet j'})
\neq \emptyset$ and 
$s_{\frac{1}{2}}(A_{\bullet j})\cap s_{\frac{1}{2}}(A_{\bullet j'})\neq \emptyset$ is called a \emph{connective pair}\index{connective pair}.

\begin{lem}\label{lemhalf1} Suppose that $A$ is binet and let $1\le j,j'\le m$ be such that  the pair $(A_{\bullet j},A_{\bullet j'})$ is connective. Then in any binet representation of $A$,
$s_{\frac{1}{2}}(A_{\bullet j})\cap s_{\frac{1}{2}}(A_{\bullet j'})$ is the edge index set of a basic full cycle.
\end{lem}

\begin{lem}\label{lemhalf1b} Suppose that $A$ is bicyclic and not $\frac{1}{2}$-equisupported. For any column index $j$ ($1\le j\le m$) such that $s_{\frac{1}{2}}(A_{\bullet j})\neq \emptyset$,   there exists $j'$ ($1 \le j' \le m$) such that the pair $(A_{\bullet j},A_{\bullet j'})$ is connective.
\end{lem}

\begin{lem}\label{lemhalf2} Suppose that $A$ has a binet representation $G(A)$ which is not bicyclic. For any column index $j$ ($1\le j\le m$) such that $s_{\frac{1}{2}}(A_{\bullet j})\neq \emptyset$ and there is no connective pair $(A_{\bullet j},A_{\bullet j'})$ with $1 \le j' \le m$, the set $s_{\frac{1}{2}}(A_{\bullet j})$ corresponds to the edge index set of a (basic) full cycle in $G(A)$.
\end{lem}

\noindent
{\bf Proof of Lemmas \ref{lemhalf1}, \ref{lemhalf1b} and \ref{lemhalf2}.} \quad 
Let $G(A)$ be a binet representation of $A$. By Lemma \ref{lemdefiWeight1}, an entry $a_{ij}$ of $A$ equal to $\frac{1}{2}$ corresponds to the weight of an edge $e_i$ belonging to a basic full cycle in the fundamental circuit of a half-edge or 2-edge $f_j$. Then, by Corollary \ref{corBidirectedCircuit} the $\frac{1}{2}$-support of a half-edge (respectively, a 2-edge) represents the edge index set of one basic full cycle (respectively, one or two basic full cycles). This implies Lemma \ref{lemhalf1}.

Let $j$ ($1\le j\le m$) be a column index such that $s_{\frac{1}{2}}(A_{\bullet j})\neq \emptyset$. Let us prove Lemma \ref{lemhalf1b}.  If $A$ is bicyclic and not $\frac{1}{2}$-equisupported, then by above arguments it follows that there exists a column index $j'$ such that $f_j$ and $f_{j'}$ are a $2$-edge and a half-edge, respectively, or vice versa. Hence $(A_{\bullet j},A_{\bullet j'})$ is a connective pair. 

Now let us show Lemma \ref{lemhalf2}. We assume that $G(A)$ is not bicyclic. Suppose that $s_{\frac{1}{2}}(A_{\bullet j})$ is the edge index set of two basic full cycles $C_1$ and $C_2$. Since $G(A)$ is not bicyclic, there exists at least a third basic cycle. Thus, using the fact that $A$ is connected, there exists a 2-edge $f_{j'}$ containing $C_1$ or $C_2$ in its 
fundamental circuit and a basic cycle different from these both. So the pair $(A_{\bullet j},A_{\bullet j'})$ is connective. This 
proves the contrapositive of Lemma \ref{lemhalf2}.
{\hfill$\BBox{\rule{.3mm}{3mm}}$} \\

Given a row index subset $R$ of $A$, the procedure Decomposition computes a row index subset $R'$ ($\supseteq R$)
and column index subsets $S_{\frac{1}{2}}$ and $S_2$ of $A$
as well as a submatrix $A(R)$ by using a subroutine called MatRcyclic. Then, thanks to a subroutine called Mattau, a number $q$ and matrices denoted by $\tau$ and $A(\tau)$ are computed, where $q$ corresponds to the number of rows of $\tau$. Moreover, the set $S_{\frac{1}{2}}$ is also partitioned into subsets $U_1,\ldots,U_q$. Later, we shall prove the following: 
If $A(R)$ has an $R$-cyclic representation $G(A(R))$, then the nonbasic edges with index in $S_{\frac{1}{2}}$ are half-edges denoted by $f_j'$ for all $j\in S_{\frac{1}{2}}$; further, for  $i=1,\ldots,q$, the nonbasic edges $f_j'$ with $j\in U_i$ have a common endnode denoted as $u_i$.
If $A(\tau)$ has a binet representation $G(A(\tau))$, then 
the basic edges corresponding to the rows of $\tau$ are denoted by  $e_{i_{max}+1}^\tau,\ldots,e_{i_{max}+q}^\tau$, where $i_{max}$ is the largest row index of $A$. 

\begin{tabbing}
\textbf{Procedure\,\,Decomposition($A$,$Q$)}\\

\textbf{Input:} A connected matrix $A$ with entries $0$, $1$, $2$ or $\frac{1}{2}$ and a row index subset $Q$ of $A$.\\
\textbf{Output: }\= a binet representation $G(A)$ of $A$ such that each element in $Q$ is the index \\
\>  of a basic half-edge, or stops.\\
 
1)\verb"  "\= {\bf if }\=  $A$ has a $\frac{1}{2}$-entry, {\bf then}\\
2)  \>            \> let $j$ be such that $s_{\frac{1}{2}}(A_{\bullet j}) \neq \emptyset$;\\ 
3)  \>            \> {\bf if } \=  there exists a connective pair $(A_{\bullet j},A_{\bullet j'})$,  {\bf then}\\
  \>            \> \> let $R= s_{\frac{1}{2}}(A_{\bullet j})\cap s_{\frac{1}{2}}(A_{\bullet j'})$;\\
\>            \> {\bf otherwise } \\
 \>            \> \> let $R= s_{\frac{1}{2}}(A_{\bullet j})$; \\
\>            \> {\bf endif } \\
4) \>            \> call {\tt MatRcyclic}($A$,$R$) which outputs sets $R'$, $S_{\frac{1}{2}}$ and $S_2$ and a submatrix $A(R)$ of $A$; \\
\>            \> then call {\tt Mattau}($A$,$R'$,$S_{\frac{1}{2}}$,$S_2$) which outputs a number $q$, subsets $U_1,\ldots,U_q$ of $S_{\frac{1}{2}}$ \\
\>            \> and matrices $\tau$ and $A(\tau)$; \\
5) \>            \> {\bf if }  $Q\cap R' \neq \emptyset$
or $q > |R'|$,  {\bf then}  STOP:  $A$ does not have a binet representation such \\
\> \> that $R$ is the index set of a basic cycle and $Q$ an index set of basic half-edges;\\
\> \> {\bf endif }; \\
6) \>            \> let $i_{max}$ be the largest row index of $A$, $Q=Q\cup \{i_{max}+1, \ldots , i_{max}+ q\}$;\\
\>            \> let $G(A(\tau))= {\tt Decomposition(A(\tau),Q)}$ and $G(A(R))={\tt RCyclic}(A(R),R)$;\\
7) \> \> {\bf for}  \= $i=1,\ldots,q$, identify $u_i$ with the endnode of $e_{i_{max}+i}^\tau$ {\bf endfor};  \\
\> \> delete $f_j'$ for all $j\in S_{\frac{1}{2}}$ and $e_{i_{max}+i}^\tau$ for all $1\le i \le q$; then output the binet\\
\> \> representation $G(A)$ of $A$; \\
\> {\bf otherwise}\\
8) \> \> call {\tt OnehalfbinetQ($A$,$Q$)} and output a binet representation of $A$ if we have one;\\
\> {\bf endif}\\

\end{tabbing}

For a matrix $A'$, we define the graph $H(A')$\index{graph!$H(A')$} with respect to $A'$ 
as follows. The set of vertices is the column index set of $A'$, 
and two vertices $j$ and $j'$ are adjacent if and only if 
$s(A_{\bullet j}')\cap s(A_{\bullet j'}')\neq \emptyset $. 
Let us state the subroutine MatRcyclic. See Figures \ref{fig:nonbinet} and \ref{fig:Sbinet} for an illustration of all sets computed by this procedure.

\begin{tabbing}
\textbf{Procedure\,\,MatRcyclic($A$,$R$)}\\

\textbf{Input:} A matrix $A$ and a row index subset $R$ of $A$.\\
\textbf{Output: }\=  A row index subset $R'$ and column index subsets $S_{\frac{1}{2}}$ and $S_2$ of $A$, \\
\> and a submatrix $A(R)$ of $A$.\\
 
1) \verb"  "\=  let $S_0=\{j\,:\,  s(A_{\bullet j}) \cap R \neq \emptyset \}$,
$S_{\frac{1}{2}}=\{j\,:\, R \subseteq s_{\frac{1}{2}}(A_{\bullet j})\}$,
$S_1= S_0\verb"\"S_{\frac{1}{2}}$ \\
\> and $S_2$ be the set of all nodes in $H(A_{\bullet \overline{S_{\frac{1}{2}}} })$ reachable from $S_1$; \\
\> let $R'=\underset{j\in S_2}{\cup}s(A_{\bullet j})\cup R$ and 
$A(R)=A_{R' \times (S_{\frac{1}{2}}\cup S_2)}$;\\
\> output $R'$, $S_{\frac{1}{2}}$, $S_2$ and $A(R)$;
\end{tabbing}

\begin{figure}[ht]

\psset{xunit=1cm,yunit=1cm,linewidth=0.5pt,radius=0.3mm,arrowsize=1pt,
labelsep=1.5pt,fillcolor=black}

\pspicture(0,0)(6,8)

\put(1.4,2.6){\rotateleft{$\overbrace{\hspace{3.3cm}}$}}
\put(.9,3.7){$R'$}

\put(2.4,4.2){\rotateleft{$\overbrace{\hspace{1.7cm}}$}}
\put(1.9,4.5){$ R$}

\put(2.9,6.1){$\overbrace{\hspace{2.2cm}}$}
\put(3.4,6.55){$ S_{\frac{1}{2}}$}

\put(5.1,6.1){$\overbrace{\hspace{1.9cm}}$}
\put(5.4,6.4){$ S_1$}

\put(5.1,6.65){$\overbrace{\hspace{3.8cm}}$}
\put(7.5,6.95){$ S_2$}

\put(2.9,7){$\overbrace{\hspace{4.1cm}}$}
\put(4.2,7.4){$ S_0$}

\rput(8,3){
\begin{tabular}{|ccc|ccc|ccc|ccccc|}
\hline
$\frac{1}{2}$ & \ldots & $\frac{1}{2}$ & 1 &  & & 0 & \ldots & 0 & 0 &&
\ldots&& 0\\
\vdots & & \vdots & 1 & 1 & &\vdots  & & \vdots &\vdots&&&&\vdots \\
$\frac{1}{2}$ & \ldots& $\frac{1}{2}$& & & 1 & 0 & \ldots & 0 & 0& &\ldots& &0 \\
\hline
1& & & & & 1 &  &  & 2 &0 & & \ldots& & 0\\
& & 1&2& 1 &  & 1 &  & &\vdots & & & &\vdots \\
& & & & &  & $\frac{1}{2}$ & 1  & &0 & &\ldots & & 0\\
\hline
1& 1 & &0 & & 
\multicolumn{2}{c}
\ldots 
 & & 0 & 1 &  & & &1\\
 
 1&  & & & & 
\multicolumn{2}{c}

 & &  &  & 1 & & &1\\
 
 &  &1 &\vdots & & 
\multicolumn{2}{c}

 & & \vdots &  &  & &2 &\\
 
 & $\frac{1}{2}$ & & & & 
\multicolumn{2}{c}

 & &  &  &  &$\frac{1}{2}$ & &\\
 
 &  & &0 & & 
\multicolumn{2}{c}
\ldots 
 & & 0 &  &\ldots  & &\ldots &\\

\hline
\end{tabular}
}

\endpspicture

\caption{An example of a (non-binet) matrix $A$ to illustrate the sets $R'$, $S_{\frac{1}{2}}$, $S_0$, $S_1$  
and $S_2$, with respect to $R$.}
\label{fig:nonbinet}

\end{figure}
\vspace{.3cm}

\begin{figure}[ht!]
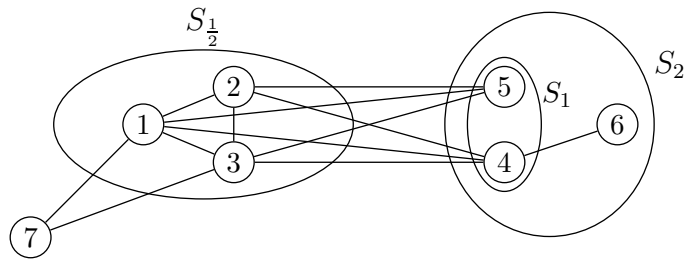


\begin{center}

\psset{xunit=1cm,yunit=1cm,linewidth=0.5pt,radius=0.3mm,arrowsize=1pt,labelsep=1.5pt,fillcolor=black}

\pspicture(0,0)(10,2.5)

\cnodeput(2.2,1){1}{1}
\cnodeput(3.4,1.5){2}{2}
\cnodeput(3.4,.5){3}{3}
\cnodeput(7,.5){4}{4}
\cnodeput(7,1.5){5}{5}
\cnodeput(8.5,1){6}{6}
\cnodeput(0.7,-0.5){7}{7}

\psellipse(3,1)(2,1)
\rput(3,2.3){\large{$S_{\frac{1}{2}}$}}

\psellipse(7,1)(.5,.9)
\rput(7.7,1.4){$S_1$}

\psellipse(7.6,1)(1.4,1.5)
\rput(9.2,1.8){\large{$S_2$}}

\ncline{-}{1}{2}
\ncline{-}{2}{3}
\ncline{-}{3}{1}
\ncline{-}{7}{1}
\ncline{-}{7}{3}
\ncline{-}{1}{4}
\ncline{-}{1}{5}
\ncline{-}{2}{4}
\ncline{-}{2}{5}
\ncline{-}{3}{4}
\ncline{-}{3}{5}
\ncline{-}{4}{6}

\endpspicture
\end{center}

\vspace{1cm}

\caption{An illustration of the graph $H(A)$ and the sets $S_{\frac{1}{2}}$, $S_1$ and $S_2$ ($S_0=S_{\frac{1}{2}} \uplus S_1$) in function of $R=\{1,2,3\}$ and $A$ given in Figure \ref{fig:exampleA} ($R'=\{1,2,3,4,5,6,7,8\}$).}
\label{fig:Sbinet}

\end{figure}

Suppose that $A$ has a binet representation $G(A)$ such that $ R$ is the edge index set of a full 
basic cycle, say $C$. Let $S_{\frac{1}{2}}$, $S_0$, $S_1$, $S_2$ and $R'$ be the sets computed by MatRcyclic.
The columns with index in $S_0$ correspond to nonbasic edges whose
fundamental circuit contains edges of $C$. As mentioned earlier, an entry $a_{ij}$ of $A$ equal to $\frac{1}{2}$ corresponds to the weight of an edge $e_i$ belonging to a basic full cycle in the fundamental circuit of a half-edge or 2-edge $f_j$. Moreover,
the $\frac{1}{2}$-support of a half-edge (respectively, a 2-edge) represents the edge index set of one basic full cycle (respectively, one or two basic full cycles). Thus 
$S_1$ is the subset of $S_0$ of $1$-edge (except half-edge) indexes, and $S_{\frac{1}{2}}$ is the index set of nonbasic half-edges and 2-edges whose fundamental circuit  
contains $C$. The set $S_2$ is an index set of $1$-edges (without any half-edge). So the $\frac{1}{2}$-support of the columns with index in $S_2$ is empty. We observe that 
 $R'$ corresponds to the edge index set of a negative $1$-tree in $G(A)$. 

Since $H(A)$ is connected, the connected components of $H(A_{\bullet \overline{S_{\frac{1}{2}}}})$ that do not 
have any vertex of $S_1$ are linked to at least one vertex of $S_{\frac{1}{2}}$ in $H(A)$. Moreover, $S_{\frac{1}{2}}$ induces a 
clique in $H(A)$ and each element of $S_{\frac{1}{2}}$ is adjacent to each vertex in 
$S_1$.

\begin{lem}\label{lemA1}
If $A$ has a binet representation such that $R$ 
is the edge index set of a basic full cycle, then $A(R)$ is $R$-cyclic.
\end{lem}

\begin{proof} Let $G(A)$ be a binet representation of $A$ such that $R$ is the edge index set of a basic full cycle. As mentioned earlier, the set $R'$ computed by MatRcyclic  corresponds to the edge index set of a basic negative $1$-tree denoted by $T$. Let us contract all basic edges of
$G(A)$ not in $T$ (in particular, 2-edges with one endnode in $T$ become half-edges). Then by deleting all left nonbasic loose edges, one obtains a binet representation $G(A(R))$ of $A(R)$ such that $R$ corresponds to a full basic cycle (see Figure \ref{fig:A1} for an example).
\end{proof}\\

Now, let us state the subroutine Mattau.

\begin{tabbing}
\textbf{Procedure\,\,Mattau($A$,$R'$,$S_{\frac{1}{2}}$,$S_2$)}\\

\textbf{Input:} A matrix $A$, a row index subset $R$ and column index subsets $S_{\frac{1}{2}}$ and $S_2$ of $A$.\\
\textbf{Output: }\= a number $q$, subsets $U_1\ldots,U_q$ of $S_{\frac{1}{2}}$ and matrices $\tau$ and $A(\tau)$
such that $q$ is\\
\> the number of rows of $\tau$.\\
 
1) \verb"  "\= let us partition $S_{\frac{1}{2}}$ into subsets $U_1,\ldots, U_q$ so that $j, \,j'\in S_{\frac{1}{2}}$ are in a
same subset\\
\> if and only if  $s(A_{\bullet j})\cap R'=s(A_{\bullet j'})\cap R'$;\\
2)\>  let $\tau$ be the matrix given by  
$\tau_{ij}=\left \{
\begin{array}{ll}
1 & \m{if } j\in U_i\\
0 & \m{otherwise},
\end{array}\right.$ for $i=1,\ldots,q$ and $j\in S_{\frac{1}{2}}$,\\
\>  $l=|\overline{S_{\frac{1}{2}}\cup S_2}|$ and
$A(\tau)=\left( \begin{array}{c}  
A_{\overline{R'}\times \overline{ S_2}}\\
\left( \begin{array}{cc} \tau & O_{q\times l} \end{array} \right) \end{array}\right)$;\\
\> output $q$, $U_1\ldots,U_q$, $\tau$ and $A(\tau)$;
\end{tabbing}

For an example of the sets and matrices computed by Mattau, see 
Figure \ref{fig:exampleAtau}, where $U_1=\{1,2\}$, $U_2=\{3\}$ and 
$\tau = \left(
\begin{matrix} 1 & 1 & 0 \\
                0 & 0 & 1 \\
\end{matrix}\right)$. 
Notice that if $S_2\neq \emptyset$, then the number of columns of $A(\tau)$ is strictly smaller than the number of columns of
$A$. In the following lemmas and propositions, if we are given row index subsets $R$ and $Q$ of $A$, then we assume that the objects $R'$, $A(R)$, $q$, $U_1,\ldots,U_q$, $\tau$ and $A(\tau)$ have been computed by the subroutines MatRcyclic and Mattau. Further, $i_{max}$ denotes the largest row index of $A$. The next lemma justifies the definition of the nodes $u_1,\ldots,u_q$ used in step 7.

\begin{lem}\label{lemRecBinetui}
Let $R$ be a row index subset of $A$ and suppose that $A(R)$ has an $R$-cyclic representation $G(A(R))$. Then, for $i=1,\ldots,q$, the nonbasic edges $f_j'$ with $j\in U_i$ are half-edges having a common endnode denoted by $u_i$.
\end{lem}

\begin{proof}
Using arguments similar to those preceding Lemma \ref{lemA1}, one can prove that each nonbasic edge $f_j'$ with $j\in S_{\frac{1}{2}}$ is a half-edge. From the way of partitioning $S_{\frac{1}{2}}$ into subsets $U_1,\ldots,U_q$ and by Corollary \ref{corBidirectedCircuit} and Lemma \ref{lemdefiWeight2}, it follows that for any $i$ ($1\le i \le q$),
$s(A_{\bullet j}) \cap (R' \verb"\" R)$  is the edge index set of a (directed) path denoted $P_i$ for all $j\in U_i$. Notice that the initial node of $P_i$ is a central node, and we denote by $u_i$ the terminal node of $P_i$, which is the endnode of $f_j'$ for all $j\in U_i$.

\end{proof}\\

\begin{lem}\label{lemA2}
Let $R$ and $Q$ be row index subsets of $A$.
If $A$ has a binet representation such that $R$ 
is the edge index set of a basic full cycle and each element in $Q$ is the index of a basic half-edge, then $q\le |R'|$,
$Q\cap R'= \emptyset$ and $A(\tau)$ has a binet representation such that each element in $Q\cup \{i_{max}+1,\ldots,i_{max}+q\} $ is a basic half-edge index.
\end{lem}

\begin{proof}
Let $G(A)$ be a binet representation  of $A$ such that
$R$ is the edge index set of
a basic full cycle, and each element in $Q$ is the index of a basic half-edge. 
We have already observed that 
$ R'$ is the edge index set of a negative $1$-tree whose
basic cycle has an index set equal to $R$. Thus $Q \cap  R'=\emptyset$. 

Let $G_1$ be the connected subgraph of $G(A)$ formed by 
basic edges with index in $R'$ and nonbasic ones indexed by $ S_2$.  
We may suppose that $G_1$ has a unique basic bidirected edge, and this one is entering. 
For each set $U_i$ ($1\le i \le q$), let $P_i$ be the directed path in $G(A)$ with edge index set $s(A_{\bullet j}) \cap (R' - R)$ for all $j\in U_i$ (as in Lemma \ref{lemRecBinetui}). Notice that the initial node of $P_i$ is a central node, and we denote by $u_i$ the terminal node of $P_i$.
Clearly, $q\le |R'|$ since $|R'|$ corresponds to the number of nodes of $G_1$. Delete from $G(A)$ all
edges of $G_1$  and then the isolated nodes. Let us call $G'$ the remaining graph. 
We observe that $G_1\cap G'=\{u_1, \ldots, u_q \}$.
Then, by adding a basic entering half-edge to
each left node of $G_1\cap G'$ it yields a binet representation of $A(\tau)$ such that for $i=i_{max}+1,\ldots,i_{max}+ q$ the $i$th row of $A(\tau)$ corresponds to a basic half-edge incident with $u_i$ and each element in $Q$ is the index of a basic half-edge. 
\end{proof}\\

\begin{prop}\label{propDecX}
Let $R$ and $Q$ be row index subsets of $A$. Then the matrix $A$ has a binet representation such that $R$ 
is the edge index set of a basic full cycle and any element in $Q$ is a basic half-edge index if and only if $q\le |R'|$,
$Q\cap R'=\emptyset$, 
$A(R)$ is $R$-cyclic and
$A(\tau)$  has a binet representation such that each element in
$Q \cup \{i_{max}+1,\ldots,i_{max}+ q\}$  is a basic half-edge index.
\end{prop}

\begin{proof}
The "only if" part has been proved in Lemmas \ref{lemA1} and \ref{lemA2}.\\
$\Leftarrow:$
Let $G(A(R))$ be an $R$-cyclic representation of
$A(R)$ having exactly one basic bidirected edge (and this one is entering).
Let $G(A(\tau))$ be a binet representation of $A(\tau)$ such
that each element $i\in Q \cup \{ i_{max}+ 1,\ldots,i_{max}+ q\}$ is the index of a basic half-edge $e_i^\tau$. By Corollary \ref{corBidirectedCircuit} and Lemma \ref{lemdefiWeight1}, it follows that the columns of $A(\tau)$ with index in $S_{\frac{1}{2}}$ correspond to half-edges or 2-edges in $G(A(\tau))$. By performing step 7 of the procedure Decomposition and using Corollary \ref{corBidirectedCircuit} and Lemmas \ref{lemdefiWeight1} and \ref{lemdefiWeight2}, we obtain a binet representation of 
$A$ such that $R$ corresponds to a basic full cycle and each element in $Q$ is a basic half-edge index 
(see Figure \ref{fig:exampleARtau-A}).
\end{proof}\\

\begin{prop}\label{propdecRep}
Given a matrix $A$ and a row index subset $Q\neq \emptyset$ of $A$, if the matrix $A$ has a binet representation such that $Q$ is an index set of half-edges then the procedure Decomposition with input $A$ and $Q$ outputs such a representation.
\end{prop}

\begin{proof}
The proof is by induction on the number of rows of $A$ having a $\frac{1}{2}$-entry. Suppose that there exists a binet representation $G(A)$ of $A$ such that $Q$ is an index set of half-edges.
If $A$ has no $\frac{1}{2}$-entry, then it follows that $G(A)$ is $\frac{1}{2}$-binet (see Lemma \ref{lemBinetintegral}). By Theorem \ref{thmOnhalfbinetQ}, in step 8 the procedure Decomposition outputs a $\frac{1}{2}$-binet representation of $A$ such that $Q$ is an index set of half-edges.

Now suppose that $A$ has at least one $\frac{1}{2}$-entry.
Let $R$ be the set computed in step 3. By Lemmas \ref{lemhalf1} and \ref{lemhalf2}, we deduce that $R$ is the edge index set of a full cycle in $G(A)$. 
By the "only if" part of Proposition \ref{propDecX}, it follows that the procedure does not stop in step 5, $A(R)$ is $R$-cyclic and $A(\tau)$ has a binet representation such that each element in $Q \cup \{i_{max}+1,\ldots,i_{max}+ q\}$  is a basic half-edge index. By Theorem \ref{thmcyclicproCyc}, the subroutine RCyclic with input $A(R)$ and $R$ outputs an $R$-cyclic representation  of $A(R)$. Since the number of rows in $A(\tau)$ with at least one $\frac{1}{2}$-entry is smaller than in $A$, by induction hypothesis, in step $6$ the procedure Decomposition computes a binet representation $G(A(\tau))$ of $A(\tau)$ such that $Q\cup \{i_{max}+1,\ldots,i_{max}+ q\}$ is an index set of basic half-edges.
Using the "if part" of Proposition \ref{propDecX} (see the proof), this concludes the proof of Proposition \ref{propdecRep}.
\end{proof}\\

\noindent
{\bf Proof of Theorem \ref{thmDecdecomp}.} \quad The "if part" is straightforward.

Let us prove the other implication.
Suppose that $A$ has a binet representation $G(A)$. Assume that if $A$ is bicyclic then it is not $\frac{1}{2}$-equisupported, and if $A$ has no $\frac{1}{2}$-entry then it is not cyclic. 
Let us show that the procedure Decomposition with input $A$ and $Q=\emptyset$ outputs a binet representation of $A$.

If $A$ has no $\frac{1}{2}$-entry, then by assumption it is not cyclic and by Lemma \ref{lemBinetintegral} it follows that $A$ is $\frac{1}{2}$-binet. By Theorem \ref{thmOnhalfbinetQ}, in step 8 the procedure Decomposition outputs a $\frac{1}{2}$-binet representation of $A$.

Now suppose that $A$ has at least one $\frac{1}{2}$-entry. Let $R$ be the set computed in step 3. Using Lemmas \ref{lemhalf1}, \ref{lemhalf1b} and \ref{lemhalf2}, 
we deduce that $R$ is the edge index set of a full cycle in $G(A)$. 
By the "only if" part of Proposition \ref{propDecX}, it follows that the procedure does not stop in step 5, $A(R)$ is $R$-cyclic and  $A(\tau)$  has a binet representation such that each element in $\{i_{max}+1,\ldots,i_{max}+ q\}$  is a basic half-edge index. So in step 6, by Theorem \ref{thmcyclicproCyc}, the subroutine RCyclic with input $A(R)$ and $R$ outputs an $R$-cyclic representation  of $A(R)$, and
by Proposition \ref{propdecRep}, the procedure Decomposition computes a binet representation $G(A(\tau))$ of $A(\tau)$ such that $\{i_{max}+ 1,\ldots,i_{max}+ q\}$ is an index set of basic half-edges.
Using the "if part" of Proposition \ref{propDecX} (see the proof), the procedure Decomposition outputs a binet representation of $A$.

Let us analyze the time needed to construct a binet representation $G(A)$ of $A$. One calls at most 
one time the procedure OnehalfbinetQ with input a matrix $A'$, and thanks to step 5, the number of rows of $A'$ does not exceed $n$. Since the number of nonzero elements in $A'$ is at most $\alpha$, by Theorem \ref{thmOnhalfbinetQ} the subroutine OnehalfbinetQ takes time $O(nm^2 \alpha)$. Denote by $R_1, \ldots, R_\delta$ all row index subsets computed in step 3. The procedure Decomposition constructs an $R_i$-cyclic representation of some matrix $A(R_i)$, for $i=1,\ldots,\delta$. Let
$n_i$, $m_i$ and $\alpha_i$ be the number of rows, columns and
nonzero elements of $A(R_i)$, respectively, for $i=1,\ldots,\delta$. Thanks to step 5 of the procedure Decomposition, we have $n_1+n_2+ \ldots + n_\delta \le n$
and $\alpha_1+\alpha_2+ \ldots + \alpha_\delta \le \alpha$ (and $m_i\le m$ for all $1\le i \le \delta$). By Theorem \ref{thmcyclicproCyc}, the computational effort to obtain an $R_i$-cyclic representation $G(A(R_i))$ of $A(R_i)$ is $O(n_i m_i \alpha_i)$.  Altogether, the time required to construct a binet representation of $A$ is bounded by 

$$C( nm^2 \alpha + n_1 m_1 \alpha_1 + \ldots + n_\delta m_\delta \alpha_\delta) \le 2Cn m^2 \alpha,$$

\noindent
for some constant $C$. This completes the proof of Theorem \ref{thmDecdecomp}.
{\hfill$\BBox{\rule{.3mm}{3mm}}$} \\


\clearpage
\thispagestyle{empty}
\cleardoublepage
\verb"   "
\newpage

\chapter{Encoding of a global structure:
a digraph $D$}\label{ch:multidiD}

Let $A$ be a connected matrix with entries $0$, $1$, $2$ or $\frac{1}{2}$, and $R^*$ a row index subset of $A$. If $A$ has an $R^*$-cyclic, $R^*$-central or $R^*$-network representation $G(A)$, then the basic forest in $G(A)$ with edge index set $\overline{R^*}$ has some nice properties that have to be recognized. For that purpose, we construct a digraph $D$ with respect to $R^*$ as well as matrices, called bonsai matrices, related with the vertices of $D$. These objects  play a significant role in the recognition of $R^*$-cyclic and $R^*$-central matrices. They are also useful for the recognition of $\{1,\rho\}$-corelated network matrices in Section \ref{sec:CentralNetcorelated}. Notice that if $A$ is $R^*$-cyclic, then for any column index $j$ such that $s_{\frac{1}{2}}(A_{\bullet j})\neq \emptyset$ we have $s_{\frac{1}{2}}(A_{\bullet j})=R^*$ (see Lemmas \ref{lemdefiWeight1} and \ref{corBidirectedCircuit}). Further, in Chapter \ref{ch:central}, we deal with $R^*$-central matrices without any $\frac{1}{2}$-entry.
So, throughout this chapter, we assume that if  
$s_{\frac{1}{2}}(A_{\bullet j})\neq \emptyset$ for some column index $j$, then $s_{\frac{1}{2}}(A_{\bullet j})= R^*$.

Let $S^*=\{j\, : \, s(A_{\bullet j})\cap R^* \neq \emptyset\}$.
We partition $\overline{ R^*}$ into subsets $E_1,\ldots,E_b$ corresponding to the vertices of $D$ as follows.
For $i,i'\in \overline{ R^*}$, $i$ and $i'$ are in a same subset $E_\ell$ ($1\le \ell \le b$) if and only if there exist column indexes $j$ and $j'$ such that $i\in s(A_{\bullet j})$, $i'\in s(A_{\bullet j'})$, and $j$ and $j'$ are in the same connected component of the graph $H(A_{\bullet \overline{ S^*} })$. See Figure \ref{fig:AH}.

\begin{figure}[h!]
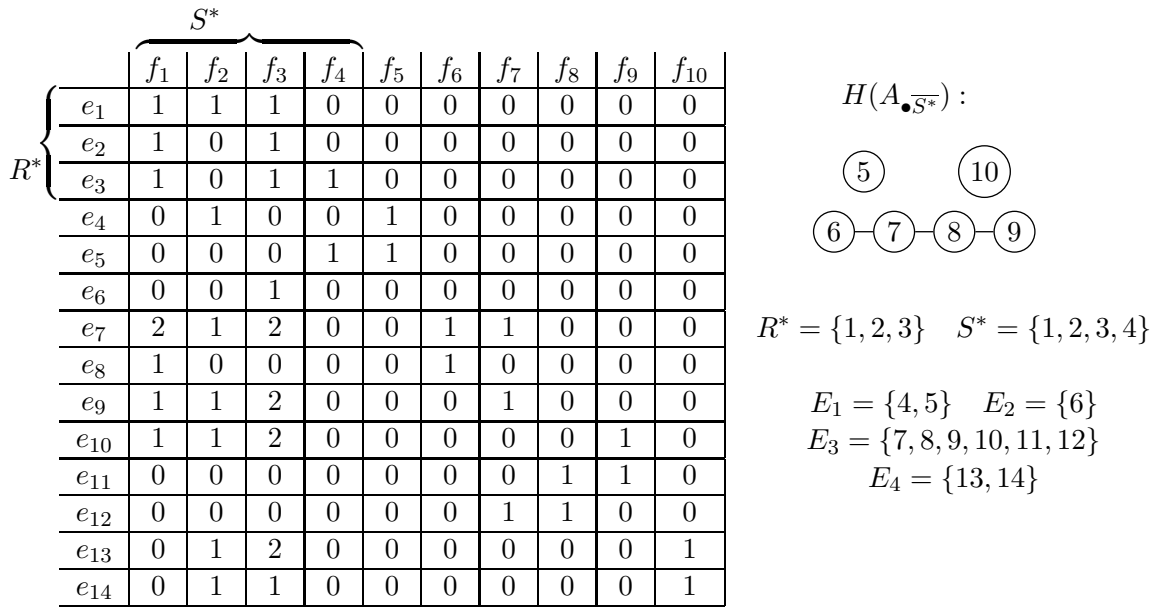

\vspace{.2cm}
$
\begin{array}{cc}

\psset{xunit=1cm,yunit=1cm,linewidth=0.5pt,radius=0.3mm,arrowsize=1pt,
labelsep=1.5pt,fillcolor=black}

\pspicture(0,0)(9,8)

\put(1.6,7.7){$\overbrace{\hspace{3cm}}$}
\put(2.3,8){$S^*$}

\put(.3,5.73){\rotateleft{$\overbrace{\hspace{1.5cm}}$}}
\put(-.1,6){$R^*$}

\rput(5,4){
\begin{tabular}{c|c|c|c|c|c|c|c|c|c|c|}
  & $f_1$ & $f_2$ & $f_3$ & $f_4$ & $f_5$ & $f_6$ & $f_7$ & $f_8$ & $f_9$ & $f_{10}$ \\
\hline
$e_1$  &1&1&1&0&0&0&0&0&0&0 \\
\hline
$e_2$  &1&0&1&0&0&0&0&0&0&0 \\
\hline
$e_3$  &1&0&1&1&0&0&0&0&0&0 \\
\hline
$e_4$  &0&1&0&0&1&0&0&0&0&0 \\
\hline
$e_5$  &0&0&0&1&1&0&0&0&0&0 \\
\hline
$e_6$  &0&0&1&0&0&0&0&0&0&0 \\
\hline
$e_7$ &2&1&2&0&0&1&1&0&0&0 \\
\hline
$e_8$  &1&0&0&0&0&1&0&0&0&0 \\
\hline
$e_9$  &1&1&2&0&0&0&1&0&0&0 \\
\hline
$e_{10}$  &1&1&2&0&0&0&0&0&1&0 \\
\hline
$e_{11}$  &0&0&0&0&0&0&0&1&1&0 \\
\hline
$e_{12}$  &0&0&0&0&0&0&1&1&0&0 \\
\hline
$e_{13}$  &0&1&2&0&0&0&0&0&0&1 \\
\hline
$e_{14}$  &0&1&1&0&0&0&0&0&0&1 \\
\hline
\end{tabular} }

\endpspicture

&

\psset{xunit=0.8cm,yunit=1cm,linewidth=0.5pt,radius=0.1mm,arrowsize=7pt,
labelsep=1.5pt,fillcolor=black}

\pspicture(-0.4,-5)(5,0)

\put(1.3,2){$H(A_{\bullet \overline{S^*}}):$}

\cnodeput(2,1.1){5}{5}
\cnodeput(1.5,0.3){6}{6}
\cnodeput(2.5,0.3){7}{7}
\cnodeput(3.5,0.3){8}{8}
\cnodeput(4.5,0.3){9}{9}
\cnodeput(4,1.1){10}{10}

\rput(3.5,-1){$R^*=\{1,2,3\}$\quad $S^*=\{1,2,3,4\}$}

\rput(3.5,-2){$E_1=\{4,5\}$\quad $E_2=\{6\}$}
\rput(3.5,-2.5){$E_3=\{7,8,9,10,11,12\}$}
\rput(3.5,-3){$E_4=\{13,14\}$}

\ncline{-}{6}{7}
\ncline{-}{7}{8}
\ncline{-}{8}{9}

\endpspicture 

\end{array}
$

\caption{A binet matrix $A$ and the graph $H(A_{\bullet \overline{S^*}})$.}  
\label{fig:AH}
\end{figure}

\begin{figure}[h!]
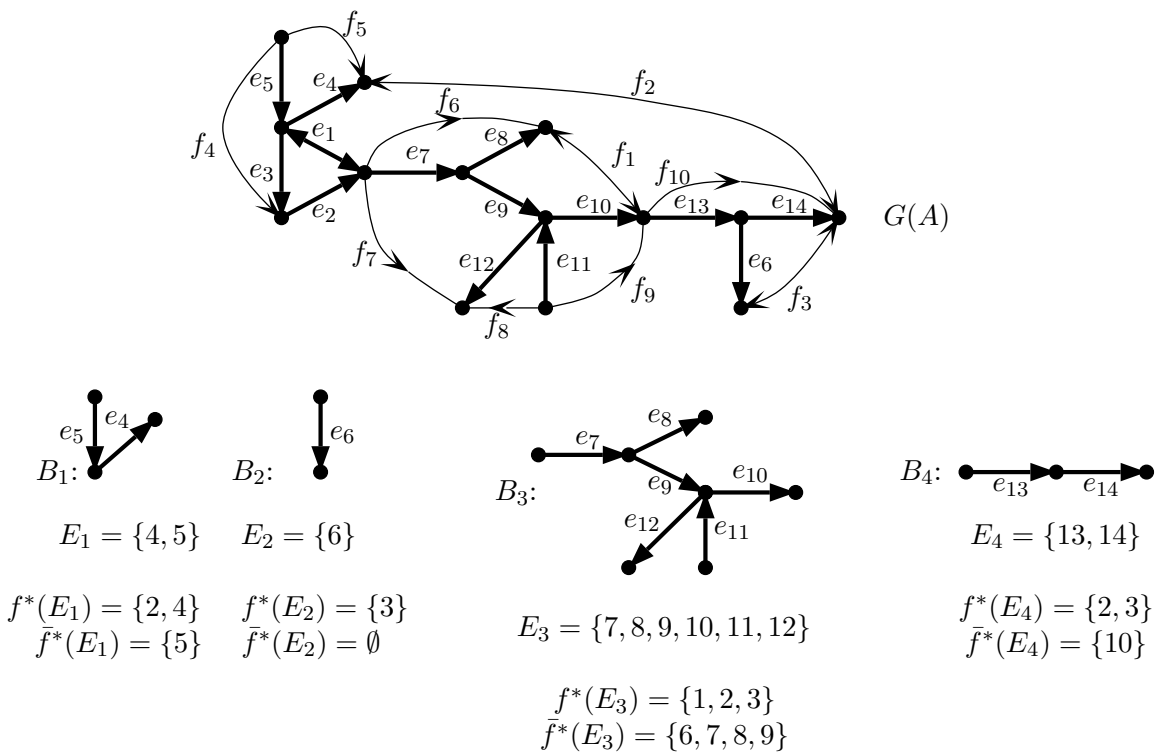


\vspace{0.4cm}
\begin{center}
\psset{xunit=1.3cm,yunit=1.2cm,linewidth=0.5pt,radius=0.1mm,arrowsize=7pt,labelsep=1.5pt,fillcolor=black}

\pspicture(0,0)(5.8,3)

\pscircle[fillstyle=solid](0,1){.1}
\pscircle[fillstyle=solid](0,2){.1}
\pscircle[fillstyle=solid](0,3){.1}
\pscircle[fillstyle=solid](0.85,2.5){.1}
\pscircle[fillstyle=solid](0.85,1.5){.1}
\pscircle[fillstyle=solid](1.85,1.5){.1}
\pscircle[fillstyle=solid](1.85,0){.1}
\pscircle[fillstyle=solid](2.7,1){.1}
\pscircle[fillstyle=solid](2.7,0){.1}
\pscircle[fillstyle=solid](2.7,2){.1}
\pscircle[fillstyle=solid](3.7,1){.1}
\pscircle[fillstyle=solid](4.7,1){.1}
\pscircle[fillstyle=solid](4.7,0){.1}
\pscircle[fillstyle=solid](5.7,1){.1}
\rput(6.5,1){$G(A)$}

\psline[linewidth=1.6pt,arrowinset=0]{<->}(0,2)(0.85,1.5)
\rput(0.44,1.95){$e_1$}

\psline[linewidth=1.6pt,arrowinset=0]{->}(0,1)(0.85,1.5)
\rput(0.44,1.05){$e_2$}

\psline[linewidth=1.6pt,arrowinset=0]{<-}(0,1)(0,2)
\rput(-0.2,1.5){$e_3$}

\psline[linewidth=1.6pt,arrowinset=0]{->}(0,2)(0.85,2.5)
\rput(0.45,2.5){$e_4$}

\psline[linewidth=1.6pt,arrowinset=0]{<-}(0,2)(0,3)
\rput(-0.2,2.5){$e_5$}

\psline[linewidth=1.6pt,arrowinset=0]{->}(4.7,1)(4.7,0)
\rput(4.9,0.5){$e_6$}

\psline[linewidth=1.6pt,arrowinset=0]{->}(0.85,1.5)(1.85,1.5)
\rput(1.4,1.7){$e_7$}

\psline[linewidth=1.6pt,arrowinset=0]{->}(1.85,1.5)(2.7,2)
\rput(2.2,1.9){$e_8$}

\psline[linewidth=1.6pt,arrowinset=0]{->}(1.85,1.5)(2.7,1)
\rput(2.2,1.1){$e_9$}

\psline[linewidth=1.6pt,arrowinset=0]{->}(2.7,1)(3.7,1)
\rput(3.2,1.15){$e_{10}$}

\psline[linewidth=1.6pt,arrowinset=0]{->}(2.7,0)(2.7,1)
\rput(3,0.5){$e_{11}$}

\psline[linewidth=1.6pt,arrowinset=0]{<-}(1.85,0)(2.7,1)
\rput(2,0.5){$e_{12}$}

\psline[linewidth=1.6pt,arrowinset=0]{->}(3.7,1)(4.7,1)
\rput(4.2,1.15){$e_{13}$}

\psline[linewidth=1.6pt,arrowinset=0]{->}(4.7,1)(5.7,1)
\rput(5.2,1.15){$e_{14}$}

\pscurve[arrowinset=.5,arrowlength=1.5]{<->}(2.7,2)(3.1,1.7)(3.7,1)
\rput(3.5,1.7){$f_1$}

\pscurve[arrowinset=.5,arrowlength=1.5]{<->}(0.85,2.5)(3.7,2.3)(5.2,1.8)(5.7,1)
\rput(3.7,2.5){$f_2$}

\pscurve[arrowinset=.5,arrowlength=1.5]{<->}(4.7,0)(5.1,0.2)(5.7,1)
\rput(5.3,0.1){$f_3$}

\pscurve[arrowinset=.5,arrowlength=1.5]{->}(0,3)(-.6,2)(0,1)
\rput(-.8,1.8){$f_4$}

\pscurve[arrowinset=.5,arrowlength=1.5]{->}(0,3)(.5,3.1)(0.85,2.5)
\rput(.75,3.15){$f_5$}

\pscurve[arrowinset=.5,arrowlength=1.5]{->}(0.85,1.5)(1,1.8)(1.85,2.1)
\pscurve[arrowinset=.5,arrowlength=1.5]{-}(1.85,2.1)(2.2,2.1)(2.7,2)
\rput(1.7,2.3){$f_6$}

\pscurve[arrowinset=.5,arrowlength=1.5]{->}(0.85,1.5)(1,0.7)(1.3,0.4)
\pscurve[arrowinset=.5,arrowlength=1.5]{-}(1.3,0.4)(1.31,0.39)(1.85,0)
\rput(0.85,0.6){$f_7$}

\psline[arrowinset=.5,arrowlength=1.5]{-<}(1.85,0)(2.4,0)
\psline[arrowinset=.5,arrowlength=1.5]{-}(2.3,0)(2.7,0)
\rput(2.2,-0.2){$f_8$}

\pscurve[arrowinset=.5,arrowlength=1.5]{-<}(3.7,1)(3.65,0.6)(3.4,0.3)
\pscurve[arrowinset=.5,arrowlength=1.5]{-}(3.5,0.4)(3.2,0.15)(2.7,0)
\rput(3.7,0.2){$f_9$}

\pscurve[arrowinset=.5,arrowlength=1.5]{->}(3.7,1)(4,1.3)(4.7,1.4)
\pscurve[arrowinset=.5,arrowlength=1.5]{-}(4.7,1.4)(5.4,1.25)(5.7,1)
\rput(4,1.5){$f_{10}$}

\endpspicture 
\end{center}

$
\begin{array}{ccc}

\begin{array}{c}

\psset{xunit=1cm,yunit=1cm,linewidth=0.5pt,radius=0.1mm,arrowsize=7pt,
labelsep=1.5pt,fillcolor=black}

\pspicture(0,0)(4,.5)

\pscircle[fillstyle=solid](.5,0){.1}
\pscircle[fillstyle=solid](.5,1){.1}
\pscircle[fillstyle=solid](1.3,0.7){.1}
\pscircle[fillstyle=solid](3.5,0){.1}
\pscircle[fillstyle=solid](3.5,1){.1}

\rput(0,0){$B_1$:}

\psline[linewidth=1.6pt,arrowinset=0]{->}(.5,0)(1.3,0.7)
\rput(.8,.7){$e_{4}$}
\psline[linewidth=1.6pt,arrowinset=0]{->}(.5,1)(.5,0)
\rput(.2,.5){$e_{5}$}

\rput(2.6,0){$B_2$:}

\psline[linewidth=1.6pt,arrowinset=0]{<-}(3.5,0)(3.5,1)
\rput(3.8,0.5){$e_{6}$}

\endpspicture\\ \\

E_1=\{4,5\}\, \quad \, E_2=\{6\} \\ \\ 
f^*(E_1)=\{2,4\}\, \quad \,f^*(E_2)=\{3\}\\ 
\bar f^*(E_1)=\{5\}\, \quad \,\bar f^*(E_2)=\emptyset  

\end{array} &

\begin{array}{c}

\psset{xunit=1.2cm,yunit=1cm,linewidth=0.5pt,radius=0.1mm,arrowsize=7pt,
labelsep=1.5pt,fillcolor=black}

\pspicture(0,0)(4.5,2)

\pscircle[fillstyle=solid](0.85,.5){.1}
\pscircle[fillstyle=solid](1.85,.5){.1}
\pscircle[fillstyle=solid](1.85,-1){.1}
\pscircle[fillstyle=solid](2.7,0){.1}
\pscircle[fillstyle=solid](2.7,1){.1}
\pscircle[fillstyle=solid](2.7,-1){.1}
\pscircle[fillstyle=solid](2.7,0){.1}
\pscircle[fillstyle=solid](3.7,0){.1}

\psline[linewidth=1.6pt,arrowinset=0]{->}(0.85,.5)(1.85,.5)
\rput(1.4,.7){$e_7$}

\psline[linewidth=1.6pt,arrowinset=0]{->}(1.85,.5)(2.7,1)
\rput(2.2,1){$e_8$}

\psline[linewidth=1.6pt,arrowinset=0]{->}(1.85,.5)(2.7,0)
\rput(2.2,.1){$e_9$}

\psline[linewidth=1.6pt,arrowinset=0]{->}(2.7,0)(3.7,0)
\rput(3.2,.25){$e_{10}$}

\psline[linewidth=1.6pt,arrowinset=0]{->}(2.7,-1)(2.7,0)
\rput(3,-0.5){$e_{11}$}

\psline[linewidth=1.6pt,arrowinset=0]{<-}(1.85,-1)(2.7,0)
\rput(2,-.4){$e_{12}$}

\rput(0.6,0){$B_3$:}

\endpspicture \\ \\ \\ \\

E_3=\{7,8,9,10,11,12\}   \\ \\
f^*(E_3)=\{1,2,3\}\\
\bar f^*(E_3)=\{6,7,8,9\}

\end{array}  &

\begin{array}{c}

\psset{xunit=1.2cm,yunit=1cm,linewidth=0.5pt,radius=0.1mm,arrowsize=7pt,
labelsep=1.5pt,fillcolor=black}

\pspicture(0,0)(3,.5)

\pscircle[fillstyle=solid](.5,0){.1}
\pscircle[fillstyle=solid](1.5,0){.1}
\pscircle[fillstyle=solid](2.5,0){.1}

\psline[linewidth=1.6pt,arrowinset=0]{->}(.5,0)(1.5,0)
\rput(1,-0.2){$e_{13}$}

\psline[linewidth=1.6pt,arrowinset=0]{->}(1.5,0)(2.5,0)
\rput(2,-.2){$e_{14}$}

\rput(0,0){$B_4$:}

\endpspicture 

\\ \\ 

E_4=\{13,14\} \\ \\
f^*(E_4)= \{2,3\}\\
\bar f^*(E_4)= \{10 \}

\end{array}

\end{array}
$
\vspace{0.3cm}
\caption{A binet representation $G(A)$ of the matrix $A$ given in Figure \ref{fig:AH} and the corresponding bonsais.}  
\label{fig:G(A)bon}
\end{figure}

For all $1\le \ell\le b$, we define the following objects. The set 
$E_\ell$ is called a \emph{bonsai}\index{bonsai $E_\ell$}. Recall that $f(E_\ell)= \{ j \, : s(A_{\bullet j})\cap E_\ell \neq \emptyset\}$. The set 
$f^*(E_\ell)=f(E_\ell) \cap S^*$ is called 
the \emph{global connecter set of $E_\ell$}\index{global connecter set} and $\bar f^*(E_\ell)= f(E_\ell) \cap \overline{S^*} $ the \emph{local connecter set of $E_\ell$}\index{local connecter set}.
From the way of partitioning $\overline{R^*}$ into $E_1,\ldots,E_b$, it results that $\bar f^*(E_\ell)= \{j\, : \, s(A_{\bullet j})\subseteq E_\ell \}$. See Figure \ref{fig:G(A)bon}.

Suppose that $A$ has an $R^*$-cyclic, $R^*$-central or $R^*$-network representation $G(A)$. Let $1\le \ell \le b$.
From the way of partitioning $\overline{R^*}$, it follows that the bonsai $E_\ell$ is the edge index set of a tree in $G(A)$, denoted as $B_\ell$, which is called a \emph{bonsai}\index{bonsai $B_\ell$}.
For any $j$ in the local connecter set $\bar f^*(E_\ell)$ of $E_\ell$, the set $s(A_{\bullet j})\cap E_\ell$ corresponds to the edge index set of a (basic) directed path in $B_\ell$.
Any column with index in the global connecter set $f^*(E_\ell)$ of $E_\ell$ corresponds to a nonbasic edge whose fundamental circuit intersects the edge set of $B_\ell$ and the edge set of the basic cycle. By Lemma \ref{lemdefiWeight2},
for any $j$ in the global connecter set $ f^*(E_\ell)$ of $E_\ell$, the intersection of the fundamental circuit of $f_j$  with $B_\ell$ forms either a directed path, called  a \emph{$B_\ell$-path generated by  $f_\beta$}\index{path@$B_\ell$-path generated by $f_\beta$}, or a union of two distinct maximal directed paths called \emph{$B_\ell$-paths} and leaving a common node (see also Figures \ref{fig:lsubstem1} and \ref{fig:lsubstem2}).

To illustrate these notions, let us study an example.
In Figure \ref{fig:G(A)bon}, we have a $\{1,2,3\}$-cyclic representation of the matrix given in Figure \ref{fig:AH} and the corresponding bonsais. Observe that the path made up of $e_{13}$ and $e_{14}$ is a $B_4$-path generated by $f_2$ and $f_3$, and the path consisting of $e_{13}$ is also 
a $B_4$-path generated by $f_3$.

Furthermore,
the nonbasic edge $f_1$ generates two $B_3$-paths, one with edge index set $\{ 7,8\}$ and another one with edge index set $\{ 7,9,10\}$. This raises the following questions. Does there exist a binet representation of $A$ such that $f_1$ generates a $B_3$-path with edge index set not equal to $\{7,8\}$ nor $\{ 7,9,10\}$? Actually, there is no such representation. Is it possible to compute the row index sets $\{7,8\}$ and $\{ 7,9,10\}$ corresponding to $B_3$-paths without knowing any binet representation of $A$? Section \ref{sec:mainthm} states a useful theorem in this direction.
In Section \ref{sec:bonsaimat}, for all $1\le \ell \le b$, we compute row index subsets of $E_\ell$, called $E_\ell$-paths, denoted as $E_\ell^1,\ldots, E_\ell^{m(\ell)}$, and define a bonsai matrix $N_\ell$.  The matrix  $A_{E_\ell\times \bar f^*(E_\ell)}$ and the vectors $\chi_{E_\ell^k}^{\{1,\ldots,n\}}$ are submatrices of $N_\ell$. 
We present the construction of a digraph $D$ in Section \ref{sec:DefDigraphD}. The vertex set of $D$ is $\{E_1,\ldots,E_b\}$. 
Provided that $A$ has an $R^*$-cyclic representation $G(A)$, an $E_\ell$-path
corresponds to the edge index set of a $B_\ell$-path,  and an edge $(E_\ell,E_{\ell'})$ in $D$ labeled $E_{\ell'}^k$, for some $1\le k \le m(\ell')$, means the following:

\begin{itemize}

\item For any nonbasic edge $f_\beta$ generating at least one $B_\ell$-path (or equivalently, $\beta \in f^*(E_\ell)$), $f_\beta$ generates a $B_{\ell'}$-path with edge index set $E_{\ell'}^k$.

\item If $f_\beta$ generates two distinct $B_\ell$-paths or $s_2(A_{\bullet \beta})\neq \emptyset$, then $f_\beta$ generates exactly one $B_{\ell'}$-path with edge index set
$E_{\ell'}^k=s(A_{\bullet \beta})\cap E_{\ell'}=s_2(A_{\bullet \beta})\cap E_{\ell'}$.

\end{itemize}

In Figure \ref{fig:G(A)bon},  
consider $B_3$ and $B_4$. We have $f^*(E_4)=\{2,3\}$.
The nonbasic edge $f_2$ generates exactly one $B_4$-path with edge index set $\{13,14\}$ and a $B_3$-path with edge index set $\{7,9,10\}$. Moreover, $f_3$ generates two $B_4$-paths with edge index sets $\{13,14\}$ and $\{13\}$, respectively, and a $B_3$-path with edge index set $\{7,9,10\}=s(A_{\bullet 3})\cap E_3=s_2(A_{\bullet 3})\cap E_3$. Thus $(E_4,E_3)$ is an edge in $D$ labeled by $\{7,9,10\}$. 

On the other hand, observe that $B_4$ and $B_2$ have a common node, corresponding to the endnode of a $B_4$-path with edge index set $\{13\}$; and $B_4$ is closer to the basic cycle than $B_2$. This implies that $(E_2,E_4)$ is an edge in $D$ labeled by $\{13\}$.
More generally, $G(A)$ induces a spanning  forest of $D$ denoted by $T_{G(A)}$ and defined as follows: for any $B_\ell,\, B_{\ell'} \subseteq G(A)$, $(E_\ell,E_{\ell'})_{E_{\ell'}^k} \in T_{G(A)}$ if and only if one node of $B_\ell$ coincides with the endnode of the $B_{\ell'}$-path with edge index set $E_{\ell'}^k$ (and $B_{\ell'}$ is closer than $B_\ell$ to the basic cycle). This forest is studied in Section \ref{sec:FordigaphD} and motivates the definition of a feasible spanning forest of $D$. Section \ref{sec:For} deals with a procedure computing (in some cases) a feasible spanning forest of $D$.

\section{An important theorem}\label{sec:mainthm}
 
Let $N$ be a connected nonnegative network matrix, $G(N)$ a network representation of $N$ and $\mathcal{I}$ a row index subset of $N$. We suppose that $\mathcal{I}$ is the edge index set of two (basic) disjoint directed paths $p$ and $p'$ having exactly one common node $v$ in $G(N)$ which is an endnode of both $p$ and $p'$ as in Figure \ref{fig7:pathThm1}. In this section, we provide a theorem which allows us to distinguish the edge indexes of $p$ from those of $p'$, by using the matrix $N$ only.

\begin{figure}[h!]
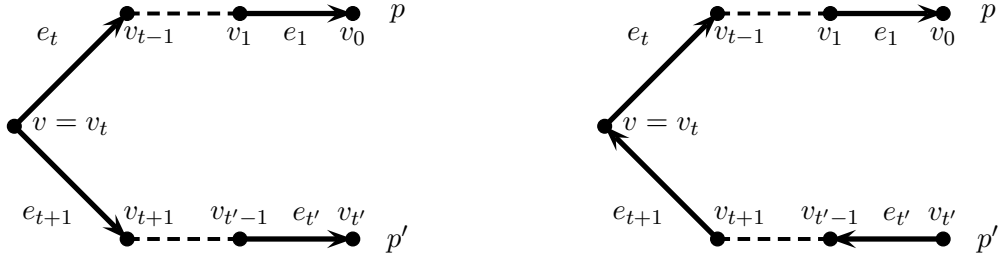

\vspace{0.5cm}
$
\begin{array}{cc}

\psset{xunit=1.5cm,yunit=1.5cm,linewidth=0.5pt,radius=0.1mm,arrowsize=7pt,
labelsep=1.5pt,fillcolor=black}

\pspicture(0,0)(5,2)

\pscircle[fillstyle=solid](0,1){.1}
\pscircle[fillstyle=solid](1,2){.1}
\pscircle[fillstyle=solid](2,2){.1}
\pscircle[fillstyle=solid](3,2){.1}
\pscircle[fillstyle=solid](1,0){.1}
\pscircle[fillstyle=solid](2,0){.1}
\pscircle[fillstyle=solid](3,0){.1}

\psline[linewidth=2pt]{->}(0,1)(1,0)
\rput(0.3,0.2){$e_{t+1}$}

\psline[linewidth=2pt]{->}(0,1)(1,2)
\rput(0.3,1.8){$e_{t}$}

\psline[linewidth=1.5pt,linestyle=dashed](1,2)(2,2)

\psline[linewidth=2pt]{->}(2,2)(3,2)
\rput(2.5,1.8){$e_{1}$}

\psline[linewidth=1.5pt,linestyle=dashed](1,0)(2,0)

\psline[linewidth=2pt]{->}(2,0)(3,0)
\rput(2.6,0.2){$e_{t'}$}

\rput(0.5,1){$v=v_t$}\rput(1.2,1.8){$v_{t-1}$}\rput(2,1.8){$v_1$}
\rput(3,1.8){$v_0$}\rput(1.2,0.2){$v_{t+1}$}
\rput(2,0.2){$v_{t'-1}$}\rput(3,0.2){$v_{t'}$}
\rput(3.4,2){$p$}\rput(3.4,0){$p'$}

\endpspicture & 

\psset{xunit=1.5cm,yunit=1.5cm,linewidth=0.5pt,radius=0.1mm,arrowsize=7pt,
labelsep=1.5pt,fillcolor=black}

\pspicture(0,0)(5,2)

\pscircle[fillstyle=solid](0,1){.1}
\pscircle[fillstyle=solid](1,2){.1}
\pscircle[fillstyle=solid](2,2){.1}
\pscircle[fillstyle=solid](3,2){.1}
\pscircle[fillstyle=solid](1,0){.1}
\pscircle[fillstyle=solid](2,0){.1}
\pscircle[fillstyle=solid](3,0){.1}

\psline[linewidth=2pt]{<-}(0,1)(1,0)
\rput(0.3,0.2){$e_{t+1}$}

\psline[linewidth=2pt]{->}(0,1)(1,2)
\rput(0.3,1.8){$e_{t}$}

\psline[linewidth=1.5pt,linestyle=dashed](1,2)(2,2)

\psline[linewidth=2pt]{->}(2,2)(3,2)
\rput(2.5,1.8){$e_{1}$}

\psline[linewidth=1.5pt,linestyle=dashed](1,0)(2,0)

\psline[linewidth=2pt]{<-}(2,0)(3,0)
\rput(2.6,0.2){$e_{t'}$}

\rput(0.5,1){$v=v_t$}\rput(1.2,1.8){$v_{t-1}$}\rput(2,1.8){$v_1$}
\rput(3,1.8){$v_0$}\rput(1.2,0.2){$v_{t+1}$}
\rput(2,0.2){$v_{t'-1}$}\rput(3,0.2){$v_{t'}$}
\rput(3.4,2){$p$}\rput(3.4,0){$p'$}

\endpspicture

\end{array}
$

\caption{the paths $p$ and $p'$ in two possible configurations.}
\label{fig7:pathThm1}
\end{figure}

Up to row permutations and reversing the orientation of all edges if necessary, we may 
suppose that the vertices of $p$ are $v_0$, $v_1, \ldots, v_t=v$, those of $p'$ are $v_t=v,v_{t+1}, \ldots,v_{t'}$, and 
 $e_i=(v_{i},v_{i-1})$ for $i=1,\ldots,t$; moreover, either $e_i=(v_{i-1},v_{i})$ for all $i=t+1,\ldots,t'$ or $e_i=(v_{i},v_{i-1})$ for all $i=t+1,\ldots, t'$. See Figure \ref{fig7:pathThm1}. A path $h$ in $H(N)$ between two nodes $k$ and $k'$ is said to be \emph{minimal}\index{minimal path} if any two non-consecutive vertices of $h$ are not adjacent in $H(N)$.
We denote by  $H_{\mathcal{I}}(N)$ the graph 
$H(N)$ in which each vertex $k$ is blue if $s(N_{\bullet k})\cap \mathcal{I}\neq \emptyset$ and yellow otherwise, and each edge $(k,k')$ has a blue color if  $s(N_{\bullet k})\cap s(N_{\bullet k'})
\cap \mathcal{I}\neq \emptyset$ and yellow otherwise. See Figure
\ref{fig7:thm1}.
Using these preliminaries, we can state the main theorem.

\begin{thm}\label{stem}
Let $k$ and $k'$ be two blue vertices in $H_{\mathcal{I}}(N)$. 
Then $\mathcal{I}\cap s(N_{\bullet k})$
and $\mathcal{I}\cap s(N_{\bullet k'})$ are edge index subsets of a same basic directed path of $G(N)$ if and only if all minimal  paths between $k$ and $k'$ in $H_{\mathcal{I}}(N)$ have an even number of yellow edges.
\end{thm}

\begin{figure}
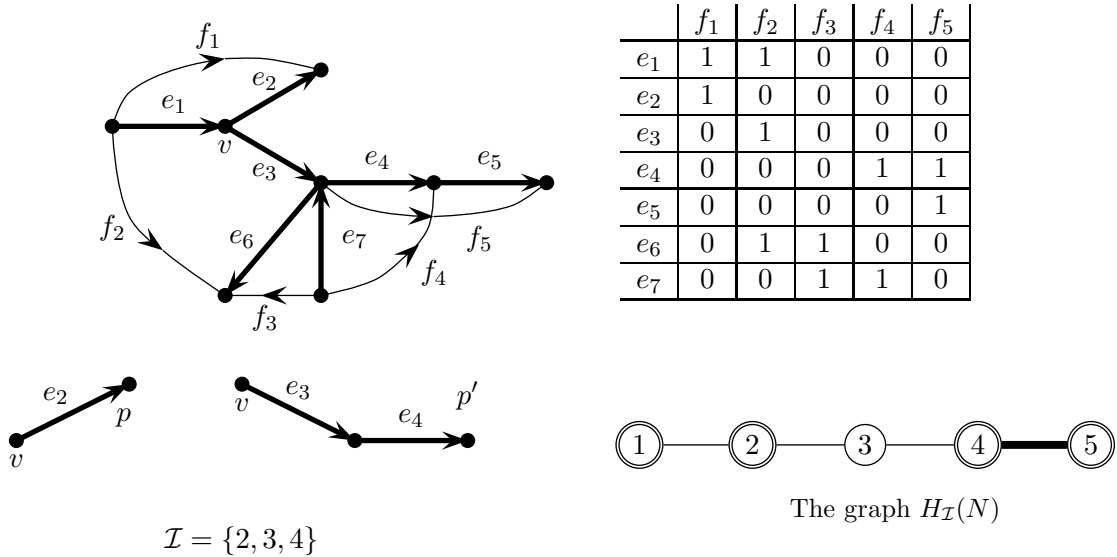


$
\begin{array}{cc}
\begin{array}{c}
\\ \\
\psset{xunit=1.5cm,yunit=1.5cm,linewidth=0.5pt,radius=0.1mm,arrowsize=7pt,
labelsep=1.5pt,fillcolor=black}

\pspicture(0,0)(5,2)

\pscircle[fillstyle=solid](0.85,1.5){.1}
\pscircle[fillstyle=solid](1.85,1.5){.1}
\pscircle[fillstyle=solid](1.85,0){.1}
\pscircle[fillstyle=solid](2.7,1){.1}
\pscircle[fillstyle=solid](2.7,0){.1}
\pscircle[fillstyle=solid](2.7,2){.1}
\pscircle[fillstyle=solid](3.7,1){.1}
\pscircle[fillstyle=solid](4.7,1){.1}

\rput(1.85,1.32){$v$}

\psline[linewidth=2pt]{->}(0.85,1.5)(1.85,1.5)
\rput(1.4,1.7){$e_1$}

\psline[linewidth=2pt]{->}(1.85,1.5)(2.7,2)
\rput(2.2,1.9){$e_2$}

\psline[linewidth=2pt]{->}(1.85,1.5)(2.7,1)
\rput(2.2,1.1){$e_3$}

\psline[linewidth=2pt]{->}(2.7,1)(3.7,1)
\rput(3.2,1.2){$e_{4}$}

\psline[linewidth=2pt]{->}(3.7,1)(4.7,1)
\rput(4.2,1.2){$e_{5}$}

\psline[linewidth=2pt]{<-}(1.85,0)(2.7,1)
\rput(2,0.5){$e_6$}

\psline[linewidth=2pt]{->}(2.7,0)(2.7,1)
\rput(3,0.5){$e_7$}

\pscurve{->}(0.85,1.5)(1,1.8)(1.85,2.1)
\pscurve{-}(1.85,2.1)(2.2,2.1)(2.7,2)
\rput(1.7,2.3){$f_1$}

\pscurve{->}(0.85,1.5)(1,0.7)(1.3,0.4)
\pscurve{-}(1.3,0.4)(1.31,0.39)(1.85,0)
\rput(0.85,0.6){$f_2$}

\psline{-<}(1.85,0)(2.4,0)
\psline{-}(2.3,0)(2.7,0)
\rput(2.2,-0.2){$f_3$}

\pscurve{-<}(3.7,1)(3.65,0.6)(3.4,0.3)
\pscurve{-}(3.5,0.4)(3.2,0.15)(2.7,0)
\rput(3.7,0.2){$f_4$}

\pscurve{->}(2.7,1)(3,0.8)(3.7,0.7)
\pscurve{-}(3.7,0.7)(4.4,0.8)(4.7,1)
\rput(4.1,0.5){$f_5$}

\endpspicture 
\end{array} &

\begin{tabular}{c|c|c|c|c|c|}
  & $f_1$ & $f_2$ & $f_3$ & $f_4$ & $f_5$ \\
\hline
$e_1$  &1&1&0&0&0 \\
\hline
$e_2$  &1&0&0&0&0 \\
\hline
$e_3$ &0&1&0&0&0 \\
\hline
$e_4$  &0&0&0&1&1 \\
\hline
$e_5$  &0&0&0&0&1 \\
\hline
$e_6$  &0&1&1&0&0 \\
\hline
$e_7$  &0&0&1&1&0 \\
\hline
\end{tabular} 

\end{array}
$

\vspace{1cm}

$
\begin{array}{ccccccc}

\begin{array}{c}

\psset{xunit=1.5cm,yunit=1.5cm,linewidth=0.5pt,radius=0.1mm,arrowsize=7pt,
labelsep=1.5pt,fillcolor=black}

\pspicture(0,0)(4,0.5)

\pscircle[fillstyle=solid](0,0){.1}
\pscircle[fillstyle=solid](1,0.5){.1}
\pscircle[fillstyle=solid](2,0.5){.1}
\pscircle[fillstyle=solid](3,0){.1}
\pscircle[fillstyle=solid](4,0){.1}

\psline[linewidth=2pt]{->}(0,0)(1,0.5)
\rput(0.35,0.4){$e_2$}

\psline[linewidth=2pt]{->}(2,0.5)(3,0)
\rput(2.5,0.45){$e_3$}

\psline[linewidth=2pt]{->}(3,0)(4,0)
\rput(3.5,0.2){$e_4$}

\rput(0,-0.18){$v$}

\rput(2,0.32){$v$}

\rput(0.95,0.2){$p$}

\rput(4,0.4){$p'$}

\endpspicture\\ \\ \\

\mathcal{I}=\{2,3,4\}
\end{array}& & & & & &

\psset{xunit=1.5cm,yunit=1.5cm,linewidth=0.5pt,radius=0.1mm,arrowsize=7pt,
labelsep=1.5pt,fillcolor=black}

\pspicture(0,0)(4,0.5)

\cnodeput[doubleline=true](0,0.3){2}{1}
\cnodeput[doubleline=true](1,0.3){3}{2}
\cnodeput(2,0.3){4}{3}
\cnodeput[doubleline=true](3,0.3){5}{4}
\cnodeput[doubleline=true](4,0.3){6}{5}

\ncline{-}{2}{3}
\ncline{-}{3}{4}
\ncline{-}{4}{5}
\ncline[linewidth=3pt]{-}{5}{6}

\put(2,-0.5){\shortstack{\small{The graph $ H_{\mathcal{I}}(N)$}}}

\endpspicture

\end{array}$

\vspace{0.5cm}

\caption{A network matrix with a network representation, two directed paths 
$p$ and $p'$ starting at $v$ and the 
graph $H_{\mathcal{I}}(N)$. (Heavy edges correspond to blue edges and vertices with a double 
circle to blue ones.)}
\label{fig7:thm1}
\end{figure}

A \emph{blue}\index{blue path} path is a path 
all of whose vertices and edges are blue.  
A \emph{yellow}\index{yellow path} path is a path whose endnodes are blue and all inner vertices and edges yellow. 
For proving Theorem \ref{stem}, we view any path in 
$H_{\mathcal{I}}(N)$ as a succession of blue and yellow paths. So we need 
the following lemma.

\begin{lem}\label{lemstem1}
Let $(k,k')$ be a blue edge in $H_{\mathcal{I}} (N)$. Then $\mathcal{I}\cap s(N_{\bullet k})$ and $\mathcal{I}\cap s(N_{\bullet k'})$ are edge index subsets of a same basic directed path of $G(N)$.
\end{lem}

\begin{proof}
Since $N$ is nonnegative, $\mathcal{I}\cap s(N_{\bullet k})$ and $\mathcal{I}\cap s(N_{\bullet k'})$ are edge index sets of (basic) directed paths included in $p\cup p'$.
If  $k$ and $k'$ are linked by a blue
edge, then the fundamental cycles of $f_{k}$ and $f_{k'}$
have a common edge contained in $p$ or $ p'$, hence
$\mathcal{I}\cap s(N_{\bullet k})$ and $\mathcal{I}\cap 
s(N_{\bullet k'})$ are edge index subsets of a same directed path.
\end{proof}\\

Let $j$ and $j'$ be two blue vertices in $H_{\mathcal{I}}(N)$ and $h'$ a yellow minimal path from $j$ to $j'$. 
Up to column permutations, we may note $h'=(j=1,(1,2),2,\ldots, j')$. Since $h'$ is a yellow minimal path, it follows that $s(N_{\bullet 1})\cap 
s(N_{\bullet j'}) \cap \mathcal{I}= \emptyset.$ Then,
using the ordering of the edges in
$p\cup p'$, we may suppose that 
$i_{\max}=\max\{i\,:\,i\in s(N_{\bullet 1})\cap \mathcal{I}\}<i_{\min}=\min\{i\,:\,i\in s(N_{\bullet j'})\cap \mathcal{I}\}$ (the case $\min\{i\,:\,i\in s(N_{\bullet 1})\cap \mathcal{I}\}>\max\{i\,:\,i\in s(N_{\bullet j'})\cap \mathcal{I}\}$ is similar for the proof of next lemmas). 
We denote by $T$ the basic subgraph of $G(N)$ with edge index set $\cup^{j'}_{\beta =1} s(N_{\bullet \beta})$. Clearly, $T$ is connected. The following lemma asserts that $e_{i_{max}}$ and 
$e_{i_{min}}$ are adjacent.

\begin{lem}
The edges $e_{i_{max}}$ and $e_{i_{min}}$ are adjacent in $T$.
\end{lem}

\begin{proof}
Suppose, to the contrary, that $e_{i_{max}}$ and $e_{i_{min}}$ are not adjacent. Then there would exist a basic edge, say $e_{i_0}$, with $i_0\in R$ between 
$e_{i_{max}}$ and $e_{i_{min}}$ in $T$. By definition of $e_{i_{max}}$ and $e_{i_{min}}$, $e_{i_0}$ would not be in the fundamental cycle of $f_{1}$ and
$f_{j'}$. Furthermore, since $\beta$ is yellow for all $1<\beta<j'$, $e_{i_0}$ would not be in the fundamental cycle of $f_\beta$ for all $1<\beta<j'$. Therefore $e_{i_{max}}, e_{i_{min}}\in T$ and $e_{i_0}\notin T$. Thus $T$ would not be connected, which is a contradiction. 
\end{proof}\\

\begin{figure}[h!]
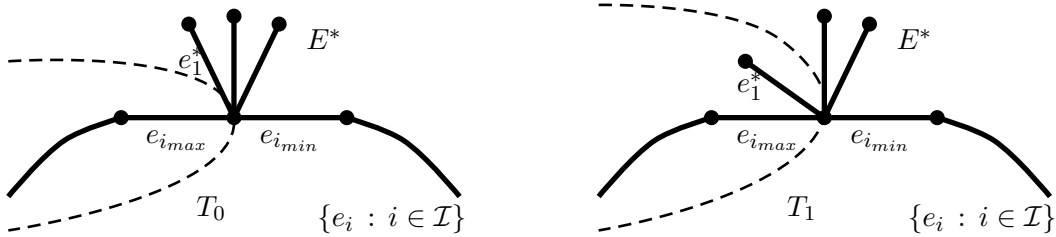

$\begin{array}{cc}
\psset{xunit=1.5cm,yunit=1.5cm,linewidth=0.5pt,radius=0.1mm,arrowsize=7pt,
labelsep=1.5pt,fillcolor=black}

\pspicture(0,0)(5,2)

\pscircle[fillstyle=solid](1,1){.1}
\pscircle[fillstyle=solid](2,1){.1}
\pscircle[fillstyle=solid](3,1){.1}
\pscircle[fillstyle=solid](2,1.9){.1}
\pscircle[fillstyle=solid](2.4,1.83){.1}
\pscircle[fillstyle=solid](1.6,1.83){.1}

\psline[linewidth=2pt]{-}(1,1)(2,1)
\rput(1.5,0.8){$e_{i_{max}}$}

\psline[linewidth=2pt]{-}(1.6,1.83)(2,1)
\rput(1.62,1.5){$e_1^*$}

\psline[linewidth=2pt]{-}(2,1.9)(2,1)

\psline[linewidth=2pt]{-}(2.4,1.83)(2,1)

\psline[linewidth=2pt]{-}(2,1)(3,1)
\rput(2.5,0.8){$e_{i_{min}}$}

\rput(2.8,1.7){$E^*$}

\pscurve[linewidth=2pt](0,0.3)(0.5,0.8)(1,1)

\pscurve[linewidth=2pt](3,1)(3.5,0.8)(4,0.3)
\rput(3.4,0.1){$\{e_i\,: \, i\in \mathcal{I}\}$}

\pscurve[linestyle=dashed,linewidth=1pt](0,1.5)(2,1)(0,0)
\rput(1.8,0.2){$T_0$}

\endpspicture &

\psset{xunit=1.5cm,yunit=1.5cm,linewidth=0.5pt,radius=0.1mm,arrowsize=7pt,
labelsep=1.5pt,fillcolor=black}

\pspicture(0,0)(5,2)

\pscircle[fillstyle=solid](1,1){.1}
\pscircle[fillstyle=solid](2,1){.1}
\pscircle[fillstyle=solid](3,1){.1}
\pscircle[fillstyle=solid](1.3,1.5){.1}
\pscircle[fillstyle=solid](2,1.9){.1}
\pscircle[fillstyle=solid](2.4,1.83){.1}

\psline[linewidth=2pt]{-}(1,1)(2,1)
\rput(1.5,0.8){$e_{i_{max}}$}

\psline[linewidth=2pt]{-}(1.3,1.5)(2,1)
\rput(1.34,1.28){$e_1^*$}

\psline[linewidth=2pt]{-}(2,1)(2,1.9)

\psline[linewidth=2pt]{-}(2.4,1.83)(2,1)

\psline[linewidth=2pt]{-}(2,1)(3,1)
\rput(2.5,0.8){$e_{i_{min}}$}

\rput(2.8,1.7){$E^*$}

\pscurve[linewidth=2pt](0,0.3)(0.5,0.8)(1,1)

\pscurve[linewidth=2pt](3,1)(3.5,0.8)(4,0.3)
\rput(3.4,0.1){$\{e_i\,: \, i\in \mathcal{I} \}$}

\pscurve[linestyle=dashed,linewidth=1pt](0,2)(1.6,1.7)(2,1)(0,0)
\rput(1.8,0.2){$T_1$}

\endpspicture 

\end{array}$

\caption{An illustration of $T_0$ and $T_1$.}
\label{fig:TimaxBare}
\end{figure}

Let $E^*$ be the set of basic edges in $G(N)$ incident to  $v_{i_{\max}}$. From the following lemma we may deduce that the fundamental cycle of any edge with index in $h'$ contains $v_{i_{max}}$.

\begin{lem}\label{lemthm1h}
For $1\le \beta \le j'$, the set $s(N_{\bullet \beta})$ contains the indexes of one edge entering $v_{i_{max}}$ and one edge leaving $v_{i_{max}}$. Moreover, for $1\le \beta < j'$, there is exactly one edge of $E^*$ with index in $s(N_{\bullet \beta})\cap s(N_{\bullet \beta+1})$.
\end{lem}

\begin{proof}
We prove the lemma by induction on $\beta$ (with $1\le \beta \le j'$).\\
Let $T_0$ be the connected component  of
$(T-E^*+\{e_{i_{max}}\})$ containing $e_{i_{max}}$. See Figure \ref{fig:TimaxBare}. Suppose by contradiction
that the basic edges with index in $s(N_{\bullet 1})$ are all in
$T_0$. Let $\delta_0=\min \{ \beta\in \mathbb{N} \, : \, \{e_i \, : \, i \in s(N_{\bullet \beta+ 1}) \} \not\subseteq T_0 \}$. Clearly, $1 \le \delta_0 < j'$. As $s(N_{\bullet {\delta_0}})
\cap s(N_{\bullet {\delta_0+1}})\neq \emptyset$, we deduce that $i_{max}\in 
s(N_{\bullet \delta_0+1})$. So $(1,\delta_0+1)$ is a blue edge in $H_\mathcal{I}(N)$, which contradicts the fact that $h'$ is either minimal or without any blue edge.
Let $e_1^*$ ($\neq e_{i_{max}}$) be the edge belonging to $E^*$ and  the
fundamental cycle of $f_{1}$. Denote by $T_1$ the connected component of 
$(T-E^*+\{e_{i_{max}},e_1^*\})$ containing $e_1^*$ (and $e_{i_{max}}$). Notice that $e_1^* \neq e_{i_{min}}$ since $(1,j')$ is not blue. Hence,
if $\{ e_i \, :\, i \in s(N_{\bullet 2}) \} \subseteq T_1$, then the node $1$ would be adjacent to $\min \{ \beta \in h' \, :\, \{ e_i \, : \, i \in s(N_{\bullet \beta}) \} \nsubseteq T_1 \} \neq 2$ in $H_\mathcal{I}(N)$, contradicting the minimality of $h'$. Since $(1,2)$ is yellow, it results that the fundamental cycle of $f_2$ contains $e_1^*$ and an edge in $E^*-T_1$.

Let $1\le \beta_0 < j'$ and suppose that the lemma is true for $\beta=1, \ldots, \beta_0$. Denote by $e_{\beta_0}^*$ and $e_{\beta_0'}^*$ the edges of $E^*$ in the fundamental cycle of $f_{\beta_0}$. Let $T_{\beta_0}$ be the connected component of $(T-E^*+ \{ e_{\beta_0}^*,e_{\beta_0'}^*\})$ containing $e_{\beta_0}^*$. As previously, if $\{e_i\, :\, i \in s(N_{\beta_0+1}) \} \subseteq T_{\beta_0}$, then $\beta_0$ would be adjacent to a vertex of $h'$ distinct from $\beta_0+1$, a contradiction. So, exactly one of the edges $e_{\beta_0}^*$ and $e_{\beta_0'}^*$ with one other edge of $E^*$ are
in the fundamental cycle of $f_{\beta_0+1}$. 
Since $N$ is nonnegative, for $1\le \beta \le j'$, one both edges belonging to the fundamental cycle of $f_\beta$ and $E^*$ is entering $v_{i_{max}}$ and the other is leaving $v_{i_{max}}$. This completes the proof.
\end{proof}\\

The proof of Theorem \ref{stem} is based on the following result (see Figure \ref{fig:EX}).

\begin{lem}\label{lemstem2}
The sets $\mathcal{I}\cap s(N_{\bullet j})$
and $\mathcal{I}\cap s(N_{\bullet j'})$ are edge index subsets of the same basic directed path of $G(N)$ if and only if $h'$ has an even number of (yellow) edges.\\
If $h'$ has an odd number of yellow edges, then the index $t$ (resp., $t+1$) belongs to $\mathcal{I}\cap s(N_{\bullet j})$ (resp., $\mathcal{I}\cap s(N_{\bullet j'})$) or vice versa, and the paths $p$ and $p'$ are both leaving $v$.
\end{lem}

\begin{proof}
Assume that $e_{i_{max}}$ is leaving $v_{i_{max}}$. By Lemma \ref{lemthm1h}, it results that $e_{i_{min}}$ is entering $v_{i_{max}}$ if and only if $h'$ has an even length.
We deduce that $h'$ has an even length if and only if
$\mathcal{I}\cap s(N_{\bullet 1})$ and $\mathcal{I}\cap s(N_{\bullet j'})$ 
are edge index subsets of a same directed path. In the case where $h'$ has an odd length, we have 
$e_{i_{\max}}=e_t$ and $e_{i_{\min}}=e_{t+1}$ (since $i_{max} < i_{min}$).
See Figure \ref{fig:EX}.
\end{proof}\\

\vspace{0.3cm}

\begin{figure}[h!]
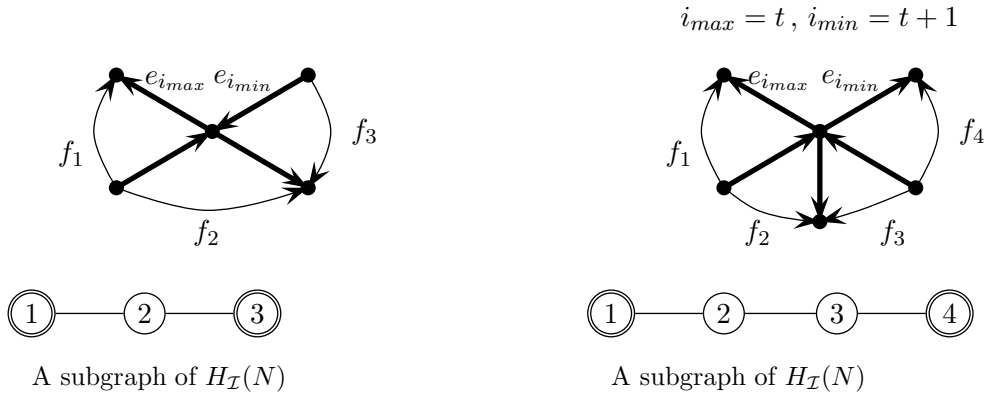

\begin{center}
$\begin{array}{ccc}
\psset{xunit=1.5cm,yunit=1.5cm,linewidth=0.5pt,radius=0.1mm,arrowsize=7pt,
labelsep=1.5pt,fillcolor=black}

\pspicture(0,0)(2.5,2)

\pscircle[fillstyle=solid](0,0.5){.1}
\pscircle[fillstyle=solid](0,1.5){.1}
\pscircle[fillstyle=solid](0.85,1){.1}
\pscircle[fillstyle=solid](1.7,1.5){.1}
\pscircle[fillstyle=solid](1.7,0.5){.1}

\psline[linewidth=2pt]{<-}(0,1.5)(0.85,1)
\rput(0.53,1.45){$e_{i_{max}}$}

\psline[linewidth=2pt]{->}(0,.5)(0.85,1)

\psline[linewidth=2pt]{->}(0.85,1)(1.7,.5)

\psline[linewidth=2pt]{<-}(0.85,1)(1.7,1.5)
\rput(1.13,1.45){$e_{i_{min}}$}

\pscurve{<-}(0,1.5)(-0.2,1)(0,.5)
\rput(-0.4,.8){$f_{1}$}

\pscurve{<-}(1.7,.5)(0.8,0.3)(0,.5)
\rput(0.8,0.1){$f_{2}$}

\pscurve{->}(1.7,1.5)(1.9,1)(1.7,.5)
\rput(2.2,1){$f_{3}$}

\endpspicture & 

\hspace{1cm} & 

\psset{xunit=1.5cm,yunit=1.5cm,linewidth=0.5pt,radius=0.1mm,arrowsize=7pt,
labelsep=1.5pt,fillcolor=black}

\pspicture(0,0)(2,2)

\pscircle[fillstyle=solid](0,0.5){.1}
\pscircle[fillstyle=solid](0,1.5){.1}
\pscircle[fillstyle=solid](0.85,1){.1}
\pscircle[fillstyle=solid](1.7,1.5){.1}
\pscircle[fillstyle=solid](1.7,0.5){.1}
\pscircle[fillstyle=solid](0.85,0.2){.1}

\psline[linewidth=2pt]{<-}(0,1.5)(0.85,1)
\rput(0.49,1.45){$e_{i_{max}}$}

\psline[linewidth=2pt]{->}(0,.5)(0.85,1)

\psline[linewidth=2pt]{->}(0.85,1)(0.85,.2)

\psline[linewidth=2pt]{<-}(0.85,1)(1.7,.5)

\psline[linewidth=2pt]{->}(0.85,1)(1.7,1.5)
\rput(1.13,1.45){$e_{i_{min}}$}

\pscurve{<-}(0,1.5)(-0.2,1)(0,.5)
\rput(-0.4,.8){$f_{1}$}

\pscurve{->}(0,.5)(0.3,0.3)(0.85,0.2)
\rput(0.3,0.1){$f_{2}$}

\pscurve{<-}(0.85,0.2)(1.2,0.3)(1.7,.5)
\rput(1.5,0.1){$f_{3}$}

\pscurve{<-}(1.7,1.5)(1.9,1)(1.7,.5)
\rput(2.2,1){$f_{4}$}

\rput(0.85,2){$i_{max}=t \, , \,i_{min}=t+1$}

\endpspicture \\
& & \\

\psset{xunit=1.5cm,yunit=1.5cm,linewidth=0.5pt,radius=0.1mm,arrowsize=7pt,
labelsep=1.5pt,fillcolor=black}

\pspicture(0,0)(4,0.5)

\cnodeput[doubleline=true](0,0.3){1}{1}
\cnodeput(1,0.3){2}{2}
\cnodeput[doubleline=true](2,0.3){3}{3}

\ncline{-}{2}{3}
\ncline{-}{1}{2}

\put(0,-0.5){\shortstack{\small{A subgraph of $ H_\mathcal{I}(N)$}}}

\endpspicture

 & &

\psset{xunit=1.5cm,yunit=1.5cm,linewidth=0.5pt,radius=0.1mm,arrowsize=7pt,
labelsep=1.5pt,fillcolor=black}

\pspicture(0,0)(4,0.5)

\cnodeput[doubleline=true](0,0.3){1}{1}
\cnodeput(1,0.3){2}{2}
\cnodeput(2,0.3){3}{3}
\cnodeput[doubleline=true](3,0.3){4}{4}

\ncline{-}{2}{3}
\ncline{-}{3}{4}
\ncline{-}{1}{2}

\put(0,-0.5){\shortstack{\small{A subgraph of $H_\mathcal{I}(N)$}}}

\endpspicture
\end{array}$

\end{center}

\caption{An illustration of Lemma \ref{lemstem2} in the case where $h'$ has an even length (on the left) and an odd length (on the right).}
\label{fig:EX}
\end{figure}

We are now able to prove Theorem \ref{stem}.\\

\noindent
{\bf Proof of Theorem \ref{stem}.}
Let $h$ be a minimal path between $k$ and $k'$.
By minimality, $h$ has at most one yellow subpath of odd length. Otherwise, using Lemma \ref{lemstem2}, $t$ or $t+1$ would be in the
support of two non-consecutive vertices of $h$, a contradiction. 
Viewing $h$ as a succession of blue subpaths and yellow subpaths, 
it follows from Lemmas \ref{lemstem1} and \ref{lemstem2} that $h$ has an even
number of yellow edges if and only if  $\mathcal{I}\cap s(N_{\bullet k})$
and $\mathcal{I}\cap s(N_{\bullet k'})$ are edge index subsets of a same (basic) directed path.
{\hfill$\BBox{\rule{.3mm}{3mm}}$}\\

In Figure \ref{fig7:thm1}, $\mathcal{I}\cap s(N_{\bullet 1})$ and $\mathcal{I}\cap s(N_{\bullet 4})$
are not edge index subsets of a same directed path, because the minimal path from 
\pscirclebox[linewidth=0.5pt]{1} to \pscirclebox[linewidth=0.5pt]{4} in $H_{\mathcal{I}}(N)$ has 
three yellow edges. But 
$\mathcal{I}\cap s(N_{\bullet 2})$ and $\mathcal{I}\cap s(N_{\bullet 5})$
are edge index subsets of a same directed path ($p'$), since the minimal path from 
\pscirclebox[linewidth=0.5pt]{2} to \pscirclebox[linewidth=0.5pt]{5} in $H_{\mathcal{I}}(N)$ has 
two yellow edges.

\section{Bonsai matrices}\label{sec:bonsaimat}

Let $A$ be a nonnegative connected matrix, $R^*$ a row index subset of $A$ and 
$E_1,\ldots,E_b$ a partitioning of $\overline{ R^*}$ as described in the introduction of the chapter.
The object of the present section is to compute subsets of 
$E_\ell$ called $E_\ell$-paths, and to define a bonsai matrix associated with $E_\ell$, for all $1\le \ell \le b$.
We will also prove some results about bonsai matrices.

For all $1\le \ell\le b$ and $\beta\in f^*(E_\ell)$, let us compute subsets of $E_\ell$ as follows. 
If $|E_\ell| > 1$, then by the way of partitioning $\overline{R^*}$, for all $i\in E_\ell$, there exists a column index $j_0 \in \bar f^*(E_\ell)$ such that $i\in s(A_{\bullet j_0}) $. Let us consider the following procedure.

\begin{tabbing}
\textbf{Procedure\,\,E$\ell$path($A$,$E_\ell$,$\beta$)}\\

\textbf{Input:} A matrix $A$, a bonsai $E_\ell$ ($1\le \ell \le b$) and some $\beta \in f^*(E_\ell)$.\\
\textbf{Output:} Row index subsets $E_\ell^I(A_{\bullet \beta})$ and 
$E_\ell^{II}(A_{\bullet \beta})$ of $E_\ell$.\\
 
\verb"  "\= let $\mathcal{I}=s_1(A_{\bullet \beta}) \cap E_\ell$; \\
\> {\bf if }\=  $|\mathcal{I}| > 1$, {\bf then } \\
\> \> let $j_0 \in \bar f^*(E_\ell)$ such that $s(A_{\bullet j_0}) \cap \mathcal{I}\neq \emptyset$ and  $S_0=\{j_0 \}$;\\
\> \> compute a spanning tree $T_{ \mathcal{I}}$ in 
$H_\mathcal{I}(A_{E_\ell \times \bar f^*(E_\ell)})$ rooted at $j_0$ such that  \\
\> \> for each $j\in V(T_\mathcal{I})$, the path from $j$ to $j_0$ in $T_\mathcal{I}$ is minimal in $H_\mathcal{I}(A_{E_\ell \times \bar f^*(E_\ell)})$;\\
\> \> {\bf for }\=  every $j\in \bar f^*(E_\ell)$ such that $s(A_{\bullet j}) \cap \mathcal{I} \neq \emptyset$, {\bf do} \\
 \> \> \> if the path from $j_0$ to $j$ in $T_\mathcal{I}$ has an even number of yellow edges, \\
 \> \> \>  then add $j$ in $S_0$; \\
\> \> {\bf endforÊ}\\
\> \> let $\mathcal{I}^I=s(A_{\bullet S_0})\cap \mathcal{I}$ and $\mathcal{I}^{II}=\mathcal{I}\verb"\" \mathcal{I}^I$; \\
\> {\bf otherwise } \\
\> \> let $\mathcal{I}^I= \mathcal{I}$ and $\mathcal{I}^{II}=\emptyset$; \\
\> {\bf endif }\\
\> let $E_\ell^I(A_{\bullet \beta})=\{i\in E_\ell\,:\,
A_{i\beta} =2\}\cup \mathcal{I}^I$,  $E_\ell^{II}(A_{\bullet \beta})=\{i\in E_\ell\,:\,
A_{i\beta} =2\}\cup  \mathcal{I}^{II}$;\\
\>output $E_\ell^I(A_{\bullet \beta})$ and 
$E_\ell^{II}(A_{\bullet \beta})$;
\end{tabbing}

\noindent
For all $1\le \ell \le b$ and $\beta\in f^*(E_\ell)$, since $s_{\frac{1}{2}}(A_{\bullet \beta})= R^*$ by assumption (see page \pageref{ch:multidiD}), we observe that 
$s_1(A_{\bullet \beta})\neq \emptyset$ or $s_2(A_{\bullet \beta})\neq \emptyset$, hence $E_\ell^I(A_{\bullet \beta})\neq \emptyset$. The procedure above is motivated by the following lemma.

\begin{lem}\label{lemdigraphBlI}
Suppose that $A$ has an $R^*$-cyclic, $R^*$-central or $R^*$-network representation $G(A)$. Then for any bonsai $E_\ell$ ($1\le \ell \le b$) and $\beta \in f^*(E_\ell)$,  the row index subsets  $E_\ell^I(A_{\bullet \beta})$ and $E_\ell^{II}(A_{\bullet \beta})$ (if not empty) output by the procedure E$\ell$path are the edge index sets of $B_\ell$-paths generated by $f_\beta$.
\end{lem}

\begin{proof}
Suppose that $A$ is $R^*$-cyclic (the cases $R^*$-central or $R^*$-network are similar or simpler).
Let $1\le \ell\le b$, $\beta \in f^*(E_\ell)$ and $\mathcal{I}=s_1(A_{\bullet \beta}) \cap E_\ell$.
Notice that if $|\mathcal{I}| > 1$, then the way of partitioning $\overline{R^*}$ implies that for all $i\in \mathcal{I}$ there exists some $j_0 \in \bar f^*(E_\ell)$ such that $A_{ij_0}=1$. By Lemmas \ref{lemdefiWeight1} and \ref{lemdefiWeight2} and the description of the different types of fundamental circuits given in page \pageref{mycounter2} (see 
 Figure \ref{fig:Apositive}), it results that $\mathcal{I}$ is the edge index set of a directed path, say $p$, or two directed paths, say $p$ and $p'$, having their initial node in common.
Furthermore, if $s_2(A_{\bullet \beta})\cap E_\ell\neq \emptyset$, then the set $s_2(A_{\bullet \beta})\cap E_\ell$ is  the edge index set of a directed path, say $p''$, and
the initial node of $p$ and $p'$ coincides with the terminal node of $p''$ (the terminal node of $p''$ is the unique node of $p''$ in common with  $p$ and $p'$).
Then the proof of Lemma \ref{lemdigraphBlI} follows from Theorem \ref{stem}.
\end{proof}\\

For any bonsai $E_\ell$, let $E_\ell^{I}(A_{\bullet \beta})$ and $E_\ell^{II}(A_{\bullet \beta})$ be output by the procedure E$\ell$path and $m(\ell)$\index{$m(\ell)$}  the cardinality of the set 
$\{E_\ell^{I}(A_{\bullet \beta}), E_\ell^{II}(A_{\bullet \beta}):\,\,\, \beta \in f^*(E_\ell)\}\verb"\"
\{\emptyset\}$, and denote by  $E^1_\ell,E^2_\ell,\ldots,E^{m(\ell)}_\ell$ all distinct elements of this set. 
For all $1\le k \le m(\ell)$, $E_\ell^k$ is called an \emph{$E_\ell$-path}\index{path@$E_\ell$-path}, and let 
$f^*(E^k_\ell)=\{\beta \in f^*(E_\ell) \,\,:\,\, E^k_\ell= E_\ell^{I}(A_{\bullet \beta})\m{ or }
E_\ell^{II}(A_{\bullet \beta})\}$. By Lemma \ref{lemdigraphBlI}, if $A$ has an $R^*$-cyclic, $R^*$-central or $R^*$-network representation $G(A)$, then for any $1 \le \ell \le b$, $1\le k \le m(\ell)$ and $\beta \in f^*(E_\ell^k)$, the $E_\ell$-path $E_\ell^k$ is the edge index set of a $B_\ell$-path\index{path@$B_\ell$-path} generated by $f_\beta$, denoted as $B_\ell^k$.
For any $\beta \in f^*(E_\ell^k)$, we say that \emph{$A_{\bullet \beta}$ generates the $E_\ell$-path $E_\ell^k$}\index{path@$E_\ell$-path generated by $A_{\bullet \beta}$}. See Figure \ref{fig:digraphD}.

\begin{figure}[t!]
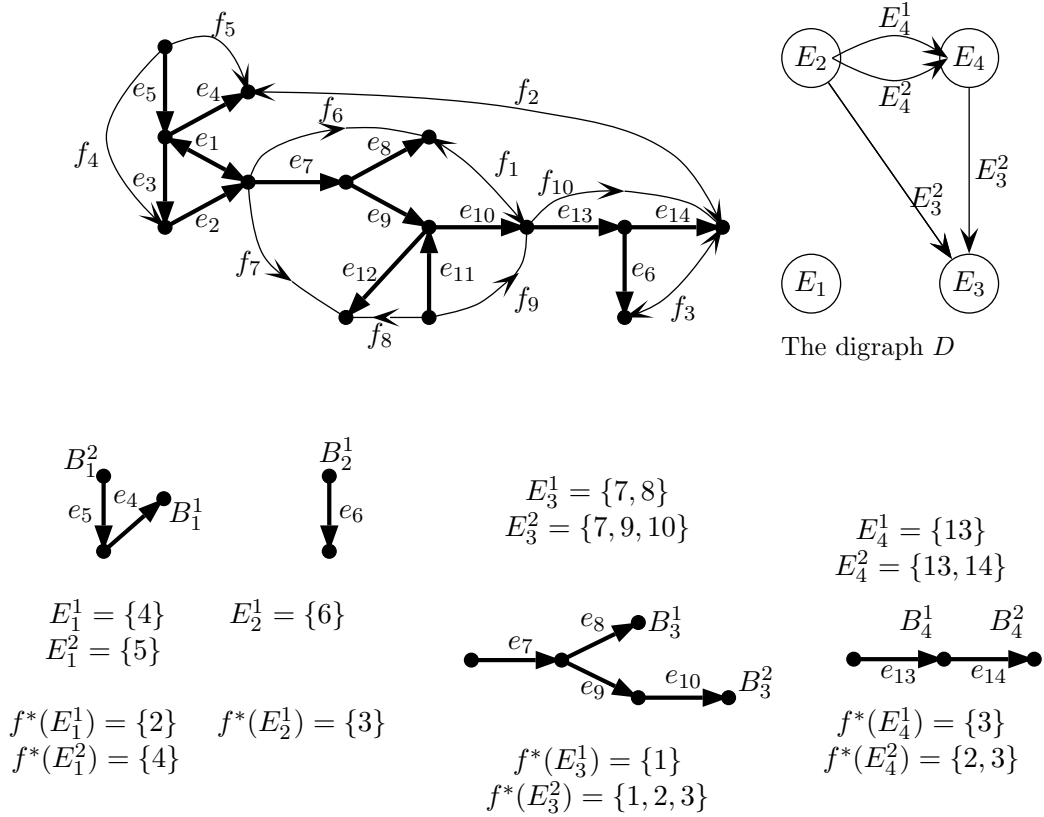


\vspace{0.8cm}
\begin{center}
$
\begin{array}{cc}
\psset{xunit=1.3cm,yunit=1.2cm,linewidth=0.5pt,radius=0.1mm,arrowsize=7pt,labelsep=1.5pt,fillcolor=black}

\pspicture(0,0)(5.8,3)

\pscircle[fillstyle=solid](0,1){.1}
\pscircle[fillstyle=solid](0,2){.1}
\pscircle[fillstyle=solid](0,3){.1}
\pscircle[fillstyle=solid](0.85,2.5){.1}
\pscircle[fillstyle=solid](0.85,1.5){.1}
\pscircle[fillstyle=solid](1.85,1.5){.1}
\pscircle[fillstyle=solid](1.85,0){.1}
\pscircle[fillstyle=solid](2.7,1){.1}
\pscircle[fillstyle=solid](2.7,0){.1}
\pscircle[fillstyle=solid](2.7,2){.1}
\pscircle[fillstyle=solid](3.7,1){.1}
\pscircle[fillstyle=solid](4.7,1){.1}
\pscircle[fillstyle=solid](4.7,0){.1}
\pscircle[fillstyle=solid](5.7,1){.1}

\psline[linewidth=1.6pt,arrowinset=0]{<->}(0,2)(0.85,1.5)
\rput(0.44,1.95){$e_1$}

\psline[linewidth=1.6pt,arrowinset=0]{->}(0,1)(0.85,1.5)
\rput(0.44,1.05){$e_2$}

\psline[linewidth=1.6pt,arrowinset=0]{<-}(0,1)(0,2)
\rput(-0.2,1.5){$e_3$}

\psline[linewidth=1.6pt,arrowinset=0]{->}(0,2)(0.85,2.5)
\rput(0.45,2.5){$e_4$}

\psline[linewidth=1.6pt,arrowinset=0]{<-}(0,2)(0,3)
\rput(-0.2,2.5){$e_5$}

\psline[linewidth=1.6pt,arrowinset=0]{->}(4.7,1)(4.7,0)
\rput(4.9,0.5){$e_6$}

\psline[linewidth=1.6pt,arrowinset=0]{->}(0.85,1.5)(1.85,1.5)
\rput(1.4,1.7){$e_7$}

\psline[linewidth=1.6pt,arrowinset=0]{->}(1.85,1.5)(2.7,2)
\rput(2.2,1.9){$e_8$}

\psline[linewidth=1.6pt,arrowinset=0]{->}(1.85,1.5)(2.7,1)
\rput(2.2,1.1){$e_9$}

\psline[linewidth=1.6pt,arrowinset=0]{->}(2.7,1)(3.7,1)
\rput(3.2,1.15){$e_{10}$}

\psline[linewidth=1.6pt,arrowinset=0]{->}(2.7,0)(2.7,1)
\rput(3,0.5){$e_{11}$}

\psline[linewidth=1.6pt,arrowinset=0]{<-}(1.85,0)(2.7,1)
\rput(2,0.5){$e_{12}$}

\psline[linewidth=1.6pt,arrowinset=0]{->}(3.7,1)(4.7,1)
\rput(4.2,1.15){$e_{13}$}

\psline[linewidth=1.6pt,arrowinset=0]{->}(4.7,1)(5.7,1)
\rput(5.2,1.15){$e_{14}$}

\pscurve[arrowinset=.5,arrowlength=1.5]{<->}(2.7,2)(3.1,1.7)(3.7,1)
\rput(3.5,1.7){$f_1$}

\pscurve[arrowinset=.5,arrowlength=1.5]{<->}(0.85,2.5)(3.7,2.3)(5.2,1.8)(5.7,1)
\rput(3.7,2.5){$f_2$}

\pscurve[arrowinset=.5,arrowlength=1.5]{<->}(4.7,0)(5.1,0.2)(5.7,1)
\rput(5.3,0.1){$f_3$}

\pscurve[arrowinset=.5,arrowlength=1.5]{->}(0,3)(-.6,2)(0,1)
\rput(-.8,1.8){$f_4$}

\pscurve[arrowinset=.5,arrowlength=1.5]{->}(0,3)(.5,3.1)(0.85,2.5)
\rput(.6,3.25){$f_5$}

\pscurve[arrowinset=.5,arrowlength=1.5]{->}(0.85,1.5)(1,1.8)(1.85,2.1)
\pscurve[arrowinset=.5,arrowlength=1.5]{-}(1.85,2.1)(2.2,2.1)(2.7,2)
\rput(1.7,2.3){$f_6$}

\pscurve[arrowinset=.5,arrowlength=1.5]{->}(0.85,1.5)(1,0.7)(1.3,0.4)
\pscurve[arrowinset=.5,arrowlength=1.5]{-}(1.3,0.4)(1.31,0.39)(1.85,0)
\rput(0.85,0.6){$f_7$}

\psline[arrowinset=.5,arrowlength=1.5]{-<}(1.85,0)(2.4,0)
\psline[arrowinset=.5,arrowlength=1.5]{-}(2.3,0)(2.7,0)
\rput(2.2,-0.2){$f_8$}

\pscurve[arrowinset=.5,arrowlength=1.5]{-<}(3.7,1)(3.65,0.6)(3.4,0.3)
\pscurve[arrowinset=.5,arrowlength=1.5]{-}(3.5,0.4)(3.2,0.15)(2.7,0)
\rput(3.7,0.2){$f_9$}

\pscurve[arrowinset=.5,arrowlength=1.5]{->}(3.7,1)(4,1.3)(4.7,1.4)
\pscurve[arrowinset=.5,arrowlength=1.5]{-}(4.7,1.4)(5.4,1.25)(5.7,1)
\rput(4,1.5){$f_{10}$}

\endpspicture  &

\psset{xunit=1.4cm,arrows=->,yunit=1.5cm,linewidth=0.5pt,radius=0.1mm,arrowsize=7pt,
labelsep=1.5pt,fillcolor=black}

\pspicture(0,0)(2,2)

\put(0.3,-0.5){\small{The digraph $D$}}

\cnodeput(0.5,0.3){1}{$E_1$}
\cnodeput(2,0.3){3}{$E_3$}
\cnodeput(0.5,2.3){2}{$E_2$}
\cnodeput(2,2.3){4}{$E_4$}

\ncline{2}{3}
\ncline{2}{3}
\rput(1.6,1.05){$E_3^2$}
\pscurve(0.7,2.3)(1.3,2.5)(1.8,2.3)
\rput(1.3,2.65){$E_4^1$}
\pscurve(0.7,2.3)(1.3,2.1)(1.8,2.3)
\rput(1.3,1.95){$E_4^2$}
\ncline{4}{3}
\naput{$E_3^2$}
      
\endpspicture 

\end{array}$
\end{center}

\vspace{1.5cm}

$
\begin{array}{ccc}

\begin{array}{c}

\psset{xunit=1cm,yunit=1cm,linewidth=0.5pt,radius=0.1mm,arrowsize=7pt,
labelsep=1.5pt,fillcolor=black}

\pspicture(0,0)(3.5,.5)

\pscircle[fillstyle=solid](.5,0){.1}
\pscircle[fillstyle=solid](.5,1){.1}
\pscircle[fillstyle=solid](1.3,0.7){.1}
\pscircle[fillstyle=solid](3.5,0){.1}
\pscircle[fillstyle=solid](3.5,1){.1}


\psline[linewidth=1.6pt,arrowinset=0]{->}(.5,0)(1.3,0.7)
\psline[linewidth=1.6pt,arrowinset=0]{<-}(.5,0)(.5,1)
\rput(.8,.7){$e_{4}$}
\rput(.2,.5){$e_{5}$}
\rput(1.6,0.5){$B_1^1$}
\rput(.2,1.2){$B_1^2$}


\psline[linewidth=1.6pt,arrowinset=0]{<-}(3.5,0)(3.5,1)
\rput(3.8,0.5){$e_{6}$}
\rput(3.6,1.3){$B_2^1$}

\endpspicture \\ \\

E_1^1=\{4\}\, \,\,\,\,\,\quad \, E_2^1=\{6\}\\
E_1^2=\{5\}\hspace{2.5cm}\\ \\
f^*(E_1^1)=\{2\}\, \quad \, f^*(E_2^1)=\{3\}\\
f^*(E_1^2)=\{4\} \hspace{2.7cm}\\

\end{array} &

\begin{array}{c}

E^1_3=\{7,8\} \\
E^2_3=\{7,9,10\} \\ 

\psset{xunit=1.2cm,yunit=1cm,linewidth=0.5pt,radius=0.1mm,arrowsize=7pt,
labelsep=1.5pt,fillcolor=black}

\pspicture(.5,0)(4,2)

\pscircle[fillstyle=solid](0.85,.5){.1}
\pscircle[fillstyle=solid](1.85,.5){.1}
\pscircle[fillstyle=solid](2.7,0){.1}
\pscircle[fillstyle=solid](2.7,1){.1}
\pscircle[fillstyle=solid](3.7,0){.1}

\psline[linewidth=1.6pt,arrowinset=0]{->}(0.85,.5)(1.85,.5)
\rput(1.4,.7){$e_7$}

\psline[linewidth=1.6pt,arrowinset=0]{->}(1.85,.5)(2.7,1)
\rput(2.2,1){$e_8$}

\psline[linewidth=1.6pt,arrowinset=0]{->}(1.85,.5)(2.7,0)
\rput(2.2,.1){$e_9$}

\psline[linewidth=1.6pt,arrowinset=0]{->}(2.7,0)(3.7,0)
\rput(3.2,.25){$e_{10}$}



\rput(3,1){$B_3^1$}
\rput(4,0.2){$B_3^2$}

\endpspicture \\ \\ 

f^*(E_3^1)=\{1\}\\
f^*(E_3^2)=\{1,2,3\}\\

\end{array}  &

\begin{array}{c}

E_4^1=\{13\}\\
E_4^2=\{13,14\}\\ \\

\psset{xunit=1.2cm,yunit=1cm,linewidth=0.5pt,radius=0.1mm,arrowsize=7pt,
labelsep=1.5pt,fillcolor=black}

\pspicture(0,0)(2.5,.5)

\pscircle[fillstyle=solid](.5,0){.1}
\pscircle[fillstyle=solid](1.5,0){.1}
\pscircle[fillstyle=solid](2.5,0){.1}

\psline[linewidth=1.6pt,arrowinset=0]{->}(.5,0)(1.5,0)
\rput(1,-0.2){$e_{13}$}

\psline[linewidth=1.6pt,arrowinset=0]{->}(1.5,0)(2.5,0)
\rput(2,-.2){$e_{14}$}

\rput(1.2,0.5){$B_4^1$}
\rput(2.2,.5){$B_4^2$}

\endpspicture 

\\ \\ 

f^*(E_4^1)=\{3\}\\
f^*(E_4^2)=\{2,3\}\\

\end{array}

\end{array}
$

\vspace{.3cm}
\caption{An $R^*$-cyclic representation of the matrix $A$ given in Figure \ref{fig:AH}, the digraph $D$ and all computed $E_\ell$-paths and corresponding $B_\ell$-paths, where $E_1,\ldots,E_4$ are defined in Figure \ref{fig:AH}.}
\label{fig:digraphD}
\end{figure}

Given the set $\{E_\ell^1,\ldots,E_\ell^{m(\ell)}\}$ of $E_\ell$-paths for some $1\le \ell \le b$, the following procedure partitions this 
set into two classes. A graphical interpretation
of these classes is given in Lemma \ref{lemdigraphutile} below.

\begin{tabbing}
\textbf{Procedure\,\,TwoClasses($A$,$\{E_l^1,\ldots,E_l^{m(l)}\}$)}\\

\textbf{Input:} A matrix $A$ and a set of $E_\ell$-paths $\{E_\ell^1,\ldots, E_\ell^{m(\ell)}\}$ for some $1\le \ell \le b$.\\
\textbf{Output:} Two sets $J_\ell^1$ and $J_\ell^2$ such that $\{E_\ell^1,\ldots, E_\ell^{m(\ell)}\}= J_\ell^1 \biguplus J_\ell^2$.\\ 

\verb"  "\= let $J_\ell^1= \{ E_\ell^1\}$; $J_\ell^2= \{ \} $
and $j_1\in \bar f^*(E_\ell)$ such that $s(A_{\bullet j_1}) \cap E_\ell^1 \neq \emptyset$;\\
\>  compute a spanning tree $T$ in 
$H(A_{E_\ell \times \bar f^*(E_\ell)})$ rooted at $j_1$ such that for each $j\in V(T)$\\
\>  the path from $j$ to $j_1$ in $T$ is minimal in $H(A_{E_\ell \times \bar f^*(E_\ell)})$;\\
 \> {\bf for }\=  $k=2,\ldots, m(\ell)$, {\bf do}\\ 
 \> \>let $j_k \in V(T)$  such that $s(A_{\bullet j_k})\cap E_\ell^k\neq \emptyset$ and $h$ be the path from $j_1$ to $j_k$ in $T$;\\
 \> \> color $h$ as if $T\subseteq H_{E_\ell^1 \cup E_\ell^k}(A_{ E_\ell \times \bar f^*(E_\ell)})$;\\
\> \> {\bf if }\= $E_\ell^1 \cap E_\ell^k \neq \emptyset$ or
$h$ has an odd number of yellow edges, {\bf then}\\
\> \> \> let $J_\ell^1=J_\ell^1 \cup \{ E_\ell^k \}$;\\
\> \>  {\bf otherwise }\\
\> \> \> let $J_\ell^2=J_\ell^2 \cup \{ E_\ell^k \}$;\\
\> \> {\bf endif }\\
 \> {\bf endfor }\\
\> output $J_\ell^1$ and $J_\ell^2$;
\end{tabbing}

For all $1\le \ell \le b$, let $J_\ell^1$\label{mycounter6} and $J_\ell^2$ be output by the procedure TwoClasses; for any $E_\ell$-paths $E_\ell^k$ and $E_\ell^{k'}$ ($1\le k,k' \le m(\ell)$), we define the following relation\index{relation!$\sim_{E_\ell}$}: $E_\ell^k \sim_{E_\ell} E_\ell^{k'}$ if and only if $E_\ell^k$ and $E_\ell^{k'}$ belong to the same class ($J_\ell^1$ or $J_\ell^2$). 
In Figure \ref{fig:digraphD}, $E_1^1$ and $E_1^2$ are two $E_1$-paths. Since $s(A_{\bullet 5}) \cap E_1^k \neq \emptyset$ for $k=1$ and $2$ and the constant path at node $5$ in $H_{E_1^1 \cup E_1^2} (A_{E_1 \times \bar f^*(E_1)})$ has zero yellow edges, $E_1^1 \nsim_{E_1} E_1^2$ and $J_1^k=\{E_1^k\}$ for $k=1$ and $2$.

\begin{lem}\label{lemdigraphutile}
Suppose that $A$ has an $R^*$-cyclic, $R^*$-central or $R^*$-network representation $G(A)$. Let $1\le \ell\le b$ and suppose that all $B_\ell$-paths in $G(A)$ have a common endnode say $v_\ell$. Then for all $1\le k, k' \le m(\ell)$, the $B_\ell$-paths  $B_\ell^k$ and $B_\ell^{k'}$ both enter, or leave $v_\ell$ if and only if $E_\ell^k \sim_{E_\ell} E_\ell^{k'}$.
\end{lem}

\begin{proof}
By assumption, all $B_\ell$-paths in $G(A)$ have a common endnode $v_\ell$. Let  $1\le k,k'\le m(\ell)$. By Theorem \ref{stem}, the $B_\ell$-paths $B_\ell^k$ and $B_\ell^{k'}$ both  enter, or leave $v_\ell$ if and only if $E_\ell^k \cap E_\ell^{k'}\neq \emptyset$, or for any $j,j'$ such that $s(A_{\bullet j})\cap E_\ell^k\neq \emptyset$ and $s(A_{\bullet j'})\cap E_\ell^{k'}\neq \emptyset$, all minimal paths in $H_{E_\ell^{k} \cup E_\ell^{k'}}(A_{ E_\ell \times \bar f^*(E_\ell)})$ between $j$ and $j'$ have an odd number of yellow edges. Hence, by construction of $J_\ell^1$ and $J_\ell^2$, the $B_\ell$-paths $B_\ell^k$ and $B_\ell^{k'}$ both enter, or leave $v_\ell$ if and only if $E_\ell^k \sim_{E_\ell} E_\ell^{k'}$. 
\end{proof}\\

For all $1\le \ell \le b$, let $n(\ell)$ denote the cardinality of the set $J_\ell^1$.
Without loss of generality, we may assume that $J_\ell^1=\{E_\ell^1,\ldots, E_\ell^{n(\ell)}\}$. 
Notice that in general the bipartition $J_\ell^1 \uplus J_\ell^2$ given by the procedure TwoClasses may be not unique.

Now we can define bonsai matrices.
For all $1\le \ell \le b$, if $J_\ell^2=\emptyset$, then let  
$$N_\ell=\left( \begin{array}{cc}
A_{E_\ell\times \bar f^*(E_\ell)} &  \chi_{E_\ell^1}^{E_\ell} \cdots \chi_{E_\ell^{m(\ell)}}^{E_\ell} \\
\bf{ 0_{1\times |\bar f^*(E_\ell)|}} & \bf{ 1_{1\times m(\ell) }}
\end{array} \right); $$
otherwise ($J_\ell^1\neq \emptyset$ and $J_\ell^2\neq \emptyset$), let $$N_\ell=\left( \begin{array}{cccc}
A_{E_\ell\times \bar f^*(E_\ell)} &  \chi_{E_\ell^1}^{E_\ell} \cdots \chi_{E_\ell^{n(\ell)}}^{E_\ell} &
\chi_{E_\ell^{n(\ell)+1}}^{E_\ell} \cdots \chi_{E_\ell^{m(\ell)}}^{E_\ell} & \bf{ 0_{|E_\ell|\times 1}} \\
\bf{ 0_{1\times |\bar f^*(E_\ell)|}} & \bf{ 1_{1\times n(\ell) }} & 
\bf{ 0_{1\times (m(\ell)-n(\ell)) }} & 1 \\
\bf{ 0_{1\times |\bar f^*(E_\ell)|}} & \bf{ 0_{1\times n(\ell) }} & 
\bf{ 1_{1\times (m(\ell)-n(\ell)) }} & 1 
\end{array} \right). $$

\noindent
The matrix $N_\ell$ is called the \emph{bonsai}\index{bonsai matrix} matrix associated with the bonsai $E_\ell$. Any row of $N_\ell$ of type $[A_{\{i \} \times \bar f^*(E_\ell) } \cdots ]$ for some $i\in E_\ell$ is indexed by $i$. A row (respectively, a column) of $N_\ell$ that is not indexed by an element in $E_\ell$ (respectively, $\bar f^*(E_\ell)$) is said to be   \emph{artificial}\index{artificial edge} as well as the corresponding edge in any network representation of $N_\ell$ (provided that such a representation exists). 
In the case where $R^*=\{i^*\}$ and $J_\ell^2= \emptyset $, the index of the artificial row of $N_\ell$ is denoted as $i^*$.

\begin{lem}\label{lembonsaicel}
Let $1\le \ell\le b$ and suppose that $A$ has an $R^*$-cyclic, $R^*$-central or $R^*$-network representation $G(A)$. If all $B_\ell$-paths in $G(A)$ have a common endnode $v_\ell$, then the bonsai matrix $N_\ell$ is a network matrix.
\end{lem}

\begin{proof}
Suppose first that $J_\ell^2\neq \emptyset$. Let us denote $B_\ell'$ the subtree $B_\ell$ of $G(A)$ with two more basic directed edges (one entering $v_\ell$ and the other one leaving $v_\ell$). Since $v_\ell$ is an endnode of every $B_\ell$-path, by Lemma \ref{lemdigraphutile}, 
we deduce that each column of $N_\ell$ is the edge incidence vector of a directed path in  $B_\ell'$.

Now suppose that $J_\ell^2= \emptyset$. This implies that either all $B_\ell$-paths of $G(A)$ enter $v_\ell$ or leave $v_\ell$. Let $B_\ell'$ be the subtree $B_\ell$ of $G(A)$ with one more basic directed edge which either enters $v_\ell$ (if all $B_\ell$-paths leave $v_\ell$) or leaves $v_\ell$ (otherwise) and we conclude as before.
\end{proof}\\

Let $1\le \ell\le b$ and suppose that $N_\ell$ has a basic network representation $G(N_\ell)$. In the case where $J_\ell^2=\emptyset$, since $A_{E_\ell\times \bar f^*(E_\ell)}$ is a connected matrix, the artificial edge, call it $\tilde e$, is an end-edge in $G(N_\ell)$. We denote by $v_\ell$  the endnode of $\tilde e$ which is a cutvertex. In the case where $J_\ell^2 \neq \emptyset $, due to the last column of $N_\ell$, we may denote by $v_\ell$\index{vertex!$v_\ell$} the node in $G(N_\ell)$ incident with both artificial edges. 

\begin{lem}\label{lembonsainet2}
Let $1\le \ell\le b$ and suppose that the bonsai matrix $N_\ell$ has a network representation $G(N_\ell)$. Then, for all $1\le k \le m(\ell)$, $E_\ell^k$ is the edge index set of a path with one endnode equal to $v_\ell$. Moreover, for any $1\le k, k' \le m(\ell)$, both paths in $G(N_\ell)$ with edge index sets $E_\ell^k$ and $E_\ell^{k'}$ enter or leave $v_\ell$ if and only if $E_\ell^k \sim_{E_\ell} E_\ell^{k'}$.
\end{lem}

\begin{proof}
Suppose $J_\ell^2\neq \emptyset$ (the case $J_\ell^2=\emptyset$ is similar and simpler). 
Up to a reversing of the orientation of all edges, we may suppose that the first artificial edge goes from a node $x$ to $v_\ell$ and the second one (corresponding to the last row) is equal to $(v_\ell,y)$.
If $x$ is a node of $ B_\ell$, then each column of the matrix 
$\left( \begin{array}{c}
\chi_{E_\ell^{n(\ell)+1}}^{E_\ell} \cdots \chi_{E_\ell^{m(\ell)}}^{E_\ell} \\
\bf{ 0_{1\times (m(\ell)-n(\ell)) }} \\
\bf{ 1_{1\times (m(\ell)-n(\ell)) }}  
\end{array} \right) $
can not be the edge incidence vector of a path in $G(N_\ell)$. We have a similar contradiction with $y\in B_\ell$ instead of $x\in B_\ell$. So $v_\ell$ belongs to $B_\ell$. Since $(x,v_\ell)$ is entering $v_\ell$ and $[(\chi_{E_\ell^k}^{E_\ell})^T \, 1 \, 0 ]^T$ is the incidence vector of a path for $k=1,\ldots, n(\ell)$, it follows that $E_\ell^k$ is the edge index set of a path leaving $v_\ell$ for $k=1,\ldots, n(\ell)$. Similarly, we prove that for $k=n(\ell) +1,\ldots, m(\ell)$, $E_\ell^k$ is the edge index set of a path entering $v_\ell$. 
\end{proof}\\

If the matrix $N_\ell$ has a network representation $G(N_\ell)$, by Lemma \ref{lembonsainet2} and for ease of notation, the path in $G(N_\ell)$ with edge index set $E_\ell^k$ ($1\le k \le m(\ell)$) is called a \emph{$B_\ell$-path}\index{path@$B_\ell$-path}.
Moreover, the basic connected subgraph of $G(N_\ell)$ with edge index set $E_\ell$ is called a \emph{bonsai}\index{bonsai $B_\ell$} and is denoted $B_\ell$.
We say that the bonsai $B_\ell$ is a \emph{$v_\ell$-rooted network representation of $N_\ell$}\index{rooted@$v_\ell$-rooted!network representation}. Given $N_\ell$ and a $v_\ell$-rooted network representation of $N_\ell$ it is easy to construct a network representation of $N_\ell$.

\begin{lem}\label{lembonsaicontra}
Let $1\le \ell\le b$ and $1\le k,k'\le m(\ell)$. Suppose that the bonsai matrix $N_\ell$ has a $v_\ell$-rooted network representation $B_\ell$. Suppose also that the matrix $A_{E_\ell\times \bar f^*(E_\ell)}$ has a network representation $G(A_{E_\ell\times \bar f^*(E_\ell)})$ in which $E_\ell^k$ and $E_\ell^{k'}$ are edge index sets of directed paths having exactly one common endnode, say $z_\ell$. Then the paths with edge index sets $E_\ell^k$ and $E_\ell^{k'}$ in $G(A_{E_\ell\times \bar f^*(E_\ell)})$  both enter, or leave $z_\ell$ if and only if $E_\ell^k \sim_{E_\ell} E_\ell^{k'}$.
\end{lem}

\begin{proof}
By Theorem \ref{stem}, the $B_\ell$-paths  $B_\ell^k$ and $B_\ell^{k'}$ in $B_\ell$ both enter, or leave $v_\ell$ if and only if $E_\ell^k \cap E_\ell^{k'}\neq \emptyset$, or for any $j,j'$ such that $s(A_{\bullet j})\cap E_\ell^k\neq \emptyset$ and $s(A_{\bullet j'})\cap E_\ell^{k'}\neq \emptyset$, all minimal paths in $H_{E_\ell^{k} \cup E_\ell^{k'}}(A_{ E_\ell \times \bar f^*(E_\ell)})$ between $j$ and $j'$ have an odd number of yellow edges.
Similarly, the basic  paths in $G(A_{E_\ell\times \bar f^*(E_\ell)})$ with edge index set $E_\ell^k$ and $E_\ell^{k'}$  both enter, or leave $z_\ell$ if and only if $E_\ell^k \cap E_\ell^{k'}\neq \emptyset$, or for any $j,j'$ such that $s(A_{\bullet j})\cap E_\ell^k\neq \emptyset$ and $s(A_{\bullet j'})\cap E_\ell^{k'}\neq \emptyset$, all minimal paths in $H_{E_\ell^{k} \cup E_\ell^{k'}}(A_{ E_\ell \times \bar f^*(E_\ell)})$ between $j$ and $j'$ have an odd number of yellow edges.
Using Lemma \ref{lembonsainet2}, this completes the proof.
\end{proof}\\

\section{A digraph $D$}\label{sec:DefDigraphD}

Let $A$ be a connected $\{0,1,2,\frac{1}{2}\}$-matrix of size $n\times m$ and $R^*$ a row index subset of $A$. Let us denote by $\alpha$ the number of nonzero elements in $A$.
We assume that $\overline{R^*}$ has been partitioned into subsets $E_1,\ldots,E_b$ and for all $\ell=1,\ldots,b$, the classes $J_\ell^1$ and $J_\ell^2$ have been computed in Section \ref{sec:bonsaimat}.
We give the definition of a digraph denoted by $D$ and prove related results.

The vertex set of $D$ is $V=\{E_1,\ldots,E_{b}\}$ and its edge index set is denoted by $\Upsilon$ ($D=(V,\Upsilon)$). Up to column permutations, we may assume $S^*=\{1,\ldots,s\}$ where $s=|S^*|$. Let  
$g: \begin{array}{clll}
V \cup_{\ell,k} E_\ell^k & \rightarrow & \{0,1,2\}^{s} 
\end{array}
$ be the function given by

\begin{center}
$$
 g_\beta(E_\ell^k)=\left\{
\begin{array}{ll}
2 & \mbox{if}  \,\,\, E_\ell^k=E_\ell^I(A_{\bullet \beta})=E_\ell^{II}(A_{\bullet \beta})  \\
1 &\mbox{if}  \,\,\, E_\ell^k=E_\ell^i(A_{\bullet \beta}) \,\,\,(i=I\m{ or } II) \,\m{ and }\, E_\ell^I(A_{\bullet \beta})\neq E_\ell^{II}(A_{\bullet \beta}) \\
0 & \mbox{otherwise }  
\end{array}
\right.   $$\label{mycounter7}
\end{center}

\noindent
and $g_\beta(E_\ell)=\displaystyle{\sum_{h=1}^{m(\ell)}} g_\beta(E_\ell^{h})$,
for all $1\le \ell\le b$, $1\le k \le m(\ell)$ and $\beta\in S^*$.
Observe that for all $1\le \ell\le b$ and $1\le k \le m(\ell)$, we have the inequality $g(E_\ell)\geq g(E_\ell^k)$. For $1\le \ell,\ell' \le b$ and $1\le k \le m(\ell)$, there is an arc $(E_\ell,E_{\ell'})
\in \Upsilon$ labeled $E_{\ell'}^k$ if and only if  $ g(E_\ell)\le g(E^{k}_{\ell'})$ (componentwise) and $J_\ell^2$ is empty. 
We denote by $(E_\ell,E_{\ell'})_{E_{\ell'}^k}$ the arc $(E_\ell,E_{\ell'})$ labeled $E_{\ell'}^k$. We mention that if $A$ is $R^*$-cyclic, then for all $1\le \ell\le b$ the assumption of Lemma \ref{lemdigraphutile} holds; this implies that the definition of $D$ is unique. \label{mycounter}
We also define a relation \emph{$\prec_D$}\index{relation!$\prec_D$} on $V$ as follows: $E_{\ell'}\prec_D E_\ell$ if and only if $(E_\ell,E_{\ell'})\in D$. The relation $\prec_D$ is clearly not symmetric, but transitive.
An example of the digraph $D$ with respect to an $R^*$-cyclic matrix is given in Figure \ref{fig:digraphD}. 

\begin{lem}\label{lemdigraphDtimeD}
The computational effort required for the construction of $D$ is bounded by $O(nm\alpha)$.
\end{lem}

\begin{proof}
For all $1\le \ell \le b$, let $n_\ell$ and $\alpha_\ell$ be the number of rows and nonzero elements of $A_{E_\ell \times \bar f^*(E_\ell)}$. For a given bonsai $E_\ell$, the
time needed to compute the tree $T_\mathcal{I}$ and $T$ in the procedures E$\ell$path and TwoClasses, respectively, is bounded by $O(n_\ell \alpha_\ell)$, and we need to compute E$\ell$path($A$,$E_\ell$,$\beta$) for all $\beta \in f^*(E_\ell)$. Since $n_1 + \ldots + n_b \le n$ and $\alpha_1+ \ldots + \alpha_b \le \alpha$, we deduce that the construction of $D$ requires $O(nm\alpha)$ operations.
\end{proof}\\

If $A$ has an $R^*$-cyclic or $R^*$-network representation $G(A)$, then for all $1\le \ell \le b$ we denote by $v_\ell$\index{vertex!$v_\ell$} the closest vertex of $B_\ell$ to the basic subgraph in $G(A)$ with edge index set $R^*$; for any node $v_i$ in $G(A)$, we denote by $v_i^*$\index{vertex!$v_i^*$}\label{mycounter5}
the (other) endnode of the basic path joining $v_i$ and the previous subgraph. If $A$ has a $\{1,\rho\}$-central representation $G(A)$,  for any bonsai $B_\ell$ on the right of $\{e_1,e_\rho\}$ we define $v_\ell$ and $v_\ell^*$ as previously. 
Provided that $A$ has an $R^*$-cyclic, $R^*$-network or $R^*$-central representation $G(A)$, for all $1\le \ell\le b$ and $\beta\in S^*$, we distinguish four cases (see Figures \ref{fig:lsubstem1} and \ref{fig:lsubstem2}).

\begin{itemize}

\item[$a_0$)] The nonbasic edge $f_\beta$ does not generate any $B_\ell$-path.

\item[$a_1$)] The nonbasic edge $f_\beta$ generates exactly one $B_\ell$-path\\ and $A_{ij}=1$ for all $i\in E_\ell \cap s(A_{\bullet j})$.

\item[$a_2$)] The nonbasic edge $f_\beta$ generates exactly one $B_\ell$-path\\ and $A_{ij}=2$ for all $i\in E_\ell \cap s(A_{\bullet j})$.

\item[$a_3$)] The nonbasic edge $f_\beta$ generates two distinct $B_\ell$-paths.

\end{itemize}

\begin{figure}[h!]
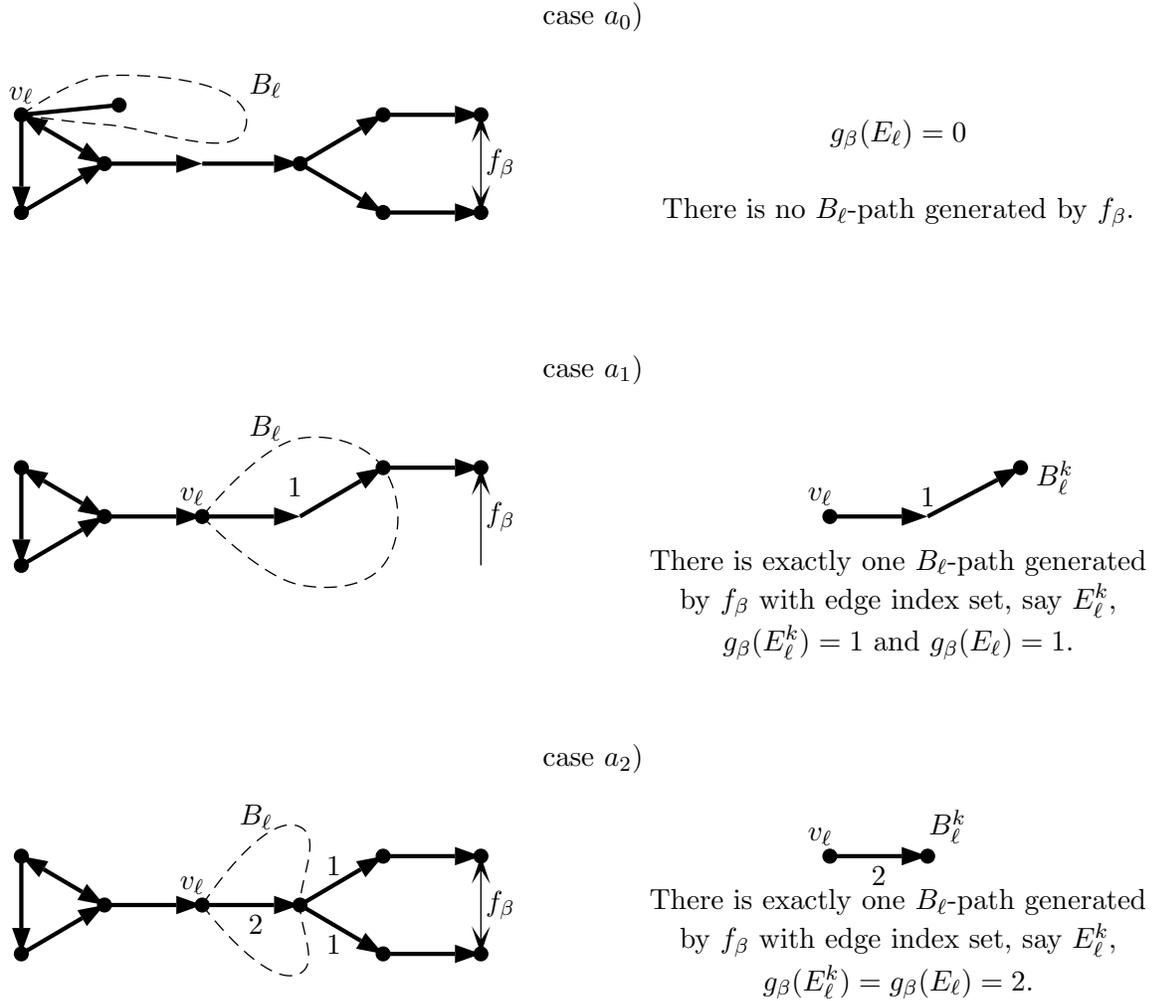


\begin{center}
case $a_0$)
\end{center}
$
\begin{array}{cl}

\psset{xunit=1.3cm,yunit=1.3cm,linewidth=0.5pt,radius=0.1mm,arrowsize=7pt,
labelsep=1.5pt,fillcolor=black}

\pspicture(0,0)(6,2.5)

\pscircle[fillstyle=solid](0,1){.1}
\pscircle[fillstyle=solid](0,2){.1}
\pscircle[fillstyle=solid](1,2.1){.1}
\pscircle[fillstyle=solid](0.85,1.5){.1}
\pscircle[fillstyle=solid](2.85,1.5){.1}
\pscircle[fillstyle=solid](3.7,1){.1}
\pscircle[fillstyle=solid](3.7,2){.1}
\pscircle[fillstyle=solid](4.7,1){.1}
\pscircle[fillstyle=solid](4.7,2){.1}

\psline[linewidth=1.6pt,arrowinset=0]{->}(0,1)(0.85,1.5)

\psline[linewidth=1.6pt,arrowinset=0]{-}(0,2)(1,2.1)

\psline[linewidth=1.6pt,arrowinset=0]{<-}(0,1)(0,2)

\psline[linewidth=1.6pt,arrowinset=0]{<->}(0,2)(0.85,1.5)

\psline[linewidth=1.6pt,arrowinset=0]{->}(0.85,1.5)(1.85,1.5)

\psline[linewidth=1.6pt,arrowinset=0]{->}(1.85,1.5)(2.85,1.5)

\psline[linewidth=1.6pt,arrowinset=0]{->}(2.85,1.5)(3.7,2)

\psline[linewidth=1.6pt,arrowinset=0]{->}(2.85,1.5)(3.7,1)

\psline[linewidth=1.6pt,arrowinset=0]{->}(3.7,1)(4.7,1)

\psline[linewidth=1.6pt,arrowinset=0]{->}(3.7,2)(4.7,2)

\psline[arrowinset=.5,arrowlength=1.5]{<->}(4.7,1)(4.7,2)
\rput(4.9,1.5){$f_\beta$}

\pscurve[linestyle=dashed]{-}(0,2)(0.85,2.4)(2.3,1.9)(0.85,1.9)(0,2)
\rput(2.5,2.3){$B_\ell$}
\rput(0,2.2){$v_\ell$}

\endpspicture 
& 

\psset{xunit=1.3cm,yunit=1.3cm,linewidth=0.5pt,radius=0.1mm,arrowsize=7pt,
labelsep=1.5pt,fillcolor=black}

\pspicture(0,0)(5,1)

\rput(2.7,1.8){$ g_\beta(E_\ell)=0$}
\rput(2.7,1){There is no $B_\ell$-path generated by $f_\beta$.}

\endpspicture 
\end{array}
$

\begin{center}
case $a_1$)
\end{center}

$
\begin{array}{cl}

\psset{xunit=1.3cm,yunit=1.3cm,linewidth=0.5pt,radius=0.1mm,arrowsize=7pt,
labelsep=1.5pt,fillcolor=black}

\pspicture(0,0)(6,2.5)

\pscircle[fillstyle=solid](0,1){.1}
\pscircle[fillstyle=solid](0,2){.1}
\pscircle[fillstyle=solid](0.85,1.5){.1}
\pscircle[fillstyle=solid](1.85,1.5){.1}
\pscircle[fillstyle=solid](3.7,2){.1}
\pscircle[fillstyle=solid](4.7,2){.1}

\psline[linewidth=1.6pt,arrowinset=0]{->}(0,1)(0.85,1.5)

\psline[linewidth=1.6pt,arrowinset=0]{<-}(0,1)(0,2)

\psline[linewidth=1.6pt,arrowinset=0]{<->}(0,2)(0.85,1.5)

\psline[linewidth=1.6pt,arrowinset=0]{->}(0.85,1.5)(1.85,1.5)

\psline[linewidth=1.6pt,arrowinset=0]{->}(1.85,1.5)(2.85,1.5)

\psline[linewidth=1.6pt,arrowinset=0]{->}(2.85,1.5)(3.7,2)

\rput(2.8,1.8){$1$}

\psline[linewidth=1.6pt,arrowinset=0]{->}(3.7,2)(4.7,2)

\psline[arrowinset=.5,arrowlength=1.5]{->}(4.7,1)(4.7,2)
\rput(4.9,1.5){$f_\beta$}

\pscurve[linestyle=dashed]{-}(1.85,1.5)(2.85,2.3)(3.7,2)(3.7,1)(2.85,0.8)
(1.85,1.5)
\rput(2.5,2.4){$B_\ell$}
\rput(1.75,1.7){$v_\ell$}

\endpspicture 
& 

\psset{xunit=1.3cm,yunit=1.3cm,linewidth=0.5pt,radius=0.1mm,arrowsize=7pt,
labelsep=1.5pt,fillcolor=black}

\pspicture(0,0)(5,1)

\pscircle[fillstyle=solid](2,1.5){.1}
\pscircle[fillstyle=solid](3.95,2){.1}

\psline[linewidth=1.6pt,arrowinset=0]{->}(2,1.5)(3,1.5)

\psline[linewidth=1.6pt,arrowinset=0]{->}(3,1.5)(3.95,2)

\rput(3,1.7){$1$}

\rput(1.9,1.7){$v_\ell$}
\rput(4.3,1.9){$B_\ell^k$}

\rput(2.7,0.6){\shortstack{There is exactly one $B_\ell$-path generated \\ by $f_\beta$ with edge index set, say $E_\ell^k$, \\ 
$g_\beta(E_\ell^k)=1$ and $ g_\beta(E_\ell)=1$.}}

\endpspicture \\ & \\
\end{array}
$

\begin{center}
case $a_2$)
\end{center}

$
\begin{array}{cl}

\psset{xunit=1.3cm,yunit=1.3cm,linewidth=0.5pt,radius=0.1mm,arrowsize=7pt,
labelsep=1.5pt,fillcolor=black}

\pspicture(0,0)(6,2.5)

\pscircle[fillstyle=solid](0,1){.1}
\pscircle[fillstyle=solid](0,2){.1}
\pscircle[fillstyle=solid](0.85,1.5){.1}
\pscircle[fillstyle=solid](1.85,1.5){.1}
\pscircle[fillstyle=solid](2.85,1.5){.1}
\pscircle[fillstyle=solid](3.7,1){.1}
\pscircle[fillstyle=solid](3.7,2){.1}
\pscircle[fillstyle=solid](4.7,1){.1}
\pscircle[fillstyle=solid](4.7,2){.1}

\psline[linewidth=1.6pt,arrowinset=0]{->}(0,1)(0.85,1.5)

\psline[linewidth=1.6pt,arrowinset=0]{<-}(0,1)(0,2)

\psline[linewidth=1.6pt,arrowinset=0]{<->}(0,2)(0.85,1.5)

\psline[linewidth=1.6pt,arrowinset=0]{->}(0.85,1.5)(1.85,1.5)

\psline[linewidth=1.6pt,arrowinset=0]{->}(1.85,1.5)(2.85,1.5)
\rput(2.4,1.3){$2$}

\psline[linewidth=1.6pt,arrowinset=0]{->}(2.85,1.5)(3.7,2)
\rput(3.2,1.9){$1$}

\psline[linewidth=1.6pt,arrowinset=0]{->}(2.85,1.5)(3.7,1)
\rput(3.2,1.1){$1$}

\psline[linewidth=1.6pt,arrowinset=0]{->}(3.7,1)(4.7,1)

\psline[linewidth=1.6pt,arrowinset=0]{->}(3.7,2)(4.7,2)

\psline[arrowinset=.5,arrowlength=1.5]{<->}(4.7,1)(4.7,2)
\rput(4.9,1.5){$f_\beta$}

\pscurve[linestyle=dashed]{-}(1.85,1.5)(2.85,2.3)(2.85,1.5)(2.85,0.8)
(1.85,1.5)
\rput(2.4,2.4){$B_\ell$}
\rput(1.75,1.7){$v_\ell$}

\endpspicture & 

\psset{xunit=1.3cm,yunit=1.3cm,linewidth=0.5pt,radius=0.1mm,arrowsize=7pt,
labelsep=1.5pt,fillcolor=black}

\pspicture(0,0)(5,1)

\pscircle[fillstyle=solid](2,2){.1}
\pscircle[fillstyle=solid](3,2){.1}

\psline[linewidth=1.6pt,arrowinset=0]{->}(2,2)(3,2)
\rput(2.5,1.8){$2$}

\rput(3.2,2.3){$B_\ell^k$}
\rput(1.9,2.2){$v_\ell$}

\rput(2.7,1.1){\shortstack{There is exactly one $B_\ell$-path generated \\ by $f_\beta$ with edge index set, say $E_\ell^k$,  \\ $g_\beta(E_\ell^k)=g_\beta(E_\ell)=2$.}}

\endpspicture

\end{array}
$

\caption{A description of all possible $B_\ell$-paths generated by a nonbasic edge $f_\beta$ and the corresponding $E_\ell$-paths generated by $A_{\bullet \beta}$ in cases $a_0$, $a_1$ and $a_2$. (A number close to an edge corresponds to the weight of the corresponding edge in the fundamental circuit of $f_\beta$.)}  
\label{fig:lsubstem1}

\end{figure}

\begin{figure}[h!]

\begin{center}
case $a_3$)
\end{center}

$
\begin{array}{cl}
\psset{xunit=1.3cm,yunit=1.3cm,linewidth=0.5pt,radius=0.1mm,arrowsize=7pt,
labelsep=1.5pt,fillcolor=black}

\pspicture(0,0)(6,2.5)

\pscircle[fillstyle=solid](0,1){.1}
\pscircle[fillstyle=solid](0,2){.1}
\pscircle[fillstyle=solid](0.85,1.5){.1}
\pscircle[fillstyle=solid](1.85,1.5){.1}
\pscircle[fillstyle=solid](2.85,1.5){.1}
\pscircle[fillstyle=solid](3.7,1){.1}
\pscircle[fillstyle=solid](3.7,2){.1}
\pscircle[fillstyle=solid](4.7,1){.1}
\pscircle[fillstyle=solid](4.7,2){.1}

\psline[linewidth=1.6pt,arrowinset=0]{->}(0,1)(0.85,1.5)

\psline[linewidth=1.6pt,arrowinset=0]{<-}(0,1)(0,2)

\psline[linewidth=1.6pt,arrowinset=0]{<->}(0,2)(0.85,1.5)

\psline[linewidth=1.6pt,arrowinset=0]{->}(0.85,1.5)(1.85,1.5)

\psline[linewidth=1.6pt,arrowinset=0]{->}(1.85,1.5)(2.85,1.5)
\rput(2.4,1.3){$2$}

\psline[linewidth=1.6pt,arrowinset=0]{->}(2.85,1.5)(3.7,2)
\rput(3.2,1.9){$1$}

\psline[linewidth=1.6pt,arrowinset=0]{->}(2.85,1.5)(3.7,1)
\rput(3.2,1.1){$1$}

\psline[linewidth=1.6pt,arrowinset=0]{->}(3.7,1)(4.7,1)

\psline[linewidth=1.6pt,arrowinset=0]{->}(3.7,2)(4.7,2)

\psline[arrowinset=.5,arrowlength=1.5]{<->}(4.7,1)(4.7,2)
\rput(4.9,1.5){$f_\beta$}

\pscurve[linestyle=dashed]{-}(1.85,1.5)(2.85,2.3)(3.7,2)(3.7,1)(2.85,0.8)
(1.85,1.5)
\rput(2.5,2.4){$B_\ell$}
\rput(1.75,1.7){$v_\ell$}

\endpspicture &

\psset{xunit=1.3cm,yunit=1.3cm,linewidth=0.5pt,radius=0.1mm,arrowsize=7pt,
labelsep=1.5pt,fillcolor=black}

\pspicture(0,0)(5,1)

\pscircle[fillstyle=solid](0,1.5){.1}
\pscircle[fillstyle=solid](1,1.5){.1}
\pscircle[fillstyle=solid](1.95,2){.1}

\pscircle[fillstyle=solid](3,1.5){.1}
\pscircle[fillstyle=solid](4,1.5){.1}
\pscircle[fillstyle=solid](4.95,1){.1}

\psline[linewidth=1.6pt,arrowinset=0]{->}(0,1.5)(1,1.5)
\rput(0.5,1.3){$2$}

\psline[linewidth=1.6pt,arrowinset=0]{->}(1,1.5)(1.95,2)
\rput(1.45,1.9){$1$}

\rput(-.1,1.7){$v_\ell$}
\rput(2,1.5){$B_\ell^k$}

\psline[linewidth=1.6pt,arrowinset=0]{->}(3,1.5)(4,1.5)
\rput(3.5,1.3){$2$}

\psline[linewidth=1.6pt,arrowinset=0]{->}(4,1.5)(4.95,1)
\rput(4.45,1.1){$1$}

\rput(2.9,1.7){$v_\ell$}
\rput(4.5,1.7){$B_\ell^{k'}$}

\endpspicture\\

\psset{xunit=1.3cm,yunit=1.3cm,linewidth=0.5pt,radius=0.1mm,arrowsize=7pt,
labelsep=1.5pt,fillcolor=black}

\pspicture(0,0)(6,1.7)

\pscircle[fillstyle=solid](0,1){.1}
\pscircle[fillstyle=solid](0,2){.1}
\pscircle[fillstyle=solid](0.85,1.5){.1}
\pscircle[fillstyle=solid](1.85,1.5){.1}
\pscircle[fillstyle=solid](2.85,1.5){.1}
\pscircle[fillstyle=solid](3.7,1){.1}
\pscircle[fillstyle=solid](3.7,2){.1}
\pscircle[fillstyle=solid](4.7,1){.1}
\pscircle[fillstyle=solid](4.7,2){.1}

\psline[linewidth=1.6pt,arrowinset=0]{->}(0,1)(0.85,1.5)

\psline[linewidth=1.6pt,arrowinset=0]{<-}(0,1)(0,2)

\psline[linewidth=1.6pt,arrowinset=0]{<->}(0,2)(0.85,1.5)

\psline[linewidth=1.6pt,arrowinset=0]{->}(0.85,1.5)(1.85,1.5)

\psline[linewidth=1.6pt,arrowinset=0]{->}(1.85,1.5)(2.85,1.5)
\rput(2.4,1.3){$2$}

\psline[linewidth=1.6pt,arrowinset=0]{->}(2.85,1.5)(3.7,2)
\rput(3.2,1.9){$1$}

\psline[linewidth=1.6pt,arrowinset=0]{->}(2.85,1.5)(3.7,1)
\rput(3.2,1.1){$1$}

\psline[linewidth=1.6pt,arrowinset=0]{->}(3.7,1)(4.7,1)

\psline[linewidth=1.6pt,arrowinset=0]{->}(3.7,2)(4.7,2)

\psline[arrowinset=.5,arrowlength=1.5]{<->}(4.7,1)(4.7,2)
\rput(4.9,1.5){$f_\beta$}

\pscurve[linestyle=dashed]{-}(1.85,1.5)(2.85,2.3)(3.7,2)(3.7,1.5)(2.85,1.5)
(2.85,1)(1.85,1.5)
\rput(2.5,2.4){$B_\ell$}
\rput(1.75,1.7){$v_\ell$}

\endpspicture & 

\psset{xunit=1.3cm,yunit=1.3cm,linewidth=0.5pt,radius=0.1mm,arrowsize=7pt,
labelsep=1.5pt,fillcolor=black}

\pspicture(0,0)(5,1)

\pscircle[fillstyle=solid](0,1.5){.1}
\pscircle[fillstyle=solid](1,1.5){.1}
\pscircle[fillstyle=solid](1.95,2){.1}

\pscircle[fillstyle=solid](3,1.5){.1}
\pscircle[fillstyle=solid](4,1.5){.1}

\psline[linewidth=1.6pt,arrowinset=0]{->}(0,1.5)(1,1.5)
\rput(0.5,1.3){$2$}

\psline[linewidth=1.6pt,arrowinset=0]{->}(1,1.5)(1.95,2)
\rput(1.45,1.9){$1$}

\rput(-.1,1.7){$v_\ell$}
\rput(2,1.5){$B_\ell^k$}

\psline[linewidth=1.6pt,arrowinset=0]{->}(3,1.5)(4,1.5)
\rput(3.5,1.3){$2$}

\rput(2.9,1.7){$v_\ell$}
\rput(4.5,1.5){$B_\ell^{k'}$}

\endpspicture

 \\

\psset{xunit=1.3cm,yunit=1.3cm,linewidth=0.5pt,radius=0.1mm,arrowsize=7pt,
labelsep=1.5pt,fillcolor=black}

\pspicture(0,0)(6,1.6)

\pscircle[fillstyle=solid](0,1){.1}
\pscircle[fillstyle=solid](0,2){.1}
\pscircle[fillstyle=solid](0.85,1.5){.1}
\pscircle[fillstyle=solid](1.85,1.5){.1}
\pscircle[fillstyle=solid](2.85,1.5){.1}
\pscircle[fillstyle=solid](3.7,1){.1}
\pscircle[fillstyle=solid](3.7,2){.1}
\pscircle[fillstyle=solid](4.7,1){.1}
\pscircle[fillstyle=solid](4.7,2){.1}

\psline[linewidth=1.6pt,arrowinset=0]{->}(0,1)(0.85,1.5)

\psline[linewidth=1.6pt,arrowinset=0]{<-}(0,1)(0,2)

\psline[linewidth=1.6pt,arrowinset=0]{<->}(0,2)(0.85,1.5)

\psline[linewidth=1.6pt,arrowinset=0]{->}(0.85,1.5)(1.85,1.5)

\psline[linewidth=1.6pt,arrowinset=0]{->}(1.85,1.5)(2.85,1.5)
\rput(2.4,1.3){$2$}

\psline[linewidth=1.6pt,arrowinset=0]{->}(2.85,1.5)(3.7,2)
\rput(3.2,1.9){$1$}

\psline[linewidth=1.6pt,arrowinset=0]{->}(2.85,1.5)(3.7,1)
\rput(3.2,1.1){$1$}

\psline[linewidth=1.6pt,arrowinset=0]{->}(3.7,1)(4.7,1)

\psline[linewidth=1.6pt,arrowinset=0]{->}(3.7,2)(4.7,2)

\psline[arrowinset=.5,arrowlength=1.5]{<->}(4.7,1)(4.7,2)
\rput(4.9,1.5){$f_\beta$}

\pscurve[linestyle=dashed]{-}(2.85,1.5)(2.85,2.3)(3.7,2)(3.7,1)(2.85,0.8)
(2.85,1.5)
\rput(2.5,2.3){$B_\ell$}
\rput(2.65,1.7){$v_\ell$}

\endpspicture  & 

\psset{xunit=1.3cm,yunit=1.3cm,linewidth=0.5pt,radius=0.1mm,arrowsize=7pt,
labelsep=1.5pt,fillcolor=black}

\pspicture(0,0)(5,1)

\pscircle[fillstyle=solid](1,1.7){.1}
\pscircle[fillstyle=solid](1.95,2.2){.1}

\pscircle[fillstyle=solid](4,1.7){.1}
\pscircle[fillstyle=solid](4.95,1.2){.1}

\psline[linewidth=1.6pt,arrowinset=0]{->}(1,1.7)(1.95,2.2)
\rput(1.45,2.1){$1$}

\rput(0.9,1.9){$v_\ell$}
\rput(2,1.7){$B_\ell^k$}

\psline[linewidth=1.6pt,arrowinset=0]{->}(4,1.7)(4.95,1.2)
\rput(4.45,1.3){$1$}

\rput(3.9,1.9){$v_\ell$}
\rput(4.8,1.7){$B_\ell^{k'}$}

\rput(2.7,0.3){\shortstack{There are two $B_\ell$-paths generated by $f_\beta$
whose  \\ edge index sets are say  $E_\ell^k$ and $E_\ell^{k'}$,  \\ $g_\beta(E_\ell^k)=g_\beta(E_\ell^{k'})=1$ and $g_\beta(E_\ell)=2$.}}

\endpspicture

\end{array}$

\vspace{0.5cm}

\caption{A description of all possible $B_\ell$-paths generated by a nonbasic edge $f_\beta$ and the corresponding $E_\ell$-paths generated by $A_{\bullet \beta}$ in case $a_3$.  (A number close to an edge corresponds to the weight of the corresponding edge in the fundamental circuit of $f_\beta$.)}  
\label{fig:lsubstem2}

\end{figure}

\begin{prop}\label{propBlsubstems}
Suppose that $A$ has an $R^*$-cyclic, $R^*$-central or $R^*$-network representation $G(A)$.
Let $1\le \ell\le b$ and $\beta \in S^*$. If one faces  case  $a_0$ (respectively, $a_1$, $a_2$ and $a_3$), then the following respective properties hold.

\begin{itemize}

\item[$a_0$)] $g_\beta(E_\ell)=0$.

\item[$a_1$)] $E_\ell^{II}(A_{\bullet \beta})=\emptyset$ and $g_\beta(E_\ell)=g_\beta(E_\ell^I(A_{\bullet \beta}))=1$.

\item[$a_2$)] $E_\ell^I(A_{\bullet \beta})=E_\ell^{II}(A_{\bullet \beta})$ and $g_\beta(E_\ell)=g_\beta(E_\ell^I(A_{\bullet \beta}))=2$.

\item[$a_3$)] $E_\ell^I(A_{\bullet \beta})\neq E_\ell^{II}(A_{\bullet \beta})$, $g_\beta(E_\ell^I(A_{\bullet \beta}))=g_\beta(E_\ell^{II}(A_{\bullet \beta}))=1$ and  $g_\beta(E_\ell)=2$.

\end{itemize}

\noindent
The converse sense is true.\\
If $A$ is an $R^*$-network matrix, then cases $a_2$ and $a_3$ do not happen.
\end{prop}

\begin{proof}
This directly follows from the construction of of $E_\ell^I(A_{\bullet \beta})$ and $E_\ell^{II}(A_{\bullet \beta})$,
Lemma \ref{lemdigraphBlI} and the definition of $g$.
\end{proof}\\

\begin{cor}\label{cordigraphDbeta2}
Suppose that $A$ is $R^*$-cyclic. Let $1\le \ell\le b$ and $\beta \in S^*$ such that $g_\beta (E_\ell) =2$. Then the fundamental circuit of $f_\beta$ contains the basic cycle.
\end{cor}

\begin{proof}
The proof follows from Corollary \ref{corBidirectedCircuit}, Lemma \ref{lemdefiWeight1} and Proposition \ref{propBlsubstems} (see the description of the different types of fundamental circuits at page \pageref{mycounter2} and Figure \ref{fig:Apositive}).\\
\end{proof}

The next proposition gives a graphical interpretation of an arc in $D$ as illustrated in Figure \ref{fig:digrapharc}.
In this figure, the  fundamental circuit of the nonbasic edge $f_3$ intersects the trees $B_2$, $B_3$ and $B_4$. 
We have the inequalities $g_3(E_2)\le g_3(E_4^1)$ and $g_3(E_4)\le g_3(E_3^2)$.

\begin{prop}\label{lemdigraphstem}
Suppose that $A$ has an $R^*$-cyclic or $R^*$-network representation $G(A)$. Let $1\le \ell,\ell'\le b$ and $q$ be the basic path in $G(A)$ from $v_\ell$ to $v_\ell^*$. Suppose that $E_{\ell'}$ contains an edge of $q$. Then $\{ i\, :\, e_i \in q\cap B_{\ell'}\}=E_{\ell'}^k$ for some $1\le k\le m(\ell')$ and $(E_\ell,E_{\ell'})_{E_{\ell'}^k}\in D$.
\end{prop}

\begin{proof}
Since $A$ is connected, there exists a column index $j$ in the global connecter set $f^*(E_\ell)$ of $E_\ell$, such that $q$ is a subpath of a stem issued from the nonbasic edge $f_j$. 
Then, $\{ i\, :\, e_i \in q\cap B_{\ell'}\}$ is an $E_{\ell'}$-path, say $E_{\ell'}^k$, generated by $A_{\bullet j}$. By definition of $g$, we have the inequality $g_j(E_{\ell'}^k)\geq 1$.
If $g_j(E_\ell)=2$ for some column index $j$, then by Proposition \ref{propBlsubstems} one is faced with case $a_2$ or $a_3$ with respect to $B_\ell$ and $f_j$. It follows that
$E_{\ell'}^k$ is the unique $E_{\ell'}$-path generated by $A_{\bullet j}$ and
$A_{ij}=2$ for all $i\in E_{\ell'}\cap s(A_{\bullet j})$. Hence $g_j(E_{\ell'}^k)=2$. Therefore, $g_j(E_\ell)\le g_j(E_{\ell'}^k)$.
Moreover, since $A$ is nonnegative, if $q$ is leaving $v_\ell$, then all $B_\ell$-paths of $G(A)$ are entering $v_\ell$. Similarly, if $q$ is entering $v_\ell$, then all $B_\ell$-paths of $G(A)$ are leaving $v_\ell$. Thus by Lemma \ref{lemdigraphutile} $J_\ell^2=\emptyset$.
We conclude that $(E_\ell,E_{\ell'})_{E_{\ell'}^k}\in D$. 
\end{proof}\\

\begin{prop}\label{lemdigraphstem2}
Suppose that $A$ has a $\{1,\rho\}$-central representation $G(A)$. Let $B_\ell$ and $B_{\ell'}$ be two bonsais on the right of $\{e_1,e_\rho\}$ such that $f^*(E_\ell) \cap f^*(E_{\ell'})\neq \emptyset$. Then $(E_\ell,E_{\ell'})\in D$ or $(E_{\ell'},E_\ell)\in D$.
\end{prop}

\begin{proof}
Since $f^*(E_\ell) \cap f^*(E_{\ell'})\neq \emptyset$,  $B_\ell$ and  $B_{\ell'}$ contain each one at least one edge of the fundamental circuit of a nonbasic edge $f_j$ with $j\in f^*(E_\ell) \cap f^*(E_{\ell'})$. Then, the proof is similar to the proof of Proposition \ref{lemdigraphstem}.
\end{proof}\\

\begin{figure}[ht!]
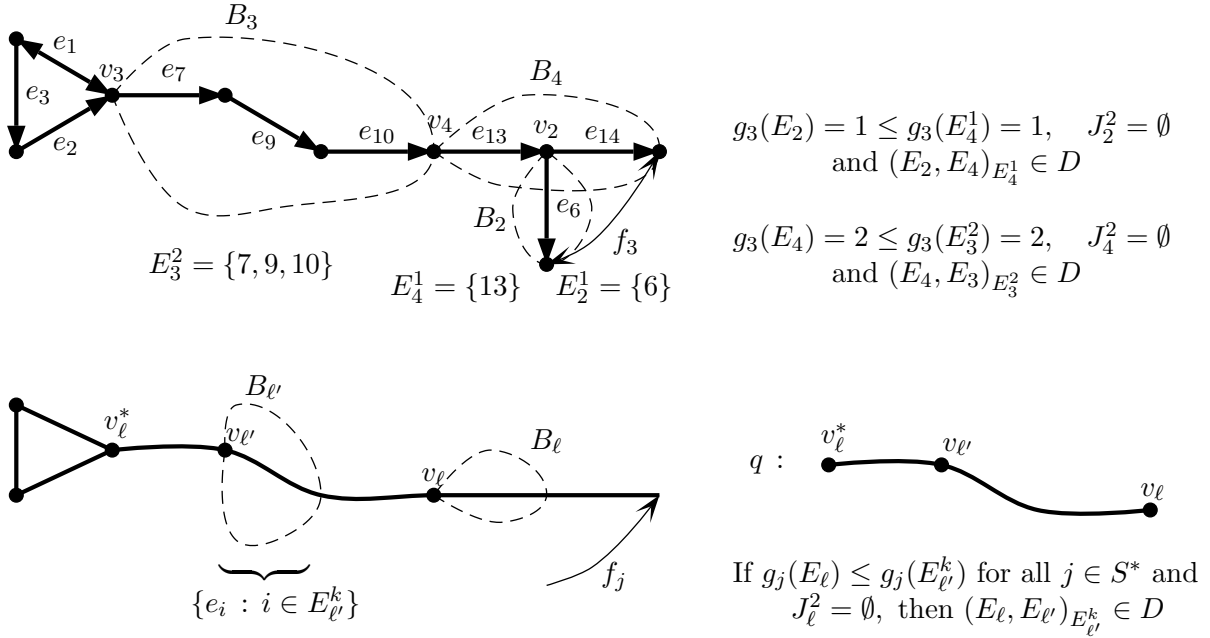

$
\begin{array}{cc}

\psset{xunit=1.5cm,yunit=1.5cm,linewidth=0.5pt,radius=0.1mm,arrowsize=7pt,
labelsep=1.5pt,fillcolor=black}

\pspicture(0,0)(6,2.5)

\pscircle[fillstyle=solid](0,1){.1}
\pscircle[fillstyle=solid](0,2){.1}
\pscircle[fillstyle=solid](0.85,1.5){.1}

\pscircle[fillstyle=solid](1.85,1.5){.1}

\pscircle[fillstyle=solid](2.7,1){.1}

\pscircle[fillstyle=solid](3.7,1){.1}

\pscircle[fillstyle=solid](4.7,1){.1}
\pscircle[fillstyle=solid](4.7,0){.1}

\pscircle[fillstyle=solid](5.7,1){.1}

\psline[linewidth=1.6pt,arrowinset=0]{->}(4.7,1)(4.7,0)
\rput(4.9,0.5){$e_6$}

\psline[linewidth=1.6pt,arrowinset=0]{->}(0,1)(0.85,1.5)
\rput(0.44,1.05){$e_2$}

\psline[linewidth=1.6pt,arrowinset=0]{<-}(0,1)(0,2)
\rput(0.2,1.5){$e_3$}

\psline[linewidth=1.6pt,arrowinset=0]{<->}(0,2)(0.85,1.5)
\rput(0.44,1.95){$e_1$}

\psline[linewidth=1.6pt,arrowinset=0]{->}(0.85,1.5)(1.85,1.5)
\rput(1.4,1.7){$e_7$}

\psline[linewidth=1.6pt,arrowinset=0]{->}(1.85,1.5)(2.7,1)
\rput(2.2,1.1){$e_9$}

\psline[linewidth=1.6pt,arrowinset=0]{->}(2.7,1)(3.7,1)
\rput(3.2,1.15){$e_{10}$}

\psline[linewidth=1.6pt,arrowinset=0]{->}(3.7,1)(4.7,1)
\rput(4.2,1.15){$e_{13}$}

\psline[linewidth=1.6pt,arrowinset=0]{->}(4.7,1)(5.7,1)
\rput(5.2,1.15){$e_{14}$}

\pscurve[arrowinset=.5,arrowlength=1.5]{<->}(4.7,0)(5.1,0.2)(5.7,1)
\rput(5.4,0.2){$f_3$}

\rput(0.85,1.7){$v_3$}
\rput(3.76,1.25){$v_4$}
\rput(4.7,1.2){$v_2$}

\rput(2,2.2){$B_3$}
\rput(4.7,1.7){$B_4$}
\rput(4.2,0.4){$B_2$}

\rput(2,0){$E^2_3=\{7,9,10\}$}
\rput(3.9,-0.2){$E^1_4=\{13\}$}
\rput(5.3,-0.2){$E^1_2=\{6\}$}

\pscurve[linestyle=dashed](0.85,1.5)(1.5,2)(3.7,1)(2.3,0.5)(1.5,0.5)(0.85,1.5)

\pscurve[linestyle=dashed](3.7,1)(4.5,1.5)(5.7,1)(4.3,0.7)(3.7,1)

\pscurve[linestyle=dashed](4.7,1)(5.1,0.5)(4.7,0)(4.4,0.5)(4.7,1)

\endpspicture &

\begin{array}{c}
g_3(E_2)=1\le g_3(E_4^1)=1, \quad J_2^2=\emptyset\\
\m{ and } (E_2,E_4)_{E_{4}^1}\in D \\
\\
g_3(E_4)=2 \le  g_3(E_3^2)=2, \quad J_4^2=\emptyset\\
\m{ and } (E_4,E_3)_{E_{3}^2}\in D \\
\\ \\ \\
\end{array}

\end{array}$
\vspace{1cm}

$
\begin{array}{cc}
\psset{xunit=1.5cm,yunit=1.2cm,linewidth=0.5pt,radius=0.1mm,arrowsize=7pt,
labelsep=1.5pt,fillcolor=black}

\pspicture(0,0)(6,1)

\pscircle[fillstyle=solid](0,1){.1}
\pscircle[fillstyle=solid](0,2){.1}
\pscircle[fillstyle=solid](0.85,1.5){.1}
\pscircle[fillstyle=solid](1.85,1.5){.1}
\pscircle[fillstyle=solid](3.7,1){.1}

\psline[linewidth=1.6pt,arrowinset=0]{-}(0,1)(0.85,1.5)

\psline[linewidth=1.6pt,arrowinset=0]{-}(0,1)(0,2)

\psline[linewidth=1.6pt,arrowinset=0]{-}(0,2)(0.85,1.5)

\pscurve[linewidth=1.6pt]{-}(0.85,1.5)(1.85,1.5)(2.7,1)(3.7,1)(4.7,1)(5.7,1)

\pscurve[arrowinset=.5,arrowlength=1.5]{->}(4.7,0)(5.1,0.2)(5.7,1)
\rput(5.3,0.1){$f_j$}

\rput(3.7,1.2){$v_\ell$}
\rput(2,1.65){$v_{\ell'}$}
\rput(0.9,1.8){$v_\ell^*$}

\rput(2.2,2.2){$B_{\ell'}$}
\rput(4.7,1.6){$B_\ell$}

\pscurve[linestyle=dashed](1.85,1.5)(1.95,2)(2.7,1)(2.3,0.5)(1.95,0.5)(1.85,1.5)

\pscurve[linestyle=dashed](3.7,1)(4.2,1.5)(4.7,1)(4.3,0.7)(3.7,1)

\rput(2.2,0.15){$\underbrace{\hspace{1.2cm}}$}
\rput(2.3,-0.2){$\{e_i\,:\, i\in E^k_{\ell'}\}$}

\endpspicture &

\begin{array}{l}

\psset{xunit=1.5cm,yunit=1.2cm,linewidth=0.5pt,radius=0.1mm,arrowsize=7pt,
labelsep=1.5pt,fillcolor=black}

\pspicture(0,0)(6,0)

\pscircle[fillstyle=solid](0.85,.5){.1}
\pscircle[fillstyle=solid](1.85,.5){.1}
\pscircle[fillstyle=solid](3.7,0){.1}

\pscurve[linewidth=1.6pt]{-}(0.85,.5)(1.85,.5)(2.7,0)(3.7,0)

\rput(3.72,.2){$v_\ell$}
\rput(2,.7){$v_{\ell'}$}
\rput(0.9,.8){$v_\ell^*$}

\rput(.3,.5){$q\,:$}

\endpspicture\\ \\

\m{If } g_j(E_\ell)\le g_j (E_{\ell'}^k)\m{ for all } j\in S^* \m{ and } \\
\quad \quad J_\ell^2=\emptyset, \m{ then }(E_\ell,E_{\ell'})_{E_{\ell'}^k}\in D \\
\\

\end{array}

\end{array}$

\caption{The fundamental circuit of $f_3$ as in Figure 
\ref{fig:G(A)bon} and an illustration of a stem (issued from some $f_j$) which
intersects a bonsai $B_\ell$ and thus contains the basic path from $v_l$ to $v_l^*$.  }
\label{fig:digrapharc}
\end{figure}

Finally, to deal with directed cycles in $D$, we define the following relation. For any $E_\ell,E_{\ell'} \in D$, $E_\ell \sim_s E_{\ell'}$\index{relation!$\sim_s$} if and only if $E_\ell$ and $E_{\ell'}$ are in a same strongly connected component of $D$. This is clearly an equivalence relation. We state a useful lemma in the case where two bonsais are equivalent under the relation $\sim_s$.

\begin{lem}\label{lemcycle}
Let $E_\ell,E_{\ell'} \in D$ such that $E_\ell \sim_s E_{\ell'}$. Then $g(E_\ell)=g(E_{\ell'})$ and $m(\ell)=m(\ell')=1$.
\end{lem}

\begin{proof} 
Since the relation "$\,E_\ell\,{\prec_D}\, E_{\ell'} \Leftrightarrow 
\,(E_\ell,E_{\ell'})\in D\,$" is
transitive, it follows that  $g(E_\ell^k)=g(E_\ell)=g(E_{\ell'})= 
g(E_{\ell'}^{k'})$ for some $1 \le k \le m(\ell)$ and $1\le k' \le m(\ell')$. \\
On the other hand, suppose by contradiction that $m(\ell)\geq 2$. 
Let $h$ be such that $E_\ell^h \neq E_\ell^k$ and $j\in f^* (E_\ell^h)$. We have $1 \le g_j(E_\ell^h) \le g_j(E_\ell) =g_j(E_\ell^k)$. Thus we may assume that $E_\ell^I(A_{\bullet j})=E_\ell^h$, $E_\ell^{II}(A_{\bullet j})=E_\ell^k$ and we have the inequality $E_\ell^I(A_{\bullet j})\neq E_\ell^{II}(A_{\bullet j})$. By definition of $g$, it results that $g_j(E_\ell^k)=1$ and $g_j(E_\ell)=2$, contradicting the equality  
$ g_j(E_\ell) =g_j(E_\ell^k)$.
\end{proof}

\section{A spanning forest in $D$}\label{sec:FordigaphD}

Let $A$ be a connected matrix and $R^*$ a row index subset of $A$. In this section, we deal with the case where $A$ is cyclic 
or $R^*$-network, and we study some forests in $D$ deriving from an $R^*$-cyclic or $R^*$-central representation of $A$. Throughout this section, we assume that there is no directed cycle in $D$.

Suppose that $A$ has an $R^*$-cyclic or $R^*$-network representation $G(A)$. By Propositions \ref{lemdigraphstem}, the bidirected graph $G(A)$  \emph{induces the forest $T_{G(A)}$} \emph{ in $D$ }\index{forest!$T_{G(A)}$} defined as follows: 
for any $B_\ell,\, B_{\ell'} \subseteq G(A)$, $(E_\ell,E_{\ell'})_{E_{\ell'}^k} \in T_{G(A)}$ if and only if  $v_\ell$ is the endnode  ($\neq v_{\ell'}$) of the $B_{\ell'}$-path of $G(A)$ with edge index set $E_{\ell'}^k$. See Figure \ref{fig8:introalgo1} for an example.
Whenever $A$ has a $\{1,\rho\}$-central representation $G(A)$ whose $T$ is the basic maximal $1$-tree, we recall that $G_1(A)$ denotes the component of $T\verb"\"\{e_1,e_\rho\}$ containing $w_\rho$ (on the right of $\{e_1,e_\rho\}$). If $A$  
has a $\{1,\rho\}$-central representation $G(A)$, then 
using Proposition \ref{lemdigraphstem2}
$T_{G_1(A)}$\index{forest!$T_{G_1(A)}$}\label{mycounter9} is defined with respect to $G_1(A)$ in a same way as $T_{G(A)}$ with respect to $G(A)$. 
With the intention of capturing the main properties of this forest, we give some more definitions in the next paragraph.

\begin{figure}[h!]
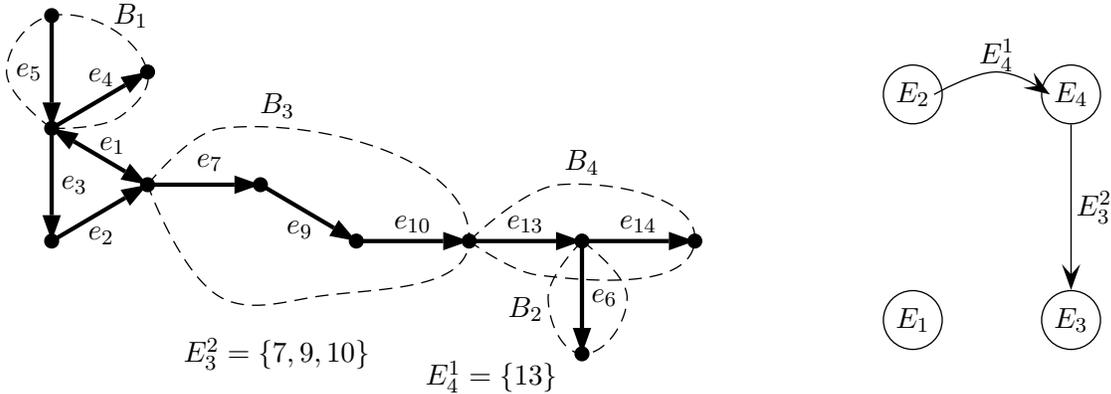


\vspace{1cm}
$
\begin{array}{cc}
\psset{xunit=1.5cm,yunit=1.5cm,linewidth=0.5pt,radius=0.1mm,arrowsize=7pt,
labelsep=1.5pt,fillcolor=black}

\pspicture(-0.5,0)(6,2.5)

\pscircle[fillstyle=solid](0,1){.1}
\pscircle[fillstyle=solid](0,2){.1}
\pscircle[fillstyle=solid](0,3){.1}
\pscircle[fillstyle=solid](0.85,1.5){.1}
\pscircle[fillstyle=solid](0.85,2.5){.1}
\pscircle[fillstyle=solid](1.85,1.5){.1}

\pscircle[fillstyle=solid](2.7,1){.1}

\pscircle[fillstyle=solid](3.7,1){.1}

\pscircle[fillstyle=solid](4.7,1){.1}
\pscircle[fillstyle=solid](4.7,0){.1}

\pscircle[fillstyle=solid](5.7,1){.1}

\psline[linewidth=1.6pt,arrowinset=0]{<->}(0,2)(0.85,1.5)
\rput(0.54,1.85){$e_1$}

\psline[linewidth=1.6pt,arrowinset=0]{->}(0,1)(0.85,1.5)
\rput(0.44,1.05){$e_2$}

\psline[linewidth=1.6pt,arrowinset=0]{<-}(0,1)(0,2)
\rput(0.2,1.5){$e_3$}

\psline[linewidth=1.6pt,arrowinset=0]{->}(0,2)(0.85,2.5)
\rput(0.44,2.45){$e_4$}

\psline[linewidth=1.6pt,arrowinset=0]{<-}(0,2)(0,3)
\rput(-0.2,2.5){$e_5$}

\psline[linewidth=1.6pt,arrowinset=0]{->}(4.7,1)(4.7,0)
\rput(4.9,0.5){$e_6$}

\psline[linewidth=1.6pt,arrowinset=0]{->}(0.85,1.5)(1.85,1.5)
\rput(1.4,1.7){$e_7$}

\psline[linewidth=1.6pt,arrowinset=0]{->}(1.85,1.5)(2.7,1)
\rput(2.2,1.1){$e_9$}

\psline[linewidth=1.6pt,arrowinset=0]{->}(2.7,1)(3.7,1)
\rput(3.2,1.15){$e_{10}$}

\psline[linewidth=1.6pt,arrowinset=0]{->}(3.7,1)(4.7,1)
\rput(4.2,1.15){$e_{13}$}

\psline[linewidth=1.6pt,arrowinset=0]{->}(4.7,1)(5.7,1)
\rput(5.2,1.15){$e_{14}$}



\rput(2,0){$E^2_3=\{7,9,10\}$}
\rput(3.9,-0.2){$E^1_4=\{13\}$}

\pscurve[linestyle=dashed](0,2)(0.6,2.1)(0.85,2.5)(0,3)(-.4,2.5)(0,2)
\rput(.7,3){$B_1$}

\pscurve[linestyle=dashed](0.85,1.5)(1.5,2)(3.7,1)(2.3,0.5)(1.5,0.5)(0.85,1.5)
\rput(2,2.2){$B_3$}

\pscurve[linestyle=dashed](3.7,1)(4.5,1.5)(5.7,1)(4.3,0.7)(3.7,1)
\rput(4.7,1.7){$B_4$}

\pscurve[linestyle=dashed](4.7,1)(5.1,0.5)(4.7,0)(4.4,0.5)(4.7,1)
\rput(4.2,0.4){$B_2$}

\endpspicture &

\psset{xunit=1.4cm,arrows=->,yunit=1.5cm,linewidth=0.5pt,radius=0.1mm,arrowsize=7pt,
labelsep=1.5pt,fillcolor=black}

\pspicture(0,0)(3,2)

\cnodeput(1.5,0.3){1}{$E_1$}
\cnodeput(3,0.3){3}{$E_3$}
\cnodeput(1.5,2.3){2}{$E_2$}
\cnodeput(3,2.3){4}{$E_4$}

\pscurve(1.7,2.3)(2.3,2.5)(2.8,2.3)
\rput(2.3,2.65){$E_4^1$}
\ncline{4}{3}
\naput{$E_3^2$}
      
\endpspicture

\end{array}
$

\vspace{.5cm}
\caption{A portion of $G(A)$ given in Figure \ref{fig:G(A)bon}
 and the spanning forest $T_{G(A)}$ of $D$ induced by $G(A)$.}
\label{fig8:introalgo1}
\end{figure}

For any $\beta \in S^*$, we define the following objects.
A directed path in $D$ is a \emph{ $\beta$-path}\index{path@$\beta$-path} if for each vertex 
$E_\ell$ of the path, we have $ g_\beta(E_\ell)=1$.
A \emph{ $\beta$-fork}\index{fork@$\beta$-fork} in $D$ is a graph consisting of three  directed paths $P_1=(E_{l_1},(E_{l_2},
E_{l_1}),E_{l_2},(E_{l_3},E_{l_2}),E_{l_3},\ldots, (E_{l_{t}},
E_{l_{t-1}}),E_{l_t})$, $P_2=(E_{l_t},(E_{l_{t+1}},E_{l_t}),
E_{l_{t+1}},\ldots,(E_{l_u},E_{l_{u-1}}),E_{l_{u}})$ and $P_3=(E_{l_t},(E_{l_{t+1}'},
E_{l_t}),\ldots,(E_{l_{v}'},E_{l_{v-1}'}),E_{l_{v}'})$  such that 
$1\le t\le u,v$ 
and $E_{l_t}$ is the unique vertex belonging to two of these paths. 
We have the equalities $g_\beta (E_{l_k})=1$
for $t+1 \le k \le u$, $ g_\beta (E_{l_k'})=1$ for $t+1\le k \le v$,
$g_\beta (E_{l_k})=2$ for $1\le k\le t$. Moreover, if $(E_{l_{t+1}},E_{l_t})$ and $(E_{l_{t+1}'},E_{l_t})$ 
have the same label, say $E_{l_t}^k$, then $g_\beta(E_{l_t}^k)=2$. If a $\beta$-fork is a path, then it is called  \emph{simple}\index{simple $\beta$-fork}. 
If a $\beta$-fork is simple (or equivalently $P_2=(E_{l_t})$ or $P_3=(E_{l_t})$), then the node $E_{l_t}$ is said to be \emph{$\beta$-free}\index{free@$\beta$-free}.
A source vertex of a $\beta$-path or a  $\beta$-fork is also called a \emph{$\beta$-free} vertex.
See Figure \ref{fig:betafork}. 

\begin{figure}[h!]
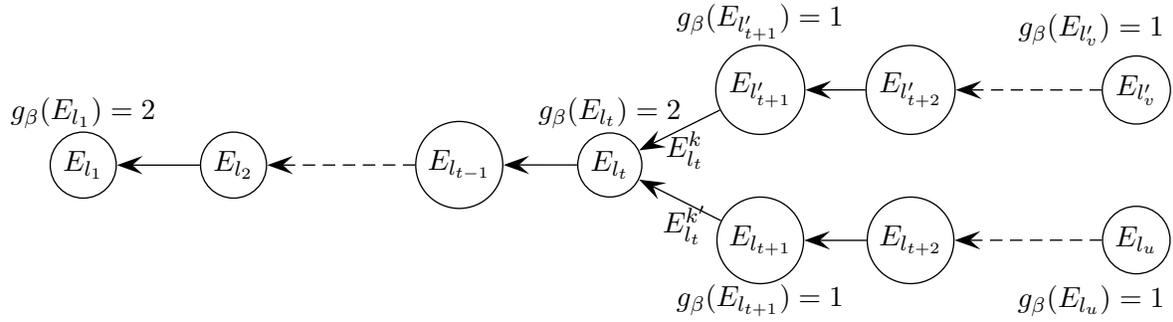

\vspace{.3cm}

\psset{xunit=2cm,yunit=1cm,linewidth=0.5pt,radius=0.1mm,arrowsize=7pt,
labelsep=1.5pt,fillcolor=black}

\pspicture(0,0)(15,3.5)

\cnodeput(0.5,1.5){1}{$E_{l_1}$}
\rput(0.5,2.2){$g_\beta(E_{l_1})=2$}
\cnodeput(1.5,1.5){2}{$E_{l_2}$}
\cnodeput(3,1.5){3}{$E_{l_{t-1}}$}
\cnodeput(4,1.5){4}{$E_{l_t}$}
\rput(4,2.2){$g_\beta(E_{l_t})=2$}

\rput(4.5,1.7){$E_{l_t}^k$}
\rput(4.5,0.7){$E_{l_t}^{k'}$}

\cnodeput(5,2.5){5}{$E_{l_{t+1}'}$}
\rput(5,3.4){$ g_\beta(E_{l_{t+1}'})=1$}
\cnodeput(6,2.5){6}{$E_{l_{t+2}'}$}
\cnodeput(7.5,2.5){7}{$E_{l_v'}$}
\rput(7.2,3.3){$ g_\beta(E_{l_{v}'})=1$}

\cnodeput(5,0.5){8}{$E_{l_{t+1}}$}
\rput(5,-0.3){$ g_\beta(E_{l_{t+1}})=1$}
\cnodeput(6,0.5){9}{$E_{l_{t+2}}$}
\cnodeput(7.5,0.5){10}{$E_{l_u}$}
\rput(7.2,-0.3){$g_\beta(E_{l_u})=1$}

\ncline{<-}{1}{2}
\ncline[linestyle=dashed]{<-}{2}{3}
\ncline{<-}{3}{4}

\ncline{<-}{4}{5}
\ncline{<-}{5}{6}
\ncline[linestyle=dashed]{<-}{6}{7}

\ncline{<-}{4}{8}
\ncline{<-}{8}{9}
\ncline[linestyle=dashed]{<-}{9}{10}

\endpspicture
\vspace{0.5cm}

If $k=k'$, then $g_\beta(E_{l_t}^k)=2$.

\caption{An illustration of a $\beta$-fork.}  
\label{fig:betafork}
\end{figure}

\begin{figure}[h!]
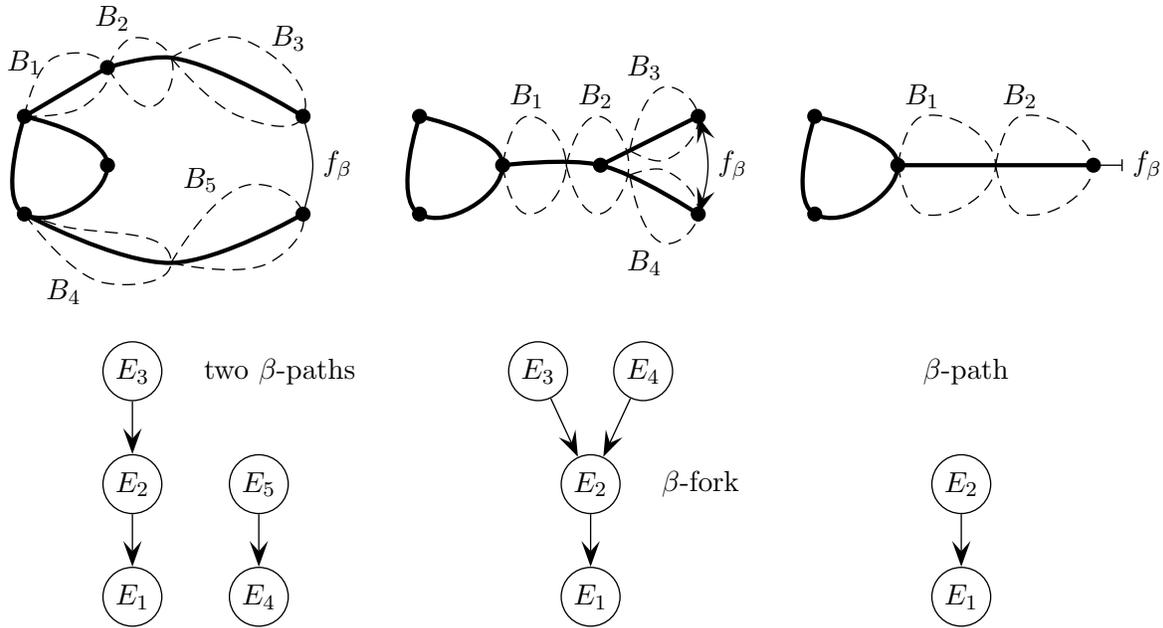


\vspace{1.1cm}
$
\begin{array}{ccc}
\begin{array}{c}
\psset{xunit=1.3cm,yunit=1.3cm,linewidth=0.5pt,radius=0.1mm,arrowsize=7pt,
labelsep=1.5pt,fillcolor=black}

\pspicture(0,-0.5)(3.5,2.5)

\pscircle[fillstyle=solid](0,1){.1}
\pscircle[fillstyle=solid](0,2){.1}
\pscircle[fillstyle=solid](0.85,1.5){.1}
\pscircle[fillstyle=solid](0.85,2.5){.1}

\pscircle[fillstyle=solid](2.85,1){.1}
\pscircle[fillstyle=solid](2.85,2){.1}

\pscurve[linewidth=1.6pt](0,2)(0.85,1.5)(0,1)(0,2)
\psline[linewidth=1.6pt,arrowinset=0]{-}(0,2)(0.85,2.5)

\pscurve[linewidth=1.6pt](0.85,2.5)(1.5,2.6)(2.85,2)
\pscurve[linewidth=1.6pt](0,1)(1.5,.5)(2.85,1)

\pscurve(2.85,1)(2.95,1.5)(2.85,2)
\rput(3.2,1.5){$f_\beta$}

\pscurve[linestyle=dashed](0,2)(0.6,2.1)(0.85,2.5)(0.3,2.6)(0,2)
\rput(0,2.55){$B_1$}

\pscurve[linestyle=dashed](0.85,2.5)(1.1,2.8)(1.5,2.6)(1.3,2.1)(.85,2.5)
\rput(.9,3){$B_2$}

\pscurve[linestyle=dashed](1.5,2.6)(2.2,2)(2.85,2)(2.2,2.8)(1.5,2.6)
\rput(2.7,2.8){$B_3$}

\pscurve[linestyle=dashed](0,1)(.8,.3)(1.5,.5)(.3,.9)(0,1)
\rput(.4,.2){$B_4$}

\pscurve[linestyle=dashed](1.5,.5)(2.5,.5)(2.85,1)(2.3,1.3)(1.5,.5)
\rput(1.8,1.35){$B_5$}

\endpspicture \\

\psset{xunit=1.4cm,arrows=->,yunit=1.5cm,linewidth=0.5pt,radius=0.1mm,arrowsize=7pt,
labelsep=1.5pt,fillcolor=black}

\pspicture(0,0)(2,2)

\cnodeput(.4,0){1}{$E_1$}
\cnodeput(.4,1){2}{$E_2$}
\cnodeput(.4,2){3}{$E_3$}
\cnodeput(1.6,0){4}{$E_4$}
\cnodeput(1.6,1){5}{$E_5$}

\ncline{3}{2}
\ncline{2}{1}
\ncline{5}{4}

\rput(1.8,2){two $\beta$-paths}
      
\endpspicture

\end{array} &

\begin{array}{c}
\psset{xunit=1.3cm,yunit=1.3cm,linewidth=0.5pt,radius=0.1mm,arrowsize=7pt,
labelsep=1.5pt,fillcolor=black}

\pspicture(0,-0.5)(3.5,2.5)

\pscircle[fillstyle=solid](0,1){.1}
\pscircle[fillstyle=solid](0,2){.1}
\pscircle[fillstyle=solid](0.85,1.5){.1}
\pscircle[fillstyle=solid](1.85,1.5){.1}

\pscircle[fillstyle=solid](2.85,1){.1}
\pscircle[fillstyle=solid](2.85,2){.1}

\pscurve[linewidth=1.6pt](0,2)(0.85,1.5)(0,1)(0,2)

\pscurve[linewidth=1.6pt](0.85,1.5)(1.85,1.5)(2.85,1)
\psline[linewidth=1.6pt,arrowinset=0](1.85,1.5)(2.85,2)

\pscurve{<->}(2.85,1)(2.95,1.5)(2.85,2)
\rput(3.2,1.5){$f_\beta$}


\pscurve[linestyle=dashed](0.85,1.5)(1.1,2)(1.5,1.5)(1.1,1)(.85,1.5)
\rput(1.1,2.2){$B_1$}

\pscurve[linestyle=dashed](1.5,1.5)(1.8,1)(2.1,1.3)(2.1,1.7)(1.8,2)(1.5,1.5)
\rput(1.8,2.2){$B_2$}

\pscurve[linestyle=dashed](2.15,1.4)(2.5,.7)(2.85,1)(2.5,1.45)(2.15,1.4)
\rput(2.3,.5){$B_4$}

\pscurve[linestyle=dashed](2.15,1.65)(2.5,2.3)(2.85,2)(2.5,1.55)(2.15,1.65)
\rput(2.3,2.5){$B_3$}

\endpspicture \\

\psset{xunit=1.4cm,arrows=->,yunit=1.5cm,linewidth=0.5pt,radius=0.1mm,arrowsize=7pt,
labelsep=1.5pt,fillcolor=black}

\pspicture(0,0)(2,2)

\cnodeput(1,0){1}{$E_1$}
\cnodeput(1,1){2}{$E_2$}
\cnodeput(.5,2){3}{$E_3$}
\cnodeput(1.5,2){4}{$E_4$}

\ncline{3}{2}
\ncline{4}{2}
\ncline{2}{1}

\rput(2,1){ $\beta$-fork}
      
\endpspicture

\end{array} & 

\begin{array}{c}
\psset{xunit=1.3cm,yunit=1.3cm,linewidth=0.5pt,radius=0.1mm,arrowsize=7pt,
labelsep=1.5pt,fillcolor=black}

\pspicture(0,-0.5)(3,2.5)

\pscircle[fillstyle=solid](0,1){.1}
\pscircle[fillstyle=solid](0,2){.1}
\pscircle[fillstyle=solid](0.85,1.5){.1}

\pscircle[fillstyle=solid](2.85,1.5){.1}

\pscurve[linewidth=1.6pt](0,2)(0.85,1.5)(0,1)(0,2)

\pscurve[linewidth=1.6pt](0.85,1.5)(1.85,1.5)(2.85,1.5)

\psline[arrowinset=.5,arrowlength=1.5]{|-}(3.15,1.5)(2.85,1.5)
\rput(3.4,1.5){$f_\beta$}


\pscurve[linestyle=dashed](0.85,1.5)(1.1,2)(1.85,1.5)(1.1,1)(.85,1.5)
\rput(1.1,2.2){$B_1$}

\pscurve[linestyle=dashed](1.85,1.5)(2.1,1)(2.85,1.5)(2.1,2)(1.85,1.5)
\rput(2.1,2.2){$B_2$}

\endpspicture \\

\psset{xunit=1.4cm,arrows=->,yunit=1.5cm,linewidth=0.5pt,radius=0.1mm,arrowsize=7pt,
labelsep=1.5pt,fillcolor=black}

\pspicture(0,0)(2,2)

\cnodeput(1,0){1}{$E_1$}
\cnodeput(1,1){2}{$E_2$}

\ncline{2}{1}

\rput(1,2){ $\beta$-path}
      
\endpspicture

\end{array}
\end{array}
$

\vspace{.3cm}
\caption{An illustration of different configurations of $T_{G(A)}(\{E_\ell\, :\, \beta \in f^*(E_\ell) \})$ (at the bottom) with respect to the fundamental circuit of a nonbasic edge $f_\beta$ with $\beta \in S^*$ (at the top). In each configuration, the basic cycle in $G(A)$ is drawn.}
\label{figcyclicfund}
\end{figure}

\noindent
For all $\beta \in S^*$, let $R_\beta =s(A_{\bullet \beta})\cap R^*$. If $R_\beta\neq R^*$ ($\beta \in S^*$), then
the set $R_\beta$ is called an \emph{interval}\index{interval}. 
If $A$ has an $R^*$-cyclic representation $G(A)$, then an interval $R_\beta$ corresponds to the edge index set of a consistently oriented path in the basic cycle, denoted as $p_\beta$, which is called an \emph{interval}\index{interval} in $G(A)$.
A forest $T_\Theta=(V',\Theta)$ in $D$ is said to be \emph{feasible}\index{feasible forest} if it has the following properties.

\begin{itemize}

\item[$\Pi_1$:] For all $\beta \in S^*$, the subgraph $T_{\Theta}(\{E_\ell\in V'\,\,:\,\, \beta \in f^*(E_\ell)\})$ is a $\beta$-fork, a $\beta$-path or a union of two disjoint $\beta$-paths.

\item[$\Pi_2$:]  For all $\beta\in S^*$, If $s_{\frac{1}{2}}(A_{\bullet \beta})\neq \emptyset$, then the subgraph $T_{\Theta}(\{E_\ell\in V' \, :\, \beta \in f^*(E_\ell)\})$ is a $\beta$-path.

\item[$\Pi_3$:] For all $\beta, \beta' \in S^*$  such that $R_\beta \neq R_{\beta'}$, the subgraph $T_{\Theta}(\{E_\ell\in V' \, :\, \beta,\beta' \in f^*(E_\ell)\,,\,$\\ $ E_\ell^I(A_{\bullet \beta}) \sim_{E_\ell} E_\ell^I(A_{\bullet \beta'}) \})$ is either a $\beta$-path or a $\beta'$-path.

\end{itemize}

We denote by $Sink(D)$ the set of sink vertices of $D$.
We observe that the property $\Pi_1$ is equivalent to the following.

\begin{itemize}

\item[$\Pi_1^*$:] For each vertex $E_{\ell}\in V$ and $1\le h\le m(\ell)$, we have 
$\begin{displaystyle}
\sum_{\substack{E_{\ell'}\,:\, (E_{\ell'},E_\ell)_{E^h_\ell}\in 
\Theta }}  
 g(E_{\ell'})\le g(E_\ell^h)
\end{displaystyle}$. Moreover, 
$\displaystyle{\sum_{E_\ell\,\in\, Sink(D)}} g(E_\ell) \le {\bf 2_{1\times s}}$.

\end{itemize}

On the other hand, in the case where $R^*=\{1,\rho\}$, we define $S_1=\{j\in S^* \, :\, 1 \in s(A_{\bullet j}),\, \rho \notin s(A_{\bullet j})\}$ and $S_2=\{j\in S^* \, :\, 1 \notin s(A_{\bullet j}), \,\rho \in s(A_{\bullet j})\}$. A forest $T_\Theta=(V',\Theta)$ in $D$ is said to be \emph{right-feasible}\index{right-feasible!forest} if it satisfies $\Pi_1$ and the following property.

\begin{itemize}

\item[$\Pi_2'$:]  For any $\beta\in S_k$ with $k\in \{1,2\}$, the subgraph $T_{\Theta}(\{E_\ell\in V' \, :\, \beta \in f^*(E_\ell)\})$ is a $\beta$-path.

\end{itemize}

Suppose that the matrix $A$ has an $R^*$-cyclic representation $G(A)$. For any $\beta \in S^*$, the different configurations of  the graph $T_{G(A)}(\{E_\ell \,: \, \beta \in f^*(E_\ell) \})$ with respect to the fundamental circuit of the nonbasic edge $f_\beta$ are shown in Figure \ref{figcyclicfund}.
For any $\beta,\beta'\in S^*$, denote by 
$v_{\beta_1}$ and $v_{\beta_2}$ (respectively, $v_{\beta_1'}$ and $v_{\beta_2'}$) the endnodes of the nonbasic edge $f_\beta$ (respectively, $f_{\beta'}$). 
(If $f_\beta$ has one endnode, then $v_{\beta_1}= v_{\beta_2}$.)
The pair of nonbasic edges $\{f_\beta,f_{\beta'}\}$ is said to be \emph{singular}\index{singular} if $\{v_{\beta_1}^* , v_{\beta_2}^*\}= \{ v_{\beta_1'}^* , v_{\beta_2'}^*\}$. Let us see two auxiliary lemmas.

\begin{lem}\label{lemdigraph12}
Suppose that $A$ has a cyclic representation $G(A)$ such that $e_1$ and $e_\rho$ are edges of the basic cycle incident with a common central node. If $1,\rho \in s(A_{\bullet j})$ for some column index  $j$, then the whole basic cycle is contained in the fundamental circuit of $f_j$.
\end{lem}

\begin{proof}
Since $A$ is nonnegative, the proof follows from Corollary \ref{corBidirectedCircuit} and Lemma \ref{lemdefiWeight2} (see Figure \ref{fig:Apositive}).
\end{proof}\\

\begin{lem}\label{lemdigraphsim}
Suppose that $A$ has an $R^*$-cyclic representation $G(A)$. Then
the following holds.
\begin{itemize}

\item[1)] For any $1\le \ell\le b$ and $\beta \in f^*(E_\ell)$, if the set $E_\ell^{II}(A_{\bullet \beta})\neq \emptyset$, then $E_\ell^I(A_{\bullet \beta}) \sim_{E_\ell} E_\ell^{II}(A_{\bullet \beta})$.

\item[2)] For any $1\le \ell,\ell'\le b$ and $1\le k \le m(\ell)$, if $(E_{\ell'},E_\ell)_{E_\ell^k}\in D$ and $\beta, \beta' \in f^*(E_{\ell'})$, then $\beta, \beta' \in f^*(E_\ell)$, 
$E_{\ell'}^I(A_{\bullet \beta}) \sim_{E_{\ell'}} E_{\ell'}^{I}(A_{\bullet \beta'})$ and $E_\ell^I(A_{\bullet \beta}) \sim_{E_\ell} E_\ell^I(A_{\bullet \beta'})$.

\end{itemize}

\end{lem}

\begin{proof}
Using Corollary \ref{corBidirectedCircuit} and
Lemmas  \ref{lemdefiWeight2} and \ref{lemdigraphBlI}, one may deduce that the $B_\ell$-paths with edge index sets 
$E_\ell^I(A_{\bullet \beta})$ and $E_\ell^{II}(A_{\bullet \beta})$
are both leaving $v_\ell$. Then, the proof of part 1 follows from Lemma \ref{lemdigraphutile}.

Now let $1\le \ell,\ell'\le b$ and $1\le k \le m(\ell)$
such that $(E_{\ell'},E_\ell)_{E_\ell^k}\in D$ and there exist two elements $\beta, \beta' \in f^*(E_{\ell'})$. By definition of an arc in $D$, the set $J_{\ell'}^2$ is empty. It results that 
$E_{\ell'}^I(A_{\bullet \beta}) \sim_{E_{\ell'}} E_{\ell'}^I(A_{\bullet \beta'})$. Since the columns $A_{\bullet \beta}$ and $A_{\bullet \beta'}$ generate the same $E_\ell$-path $E_\ell^k$, 
we have the equality $E_\ell^i(A_{\bullet \beta})=E_\ell^{i'}(A_{\bullet \beta'})$ for some $i,i'\in \{I, II \}$. Finally, part 1 of the lemma implies that $E_\ell^I(A_{\bullet \beta}) \sim_{E_\ell} E_\ell^{I}(A_{\bullet \beta'})$.
\end{proof}\\

The following Lemma shall be useful in Section \ref{sec:cyc}.

\begin{lem}\label{lembonsaisim}
Let $1\le \ell,\ell'\le b$ and suppose that the bonsai matrices $N_\ell$ and $N_{\ell'}$ are network matrices. Then the following holds.
\begin{itemize}

\item[1)] For any $\beta \in f^*(E_\ell)$, if the set $E_\ell^{II}(A_{\bullet \beta})\neq \emptyset$, then $E_\ell^I(A_{\bullet \beta}) \sim_{E_\ell} E_\ell^{II}(A_{\bullet \beta})$.

\item[2)] If $(E_{\ell'},E_\ell)_{E_\ell^k}\in D$ and $\beta, \beta' \in f^*(E_{\ell'})$, then $\beta, \beta' \in f^*(E_\ell)$, 
$E_{\ell'}^I(A_{\bullet \beta}) \sim_{E_{\ell'}} E_{\ell'}^{I}(A_{\bullet \beta'})$ and $E_\ell^I(A_{\bullet \beta}) \sim_{E_\ell} E_\ell^I(A_{\bullet \beta'})$.

\end{itemize}

\end{lem}

\begin{proof}
The proof is similar to the proof of Lemma \ref{lemdigraphsim} by using Lemma \ref{lembonsainet2} instead of Lemma \ref{lemdigraphutile}.
\end{proof}\\

Now, we are ready to state the main result about the forest $T_{G(A)}$, whenever $A$ has an $R^*$-cyclic representation $G(A)$.

\begin{thm}\label{thmdigraphfeasible}
Suppose that $A$ has an $R^*$-cyclic representation $G(A)$. 
Then $T_{G(A)}$ is a feasible spanning forest of $D$.
\end{thm}

\begin{proof}
The properties $\Pi_1$ and $\Pi_2$ follow from Proposition \ref{propBlsubstems} and
the description of the different types of fundamental circuits
described at page \pageref{mycounter2}
(see Figures \ref{fig:Apositive} and \ref{figcyclicfund}). 

Let us show the property $\Pi_3$. Let $\beta,\beta'\in S^*$ such that $R_\beta \neq R_{\beta'}$ and $\{E_\ell \, : \, \beta, \beta' \in f^*(E_\ell) \}\neq \emptyset$. We may assume that $R_{\beta'}\neq R^*$, so $R_{\beta'}$ is an interval.

Suppose first that the pair $\{f_\beta, f_{\beta'}\}$ is not singular. W.l.o.g, we have $v_{\beta_1'}^*=v_{\beta_1}^*$, $v_{\beta_2'}^*\neq v_{\beta_1}^*$ and $v_{\beta_2'}^*\neq v_{\beta_2}^*$. Therefore, for any $E_\ell$ such that $\beta,\beta' \in f^*(E_\ell)$, we have $v_\ell*=v_{\beta_1'}^*$. So
the graph $T_{G(A)}(\{E_\ell\,:\, \beta,\beta'\in f^*(E_\ell)\})$ is a $\beta'$-path (it can not be a $\beta'$-fork, otherwise $R_{\beta'}=R^*$). Hence, by Lemma \ref{lemdigraphsim} (part 2), if the subgraph $T_{G(A)}(\{E_\ell\,:\, \beta,\beta'\in f^*(E_\ell)\})$ 
has at least two vertices, then
$T_{G(A)}(\{E_\ell\,:\, \beta,\beta'\in f^*(E_\ell)\})=T_{G(A)}(\{E_\ell\,:\, \beta,\beta'\in f^*(E_\ell)\,,\,E_\ell^I(A_{\bullet \beta}) \sim_{E_\ell} E_\ell^I(A_{\bullet \beta'}) \})$. Thus
$T_{G(A)}(\{E_\ell\,:\, \beta,\beta'\in f^*(E_\ell)\,,\,E_\ell^I(A_{\bullet \beta}) \sim_{E_\ell} E_\ell^I(A_{\bullet \beta'}) \})$ is a $\beta'$-path.

Now suppose that the pair $\{f_\beta, f_{\beta'}\}$ is singular. 
It follows that $R^*=R_\beta \uplus R_{\beta'}$. Recall that
$w_1,\ldots,w_\rho$ are the vertices of the basic cycle, $e_1=[w_1,w_\rho]$ and $e_i=]w_{i-1},w_i]$ for $i=2,\ldots, \rho$.
By Lemma \ref{lemdigraph12}, we may assume that $1\in R_\beta \verb"\" R_{\beta'}$ and $\rho \in R_{\beta'} \verb"\" R_\beta$. Let $k=\max\{ i\,:\, e_i \in p_\beta\}$. Clearly, $k< \rho$,  $e_k$ ($\in p_\beta$) enters $w_k$
and $]w_k,w_{k+1}] \in p_{\beta'}$. If there exists a bonsai $B_\ell$ such that $v_\ell=w_k$ and $\beta,\beta' \in f^*(E_\ell)$, then the $B_\ell$-paths generated by $f_\beta$ and $f_{\beta'}$  leave and enter $v_\ell$, respectively. Thus, by Lemma \ref{lemdigraphutile}, we have that $E_\ell^I(A_{\bullet \beta}) \nsim_{E_\ell} E_\ell^I(A_{\bullet \beta'})$. Hence, by Lemma \ref{lemdigraphsim}, it does not exist an edge $(E_{\ell'},E_\ell)\in D$ such that $\beta,\beta'\in f^*(E_{\ell'})$.\\
On the other hand, if there is a bonsai $B_\ell$ such that $v_\ell^*=w_\rho$ and $\beta,\beta' \in f^*(E_\ell)$, then the $B_\ell$-paths generated by $f_\beta$ and $f_{\beta'}$ are both leaving $v_\ell$. Then by Lemma \ref{lemdigraphutile}
$E_\ell^I(A_{\bullet \beta}) \sim_{E_\ell} E_\ell^I(A_{\bullet \beta'})$. Therefore $T_{G(A)}(\{E_\ell\,:\, \beta,\beta'\in f^*(E_\ell)\,,\,E_\ell^I(A_{\bullet \beta}) \sim_{E_\ell} E_\ell^I(A_{\bullet \beta'}) \})$ is a $\beta$-path and a $\beta'$-path.
\end{proof}\\

Eventually, we can provide a lemma analoguous to Lemma \ref{lemdigraphsim} in the case where $A$ has a $\{1,\rho\}$-central representation $G(A)$, by restricting ourselves to the set of bonsais on the right of $\{e_1,e_\rho\}$. Then, we can prove the second main theorem.

\begin{thm}\label{thmdigraphrightfeasible}
Suppose that $A$ has a $\{1,\rho\}$-central representation $G(A)$. Then $T_{G_1(A)}$ is a right-feasible forest in $D$.
\end{thm}

\begin{proof}
The property $\Pi_2'$ is simple to show and the remaining part of the proof is similar to the proof of Theorem \ref{thmdigraphfeasible}.
\end{proof}\\

\section{The procedure Forest}\label{sec:For}

Let $A$ be a matrix with entries $0$, $1$, $2$ or $\frac{1}{2}$ and $R^*$ a row index subset of $A$. Let $D=(V,\Upsilon)$ be a digraph as constructed in Section \ref{sec:DefDigraphD} and $D'=(V',\Upsilon')\subseteq D$ an induced subgraph without any directed cycle.
In this section,  
we describe a procedure called Forest taking $A$, $D'$ as well as a subset $V_0$ of $V'$ as input, and computing a spanning forest of $D'$, when it does not stop. 

Let $V_0'$ denote the set of $E_\ell \in V'$ such that $E_\ell$ is a sink vertex of $D'$, or there exists $\beta \in f^*(E_\ell)$ such that $s_{\frac{1}{2}}(A_{\bullet \beta})\neq \emptyset$, or there exist $\beta,\beta' \in f^*(E_\ell)$ such that $R_\beta$ and $R_{\beta'}$ are not equal ($R_\beta\neq R_{\beta'}$) and $E_\ell^I(A_{\bullet \beta}) \sim_{E_\ell} E_\ell^I(A_{\bullet \beta'})$.

On the other hand, given a row index $\rho\neq 1$ such that $R^*=\{1,\rho\}$, we define $S_1=\{j\in S^* \, :\, 1 \in s(A_{\bullet j}),\, \rho \notin s(A_{\bullet j})\}$, $S_2=\{j\in S^* \, :\, 1 \notin s(A_{\bullet j}), \,\rho \in s(A_{\bullet j})\}$ and let
$V_0''$ denote the set of $E_\ell \in V'$ such that $E_\ell$ is a sink vertex of $D'$ or $f^*(E_\ell)\cap (S_1 \cup S_2)\neq \emptyset$. We will prove the following.

\begin{thm}\label{thmForest1}
The procedure Forest with input $A$, $D'$ and $V_0=V_0'$ 
outputs a feasible spanning forest of $D'$ if and only if one exists.
\end{thm}

\begin{thm}\label{thmForest2}
Suppose $R^*=\{ 1,\rho\}$. The procedure Forest with input $A$, $D'$ and $V_0=V_0''$ 
outputs a right-feasible spanning forest of $D'$ if and only if one exists.
\end{thm}

The construction of a spanning forest of $D'$ can be reduced to a sequence of 2-SAT problems.
Assume that vertex sets $V_0,\ldots, V_{k-1}$ in  $V'$ and a forest $T_{\Theta}=(U=\cup _{j=0}^{k-1}V_j,\Theta)$ have been 
constructed, and $V_k\subseteq V' \verb"\" U$ for some $k\geq 1$.
We construct an instance $I_k$ of 2-SAT.

For any $E_{\ell}\in V_k$, $E_u\in U$ and $1\le h\le m(u)$,
an arc $(E_{\ell},E_{u})_{E^h_u}\in \Upsilon'$  is said to be \emph{legal}\index{legal arc}
if and only if we have the inequality
$$ g(E_{\ell}) +
\begin{displaystyle}
\sum_{\substack{E_{u'}\,:\, (E_{u'},E_{u})_{E^h_u}\in 
\Theta }}  
 g(E_{u'})\le g(E_u^h)
\end{displaystyle}.$$
For all $E_{\ell}\in V_k$, if there exists a legal arc $(E_{\ell},
E_{u})\in \Upsilon'$
labeled $E^h_u$, we define the variable 
$X^l_{E_u^h}$. The variable $X^l_{E_u^h}$ gets a true value if and only 
if we add the arc $(E_{\ell},E_{u})_{E_u^h}$ in $\Theta$. 
Now let us see the clauses. If there
exist  two legal arcs  $(E_{\ell},E_{u})_{E_u^h},
(E_{\ell'},E_{u})_{E_u^h} \in \Upsilon'$ such that  
$$ g_\beta (E_{\ell})+ g_\beta (E_{\ell'})+
\begin{displaystyle}
\sum_{\substack{E_{u'}\,:\, (E_{u'},E_{u})_{E_u^h}
\in \Theta }}  
 g_\beta (E_{u'})
\end{displaystyle} > g_\beta (E_{u}^h)$$ for some
$E_{\ell},E_{\ell'}\in V_k$ ($\ell\neq \ell'$), $E_{u}\in U$ and $\beta\in S^*$, we 
add in $I_k$ the clauses 
$X^l_{E_u^h}\vee X^{\ell'}_{E_u^h}$ and $\bar X^l_{E_u^h}\vee \bar X ^{\ell'}_{E_u^h}$. 
Thus if these clauses are true, exactly one
of the arcs $(E_{\ell},E_{u})_{E_u^h}$ and
$(E_{\ell'},E_{u})_{E^h_u} $ will be in $T_\Theta$.
The variable $X^l_{E_u^h}$ is said to be \emph{associated to the vertex $E_\ell\in V$ }\index{variable associated to a vertex}.

Let us assume that there are at
most two variables associated to each vertex in $V_k$.
If exactly one variable, say $X_{E_u^h}^l$, is associated to a vertex $E_\ell$
and there exists some $\beta \in f^*(E_\ell)$ such that 
$\displaystyle{\sum_{E_u \in Sink(T_{\Theta})}} g_\beta (E_u)=2$, 
we set $X_{E_u^h}^l=1$. If there are exactly two variables, say $X_{E_u^h}^l$ and $X_{E_{u'}^{h'}}^l$, associated to $E_\ell$, then we include in $I_k$ the clauses $X_{E_u^h}^l\vee X_{E_{u'}^{h'}}^l$ and  $\bar X_{E_u^h}^l\vee\bar  X_{E_{u'}^{h'}}^l$. 
Note as before that if these clauses are true, exactly one of both legal arcs leaving $E_\ell$ will be in $T_\Theta$.

In the procedure Forest, we use a subroutine called Spanning$V_0$ which outputs a spanning forest of $D'(V_0)$.

\begin{tabbing}
\textbf{Procedure\,\,Spanning$V_0$($D'(V_0)$)}\\

\textbf{Input:} An induced subgraph $D'(V_0) \subseteq D'$  where $V_0\subseteq V'$.\\
\textbf{Output:} A spanning forest $T_\Theta=(V_0,\Theta)$ of $D'(V_0)$.\\ 
1)\verb"  "\= let $\Theta=\emptyset$; $U=\emptyset$; $\tilde D=D'(V_0)$, $T_\Theta=(U,\Theta)$;\\
2) \> {\bf while }\= $\tilde D \neq \emptyset$ {\bf do}\\
3)  \>            \> compute $W$ the set of sink vertices of $\tilde D$;\\ 
  \>            \> add each legal arc from a node of $W$ to a node of $U$ in $\Theta$; \\
  \>            \> $\tilde D=\tilde D \verb"\" W$; $U=U \cup W$;\\
\> {\bf endwhile }\\
\> output $T_\Theta=(V_0,\Theta)$;
\end{tabbing}

\begin{tabbing}
\textbf{Procedure\,\,Forest($D'$,$V_0$)}\\

\textbf{Input: }\= An induced subgraph $D'\subseteq D$ without any directed\\
 \> cycle and a subset $V_0\subseteq V'$.\\
\textbf{Output: } either a spanning forest $T_\Theta=(V',\Theta)$ of $D'$, or stops.\\ 

1)\verb"  "\=  call {\tt Spanning$V_0$}($D'(V_0)$);\\
2)\> let $k=1$; $U=V_0$; $\tilde D=D'\verb"\" V_0$; $T_\Theta=(U,\Theta)$;\\
3) \> {\bf while }\= $\tilde D\neq \emptyset$ {\bf do}\\
4) \>            \> compute $V_k$ the set of sink vertices of $\tilde D$;\\ 
5)    \>		 \>  if some vertex of $V_k$ has more than two associated                                variables: STOP;\\
6) \>            \> compute a truth assignment of $I_k$, if one exists, otherwise STOP; \\
\>           \>  for each true variable,  
                    add the corresponding arc in $\Theta$;\\    
 \>            \> $\tilde D=\tilde D\verb"\"V_k$; $U=U \cup V_k$ and 
                    $k=k+1$;\\
\> {\bf endwhile }\\
\> output $T_\Theta$;
\end{tabbing}

\noindent
{\bf Proof of Theorem \ref{thmForest1}.}
Suppose that $V_0=V_0'$ and $D'$ has a feasible spanning forest.
Let  $n$ be the number of passages in loop 3. 
For all $0\le k\le n$, let $\Theta_k:=\Theta$ as in the procedure Forest after $k$ passages in loop 3 and 
$T_k=( \cup _{j=0}^{k}V_j,\Theta_k)$, where $T_0=(V_0,\Theta_0)$ denotes the forest output by Spanning$V_0$.
Let us prove inductively the following.

\begin{tabbing}
{\bf statement:} \= (i) \,  \= For all $0\le k \le n$, there exists a feasible spanning forest  $T_f$ in $D'$ \\
 \> \> such that $T_f(\cup_{j=0}^k V_j)=T_k$.\\
\> (ii)  \> $T_n$ spans $V'$.
\end{tabbing}

Let $r$ be the number of passages through loop 2 in the procedure SpanningV0.
For $k=0$, statement (i) follows from Claim 2 below. For $j=1,\ldots,r$ denote by $U_j$ the set $U$ after $j$ passages through loop 2. For proving Claim 2, we use the following.\\

\noindent
{\bf Claim 1.} \quad
For all $2\le j \le r$ and any vertex $E_\ell\in U_j$,
there exists $E_{\ell'}\in U_{j-1}$ such that $(E_\ell,E_{\ell'})\in \Upsilon'$. Moreover, 
for each arc $(E_\ell,E_{\ell'})\in \Upsilon'$ , if  $E_\ell\in
U_j$ and $E_{\ell'}\in U_{j'}$  with $j,j'\geq 1$, 
then  $j>j'$. \\

\noindent
{\bf Proof of Claim 1.} \quad This is due to step 3 in the procedure SpanningV0.
{\hfill$\BBox{\rule{.3mm}{3mm}}$} \\

\noindent
{\bf Claim 2.} \quad For any feasible spanning forest $T_f$ of $D'$, $T_f(V_0)=T_0$.\\

\noindent
{\bf Proof of Claim 2.} \quad
Let $E_\ell\in V_0=V_0'$ and $T_f$ be a feasible spanning forest of $D'$. If $E_\ell$ is a sink vertex of $D'$, then clearly $E_\ell$ is a sink vertex of $T_f$ and $T_0$. 
 
Suppose now there exist $\beta,\beta' \in f^*(E_\ell)$ 
such that $R_\beta\neq R_{\beta'}$ and $E_\ell^I(A_{\bullet \beta}) \sim_{E_\ell}  E_\ell^I(A_{\bullet \beta'})$. By definition of $V_0'$, the graph $T_f(\{ E_{\ell'}\in V_0\,:\, \beta,\beta' \in f^*(E_{\ell'})
\,,\, E_{\ell'}^I(A_{\bullet \beta}) \sim_{E_{\ell'}}   
E_{\ell'}^I(A_{\bullet \beta'}) \})$ is a $\beta$-path. By Claim 1 
and since $D'$ has no directed cycle, it follows  that $T_f(\{ E_{\ell'}\in V_0\,:\, \beta,\beta' \in f^*(E_{\ell'})
\,,\, E_{\ell'}^I(A_{\bullet \beta}) \sim_{E_{\ell'}}   
E_{\ell'}^I(A_{\bullet \beta'})
\})= 
T_0(\{ E_{\ell'}\in V_0\,:\, \beta,\beta' \in f^*(E_{\ell'}) \,,\, E_{\ell'}^I(A_{\bullet \beta}) \sim_{E_{\ell'}}   
E_{\ell'}^I(A_{\bullet \beta'})\})$.
Thus, for any $E_{\ell'}\in V_0$ $1\le k \le m(\ell')$, we have $(E_\ell,E_{\ell'})_{E_{\ell'}^k} \in T_f \Leftrightarrow (E_\ell,E_{\ell'})_{E_{\ell'}^k}\in T_0$. We obtain the same conclusion if $s_{\frac{1}{2}}(A_{\bullet \beta})\neq \emptyset$ for some $\beta\in f^*(E_\ell)$, because $T_f(\{ E_{\ell'}\in V_0\,:\, \beta \in f^*(E_{\ell'}) \})$ is a $\beta$-path.
This completes the proof.
{\hfill$\BBox{\rule{.3mm}{3mm}}$} \\

Let $k\geq 1$. Suppose that the procedure Forest has constructed a forest $T_{k-1}$, $V_k\neq \emptyset$, and there exists a feasible spanning forest $T_f$ of $D'$ such that $T_f(\cup_{j=0}^{k-1} V_j)=T_{k-1}$. Since $T_f$ satisfies the property $\Pi_1$ and $T_f( \cup_{j=0}^{k-1} V_j)= T_{k-1}$, there are at most two legal arcs leaving each vertex of $V_k$. So there are at most two variables associated to each vertex of 
$V_k$ and the procedure Forest does not stop at step 5.
Moreover, for any variable $X_{E_u^h}^l$ in $I_k$, by setting 
$X_{E_u^h}^l=1$ if and only if $(E_\ell,E_u)_{E_u^h}\in T_f$, one obtains a truth assignment of $I_k$. This proves that the procedure Forest computes a forest $T_k$.

Finally, let us prove that there exists a feasible spanning forest $T_f'$ of $D'$ such that $T_f'(\cup_{j=0}^k V_j) = T_k$. By previous arguments, this implies statement (ii). The analogue of Claim 1 is the following.\\

\begin{figure}[h!]
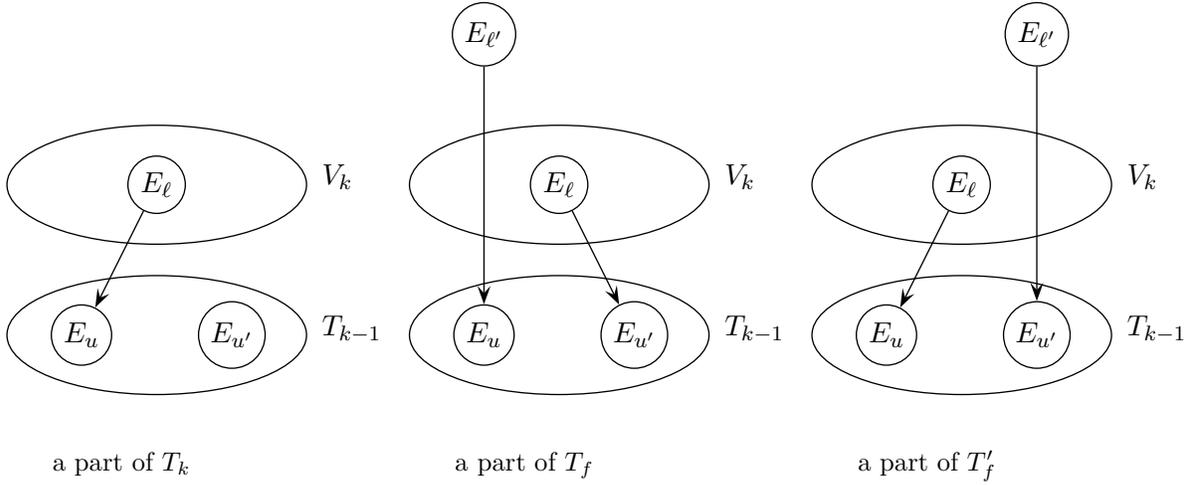

\vspace{.3cm}
$
\begin{array}{ccc}

\psset{xunit=1cm,yunit=1cm,linewidth=0.5pt,radius=0.3mm,arrowsize=5pt,
labelsep=1.5pt}

\pspicture(0,0)(5,5)

\cnodeput(2,3){1}{$E_{\ell}$}
\cnodeput(1,1){2}{$E_u$}
\cnodeput(3,1){3}{$E_{u'}$}

\psellipse(2,3)(2,0.8)
\psellipse(2,1)(2,0.8)

\ncline{->}{1}{2}


\put(4.2,3){\shortstack{$V_k$}}
\put(4.2,1){\shortstack{$T_{k-1}$}}

\put(.5,-0.8){\shortstack{\small{ a part of  $T_k$}}}
\endpspicture

&

\psset{xunit=1cm,yunit=1cm,linewidth=0.5pt,radius=0.3mm,arrowsize=5pt,
labelsep=1.5pt}

\pspicture(0,0)(5,5)

\cnodeput(1,5){0}{$E_{\ell'}$}
\cnodeput(2,3){1}{$E_{\ell}$}
\cnodeput(1,1){2}{$E_u$}
\cnodeput(3,1){3}{$E_{u'}$}

\psellipse(2,3)(2,0.8)
\psellipse(2,1)(2,0.8)

\ncline{->}{1}{3}
\ncline{->}{0}{2}


\put(4.2,3){\shortstack{$V_k$}}
\put(4.2,1){\shortstack{$T_{k-1}$}}

\put(.5,-0.8){\shortstack{\small{ a part of  $T_f$}}}
\endpspicture    &

\psset{xunit=1cm,yunit=1cm,linewidth=0.5pt,radius=0.3mm,arrowsize=5pt,
labelsep=1.5pt}

\pspicture(0,0)(4,5)

\cnodeput(3,5){0}{$E_{\ell'}$}
\cnodeput(2,3){1}{$E_{\ell}$}
\cnodeput(1,1){2}{$E_u$}
\cnodeput(3,1){3}{$E_{u'}$}

\psellipse(2,3)(2,0.8)
\psellipse(2,1)(2,0.8)

\ncline{->}{1}{2}
\ncline{->}{0}{3}


\put(4.2,3){\shortstack{$V_k$}}
\put(4.2,1){\shortstack{$T_{k-1}$}}

\put(.5,-0.8){\shortstack{\small{ a part of  $T_f'$}}}
\endpspicture   

\end{array}
$

\vspace{1cm} 
\caption{An illustration of the proof of Theorem \ref{thmForest1}, in case $(E_{\ell'},E_\ell)\in D'$, $(E_\ell,E_u)_{E_u^h}\in T_k$ and $g(E_\ell)+g(E_{\ell'}) \nleqslant g(E_u^h)$ for some $E_{\ell'} \in V'-\cup_{j=0}^k V_j$, $E_\ell \in V_k$ and $E_u \in \cup_{j=0}^{k-1} V_j$.} 
\label{fig:situationa2}
\end{figure}

\noindent
{\bf Claim 3.} \quad
For any vertex $E_\ell\in V_j$ $1\le j \le n$ and $0\le j' \le j$,
there exists $E_u\in V_{j'}$ such that $(E_\ell,E_u)\in \Upsilon'$. Moreover, 
for each arc $(E_\ell,E_u)\in \Upsilon'$ , if  $E_\ell\in
V_j$ and $E_u \in V_{j'}$  with $j,j'\geq 0$, 
then  $j>j'$. \\

\noindent
{\bf Proof of Claim 3.} \quad This is due to step 4 in the procedure Forest and the transitivity of the relation $\prec_D$.
{\hfill$\BBox{\rule{.3mm}{3mm}}$} \\

Now, let us construct a new forest $T_{f}'$ such that $T_f'(\cup_{j=0}^k V_j)=T_k$. Add any arc $(E_i,E_{i'})$ of $T_f$ such that $E_i,E_{i'} \in V' \verb"\" \cup_{j=0}^k V_j$ or $E_i,E_{i'} \in \cup_{j=0}^{k-1} V_j$
and any arc of $\Theta_k\verb"\" \Theta_{k-1}$ in $T_f'$. Then, for any arc $(E_{\ell'},E_u)_{E_u^h} \in T_f$ such that $
E_\ell \in V' \verb"\"\cup_{j=0}^k V_j$ and $E_u \in \cup_{j=0}^{k-1} V_j$, proceed as follows. By Claim 3, there exists $E_\ell \in V_k$ such that $(E_{\ell'},E_\ell)\in D'$. If $(E_\ell,E_u)_{E_u^h} \notin T_k$ or $g(E_\ell)+g(E_{\ell'}) \le g(E_u^h)$, then add $(E_{\ell'},E_u)_{E_u^h}$ in $T_f'$. Otherwise, since $T_f$ satisfies $\Pi_1$, $E_\ell$ is an isolated node in $T_f$ or $(E_\ell,E_{u'}^{h'}) \in T_f$ for some vertex $E_{u'} \in 
\cup_{j=0}^{k-1} V_j$ and $E_u^h \neq E_{u'}^{h'}$.
By transitivity of the relation $\prec_D$, it results that $(E_{\ell'},E_{u'})_{E_{u'}^{h'}}\in D$, and add $(E_{\ell'},E_{u'})_{E_{u'}^{h'}}$ in $T_f'$. Since $T_f$ satisfies $\Pi_1$, for any $\beta\in f^*(E_{\ell'})$,
we observe that $T_f(\{E_i\in V' \, : \, \beta \in f^*(E_i) \})$ is a non-simple $\beta$-fork and $\beta\in f^*(E_\ell)$, hence $E_{u'} $ has no predecessor $E_i$ in $T_f$, except $E_i=E_\ell$, such that $\beta \in f^*(E_i)$. 
Thus $T_f'$ satisfies $\Pi_1$. This completes the proof of Theorem \ref{thmForest1}.
{\hfill$\BBox{\rule{.3mm}{3mm}}$} \\

\noindent
{\bf Proof of Theorem \ref{thmForest2}.} The proof uses the same ideas as the proof of Theorem \ref{thmForest1}.
{\hfill$\BBox{\rule{.3mm}{3mm}}$} \\


\clearpage
\thispagestyle{empty}
\cleardoublepage
\verb"   "
\newpage

\chapter{Recognizing $R^*$-cyclic matrices}\label{ch:cyc}

Let $A$ be a connected matrix of size $n\times m$ with entries $0$, $1$, $2$, or $\frac{1}{2}$. Let $\alpha$  be the number of nonzero entries of $A$ and $R^*$ a row index subset of $A$ such that $R^*\subseteq s(A_{\bullet j})$ for at least one column index $j$. In this chapter, we describe a procedure called RCyclic and provide a proof of Theorems \ref{thmcyclicproCyc} and \ref{thmcyclic1} below. Before reading this chapter, the reader is referred to Chapter \ref{ch:multidiD}.

\begin{thm}\label{thmcyclicproCyc}
The matrix $A$ can be tested for having an $R^*$-cyclic representation by the procedure RCyclic.
The running time of this procedure is $O(n m \alpha)$.
\end{thm}

Let us introduce some notations and definitions. We note $S^*=\{j\, : \, s(A_{\bullet j}) \cap R^*\neq \emptyset \}$, $\rho=|R^*|$ and $R_j=s(A_{\bullet j}) \cap R^*$, for all $j\in S^*$; the set $R_j$ is called an \emph{interval}\index{interval}. Up to row permutations, we may assume $R^*=\{1,\ldots,\rho\}$. Let $D=(V,\Upsilon)$ be a digraph with respect to $R^*$ as computed in Chapter \ref{ch:multidiD}. We will see a procedure called Initialization which produces an induced subgraph of $D$, denoted as $D'$, having no directed cycle. Further, we will define and construct a set $V_c$ of particular bonsais, called central bonsais, and a matrix $O(R^*)$ called the \emph{$R^*$-open matrix}\index{open@$R^*$-open matrix} with respect to $D'$.
Let $$\Lambda_c=\{R_j\, : \,j\in f^*(E_\ell) \m{ and } E_\ell \in V_c\}.$$  Whenever $A$ is $R^*$-cyclic, it will be proved that the poset $(\Lambda_c, \subseteq)$ can be decomposed into two chains forming an $R^*$-double chain, a notion that will be defined later.
The set $V_c$ is called \emph{$R^*$-compatible}\index{compatible@$R^*$-compatible} if for any $\beta \in S^*$ such that $\sum_{E_\ell \in V_c} g_\beta(E_\ell)\geq 2$, we have the equality $R_\beta = R^*$. 

Let $D'$ be a maximal induced subgraph of $D$ having no directed cycle, and $O(R^*)$ the corresponding $R^*$-open matrix.
Let us give two simple necessary conditions for $A$ to be $R^*$-cyclic that directly follow from Corollary \ref{corBidirectedCircuit} and Lemma \ref{lemdefiWeight1}. For any $1\le j\le m$, if $s_{\frac{1}{2}}(A_{\bullet j})\neq \emptyset$, then $s_{\frac{1}{2}}(A_{\bullet j})=R^*$ and $s_2(A_{\bullet j})=\emptyset$; and if $s_2(A_{\bullet j})\neq \emptyset$, then $s_2(A_{\bullet j})\cap R^*=\emptyset$ and $R^*\subseteq s(A_{\bullet j})$. 
Under these two conditions, the following holds.

\begin{thm}\label{thmcyclic1}
The matrix $A$ is $R^*$-cyclic if and only if the digraph $D$ has a feasible spanning forest, the poset $(\Lambda_c, \subseteq)$ is an $R^*$-double chain,
the set $V_c$ is $R^*$-compatible and the $R^*$-open matrix $O(R^*)$ as well as each bonsai matrix $N_\ell$ with $E_\ell \in V_c\cup  \overline{Sink(D')}$ are network matrices.
\end{thm}

Suppose that $A$ has an $R^*$-cyclic representation $G(A)$. If $v_i$ is a node of $G(A)$, then we note $v_i^*$\index{vertex!$v_i^*$} the endnode of the basic path from $v_i$ to the basic cycle. 
Recall that, for all $1\le \ell \le b$, $v_\ell$ denotes the closest vertex of $B_\ell$ to the basic cycle. Moreover,
$G(A)$ induces the following spanning forest $T_{G(A)}$ of $D$: for all $1\le \ell,\ell'\le b$ ($\ell\neq \ell'$), $(E_\ell,E_{\ell'})\in T_{G(A)}$ if and only if $v_\ell$ is a node of $B_{\ell'}$ (distinct from $v_{\ell'}$).
For all $j\in S^*$, the interval $R_j$ is equal to the edge index set of a consistently oriented path $p_j$ which  is called an \emph{interval}\index{interval} in $G(A)$ and lies on the basic cycle.

Before embarking on the proof of Theorems \ref{thmcyclicproCyc} and \ref{thmcyclic1}, Section \ref{sec:infcyc} deals with some intuitive notions and graphical ideas on which these are based using an example. Then, in Section \ref{sec:cyc}, a formal proof of these theorems is given.

\section{An informal sketch of a recognition procedure}\label{sec:infcyc}

Let us consider the $R^*$-cyclic matrix $A$ given in Figure \ref{fig:cyclicA} where $R^*=\{1,2,3\}$. The goal is to construct an $R^*$-cyclic representation of $A$, for instance the one given in Figure \ref{fig:cyclicA}, without knowing that such a representation exists. We depict several steps of the recognition procedure RCyclic applied on $A$, and motivate some definitions in an informal way.

\vspace{1cm}

\begin{figure}[ht!]
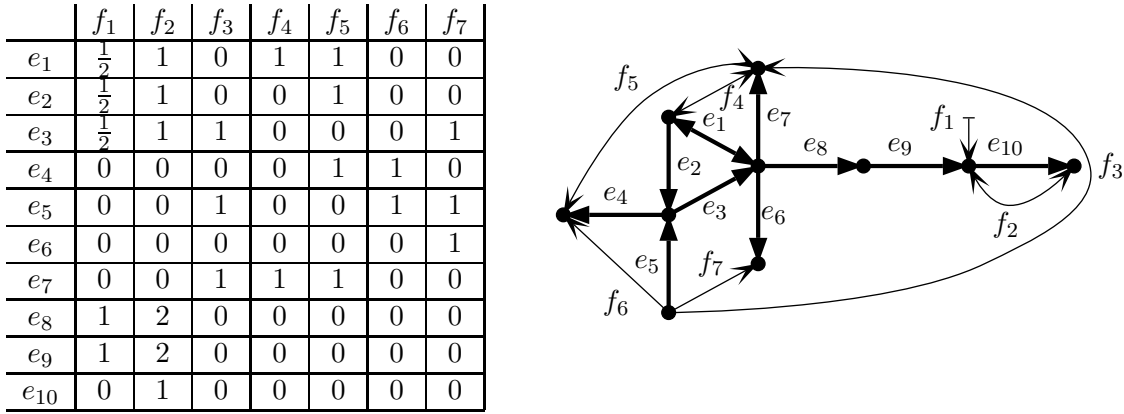


$
\begin{array}{cl}

\begin{tabular}{c|c|c|c|c|c|c|c|}
  & $f_1$ & $f_2$ & $f_3$ & $f_4$ & $f_5$ & $f_6$ & $f_7$  \\
  \hline
$e_1$  &$\frac{1}{2}$&1&0&1&1&0&0 \\
\hline
$e_2$  &$\frac{1}{2}$&1&0&0&1&0&0 \\
\hline
$e_3$  &$\frac{1}{2}$&1&1&0&0&0&1 \\
\hline
$e_4$  &0&0&0&0&1&1&0 \\
\hline
$e_5$  &0&0&1&0&0&1&1\\
\hline
$e_6$  &0&0&0&0&0&0&1 \\
\hline
$e_7$  &0&0&1&1&1&0&0 \\
\hline
$e_8$  &1&2&0&0&0&0&0 \\
\hline
$e_9$  &1&2&0&0&0&0&0 \\
\hline
$e_{10}$  &0&1&0&0&0&0&0 \\
\hline
\end{tabular}  &

\psset{xunit=1.4cm,yunit=1.3cm,linewidth=0.5pt,radius=0.1mm,arrowsize=7pt,
labelsep=1.5pt,fillcolor=black}

\pspicture(-1.5,1)(5,2)

\pscircle[fillstyle=solid](0,0){.1}
\pscircle[fillstyle=solid](-1,1){.1}
\pscircle[fillstyle=solid](0,1){.1}
\pscircle[fillstyle=solid](0,2){.1}
\pscircle[fillstyle=solid](0.85,1.5){.1}
\pscircle[fillstyle=solid](1.85,1.5){.1}
\pscircle[fillstyle=solid](.85,.5){.1}
\pscircle[fillstyle=solid](.85,2.5){.1}
\pscircle[fillstyle=solid](2.85,1.5){.1}
\pscircle[fillstyle=solid](3.85,1.5){.1}

\psline[linewidth=1.6pt,arrowinset=0]{<->}(0,2)(0.85,1.5)
\rput(0.44,1.95){$e_1$}

\psline[linewidth=1.6pt,arrowinset=0]{<-}(0,1)(0,2)
\rput(0.2,1.5){$e_2$}

\psline[linewidth=1.6pt,arrowinset=0]{->}(0,1)(0.85,1.5)
\rput(0.44,1.05){$e_3$}

\psline[linewidth=1.6pt,arrowinset=0]{<-}(-1,1)(0,1)
\rput(-0.5,1.2){$e_4$}

\psline[linewidth=1.6pt,arrowinset=0]{->}(0,0)(0,1)
\rput(-0.2,0.5){$e_5$}

\psline[linewidth=1.6pt,arrowinset=0]{->}(.85,1.5)(.85,.5)
\rput(1,1){$e_6$}

\psline[linewidth=1.6pt,arrowinset=0]{->}(.85,1.5)(.85,2.5)
\rput(1.05,2){$e_7$}

\psline[linewidth=1.6pt,arrowinset=0]{->}(0.85,1.5)(1.85,1.5)
\rput(1.4,1.7){$e_8$}

\psline[linewidth=1.6pt,arrowinset=0]{->}(1.85,1.5)(2.85,1.5)
\rput(2.2,1.7){$e_9$}

\psline[linewidth=1.6pt,arrowinset=0]{->}(2.85,1.5)(3.85,1.5)
\rput(3.2,1.7){$e_{10}$}


\psline[arrowinset=.5,arrowlength=1.5]{|->}(2.85,2)(2.85,1.5)
\rput(2.6,2){$f_{1}$}

\pscurve[arrowinset=.5,arrowlength=1.5]{<->}(2.85,1.5)(3.2,1.1)(3.85,1.5)
\rput(3.2,.9){$f_2$}

\pscurve[arrowinset=.5,arrowlength=1.5]{->}(0,0)(3,0.5)(4,1.5)(.85,2.5)
\rput(4.2,1.5){$f_3$}

\psline[arrowinset=.5,arrowlength=1.5]{<->}(0,2)(.85,2.5)
\rput(.6,2.2){$f_4$}

\pscurve[arrowinset=.5,arrowlength=1.5]{<->}(-1,1)(-0.2,2.2)(.7,2.55)(.85,2.5)
\rput(-.4,2.4){$f_5$}

\psline[arrowinset=.5,arrowlength=1.5]{->}(0,0)(-1,1)
\rput(-0.5,0.1){$f_6$}

\psline[arrowinset=.5,arrowlength=1.5]{->}(0,0)(.85,.5)
\rput(.4,.5){$f_7$}

\endpspicture 

\end{array}$

\vspace{.7cm}

\caption{A binet matrix $A$ and a $\{1,2,3\}$-cyclic representation $G(A)$ of $A$.}
\label{fig:cyclicA}

\end{figure}

Suppose that $\overline{R^*}$ has been partitioned into $E_1=\{4,5\}$, $E_2=\{6\}$, $E_3=\{7\}$, $E_4=\{8\}$, $E_5=\{9\}$ and  $E_6=\{10\}$. Let $D$ be a digraph as constructed in Chapter \ref{ch:multidiD} with respect to $R^*$  (see Figure \ref{fig:cyclicdigraphD}).
In general, let us recall that this construction makes only use of the matrix $A$, and $D$ is unique, provided that $A$ is $R^*$-cyclic as mentioned at page \pageref{mycounter}.
We observe that $D$ contains a directed cycle, namely $(E_4,(E_4,E_5),E_5,(E_5,E_4),E_4)$. At first, the procedure RCyclic searches for a maximal induced subgraph in $D$ having no directed cycle, for instance $D\verb"\"\{E_5\}$. So let $D'=D\verb"\"\{E_5\}$.

\vspace{1cm}

\begin{figure}[ht!]
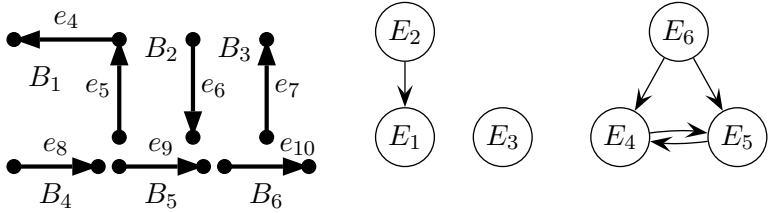


$
\begin{array}{cc}

\psset{xunit=1.4cm,yunit=1.3cm,linewidth=0.5pt,radius=0.1mm,arrowsize=7pt,
labelsep=1.5pt,fillcolor=black}

\pspicture(-1.5,0)(2,2.5)

\pscircle[fillstyle=solid](0,0){.1}
\pscircle[fillstyle=solid](-1,1){.1}
\pscircle[fillstyle=solid](0,1){.1}

\pscircle[fillstyle=solid](.7,0){.1}
\pscircle[fillstyle=solid](.7,1){.1}

\pscircle[fillstyle=solid](1.4,0){.1}
\pscircle[fillstyle=solid](1.4,1){.1}

\pscircle[fillstyle=solid](-1,-0.3){.1}
\pscircle[fillstyle=solid](-.2,-.3){.1}

\pscircle[fillstyle=solid](0,-0.3){.1}
\pscircle[fillstyle=solid](.8,-.3){.1}

\pscircle[fillstyle=solid](1,-0.3){.1}
\pscircle[fillstyle=solid](1.8,-.3){.1}

\psline[linewidth=1.6pt,arrowinset=0]{<-}(-1,1)(0,1)
\rput(-0.5,1.2){$e_4$}

\psline[linewidth=1.6pt,arrowinset=0]{->}(0,0)(0,1)
\rput(-0.2,0.5){$e_5$}

\rput(-.7,.6){$B_1$}

\psline[linewidth=1.6pt,arrowinset=0]{->}(.7,1)(.7,0)
\rput(.9,.5){$e_6$}

\rput(.4,.9){$B_2$}

\psline[linewidth=1.6pt,arrowinset=0]{<-}(1.4,1)(1.4,0)
\rput(1.6,.5){$e_7$}

\rput(1.1,.9){$B_3$}

\psline[linewidth=1.6pt,arrowinset=0]{->}(-1,-.3)(-.2,-.3)
\rput(-.6,-.1){$e_8$}

\rput(-.6,-.6){$B_4$}

\psline[linewidth=1.6pt,arrowinset=0]{->}(0,-.3)(.8,-.3)
\rput(.4,-.1){$e_9$}

\rput(.4,-.6){$B_5$}

\psline[linewidth=1.6pt,arrowinset=0]{->}(1,-.3)(1.8,-.3)
\rput(1.7,-.1){$e_{10}$}

\rput(1.4,-.6){$B_6$}

\endpspicture &

\psset{xunit=1.3cm,arrows=->,yunit=1.4cm,linewidth=0.5pt,radius=0.1mm,arrowsize=6pt,
labelsep=1pt,fillcolor=black}

\pspicture(-0.3,0)(2,1.5)


\cnodeput(0.2,0){1}{$E_1$}
\cnodeput(0.2,1){2}{$E_2$}
\cnodeput(3,1){6}{$E_6$}
\cnodeput(2.4,0){4}{$E_4$}
\cnodeput(3.6,0){5}{$E_5$}
\cnodeput(1.2,0){3}{$E_3$}

\ncline{2}{1}
\ncline{6}{4}
\ncline{6}{5}
\ncarc{5}{4}
\ncarc{4}{5}
      
\endpspicture 
\end{array}$

\vspace{.7cm}

\caption{the bonsais in $G(A)$ and the digraph $D$ (with respect to $R^*=\{1,2,3\}$), where $A$ is given in Figure \ref{fig:cyclicA}.}
\label{fig:cyclicdigraphD}

\vspace{.3cm}
\end{figure}

Then, the procedure RCyclic constructs an $R^*$-cyclic representation of the matrix $$A'=A_{\cup_{E_\ell \in Sink(D')} E_\ell \cup R^* \bullet},$$ or determines that $A$ is not $R^*$-cyclic. How does it proceed? Let us look at the submatrix $A_{R^*\bullet}^{\frac{1}{2}\rightarrow 1}$ in Figure \ref{fig:cyclicA}. We observe that it is an interval matrix. It will be proved that this follows from the nonnegativity of $A$ and the fact that this matrix is $R^*$-cyclic. Further, one can extend the matrix $A_{R^*\bullet}^{\frac{1}{2}\rightarrow 1}$ to a larger one, say $A_{R\bullet}^{\frac{1}{2}\rightarrow 1}$, so that $R^*\subseteq R$ and $A_{R\bullet}^{\frac{1}{2}\rightarrow 1}$ is a network matrix. Following this idea, one can try to locate some special row index sets that have to be disjoint from $R$, otherwise $A_{R\bullet}^{\frac{1}{2}\rightarrow 1}$ might be a non-network matrix. This motivates a notion of "central" bonsai. 

Let $E_\ell$ be a bonsai. Provided that $A$ has an $R^*$-cyclic representation $G(A)$, it will be proved that if $E_\ell$ is central, then $v_\ell$ corresponds to a central node in $G(A)$, and so all $B_\ell$-paths in $G(A)$ are leaving $v_\ell$. 

Two necessary conditions for $E_\ell$ to be central are that $J_\ell^2= \emptyset$ (see Lemmas \ref{lembonsaicel} and \ref{lembonsainet2}) and $E_\ell$ is a sink vertex of $D'$. Let us consider the bonsais $E_1$ and $E_3$ in $Sink(D')$. We have that $f^*(E_1)=\{3,5,7\}$, $f^*(E_3)=\{3,4,5\}$, $R_3=\{3\}$, $R_4=\{1\}$, $R_5=\{1,2\}$ and  $R_7=\{3\}$. We observe that none of the posets $(\{R_j \, : \, j \in f^*(E_1) \}, \subseteq )$ and $(\{R_j \, : \, j \in f^*(E_3) \}, \subseteq )$ is a chain. On the other hand, the $E_1$-paths (resp., $E_3$-paths) generated by the columns of $A$ are $\{4\}$ and $\{5 \}$ (resp., $\{7\}$), and $\{4 \} \nsim_{E_1} \{ 5\}$ (see Lemma \ref{lemdigraphutile}). Thus
$J_1^2\neq \emptyset$ and $J_3^2=\emptyset$ (see Lemmas \ref{lembonsaicel} and \ref{lembonsainet2}). Hence $E_1$ is not central, and we will see that
$E_3$ is central because $E_3 \in Sink(D')$, $J_3^2=\emptyset$ and $(\{ R_j\, ,\, j\in f^*(E_3) \}, \subseteq )$ is not a chain. Since $J_1^2\neq \emptyset$ and $(\{ R_j\, ,\, j\in f^*(E_1) \}, \subseteq )$ is not a chain, provided that $A$ is $R^*$-cyclic, it can be proved that $v_1$ is not equal to the central node in any $R^*$-cyclic representation of $A$.

Moreover, besides the fact that $E_4 \in Sink(D')$ and $J_4^2 = \emptyset$, the bonsai $E_4$ is central for two reasons. The first one  is that there exists some $j\in f^*(E_4)$, namely $j=1$, such that $R_j=R^*$.  The second is that there exists some $j\in f^*(E_4)$, namely $j=2$, such that $g_j(E_4)=2$. The bonsai $E_6$ is not central because it is not a sink vertex of $D'$. Actually, $E_3$ and $E_4$ are the only central bonsais in $D'$. Let $$R= \cup_{E_\ell \in Sink(D')\verb"\"V_c} E_\ell \cup R^*=E_1 \cup R^*.$$ Is the matrix $A_{R \bullet}^{\frac{1}{2}\rightarrow 1}$ a network matrix? 
In general, one can show that it is whenever $A$ is $R^*$-cyclic. Given network representations of $A_{R \bullet}^{\frac{1}{2}}$, $N_3$ and $N_4$, it turns out sometimes that it is not sufficient to construct an $R^*$-cyclic representation of $A'= A_{\cup_{E_\ell \in Sink(D')}E_\ell \cup R^*\bullet }$.

\begin{figure}[h!]
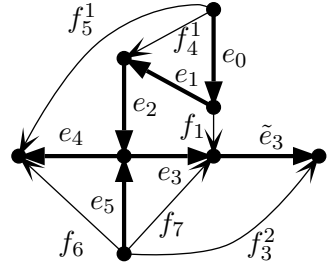


$
\begin{array}{cl}

\begin{tabular}{c|c|c|c|c|c|c|c|c|c|c|c|c|}
  & $f_1$ & $f_2$ & $f_6$ & $f_7$ &
 $f_4^1$ & $f_5^1$ & $f_3^2$ & &  \\
  \hline
$e_1$  &1&1&0&0& 1&1&0&1&1    \\
\hline
$e_2$  &1&1&0&0&0&1& 0&1&1\\
\hline
$e_3$  &1&1&0&1&0&0& 1&1&1\\
\hline
$e_4$  &0&0&1&0&0&1& 0&0&0\\
\hline
$e_5$  &0&0&1&1&0&0& 1&0&0\\
\hline
$e_0$  &0&0&0&0&1&1& 0&0&1\\
\hline
$\tilde e_3$  &0&0&0&0&0&0& 1&0&1\\

\hline
\end{tabular}  &

\psset{xunit=1.4cm,yunit=1.3cm,linewidth=0.5pt,radius=0.1mm,arrowsize=7pt,
labelsep=1.5pt,fillcolor=black}

\pspicture(-1.5,1)(5,2.5)

\pscircle[fillstyle=solid](0,0){.1}
\pscircle[fillstyle=solid](-1,1){.1}
\pscircle[fillstyle=solid](0,1){.1}
\pscircle[fillstyle=solid](0,2){.1}
\pscircle[fillstyle=solid](0.85,1.5){.1}
\pscircle[fillstyle=solid](1.85,1){.1}
\pscircle[fillstyle=solid](.85,1){.1}
\pscircle[fillstyle=solid](.85,2.5){.1}

\psline[linewidth=1.6pt,arrowinset=0]{<-}(0,2)(0.85,1.5)
\rput(0.6,1.8){$e_1$}

\psline[linewidth=1.6pt,arrowinset=0]{<-}(0,1)(0,2)
\rput(0.2,1.5){$e_2$}

\psline[linewidth=1.6pt,arrowinset=0]{->}(0,1)(.85,1)
\rput(0.44,.8){$e_3$}

\psline[linewidth=1.6pt,arrowinset=0]{<-}(-1,1)(0,1)
\rput(-0.5,1.2){$e_4$}

\psline[linewidth=1.6pt,arrowinset=0]{->}(0,0)(0,1)
\rput(-0.2,0.5){$e_5$}


\psline[linewidth=1.6pt,arrowinset=0]{<-}(.85,1.5)(.85,2.5)
\rput(1.05,2){$e_0$}

\psline[linewidth=1.6pt,arrowinset=0]{->}(0.85,1)(1.85,1)
\rput(1.4,1.2){$\tilde e_3$}



\psline[arrowinset=.5,arrowlength=1.5]{->}(.85,1.5)(.85,1)
\rput(.65,1.3){$f_{1}$}


\pscurve[arrowinset=.5,arrowlength=1.5]{->}(0,0)(1,0.1)(1.85,1)
\rput(1.3,.1){$f_3^2$}

\psline[arrowinset=.5,arrowlength=1.5]{<-}(0,2)(.85,2.5)
\rput(.6,2.2){$f_4^1$}

\pscurve[arrowinset=.5,arrowlength=1.5]{<-}(-1,1)(-0.2,2.2)(.7,2.55)(.85,2.5)
\rput(-.4,2.4){$f_5^1$}

\psline[arrowinset=.5,arrowlength=1.5]{->}(0,0)(-1,1)
\rput(-0.5,0.1){$f_6$}

\psline[arrowinset=.5,arrowlength=1.5]{->}(0,0)(.85,1)
\rput(.45,.3){$f_7$}

\endpspicture

\end{array}$

\caption{the network matrix $O(R^*)$ and a network representation $G(O(R^*))$ of $O(R^*)$ (without some nonbasic edges), where the basic tree is denoted by $T$.}
\label{fig:cyclicM*}

\end{figure}

Let $V_c$ be the set of central bonsais and 
$\Lambda_c=\{R_j \, :\, j\in f^*(E_\ell)\m{ and } E_\ell \in V_c \}$. For our example, $\Lambda_c=\{R^*, \{1\},\{1,2\},\{3 \} \}$. Provided that $A$ has an $R^*$-cyclic representation $G(A)$ as in Figure \ref{fig:cyclicA},
we notice that the intervals in $G(A)$ with edge index set in $\Lambda_c$ are all incident with the central node, so the poset $(\Lambda_c\verb"\"\{R^*\}, \subseteq)$  can be split up  into two chains, namely $\Lambda_c^1=(\{\{ 3\}\},\subseteq)$ and $\Lambda_c^2=(\{\{1\} ,\{1,2\} \},\subseteq)$. The procedure RCyclic produces these two chains or stops, if they do not exist. Then, it constructs a network representation $G(A_{R \bullet})$ of $A_{R \bullet}$ such that for $i=1$ and $2$, every
interval in $\Lambda_c^i$ corresponds to the edge index set of a path, and all these paths have a common endnode, provided that such a representation exists. On this purpose the matrix $O(R^*)$ as given in Figure \ref{fig:cyclicM*}
is constructed. Let us remark that $O(R^*)$ has two more rows than 
$A_{R \bullet}$.

Given a basic network representation $G(O(R^*))$ of $O(R^*)$ and, for every $E_\ell \in V_c$, a $v_\ell$-rooted network representation $B_\ell$ of $N_\ell$, under certain conditions, the procedure RCyclic constructs a basic $R^*$-cyclic representation $G(A')$ of the matrix $A'=A_{\cup_{E_\ell\in Sink(D')}E_\ell\cup R^* \bullet}$ as illustrated in Figure \ref{fig:cyclicGA'}. This will be accomplished by a subroutine called GASink.

\begin{figure}[h!]
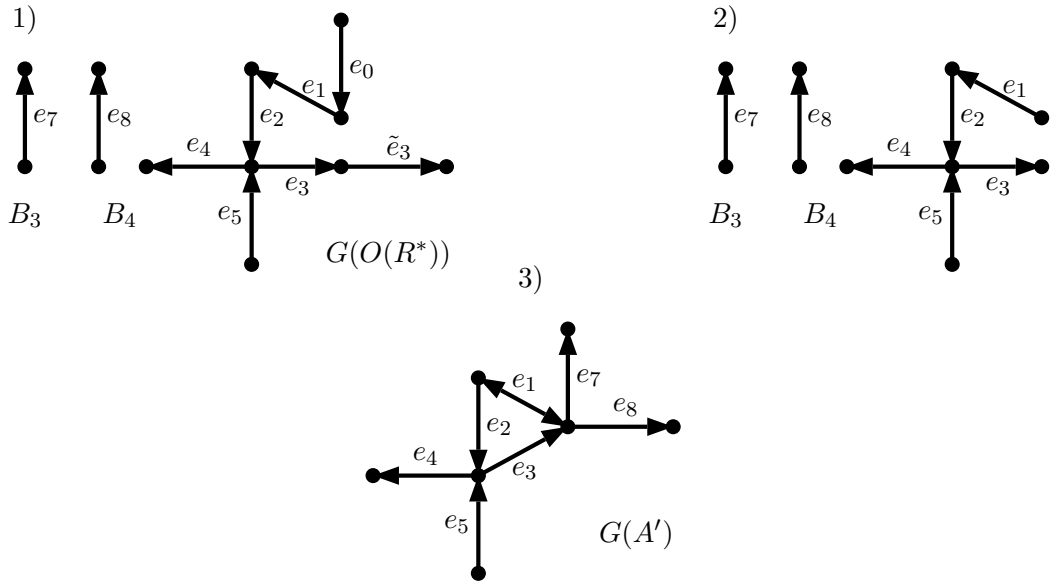

\vspace{.3cm}
\begin{center}
$
\begin{array}{cc}

\begin{array}{cc}

\psset{xunit=1.4cm,yunit=1.3cm,linewidth=0.5pt,radius=0.1mm,arrowsize=7pt,
labelsep=1.5pt,fillcolor=black}

\pspicture(-.5,0)(.7,2)

\pscircle[fillstyle=solid](0,1){.1}
\pscircle[fillstyle=solid](0,2){.1}

\pscircle[fillstyle=solid](.7,1){.1}
\pscircle[fillstyle=solid](0.7,2){.1}

\rput(0,2.5){$1)$}

\psline[linewidth=1.6pt,arrowinset=0]{->}(0,1)(0,2)
\rput(0.2,1.5){$e_7$}

\rput(0,.5){$B_3$}

\psline[linewidth=1.6pt,arrowinset=0]{->}(.7,1)(0.7,2)
\rput(0.9,1.5){$e_8$}

\rput(.9,.5){$B_4$}

\endpspicture &

\psset{xunit=1.4cm,yunit=1.3cm,linewidth=0.5pt,radius=0.1mm,arrowsize=7pt,
labelsep=1.5pt,fillcolor=black}

\pspicture(-1.2,0)(4,2.5)

\pscircle[fillstyle=solid](0,0){.1}
\pscircle[fillstyle=solid](-1,1){.1}
\pscircle[fillstyle=solid](0,1){.1}
\pscircle[fillstyle=solid](0,2){.1}
\pscircle[fillstyle=solid](0.85,1.5){.1}
\pscircle[fillstyle=solid](1.85,1){.1}
\pscircle[fillstyle=solid](.85,1){.1}
\pscircle[fillstyle=solid](.85,2.5){.1}

\rput(1.3,0.1){$G(O(R^*))$}

\psline[linewidth=1.6pt,arrowinset=0]{<-}(0,2)(0.85,1.5)
\rput(0.6,1.8){$e_1$}

\psline[linewidth=1.6pt,arrowinset=0]{<-}(0,1)(0,2)
\rput(0.2,1.5){$e_2$}

\psline[linewidth=1.6pt,arrowinset=0]{->}(0,1)(.85,1)
\rput(0.44,.8){$e_3$}

\psline[linewidth=1.6pt,arrowinset=0]{<-}(-1,1)(0,1)
\rput(-0.5,1.2){$e_4$}

\psline[linewidth=1.6pt,arrowinset=0]{->}(0,0)(0,1)
\rput(-0.2,0.5){$e_5$}


\psline[linewidth=1.6pt,arrowinset=0]{<-}(.85,1.5)(.85,2.5)
\rput(1.05,2){$e_0$}

\psline[linewidth=1.6pt,arrowinset=0]{->}(0.85,1)(1.85,1)
\rput(1.4,1.2){$\tilde e_3$}

\endpspicture
\end{array}    &

\begin{array}{cc}

\psset{xunit=1.4cm,yunit=1.3cm,linewidth=0.5pt,radius=0.1mm,arrowsize=7pt,
labelsep=1.5pt,fillcolor=black}

\pspicture(0,0)(.7,2)

\pscircle[fillstyle=solid](0,1){.1}
\pscircle[fillstyle=solid](0,2){.1}

\pscircle[fillstyle=solid](.7,1){.1}
\pscircle[fillstyle=solid](0.7,2){.1}

\rput(0,2.5){$2)$}

\psline[linewidth=1.6pt,arrowinset=0]{->}(0,1)(0,2)
\rput(0.2,1.5){$e_7$}

\rput(0,.5){$B_3$}

\psline[linewidth=1.6pt,arrowinset=0]{->}(.7,1)(0.7,2)
\rput(0.9,1.5){$e_8$}

\rput(.9,.5){$B_4$}

\endpspicture &

\psset{xunit=1.4cm,yunit=1.3cm,linewidth=0.5pt,radius=0.1mm,arrowsize=7pt,
labelsep=1.5pt,fillcolor=black}

\pspicture(-1.2,0)(3,2.5)

\pscircle[fillstyle=solid](0,0){.1}
\pscircle[fillstyle=solid](-1,1){.1}
\pscircle[fillstyle=solid](0,1){.1}
\pscircle[fillstyle=solid](0,2){.1}
\pscircle[fillstyle=solid](0.85,1.5){.1}
\pscircle[fillstyle=solid](.85,1){.1}

\psline[linewidth=1.6pt,arrowinset=0]{<-}(0,2)(0.85,1.5)
\rput(0.6,1.8){$e_1$}

\psline[linewidth=1.6pt,arrowinset=0]{<-}(0,1)(0,2)
\rput(0.2,1.5){$e_2$}

\psline[linewidth=1.6pt,arrowinset=0]{->}(0,1)(.85,1)
\rput(0.44,.8){$e_3$}

\psline[linewidth=1.6pt,arrowinset=0]{<-}(-1,1)(0,1)
\rput(-0.5,1.2){$e_4$}

\psline[linewidth=1.6pt,arrowinset=0]{->}(0,0)(0,1)
\rput(-0.2,0.5){$e_5$}

\endpspicture
\end{array}  
\end{array}
$
\end{center}

\begin{center}
 
\psset{xunit=1.4cm,yunit=1.3cm,linewidth=0.5pt,radius=0.1mm,arrowsize=7pt,
labelsep=1.5pt,fillcolor=black}

\pspicture(-1,0)(2,2.7)

\pscircle[fillstyle=solid](0,0){.1}
\pscircle[fillstyle=solid](-1,1){.1}
\pscircle[fillstyle=solid](0,1){.1}
\pscircle[fillstyle=solid](0,2){.1}
\pscircle[fillstyle=solid](0.85,1.5){.1}
\pscircle[fillstyle=solid](1.85,1.5){.1}
\pscircle[fillstyle=solid](.85,2.5){.1}

\rput(.5,3){$3)$}

\rput(1.5,.4){$G(A')$}

\psline[linewidth=1.6pt,arrowinset=0]{<->}(0,2)(0.85,1.5)
\rput(0.44,1.95){$e_1$}

\psline[linewidth=1.6pt,arrowinset=0]{<-}(0,1)(0,2)
\rput(0.2,1.5){$e_2$}

\psline[linewidth=1.6pt,arrowinset=0]{->}(0,1)(0.85,1.5)
\rput(0.44,1.05){$e_3$}

\psline[linewidth=1.6pt,arrowinset=0]{<-}(-1,1)(0,1)
\rput(-0.5,1.2){$e_4$}

\psline[linewidth=1.6pt,arrowinset=0]{->}(0,0)(0,1)
\rput(-0.2,0.5){$e_5$}


\psline[linewidth=1.6pt,arrowinset=0]{->}(.85,1.5)(.85,2.5)
\rput(1.05,2){$e_7$}

\psline[linewidth=1.6pt,arrowinset=0]{->}(0.85,1.5)(1.85,1.5)
\rput(1.4,1.7){$e_8$}

\endpspicture 

\end{center}

\caption{Two steps in the construction of a basic
$\{1,2,3\}$-cyclic representation $G(A')$ of the matrix $A'$, where $A$ is given in Figure \ref{fig:cyclicA}. (See pictures from 1 to 3.)}
\label{fig:cyclicGA'}
\end{figure}

Finally, given a basic $R^*$-cyclic representation $G(A')$ of $A'$, a $v_\ell$-rooted network representation of $N_\ell$ for $l=2$, $5$ and $6$, as well as a feasible spanning forest $T_\Theta$ of $D$ such that $Sink(D') \subseteq Sink(T_\Theta)$, the 
procedure RCyclic provides a basic $R^*$-cyclic representation of $A$ as illustrated in Figure \ref{fig:cyclicrecA}.\\

\begin{figure}[h!]
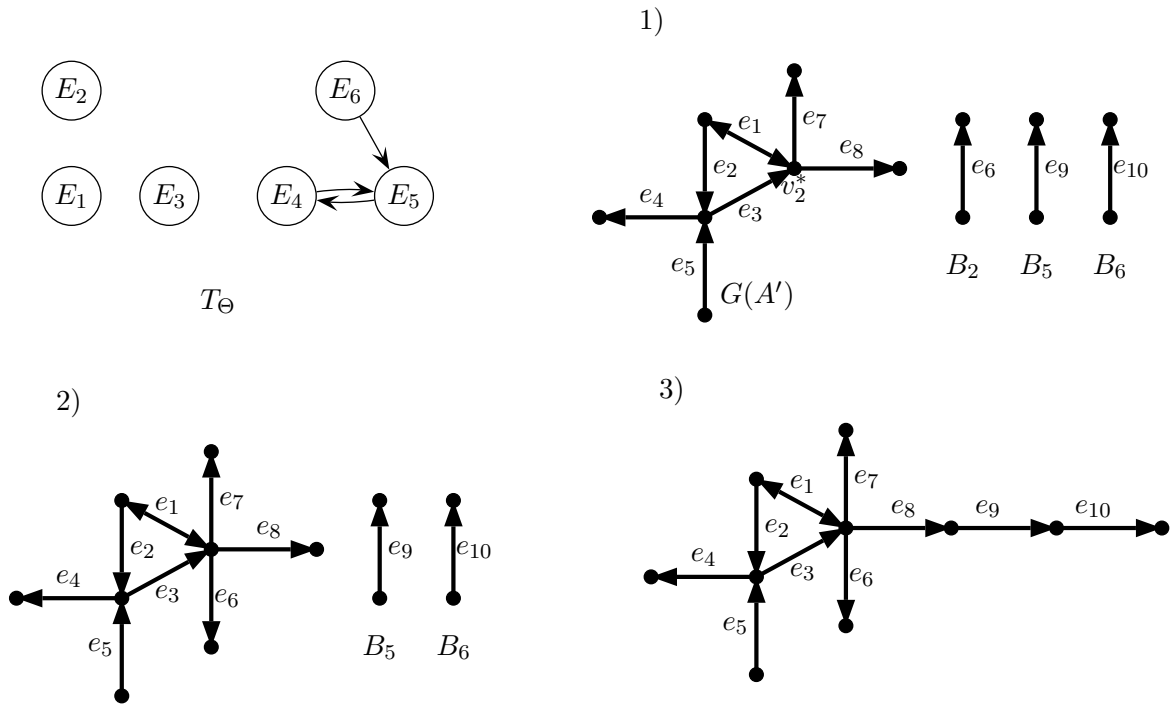

\vspace{.7cm}
\begin{center}
$
\begin{array}{cc}

\psset{xunit=1.3cm,arrows=->,yunit=1.4cm,linewidth=0.5pt,radius=0.1mm,arrowsize=6pt,
labelsep=1pt,fillcolor=black}

\pspicture(-0.5,0)(2.5,1.5)


\cnodeput(0.2,0){1}{$E_1$}
\cnodeput(0.2,1){2}{$E_2$}
\cnodeput(3,1){6}{$E_6$}
\cnodeput(2.4,0){4}{$E_4$}
\cnodeput(3.6,0){5}{$E_5$}
\cnodeput(1.2,0){3}{$E_3$}

\rput(1.7,-1){$T_\Theta$}
\ncline{6}{5}
\ncarc{5}{4}
\ncarc{4}{5}
      
\endpspicture   &

\begin{array}{cc}

\psset{xunit=1.4cm,yunit=1.3cm,linewidth=0.5pt,radius=0.1mm,arrowsize=7pt,
labelsep=1.5pt,fillcolor=black}

\pspicture(-3.5,0)(2,2.7)

\pscircle[fillstyle=solid](0,0){.1}
\pscircle[fillstyle=solid](-1,1){.1}
\pscircle[fillstyle=solid](0,1){.1}
\pscircle[fillstyle=solid](0,2){.1}
\pscircle[fillstyle=solid](0.85,1.5){.1}
\pscircle[fillstyle=solid](1.85,1.5){.1}
\pscircle[fillstyle=solid](.85,2.5){.1}

\rput(-.5,3){$1)$}

\rput(.85,1.3){$v_2^*$}

\rput(.5,.2){$G(A')$}

\psline[linewidth=1.6pt,arrowinset=0]{<->}(0,2)(0.85,1.5)
\rput(0.44,1.95){$e_1$}

\psline[linewidth=1.6pt,arrowinset=0]{<-}(0,1)(0,2)
\rput(0.2,1.5){$e_2$}

\psline[linewidth=1.6pt,arrowinset=0]{->}(0,1)(0.85,1.5)
\rput(0.44,1.05){$e_3$}

\psline[linewidth=1.6pt,arrowinset=0]{<-}(-1,1)(0,1)
\rput(-0.5,1.2){$e_4$}

\psline[linewidth=1.6pt,arrowinset=0]{->}(0,0)(0,1)
\rput(-0.2,0.5){$e_5$}


\psline[linewidth=1.6pt,arrowinset=0]{->}(.85,1.5)(.85,2.5)
\rput(1.05,2){$e_7$}

\psline[linewidth=1.6pt,arrowinset=0]{->}(0.85,1.5)(1.85,1.5)
\rput(1.4,1.7){$e_8$}

\endpspicture &

\psset{xunit=1.4cm,yunit=1.3cm,linewidth=0.5pt,radius=0.1mm,arrowsize=7pt,
labelsep=1.5pt,fillcolor=black}

\pspicture(-.2,0)(2.3,2)

\pscircle[fillstyle=solid](0,1){.1}
\pscircle[fillstyle=solid](0,2){.1}

\pscircle[fillstyle=solid](.7,1){.1}
\pscircle[fillstyle=solid](0.7,2){.1}

\pscircle[fillstyle=solid](1.4,1){.1}
\pscircle[fillstyle=solid](1.4,2){.1}


\psline[linewidth=1.6pt,arrowinset=0]{->}(0,1)(0,2)
\rput(0.2,1.5){$e_6$}

\rput(0,.5){$B_2$}

\psline[linewidth=1.6pt,arrowinset=0]{->}(0.7,1)(0.7,2)
\rput(0.9,1.5){$ e_9$}

\rput(.7,.5){$B_5$}

\psline[linewidth=1.6pt,arrowinset=0]{->}(1.4,1)(1.4,2)
\rput(1.6,1.5){$ e_{10}$}

\rput(1.4,.5){$B_6$}

\endpspicture 

\end{array}
\end{array}
$
\end{center}

\vspace{.5cm}

\begin{center}
$
\begin{array}{cc}
 
\begin{array}{cc}

\psset{xunit=1.4cm,yunit=1.3cm,linewidth=0.5pt,radius=0.1mm,arrowsize=7pt,
labelsep=1.5pt,fillcolor=black}

\pspicture(-1,0)(2,2.7)

\pscircle[fillstyle=solid](0,0){.1}
\pscircle[fillstyle=solid](-1,1){.1}
\pscircle[fillstyle=solid](0,1){.1}
\pscircle[fillstyle=solid](0,2){.1}
\pscircle[fillstyle=solid](0.85,1.5){.1}
\pscircle[fillstyle=solid](1.85,1.5){.1}
\pscircle[fillstyle=solid](.85,.5){.1}
\pscircle[fillstyle=solid](.85,2.5){.1}

\rput(-.5,3){$2)$}


\psline[linewidth=1.6pt,arrowinset=0]{<->}(0,2)(0.85,1.5)
\rput(0.44,1.95){$e_1$}

\psline[linewidth=1.6pt,arrowinset=0]{<-}(0,1)(0,2)
\rput(0.2,1.5){$e_2$}

\psline[linewidth=1.6pt,arrowinset=0]{->}(0,1)(0.85,1.5)
\rput(0.44,1.05){$e_3$}

\psline[linewidth=1.6pt,arrowinset=0]{<-}(-1,1)(0,1)
\rput(-0.5,1.2){$e_4$}

\psline[linewidth=1.6pt,arrowinset=0]{->}(0,0)(0,1)
\rput(-0.2,0.5){$e_5$}

\psline[linewidth=1.6pt,arrowinset=0]{->}(.85,1.5)(.85,.5)
\rput(1,1){$e_6$}

\psline[linewidth=1.6pt,arrowinset=0]{->}(.85,1.5)(.85,2.5)
\rput(1.05,2){$e_7$}

\psline[linewidth=1.6pt,arrowinset=0]{->}(0.85,1.5)(1.85,1.5)
\rput(1.4,1.7){$e_8$}

\endpspicture &

\psset{xunit=1.4cm,yunit=1.3cm,linewidth=0.5pt,radius=0.1mm,arrowsize=7pt,
labelsep=1.5pt,fillcolor=black}

\pspicture(-.2,0)(1.7,2)

\pscircle[fillstyle=solid](0,1){.1}
\pscircle[fillstyle=solid](0,2){.1}

\pscircle[fillstyle=solid](.7,1){.1}
\pscircle[fillstyle=solid](0.7,2){.1}


\psline[linewidth=1.6pt,arrowinset=0]{->}(0,1)(0,2)
\rput(0.2,1.5){$e_9$}

\rput(0,.5){$B_5$}

\psline[linewidth=1.6pt,arrowinset=0]{->}(0.7,1)(0.7,2)
\rput(0.9,1.5){$ e_{10}$}

\rput(.7,.5){$B_6$}

\endpspicture 

\end{array} &

\psset{xunit=1.4cm,yunit=1.3cm,linewidth=0.5pt,radius=0.1mm,arrowsize=7pt,
labelsep=1.5pt,fillcolor=black}

\pspicture(-1.5,1)(5,2)

\pscircle[fillstyle=solid](0,0){.1}
\pscircle[fillstyle=solid](-1,1){.1}
\pscircle[fillstyle=solid](0,1){.1}
\pscircle[fillstyle=solid](0,2){.1}
\pscircle[fillstyle=solid](0.85,1.5){.1}
\pscircle[fillstyle=solid](1.85,1.5){.1}
\pscircle[fillstyle=solid](.85,.5){.1}
\pscircle[fillstyle=solid](.85,2.5){.1}
\pscircle[fillstyle=solid](2.85,1.5){.1}
\pscircle[fillstyle=solid](3.85,1.5){.1}

\rput(-.8,2.9){$3)$}

\psline[linewidth=1.6pt,arrowinset=0]{<->}(0,2)(0.85,1.5)
\rput(0.44,1.95){$e_1$}

\psline[linewidth=1.6pt,arrowinset=0]{<-}(0,1)(0,2)
\rput(0.2,1.5){$e_2$}

\psline[linewidth=1.6pt,arrowinset=0]{->}(0,1)(0.85,1.5)
\rput(0.44,1.05){$e_3$}

\psline[linewidth=1.6pt,arrowinset=0]{<-}(-1,1)(0,1)
\rput(-0.5,1.2){$e_4$}

\psline[linewidth=1.6pt,arrowinset=0]{->}(0,0)(0,1)
\rput(-0.2,0.5){$e_5$}

\psline[linewidth=1.6pt,arrowinset=0]{->}(.85,1.5)(.85,.5)
\rput(1,1){$e_6$}

\psline[linewidth=1.6pt,arrowinset=0]{->}(.85,1.5)(.85,2.5)
\rput(1.05,2){$e_7$}

\psline[linewidth=1.6pt,arrowinset=0]{->}(0.85,1.5)(1.85,1.5)
\rput(1.4,1.7){$e_8$}

\psline[linewidth=1.6pt,arrowinset=0]{->}(1.85,1.5)(2.85,1.5)
\rput(2.2,1.7){$e_9$}

\psline[linewidth=1.6pt,arrowinset=0]{->}(2.85,1.5)(3.85,1.5)
\rput(3.2,1.7){$e_{10}$}

\endpspicture 

\end{array}
$

\end{center}

\caption{A feasible spanning forest $T_\Theta$ of $D$ and two steps in the reconstruction of a basic $\{1,2,3\}$-cyclic representation $G(A)$ of $A$, where $A$ is given in Figure \ref{fig:cyclicA}.}
\label{fig:cyclicrecA}
\end{figure}

\section{The procedure RCyclic}\label{sec:cyc}

In this section, we deal with the general framework of the recognition problem, and provide a proof of Theorems \ref{thmcyclicproCyc} and \ref{thmcyclic1}. We describe here the initialization step of the procedure RCyclic. This enables us to focus on an induced subgraph of $D$ without any directed cycle.

\begin{tabbing}
\textbf{Procedure\,\,Initialization($A$,$R^*$)}\\

\textbf{Input:} A matrix $A$ and a row index subset $R^*$ of $A$.\\
\textbf{Output: }\= An induced subgraph $D' \subseteq D$ without any directed cycle,\\
\> or determines that $A$ is not $R^*$-cyclic.\\ 
1)\verb"  "\= {\bf for }\= every column index $j$, {\bf do} \\ \> \> check that if 
 $s_{\frac{1}{2}}(A_{\bullet j})\neq \emptyset$, then $s_{\frac{1}{2}}(A_{\bullet j})=R^*$ and $s_2(A_{\bullet j})=\emptyset$,\\
\> \> and if $s_2(A_{\bullet j})\neq \emptyset$, then $s_2(A_{\bullet j})\cap R^*=\emptyset$ and $R^*\subseteq s(A_{\bullet j})$;\\
\> \> otherwise STOP: output that $A$ is not $R^*$-cyclic;\\
\> {\bf endfor }\\
2) \> let $D'=D$;\\
\> {\bf while } \= $D'$ has a directed cycle $C$, ${\bf  do}$\\
\> \> remove from $D'$ all vertices of $C$ except one;\\
\> {\bf endwhile }\\
\> output $D'$;\\
\end{tabbing}

Suppose that the procedure Initialization has output an induced subgraph $D'$ of $D$. Let $S_0^*=\{j\,:\, R_j=R^*\}$.
A spanning forest $T_\Theta$ of $D$ is called \emph{$D'$-clean}\index{clean@$D'$-clean!forest} if $Sink(D')\subseteq Sink(T_\Theta)$.
Provided that $A$ is $R^*$-cyclic, a \emph{$D'$-clean}\index{clean@$D'$-clean!representation} $R^*$-cyclic representation $G(A)$ of $A$ verifies the inclusion $Sink(D') \subseteq Sink(T_{G(A)})$. 
A bonsai $E_\ell\in D'$ is designated as a \emph{central}\index{central!bonsai} bonsai if $J_\ell^2=\emptyset$ and $E_\ell$ is a sink vertex of $D'$ satisfying at least one of the following conditions:

\begin{itemize}

\item[1.] The vertex $E_\ell$ is a $\beta$-fork for some $\beta\in f^*(E_\ell)$.

\item[2.] The global connecter set of $E_\ell$ intersects $S_0^*$: $f^*(E_\ell) \cap S_0^* \neq \emptyset$.

\item[3.] The poset $(\{R_j \, : \, j\in f^*(E_\ell) \},\subseteq )$ is not a chain.

\end{itemize}

Let $V_c$\index{set!$V_c$} be the set of central bonsais. The set $V_c$ is called \emph{$R^*$-compatible}\index{compatible@$R^*$-compatible} if for any $\beta \in S^*$ such that $\sum_{E_\ell \in V_c} g_\beta(E_\ell)\geq 2$, we have the equality $R_\beta = R^*$. 

Throughout this section, whenever $A$ has an $R^*$-cyclic representation $G(A)$, since $R^*\subseteq s(A_{\bullet j})$ for at least one column index $j$ and $A$ is nonnegative,
we may assume that $w_1,\ldots,w_\rho$ are the vertices of the basic cycle and
$e_1=[w_1,w_\rho]$, $e_{i+1}=]w_i,w_{i+1}] \in G(A)$ for $i=1,\ldots, \rho -1$. (If $\rho=1$, then $e_1=[w_1,w_1] \in G(A)$.) The following statements are good 
spices in the cooking of Theorem \ref{thmcyclic1}.

\begin{prop}\label{propcyclicgood}
If $A$ is $R^*$-cyclic, then there exists a $D'$-clean $R^*$-cyclic representation of $A$.
\end{prop}

\begin{proof}
Suppose that $A$ has an $R^*$-cyclic representation $G(A)$.
Oberve that any two bonsais in $Sink(D')$ are not in a same subpath of $T_{G(A)}$ (otherwise, by transitivity of the relation $\prec_D$ and since $D'$ is an induced subgraph of $D$, one of both bonsais would not be in $Sink(D')$, a contradiction). Then, by Lemma \ref{lemcycle} and construction of $D'$, it is possible to move if necessary some bonsais in $G(A)$ corresponding to vertices in a directed cycle of $D$, in order to obtain a $D'$-clean $R^*$-cyclic representation of $A$.
\end{proof}\\

\begin{lem}\label{lemcyclic4}
Suppose that $A$ has a $D'$-clean $R^*$-cyclic representation $G(A)$. Then for any bonsai $E_\ell\in Sink(D')$,  the node $v_\ell$ belongs to the basic cycle of $G(A)$.
\end{lem}

\begin{proof}
For all $1\le \ell \le b$, if $v_\ell$ is not a node of the basic cycle, then there exists an edge leaving $E_\ell$ in the digraph $T_{G(A)}$, contradicting the inclusion $Sink(D') \subseteq Sink(T_{G(A)})$.
\end{proof}

\begin{lem}\label{lemcyclicchain}
Suppose that $A$ has an $R^*$-cyclic representation $G(A)$.
If some bonsai $E_\ell \in Sink(D')\verb"\"V_c$ is such that $v_\ell=w_{\rho}$, then $J_\ell^2=\emptyset$ and the poset $(\{R_j\,:\, j\in f^*(E_\ell)\},\subseteq )$ is a chain.
\end{lem}

\begin{proof}
Let $E_\ell\in Sink(D')\verb"\"V_c$ such that $v_\ell=w_{\rho}$.
Since the edges $e_1$ and $e_\rho$ enter the central node $w_{\rho}$ in $G(A)$ and the matrix $A$ is nonnegative, by Lemma \ref{lemdefiWeight2} it follows that all $B_\ell$-paths in $G(A)$ are leaving $v_\ell$. Then, by Lemma \ref{lemdigraphutile}, we have that $J_\ell^2=\emptyset$. By the third condition in the definition of a central bonsai, the proof is done.
\end{proof}\\

\begin{lem}\label{lemcyclic2}
Suppose that $A$ has a $D'$-clean $R^*$-cyclic representation $G(A)$. Then for any central bonsai $E_\ell$ ($1\le \ell\le b$), $v_\ell$ is equal to the central node in $G(A)$.
\end{lem}

\begin{proof}
Let $E_\ell$ be a sink vertex of $D'$ such that $J_\ell^2=\emptyset$.
If the vertex $E_\ell$ is a $\beta$-fork  for some $\beta \in S^*$ or $f^*(E_\ell) \cap S_0^*\neq \emptyset$, since $A$ is nonnegative, the conclusion of the lemma follows from Lemma \ref{lemdefiWeight2} and Corollary \ref{cordigraphDbeta2}.

Now assume that the poset $(\{R_j \, : \, j\in f^*(E_\ell)\},\subseteq)$ is not a chain. By Lemma \ref{lemcyclic4}  $v_\ell=v_\ell^*$.
Suppose by contradiction that $v_\ell$ is not a central node, say $v_\ell=w_k$ for some $k< \rho$. Since $J_\ell^2=\emptyset$,
by Lemma \ref{lemdigraphutile}, either all
$B_\ell$-paths of $G(A)$ enter $v_\ell$ or leave $v_\ell$. Assume that all $B_\ell$-paths of $G(A)$ are entering $v_\ell$. It results that any interval $p_j$ of $G(A)$ (with $j \in f^*(E_\ell)$) is leaving $v_\ell=w_k$, so $p_j$ contains the edge $]w_{k},w_{k+1}]$ and is included in the path $]w_k,]w_{k},w_{k+1}],\ldots,]w_{\rho-1},w_\rho],w_\rho]$. Thus the poset $(\{R_j \, : \, j\in f^*(E_\ell)\},\subseteq)$ is a chain, a contradiction. We obtain a similar contradiction if all $B_\ell$-paths of $G(A)$ are leaving $v_\ell$.
\end{proof}\\

Let $\Lambda$ and $\Lambda'$ be two sets whose elements are subsets of $R^*$.
A pair of posets $(\Lambda, \subseteq)$ and $(\Lambda', \subseteq)$ is said to be \emph{exclusive}\index{exclusive pair of posets} if for all $R\in \Lambda, R' \in \Lambda'$, we have $R\nsubseteq R'$ and $R'\nsubseteq R$.
A poset $(\Lambda, \subseteq)$  is called an \emph{$R^*$-double chain}\index{double@$R^*$-double chain} if $(\Lambda\verb"\"\{R^*\}, \subseteq)$ is spanned by an exclusive pair of chains. 

\begin{lem}\label{lemcyclic5}
If a poset $(\Lambda,\subseteq)$ is an $R^*$-double chain, then a spanning of 
$(\Lambda\verb"\"\{R^*\},\subseteq)$ by an exclusive pair of chains is unique.
\end{lem}

\begin{proof}
Clearly, if $(\Lambda\verb"\"\{R^*\},\subseteq)$ is a chain, then we can not span it by an exclusive pair of nonempty chains and so there is a unique spanning chain. Suppose by contradiction that there exist four nonempty chains $(\Lambda_1,\subseteq)$, $(\Lambda_2,\subseteq)$, $(\Lambda_3,\subseteq)$ and $(\Lambda_4,\subseteq)$ such that for $i=1\m{ and }3$, the pair of chains $(\Lambda_i,\subseteq)$ and $(\Lambda_{i+1},\subseteq)$ is  exclusive and spans $(\Lambda\verb"\"\{R^*\},\subseteq)$, and $\Lambda_3\neq \Lambda_1, \Lambda_2$. Up to a renumbering of the chains, we may suppose that there exist $R\in \Lambda_3\cap \Lambda_1 $ and
$R' \in \Lambda_3\cap \Lambda_2$. Since $(\Lambda_3,\subseteq)$ is a chain, it follows that $R\subseteq R'$ or $R'\subseteq R$, which contradicts the fact that the pair of chains $(\Lambda_1,\subseteq)$ and $(\Lambda_2,\subseteq)$ is exclusive.
\end{proof}\\

Now let us prove that the poset $(\Lambda_c, \subseteq)$ is an $R^*$-double chain, provided that $A$ is $R^*$-cyclic.

\begin{prop}\label{propcyclic3}
If $A$ is $R^*$-cyclic, then the poset $(\Lambda_c,\subseteq)$ is an $R^*$-double chain and $V_c$ is $R^*$-compatible.
\end{prop}

\begin{proof}
Suppose that $A$ has an $R^*$-cyclic representation $G(A)$. 
By Lemmas  \ref{lemdigraph12} and \ref{lemcyclic2}, we deduce that the poset $(\Lambda_c\verb"\"\{R^*\},\subseteq)$ is spanned by the exclusive pair of chains $(\{R_j \in \Lambda_c \,:\, 1\in R_j, \rho \notin R_j\},\subseteq)$ and $(\{R_j \in \Lambda_c \,:\, 1\notin R_j, \rho\in R_j\},\subseteq)$.

On the other hand, let $E_\ell\in V_c$ such that $g_\beta(E_\ell)=2$ for some $\beta \in S^*$. By Corollary \ref{cordigraphDbeta2} $R_j=R^*$. Similarly, if there exist $E_\ell,E_{\ell'}\in V_c$ such that $g_\beta(E_\ell)=1$ and $g_\beta(E_{\ell'})=1$ for some $\beta \in S^*$, then by Lemma \ref{lemcyclic2} $v_\ell=v_{\ell'}=w_\rho$.
Following the lines of the proof of Corollary \ref{cordigraphDbeta2}, we obtain the same conclusion.  
\end{proof}\\

By assuming that the poset $(\Lambda_c,\subseteq)$ is an $R^*$-double chain, the $R^*$-open matrix $O(R^*)$ is defined as follows. Let us partition $\Lambda_c\verb"\"\{R^*\}$ into two subsets $\Lambda_c^1$  and $\Lambda_c^2$ such that the posets $(\Lambda_c^1,\subseteq)$ and $(\Lambda_c^2,\subseteq)$ build up an exclusive pair of chains. Let $F_c^i=\{j\, :\,\exists\hspace{.07cm} E_\ell\in V_c \,\, \m{s.t.}\, j\in f^*(E_\ell),\,  R_j \in \Lambda^i_c \}$ for $i=1$ and $2$, $R=\displaystyle{\cup_{E_\ell\in Sink(D')\verb"\"V_c}} E_\ell \cup R^*$ and $F=f(R) \verb"\" (F_c^1 \cup F_c^2)$. Then

$$
O(R^*)=\left[ 
\begin{array}{ccccc} 
A^{\frac{1}{2}\rightarrow 1}_{R\times F} & 
A_{R \times F_c^1} & A_{R \times F_c^2} & 
\chi^{R}_{R^*} & \chi^{R}_{R^*} \\
\bf{0_{1\times m}} & \bf{1_{1\times |F_c^1|}} &\bf{0_{1\times |F_c^2|}} & 0& 1 \\
\bf{0_{1\times m}} & \bf{0_{1\times |F_c^1|}} &\bf{1_{1\times |F_c^2|}} & 0& 1 \\ 
\end{array} \right],$$ 

\noindent
where the rows of $O(R^*)$ except the two last ones are indexed by the elements of the set $R$. If $O(R^*)$ has a network representation $G(O(R^*))$, then the two last rows of $O(R^*)$ correspond to basic \emph{artificial}\index{artificial edge} edges denoted as $e_0$ and $\tilde e_\rho$.
Let us see the graphical interpretation of $O(R^*)$.

\begin{prop}\label{propcyclicM*}
Suppose that $A$ is $R^*$-cyclic. Then the $R^*$-open matrix $O(R^*)$ is a network matrix.
\end{prop}

\begin{proof}
By Proposition \ref{propcyclicgood}, let $G(A)$ be a basic  $D'$-clean $R^*$-cyclic representation of $A$. 
Let $\Lambda^1=\{R_j \, : \, \rho \notin R_j, 1\in R_j \}$ and $\Lambda^2=\{R_j \, : \, 1 \notin R_j, \rho \in R_j \}$.
For $i=1$ and $2$, let $\mathcal{B}^i=\{ E_\ell \in Sink(D')\verb"\"V_c \, : \, v_\ell=w_\rho\,,\, \{R_j,\, j\in f^*(E_\ell)\}\subseteq \Lambda^i\}$. 

For any $E_\ell \in Sink(D')\verb"\"V_c$ such that $v_\ell=w_\rho$, we make the following observations. From the definition of a central bonsai and Lemma \ref{lemcyclicchain}, it results that $R_j\neq R^*$ for all $j\in f^*(E_\ell)$ and the poset 
$(\{R_j\,:\, j\in f^*(E_\ell)\},\subseteq )$ is a chain. Hence, using Lemma \ref{lemdigraph12}, the set $\{R_j\,:\, j\in f^*(E_\ell)\}$ is included in $\Lambda^1$ or $\Lambda^2$. Thus 

\begin{eqnarray}\label{eqncyclicSink}
\{ E_\ell \in Sink(D')\verb"\"V_c \, : \, v_\ell=w_\rho\}= \mathcal{B}^1\uplus \mathcal{B}^2. 
\end{eqnarray}

\noindent
Furthermore, by Lemmas \ref{lemdigraph12} and \ref{lemcyclic2},  $\Lambda_c \verb"\" \{R^* \} \subseteq \Lambda^1 \cup \Lambda^2$. So by Lemma \ref{lemcyclic5}, we may assume that

\begin{eqnarray}\label{eqncyclicSinkVc}
\Lambda_c^1=\Lambda_c\cap \Lambda^1 \m{ and } \Lambda_c^2=\Lambda_c \cap \Lambda^2.
\end{eqnarray} 

 Since $G(A)$ is $D'$-clean, by Lemma \ref{lemcyclic4} we deduce that $R$ is the edge index set of a connected graph in $G(A)$. Using $G(A)$, we construct a basic network representation $T$ of the matrix $O(R^*)$. First, let us 
contract all basic edges with index in $\cup_{E_\ell\in\overline{ (Sink(D'))} \cup V_c} E_\ell$. Then create two new basic edges (and vertices) $e_0=]\tilde w_0,w_0]$ and $\tilde e_\rho =]w_\rho,\tilde w_\rho]$ and replace the bidirected edge $[w_1,w_\rho]$ by $]w_0,w_1]$. Finally, for every $E_\ell\in \mathcal{B}^1$, let us get the bonsai $B_\ell$ loose from the central node (by making a copy of $v_\ell=w_\rho$), reverse the orientation of all its basic edges and identify the copy of $v_\ell$ with $w_0$. (See Figure \ref{fig:cyclicM*} for an example.)

Let us prove that $T$ is a basic network representation of $O(R^*)$. To see that each column of $O(R^*)$ is the edge incidence vector of a directed path in $T$, we state two claims.\\

\noindent
{\bf Claim 1:} \quad For all $j\in F$, the column $[ (A^{\frac{1}{2}\rightarrow 1}_{R\times \{j\} })^T\,\, 0 \,\, 0]^T$ is the edge incidence vector of a directed path in $T$.\\

\noindent
{\bf Claim 2:} \quad For all $j\in F_c^1$ (respectively, $j \in F_c^2$), $s(A_{\bullet j})\cap R$ is the edge index set of a directed path in $T$ whose $w_0$ (resp., $w_\rho$) is an endnode.\\

\noindent
{\bf Proof of Claim 1.} \quad Let $j\in F$. If $1,\rho \notin s(A_{\bullet j})$, then by construction of $T$, up to a reversing of all edges the paths in $G(A)$ and in $T$ with edge index set $s(A_{\bullet j})\cap R$ are isomorphic.
Suppose that $1,\rho \in s(A_{\bullet j})$.  By Lemma \ref{lemdigraph12} $R_j=R^*$. If $j$ is in the global connector set of a bonsai $E_\ell \in Sink(D')$, then by Lemma \ref{lemcyclic4} $v_\ell=w_\rho$. Hence, by Lemma \ref{lemcyclicchain} or the definition of a central bonsai  $J_\ell^2=\emptyset$. Therefore $E_\ell\in V_c$, a contradiction. Thus $s(A_{\bullet j})\cap R=R^*$ is the edge index set of a directed path in $T$. At last, if exactly one of the indexes $1$ or $\rho$ is in $s(A_{\bullet j})$, then by equation (\ref{eqncyclicSink}) and construction of $T$, the proof follows.\\

\noindent
{\bf Proof of Claim 2.} \quad  
Let $j\in F_c^1$ (the proof is symmetric with $j\in F_c^2$). By definition of $F_c^1$, there exists a bonsai $E_\ell\in V_c$ such that $j\in f^*(E_\ell)$ and $R_j \in \Lambda_c^1$. 
Hence $1\in s(A_{\bullet j})\cap R$, $\rho \notin s(A_{\bullet j})\cap R$ and by Lemma \ref{lemcyclic2} $v_\ell=w_\rho$. Since $E_\ell \cap R = \emptyset$, it follows that $s(A_{\bullet j})\cap R$ is the edge index set of a directed path in $G(A)$ with one endnode equal to $w_\rho$. Moreover, by (\ref{eqncyclicSinkVc}), $R_j\in \Lambda^1$. Then, by construction of $T$, the proof of the claim follows.\\

\noindent
The vectors $[(\chi_{R^*}^{R})^T \,\, 0 \,\, 0]^T$ and $[ (\chi_{R^*}^{R})^T \,\, 1 \,\, 1]^T$ are clearly edge incidence vectors  of directed paths in $T$. Using Claims 1 and 2, it is easy to conclude the proof of Proposition \ref{propcyclicM*}.
\end{proof}\\

Suppose that the matrix $O(R^*)$ has a basic network representation $G(O(R^*))$. Since $[(\chi_{R^*}^{R})^T \, 0 \, 0]^T$ is a column of $O(R^*)$, the set $R^*$ is the edge index set of a directed path denoted as $p^*$ in $G(O(R^*))$. Moreover, since $A$ is connected $(O(R^*))_{R \bullet}$ is connected, hence $R$ is the edge index set of a tree in $G(O(R^*))$. Thus, if $F_c^1 \neq \emptyset$ and $F_c^2 \neq \emptyset$, since the vectors $[(\chi_{R^*}^{R})^T \, \, 1 \, \, 1 ]^T$, $[A_{R \times \{j_1\}}^T \, \, 1 \, \, 0 ]^T$  and $[A_{R \times \{j_2\} }^T \, \, 0 \, \, 1 ]^T$ (for some $j_1\in F_c^1$, $j_2\in F_c^2$) are edge incidence vectors of paths in $G(O(R^*))$, we deduce that $e_0$ is incident with one endnode of $p^*$ and $\tilde e_\rho$ with the other one.
We obtain the same conclusion that whenever $F_c^1=\emptyset$ or $F_c^2=\emptyset$.
By reversing the orientation of all edges in $G(O(R^*))$ if necessary, we may note 
$$p^*=]w_0,]w_0,w_1],\ldots,]w_{\rho-1},w_\rho],w_\rho],$$ $ e_0=]\tilde w_0,w_0]$ and $\tilde e_\rho=]w_\rho,\tilde w_\rho]$. For ease of notation, a basic subtree in $G(O(R^*))$ with edge index set equal to $E_\ell$ for some $1\le \ell\le b$ is called a bonsai $B_\ell$\index{bonsai $B_\ell$}, and $v_\ell$\index{vertex!$v_\ell$} denotes the closest node in $B_\ell$ to the path $p^*$.

At this point, we are ready to describe the main subroutine of the procedure RCyclic. Let $$A'= A_{\cup_{E_\ell\in Sink(D')} E_\ell
\cup R^* \bullet }.$$ In the following procedure, 
for each $E_\ell\in V_c$, if $N_\ell$ has a $v_\ell$-rooted network representation $B_\ell$, then up to a reversing of all edges in $B_\ell$,
we assume that all $B_\ell$-paths in $B_\ell$ leave $v_\ell$
(since $J_\ell^2= \emptyset$,  this is justified by Lemma \ref{lembonsainet2}).

\begin{tabbing}
\textbf{Procedure\,\,GASink($A$,$R^*$)}\\

\textbf{Input: }\= A matrix $A$ and a row index subset $R^*$ of $A$ such that $R^*\subseteq s(A_{\bullet j})$\\
\> for at least one column index $j$.\\
\textbf{Output: }\= Either a basic $R^*$-cyclic representation $G(A')$ of $A'= A_{\cup_{E_\ell\in Sink(D')} E_\ell
\cup R^* \bullet }$ and  \\
\> $D'\subseteq D$, or determines that $A$ is not $R^*$-cyclic;\\
1) \verb"  "\=  call {\tt Initialization}($A$,$R^*$) outputing a subgraph $D'\subseteq D$, or the fact that $A$ is not\\
\>  $R^*$-cyclic.\\
2) \> compute, if they exist, $O(R^*)$, a basic network representation $G(O(R^*))$ of $O(R^*)$\\
\> and a $v_\ell$-rooted network representation $B_\ell$ of $N_\ell$ for every $E_\ell \in V_c$,\\
\>  otherwise STOP: output that $A$ is not $R^*$-cyclic;\\
3) \> check that $V_c$ is $R^*$-compatible, otherwise STOP: output that $A$ is not $R^*$-cyclic;\\
4) \> check that $Sink(D')$ is feasible, otherwise STOP: output that $A$ is not $R^*$-cyclic;\\
5) \> delete $e_0$, $\tilde e_\rho$, $\tilde w_0$ and $\tilde w_\rho$; replace the edge $]w_0,w_1]$ by $[w_0,w_1]$; \\ 
\> reverse the orientation of all  edges of the bonsais connected at $w_0$; \\
\> identify $w_0$ with $w_\rho$, and for each $E_\ell \in V_c$ identify $v_\ell$ with $w_\rho$;\\
\> output the resulting $1$-tree and $D'\subseteq D$;
\end{tabbing}

Let us show the correctness of the procedure GASink.

\begin{prop}\label{propcyclicA'}
The output of the procedure GASink is correct.
\end{prop}

\begin{proof} 
Suppose that $A$ is $R^*$-cyclic.
By Corollary \ref{corBidirectedCircuit}, Lemma \ref{lemdefiWeight1} and the description of the different types of fundamental circuits given at page \pageref{mycounter2} (see Figure \ref{fig:Apositive}), the subroutine Initialization does not stop at its step 1 and produces an induced subgraph $D'\subseteq D$ without any directed cycle by construction.
By Proposition \ref{propcyclic3}, the $R^*$-open matrix $O(R^*)$ is well defined. Furthermore,  by Proposition \ref{propcyclicM*} and Lemma \ref{lembonsaicel} (respectively, Proposition \ref{propcyclic3} and Theorem \ref{thmdigraphfeasible}), the procedure GASink does not stop at step 2 (respectively, step 3 and 4). Finally, by construction, it outputs a $1$-tree.

Now suppose that the procedure outputs a $1$-tree denoted as
$G(A')$. One needs to show that each column in $A'$ can be interpreted as a binet matrix whose $G(A')$ is a basic binet representation. Let $j_0$ be an index of a nonzero column in $A'$. 

If $s(A_{\bullet j_0}') \cap R^* =\emptyset$, then since $s(A_{\bullet j_0}') \subseteq E_\ell$ for some bonsai $E_\ell \in Sink(D')$ and $B_\ell$ is a $v_\ell$-rooted network representation of $N_\ell$ it follows that $A_{\bullet j_0}'$ is the edge incidence vector of a directed path in $G(A')$. 
Now assume that $j_0\in S^*$. Since $Sink(D')$ is feasible, $j_0$ is in the global connecter set of at most two bonsais. 
If $j_0$ is not in the global connecter set of a central bonsai, then $[A_{R\times \{j_0\}}^T\, 0 \, 0]^T$ is the edge incidence vector of a directed path in $G(O(R^*))$. So by construction $s(A_{\bullet j_0}') \cap R= s_1(A_{\bullet j_0}') \cap R$ and $s(A_{\bullet j_0}') \cap R$ is the edge index set of a consistently oriented path or a pathcycle.

Suppose that $j_0$ is in the global connecter set of a central bonsai, say $E_{l_c}$, and a non-central one, say $E_\ell$ in $Sink(D')$. Suppose by contradiction that $R_{j_0}=R^*$. Since $s(A_{\bullet j_0}')=E_\ell^k \cup R^*$ for some $1\le k \le m(\ell)$, it follows that 
$v_\ell= w_0$ or $v_\ell= w_\rho$ in $G(O(R^*))$; this implies that all paths in $G(O(R^*))$ with edge index set equal to $E_\ell^1, \ldots , E_\ell^{m(\ell)}$ leave $v_\ell$ or they all enter $v_\ell$. Then, the bonsai $B_\ell$ of $G(O(R^*))$ with one more edge entering $v_\ell$ is a $v_\ell$-rooted network representation of $N_\ell$, so by Lemma \ref{lemdigraphutile} $J_\ell^2=\emptyset$. Hence $E_\ell$ is central, a contradiction. Therefore $R_{j_0}\neq R^*$. It follows that $j_0$ belongs to $F_c^1$ or $F_c^2$. Assume $j_0 \in F_c^1$ (if $j_0 \in F_c^2$, then the proof is similar). Since
$[A_{R\times \{j_0\}}^T  \, 1 \, 0]^T$ is the edge incidence vector of a directed path in $G(O(R^*))$ containing $e_0=]\tilde w_0,w_0]$ and a subpath of $p^*=]w_0,]w_o,w_1],\ldots, w_\rho]$, this implies that $v_\ell$ is an inner node of the path $p^*$ ($v_\ell\neq w_0$ and $v_\ell\neq w_\rho$). Moreover, since $V_c$ is $R^*$-compatible, $g_{j_0}(E_{l_c})=1$. So $A'_{\bullet j_0}$ is the incidence vector of a consistently oriented path in $G(A')$.

If $j_0$ is in the global connecter set of one or two central bonsais, with similar arguments one may prove that $A'_{\bullet j_0}$ is the edge index set of some basic fundamental circuit in $G(A')$. This completes the proof.
\end{proof}\\

The following lemma will justify a way of connecting some 
bonsais to a basic $R^*$-cyclic representation of $A'$ whenever $A$ is $R^*$-cyclic.

\begin{lem}\label{lemcyclicpathA'}
Suppose that $A'$ has a basic $R^*$-cyclic representation $G(A')$. Let $T_\Theta$ be a $D'$-clean feasible spanning forest of $D$ and $E_\ell\in Sink(T_\Theta) \verb"\" Sink(D')$ such that $N_\ell$ is a network matrix.
Then $s(A'_{\bullet j})=s(A'_{\bullet j'})$ for all $j,j' \in f^*(E_\ell)$. Moreover, for any $j\in f^*(E_\ell)$, $s(A'_{\bullet j})$ is the edge index set of a pathcycle, or a consistently oriented path with exactly one endnode lying on the basic cycle of $G(A')$.
\end{lem}

\begin{proof}
By construction of $D'$ and using the transitivity of the relation $\prec_{D}$, there exists a sink vertex of $D'$, say $E_{\ell'}$, such that $(E_\ell,E_{\ell'})\in D$. By the property $\Pi_1$ or $\Pi_1^*$ satisfied by $T_\Theta$, it follows that $g_\beta(E_\ell) =1$, $g_\beta(E_{\ell'})=1$ and $g_\beta(E_{\ell''})=0$ for any $\beta \in f^*(E_\ell)$ and $E_{\ell''} \in Sink(T_\Theta) \verb"\"\{E_\ell,E_{\ell'}\}$. Suppose that there exist two indexes $\beta,\beta'\in f^*(E_\ell)$. By Lemma \ref{lembonsaisim} (part 2) this implies that $E_{\ell'}^I(A_{\bullet \beta}) \sim_{E_{\ell'}} E_{\ell'}^{I}(A_{\bullet \beta'})$ and $E_\ell^I(A_{\bullet \beta}) \sim_{E_\ell} E_\ell^I(A_{\bullet \beta'})$, 
hence the graph $T_\Theta(\{ E_u\in Sink(T_\Theta)\, :\, \beta,\beta'  \in f^*(E_u) \,,\, E_u^I(A_{\bullet \beta})\sim_{E_u} E_u^I(A_{\bullet \beta'})\})$ is neither a $\beta$-path nor a $\beta'$-path. As $T_\Theta$ is feasible, $T_\Theta$ satisfies in particular property $\Pi_3$ and we deduce that $R_\beta=R_{\beta'}$.
This implies Lemma \ref{lemcyclicpathA'}.
\end{proof}\\

By Lemma \ref{lemcyclicpathA'}, if $A'$ has a basic $R^*$-cyclic representation $G(A')$ and $D$ some $D'$-clean feasible spanning forest $T_\Theta$, then for any $E_\ell\in Sink(T_\Theta) \verb"\" Sink(D')$, we define a vertex $v_\ell^*$\index{vertex!$v_\ell^*$} in $G(A')$ as follows.
For any $j\in f^*(E_\ell)$, if $s(A_{\bullet j})\cap R^*\neq R^*$, then $v_\ell^*$ is equal to the endnode of the path with edge index set $s(A_{\bullet j})\cap E$ lying on the basic cycle, otherwise $v_\ell^*$ is the central node. 

Finally, we provide the procedure RCyclic for testing whether $A$ is $R^*$-cyclic or not. 
Let $V_0'$ denote the set of $E_\ell \in V'$ such that $E_\ell$ is a sink vertex of $D'$, or there exists $\beta \in f^*(E_\ell)$ such that $s_{\frac{1}{2}}(A_{\bullet \beta})\neq \emptyset$, or there exist $\beta,\beta' \in f^*(E_\ell)$ such that 
the intervals $R_\beta$ and $R_{\beta'}$ are not equal ($R_\beta\neq R_{\beta'}$) and $E_\ell^I(A_{\bullet \beta}) \sim_{E_\ell} E_\ell^I(A_{\bullet \beta'})$.
We use the subroutine Forest described in Section \ref{sec:For} taking $A$, $D'$ and a subset $V_0\subseteq V'$ as input and outputing, for $V_0= V_0'$, a feasible spanning forest of $D'$ whenever $A$ is $R^*$-cyclic.

\begin{tabbing}
\textbf{Procedure\,\,RCyclic($A$,$R^*$)}\\

\textbf{Input: }\= A matrix $A$ and a row index subset $R^*$ of $A$ such that $R^*\subseteq s(A_{\bullet j})$\\
\> for at least one column index $j$.\\
\textbf{Output: } \= Either a basic $R^*$-cyclic representation $G(A)$ of $A$,\\
\> or determines that none exists.\\ 
1) \verb"  "\=  call {\tt GASink}($A$,$R^*$) outputing    
a basic $R^*$-cyclic representation $G(A')$ of some submatrix\\
\> $A'$ of $A$ and $D'\subseteq D$, or the fact that $A$ is not $R^*$-cyclic; \\
2) \> call {\tt Forest}($D'$,$V_0= V_0'$) and check that the output of {\tt Forest} is a feasible forest,\\
\> otherwise STOP: output that $A$ is not $R^*$-cyclic;\\
\> compute a $D'$-clean feasible spanning forest $T_\Theta$ of $D$;\\
3) \>  for every $E_\ell\in V(D)\verb"\"V(D')$, compute a $v_\ell$-rooted network representation $B_\ell$ of $N_\ell$,\\
\> if one exists; otherwise STOP: output that $A$ is not $R^*$-cyclic;\\
4) \> {\bf for }\= each $E_\ell\in Sink(T_\Theta)\verb"\"Sink(D')$, {\bf do}\\
   \>            \> identify the node $v_\ell$ with $v_\ell^*$ and orient the edges of $B_\ell$ so that 
for any $1\le k \le m(\ell)$\\
\> \>  and $\beta \in f^*(E_\ell)$, $E_\ell^k\cup R_\beta $  is the edge index set of a consistently oriented path;\\ 
\> {\bf endfor }\\
5) \> {\bf for }\= any  $E_{\ell'} \in T_\Theta\verb"\"Sink(T_\theta)$, {\bf do}\\
   \>     \> let $(E_{\ell'},E_\ell)_{E_\ell^k}\in T_{\Theta}$ and identify $v_{\ell'}$ with the endnode ($\neq v_\ell$) of the $B_\ell$-path\\
\> \> with  edge index set $E_\ell^k$, and orient the edges of $B_{\ell'}$ so that for all $1\le j \le m(\ell')$,\\
\> \> $E_{\ell'}^j\cup E_\ell^k $ is the edge index set  of a directed path;\\ 
   \> {\bf endfor }\\
\> output the obtained $1$-tree;
\end{tabbing}

\noindent
{\bf Proof of Theorems \ref{thmcyclicproCyc} and \ref{thmcyclic1}.} \quad The proof of the "only if" part of Theorem \ref{thmcyclic1} follows from Propositions \ref{propcyclicgood}, \ref{propcyclic3} and \ref{propcyclicM*} and Lemma \ref{lembonsaicel}.

Let us show the correctness of the procedure RCyclic. Suppose that $A$ is $R^*$-cyclic. By Proposition \ref{propcyclicA'}, the subroutine  GASink produces a basic $R^*$-cyclic representation $G(A')$ of $A'$. By Theorem \ref{thmForest1}, the subroutine Forest outputs a feasible spanning forest of $D'$. Then, by Lemma \ref{lemcycle}, it is easy to compute a $D'$-clean feasible spanning forest of $D$. By Lemma 
\ref{lembonsaicel}, in step 3 a $v_\ell$-rooted network representation of the bonsai matrix $N_\ell$ is computed, for every $E_\ell \in V(D)\verb"\" V(D')$. By Lemma \ref{lemcyclicpathA'} it follows that for any $E_\ell \in Sink(T_\Theta)\verb"\"Sink(D')$ the vertex $v_\ell^*$ is well defined. At step 4 and 5, for any $E_\ell\in T_\Theta$ and $1\le k \le m(\ell)$, $E_\ell^k$ is the edge index set of a path in $B_\ell$ by Lemma \ref{lembonsainet2}. Finally, by construction, the procedure RCyclic outputs a $1$-tree $G(A)$. 

Now suppose that the the procedure RCyclic outputs a $1$-tree $G(A)$. One need to show that $G(A)$ is a basic $R^*$-cyclic representation of $A$. At the end of step 1, by Proposition \ref{propcyclicA'} we have a basic $R^*$-cyclic representation of $A'$. Then, using Lemmas \ref{lemcyclicpathA'}
and \ref{lembonsainet2} and the fact that $T_\Theta$ is feasible, we deduce that $G(A)$ is a basic $R^*$-cyclic representation of $A$. This implies the "if" part of Theorem \ref{thmcyclic1}.

At last, let us analyse the running time of the procedure RCyclic. By Lemma \ref{lemdigraphDtimeD}, the computation of $D$ requires time $O(nm \alpha)$.
Let $k=|V_c \cup \overline{Sink(D')}|$, $n_0=|R|+2$ and $n_\ell= |E_\ell|+2$ for all $E_\ell \in V_c \cup \overline{Sink(D')}$.
Denote by $\alpha_0$ (respectively, $\alpha_\ell$) the number of nonzero elements of the matrix $O(R^*)$ (respectively, $N_\ell$) for all $1\le \ell \le b$ such that $E_\ell \in V_c \cup \overline{Sink(D')}$. 
Since $\sum_{E_\ell \in V_c \cup \overline{Sink(D')}}\, n_\ell +n_0 = n+ 2(k+1)$ and $\sum_{E_\ell \in V_c \cup \overline{Sink(D')}}\, \alpha_\ell +\alpha_0 \le 3 \alpha$, using Theorem \ref{thmSubclassNTutteCunNet}, step 2 in the subroutine GASink and step 3 in the procedure RCyclic altogether take time at most
$$\sum_{E_\ell \in V_c \cup \overline{Sink(D')}} n_\ell \alpha_\ell +n_0 \alpha_0 \le (3 n) (3\alpha ) =C n \alpha,$$

\noindent
for some constant $C$. All other steps perform in time $O(\alpha)$. This concludes the proof of Theorem
\ref{thmcyclicproCyc}.
{\hfill$\BBox{\rule{.3mm}{3mm}}$}


\chapter{Recognizing $\frac{1}{2}$-binet matrices}
\label{ch:halfbinet}

Let $A$ be a matrix of size $n\times m$ with $0$, $1$ or $2$ coefficients and $\alpha$ the number of nonzero elements in $A$.
In this chapter, we address the problem of recognizing whether $A$ is $\frac{1}{2}$-binet.  More generally, given $A$ 
and a row index subset $Q$ of $A$, 
we describe a polynomial-time procedure, called OnehalfbinetQ,
that computes a $\frac{1}{2}$-binet representation of $A$ such that each element in $Q$ is a basic half-edge index, or determines that none exists. We will prove the following theorem.

\begin{thm}\label{thmOnhalfbinetQ}
Let $Q$ be a row index subset of $A$. The matrix $A$ can be tested for having a $\frac{1}{2}$-binet representation such that each element in $Q$ is a basic half-edge index, in time $O(n m^2 \alpha)$, by the procedure OnehalfbinetQ.
\end{thm}

Let $S_2=\{ j \,: \, s_2(A_{\bullet j}) \neq \emptyset\}$.
A family $\mathscr{I}$ of row index subsets of $A$ is said to be \emph{$1$-forest}\index{forest@$1$-forest} if one can prove the following: If $A$ is $\frac{1}{2}$-binet, then in any $\frac{1}{2}$-binet representation of $A$, each element of $\mathscr{I}$ corresponds to the edge index set of a basic (negative) $1$-tree and the elements of $\mathscr{I}$ are pairwise disjoint. 

Our first aim is to compute some $1$-forest family $\mathscr{I}$ such that if $A$ is $\frac{1}{2}$-binet, then there exists a $\frac{1}{2}$-binet representation of $A$ in which all nonbasic bidirected edges (except half-edges) correspond to columns of $A$ whose support intersects two distinct elements of $\mathscr{I}$. Then, we reduce our problem to the recognition of an $\{i\}$-cyclic matrix for some row index $i$. 

\section{An informal sketch of a recognition procedure}\label{sec:infcycOnehalfbinet}

We first illustrate and comment the procedure OnehalfbinetQ with input the $\frac{1}{2}$-binet matrix $A$ given in Figure \ref{fig:OnehalfbinetA} and the empty row index subset ($Q=\emptyset$). We have to build up a $\frac{1}{2}$-binet representation of $A$ without knowing that such a representation exists. Since the first column of $A$ has a nonempty $2$-support, this implies the following. Provided that $A$ has a binet representation $G(A)$, $f_1$ is a bidirected non-half
$1$-edge whose fundamental circuit contains a basic half-edge, and the set $R_1=s(A_{\bullet 1})=\{1,2\}$ is the edge index set of a $1$-tree. The procedure OnehalfbinetQ sets
$\mathscr{I}=\{R_1\}$ which is $1$-forest and $H=H(A)\verb"\"\{1\}$.

\vspace{.3cm}
\begin{figure}[h!]
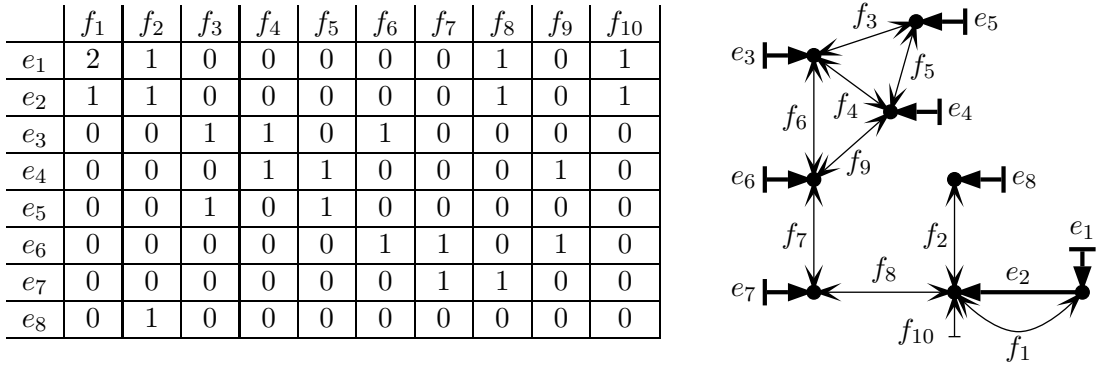

$
\begin{array}{cc}

\begin{tabular}{c|c|c|c|c|c|c|c|c|c|c|}
  & $f_1$ & $f_2$ & $f_3$ & $f_4$ & $f_5$ & $f_6$ & $f_7$ & $f_8$ & $f_9$ & $f_{10}$  \\
  \hline
$e_1$  &2&1&0&0&0&0&0&1&0&1 \\
\hline
$e_2$  &1&1&0&0&0&0&0&1&0&1 \\
\hline
$e_3$  &0&0&1&1&0&1&0&0&0&0 \\
\hline
$e_4$  &0&0&0&1&1&0&0&0&1&0\\
\hline
$e_5$  &0&0&1&0&1&0&0&0&0&0 \\
\hline
$e_6$  &0&0&0&0&0&1&1&0&1&0 \\
\hline
$e_7$  &0&0&0&0&0&0&1&1&0&0 \\
\hline
$e_8$  &0&1&0&0&0&0&0&0&0&0 \\
\hline
\end{tabular}  &

\psset{xunit=1.7cm,yunit=1.5cm,linewidth=0.5pt,radius=0.1mm,arrowsize=7pt,labelsep=1.5pt,fillcolor=black}

\pspicture(-1,1)(3,2)

\pscircle[fillstyle=solid](0,0){.1}
\pscircle[fillstyle=solid](0,1){.1}
\pscircle[fillstyle=solid](0,2.1){.1}
\pscircle[fillstyle=solid](0.6,1.6){.1}
\pscircle[fillstyle=solid](1.1,0){.1}
\pscircle[fillstyle=solid](1.1,1){.1}
\pscircle[fillstyle=solid](.8,2.4){.1}
\pscircle[fillstyle=solid](2.1,0){.1}

\psline[linewidth=1.6pt,arrowinset=0]{|->}(2.1,.4)(2.1,0)
\rput(2.1,.55){$e_1$}

\psline[linewidth=1.6pt,arrowinset=0]{<-}(1.1,0)(2.1,0)
\rput(1.6,.15){$e_2$}

\psline[linewidth=1.6pt,arrowinset=0]{<-|}(.8,2.4)(1.2,2.4)
\rput(1.35,2.4){$e_5$}

\psline[linewidth=1.6pt,arrowinset=0]{|->}(-.4,2.1)(0,2.1)
\rput(-0.55,2.1){$e_3$}

\psline[linewidth=1.6pt,arrowinset=0]{<-|}(.6,1.6)(1,1.6)
\rput(1.15,1.6){$e_4$}

\psline[linewidth=1.6pt,arrowinset=0]{|->}(-.4,1)(0,1)
\rput(-0.55,1){$e_6$}

\psline[linewidth=1.6pt,arrowinset=0]{|->}(-.4,0)(0,0)
\rput(-0.55,0){$e_7$}

\psline[linewidth=1.6pt,arrowinset=0]{|->}(1.5,1)(1.1,1)
\rput(1.65,1){$e_8$}

\pscurve[arrowinset=.5,arrowlength=1.5]{<->}(1.1,0)(1.6,-.35)(2.1,0)
\rput(1.6,-.5){$f_1$}

\psline[arrowinset=.5,arrowlength=1.5]{<->}(1.1,0)(1.1,1)
\rput(.95,.5){$f_2$}

\psline[arrowinset=.5,arrowlength=1.5]{<->}(0,2.1)(.8,2.4)
\rput(.4,2.45){$f_3$}

\psline[arrowinset=.5,arrowlength=1.5]{<->}(0,2.1)(.6,1.6)
\rput(.25,1.65){$f_4$}

\psline[arrowinset=.5,arrowlength=1.5]{<->}(0.6,1.6)(.8,2.4)
\rput(0.85,2){$f_5$}

\psline[arrowinset=.5,arrowlength=1.5]{<->}(0,1)(0,2.1)
\rput(-.15,1.55){$f_6$}

\psline[arrowinset=.5,arrowlength=1.5]{<->}(0,1)(0,0)
\rput(-.15,0.5){$f_7$}

\psline[arrowinset=.5,arrowlength=1.5]{<->}(0,0)(1.1,0)
\rput(.55,.2){$f_8$}

\psline[arrowinset=.5,arrowlength=1.5]{<->}(0,1)(0.6,1.6)
\rput(.35,1.15){$f_9$}

\psline[arrowinset=.5,arrowlength=1.5]{|->}(1.1,-.4)(1.1,0)
\rput(.8,-0.35){$f_{10}$}

\endpspicture 
\end{array}$

\vspace{.2cm}
\caption{A $\frac{1}{2}$-binet matrix $A$ and a $\frac{1}{2}$-binet representation $G(A)$ of $A$.}  
\label{fig:OnehalfbinetA}
\end{figure}

Then consider the submatrix $A_{\bullet V(H)}$. We notice that it is not a network matrix, since it contains the submatrix $A_{\{3,4,5\} \times \{3,4,5\}}$ of determinant $2$. The second step consists in computing a minimal column index subset $J$ such that $A_{\bullet J}$ is not a network matrix. For instance, $J=\{3,4,5\}$ is such a subset. Observe that  $\{3,4,5\}$ induces an odd cycle in $H(A)$ and $s(A_{\bullet 3})\cap s(A_{\bullet 4}) \cap s(A_{\bullet 5})$ is empty. Provided that $A$ has a $\frac{1}{2}$-binet representation $G(A)$,
this implies that $f_3$, $f_4$ and $f_5$ are $2$-edges in $G(A)$ and $Q_3=s(A_{\bullet 3})\cap s(A_{\bullet 4})=\{3\}$, $Q_4=s(A_{\bullet 4})\cap s(A_{\bullet 5})=\{4\}$ and $Q_5=s(A_{\bullet 5})\cap s(A_{\bullet 3})=\{5\}$ are edge index sets of $1$-trees. So, the procedure sets $\mathscr{I}:= \{R_1, Q_3,Q_4,Q_5\}$, and $\mathscr{I}$ is $1$-forest. 

\vspace{.3cm}
\begin{figure}[h!]
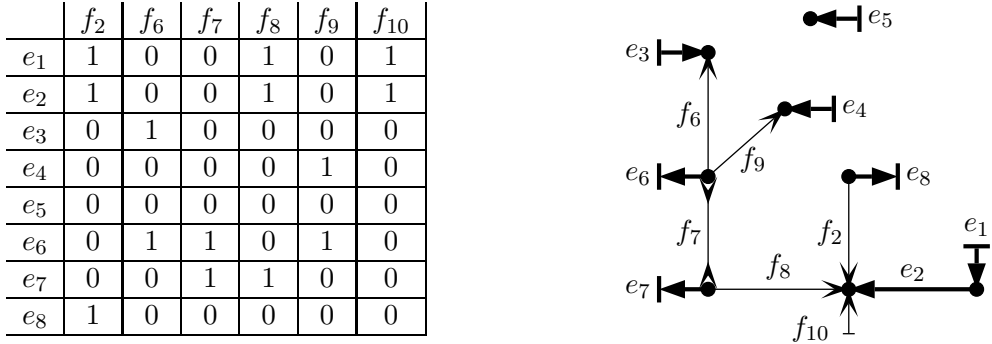


\begin{center}
$
\begin{array}{cc}

\begin{tabular}{c|c|c|c|c|c|c|}
   & $f_2$ & $f_6$ & $f_7$ & $f_8$ & $f_9$ & $f_{10}$  \\
  \hline
$e_1$  &1&0&0&1&0&1 \\
\hline
$e_2$  &1&0&0&1&0&1 \\
\hline
$e_3$  &0&1&0&0&0&0 \\
\hline
$e_4$  &0&0&0&0&1&0\\
\hline
$e_5$  &0&0&0&0&0&0 \\
\hline
$e_6$  &0&1&1&0&1&0 \\
\hline
$e_7$  &0&0&1&1&0&0 \\
\hline
$e_8$  &1&0&0&0&0&0 \\
\hline
\end{tabular}  &

\psset{xunit=1.7cm,yunit=1.5cm,linewidth=0.5pt,radius=0.1mm,arrowsize=7pt,labelsep=1.5pt,fillcolor=black}

\pspicture(-2,1)(3,2)

\pscircle[fillstyle=solid](0,0){.1}
\pscircle[fillstyle=solid](0,1){.1}
\pscircle[fillstyle=solid](0,2.1){.1}
\pscircle[fillstyle=solid](0.6,1.6){.1}
\pscircle[fillstyle=solid](1.1,0){.1}
\pscircle[fillstyle=solid](1.1,1){.1}
\pscircle[fillstyle=solid](.8,2.4){.1}
\pscircle[fillstyle=solid](2.1,0){.1}

\psline[linewidth=1.6pt,arrowinset=0]{|->}(2.1,.4)(2.1,0)
\rput(2.1,.55){$e_1$}

\psline[linewidth=1.6pt,arrowinset=0]{<-}(1.1,0)(2.1,0)
\rput(1.6,.15){$e_2$}

\psline[linewidth=1.6pt,arrowinset=0]{|->}(-.4,2.1)(0,2.1)
\rput(-0.55,2.1){$e_3$}

\psline[linewidth=1.6pt,arrowinset=0]{<-|}(.6,1.6)(1,1.6)
\rput(1.15,1.6){$e_4$}

\psline[linewidth=1.6pt,arrowinset=0]{<-|}(.8,2.4)(1.2,2.4)
\rput(1.35,2.4){$e_5$}

\psline[linewidth=1.6pt,arrowinset=0]{|<-}(-.4,1)(0,1)
\rput(-0.55,1){$e_6$}

\psline[linewidth=1.6pt,arrowinset=0]{|<-}(-.4,0)(0,0)
\rput(-0.55,0){$e_7$}

\psline[linewidth=1.6pt,arrowinset=0]{|<-}(1.5,1)(1.1,1)
\rput(1.65,1){$e_8$}

\psline[arrowinset=.5,arrowlength=1.5]{<-}(1.1,0)(1.1,1)
\rput(.95,.5){$f_2$}

\psline[arrowinset=.5,arrowlength=1.5]{->}(0,1)(0,2.1)
\rput(-.15,1.55){$f_6$}

\psline[arrowinset=.5,arrowlength=1.5]{>-<}(0,1)(0,0)
\rput(-.15,0.5){$f_7$}

\psline[arrowinset=.5,arrowlength=1.5]{->}(0,0)(1.1,0)
\rput(.55,.2){$f_8$}

\psline[arrowinset=.5,arrowlength=1.5]{->}(0,1)(0.6,1.6)
\rput(.35,1.15){$f_9$}

\psline[arrowinset=.5,arrowlength=1.5]{|->}(1.1,-.4)(1.1,0)
\rput(.8,-0.35){$f_{10}$}

\endpspicture 
\end{array}$
\end{center}

\vspace{.2cm}
\caption{The binet matrix $A_{\bullet V(H)}$, where $V(H)= \{Ê2,6,7,8,9,10\}$, and a $\frac{1}{2}$-binet representation $G(A_{\bullet V(H)})$ of $A_{\bullet V(H)}$ all of whose basic half-edges with index in $R$, for any $R\in \mathscr{I}$, are entering, with $\mathscr{I}:= \{R_1, Q_3,Q_4,Q_5\}$.}  
\label{fig:OnehalfbinetnetH}
\end{figure}

By removing $3$, $4$ and $5$ from $H$, the matrix $A_{\bullet V(H)}$ becomes a network matrix. Now assume that $H$ is the subgraph of $H(A)$ induced by the vertex set $\{2,6,7,8,9,10\}$.
Let $G(A_{\bullet V(H)})$ be a $\frac{1}{2}$-binet representation of $A_{\bullet V(H)}$  all of whose basic half-edges with index in $R$, for any $R\in \mathscr{I}$, are entering. See Figure \ref{fig:OnehalfbinetnetH}. Since $s(A_{\bullet 6}) \cap Q_3 \neq \emptyset$, $s(A_{\bullet 8}) \cap R_1 \neq \emptyset$ and $\mathscr{I}$ is $1$-forest, one can prove that for any minimal path $h$ in $H$, any nonbasic edge with index in $V(h)$ is a $2$-edge. Let $h=(6,(6,7),7,(7,8),8)$. Since the length of $h$ is odd (and the basic half-edges with index in $Q_3$ and $R_1$ are entering),
at least one of the nonbasic $2$-edges $f_6$, $f_7$ or $f_8$ is bidirected; that is undesirable (see $f_7$ in Figure \ref{fig:OnehalfbinetnetH}). 

To bypass this problem, the procedure OnehalfbinetQ computes a minimal path from $6$ to $8$ in $H$, for instance $h$ as above. Then it adds the row index subsets $Q_6=s(A_{\bullet 6}) \cap s(A_{\bullet 7})=\{6\}$ and $Q_7=s(A_{\bullet 7}) \cap s(A_{\bullet 8})=\{7\}$ in $\mathscr{I}$, and deletes the vertices $6$, $7$ and $8$ in  $H$. Then $9$ is also removed from $H$, since the support of $A_{\bullet 9}$ intersects two elements of $\mathscr{I}$, namely $Q_4$ and $Q_6$. 
Following this way, the procedure computes a graph $H\subseteq H(A)$ and a $1$-forest family $\mathscr{I}$ such that $A_{\bullet V(H)}$ is a network matrix decomposable into blocks, and the row index set of any block intersects at most one element in $\mathscr{I}$. This implies the following. There exists a $\frac{1}{2}$-binet representation $G(A_{\bullet V(H)})$ of $A_{\bullet V(H)}$ such that all $2$-edges in $G(A_{\bullet V(H)})$ are directed, and the basic half-edge of any $1$-tree with edge index set in $\mathscr{I}$ is entering (see Figure \ref{fig:OnehalfbinetnetHH}).

\vspace{.3cm}
\begin{figure}[h!]
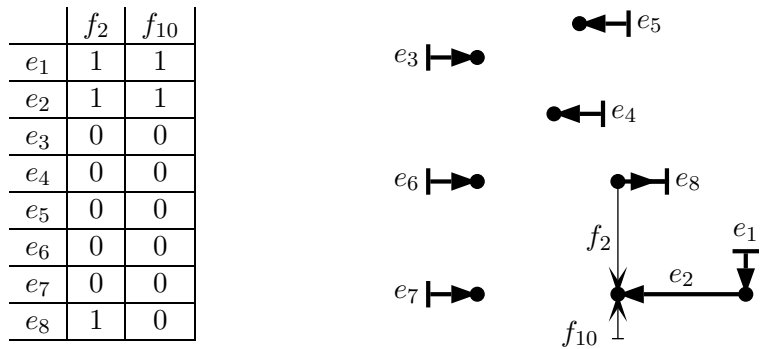


\begin{center}
$
\begin{array}{cc}

\begin{tabular}{c|c|c|}
   & $f_2$ & $f_{10}$  \\
  \hline
$e_1$  &1&1 \\
\hline
$e_2$  &1&1 \\
\hline
$e_3$  &0&0 \\
\hline
$e_4$  &0&0 \\
\hline
$e_5$  &0&0\\
\hline
$e_6$  &0&0 \\
\hline
$e_7$  &0&0 \\
\hline
$e_8$  &1&0 \\
\hline
\end{tabular}  &

\psset{xunit=1.7cm,yunit=1.5cm,linewidth=0.5pt,radius=0.1mm,arrowsize=7pt,labelsep=1.5pt,fillcolor=black}

\pspicture(-2,1)(3,2)

\pscircle[fillstyle=solid](0,0){.1}
\pscircle[fillstyle=solid](0,1){.1}
\pscircle[fillstyle=solid](0,2.1){.1}
\pscircle[fillstyle=solid](0.6,1.6){.1}
\pscircle[fillstyle=solid](1.1,0){.1}
\pscircle[fillstyle=solid](1.1,1){.1}
\pscircle[fillstyle=solid](.8,2.4){.1}
\pscircle[fillstyle=solid](2.1,0){.1}

\psline[linewidth=1.6pt,arrowinset=0]{|->}(2.1,.4)(2.1,0)
\rput(2.1,.55){$e_1$}

\psline[linewidth=1.6pt,arrowinset=0]{<-}(1.1,0)(2.1,0)
\rput(1.6,.15){$e_2$}

\psline[linewidth=1.6pt,arrowinset=0]{|->}(-.4,2.1)(0,2.1)
\rput(-0.55,2.1){$e_3$}

\psline[linewidth=1.6pt,arrowinset=0]{<-|}(.6,1.6)(1,1.6)
\rput(1.15,1.6){$e_4$}

\psline[linewidth=1.6pt,arrowinset=0]{<-|}(.8,2.4)(1.2,2.4)
\rput(1.35,2.4){$e_5$}

\psline[linewidth=1.6pt,arrowinset=0]{|->}(-.4,1)(0,1)
\rput(-0.55,1){$e_6$}

\psline[linewidth=1.6pt,arrowinset=0]{|->}(-.4,0)(0,0)
\rput(-0.55,0){$e_7$}

\psline[linewidth=1.6pt,arrowinset=0]{|-<}(1.5,1)(1.2,1)
\psline[linewidth=1.6pt,arrowinset=0]{-}(1.5,1)(1.1,1)
\rput(1.65,1){$e_8$}

\psline[arrowinset=.5,arrowlength=1.5]{<-}(1.1,0)(1.1,1)
\rput(.95,.5){$f_2$}

\psline[arrowinset=.5,arrowlength=1.5]{|->}(1.1,-.4)(1.1,0)
\rput(.8,-0.35){$f_{10}$}

\endpspicture 
\end{array}$
\end{center}

\vspace{.2cm}
\caption{The binet matrix $A_{\bullet V(H)}$ and a $\frac{1}{2}$-binet representation $G(A_{\bullet V(H)})$ of $A_{\bullet V(H)}$, where $V(H)= \{Ê2,10\}$, such that every $2$-edge is directed and the basic half-edge of any $1$-tree with edge index set in $\mathscr{I}$ is entering, with $\mathscr{I}:= \{R_1, Q_3,Q_4,Q_5, Q_6,Q_7\}$.}  
\label{fig:OnehalfbinetnetHH}
\end{figure}

At last, let $S(\mathscr{I})$ be the index set of columns whose support intersects two elements of $\mathscr{I}$. Whenever $A$ is binet, it will be proved that there exists a $\frac{1}{2}$-binet representation $G(A)$ of $A$
such that $S(\mathscr{I})\cup S_2$ is the index set of nonbasic bidirected edges (except half-edges), and for each $R \in \mathscr{I}$, the basic half-edge with index in $R$ is entering. Let $\delta$ be the row vector of size $m$ given by  $\delta_j= \left\{
\begin{array}{ll}
1 & \mbox{ if } j\in S(\mathscr{I})\cup S_2 \\
0 & \mbox{ Otherwise } \\
\end{array}\right.
$ and $A'=\left[ \begin{array}{c}
A \\
\delta 
\end{array} \right]$. 

\begin{figure}[h!]
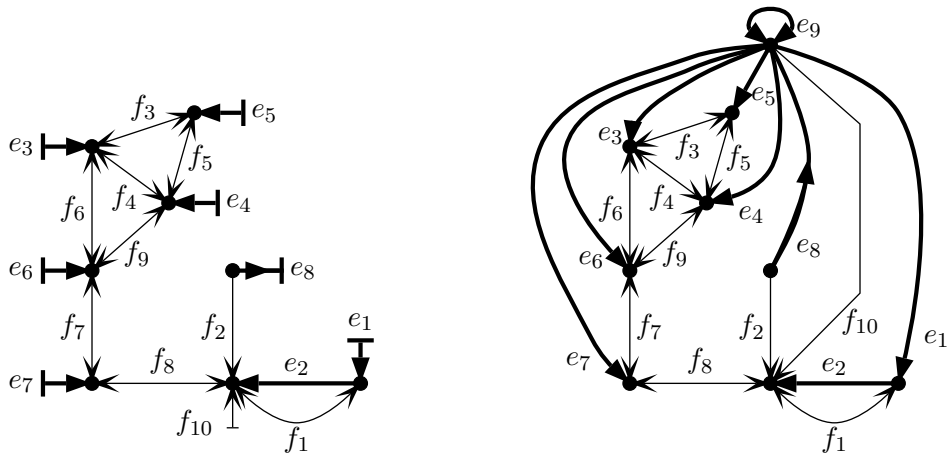

\vspace{1.3cm}
$
\begin{array}{cc}

\psset{xunit=1.7cm,yunit=1.5cm,linewidth=0.5pt,radius=0.1mm,arrowsize=7pt,labelsep=1.5pt,fillcolor=black}

\pspicture(-1,0)(3,3)

\pscircle[fillstyle=solid](0,0){.1}
\pscircle[fillstyle=solid](0,1){.1}
\pscircle[fillstyle=solid](0,2.1){.1}
\pscircle[fillstyle=solid](0.6,1.6){.1}
\pscircle[fillstyle=solid](1.1,0){.1}
\pscircle[fillstyle=solid](1.1,1){.1}
\pscircle[fillstyle=solid](.8,2.4){.1}
\pscircle[fillstyle=solid](2.1,0){.1}

\psline[linewidth=1.6pt,arrowinset=0]{|->}(2.1,.4)(2.1,0)
\rput(2.1,.55){$e_1$}

\psline[linewidth=1.6pt,arrowinset=0]{<-}(1.1,0)(2.1,0)
\rput(1.6,.15){$e_2$}

\psline[linewidth=1.6pt,arrowinset=0]{|->}(-.4,2.1)(0,2.1)
\rput(-0.55,2.1){$e_3$}

\psline[linewidth=1.6pt,arrowinset=0]{<-|}(.6,1.6)(1,1.6)
\rput(1.15,1.6){$e_4$}

\psline[linewidth=1.6pt,arrowinset=0]{<-|}(.8,2.4)(1.2,2.4)
\rput(1.35,2.4){$e_5$}

\psline[linewidth=1.6pt,arrowinset=0]{|->}(-.4,1)(0,1)
\rput(-0.55,1){$e_6$}

\psline[linewidth=1.6pt,arrowinset=0]{|->}(-.4,0)(0,0)
\rput(-0.55,0){$e_7$}

\psline[linewidth=1.6pt,arrowinset=0]{|-<}(1.5,1)(1.2,1)
\psline[linewidth=1.6pt,arrowinset=0]{-}(1.5,1)(1.1,1)
\rput(1.65,1){$e_8$}

\pscurve[arrowinset=.5,arrowlength=1.5]{<->}(1.1,0)(1.6,-.35)(2.1,0)
\rput(1.6,-.5){$f_1$}

\psline[arrowinset=.5,arrowlength=1.5]{<-}(1.1,0)(1.1,1)
\rput(.95,.5){$f_2$}

\psline[arrowinset=.5,arrowlength=1.5]{<->}(0,2.1)(.8,2.4)
\rput(.4,2.45){$f_3$}

\psline[arrowinset=.5,arrowlength=1.5]{<->}(0,2.1)(.6,1.6)
\rput(.25,1.65){$f_4$}

\psline[arrowinset=.5,arrowlength=1.5]{<->}(0.6,1.6)(.8,2.4)
\rput(0.85,2){$f_5$}

\psline[arrowinset=.5,arrowlength=1.5]{<->}(0,1)(0,2.1)
\rput(-.15,1.55){$f_6$}

\psline[arrowinset=.5,arrowlength=1.5]{<->}(0,1)(0,0)
\rput(-.15,0.5){$f_7$}

\psline[arrowinset=.5,arrowlength=1.5]{<->}(0,0)(1.1,0)
\rput(.55,.2){$f_8$}

\psline[arrowinset=.5,arrowlength=1.5]{<->}(0,1)(0.6,1.6)
\rput(.35,1.15){$f_9$}

\psline[arrowinset=.5,arrowlength=1.5]{|->}(1.1,-.4)(1.1,0)
\rput(.8,-0.35){$f_{10}$}

\endpspicture  &

\psset{xunit=1.7cm,yunit=1.5cm,linewidth=0.5pt,radius=0.1mm,arrowsize=7pt,labelsep=1.5pt,fillcolor=black}

\pspicture(-1,0)(3,3)

\pscircle[fillstyle=solid](0,0){.1}
\pscircle[fillstyle=solid](0,1){.1}
\pscircle[fillstyle=solid](0,2.1){.1}
\pscircle[fillstyle=solid](0.6,1.6){.1}
\pscircle[fillstyle=solid](1.1,0){.1}
\pscircle[fillstyle=solid](1.1,1){.1}
\pscircle[fillstyle=solid](.8,2.4){.1}
\pscircle[fillstyle=solid](1.1,3){.1}
\pscircle[fillstyle=solid](2.1,0){.1}

\pscurve[linewidth=1.6pt,arrowinset=0]{<-}(2.1,0)(2.1,2.4)(1.1,3)
\rput(2.4,.4){$e_1$}

\psline[linewidth=1.6pt,arrowinset=0]{<-}(1.1,0)(2.1,0)
\rput(1.6,.15){$e_2$}

\pscurve[linewidth=1.6pt,arrowinset=0]{->}(1.1,3)(.5,2.7)(0.05,2.3)(0,2.1)
\rput(-0.15,2.2){$e_3$}

\pscurve[linewidth=1.6pt,arrowinset=0]{<-}(.6,1.6)(1.1,1.9)(1.1,3)
\rput(.95,1.5){$e_4$}

\psline[linewidth=1.6pt,arrowinset=0]{<-}(.8,2.4)(1.1,3)
\rput(1.05,2.55){$e_5$}

\pscurve[linewidth=1.6pt,arrowinset=0]{->}(1.1,3)(.5,2.8)(-.5,2.1)(0,1)
\rput(-0.3,1.1){$e_6$}

\pscurve[linewidth=1.6pt,arrowinset=0]{->}(1.1,3)(.5,2.9)(-.7,2.1)(-.3,.4)(0,0)
\rput(-0.4,0.2){$e_7$}

\pscurve[linewidth=1.6pt,arrowinset=0]{-}(1.1,1)(1.4,2)(1.1,3)
\pscurve[linewidth=1.6pt,arrowinset=0]{->}(1.1,1)(1.3,1.45)(1.4,2)
\rput(1.4,1.2){$e_8$}

\pscurve[linewidth=1.6pt,arrowinset=0]{<->}(1.1,3)(1.25,3.25)(0.95,3.25)(1.1,3)
\rput(1.4,3.15){$e_9$}

\pscurve[arrowinset=.5,arrowlength=1.5]{<->}(1.1,0)(1.6,-.35)(2.1,0)
\rput(1.6,-.5){$f_1$}

\psline[arrowinset=.5,arrowlength=1.5]{<-}(1.1,0)(1.1,1)
\rput(.95,.5){$f_2$}

\psline[arrowinset=.5,arrowlength=1.5]{<->}(0,2.1)(.8,2.4)
\rput(.45,2.1){$f_3$}

\psline[arrowinset=.5,arrowlength=1.5]{<->}(0,2.1)(.6,1.6)
\rput(.25,1.65){$f_4$}

\psline[arrowinset=.5,arrowlength=1.5]{<->}(0.6,1.6)(.8,2.4)
\rput(0.85,2){$f_5$}

\psline[arrowinset=.5,arrowlength=1.5]{<->}(0,1)(0,2.1)
\rput(-.15,1.55){$f_6$}

\psline[arrowinset=.5,arrowlength=1.5]{<->}(0,1)(0,0)
\rput(.15,0.5){$f_7$}

\psline[arrowinset=.5,arrowlength=1.5]{<->}(0,0)(1.1,0)
\rput(.55,.2){$f_8$}

\psline[arrowinset=.5,arrowlength=1.5]{<->}(0,1)(0.6,1.6)
\rput(.35,1.15){$f_9$}

\psline[arrowinset=.5,arrowlength=1.5]{->}(1.1,3)(1.8,2.3)(1.8,.8)(1.1,0)
\rput(1.8,.55){$f_{10}$}

\endpspicture 
\end{array}$

\vspace{.8cm}
\caption{A $\frac{1}{2}$-binet representation $G(A)$ of $A$ and a $\{n+1\}$-cyclic representation of $A'$.}  
\label{fig:OnehalfbinetA'}
\end{figure}

One claims that $A$ is $\frac{1}{2}$-binet if and only if $A'$ is $\{ n+1 \}$-cyclic. 
Indeed, suppose that $A$ is $\frac{1}{2}$-binet and let $G(A)$
be a $\frac{1}{2}$-binet representation of $A$ as described above. Let us construct a $\{n+1\}$-cyclic representation of $A'$ as follows. Create a new vertex $v$ and a loop entering $v$. Then, replace each half-edge entering (respectively, leaving) a node, say $v'$, by a directed edge $]v,v']$ (respectively, $]v',v]$). 
Conversely, given a $\{ n+1 \}$-cyclic representation of the matrix $A'$, by contracting the basic loop, one obtains a $\frac{1}{2}$-binet representation of $A$. See Figure \ref{fig:OnehalfbinetA'}.

\section{The procedure OnehalfbinetQ}\label{sec:recOnehalfbinet}

In this section, we deal with the general framework of the recognition problem. 
A column index subset $J$ of $A$ is said to be \emph{odd-cyclic}\index{odd-cyclic} if and only if 
the subgraph of $H(A)$ induced by $J$ is
an odd cycle of length at least 5, or a  triangle such that  $s(A_{\bullet j_1})\cap s(A_{\bullet j_2})\cap s(A_{\bullet j_3})=\emptyset$ where $J=\{j_1,j_2,j_3\}$. Observe that in the graph $H(A)$, the vertex set of a triangle may be not odd-cyclic. Before describing a procedure for our recognition problem, we state an auxiliary lemma and the main theorem.

\begin{lem}\label{lemsimponetwo}
If A is a $\{0,1\}$-matrix having a $\frac{1}{2}$-binet representation with exactly one or two basic maximal 1-trees, then $A$ is a network matrix.
\end{lem}

\begin{proof}
Let $G(A)$ be a $\frac{1}{2}$-binet representation of $A$ with exactly one or two basic maximal
$1$-trees. By switching operations if necessary, we may suppose that one basic half-edge of $G(A)$ is entering and the other one (if one exists) is leaving. Since any nonbasic nonhalf and
bidirected $1$-edge would have a non-empty $2$-support, it follows that all edges except half-edges are directed (see Lemmas \ref{lemdefiWeight1} and \ref{lemdefiWeight2} ).

Then create a new vertex $v$ and replace each half-edge  entering (respectively, leaving) a node, say $v'$, by a directed edge from $v$ to $v'$ (respectively, $v'$ to $v$). The obtained bidirected graph is a network representation of $A$. See Figure \ref{fig:vsimplenet}.
\end{proof}\\

\begin{figure}[ht]
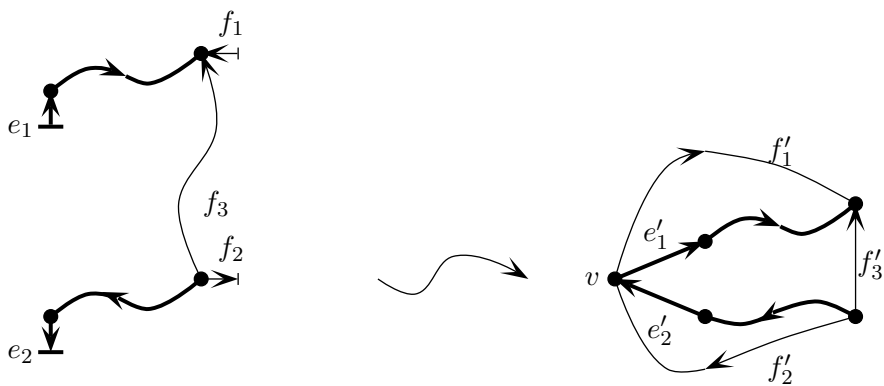


$
\begin{array}{ccc}

\psset{xunit=1cm,yunit=1cm,linewidth=0.5pt,radius=0.1mm,arrowsize=7pt,
labelsep=1.5pt,fillcolor=black}

\pspicture(0,0)(6,4.5)

\pscircle[fillstyle=solid](2,.5){.1}
\pscircle[fillstyle=solid](4,1){.1}

\psline[arrowinset=.5,arrowlength=1.5]{->|}(4,1)(4.5,1)
\rput(4.4,1.4){$f_2$}

\psline[linewidth=1.5pt]{->|}(2,.5)(2,0)

\pscurve[linewidth=1.5pt]{-<}(2,.5)(2.5,.8)(3,.7)
\pscurve[linewidth=1.5pt]{-}(2.9,.75)(3.3,.6)(4,1)

\rput(1.6,0){$e_2$}

\pscircle[fillstyle=solid](2,3.5){.1}
\pscircle[fillstyle=solid](4,4){.1}

\psline[arrowinset=.5,arrowlength=1.5]{<-|}(4,4)(4.5,4)
\rput(4.4,4.4){$f_1$}

\psline[linewidth=1.5pt]{<-|}(2,3.5)(2,3)

\pscurve[linewidth=1.5pt]{->}(2,3.5)(2.5,3.8)(3,3.7)
\pscurve[linewidth=1.5pt]{-}(3,3.7)(3.3,3.6)(4,4)

\rput(1.6,3){$e_1$}

\pscurve[arrowinset=.5,arrowlength=1.5]{<-}(4,4)(4.2,3)(3.7,2)(4,1)
\rput(4.2,2){$f_3$}

\endpspicture &

\psset{xunit=1cm,yunit=1cm,linewidth=0.5pt,radius=0.1mm,arrowsize=7pt,
labelsep=1.5pt,fillcolor=black}

\pspicture(0,0)(2,2)
\pscurve[arrowinset=.5,arrowlength=1.5]{->}(0,1)(.5,0.8)(1,1.3)(2,1)

\endpspicture

 &

\psset{xunit=1cm,yunit=1cm,linewidth=0.5pt,radius=0.1mm,arrowsize=7pt,
labelsep=1.5pt,fillcolor=black}

\pspicture(0,0)(5,3)

\pscircle[fillstyle=solid](0.8,1){.1}

\pscircle[fillstyle=solid](2,0.5){.1}
\pscircle[fillstyle=solid](2,1.5){.1}
\pscircle[fillstyle=solid](4,2){.1}
\pscircle[fillstyle=solid](4,0.5){.1}

\psline[linewidth=1.5pt]{->}(0.8,1)(2,1.5)
\rput(1.35,1.6){$e_1'$}

\pscurve[linewidth=1.5pt]{->}(2,1.5)(2.5,1.8)(3,1.7)
\pscurve[linewidth=1.5pt]{-}(3,1.7)(3.3,1.6)(4,2)

\psline[linewidth=1.5pt]{<-}(0.8,1)(2,0.5)
\rput(1.4,0.32){$e_2'$}

\pscurve[linewidth=1.5pt]{-<}(2,0.5)(2.5,0.4)(3,0.6)
\pscurve[linewidth=1.5pt]{-}(2.9,0.55)(3.5,0.7)(4,0.5)

\rput(0.5,1){$v$}

\psline[arrowinset=.5,arrowlength=1.5]{<-}(4,2)(4,0.5)
\rput(4.2,1.2){$f_3'$}

\pscurve[arrowinset=.5,arrowlength=1.5]{->}(0.8,1)(1.5,2.5)(2,2.7)
\pscurve[arrowinset=.5,arrowlength=1.5]{-}(2,2.7)(3,2.5)(4,2)
\rput(3,2.7){$f_1'$}

\pscurve[arrowinset=.5,arrowlength=1.5]{-}(0.8,1)(1.5,-.2)(2,-.2)
\pscurve[arrowinset=.5,arrowlength=1.5]{->}(4,.5)(3,.2)(2,-.2)
\rput(3,-.2){$f_2'$}

\endpspicture 
\end{array}$

\vspace{1cm}

\caption{How to obtain a network representation of a matrix $A$, given a $\frac{1}{2}$-binet representation of $A$ in which 
all nonbasic edges except half-edges are directed.
An edge $e_i$, respectively $f_j$, is replaced by the directed edge $e_i'$, respectively $f_j'$.}  
\label{fig:vsimplenet}

\end{figure}

\begin{thm}\label{thm:simplenet}
Suppose that $A$ is a $\frac{1}{2}$-binet $\{0,1\}$-matrix.
Then $A$  is a network matrix if and only if each column index subset of $A$ is not odd-cyclic.
\end{thm}

\begin{proof}
Let $G(A)$ be a $\frac{1}{2}$-binet representation of $A$.

Suppose that there exists an odd-cyclic set, say $\{1,\ldots, d\}$, in $H(A)$. From the definition of an odd-cyclic set, it follows that $s(A_{\bullet j})\cap s(A_{\bullet k})\cap s(A_{\bullet l})=\emptyset$ for any triplet $\{j,k,l\}\subseteq \{1,\ldots,d\}$. By a relabeling of the edges we may suppose
that $e_k$ (corresponding to the $k$th row of $A$) is a basic edge contained in the fundamental circuit of $f_k$
and $f_{k+1}$ for  $k=1,\ldots,d-1$, and $e_d$ is in the fundamental circuit of 
$f_d$ and $f_{1}$. Then, since $d$ is odd, the submatrix 

$$
A_{\{1,\ldots, d\}^2}=
\left( \begin{array}{ccccc}
1 & 1 & 0 & \cdots & 0 \\
0 & 1 & 1 & & \vdots \\
\vdots & \ddots & \ddots & \ddots & 0\\
 0 & & & 1 & 1\\
 1 & 0 & \cdots &  0 & 1 
 \end{array} \right)$$

\noindent
of $A$ has a determinant equal to $2$. Thus $A$ is not totally unimodular, and so not a network matrix by Theorem \ref{thmSubclassNNetTot}. 

Conversely, suppose that each column index subset of $A$ is not odd-cyclic. We may assume that $e_1, \ldots,e_r$ are the basic half-edges of $G(A)$ and let 
$L_{G(A)}=(\{e_1,\ldots,e_r\},L)$ be the graph 
such that for $1\le i,i' \le r$, $(e_i,e_{i'})\in L$
if and only if the basic half-edges $e_i$ and $e_{i'}$ in $G(A)$ are in the
fundamental circuit of a $2$-edge.
Since an odd cycle in $L_{G(A)}$ would imply the existence of an odd-cyclic set, it results from the hypothesis that $L_{G(A)}$
is bipartite. Let $B_1\uplus B_2$ be a bipartition of $\{e_1, \ldots,e_r\}$  into two colour classes.
 
Using switching operations, one can construct a new $\frac{1}{2}$-binet representation $G'(A)$ such that all half-edges belonging to $B_1$ (respectively, $B_2$) are entering (respectively, leaving). So, since the fundamental circuit of any $2$-edge contains an entering and a leaving basic half-edge, all $2$-edges are directed. We obtain from $G'(A)$ a network representation of $A$ by
creating a new vertex $v$ and replacing each half-edge by a directed edge incident with $v$ as illustrated in Figure
\ref{fig:vsimplenet}.
\end{proof}\\

The first step of the recognition procedure deals with the columns of $A$ having an entry equal to $2$.
Let $S_2= \{ j\,:\, s_2(A_{\bullet j})\neq \emptyset \}$. 
Denote by $C_1, \ldots, C_{l_0}$ the connected components of $H(A_{\bullet S_2})$ and we note $R_l:=\underset{\small{j\in V(C_l)}}{\cup} s(A_{\bullet j})$
for all $1\le l\le l_0$. In the next lemma, we give the graphical interpretation of these sets.

\begin{lem}\label{lemhalfbinet2Sup}
Suppose that the matrix $A$ has a $\frac{1}{2}$-binet representation $G(A)$. Then the family $\{ R_1,\ldots, R_{l_0} \}$ is $1$-forest.
\end{lem}

\begin{proof}
It is immediate from the definition of $H(A_{\bullet S_2})$ that the sets $R_1, \ldots, R_{l_0}$ are pairwise disjoint. Further,
by Corollary \ref{corBidirectedCircuit} and Lemma \ref{lemdefiWeight1}, for all $j\in S_2$, the nonbasic edge $f_j$ is a $1$-edge whose fundamental circuit contains a basic half-edge.
So for all $1\le l\le l_0$, since $C_l$ is a connected subgraph of $H(A)$, $R_l$ is the edge index set of a $1$-tree. 
\end{proof}\\

Furthermore, we will search for odd-cyclic vertex sets in $H(A_{\bullet \overline{S_2}})$. The following lemma shows that if $A$ is $\frac{1}{2}$-binet, then nonbasic edges with index in an odd-cyclic subset of $\overline{S_2}$ are $2$-edges.

\begin{lem}\label{lemhalfbinet2edge}
Suppose that $A$ has a $\frac{1}{2}$-binet representation $G(A)$. Then, in any $\frac{1}{2}$-binet representation of $A$, each odd-cyclic subset of $\overline{S_2}$ is an index set of $2$-edges. 
\end{lem}

\begin{proof}
Let $S$ be an odd-cyclic subset of $\overline{S_2}$. We may assume  that $S=\{1,\ldots, d\}$. By Theorem \ref{thm:simplenet}, we deduce that
the matrix $A_{\bullet S}$ is not a network matrix. (*)\\
Up to a relabeling of the edges, we may assume that the subgraph of $H(A)$ induced by $\{1,\ldots, d\}$ is the cycle $(1,(1,2),2,...,d,(d,1),1)$. If all $f_j$ ($1\le j \le d$) are $1$-edges, then since $(j,j+1)\in H(A)$ for $j=1,\ldots,d-1$, the basic subgraph of $G(A)$ with edge index set $\cup_{j=1}^{d} s(A_{\bullet j})$ is connected. If we delete all nonbasic edges with index in $\overline{S}$, we obtain a $\frac{1}{2}$-binet representation of $A_{\bullet S}$ with exactly one basic maximal $1$-tree, so using Lemma \ref{lemsimponetwo} that contradicts (*). 

Denote by $f_{j_1},f_{j_2},\ldots, f_{j_t}$ the succession of $2$-edges with index between $1$ and $d$ so that
$1\le j_1< j_2 < \ldots < j_t \le d$. If $t=1$, then using Lemma \ref{lemsimponetwo} it contradicts the observation (*). Assume $t=2$. If $f_{j_1}$ and $f_{j_2}$ have two common basic half-edges in their fundamental circuits, using Lemma \ref{lemsimponetwo} we also obtain a contradiction. Otherwise, let $e_{h_1}$ and $e_{h_2}$ (respectively, $e_{h_2}$ and $e_{h_3}$) be the basic half-edges in the fundamental circuit of $f_{j_1}$ (respectively, $f_{j_2}$). Denote by $T_1$, $T_2$ and $T_3$ the basic maximal $1$-trees containing $e_{h_1}$, $e_{h_2}$ and $e_{h_3}$, respectively.
Then, any $1$-edge $f_j$ with $1\le j < j_1$ (respectively, $j_t < j \le d$) has its endnodes in $T_1$ (respectively, $T_3$). Hence $1$ is not adjacent to $d$ in $H(A)$, except if $1=j_1$ and $d=j_2$. So $1=j_1$ and $d=j_2$. Since all edges $f_j$ with $1< j < d$ are $1$-edges with endnodes in $T_2$, one can construct a network representation of $A_{\bullet S}$, using the same transformation as illustrated in Figure \ref{fig:vsimplenet}. This is in contradiction with (*).
Thus $t\geq 3$.

Since there is a path from $j_1$ to $j_2$ in $H(A)$, all of whose inner vertices are indexes of $1$-edges, it follows that $f_{j_1}$ and $f_{j_2}$ have a common basic half-edge in their fundamental circuits. So $(j_1,j_2) \in H(A)$ and in a same way, one can prove that $(j_1,(j_1,j_2),j_2,\ldots, j_t, (j_t,j_1),j_1)$ is a cycle in $H(A)$. Since the cycle $(1,(1,2),2, \ldots, d, (d,1),1)$ has no chord in $H(A)$, we deduce that $l=d$, which concludes the proof.
\end{proof}\\

In what follows, if $\{ j_1,\ldots, j_t\}$ is an odd-cyclic set, then we will assume that the subgraph of $H(A)$ induced by this  set is equal to $(j_1,(j_1,j_2),j_2, \ldots, j_t,
(j_t,j_1),j_1)$. Denote by $Q_{j_l}=s(A_{\bullet j_l})\cap
s(A_{\bullet j_{l+1}})$ for $l=1,\ldots,t-1$ and $Q_{j_t}=s(A_{\bullet j_t})\cap  
s(A_{\bullet j_{1}})$.

\begin{lem}\label{lemHafbinetoddcyclic}
Suppose that the matrix $A$ has a $\frac{1}{2}$-binet representation $G(A)$. Let $\{ j_1,\ldots, j_t\}\subseteq \overline{S_2}$ be an odd-cyclic set. Then the family $\{Q_{j_1},\ldots, Q_{j_t}\}$ is $1$-forest.
\end{lem}

\begin{proof}
From the definition of an odd-cyclic set, it results that the sets $Q_{j_1},\ldots, Q_{j_t}$ are pairwise disjoint and $t\geq 3$. 
By Lemma \ref{lemhalfbinet2edge}, for all $l=1,\ldots ,t$, $f_{j_l}$ is a 2-edge. So, if the fundamental circuits of $f_{j_1}$ and $f_{j_2}$ for instance share two common basic half-edges, then $s(A_{\bullet j_1})\cap s(A_{\bullet j_2})\cap s(A_{\bullet j_3})\neq \emptyset$ which contradicts the definition of an odd-cyclic set. Thus the fundamental circuits of $f_{j_1}$ and $f_{j_2}$ have exactly one common basic half-edge. Therefore $Q_{j_1}$ is the edge index set of a negative $1$-tree. Similarly, $Q_{j_l}$ is the edge index set of a negative $1$-tree, for $l=2,\ldots,t$.
\end{proof}\\

In the following procedure, since it is hard to find an odd cycle of length at least $5$
in a graph or prove that none exists, odd-cyclic sets 
are discovered by searching for minimal column subsets of $A$ forming a non-network matrix. Odd-cyclic vertex sets are removed out of a subgraph $H$ of $H(A)$, until $H$ has no odd-cyclic subset any more. On the other hand, any row index subset $Q_{j_l}$ (for some column index $j_l$) which does not intersect any element of $\mathscr{I}$ is added in $\mathscr{I}$.

\begin{tabbing}
\textbf{Procedure\,\,S2OddCycle($A$)}\\

\textbf{Input:} A matrix $A$.\\
\textbf{Output:} Either a graph $H$ and a family $\mathscr{I}$ updated, or determines that $A$ is not $\frac{1}{2}$-binet.\\

1)\verb"  "\= let $\mathscr{I} = \{R_1,\ldots,R_{l_0}\}$ and $H=H(A)\verb"\"S_2$;\\

2) \> {\bf while }\= $A_{\bullet V(H)}$ is not a network matrix,  {\bf do}\\
3)\> \> search for a minimal subset $J=\{ j_1,\ldots,j_t\}$ in $V(H)$ such that $A_{\bullet J}$ is not\\
\> \> a network matrix;\\
4)\> \> if $J$ is not odd-cyclic, then STOP: output that $A$ is not $\frac{1}{2}$-binet;\\
5)  \> \> for $l=1,\ldots,t$, if $Q_{j_l}\cap R =  \emptyset$ for all $R\in \mathscr{I}$, then add $Q_{j_l}$ in $\mathscr{I}$; \\ 
6)    \> \> $H=H\verb"\" \{j_1,\ldots,j_t\}$;\\
\> {\bf endwhile }\\
\> output $H$ and $\mathscr{I}$;\\
\end{tabbing}

Suppose that $A$ is $\frac{1}{2}$-binet and the procedure S2OddCycle has output a graph $H$ and a family $\mathscr{I}$.
One idea is to prove the existence of a $\frac{1}{2}$-binet representation of $A$, if one exists, such that all bidirected $2$-edges are known precisely. By assuming that the matrix $A_{\bullet V(H)}$ is decomposed into (connected) blocks, provided that there exist two elements of $\mathscr{I}$ that are row index subsets of a same block, it seems unconvenient to find a $\frac{1}{2}$-binet representation in which all bidirected $2$-edges are known.
For avoiding this case, we will consider a procedure called Path.

For a given $\mathscr{I}$ and a subgraph $H$ of $H(A)$, a vertex $j$ of $H$ is \emph{marked}\index{marked} if and only if there exists at least one element $R$ of $\mathscr{I}$ such that $s(A_{\bullet j})\cap R\neq \emptyset.$ A path from a vertex $j$ to $j'$
in $H$ is called \emph{linking}\index{linking path} if all its nodes are not marked except $j$ and $j'$. Moreover, if $j=j'$, then one requires the existence of two sets $R,R' \in \mathscr{I}$ ($R\neq R'$) such that $s(A_{\bullet j})\cap R\neq \emptyset$ and 
$s(A_{\bullet j})\cap R'\neq \emptyset$, otherwise for any $R,R' \in \mathscr{I}$ such that $s(A_{\bullet j}) \cap R \neq \emptyset$ and $s(A_{\bullet j'}) \cap R'\neq \emptyset$, we have $s(A_{\bullet j'}) \cap R =\emptyset$ and $s(A_{\bullet j}) \cap R'=\emptyset$. Let us see a useful lemma.

\begin{lem}\label{lem:linkingpath}
Suppose that $A$ has a $\frac{1}{2}$-binet representation $G(A)$. Let $\mathscr{I}$ be a $1$-forest family of row index subsets of $A$ and $H$ a subgraph of $H(A)$. Then, all vertices of a minimal linking path in $H$ correspond to indexes of $2$-edges in $G(A)$.
\end{lem} 

\begin{proof}
Let $\gamma=(j_1, (j_1,j_2),j_2,\ldots,j_{t-1},(j_{t-1},j_t),j_t)$ be a minimal linking path in $H$. If $t=1$, the conclusion is clear. Now assume $t\geq 2$. Let $R,R'\in \mathscr{I}$ ($R\neq R'$) such that
$s(A_{\bullet j_1})\cap R\neq \emptyset$ and $s(A_{\bullet j_t})\cap R'\neq \emptyset$. Since $\mathscr{I}$ is a $1$-forest family, we may denote by $T$ and $T'$ the basic $1$-trees in $G(A)$ with edges index sets $R$ and $R'$, respectively. Denote by $e_i$ and $e_{i'}$ the half-edges of $T$ and $T'$, respectively. Since $T\neq T'$ and using the definition of $H$,  $e_i$ as well as $e_{i'}$ have to be contained in the fundamental circuit of a $2$-edge with index in $\gamma$. By definition of a minimal linking path, we deduce that
$f_{j_1}$ is the unique nonbasic edge whose fundamental circuit contains $e_i$, so $f_{j_1}$ is a  $2$-edge. For a same reason, $f_{j_t}$ is a  $2$-edge. 

Now, suppose that there exists a vertex in $\gamma$ corresponding to the index of a $1$-edge in $G(A)$. We may assume that $f_{j_l},f_{j_{l+1}},\ldots, f_{j_{l'}}$ are $1$-edges and $f_{j_{l-1}}$ and $f_{j_{l'+1}}$ are $2$-edges for some $1<l\le l'< t$. It results that 
$j_{l-1}$ and $j_{l'+1}$ are adjacent in $H$, which contradicts the minimality of $\gamma$.
\end{proof}\\

For a path $(j_1, (j_1,j_2),j_2,\ldots,j_{t-1},(j_{t-1},j_t),j_t)$ in $H(A)$, we denote by $Q_{j_l}=s(A_{\bullet j_l})\cap  s(A_{\bullet j_{l+1}})$ for $l=1,\ldots,t-1$.
The procedure Path is stated below.

\begin{tabbing}
\textbf{Procedure\,\,Path($A$,$H$,$\mathscr{I}$)}\\

\textbf{Input:} A matrix $A$, a graph $H \subseteq H(A)$ and a family $\mathscr{I}$ of row index subsets of $A$.\\
\textbf{Output:} A graph $H$ and a family $\mathscr{I}$ updated.\\

1) \= {\bf while }\= there exists a minimal linking path in $H$, {\bf do}\\
2)   \>            \> let $(j_1, (j_1,j_2),j_2,\ldots,j_{t-1},(j_{t-1},j_t),j_t)$ be a minimal linking path in $H$;\\
3)    \> \> for $l=1,\ldots,t-1$, if $Q_{j_l}\cap R =  \emptyset$ for all  $R\in \mathscr{I}$, then add $Q_{j_l}$ in $\mathscr{I}$; \\ 
    \> \> $H=H\verb"\" \{j_1,\ldots,j_t\}$;\\
\> {\bf endwhile }\\
\> output $H$ and $\mathscr{I}$;\\
\end{tabbing}

Finally, for a given family $\mathscr{I}$ of row index subsets of $A$, let $S(\mathscr{I})$ be the set of vertices $j$ in $H(A)$ such that $s(A_{\bullet j})\cap R\neq \emptyset$ and $s(A_{\bullet j})\cap R'\neq \emptyset$ for some $R,R' \in \mathscr{I}$ ($R\neq R'$).
We define the matrix $A'=\left[ \begin{array}{c}
A \\
\delta  \end{array} \right]$, where $\delta$ is a row vector given by  $\delta_j= \left\{
\begin{array}{ll}
1 & \mbox{ if } j\in S(\mathscr{I})\cup S_2 \\
0 & \mbox{ Otherwise } \\
\end{array}\right.
$. Now one can describe a main procedure and prove its correctness.

\begin{tabbing}
\textbf{Procedure\,\,Onehalfbinet($A$)}\\
\textbf{Input:} A  matrix $A$ with $0$, $1$, or $2$ entries.\\
\textbf{Output:} Either a $\frac{1}{2}$-binet representation $G(A)$ of $A$, or determines that none exists.\\
1)\verb"  "\= call {\tt S2OddCycle}($A$) outputing some $H\subseteq H(A)$ and $\mathscr{I}$, or the fact that $A$ is not $\frac{1}{2}$-binet;\\
2)  \> call {\tt Path}($A$,$H$,$\mathscr{I}$) outputing some $H\subseteq H(A)$ and $\mathscr{I}$;\\
3)        \> construct the matrix $A'$ as above and call {\tt RCyclic}($A'$,$\{n+1\}$) of Section \ref{sec:cyc}; \\
\> if we have a $\{n+1\}$-cyclic representation $G(A')$, then go to 4,\\
\>  otherwise STOP: output that $A$ is not $\frac{1}{2}$-binet;\\
4) \> output the bidirected graph $G(A)$ obtained from $G(A')$ by contracting the basic loop; \\

\end{tabbing}

\begin{thm}\label{thmOnehalbinet}
The output of the procedure Onehalfbinet is correct. 
\end{thm}

\begin{proof} 
Suppose that $A$ has a $\frac{1}{2}$-binet representation $G(A)$. Using Theorem \ref{thm:simplenet}, we deduce that the subroutine S2OddCycle does not stop in its step 4. Let 
$\mathscr{I}$ and $H$ be obtained by performing successively the procedures S2OddCycle and Path. By Lemmas \ref{lemhalfbinet2Sup}, \ref{lemHafbinetoddcyclic} and \ref{lem:linkingpath}, it results that $\mathscr{I}$ is $1$-forest. Thus, for any $R \in \mathscr{I}$, $R$ is the edge index set of a basic $1$-tree in $G(A)$. Moreover, by construction, we have the equality 
$$
\{1,\ldots,m\}= S_2 \uplus S(\mathscr{I}) \uplus V(H).
$$
By the subroutine S2OddCycle, the matrix $A_{\bullet V(H)}$ is a network matrix. Moreover, thanks to the subroutine Path, for $R,R'\in \mathscr{I}$ ($R\neq R'$), the basic $1$-trees with edge index sets $R$ and $R'$
are in different connected components of the subgraph $G(A_{\bullet V(H)})$. So, up to switching at all nodes of some basic maximal $1$-trees, we may assume that for all $R\in \mathscr{I}$, the basic half-edge with index in $R$ is entering and all $2$-edges  in $G(A_{\bullet V(H)})$ are directed (see also the proof of Theorem \ref{thm:simplenet}).
Thus, up to switching operations, we may assume that $G(A)$ is a $\frac{1}{2}$-binet representation such that $S(\mathscr{I})$ corresponds to the index set of bidirected $2$-edges and each basic half-edge with index in some $R\in \mathscr{I}$ is entering.

Using $G(A)$, we construct a $\{n+1\}$-cyclic representation of $A'$ as follows. Create a new vertex $v$ and a loop entering $v$ corresponding to the last row of $A'$. Then replace each half-edge which is entering (respectively, leaving) a node say $v'$, by a directed edge from $v$ to $v'$ (respectively, $v'$ to $v$). 
By Theorem \ref{thmcyclicproCyc}, the procedure RCyclic computes a $\{n+1\}$-cyclic representation of $A'$ in step 3.

Whatever input, it is straightforward that the output of the procedure Onehalfbinet (if it does not stop) is a $\frac{1}{2}$-binet representation of $A$.
\end{proof}\\


For any positive integer $q$, we define 

\begin{center}
$
N^{(1)}=
\left( \begin{array}{ccc}
1 & 1 & 0  \\
0 & 1 & 1 \\
1 & 0 & 1 \\
0 & 0 & 0 \\ 
\vdots & \vdots & \vdots \\
0 & 0 & 0 \\ 
\end{array} \right)$ \hspace{1cm}and  \hspace{1cm}
$
N^{(q)}=
\left( \begin{array}{ccccc}
1 & 1 & 0 & \cdots & 0 \\
0 & 1 & 1 & & \vdots \\
\vdots & \ddots & \ddots & \ddots & 0\\
 0 & & & 1 & 1\\
 1 & 0 & \cdots &  0 & 1 \\
0& & \cdots & & 0\\
\vdots& &  & & \vdots\\
0& & \cdots & & 0\\
 \end{array} \right)$,
\end{center}
where $N^{(1)}$ is of size $(n+2)\times 3$, and $N^{(q)}$ of size $(n+1) \times (q+1)$ (respectively, $n\times q$) if $q$ is even (respectively, odd and at least equal to $3$). Then, we define
\begin{center}
$A^{(1)}=\left( \begin{array}{c} 0_{2\times m}\\ A \end{array} \,\, N^{(1)}\right)$\hspace{.7cm} and \hspace{.7cm}
$A^{(q)}=\left( \begin{array}{c} 0_{(q+1)mod 2 \times m}\\ A \end{array} \,\, N^{(q)}\right)$ for $q\geq 2$.
\end{center} The relation between $A$ and $A^{(q)}$ appears in the following lemma.

\begin{lem}\label{lemsimplebinet_q}
The matrix $A$ has a $\frac{1}{2}$-binet representation such that the first
$q$ rows correspond to (basic) half-edges if and only if $A^{(q)}$ has a $\frac{1}{2}$-binet representation.
\end{lem}

\begin{proof}
The "only if" part is not difficult and the "if" part follows from Lemma \ref{lemhalfbinet2edge}.
\end{proof}\\

\noindent
From Lemma \ref{lemsimplebinet_q}, we easily deduce the following procedure that recognizes whether  the matrix $A$ has a $\frac{1}{2}$-binet representation such that each element in $Q$ is a basic half-edge index, where $Q$ is a given row index subset of $A$.

\begin{tabbing}
\textbf{Procedure\,\,OnehalfbinetQ($A$,$Q$)}\\
\textbf{Input:} A  matrix $A$ with entries $0$, $1$, or $2$ 
and a row index subset $Q$.\\
\textbf{Output:} \= A $\frac{1}{2}$-binet representation $G(A)$ such that each element in $Q$ corresponds to\\
\> a basic half-edge index, or determines that none exists.\\
1)\verb"  "\= let $q=|Q|$, $\tilde A=A$ and permute the rows of $\tilde A$  so that those with index in $Q$\\
\>  appear first;\\
2) \> call {\tt Onehalfbinet}($\tilde A^{(q)}$);\\
 \> if the procedure Onehalfbinet outputs a $\frac{1}{2}$-binet representation 
$G(\tilde A^{(q)})$, then let $G(\tilde A)$  \\
\> be the graph obtained from $G(\tilde A^{(q)})$ by deleting the non-basic edges corresponding to \\
\> the columns of $N^{(q)}$ and basic ones corresponding to rows with $n$ first zero entries,\\
\>    otherwise STOP: output that $A$ has no $\frac{1}{2}$-binet representation in which each element\\
\>  of $Q$ corresponds to a basic half-edge index;\\
\> up to a relabeling of basic edges, output a $\frac{1}{2}$-binet representation $G(A)$ of $A$;
 
\end{tabbing}

Finally, we can prove Theorem \ref{thmOnhalfbinetQ}.\\

\noindent
{\bf Proof of Theorem \ref{thmOnhalfbinetQ}.} \quad The correctness of the procedure OnehalfbinetQ follows from Lemma 
\ref{lemsimplebinet_q} and Theorem \ref{thmOnehalbinet}. We know by Theorem \ref{thmSubclassNTutteCunNet}
that determining whether a given matrix of size $n\times m$ with $\beta$ nonzero entries is a network matrix takes time $O(n \beta)$. Thus step 3 in the subroutine S2OddCycle works in time $O(nm\alpha)$. Then, the number of passages through step 2 in the same subroutine does not exceed $m$. Hence the subroutine S2OddCycle takes time $O(nm^2\alpha)$.
It is not difficult to see that the subroutine Path takes time at most $O(n m^2)$, and by Theorem \ref{thmcyclicproCyc}, the procedure RCyclic works in time $O(n m \alpha)$. This concludes the proof.
{\hfill$\BBox{\rule{.3mm}{3mm}}$} \\


\chapter{Recognizing bicyclic matrices}\label{ch:bicyc}

In this Chapter, we provide a characterization of nonnegative 
$\frac{1}{2}$-equisupported bicyclic matrices as well as a recognition procedure called Bicyclic for these matrices.
Let $A$ be a connected $\frac{1}{2}$-equisupported
 matrix of size $n\times m$ with entries $0$, $1$, $\frac{1}{2}$ or $2$ and at least one $\frac{1}{2}$-entry. Let $\alpha$ be the number of nonzero elements in $A$.
A proof of the following theorem will be given.

\begin{thm}\label{thmbicyclicpro}
The matrix $A$ can be tested for being bicyclic by the procedure   Bicyclic. The computational effort required is $O(nm \alpha)$.
\end{thm}

If $A$ has a bicyclic representation $G(A)$, then since $A$ is $\frac{1}{2}$-equisupported, using Corollary \ref{corBidirectedCircuit} and Lemma \ref{lemdefiWeight1} it follows that every column with a nonempty $\frac{1}{2}$-support corresponds to a $2$-edge; and the $\frac{1}{2}$-support of a $2$-edge is equal to the edge index set of both basic cycles in $G(A)$. Furthermore, each basic cycle corresponds to a (closed) consistently oriented path.

Let $S_{\frac{1}{2}}=\{j\, :\, s_{\frac{1}{2}}(A_{\bullet j})\neq \emptyset\}$, $R^*=s_{\frac{1}{2}}(A_{\bullet j})$ for any $j\in S_{\frac{1}{2}}$, and $S^*=\{ j \, : \, s(A_{\bullet j}) \cap R^* \neq \emptyset \}$. Let us partition all row indexes of $A$ into subsets $C_1,\ldots, C_r$ called \emph{cells}\index{cell} so that $i$ and $i'$ belong to a same cell if and only if $i\in s(A_{\bullet j})$, $i'\in s(A_{\bullet j'})$, and $j$ and $j'$ are in a same connected component of 
$H(A)\verb"\"S_{\frac{1}{2}}$.  
For $k=1,\ldots, r$, let $$R_k= C_k \cap R^* \,\,\, \m{ and } \,\,\,A_k= A_{C_k \times f(C_k) }. $$ The set $R_k$ ($1\le k \le r$)
is called an \emph{interval}\index{interval}. Up to a renumbering of the cells,
we may note $R_1,\ldots,R_\xi$ the nonempty intervals.
Let $$\mathcal{K}= \{ C_1, \ldots, C_\xi\}\,\,\, \m{ and } \,\,\,F=\{ j \, : \, s(A_{\bullet j}) \subseteq \cup_{k=\xi+1}^r C_k \}.$$
If $A_k^{\frac{1}{2}\rightarrow 1}$ is not a network matrix or $ R_k\neq \cup_{j\in S_{\frac{1}{2}}} s(A_{\bullet j}) \cap C_k$ for some $1\le k \le \xi$, then the cell $C_k$ is called \emph{central}\index{central!cell}. 
For a matrix $N$ and a row index subset $R$ of $N$, an \emph{$R$ $\frac{1}{2}$-binet}\index{binet@$R$ $\frac{1}{2}$-binet!representation} representation of $N$ is a $\frac{1}{2}$-binet representation  of $N$ such that $R$ is the index set of all basic half-edges. A matrix is called \emph{$R$ $\frac{1}{2}$-binet}\index{binet@$R$ $\frac{1}{2}$-binet!matrix} if and only if it has an $R$ $\frac{1}{2}$-binet representation.
Later, we shall study some bipartitions of $\mathcal{K}$, denoted by $\Sigma(\mathcal{K})$,
and construct submatrices $M_I(\Sigma)$ and $M_{II}(\Sigma)$ of $A$ and a matrix $N(\Sigma)$ with respect to $\Sigma(\mathcal{K})$. For $i=I$ and $II$, $R_i^*(\Sigma)$ will denote the row indexes of $M_i(\Sigma)$ belonging to $R^*$.
We shall denote by $n_I$ and $n_{II}$ two particular row indexes of $N(\Sigma)$.
Then a proof of the following theorem will be given.

\begin{thm}\label{thmbicyclicChar}
The matrix $A$ is bicyclic if and only if there exists a bicompatible bipartition $\Sigma(\mathcal{K})$ such that for any bipartition  $\Sigma'(\mathcal{K})$ equivalent to $\Sigma(\mathcal{K})$, the matrix $M_i(\Sigma')$ is $E_i(\Sigma')$-cyclic for $i= I $ and $II$ and $N(\Sigma')$ is $\{n_I,n_{II}\}$ $\frac{1}{2}$-binet. 
\end{thm}

In Section \ref{sec:sketchbicyc}, we provide some intuitions and ideas about Theorems \ref{thmbicyclicpro} and \ref{thmbicyclicChar} using an example. Then a formal proof of these theorems is given in Section \ref{sec:recbicyc}.

\section{An informal sketch of a recognition procedure}\label{sec:sketchbicyc}

Before embarking on the proof of Theorems \ref{thmbicyclicpro} and \ref{thmbicyclicChar}, we describe the graphical idea on which these are based using the example of Figure \ref{fig:bicyclicA}. As explained above, we partition the row indexes of $A$ into \emph{cells} $C_1,\ldots,C_r$ in the following way: two row indexes $i$ and $i'$  belong to a same cell if and only if $i\in s(A_{\bullet j})$, $i'\in s(A_{\bullet j'})$, and $j$ and $j'$ are in a same connected component of $H(A)\verb"\"S_{\frac{1}{2}}$.  
Up to a renumbering of cells, we have
$C_1=\{1,3,4\}$, $C_2=\{6,7\}$, $C_3=\{2,5\}$, $C_4=\{8,9\}$, $C_5=\{ 10 \}$ and $C_6=\{11\}$. Provided that $A$ has a bicyclic representation $G(A)$,
each cell $C_k$ ($1\le k \le r$) corresponds to the edge index set of a (basic) tree or $1$-tree in $G(A)$ which is called a \emph{cell} and denoted as $T(C_k)$ (see Figure \ref{fig:bicyclicCell}); moreover, a cell in $G(A)$ is called \emph{central} if and only if its corresponding edge index set is central.

\begin{figure}[ht!]
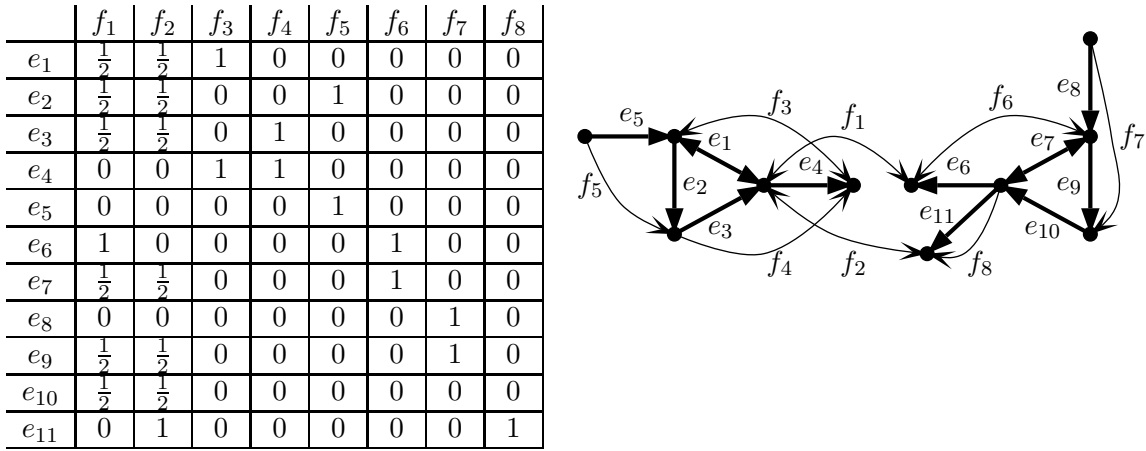

\vspace{.5cm}
$
\begin{array}{cl}

\begin{tabular}{c|c|c|c|c|c|c|c|c|}
  & $f_1$ & $f_2$ & $f_3$ & $f_4$ & $f_5$ & $f_6$ & $f_7$ & $f_8$ \\
  \hline
$e_1$  &$\frac{1}{2}$&$\frac{1}{2}$&1&0&0&0&0&0 \\
\hline
$e_2$  &$\frac{1}{2}$&$\frac{1}{2}$&0&0&1&0&0&0 \\
\hline
$e_3$  &$\frac{1}{2}$&$\frac{1}{2}$&0&1&0&0&0&0 \\
\hline
$e_4$  &0&0&1&1&0&0&0&0 \\
\hline
$e_5$  &0&0&0&0&1&0& 0&0\\
\hline
$e_6$  &1&0&0&0&0&1&0&0 \\
\hline
$e_7$  &$\frac{1}{2}$&$\frac{1}{2}$&0&0&0&1&0&0 \\
\hline
$e_8$  &0&0&0&0&0&0&1&0 \\
\hline
$e_{9}$  &$\frac{1}{2}$&$\frac{1}{2}$&0&0&0&0&1&0 \\
\hline
$e_{10}$ &$\frac{1}{2}$&$\frac{1}{2}$&0&0&0&0&0 &0\\
\hline  
$e_{11}$  &0&1&0&0&0&0&0&1 \\
\hline
\end{tabular}  &

\psset{xunit=1.4cm,yunit=1.3cm,linewidth=0.5pt,radius=0.1mm,arrowsize=7pt,
labelsep=1.5pt,fillcolor=black}

\pspicture(-1,1)(5,2.5)

\pscircle[fillstyle=solid](-.85,2){.1}
\pscircle[fillstyle=solid](0,1){.1}
\pscircle[fillstyle=solid](0,2){.1}
\pscircle[fillstyle=solid](0.85,1.5){.1}
\pscircle[fillstyle=solid](1.7,1.5){.1}
\pscircle[fillstyle=solid](2.25,1.5){.1}
\pscircle[fillstyle=solid](2.4,.8){.1}
\pscircle[fillstyle=solid](3.1,1.5){.1}
\pscircle[fillstyle=solid](3.95,1){.1}
\pscircle[fillstyle=solid](3.95,2){.1}
\pscircle[fillstyle=solid](3.95,3){.1}

\psline[linewidth=1.6pt,arrowinset=0]{<->}(0,2)(0.85,1.5)
\rput(0.44,1.95){$e_1$}

\psline[linewidth=1.6pt,arrowinset=0]{<-}(0,1)(0,2)
\rput(0.2,1.5){$e_2$}

\psline[linewidth=1.6pt,arrowinset=0]{->}(0,1)(0.85,1.5)
\rput(0.44,1.05){$e_3$}

\psline[linewidth=1.6pt,arrowinset=0]{->}(0.85,1.5)(1.7,1.5)
\rput(1.3,1.7){$e_4$}

\psline[linewidth=1.6pt,arrowinset=0]{->}(-.85,2)(0,2)
\rput(-0.4,2.2){$e_5$}

\psline[linewidth=1.6pt,arrowinset=0]{->}(3.1,1.5)(2.25,1.5)
\rput(2.7,1.7){$e_6$}

\psline[linewidth=1.6pt,arrowinset=0]{<->}(3.1,1.5)(3.95,2)
\rput(3.5,1.95){$e_7$}

\psline[linewidth=1.6pt,arrowinset=0]{->}(3.95,3)(3.95,2)
\rput(3.75,2.5){$e_8$}

\psline[linewidth=1.6pt,arrowinset=0]{<-}(3.95,1)(3.95,2)
\rput(3.75,1.5){$e_9$}

\psline[linewidth=1.6pt,arrowinset=0]{<-}(3.1,1.5)(3.95,1)
\rput(3.5,1.05){$e_{10}$}

\psline[linewidth=1.6pt,arrowinset=0]{->}(3.1,1.5)(2.4,.8)
\rput(2.5,1.2){$e_{11}$}


\pscurve[arrowinset=.5,arrowlength=1.5]{<->}(.85,1.5)(1.5,2)(2.25,1.5)
\rput(1.7,2.2){$f_1$}

\pscurve[arrowinset=.5,arrowlength=1.5]{<->}(.85,1.5)(1.5,1)(2.4,.8)
\rput(1.7,.7){$f_2$}

\pscurve[arrowinset=.5,arrowlength=1.5]{<->}(0,2)(.85,2.2)(1.7,1.5)
\rput(1,2.35){$f_3$}

\pscurve[arrowinset=.5,arrowlength=1.5]{->}(0,1)(.85,.8)(1.7,1.5)
\rput(1,.7){$f_4$}

\pscurve[arrowinset=.5,arrowlength=1.5]{->}(-.85,2)(-.5,1.3)(0,1)
\rput(-0.8,1.5){$f_5$}

\pscurve[arrowinset=.5,arrowlength=1.5]{<->}(2.25,1.5)(3.1,2.2)(3.95,2)
\rput(3.1,2.4){$f_{6}$}

\pscurve[arrowinset=.5,arrowlength=1.5]{->}(3.95,3)(4.15,2)(4.2,1.3)(3.95,1)
\rput(4.35,2){$f_{7}$}

\pscurve[arrowinset=.5,arrowlength=1.5]{->}(3.1,1.5)(2.75,.8)(2.4,.8)
\rput(2.9,.7){$f_{8}$}

\endpspicture

\end{array}$
\vspace{.5cm}

\caption{A bicyclic matrix $A$ such that $R^*=\{1,2,3,7,9,10\}$, $S_{\frac{1}{2}}=\{ 1,2 \}$, and a binet representation $G(A)$ of $A$.}
\label{fig:bicyclicA}

\end{figure}

\begin{figure}[ht!]
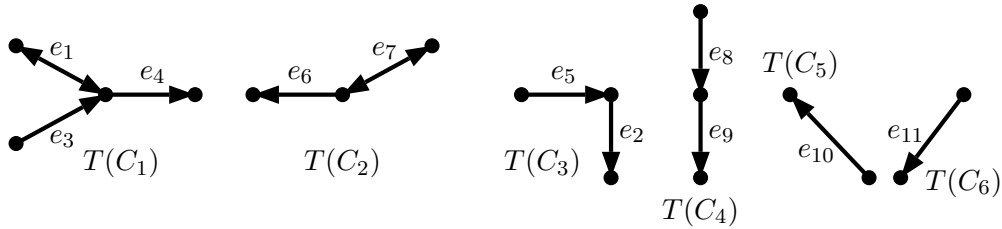


\begin{center}
\psset{xunit=1.4cm,yunit=1.3cm,linewidth=0.5pt,radius=0.1mm,arrowsize=7pt,
labelsep=1.5pt,fillcolor=black}

\pspicture(0,0)(9,2.5)

\pscircle[fillstyle=solid](0,1){.1}
\pscircle[fillstyle=solid](0,2){.1}
\pscircle[fillstyle=solid](0.85,1.5){.1}
\pscircle[fillstyle=solid](1.7,1.5){.1}

\pscircle[fillstyle=solid](2.25,1.5){.1}
\pscircle[fillstyle=solid](3.1,1.5){.1}
\pscircle[fillstyle=solid](3.95,2){.1}

\pscircle[fillstyle=solid](4.8,1.5){.1}
\pscircle[fillstyle=solid](5.65,1.5){.1}
\pscircle[fillstyle=solid](5.65,.65){.1}

\pscircle[fillstyle=solid](6.5,1.5){.1}
\pscircle[fillstyle=solid](6.5,2.35){.1}
\pscircle[fillstyle=solid](6.5,.65){.1}

\pscircle[fillstyle=solid](7.35,1.5){.1}
\pscircle[fillstyle=solid](8.1,.65){.1}

\pscircle[fillstyle=solid](8.4,.65){.1}
\pscircle[fillstyle=solid](9,1.5){.1}

\psline[linewidth=1.6pt,arrowinset=0]{<->}(0,2)(0.85,1.5)
\rput(0.44,1.95){$e_1$}

\psline[linewidth=1.6pt,arrowinset=0]{->}(0,1)(0.85,1.5)
\rput(0.44,1.05){$e_3$}

\psline[linewidth=1.6pt,arrowinset=0]{->}(0.85,1.5)(1.7,1.5)
\rput(1.3,1.7){$e_4$}

\rput(1,.8){$T(C_1)$}

\psline[linewidth=1.6pt,arrowinset=0]{<-}(2.25,1.5)(3.1,1.5)
\rput(2.7,1.7){$e_6$}

\psline[linewidth=1.6pt,arrowinset=0]{<->}(3.1,1.5)(3.95,2)
\rput(3.5,1.95){$e_7$}

\rput(3.1,.8){$T(C_2)$}

\psline[linewidth=1.6pt,arrowinset=0]{->}(5.65,1.5)(5.65,.65)
\rput(5.85,1.1){$e_2$}

\psline[linewidth=1.6pt,arrowinset=0]{->}(4.8,1.5)(5.65,1.5)
\rput(5.2,1.7){$e_5$}

\rput(5,.8){$T(C_3)$}

\psline[linewidth=1.6pt,arrowinset=0]{->}(6.5,2.35)(6.5,1.5)
\rput(6.7,1.9){$e_8$}

\psline[linewidth=1.6pt,arrowinset=0]{->}(6.5,1.5)(6.5,.65)
\rput(6.7,1.1){$e_9$}

\rput(6.5,.3){$T(C_4)$}

\psline[linewidth=1.6pt,arrowinset=0]{->}(8.1,.65)(7.35,1.5)
\rput(7.6,.9){$e_{10}$}

\rput(7.45,1.8){$T(C_5)$}

\psline[linewidth=1.6pt,arrowinset=0]{->}(9,1.5)(8.4,.65)
\rput(8.45,1.1){$e_{11}$}

\rput(9,.6){$T(C_6)$}

\endpspicture

\end{center}

\caption{The different cells in the bicyclic representation $G(A)$ of the matrix $A$ given in Figure \ref{fig:bicyclicA}.}
\label{fig:bicyclicCell}

\end{figure}

Among all cells given in Figure \ref{fig:bicyclicCell}, what are the central ones? We have $R_1= C_1 \cap R^*=\{ 1,3\}$, $R_2=C_2 \cap R^* =\{ 7 \}$ and 

$$ A_1^{\frac{1}{2} \rightarrow 1}= 
\begin{array}{c|c|c|c|c|}
 & f_1 & f_2 & f_3 & f_4 \\ 
\hline
e_1 & 1 & 1 & 1 & 0  \\
\hline
e_3 & 1 & 1 & 0 & 1  \\
\hline
e_4 & 0 & 0 & 1 & 1 \\
\hline
\end{array} \m{\quad , \quad }
A_2^{\frac{1}{2} \rightarrow 1}= 
\begin{array}{c|c|c|c|}
 & f_1 & f_2 & f_6 \\
\hline
e_6 & 1 & 0 & 1 \\
\hline
e_7 & 1 & 1 & 1   \\
\hline
\end{array}
 $$

\noindent
The matrix $A_1^{\frac{1}{2} \rightarrow 1}$ has a submatrix of determinant $2$, so it is not a network matrix. On the contrary, $A_2^{\frac{1}{2} \rightarrow 1}$ 
is a network matrix, but $R_2=\{7 \} \neq \{6,7\}= \cup_{j\in S_{\frac{1}{2}}} s(A_{\bullet j}) \cap C_2$. Hence the cells $C_1$ and $C_2$ are central, and one can verify that the other ones are not. One can prove that this implies the following. Provided that $A$ has a bicyclic representation $G(A)$, $T(C_1)$ and $T(C_2)$ contain each one at least one central edge.
Indeed, $e_1,e_3\in T(C_1)$ and $e_2\in T(C_2)$
(see Figures \ref{fig:bicyclicA} and \ref{fig:bicyclicCell}).
Nevertherless, a non-central cell in $G(A)$ might contain central edges, as evidenced by $T(C_5)$. 

Once the procedure Bicyclic has located all central cells, the goal is to split up all cells into two groups, say $\mathcal{K}_I$ and $\mathcal{K}_{II}$, satisfying the following condition. There exists a bicyclic representation $G(A)$ of $A$ in which $T_I$ and $T_{II}$ are the basic $1$-trees, and any cell in $\mathcal{K}_i$ corresponds to the edge index set of a subgraph of $T_i$ for $i=I$ and $II$. One difficulty is that the number of partitions of $\mathcal{K}$ may be huge. However, we observe in Figure \ref{fig:bicyclicA}, that it is possible to "move" a non-central cell, for instance $T(C_5)$, from one basic maximal $1$-tree to the other as illustrated in Figure \ref{fig:bicyclicmove}, in order to obtain a new bicyclic representation of $A$. This motivates a notion of "equivalent" bipartitions for reducing the number of bipartitions that have to be considered.

\begin{figure}[ht!]
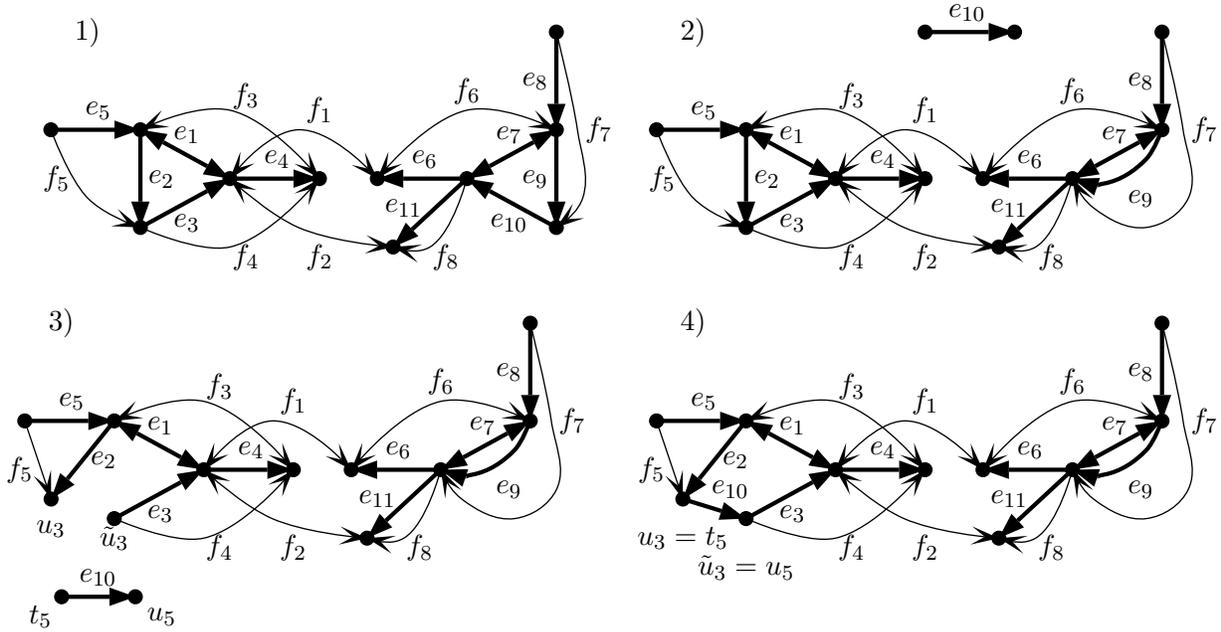

\vspace{1.6cm}
$
\begin{array}{cc}

\psset{xunit=1.4cm,yunit=1.3cm,linewidth=0.5pt,radius=0.1mm,arrowsize=7pt,
labelsep=1.5pt,fillcolor=black}

\pspicture(-1,0)(5,2.5)

\pscircle[fillstyle=solid](-.85,2){.1}
\pscircle[fillstyle=solid](0,1){.1}
\pscircle[fillstyle=solid](0,2){.1}
\pscircle[fillstyle=solid](0.85,1.5){.1}
\pscircle[fillstyle=solid](1.7,1.5){.1}
\pscircle[fillstyle=solid](2.25,1.5){.1}
\pscircle[fillstyle=solid](2.4,.8){.1}
\pscircle[fillstyle=solid](3.1,1.5){.1}
\pscircle[fillstyle=solid](3.95,1){.1}
\pscircle[fillstyle=solid](3.95,2){.1}
\pscircle[fillstyle=solid](3.95,3){.1}

\rput(-.5,3){1)}

\psline[linewidth=1.6pt,arrowinset=0]{<->}(0,2)(0.85,1.5)
\rput(0.44,1.95){$e_1$}

\psline[linewidth=1.6pt,arrowinset=0]{<-}(0,1)(0,2)
\rput(0.2,1.5){$e_2$}

\psline[linewidth=1.6pt,arrowinset=0]{->}(0,1)(0.85,1.5)
\rput(0.44,1.05){$e_3$}

\psline[linewidth=1.6pt,arrowinset=0]{->}(0.85,1.5)(1.7,1.5)
\rput(1.3,1.7){$e_4$}

\psline[linewidth=1.6pt,arrowinset=0]{->}(-.85,2)(0,2)
\rput(-0.4,2.2){$e_5$}

\psline[linewidth=1.6pt,arrowinset=0]{->}(3.1,1.5)(2.25,1.5)
\rput(2.7,1.7){$e_6$}

\psline[linewidth=1.6pt,arrowinset=0]{<->}(3.1,1.5)(3.95,2)
\rput(3.5,1.95){$e_7$}

\psline[linewidth=1.6pt,arrowinset=0]{->}(3.95,3)(3.95,2)
\rput(3.75,2.5){$e_8$}

\psline[linewidth=1.6pt,arrowinset=0]{<-}(3.95,1)(3.95,2)
\rput(3.75,1.5){$e_9$}

\psline[linewidth=1.6pt,arrowinset=0]{<-}(3.1,1.5)(3.95,1)
\rput(3.5,1.05){$e_{10}$}

\psline[linewidth=1.6pt,arrowinset=0]{->}(3.1,1.5)(2.4,.8)
\rput(2.5,1.2){$e_{11}$}


\pscurve[arrowinset=.5,arrowlength=1.5]{<->}(.85,1.5)(1.5,2)(2.25,1.5)
\rput(1.7,2.2){$f_1$}

\pscurve[arrowinset=.5,arrowlength=1.5]{<->}(.85,1.5)(1.5,1)(2.4,.8)
\rput(1.7,.7){$f_2$}

\pscurve[arrowinset=.5,arrowlength=1.5]{<->}(0,2)(.85,2.2)(1.7,1.5)
\rput(1,2.35){$f_3$}

\pscurve[arrowinset=.5,arrowlength=1.5]{->}(0,1)(.85,.8)(1.7,1.5)
\rput(1,.7){$f_4$}

\pscurve[arrowinset=.5,arrowlength=1.5]{->}(-.85,2)(-.5,1.3)(0,1)
\rput(-0.8,1.5){$f_5$}

\pscurve[arrowinset=.5,arrowlength=1.5]{<->}(2.25,1.5)(3.1,2.2)(3.95,2)
\rput(3.1,2.4){$f_{6}$}

\pscurve[arrowinset=.5,arrowlength=1.5]{->}(3.95,3)(4.15,2)(4.2,1.3)(3.95,1)
\rput(4.35,2){$f_{7}$}

\pscurve[arrowinset=.5,arrowlength=1.5]{->}(3.1,1.5)(2.75,.8)(2.4,.8)
\rput(2.9,.7){$f_{8}$}

\endpspicture
  &

\psset{xunit=1.4cm,yunit=1.3cm,linewidth=0.5pt,radius=0.1mm,arrowsize=7pt,
labelsep=1.5pt,fillcolor=black}

\pspicture(-0.5,0)(5,2.5)

\pscircle[fillstyle=solid](-.85,2){.1}
\pscircle[fillstyle=solid](0,1){.1}
\pscircle[fillstyle=solid](0,2){.1}
\pscircle[fillstyle=solid](0.85,1.5){.1}
\pscircle[fillstyle=solid](1.7,1.5){.1}
\pscircle[fillstyle=solid](1.7,3){.1}
\pscircle[fillstyle=solid](2.55,3){.1}
\pscircle[fillstyle=solid](2.25,1.5){.1}
\pscircle[fillstyle=solid](2.4,.8){.1}
\pscircle[fillstyle=solid](3.1,1.5){.1}
\pscircle[fillstyle=solid](3.95,2){.1}
\pscircle[fillstyle=solid](3.95,3){.1}

\rput(-.5,3){2)}

\psline[linewidth=1.6pt,arrowinset=0]{<->}(0,2)(0.85,1.5)
\rput(0.44,1.95){$e_1$}

\psline[linewidth=1.6pt,arrowinset=0]{<-}(0,1)(0,2)
\rput(0.2,1.5){$e_2$}

\psline[linewidth=1.6pt,arrowinset=0]{->}(0,1)(0.85,1.5)
\rput(0.44,1.05){$e_3$}

\psline[linewidth=1.6pt,arrowinset=0]{->}(0.85,1.5)(1.7,1.5)
\rput(1.3,1.7){$e_4$}

\psline[linewidth=1.6pt,arrowinset=0]{->}(-.85,2)(0,2)
\rput(-0.4,2.2){$e_5$}

\psline[linewidth=1.6pt,arrowinset=0]{->}(3.1,1.5)(2.25,1.5)
\rput(2.7,1.7){$e_6$}

\psline[linewidth=1.6pt,arrowinset=0]{<->}(3.1,1.5)(3.95,2)
\rput(3.5,1.95){$e_7$}

\psline[linewidth=1.6pt,arrowinset=0]{->}(3.95,3)(3.95,2)
\rput(3.75,2.5){$e_8$}

\pscurve[linewidth=1.6pt,arrowinset=0]{->}(3.95,2)(3.8,1.7)(3.4,1.45)(3.1,1.5)
\rput(3.75,1.3){$e_9$}

\psline[linewidth=1.6pt,arrowinset=0]{->}(1.7,3)(2.55,3)
\rput(2.1,3.2){$e_{10}$}

\psline[linewidth=1.6pt,arrowinset=0]{->}(3.1,1.5)(2.4,.8)
\rput(2.5,1.2){$e_{11}$}


\pscurve[arrowinset=.5,arrowlength=1.5]{<->}(.85,1.5)(1.5,2)(2.25,1.5)
\rput(1.7,2.2){$f_1$}

\pscurve[arrowinset=.5,arrowlength=1.5]{<->}(.85,1.5)(1.5,1)(2.4,.8)
\rput(1.7,.7){$f_2$}

\pscurve[arrowinset=.5,arrowlength=1.5]{<->}(0,2)(.85,2.2)(1.7,1.5)
\rput(1,2.35){$f_3$}

\pscurve[arrowinset=.5,arrowlength=1.5]{->}(0,1)(.85,.8)(1.7,1.5)
\rput(1,.7){$f_4$}

\pscurve[arrowinset=.5,arrowlength=1.5]{->}(-.85,2)(-.5,1.3)(0,1)
\rput(-0.8,1.5){$f_5$}

\pscurve[arrowinset=.5,arrowlength=1.5]{<->}(2.25,1.5)(3.1,2.2)(3.95,2)
\rput(3.1,2.4){$f_{6}$}

\pscurve[arrowinset=.5,arrowlength=1.5]{->}(3.95,3)(4.15,2)(4.2,1.3)(3.8,1)(3.3,1.2)(3.1,1.5)
\rput(4.35,2){$f_{7}$}

\pscurve[arrowinset=.5,arrowlength=1.5]{->}(3.1,1.5)(2.75,.8)(2.4,.8)
\rput(2.9,.7){$f_{8}$}

\endpspicture \\ \\ 

\psset{xunit=1.4cm,yunit=1.3cm,linewidth=0.5pt,radius=0.1mm,arrowsize=7pt,
labelsep=1.5pt,fillcolor=black}

\pspicture(-0.5,0)(5,2.5)

\pscircle[fillstyle=solid](-.85,2){.1}
\pscircle[fillstyle=solid](-.6,1.2){.1}
\pscircle[fillstyle=solid](0,1){.1}
\pscircle[fillstyle=solid](0,2){.1}
\pscircle[fillstyle=solid](0.85,1.5){.1}
\pscircle[fillstyle=solid](1.7,1.5){.1}
\pscircle[fillstyle=solid](-.5,.2){.1}
\pscircle[fillstyle=solid](.2,.2){.1}
\pscircle[fillstyle=solid](2.25,1.5){.1}
\pscircle[fillstyle=solid](2.4,.8){.1}
\pscircle[fillstyle=solid](3.1,1.5){.1}
\pscircle[fillstyle=solid](3.95,2){.1}
\pscircle[fillstyle=solid](3.95,3){.1}

\rput(-.5,3){3)}

\psline[linewidth=1.6pt,arrowinset=0]{<->}(0,2)(0.85,1.5)
\rput(0.44,1.95){$e_1$}

\psline[linewidth=1.6pt,arrowinset=0]{<-}(-.6,1.2)(0,2)
\rput(-0.1,1.6){$e_2$}
\rput(-.6,.9){$ u_3$}
\rput(0,.8){$\tilde u_3$}

\psline[linewidth=1.6pt,arrowinset=0]{->}(0,1)(0.85,1.5)
\rput(0.44,1.05){$e_3$}

\psline[linewidth=1.6pt,arrowinset=0]{->}(0.85,1.5)(1.7,1.5)
\rput(1.3,1.7){$e_4$}

\psline[linewidth=1.6pt,arrowinset=0]{->}(-.85,2)(0,2)
\rput(-0.4,2.2){$e_5$}

\psline[linewidth=1.6pt,arrowinset=0]{->}(3.1,1.5)(2.25,1.5)
\rput(2.7,1.7){$e_6$}

\psline[linewidth=1.6pt,arrowinset=0]{<->}(3.1,1.5)(3.95,2)
\rput(3.5,1.95){$e_7$}

\psline[linewidth=1.6pt,arrowinset=0]{->}(3.95,3)(3.95,2)
\rput(3.75,2.5){$e_8$}

\pscurve[linewidth=1.6pt,arrowinset=0]{->}(3.95,2)(3.8,1.7)(3.4,1.45)(3.1,1.5)
\rput(3.75,1.3){$e_9$}

\psline[linewidth=1.6pt,arrowinset=0]{->}(-.5,.2)(.2,.2)
\rput(-.15,.4){$e_{10}$}
\rput(-.7,0){$t_5$}
\rput(.45,0){$u_5$}

\psline[linewidth=1.6pt,arrowinset=0]{->}(3.1,1.5)(2.4,.8)
\rput(2.5,1.2){$e_{11}$}


\pscurve[arrowinset=.5,arrowlength=1.5]{<->}(.85,1.5)(1.5,2)(2.25,1.5)
\rput(1.7,2.2){$f_1$}

\pscurve[arrowinset=.5,arrowlength=1.5]{<->}(.85,1.5)(1.5,1)(2.4,.8)
\rput(1.7,.7){$f_2$}

\pscurve[arrowinset=.5,arrowlength=1.5]{<->}(0,2)(.85,2.2)(1.7,1.5)
\rput(1,2.35){$f_3$}

\pscurve[arrowinset=.5,arrowlength=1.5]{->}(0,1)(.85,.8)(1.7,1.5)
\rput(1,.7){$f_4$}

\psline[arrowinset=0.5,arrowlength=1.5]{->}(-.85,2)(-.6,1.2)
\rput(-0.9,1.5){$f_5$}

\pscurve[arrowinset=.5,arrowlength=1.5]{<->}(2.25,1.5)(3.1,2.2)(3.95,2)
\rput(3.1,2.4){$f_{6}$}

\pscurve[arrowinset=.5,arrowlength=1.5]{->}(3.95,3)(4.15,2)(4.2,1.3)(3.8,1)(3.3,1.2)(3.1,1.5)
\rput(4.35,2){$f_{7}$}

\pscurve[arrowinset=.5,arrowlength=1.5]{->}(3.1,1.5)(2.75,.8)(2.4,.8)
\rput(2.9,.7){$f_{8}$}

\endpspicture  &

\psset{xunit=1.4cm,yunit=1.3cm,linewidth=0.5pt,radius=0.1mm,arrowsize=7pt,
labelsep=1.5pt,fillcolor=black}

\pspicture(-0.5,0)(5,2.5)

\pscircle[fillstyle=solid](-.85,2){.1}
\pscircle[fillstyle=solid](-.6,1.2){.1}
\pscircle[fillstyle=solid](0,1){.1}
\pscircle[fillstyle=solid](0,2){.1}
\pscircle[fillstyle=solid](0.85,1.5){.1}
\pscircle[fillstyle=solid](1.7,1.5){.1}
\pscircle[fillstyle=solid](2.25,1.5){.1}
\pscircle[fillstyle=solid](2.4,.8){.1}
\pscircle[fillstyle=solid](3.1,1.5){.1}
\pscircle[fillstyle=solid](3.95,2){.1}
\pscircle[fillstyle=solid](3.95,3){.1}

\rput(-.5,3){4)}

\psline[linewidth=1.6pt,arrowinset=0]{<->}(0,2)(0.85,1.5)
\rput(0.44,1.95){$e_1$}

\psline[linewidth=1.6pt,arrowinset=0]{<-}(-.6,1.2)(0,2)
\rput(-0.1,1.6){$e_2$}
\rput(-.6,.8){$u_3=t_5$}
\rput(0,.5){$\tilde u_3=u_5$}

\psline[linewidth=1.6pt,arrowinset=0]{->}(0,1)(0.85,1.5)
\rput(0.44,1.05){$e_3$}

\psline[linewidth=1.6pt,arrowinset=0]{->}(0.85,1.5)(1.7,1.5)
\rput(1.3,1.7){$e_4$}

\psline[linewidth=1.6pt,arrowinset=0]{->}(-.85,2)(0,2)
\rput(-0.4,2.2){$e_5$}

\psline[linewidth=1.6pt,arrowinset=0]{->}(3.1,1.5)(2.25,1.5)
\rput(2.7,1.7){$e_6$}

\psline[linewidth=1.6pt,arrowinset=0]{<->}(3.1,1.5)(3.95,2)
\rput(3.5,1.95){$e_7$}

\psline[linewidth=1.6pt,arrowinset=0]{->}(3.95,3)(3.95,2)
\rput(3.75,2.5){$e_8$}

\pscurve[linewidth=1.6pt,arrowinset=0]{->}(3.95,2)(3.8,1.7)(3.4,1.45)(3.1,1.5)
\rput(3.75,1.3){$e_9$}

\psline[linewidth=1.6pt,arrowinset=0]{->}(-.6,1.2)(0,1)
\rput(-.15,1.3){$e_{10}$}

\psline[linewidth=1.6pt,arrowinset=0]{->}(3.1,1.5)(2.4,.8)
\rput(2.5,1.2){$e_{11}$}


\pscurve[arrowinset=.5,arrowlength=1.5]{<->}(.85,1.5)(1.5,2)(2.25,1.5)
\rput(1.7,2.2){$f_1$}

\pscurve[arrowinset=.5,arrowlength=1.5]{<->}(.85,1.5)(1.5,1)(2.4,.8)
\rput(1.7,.7){$f_2$}

\pscurve[arrowinset=.5,arrowlength=1.5]{<->}(0,2)(.85,2.2)(1.7,1.5)
\rput(1,2.35){$f_3$}

\pscurve[arrowinset=.5,arrowlength=1.5]{->}(0,1)(.85,.8)(1.7,1.5)
\rput(1,.7){$f_4$}

\psline[arrowinset=0.5,arrowlength=1.5]{->}(-.85,2)(-.6,1.2)
\rput(-0.9,1.5){$f_5$}

\pscurve[arrowinset=.5,arrowlength=1.5]{<->}(2.25,1.5)(3.1,2.2)(3.95,2)
\rput(3.1,2.4){$f_{6}$}

\pscurve[arrowinset=.5,arrowlength=1.5]{->}(3.95,3)(4.15,2)(4.2,1.3)(3.8,1)(3.3,1.2)(3.1,1.5)
\rput(4.35,2){$f_{7}$}

\pscurve[arrowinset=.5,arrowlength=1.5]{->}(3.1,1.5)(2.75,.8)(2.4,.8)
\rput(2.9,.7){$f_{8}$}

\endpspicture 

\end{array}$

\vspace{.5cm}

\caption{ How to move a non-central cell in a bicyclic representation $G(A)$ of $A$ from one basic $1$-tree to the other,  where $A$ and $G(A)$ are given in Figure \ref{fig:bicyclicA} (see pictures from 1 to 4). The vertices $u_3$, $\tilde u_3$, $u_5$ and $t_5$ are defined in the proof of Proposition \ref{propbicyclicChar}.}
\label{fig:bicyclicmove}

\end{figure}

Given a bipartition $\Sigma(\mathcal{K})=\{ \mathcal{K}_I, \mathcal{K}_{II} \}$ of $\mathcal{K}$ with $\mathcal{K}_I= \{ C_1, C_3\}$ and $\mathcal{K}_{II}= \{ C_2, C_4,C_5 \}$ for instance, we define $E_i(\Sigma)= \cup_{C_k \in \mathcal{K}_i} C_k$, $R_i^*(\Sigma)=E_i(\Sigma)\cap R^*$
and $M_i(\Sigma) = A_{E_i(\Sigma)\, f(E_i(\Sigma)) }$, for $i= I $ and $II$, as well as a matrix $N(\Sigma)$ having two particular row indexes denoted by $n_I$ and $n_{II}$. See Figure \ref{fig:bicyclicmatrices}. Then one computes if possible an $E_i(\Sigma)$-cyclic representation of $M_i(\Sigma)$, for $i=I$ and $II$, and a $\frac{1}{2}$-binet representation of $N(\Sigma)$ whose basic half-edges are $e_{n_I}$ and $e_{n_{II}}$. Provided that all these representations have been found, one can construct a bicyclic representation of $A$ (see Figure \ref{fig:bicyclicmatnet}).

\begin{figure}[h!]

\begin{center}
$
\begin{array}{ccccc}
M_{I}(\Sigma)=
\begin{tabular}{c|c|c|c|c|c|}
  & $f_1^I$ & $f_2^I$ & $f_3$ & $f_4$ & $f_5$ \\
  \hline
$e_1$  &$\frac{1}{2}$&$\frac{1}{2}$&1&0&0 \\
\hline
$e_2$  &$\frac{1}{2}$&$\frac{1}{2}$&0&0&1 \\
\hline
$e_3$  &$\frac{1}{2}$&$\frac{1}{2}$&0&1&0 \\
\hline
$e_4$  &0&0&1&1&0 \\
\hline
$e_5$  &0&0&0&0&1\\
\hline
\end{tabular}  &&&&
M_{II}(\Sigma)=
\begin{tabular}{c|c|c|c|c|}
  & $f_1^{II}$ & $f_2^{II}$ & $f_6$ & $f_7$ \\
  \hline
$e_6$  &1&0&1&0 \\
\hline
$e_7$  &$\frac{1}{2}$&$\frac{1}{2}$&1&0 \\
\hline
$e_8$  &0&0&0&1 \\
\hline
$e_{9}$  &$\frac{1}{2}$&$\frac{1}{2}$&0&1 \\
\hline
$e_{10}$ &$\frac{1}{2}$&$\frac{1}{2}$&0&0\\
\hline
\end{tabular} 
\end{array}
$
\end{center}

\begin{center}
$
N(\Sigma)=
\begin{tabular}{c|c|c|c|c|c|c|c|}
  & $f_1^I$ & $f_2^I$ & $f_1^{II}$ & $f_2^{II}$ & $f_1$ & $f_2$ & $f_8$ \\
  \hline
$e_6$  &0&0&1&0 &1&0&0 \\
\hline
$e_{11}$  &0&0&0&0&0&1&1\\
\hline
$e_{n_I}$  &1&1&0&0&1&1&0 \\
\hline
$e_{n_{II}}$ &0&0&1&1&1&1&0\\
\hline
\end{tabular}
$
\end{center}

\caption{An example of the matrices $M_I(\Sigma)$, $M_{II}(\Sigma)$ and $N(\Sigma)$ where $A$ is given in Figure \ref{fig:bicyclicA}, $\Sigma(\mathcal{K})=\{ \mathcal{K}_I, \mathcal{K}_{II} \}$ with $\mathcal{K}_I=\{ C_1,C_3\}$ and $\mathcal{K}_{II}=\{C_2, C_4, C_5 \}$.}
\label{fig:bicyclicmatrices}

\end{figure}

\begin{figure}[h!]
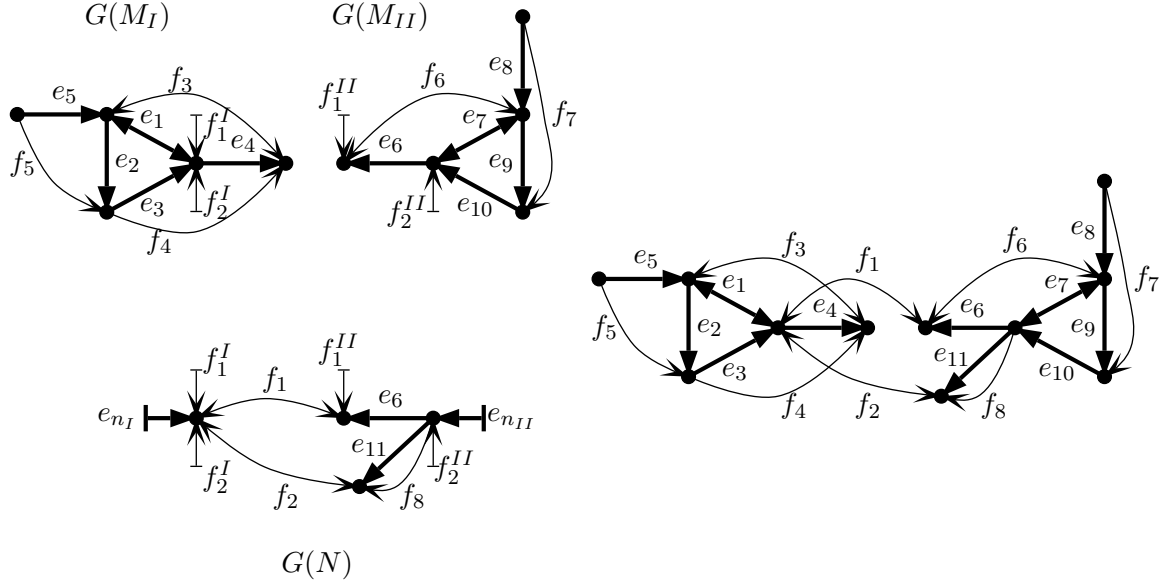


\vspace{1.5cm}

$
\begin{array}{cc}
\begin{array}{c}

\psset{xunit=1.4cm,yunit=1.3cm,linewidth=0.5pt,radius=0.1mm,arrowsize=7pt,
labelsep=1.5pt,fillcolor=black}

\pspicture(-1,0)(5,2.5)

\pscircle[fillstyle=solid](-.85,2){.1}
\pscircle[fillstyle=solid](0,1){.1}
\pscircle[fillstyle=solid](0,2){.1}
\pscircle[fillstyle=solid](0.85,1.5){.1}
\pscircle[fillstyle=solid](1.7,1.5){.1}
\pscircle[fillstyle=solid](2.25,1.5){.1}
\pscircle[fillstyle=solid](3.1,1.5){.1}
\pscircle[fillstyle=solid](3.95,1){.1}
\pscircle[fillstyle=solid](3.95,2){.1}
\pscircle[fillstyle=solid](3.95,3){.1}

\rput(0.2,3){$G(M_{I})$}
\rput(2.6,3){$G(M_{II})$}

\psline[linewidth=1.6pt,arrowinset=0]{<->}(0,2)(0.85,1.5)
\rput(0.44,1.95){$e_1$}

\psline[linewidth=1.6pt,arrowinset=0]{<-}(0,1)(0,2)
\rput(0.2,1.5){$e_2$}

\psline[linewidth=1.6pt,arrowinset=0]{->}(0,1)(0.85,1.5)
\rput(0.44,1.05){$e_3$}

\psline[linewidth=1.6pt,arrowinset=0]{->}(0.85,1.5)(1.7,1.5)
\rput(1.3,1.7){$e_4$}

\psline[linewidth=1.6pt,arrowinset=0]{->}(-.85,2)(0,2)
\rput(-0.4,2.2){$e_5$}

\psline[linewidth=1.6pt,arrowinset=0]{->}(3.1,1.5)(2.25,1.5)
\rput(2.7,1.7){$e_6$}

\psline[linewidth=1.6pt,arrowinset=0]{<->}(3.1,1.5)(3.95,2)
\rput(3.5,1.95){$e_7$}

\psline[linewidth=1.6pt,arrowinset=0]{->}(3.95,3)(3.95,2)
\rput(3.75,2.5){$e_8$}

\psline[linewidth=1.6pt,arrowinset=0]{<-}(3.95,1)(3.95,2)
\rput(3.75,1.5){$e_9$}

\psline[linewidth=1.6pt,arrowinset=0]{<-}(3.1,1.5)(3.95,1)
\rput(3.5,1.05){$e_{10}$}



\psline[arrowinset=.5,arrowlength=1.5]{<-|}(.85,1.5)(.85,2)
\rput(1.05,1.9){$f_1^I$}

\psline[arrowinset=.5,arrowlength=1.5]{|->}(2.25,2)(2.25,1.5)
\rput(2.2,2.2){$f_1^{II}$}

\psline[arrowinset=.5,arrowlength=1.5]{<-|}(.85,1.5)(.85,1)
\rput(1.05,1.1){$f_2^{I}$}

\psline[arrowinset=.5,arrowlength=1.5]{|->}(3.1,1)(3.1,1.5)
\rput(2.85,1){$f_2^{II}$}

\pscurve[arrowinset=.5,arrowlength=1.5]{<->}(0,2)(.85,2.2)(1.7,1.5)
\rput(.7,2.35){$f_3$}

\pscurve[arrowinset=.5,arrowlength=1.5]{->}(0,1)(.85,.8)(1.7,1.5)
\rput(.5,.7){$f_4$}

\pscurve[arrowinset=.5,arrowlength=1.5]{->}(-.85,2)(-.5,1.3)(0,1)
\rput(-0.8,1.5){$f_5$}

\pscurve[arrowinset=.5,arrowlength=1.5]{<->}(2.25,1.5)(3.1,2.2)(3.95,2)
\rput(3.1,2.4){$f_{6}$}

\pscurve[arrowinset=.5,arrowlength=1.5]{->}(3.95,3)(4.15,2)(4.2,1.3)(3.95,1)
\rput(4.35,2){$f_{7}$}


\endpspicture \\

\psset{xunit=1.4cm,yunit=1.3cm,linewidth=0.5pt,radius=0.1mm,arrowsize=7pt,
labelsep=1.5pt,fillcolor=black}

\pspicture(-1,0)(5,2.5)

\pscircle[fillstyle=solid](0.85,1.5){.1}
\pscircle[fillstyle=solid](2.25,1.5){.1}
\pscircle[fillstyle=solid](2.4,.8){.1}
\pscircle[fillstyle=solid](3.1,1.5){.1}

\rput(2,0){$G(N)$}

\psline[linewidth=1.6pt,arrowinset=0]{|->}(0.35,1.5)(0.85,1.5)
\rput(0.1,1.5){$e_{n_I}$}

\psline[linewidth=1.6pt,arrowinset=0]{->}(3.1,1.5)(2.25,1.5)
\rput(2.7,1.7){$e_6$}

\psline[linewidth=1.6pt,arrowinset=0]{<-|}(3.1,1.5)(3.6,1.5)
\rput(3.85,1.5){$e_{n_{II}}$}

\psline[linewidth=1.6pt,arrowinset=0]{->}(3.1,1.5)(2.4,.8)
\rput(2.5,1.2){$e_{11}$}


\psline[arrowinset=.5,arrowlength=1.5]{<-|}(.85,1.5)(.85,2)
\rput(1.05,2.1){$f_1^I$}

\psline[arrowinset=.5,arrowlength=1.5]{|->}(2.25,2)(2.25,1.5)
\rput(2.2,2.2){$f_1^{II}$}

\psline[arrowinset=.5,arrowlength=1.5]{<-|}(.85,1.5)(.85,1)
\rput(1.05,.9){$f_2^{I}$}

\psline[arrowinset=.5,arrowlength=1.5]{|->}(3.1,1)(3.1,1.5)
\rput(3.3,1){$f_2^{II}$}

\pscurve[arrowinset=.5,arrowlength=1.5]{<->}(.85,1.5)(1.5,1.7)(2.25,1.5)
\rput(1.6,1.9){$f_1$}

\pscurve[arrowinset=.5,arrowlength=1.5]{<->}(.85,1.5)(1.5,1)(2.4,.8)
\rput(1.7,.7){$f_2$}

\pscurve[arrowinset=.5,arrowlength=1.5]{->}(3.1,1.5)(2.75,.8)(2.4,.8)
\rput(2.9,.7){$f_{8}$}

\endpspicture

\end{array} &

\psset{xunit=1.4cm,yunit=1.3cm,linewidth=0.5pt,radius=0.1mm,arrowsize=7pt,
labelsep=1.5pt,fillcolor=black}

\pspicture(-0.15,1.5)(5.5,2.5)

\pscircle[fillstyle=solid](-.85,2){.1}
\pscircle[fillstyle=solid](0,1){.1}
\pscircle[fillstyle=solid](0,2){.1}
\pscircle[fillstyle=solid](0.85,1.5){.1}
\pscircle[fillstyle=solid](1.7,1.5){.1}
\pscircle[fillstyle=solid](2.25,1.5){.1}
\pscircle[fillstyle=solid](2.4,.8){.1}
\pscircle[fillstyle=solid](3.1,1.5){.1}
\pscircle[fillstyle=solid](3.95,1){.1}
\pscircle[fillstyle=solid](3.95,2){.1}
\pscircle[fillstyle=solid](3.95,3){.1}

\psline[linewidth=1.6pt,arrowinset=0]{<->}(0,2)(0.85,1.5)
\rput(0.44,1.95){$e_1$}

\psline[linewidth=1.6pt,arrowinset=0]{<-}(0,1)(0,2)
\rput(0.2,1.5){$e_2$}

\psline[linewidth=1.6pt,arrowinset=0]{->}(0,1)(0.85,1.5)
\rput(0.44,1.05){$e_3$}

\psline[linewidth=1.6pt,arrowinset=0]{->}(0.85,1.5)(1.7,1.5)
\rput(1.3,1.7){$e_4$}

\psline[linewidth=1.6pt,arrowinset=0]{->}(-.85,2)(0,2)
\rput(-0.4,2.2){$e_5$}

\psline[linewidth=1.6pt,arrowinset=0]{->}(3.1,1.5)(2.25,1.5)
\rput(2.7,1.7){$e_6$}

\psline[linewidth=1.6pt,arrowinset=0]{<->}(3.1,1.5)(3.95,2)
\rput(3.5,1.95){$e_7$}

\psline[linewidth=1.6pt,arrowinset=0]{->}(3.95,3)(3.95,2)
\rput(3.75,2.5){$e_8$}

\psline[linewidth=1.6pt,arrowinset=0]{<-}(3.95,1)(3.95,2)
\rput(3.75,1.5){$e_9$}

\psline[linewidth=1.6pt,arrowinset=0]{<-}(3.1,1.5)(3.95,1)
\rput(3.5,1.05){$e_{10}$}

\psline[linewidth=1.6pt,arrowinset=0]{->}(3.1,1.5)(2.4,.8)
\rput(2.5,1.2){$e_{11}$}


\pscurve[arrowinset=.5,arrowlength=1.5]{<->}(.85,1.5)(1.5,2)(2.25,1.5)
\rput(1.7,2.2){$f_1$}

\pscurve[arrowinset=.5,arrowlength=1.5]{<->}(.85,1.5)(1.5,1)(2.4,.8)
\rput(1.7,.7){$f_2$}

\pscurve[arrowinset=.5,arrowlength=1.5]{<->}(0,2)(.85,2.2)(1.7,1.5)
\rput(1,2.35){$f_3$}

\pscurve[arrowinset=.5,arrowlength=1.5]{->}(0,1)(.85,.8)(1.7,1.5)
\rput(1,.7){$f_4$}

\pscurve[arrowinset=.5,arrowlength=1.5]{->}(-.85,2)(-.5,1.3)(0,1)
\rput(-0.8,1.5){$f_5$}

\pscurve[arrowinset=.5,arrowlength=1.5]{<->}(2.25,1.5)(3.1,2.2)(3.95,2)
\rput(3.1,2.4){$f_{6}$}

\pscurve[arrowinset=.5,arrowlength=1.5]{->}(3.95,3)(4.15,2)(4.2,1.3)(3.95,1)
\rput(4.35,2){$f_{7}$}

\pscurve[arrowinset=.5,arrowlength=1.5]{->}(3.1,1.5)(2.75,.8)(2.4,.8)
\rput(2.9,.7){$f_{8}$}

\endpspicture

\end{array}$

\caption{An $E_i(\Sigma)$-cyclic representation of $M_i(\Sigma)$ for $i=I$ and $II$, a $\{n_I,n_{II}\}$ $\frac{1}{2}$-binet representation of $N(\Sigma)$ and a bicyclic representation of $A$, where $A$ is given in Figure \ref{fig:bicyclicA} and 
$M_I(\Sigma)$, $M_{II}(\Sigma)$ and $N(\Sigma)$ in Figure \ref{fig:bicyclicmatrices}.}
\label{fig:bicyclicmatnet}

\end{figure}

\section{The procedure Bicyclic}\label{sec:recbicyc}

In this section, we describe the procedure Bicyclic and provide a proof of Theorems \ref{thmbicyclicpro} and \ref{thmbicyclicChar}.
Suppose that $A$ has a bicyclic representation $G(A)$. For all $1\le k \le \xi$, the interval $R_k$ represents the edge index set of a basic loop, or a consistently oriented path,
denoted as $\mathcal{P}_k$, in a basic cycle. 
For every column index $j$, the nonbasic edge $f_j$ is a $2$-edge if and only if $j\in S_{\frac{1}{2}}$. Observe that for all $1\le k \le r$, the cell $C_k$ is the edge index set of a (basic) tree or $1$-tree of $G(A)$, denoted as $T(C_k)$, which is called a \emph{cell}\index{cell}. We denote by $T_I$ and $T_{II}$ the maximal basic $1$-trees of $G(A)$. We say that $G(A)$ \emph{induces the bipartition $\Sigma(\mathcal{K})=\{\mathcal{K}_I, \mathcal{K}_{II} \}$}\index{induce a bipartition}, where $\mathcal{K}_i=\{C_k \, : \, T(C_k) \subseteq T_i\}$ for $i=I$ and $II$.
A cell in $G(A)$ is said to be \emph{central}\index{central!cell} if and only if its edge index set is central. The following lemma gives an enlightenment on the notion of central cell.

\begin{lem}\label{lembicycliccentral}
Suppose that $A$ has a bicyclic representation $G(A)$. Then every central cell in $G(A)$ contains a central edge. 
\end{lem}

\begin{proof}
Let $T(C_k)$ be a central cell, for some $1 \le k \le \xi$, and suppose that it is contained in $T_I$. Since $C_k \cap R^*\neq \emptyset$, we may assume that the basic cycle of $T_I$ is not a loop. Let us denote by $e_1$ and $e_\rho$ the central edges of $T_I$ so that $e_1$ is bidirected. 
We may also assume that $e_1 \notin \mathcal{P}_k$. Hence, $ T(C_k)$ is a directed tree, and no column in $A_k$ has a $2$-entry, because the fundamental circuit of a column with a nonempty $2$-support contains a whole basic cycle. Moreover, the intersection of the fundamental circuit of any $2$-edge with $T(C_k)$ is a consistently oriented path. Therefore, 
$T(C_k)$ is a basic network representation of $A_k^{\frac{1}{2}\rightarrow 1}$. 

From the definition of a central cell, it follows that $\cup_{j\in S_{\frac{1}{2}}} s(A_{\bullet j}) \cap C_k\neq R_k$. So there exists a basic edge $e_i$ and a $2$-edge $f_j$ ($j\in S_{\frac{1}{2}}$), such that $e_i$ is in the fundamental circuit of $f_j$ and $i \in C_k\verb"\"R_k$. Then, since $A$ is nonnegative, by Lemma \ref{lemdefiWeight2}
the stem issued from $f_j$ in $T_I$ contains
the central node of $T_I$ and $e_i$ (see Figure \ref{fig:Apositive}).
As $T(C_k)$ is connected and contains at least one edge of the basic cycle in $T_I$, we deduce that $e_\rho$ belongs to $T(C_k)$.
\end{proof}\\

A bipartition of $\mathcal{K}$ into two subsets $\Sigma(\mathcal{K}) = \{ \mathcal{K}_I , \mathcal{K}_{II} \}$ (where $\mathcal{K}=  \mathcal{K}_I \biguplus \mathcal{K}_{II} $) is said to be \emph{bicompatible}\index{bicompatible bipartition} if $\mathcal{K}_i\neq \emptyset$ and $\mathcal{K}_i$ contains at most two central cells for $i=I$ and $II$. A bipartition  $\Sigma'(\mathcal{K})=\{\mathcal{K}_I', \mathcal{K}_{II}' \}$ is called \emph{equivalent}\index{equivalent bipartitions} to a bipartition $\Sigma(\mathcal{K})= \{ \mathcal{K}_I , \mathcal{K}_{II} \}$, if 
$\Sigma'(\mathcal{K})=\Sigma(\mathcal{K})$ or "$\mathcal{K}_I$ and $\mathcal{K}_{II}$ have each one at least one non-central cell, and one can obtain $\Sigma'$ from $\Sigma$ by moving some non-central cells between $\mathcal{K}_I$ and $\mathcal{K}_{II}$ and keeping at least one non-central cell in each subset". Clearly, this is an equivalence relation.
Denote by $\mathscr{S}$ the quotient set of all bicompatible bipartitions by this equivalence relation.
Given a bipartition $\Sigma(\mathcal{K}) = \{ \mathcal{K}_I , \mathcal{K}_{II} \}$, we define 
$$E_i(\Sigma)= \cup_{C_k \in \mathcal{K}_i} C_k, \quad R_i^*(\Sigma)=E_i(\Sigma)\cap R^*$$ 
and $$ M_i(\Sigma) = A_{E_i(\Sigma)\, f(E_i(\Sigma)) }$$ for $i= I $ and $II$. For $i=I$ and $II$, let $$E_{i,\frac{1}{2}} =\{ i' \in E_i(\Sigma)\verb"\"
R^* \, :\, i' \in s_{\frac{1}{2}} (A_{\bullet j}) \m{ for some } j\in S_{\frac{1}{2}} \}$$ and $E= \cup_{k=\xi+1}^{r} C_k\cup 
E_{I,\frac{1}{2}} \cup E_{II,\frac{1}{2}}$. Moreover, let 

$$N(\Sigma) = \left( \begin{array}{cccc}

\begin{array}{c}
A_{E_{I,\frac{1}{2}} \times S_{\frac{1}{2}} } \\
0_{|\overline{E_{I,\frac{1}{2}}}| \times |S_{\frac{1}{2}}| }
\end{array}  &
\begin{array}{c}
0_{|\overline{E_{II,\frac{1}{2}}}| \times |S_{\frac{1}{2}}| }\\
A_{E_{II,\frac{1}{2}} \times S_{\frac{1}{2}} } \\
\end{array}  &

A_{E \times S_{\frac{1}{2}} } & A_{E \times F } \\

1_{1 \times |S_{\frac{1}{2}}| } & 0_{1 \times |S_{\frac{1}{2}}| } &
1_{1 \times |S_{\frac{1}{2}}| } &
0_{1 \times |F| } \\
0_{1 \times |S_{\frac{1}{2}}| } & 1_{1 \times |S_{\frac{1}{2}}| } &
1_{1 \times |S_{\frac{1}{2}}| } &
0_{1 \times |F| } \\
\end{array}\right), $$

\noindent 
where the rows of $N(\Sigma)$ (except the two last ones) are indexed by the elements in $E$ ($\overline{E_{i,\frac{1}{2}}}= E \verb"\" E_{i,\frac{1}{2}}$ for $i=I$ and $II$).
See Figure \ref{fig:bicyclicmatrices}. Denote by $n_I$ and $n_{II}$ the indexes of the two last rows, respectively. 

\begin{lem}\label{lembicyclicnoncentral}
Suppose that $A$ has a bicyclic representation $G(A)$. For any non-central cell $C_l$ and a cell $C_{l'}$ for some $1\le l \le \xi$ and $1\le l' \le r$, a node belonging to $T(C_l)$ and $T(C_{l'})$ is either equal to a central node, if $R^*=R_l$, or an endnode of the interval with edge index set $R_l$ otherwise.
\end{lem}

\begin{proof}
This follows from the connectivity of $A$, the construction of the cells and the definition of a central cell.
\end{proof}\\


Let us prove the following proposition.

\begin{prop}\label{propbicyclicChar}
If the matrix $A$ is bicyclic, then there exists a bicompatible bipartition $\Sigma(\mathcal{K})$ such that for any bipartition  $\Sigma'(\mathcal{K})$ equivalent to $\Sigma(\mathcal{K})$, the matrix $M_i(\Sigma')$ is $R_i^*(\Sigma')$-cyclic for $i= I $ and $II$ and $N(\Sigma')$ is $\{n_I,n_{II}\}$ $\frac{1}{2}$-binet. 
\end{prop}

\begin{proof}
Let $G(A)$ be a bicyclic representation of $A$ and $\Sigma(\mathcal{K})=\{\mathcal{K}_I,\mathcal{K}_{II}\}$ the bipartition induced by $G(A)$. By Lemma \ref{lembicycliccentral}, we deduce that $T_I$ and $T_{II}$ contain each one at most two central cells. Thus the bipartition 
$\Sigma(\mathcal{K})$ is bicompatible. For any $i\in \{I,II\}$, by deleting all nonbasic edges with index not in $f(E_i(\Sigma))$ and contracting all basic edges with index not in $E_i(\Sigma)$ (in any order), one obtains an $R_i^*(\Sigma)$-cyclic representation of $M_i (\Sigma)$. 

A $\{n_I,n_{II}\}$ $\frac{1}{2}$-binet representation of $N(\Sigma)$ is obtained as follows. Observe that for $i=I$ and $II$, $E_{i,\frac{1}{2}}$ is the edge index set a tree rooted at the central node of $T_i$; for any $2$-edge $f_j$ ($j\in S_{\frac{1}{2}}$), the intersection of this rooted tree with the fundamental circuit of $f_j$ is a directed path denoted by $q_{j,i}$ starting at the central node of $T_i$
($q_{j,i}$ may be equal to the central node of $T_i$).
For $i=I$ and $II$, add a nonbasic half-edge entering the terminal node of $q_{j,i}$ for all $j\in S_{\frac{1}{2}}$, and create a basic half-edge $n_i$ entering the central node of $T_i$. Then, delete all nonbasic edges (except half-edges) with index not in $F\cup S_{\frac{1}{2}}$, all basic edges with index not in $E\cup \{n_I, n_{II}\}$ and the remaining isolated nodes. (See Figure \ref{fig:bicyclicmatnet}.)

Suppose now that $\mathcal{K}_I$ contains at least one non-central cell, say $C_l$, and $\mathcal{K}_{II}$ at least two, say $C_{l'}$ and $C_{l''}$. Let us prove that there exists a bicyclic representation of $A$ inducing a bicompatible bipartition $\Sigma'(\mathcal{K})= \{ \mathcal{K}_I', \mathcal{K}_{II}'\}$ of $\mathcal{K}$ with $\mathcal{K}_I'= \mathcal{K}_I \cup \{ C_{l''} \} $ and $\mathcal{K}_{II}'= \mathcal{K}_{II} \verb"\" \{ C_{l''} \}$.
From the definition of a central cell, it follows that $\cup_{j\in S_{\frac{1}{2}}} s(A_{\bullet j}) \cap  C_k =R_k$ and $A_k^{\frac{1}{2} \rightarrow 1}$ is a network matrix for $k=l,$ $l'$ and $l''$. For any $1 \le k \le r$, let $$G_k=
G(A_{C_k \times (f(C_k) \verb"\" S_{\frac{1}{2}})}) \subseteq G(A).$$

Assume that $T(C_{l''})$ has only directed edges. 
Make a copy $\tilde G_{l''}$ of $G_{l''}$ and denote by $t_{l''}$ (resp., $u_{l''}$) the first (resp., terminal) node of the copy of $\mathcal{P}_{l''}$ in $\tilde G_{l''}$. 
Suppose that there exists an endnode of $\mathcal{P}_l$, say $u_l$, which is not central and such that $\mathcal{P}_l$ enters $u_l$. Let $\mathcal{T}(u_l)$ be the set of subgraphs $G_k$ in $G(A)$ with $1\le k \le r$ containing $u_l$. 
Then contract all edges of $G_{l''}$ and cut the basic cycle in $T_I$ at $u_l$, by creating a copy $\tilde u_l$ of $u_l$ so that $u_l \in G_l$ and $\tilde u_l \in G_k$ for all
$G_k \in \mathcal{T}(u_l)$ with $k\neq l$. Finally, identify $t_{l''}$ with $u_l$, and $u_{l''}$ with $\tilde u_l$. Thus, we obtain a new bicyclic representation of $A$ inducing a bipartition $\Sigma'(\mathcal{K})$ equivalent to $\Sigma(\mathcal{K})$. 
See Figure \ref{fig:bicyclicmove} for an example with $C_l=C_3=\{2,5\}$, $C_{l'}=\{8,9\}$ and $C_{l''}=C_5= \{10 \}$. 

The other cases can be treated in a similar way
by using Lemma \ref{lembicyclicnoncentral}, a network representation of $A_k^{\frac{1}{2}\rightarrow 1}$ for $k=l$, $l'$ or $l''$, and switching operations if necessary.
By moving some cells this way, one can prove as above that 
for any bipartition  $\Sigma'(\mathcal{K})$ equivalent to $\Sigma(\mathcal{K})$, the matrix $M_i(\Sigma')$ is $R_i^*(\Sigma')$-cyclic for $i= I $ and $II$ and $N(\Sigma')$ is $\{n_I,n_{II}\}$ $\frac{1}{2}$-binet. 
\end{proof}\\

For computing a $\{n_I,n_{II}\}$ $\frac{1}{2}$-binet representation of $N(\Sigma)$, we use the following auxiliary matrix:

$$\tilde N(\Sigma) = \left( 
\begin{array}{cc}
 N(\Sigma) & 
\begin{array}{ccc}
\begin{array}{c}
A_{E_{I,\frac{1}{2}} \times S_{\frac{1}{2}} } \\
0_{|\overline{E_{I,\frac{1}{2}}}| \times |S_{\frac{1}{2}}| }
\end{array}  &
\begin{array}{c}
0_{|\overline{E_{II,\frac{1}{2}}}| \times |S_{\frac{1}{2}}| }\\
A_{E_{II,\frac{1}{2}}\times S_{\frac{1}{2}} } \\
\end{array}  &

0_{|E| \times 1 } \\

1_{1 \times |S_{\frac{1}{2}}| } & 1_{1 \times |S_{\frac{1}{2}}| } & 1 \\
1_{1 \times |S_{\frac{1}{2}}| } & 1_{1 \times |S_{\frac{1}{2}}| } & 1\\
\end{array}

\end{array}\right). $$

Suppose that the matrix $\tilde N(\Sigma)$ has a network representation $G(\tilde N)$. Due to the last column of $\tilde N$ the basic edges $e_{n_I}$ and $e_{n_{II}}$ are adjacent. Denote by $v_0$ the node in $G(\tilde N)$ incident with $e_{n_I}$ and $e_{n_{II}}$. 
Observe that for any $j\in S_{\frac{1}{2}}$ and $i\in \{I,II\}$, $(s(A_{\bullet j}) \cap E_{i,\frac{1}{2}}) \cup \{n_i\}$ and $(s(A_{\bullet j}) \cap E_{i,\frac{1}{2}})\cup \{n_I,n_{II} \}$ are edge index sets of directed paths. Moreover, since $A$ is connected, for all cell $C_k$ with $\xi < k \le r$ there exists $j' \in S_{\frac{1}{2}}$ 
such that  $s(A_{\bullet j'}) \cap C_k \neq \emptyset$. This implies that the only basic edges incident with $v_0$ are $e_{n_I}$ and $e_{n_{II}}$.

\begin{tabbing}
\textbf{Procedure\,\,Bicyclic($A$)}\\

\textbf{Input: } \=  A connected $\frac{1}{2}$-equisupported
 matrix $A$ with entries $0$, $1$, $\frac{1}{2}$ or $2$ and\\
\>  at least one $\frac{1}{2}$-entry.\\
\textbf{Output:} A bicyclic representation $G(A)$ of $A$, or determines that none exists.\\
1)\verb"  "\= compute the quotient set $\mathscr{S}$;\\
2)\>{\bf forÊ} \= every class $[\Sigma(\mathcal{K})]$ in $\mathscr{S}$, {\bf do}\\ 
3) \> \> for $i=I $ and $II$, call {\tt RCyclic}($M_i(\Sigma)$,$R_i^*(\Sigma)$) of Section \ref{sec:cyc};\\
\> \> if we have an $R_i^*(\Sigma)$-cyclic representation $G(M_i)$ of $M_i(\Sigma)$ for $i=I$ and $II$,\\
\> \> then go to 4, otherwise return to $2$;\\
4) \> \>   compute a network representation $G(\tilde N)$ of $\tilde N(\Sigma)$, if one exists, otherwise return to 2,\\
 \> \> delete $v_0$ and using switching operations, build up a $\{n_I,n_{II}\}$ $\frac{1}{2}$-binet representation \\
\> \>$G(N)$  of $N(\Sigma)$ such that $e_{n_I}$ and $e_{n_{II}}$ are entering;\\
5) \>\>  contract all half-edges in $G(M_I)$, $G(M_{II})$ and $G(N)$; and for $i=I$ and $II$, identify \\
\> \> the subtree of $G(N)$ with edge index set $E_{i,\frac{1}{2}}$ with the corresponding tree \\
\> \> in $G(M_i)$ by keeping the labeling of the basic edges in $G(M_i)$;\\
\> \>  output the resulting bicyclic representation $G(A)$ of $A$ and STOP;\\
\> {\bf endfor}\\
\> output that $A$ is not bicyclic;
\end{tabbing}

Finally, we prove Theorems \ref{thmbicyclicpro} and \ref{thmbicyclicChar}.\\

\noindent
{\bf Proof of Theorems \ref{thmbicyclicpro} and \ref{thmbicyclicChar}.} \quad The "only if" part of Theorem \ref{thmbicyclicChar} follows from Proposition \ref{propbicyclicChar}.

Let us prove the counter part of Theorem \ref{thmbicyclicChar}
and the correctness of the procedure Bicyclic.
Suppose that there exists a bicompatible bipartition 
$\Sigma_0(\mathcal{K})=\{\mathcal{K}_I, \mathcal{K}_{II} \}$ 
such that for any bipartition  $\Sigma(\mathcal{K})$ equivalent to $\Sigma_0(\mathcal{K})$, the matrix $M_i(\Sigma)$ is $R_i^*(\Sigma)$-cyclic for $i= I $ and $II$ and $N(\Sigma)$ is $\{n_I,n_{II}\}$ $\frac{1}{2}$-binet. 
Suppose that the procedure Bicyclic is dealing with the class
$[\Sigma_0(\mathcal{K})]$ in step 2. By Theorem \ref{thmcyclicproCyc}, for some bipartition $\Sigma$ equivalent to $\Sigma_0$,
the procedure RCyclic computes an $R_i^*(\Sigma)$-cyclic representation $G(M_i)$ of $M_i(\Sigma)$ for $i=I$ and $II$, and a $\{n_I,n_{II}\}$ $\frac{1}{2}$-binet representation $G(N)$ of $N(\Sigma)$ is computed.
Since $A$ is connected, a cell $C_k$ such that $R_k=\emptyset$ ($\xi < k \le r$)
necessarily intersects the support of a column with index in $S_{\frac{1}{2}}$, and for each $j\in F$, $s(A_{\bullet j})$ is the edge index set of a path in $G(N)$. We observe that
$(A^T_{E_{I,\frac{1}{2}}\times S_{\frac{1}{2}} } \,  0^T_{|\overline{E_{I,\frac{1}{2}}}| \times |S_{\frac{1}{2}}| } \, 1^T_{1 \times |S_{\frac{1}{2}}| } \,
0^T_{1 \times |S_{\frac{1}{2}}| })^T$
is a submatrix of $N(\Sigma)$, and $s(A_{\bullet j}) \cap E_{I,\frac{1}{2}}$ is the edge index set of a directed path in $G(M_I)$ starting at the central node, for all $j\in S_{\frac{1}{2}}$. Thus,
the basic subtrees of $G(N)$ and $G(M_I)$ with edge index set  $E_{I,\frac{1}{2}}$ are isomorphic; the former is rooted at the endnode of $e_{n_I}$ while the latter is rooted at the central node of $G(M_I)$. The same thing holds with $G(M_{II})$, $E_{II,\frac{1}{2}}$ and $e_{n_{II}}$ instead of $G(M_{I})$, $E_{I,\frac{1}{2}}$ and $e_{n_I}$, respectively. We deduce that the procedure Bicyclic outputs a bicyclic representation of $A$ (see Figure \ref{fig:bicyclicmatnet} for an example). 

For the proof of Theorem \ref{thmbicyclicpro}, it remains to show that the procedure Bicyclic works in time $O(nm \alpha)$.
In a bicompatible bipartition $\Sigma(C)=\{\mathcal{K}_I, \mathcal{K}_{II}\}$, since any $\mathcal{K}_i$ ($i=I$ or $II$)
has at most 2
central cells, there are at most $6$ different ways of partitioning the central cells into two groups. Furthermore, any bipartition of the central cells can be extended to
at most $3$ nonequivalent bicompatible partitions of $\mathcal{K}$. Therefore, the number of passages through step 1 is upper bounded by $18$. By Theorem \ref {thmSubclassNTutteCunNet}, the computational effort required to obtain a network representation of a matrix $A_k^{\frac{1}{2}\rightarrow 1}$ for some $k$ is $O(|C_k| \alpha_k)$ where $\alpha_k$ is the number of nonzero elements in $A_k$. Thus computing the quotient set $\mathscr{S}$ can be executed in time $O(n\alpha)$. Similarly, step 4 takes time $O(n\alpha)$. Finally, by Theorem \ref{thmcyclicproCyc}, the procedure RCyclic works in time $O(n m\alpha)$. This completes the proof.
{\hfill$\BBox{\rule{.3mm}{3mm}}$} \\


\chapter{Recognizing $\{\epsilon,\rho\}$-central matrices}\label{ch:central}

In this chapter, we provide a characterization of $\{\epsilon,\rho\}$-central (non-network) $\{0,1,2\}$-matrices, where $\epsilon$ and $\rho$ are two given row indexes, as wel as a recognition procedure.
Let $A$ be a $\{0,1,2\}$-matrix of size $n\times m$ and $\alpha$ the number of nonzero entries in $A$.
We assume that $A$ is not a network matrix. Let $\epsilon$ and $\rho$ ($\rho \neq 1$) be two row indexes and $S_0=\{j \, :\, \epsilon ,\rho \in s(A_{\bullet j})\}$. 
We will prove the following.

\begin{thm}\label{thmcentralCarEmpty}
Suppose $S_0=\emptyset$. The matrix $A$ can be tested for being $\{\epsilon,\rho\}$-central by the procedure CentralI. The running time of the procedure is $O(n^3 \alpha)$.
\end{thm}

\begin{thm}\label{thmcentralCarNotEmpty}
Suppose $S_0\neq \emptyset$. The matrix $A$ can be tested for being $\{\epsilon,\rho\}$-central by the procedure CentralII. The running time of the procedure is $O(n^3 \alpha)$.
\end{thm}

A characterization of cyclic matrices without $\frac{1}{2}$- nor $2$-entry  turns out to be the most critical part in the recognition of binet matrices. Provided that $A$ is cyclic, there is a priori no direct way of detecting the edge index set of the basic cycle, in some cyclic representation of $A$. To illustrate 
this, consider the cyclic matrix in Figure \ref{fig:central1min}. Its minimally non-network submatrices are:

$$
\begin{array}{ccc}
M_1=
\begin{tabular}{c|c|c|c|c|c|}
  & $f_1$ & $f_2$ & $f_3$  \\
  \hline
$e_1$  &0&1&1\\
\hline
$e_2$  &1&0&1 \\
\hline
$e_3$  &1&1&0\\
\hline
\end{tabular}  & &
M_2=
\begin{tabular}{c|c|c|c|c|c|}
  & $f_1$ & $f_2$ &  $f_4$ \\
  \hline
$e_3$  &1&1&0 \\
\hline
$e_4$  &1&0&1\\
\hline
$e_5$  &0&1&1 \\
\hline
\end{tabular}  

\end{array} $$

\begin{figure}[h!]
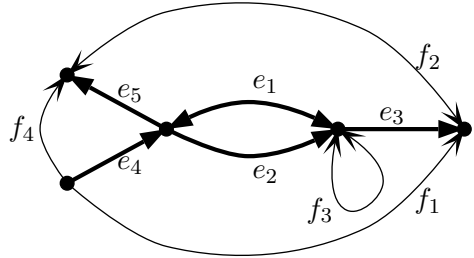


\begin{center}
$
\begin{array}{cc}

\begin{tabular}{c|c|c|c|c|}
  & $f_1$ & $f_2$ & $f_3$ & $f_4$  \\
  \hline
$e_1$  &0&1&1&0 \\
\hline
$e_2$ &1&0&1&0 \\
\hline
$e_3$  &1&1&0&0 \\
\hline
$e_4$  &1&0&0&1\\
\hline
$e_5$  &0&1&0&1 \\
\hline
\end{tabular}  &

\psset{xunit=1.2cm,yunit=1.2cm,linewidth=0.5pt,radius=0.1mm,arrowsize=7pt,
labelsep=1.5pt,fillcolor=black}

\pspicture(-1.5,0.8)(5,3)

\pscircle[fillstyle=solid](1,1){.1}
\pscircle[fillstyle=solid](2.9,1){.1}
\pscircle[fillstyle=solid](4.3,1){.1}
\pscircle[fillstyle=solid](-0.1,.4){.1}
\pscircle[fillstyle=solid](-0.1,1.6){.1}

\pscurve[linewidth=1.6pt,arrowinset=0]{<->}(1,1)(1.9,1.3)(2.9,1)
\rput(2.1,1.5){$e_{1}$}

\pscurve[linewidth=1.6pt,arrowinset=0]{->}(1,1)(1.9,.7)(2.9,1)
\rput(2.1,.5){$e_{2}$}

\psline[linewidth=1.6pt,arrowinset=0]{->}(2.9,1)(4.3,1)
\rput(3.5,1.15){$e_{3}$}

\psline[linewidth=1.6pt,arrowinset=0]{<-}(1,1)(-0.1,.4)
\rput(.6,.6){$e_{4}$}

\psline[linewidth=1.6pt,arrowinset=0]{->}(1,1)(-0.1,1.6)
\rput(.6,1.4){$e_{5}$}

\pscurve[arrowinset=.5,arrowlength=1.5]{->}(-0.1,.4)(1,-0.3)(3.2,-0.1)(4.3,1)
\rput(3.9,.2){$f_{1}$}

\pscurve[arrowinset=.5,arrowlength=1.5]{<->}(-0.1,1.6)(1,2.3)(3.2,2.1)(4.3,1)
\rput(3.9,1.8){$f_{2}$}

\pscurve[arrowinset=.5,arrowlength=1.5]{<->}(2.9,1)(3,.1)(3.4,.4)(2.9,1)
\rput(2.7,0.1){$f_{3}$}

\pscurve[arrowinset=.5,arrowlength=1.5]{<-}(-0.1,1.6)(-0.4,1)(-0.1,.4)
\rput(-.6,1){$f_{4}$}

\endpspicture 

\end{array}$
\end{center}

\vspace{.3cm}
\caption{A cyclic matrix $A$ and a $\{1,2\}$-central representation of $A$.} 
\label{fig:central1min}

\end{figure}

By careful analysis, one can show that there is no cyclic representation of $A$ whose basic cycle has an edge index set equal to the row index set of $M_1$ or $M_2$. However, given the row indexes $1$ and $2$ and since $A$ is nonnegative, we describe a procedure for computing a $\{1,2\}$-central representation of $A$.

The recognition of $\{\epsilon,\rho\}$-central matrices can be 
viewed as a generalization of Schrijver's method outlined on page \pageref{mycounter3}. Instead of considering a row index $i$, and then carrying on with the connected components of the matrix $A_{\overline{\{i\}}\times\overline{f(\{i\})} }$, we fix two row indexes, namely $\epsilon$ and $\rho$, and deal with the connected components of the matrix $A_{\overline{\{\epsilon, \rho\}}\times\overline{f(\{\epsilon, \rho\})} }$.

The chapter is organised as follows. First, we give the main notations with graphical interpretations and state two theorems 
characterizing $\{\epsilon,\rho\}$-central matrices
in case $S_0= \emptyset$ (respectively, $S_0 \neq \emptyset$) requiring certain assumptions that are discussed in Section \ref{sec:PreCentral}. Then, in Sections \ref{sec:S0empty} and \ref{sec:S0notempty}, we deal with the proof of the first and second theorem, respectively.  Section \ref{sec:CentralNetcorelated} provides a procedure for recognizing $\{1,\rho\}$-noncorelated network matrices, that is used in Section \ref{sec:S0empty}.

\section{Preliminaries}\label{sec:PreCentral}

Let $\epsilon$ and $\rho\neq 1$ be two row indexes of $A$,
$R^*= \{ \epsilon, \rho \}$ and
$D=(V,\Upsilon)$ be a digraph with respect to $R^*$, as computed in Section \ref{sec:DefDigraphD}.
We will make use of the whole terminology described in Chapter \ref{ch:multidiD}.
Let $S^*=\{ j \, :\, s(A_{\bullet j})\cap R^*\neq \emptyset \}$, $S_0=\{j\in S^* \, :\, \epsilon,\rho \in s(A_{\bullet j})\}$,
$S_1=\{j\in S^* \, :\, \epsilon \in s(A_{\bullet j}),\, \rho \notin s(A_{\bullet j})\}$ and $S_2=\{j\in S^* \, :\, \epsilon \notin s(A_{\bullet j}), \,\rho \in s(A_{\bullet j})\}$. 
For any bonsai $E_\ell$ and $k\in \{0,1,2\}$, we define the \emph{connecter set}\index{connecter set} $f_{S_k}(E_\ell)= \{j\in S_k \, : \, s(A_{\bullet j}) \cap E_\ell \neq \emptyset \}$. For any $j\in S^*$ and $1\le \ell \le b$, denote $E_{\ell j}=E_\ell \cap s(A_{\bullet j})$, and for $S\subseteq S_1 \cup S_2$,  $E_{\ell S}=\cup_{j\in S} E_{\ell j}$. Throughout this section we assume $\epsilon=1$. 

\begin{figure}[h!]
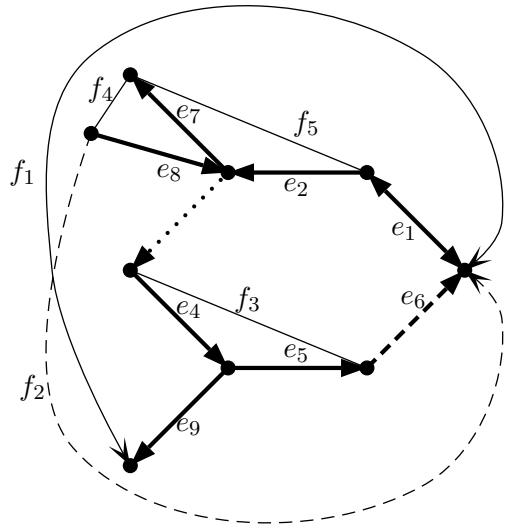

\vspace{.5cm}
$
\begin{array}{cc}

\begin{array}{c}
S_1=\{1\}; \quad S_2=\{2\};\\\\
E_1=\{2;7;8\};\,\,\, E_2=\{4;5\};\,\,\, E_3=\{9\};\\ \\
\begin{tabular}{c|c|c|c|}
 $\ell=$  &  $1$ &  $2$ &  $3$ \\
 \hline
$f_{S_1}(E_\ell)$  & \{1\} &\{1\}&\{1\} \\
 \hline
$f_{S_2}(E_\ell)$  &\{2\}&\{2\}&$\emptyset$\\
\hline
$E_{\ell  S_1}$  &\{2\}&\{4\}& \{9\}\\
\hline
$E_{\ell  S_2}$  &\{8\}&\{4,5\}&$\emptyset$\\
 \hline

\end{tabular}   
\end{array}  &

\psset{xunit=1.3cm,yunit=1.3cm,linewidth=0.5pt,radius=0.1mm,arrowsize=7pt,
labelsep=1.5pt,fillcolor=black}

\pspicture(-2,1)(3.7,3.5)

\pscircle[fillstyle=solid](0,1){.1}
\pscircle[fillstyle=solid](1,0){.1}
\pscircle[fillstyle=solid](1,2){.1}
\pscircle[fillstyle=solid](2.42,0){.1}
\pscircle[fillstyle=solid](2.42,2){.1}
\pscircle[fillstyle=solid](3.42,1){.1}
\pscircle[fillstyle=solid](-.4,2.4){.1}
\pscircle[fillstyle=solid](0,3){.1}
\pscircle[fillstyle=solid](0,-1){.1}

\psline[linewidth=1.6pt,arrowinset=0]{<->}(3.42,1)(2.42,2)
\rput(2.8,1.4){$e_{1}$}

\psline[linewidth=1.6pt,arrowinset=0]{->}(2.42,2)(1,2)
\rput(1.7,1.85){$e_{2}$}

\psline[linestyle=dotted,linewidth=1.6pt,arrowinset=0]{->}(1,2)(0,1)

\psline[linewidth=1.6pt,arrowinset=0]{->}(0,1)(1,0)
\rput(.6,.6){$e_{4}$}

\psline[linewidth=1.6pt,arrowinset=0]{->}(1,0)(2.42,0)
\rput(1.7,.15){$e_{5}$}

\psline[linewidth=1.6pt,arrowinset=0,linestyle=dashed]{->}(2.42,0)(3.42,1)
\rput(2.9,.7){$e_{6}$}

\psline[linewidth=1.6pt,arrowinset=0]{->}(1,2)(0,3)
\rput(.6,2.6){$e_{7}$}

\psline[linewidth=1.6pt,arrowinset=0]{<-}(1,2)(-.4,2.4)
\rput(.4,2){$e_{8}$}

\psline[linewidth=1.6pt,arrowinset=0]{->}(1,0)(0,-1)
\rput(.6,-.6){$e_{9}$}

\pscurve[arrowinset=.5,arrowlength=1.5]{<->}(0,-1)(-.8,1)(-.4,3.2)(3,3.2)(3.8,1.5)(3.42,1)
\rput(-1.1,2){$f_{1}$}

\pscurve[arrowinset=.5,arrowlength=1.5,linestyle=dashed]{->}(-.4,2.4)(-.6,-.8)(3,-1.2)(3.8,.5)(3.42,1)
\rput(-1,-.2){$f_{2}$}




\psline[arrowinset=.5,arrowlength=1.5]{-}(0,1)(2.42,0)
\rput(1.2,.7){$f_{3}$}

\psline[arrowinset=.5,arrowlength=1.5]{-}(0,3)(-.4,2.4)
\rput(-.3,2.85){$f_{4}$}

\psline[arrowinset=.5,arrowlength=1.5]{-}(0,3)(2.42,2)
\rput(1.8,2.5){$f_{5}$}


\endpspicture 
\end{array}
$
\vspace{1.5cm}
\caption{a $\{1,6\}$-central representation of some binet $\{0,1\}$-matrix, where $E_1$ is disjointly shared, $E_2$ is jointly shared and $E_3$ is $S_1$-dominated. ($f_3$, $f_4$ and $f_5$ are directed edges.)} 
\label{fig:central2}

\end{figure}

We say that a bonsai $E_\ell \in V$ ($1\le \ell \le b$) is \emph{$S_k$-dominated}\index{dominated@$S_k$-dominated} if $f_{S_k}(E_\ell)\neq \emptyset$ and $f_{S_{k'}}(E_\ell)= \emptyset$ for some $k,k' \in \{ 1,2\}$ ($k\neq k'$).
A bonsai $E_\ell\in V$ is  \emph{shared}\index{shared} if $f_{S_1}(E_\ell) \neq \emptyset $ and $f_{S_2}(E_\ell) \neq \emptyset $. A shared bonsai $E_\ell\in V$ is said to be \emph{disjointly shared}\index{disjointly shared} if $E_{\ell S_1} \cap E_{\ell  S_2} = \emptyset $, otherwise  \emph{ jointly shared}\index{jointly shared}. See Figure \ref{fig:central2}.

\begin{figure}[t!]
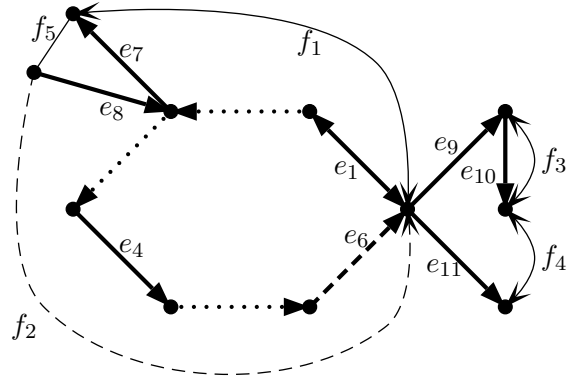

$
\begin{array}{cc}
\begin{array}{c}
S_0=\{3,4\};\,\,\,S_1=\{1\}; \,\,\, S_2=\{2\};\\\\
E_1=\{7;8\};\,\,\, E_2=\{4\}; \\
E_3=\{9\};\,\,\, E_4=\{10\};\,\,\, E_5=\{11\};\\ \\
\begin{tabular}{c|c|c|c|}
 $\ell=$  &  $1$ &  $2$ &  $3,4,5$\\
   \hline
$f_{S_0}(E_\ell)$  &$\emptyset$&\{3,4\} & \{3,4\}  \\
 \hline
$f_{S_1}(E_\ell)$  & \{1\}& $\emptyset$ &$\emptyset$ \\
 \hline
$f_{S_2}(E_\ell)$ &\{2\}&\{2\}&$\emptyset$\\
\hline
$E_{\ell  S_1}$ & \{7\}&$\emptyset$&$\emptyset$\\
\hline
$E_{\ell  S_2}$  & \{8\} &\{4\}&$\emptyset$ \\
 \hline

\end{tabular}   
\end{array}   &

\psset{xunit=1.3cm,yunit=1.3cm,linewidth=0.5pt,radius=0.1mm,arrowsize=7pt,
labelsep=1.5pt,fillcolor=black}

\pspicture(-1.5,1)(5,3.5)

\pscircle[fillstyle=solid](0,1){.1}
\pscircle[fillstyle=solid](1,0){.1}
\pscircle[fillstyle=solid](1,2){.1}
\pscircle[fillstyle=solid](2.42,0){.1}
\pscircle[fillstyle=solid](2.42,2){.1}
\pscircle[fillstyle=solid](3.42,1){.1}
\pscircle[fillstyle=solid](4.42,1){.1}
\pscircle[fillstyle=solid](4.42,2){.1}
\pscircle[fillstyle=solid](-.4,2.4){.1}
\pscircle[fillstyle=solid](0,3){.1}
\pscircle[fillstyle=solid](4.42,0){.1}

\psline[linewidth=1.6pt,arrowinset=0]{<->}(3.42,1)(2.42,2)
\rput(2.8,1.4){$e_{1}$}

\psline[linewidth=1.6pt,arrowinset=0,linestyle=dotted]{->}(2.42,2)(1,2)

\psline[linewidth=1.6pt,arrowinset=0,linestyle=dotted]{->}(1,2)(0,1)

\psline[linewidth=1.6pt,arrowinset=0]{->}(0,1)(1,0)
\rput(.6,.6){$e_{4}$}

\psline[linewidth=1.6pt,arrowinset=0,linestyle=dotted]{->}(1,0)(2.42,0)

\psline[linewidth=1.6pt,arrowinset=0,linestyle=dashed]{->}(2.42,0)(3.42,1)
\rput(2.9,.7){$e_{6}$}

\psline[linewidth=1.6pt,arrowinset=0]{->}(1,2)(0,3)
\rput(.6,2.6){$e_{7}$}

\psline[linewidth=1.6pt,arrowinset=0]{<-}(1,2)(-.4,2.4)
\rput(.4,2){$e_{8}$}



\psline[linewidth=1.6pt,arrowinset=0]{->}(3.42,1)(4.42,2)
\rput(3.82,1.65){$e_{9}$}

\psline[linewidth=1.6pt,arrowinset=0]{<-}(4.42,1)(4.42,2)
\rput(4.15,1.35){$e_{10}$}

\psline[linewidth=1.6pt,arrowinset=0]{->}(3.42,1)(4.42,0)
\rput(3.8,.4){$e_{11}$}



\pscurve[arrowinset=.5,arrowlength=1.5]{<->}(0,3)(3.1,2.5)(3.42,1)
\rput(2.42,2.7){$f_{1}$}


\pscurve[arrowinset=.5,arrowlength=1.5,linestyle=dashed]{->}(-.4,2.4)(-.3,0)(3.2,-.2)(3.42,1)
\rput(-.5,-.2){$f_{2}$}

\pscurve[arrowinset=.5,arrowlength=1.5]{<->}(4.42,1)(4.72,1.5)(4.42,2)
\rput(4.92,1.5){$f_{3}$}

\pscurve[arrowinset=.5,arrowlength=1.5]{<->}(4.42,0)(4.72,.5)(4.42,1)
\rput(4.92,.5){$f_{4}$}

\psline[arrowinset=.5,arrowlength=1.5]{-}(0,3)(-.4,2.4)
\rput(-.3,2.85){$f_{5}$}





\endpspicture 
\end{array}
$
\vspace{1cm}
\caption{a (partial) $\{1,6\}$-central representation of some binet $\{0,1,2\}$-matrix, where $E_1$ is sensitive,  $E_2$ and $E_4$ are $S_0$-straight, $E_3 \in V_2$ and 
$E_5\in V_0-(V_2 \cup V_{st})$. ($f_5$ is directed.)} 

\label{fig:central3}

\end{figure}

\begin{table}[h!]
\begin{center}

\begin{tabular}{c|c|c|c|c|}
  & $f_1$ & $f_2$ & $f_4$ &  \\
 \hline
$e_7$  &1&0&1& 0 \\
\hline
$e_8$  &0&1&1& 0\\
\hline
$\tilde e_1$  &1&0&0&1 \\
\hline
$\tilde e_2$  &0&1&0&1 \\
\hline
\end{tabular} 

\end{center}
\vspace{.2cm}
\caption{The bonsai matrix $N_1$ with respect to $E_1$
given in Figure \ref{fig:central3}.} 
\label{tab:central4}

\end{table}

A bonsai $E_\ell\in V$ is called 
\emph{sensitive}\index{sensitive}, if it is disjointly shared, $f_{S_0}(E_\ell) = \emptyset$, $J_\ell^1=\{E_{\ell j} \, :\, j\in f_{S_k}(E_\ell) \}$ and 
$J_\ell^2=\{E_{\ell j} \, :\, j\in f_{S_{k'}}(E_\ell) \}$ for some $k, k' \in \{1,2\}$ $k\neq k'$, and $N_\ell$ is a network matrix (see page \pageref{mycounter6} for the definition of $J_\ell^1$ and $J_\ell^2$). For every $1\le l \le b$, the bonsai matrix $N_l$ is defined in Section \ref{sec:bonsaimat}.
For instance, the bonsai $E_1$ in Figure \ref{fig:central2} is disjointly shared but not sensitive, because $J_1^2=\emptyset$, while 
$E_1$ in Figure \ref{fig:central3} is sensitive (see Table \ref{tab:central4}).
For $1\le \ell , \ell' \le b$,
we say that $E_\ell$ and $E_{\ell '}$ are \emph{$S_k$ linked}\index{linked@$S_k$-linked}, or
$E_\ell$ is \emph{$S_k$ linked} to $E_{\ell '}$,
 if $f_{S_k}(E_\ell) \cap f_{S_k}(E_{\ell '})\neq \emptyset $ for some $k \in \{ 1, 2\}$. (If $E_\ell$ and $E_{\ell '}$ are $S_1$ and $S_2$ linked, then we write that $E_\ell$ and $E_{\ell '}$ are $S_1,S_2$ linked.) As an example,
 in Figure \ref{fig:central2}, $E_1$ and $E_2$ are $S_1$, $S_2$ linked.

Let $V_0=\{ E_\ell \, : \, f_{S_0}(E_\ell)\neq \emptyset \}$, $V_0'=V_0 \cup \{\m{shared bonsai} \}$, $V_2=\{ E_\ell \, :\, \exists \beta \in S_0 \m{ with }$\\ $g_\beta(E_\ell)=2 \}$ (see page \pageref{mycounter7} for the definition of $g$). A bonsai $E_\ell$ is said to be \emph{$S_0$-straight}\index{straight@$S_0$-straight} if $E_\ell \in V_0-V_2$ and $E_{\ell \beta}=E_{\ell \beta'}$  for any $\beta,\beta'\in S_0$. Let $V_{st}$ denote the set of $S_0$-straight bonsais. 
In Figure \ref{fig:central3}, $E_2$ and $E_4$ are $S_0$-straight, while $E_3\in V_2$ since $g_3(E_3)=2$ (see Proposition \ref{propBlsubstems}), and $E_5$ is neither in $V_2$ nor $S_0$-straight.

For a set $V' \subseteq V$, we define the undirected graph $F^*(V')$\index{graph!$F^*(V)$}\label{mycounter4} whose vertex set is $V'$, and two vertices $E_\ell$ and $E_{\ell '}$ are adjacent if and only if $f^*(E_\ell) \cap f^*(E_{\ell '}) \neq \emptyset $, $(E_\ell,E_{\ell '})\notin D$ and $(E_{\ell '},E_\ell) \notin D$.  Provided that $A$ has a $\{1,\rho\}$-central representation $G(A)$, if some bonsai $E_\ell \in V$ has some property (for instance $E_\ell$ is sensitive), then we say that the corresponding bonsai $B_\ell$ in $G(A)$ has this property (for instance $B_\ell$ is sensitive), and conversely.

Before explaining the graphical motivations of these definitions, we state the main theorems. Later, we 
will give the definition of left-compatibility for a set $U\subseteq V$ of bonsais. For any left-compatible set $U\subseteq V$, we will define a
$U$-spanning pair $(j_1,j_2)$ of column indexes as well as a subset $V(j_1,j_2)$ of $V$, and we will describe an instance $\Omega(U,j_1,j_2)$ of the $2$-SAT problem, where variables are associated to bonsais in $V(j_1,j_2)$.
We will also need the following.
\begin{tabbing}
{\bf assumption $\mathscr{A}$:} \= for any sensitive bonsai $E_\ell$, if some bonsai $E_{\ell '}$ is  $S_1,S_2$  linked to $E_\ell$,  then \\ \> $J_{\ell '}^2=\emptyset$  and $N_{\ell '}$  is a network matrix.
\end{tabbing}
\noindent
Under this assumption, we shall prove the following theorems.

\begin{thm}\label{thmS0empty}
Suppose $S_0=\emptyset$. Then the matrix $A$ is $\{1,\rho\}$-central if and only if $g_\beta(E_\ell)\le 1$ for all $1\le \ell\le b$, $\beta \in f^*(E_\ell)$, and
there exists a left-compatible set $U$ such that for any $U$-spanning pair $(j_1,j_2)$, the graph $F^*(\overline{V(j_1,j_2)})$ is bipartite, for all $E_\ell \in \overline{V(j_1,j_2)}$ the bonsai matrix $N_\ell$ is a network matrix and $J_\ell^2= \emptyset$, and
the instance $\Omega(U,j_1,j_2)$ of the $2$-SAT problem has a truth assignment.
\end{thm}

\begin{thm}\label{thmS0notempty}
Suppose $S_0\neq \emptyset$. Then the matrix $A$ is $\{1,\rho\}$-central if and only if the graph  $F^*(\overline{V_0'})$ is bipartite, for any $E_\ell \in \overline{V_0'}$ the bonsai matrix $N_\ell$ is a network matrix and $J_\ell^2= \emptyset$, and
there exists a right-compatible set $U$ such that the instance $\Lambda(U)$ of the $2$-SAT problem has a truth assignment.
\end{thm}

Suppose that $A$ has a basic $\{1,\rho\}$-central  representation $G(A)$. Up to row permutations and changing  the value of $\rho$,
we may assume that $w_1,\ldots,w_\rho$ are the vertices of the basic cycle in this order, $e_1=[w_1,w_\rho]$ and $e_\rho=]w_{\rho-1},w_\rho]\in G(A)$. Hence $w_\rho$ is a central node. 
Recall that $T$ denotes the basic maximal $1$-tree in $G(A)$,
$G_0(A)$ the connected component of $T-\{e_1,e_\rho\}$ containing $w_1$ (and $w_{\rho-1}$) and $G_1(A)$ the other connected component. We say that a subgraph of $G_0(A)$ and $G_1(A)$ is \emph{on the left} and \emph{on the right}, respectively, \emph{of $\{e_1,e_\rho\}$}. For all $1\le \ell\le b$, the bonsai $E_\ell$ is the edge index set of a tree in $G(A)$ denoted as $B_\ell$ and called a \emph{bonsai in $G(A)$}.\\
For any $w_i$ and $w_{i'}$ on the basic cycle  with $1\le i,i' \le \rho-1$, let $p(w_i,w_{i'})$ be the basic path in $G_0(A)$ between $w_i$ and $w_{i'}$, and $I(w_i,w_{i'})$ its row index set. For any two bonsais
$B_\ell$ and $B_{\ell '}$ intersecting the basic cycle, if $i\le i'$ for all $w_i \in B_\ell$ and $w_{i'} \in B_{\ell '}$ ($1\le i,i' \le \rho-1$), then  $B_\ell$ is \emph{preceding}\index{preceding} $B_{\ell '}$ and $B_{\ell '}$ is \emph{succeeding}\index{succeeding} $B_\ell$. As an example, in Figure \ref{fig:central3}, $B_1$ is preceding $B_2$.\\
If a bonsai $B_\ell$ is on the left of $\{e_1,e_\rho\}$, then denote by $v_{\ell ,1}$ (respectively, $v_{\ell ,2}$) the node of $B_\ell$ which is the closest to $w_1$ (respectively, $w_{\rho-1}$) in $G_0(A)$.
If a bonsai $B_\ell$ is on the right of $\{ e_1, e_\rho \}$, we denote by $v_\ell$ the closest vertex of $B_\ell$ to the central node $w_\rho$ ($v_\ell$ may be equal to $w_\rho$). 

For a bonsai $E_\ell$, under certain conditions and provided that $A$ is $\{1,\rho\}$-central, a subset $E_{\ell j}$ of $E_\ell$, with $j\in f^*(E_\ell)$, can be interpreted as the edge index set of a path.

\begin{lem}\label{lemCentralElj}
Suppose that $A$ has a $\{1,\rho\}$-central representation $G(A)$. Let $j\in f_{S_k}(E_\ell)$ for some $1\le \ell \le b$ and $k\in \{0,1,2\}$. If $k\in \{1,2\}$, or $k=0$ and $B_\ell$ is $S_0$-straight, then the following holds:

\begin{itemize}

\item[(i)]  The set $E_{\ell j}$ is the edge index set of the unique $B_\ell$-path generated by $f_j$, and $g_j(E_\ell)=1$.

\item[(ii)] If $B_\ell$ is on the left of $\{e_1,e_\rho\}$, then the $B_\ell$-path generated by $f_j$  leaves $v_{\ell ,1}$ for $k=1$, and enters $v_{\ell ,2}$ for $k=2$.\\
If $B_\ell$ is on the right of $\{e_1,e_\rho\}$, then the $B_\ell$-path generated by $f_j$  leaves $v_\ell$.

\end{itemize}

\end{lem}

\begin{proof}
If $k\in \{1,2\}$, then by Corollary \ref{corBidirectedCircuit},
the basic fundamental circuit of $f_j$ is a non-closed
path. Hence, $f_j$ generates exactly one $B_\ell$-path with edge index set $E_{\ell j}$ and, using Lemma \ref{lemdefiWeight1}, $s_2(A_{\bullet j})=\emptyset$. So, by Proposition \ref{propBlsubstems}, $g_j(E_\ell)=1$.
Similarly, if $k=0$ and $B_\ell$ is $S_0$-straight, since $g_j(E_\ell)=1$, by Proposition \ref{propBlsubstems},
$E_\ell^{II}(A_{\bullet j})=\emptyset$ and $E_\ell^I(A_{\bullet j})=E_{\ell j}$  is the edge index set of the unique $B_\ell$-path generated by $f_j$. Then part (ii) is implied by Lemma \ref{lemdefiWeight2} since $A$ is nonnegative, and the fact that $e_1=[w_1,w_\rho]$ and $e_\rho=]w_{\rho-1}, w_\rho]$.
\end{proof}\\

Suppose that $A$ has a $\{1,\rho\}$-central representation $G(A)$. By Lemma \ref{lemCentralElj}, for all $1\le \ell\le b$ and  $j\in f_{S_k} (E_\ell)$, if $k\in \{1,2\}$, or $k=0$ and $B_\ell$ is $S_0$-straight, then $E_{\ell j}$   is the edge index set of a $B_\ell$-path, denoted as $B_{\ell j}$.


\begin{lem}\label{lemCentralElS}
Suppose that $A$ has a $\{1,\rho\}$-central representation $G(A)$. Let $B_\ell$ be a bonsai on the left of $\{e_1,e_\rho\}$. Then 

\begin{itemize}

\item[(i)]  $E_{\ell S_k}$ is the edge index set of an out-rooted (resp., in-rooted) tree for $k=1$ (resp., $k=2$).

\item[(ii)] For all $j\in S_0$, $E_{\ell j}$ is the edge index set of edges belonging to $B_\ell$ and the basic cycle.

\end{itemize}

\end{lem}

\begin{proof}
The proof of part (i) directly follows from Lemma \ref{lemCentralElj}. Now let $j\in f_{S_0}(E_\ell)$. Since $A$ is nonnegative, by Corollary \ref{corBidirectedCircuit} and Lemma \ref{lemdigraph12}, the fundamental circuit of $f_j$ is a handcuff. Then, by Lemma \ref{lemdefiWeight2}, it follows that the fundamental circuit of $f_j$ intersected with $G_0(A)$ is equal to the basic path from $w_1$ to $w_{\rho-1}$ (contained in the basic cycle. this concludes the proof.
\end{proof}\\

Provided that $A$ has a $\{1,\rho\}$-central representation $G(A)$, by Lemma \ref{lemCentralElS} (i), we denote by $T_{\ell ,k}$ the subgraph of $G(A)$ with edge index set $E_{\ell  S_k}$ (see Figure \ref{fig:central4}).

\begin{figure}[t!]
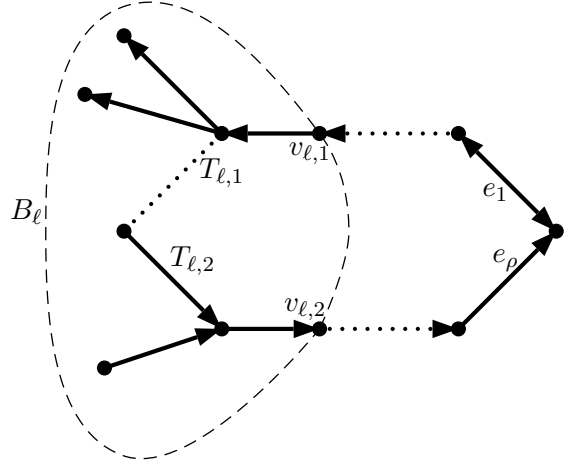


\begin{center}

\psset{xunit=1.3cm,yunit=1.3cm,linewidth=0.5pt,radius=0.1mm,arrowsize=7pt,
labelsep=1.5pt,fillcolor=black}

\pspicture(-2,-1)(3.7,3.5)

\pscircle[fillstyle=solid](-1,1){.1}
\pscircle[fillstyle=solid](0,0){.1}
\pscircle[fillstyle=solid](0,2){.1}
\pscircle[fillstyle=solid](1,0){.1}
\pscircle[fillstyle=solid](1,2){.1}
\pscircle[fillstyle=solid](2.42,0){.1}
\pscircle[fillstyle=solid](2.42,2){.1}
\pscircle[fillstyle=solid](3.42,1){.1}
\pscircle[fillstyle=solid](-1.4,2.4){.1}
\pscircle[fillstyle=solid](-1,3){.1}
\pscircle[fillstyle=solid](-1.2,-.4){.1}

\pscurve[arrowinset=.5,arrowlength=1.5,linestyle=dashed]{-}(1,0)(1.3,1)(1,2)(-1,3.3)(-1.8,1)(-1,-1.3)(1,0)

\rput(0,1.6){$T_{\ell ,1}$}
\rput(-.3,.7){$T_{\ell ,2}$}
\rput(.9,1.8){$v_{\ell ,1}$}
\rput(.85,.2){$v_{\ell ,2}$}
\rput(-2,1.2){$B_\ell$}

\psline[linewidth=1.6pt,arrowinset=0]{<->}(3.42,1)(2.42,2)
\rput(2.8,1.4){$e_{1}$}

\psline[linewidth=1.6pt,arrowinset=0,linestyle=dotted]{->}(2.42,2)(1,2)

\psline[linewidth=1.6pt,arrowinset=0]{<-}(0,2)(1,2)

\psline[linewidth=1.6pt,arrowinset=0,linestyle=dotted]{-}(0,2)(-1,1)

\psline[linewidth=1.6pt,arrowinset=0]{->}(-1,1)(0,0)

\psline[linewidth=1.6pt,arrowinset=0]{->}(0,0)(1,0)

\psline[linewidth=1.6pt,arrowinset=0,linestyle=dotted]{<-}(2.42,0)(1,0)

\psline[linewidth=1.6pt,arrowinset=0]{->}(2.42,0)(3.42,1)
\rput(2.9,.7){$e_{\rho}$}

\psline[linewidth=1.6pt,arrowinset=0]{->}(0,2)(-1,3)

\psline[linewidth=1.6pt,arrowinset=0]{->}(0,2)(-1.4,2.4)

\psline[linewidth=1.6pt,arrowinset=0]{<-}(0,0)(-1.2,-.4)

\endpspicture 
\end{center}

\caption{An illustration of the trees $T_{\ell ,1}$ and $T_{\ell ,2}$ whose edge index sets are $E_{\ell  S_1}$ and $E_{\ell  S_2}$, respectively, for some bonsai $E_\ell$.} 
\label{fig:central4}

\end{figure}

Let us see now a first enlightenment on sensitive bonsais.

\begin{lem}\label{lemroot}
Suppose that $A$ has a $\{1,\rho\}$-central  representation $G(A)$. Let $B_\ell$ be a shared bonsai on the left of $\{e_1, e_\rho\}$ in $G(A)$. Then we have
\begin{itemize}

\item If $v_{\ell ,1}=v_{\ell ,2}$, then $B_\ell$ is sensitive and $v_{\ell ,1}$ is the unique node of $B_\ell$ belonging to the basic cycle.

\item If $v_{\ell ,1} \neq v_{\ell ,2}$, then the path from $v_{\ell ,1}$ to $v_{\ell ,2}$ in $B_\ell$ lies on the basic cycle.

\end{itemize}
\end{lem}

\begin{proof}
Assume first that $v_{\ell ,1} =v_{\ell ,2}$. Suppose, to the contrary, that $v_{\ell ,1}$ does not belong to the basic cycle. Let $j_1 \in f_{S_1}(E_\ell)$, $j_2 \in f_{S_2}(E_\ell)$ and $e_i$ be a basic edge on the basic path from $v_{\ell ,1}$ to the basic cycle. By Lemma \ref{lemCentralElj}, the $B_\ell$-paths generated by $f_{j_1}$ and $f_{j_2}$ are leaving and entering $v_{\ell ,1}$, respectively. Then, since $e_i$ is in the fundamental circuit of $f_{j_1}$ and $f_{j_2}$ and using Lemma \ref{lemdefiWeight2}, we get a contradiction. Hence $v_{\ell ,1}$ is on the basic cycle. From the definition of $v_{\ell ,1}$ and $v_{\ell ,2}$, it follows that $v_{\ell ,1}$ is the unique node of $B_\ell$ belonging to the basic cycle. So
by Lemma \ref{lemCentralElS} (ii),
$f_{S_0}(E_\ell)= \emptyset$. Furthermore, using Lemmas 
\ref{lemCentralElj} and \ref{lemdigraphutile},  we deduce that $\{E_{\ell j} \, : \, j \in f_{S_1}(E_\ell) \} \subseteq J_\ell^1 $ and 
$\{E_{\ell j} \, : \, j \in f_{S_2}(E_\ell) \} \subseteq J_\ell^2$ or conversely, and $B_\ell$ is a $v_{\ell ,1}$-rooted network representation of the bonsai matrix $N_\ell$. Thus $B_\ell$ is sensitive.

Now assume $v_{\ell ,1} \neq v_{\ell ,2}$. Since $B_\ell$ is connected, there exists a (basic) path in $B_\ell$ from $v_{\ell ,1}$ to $v_{\ell ,2}$. Observe also that there exist a basic path from $v_{\ell ,1}$ to $w_1$ and an other from $v_{\ell ,2}$ to $w_{\rho-1}$ in $G_0(A)$ that do not intersect and do not contain any edge of $B_\ell$. This completes the proof.
\end{proof}\\

To make the assumption $\mathscr{A}$, we will need Lemmas \ref{lem2rid}, \ref{lemalpha} and \ref{lembeta} below. In the 
case where this assumption is not satisfied, a matrix $A'$ is obtained by adding some columns to $A$, and we consider a digraph $D'$ with respect to $A'$. By Lemma \ref{lemalpha}, it turns out that solving the recognition problem on $A$ is equivalent to solving it on $A'$, and $D'$ has one bonsai less than $D$. By adding at most $m$ columns to $A$, the assumption $\mathscr{A}$ will be satisfied. Let us see an auxiliary lemma.

\begin{lem}\label{lemright}
Suppose that $A$ has a $\{1,\rho\}$-central  representation $G(A)$.
Let $1\le \ell\le b$. If $B_\ell$ is on the right of $\{e_1 , e_\rho\} $ in $G(A)$, then $J_\ell^2 = \emptyset$ and $N_\ell$ is a network matrix.
\end{lem}

\begin{proof}
Using Lemma \ref{lemdefiWeight2}, it follows that
all $B_\ell$-paths leave $v_\ell$. Hence, by Lemma \ref{lemdigraphutile}, $J_\ell^2=\emptyset$. Moreover, $B_\ell$ is a $v_\ell$-rooted network representation of $N_\ell$.
\end{proof}\\

\begin{lem}\label{lem2rid}
Suppose that the matrix $A$ has a $\{1,\rho\}$-central  representation $G(A)$. Let $B_\ell$ be a sensitive bonsai ($1\le \ell \le b$). Then the following holds.
\begin{itemize}

\item[1)] The bonsai $B_\ell$ is on the left of $\{ e_1, e_\rho \}$.

\item[2)] If $v_{\ell ,1}\neq v_{\ell ,2}$, then for any shared bonsai $B_{\ell '}$ on the left of $\{ e_1,e_\rho\}$, we have $v_{\ell ',1}=v_{\ell ',2} \in B_\ell$ and $B_{\ell '}$ is $S_1,S_2$ linked to $B_\ell$.

\end{itemize}
\end{lem}

\begin{proof}
If $B_\ell$ is on the right of  $\{ e_1, e_\rho \}$, then
by Lemma \ref{lemright}  $J_\ell^2=\emptyset$,
contradicting the definition of a sensitive bonsai. Thus $B_\ell$ is on the left of $\{ e_1, e_\rho \}$. 

Now assume that $v_{\ell ,1} \neq v_{\ell ,2}$. Let $B_{\ell '}$ be a shared bonsai on the left of $\{ e_1,e_\rho\}$. By Lemma \ref{lemroot}, the nodes $v_{\ell ,1}$, $v_{\ell ,2}$, $v_{\ell '}^1$ and $v_{\ell '}^2$ belong to the basic cycle.
Suppose, by contradiction, that for some $k\in \{ 1,2\}$, $v_{\ell '}^k=w_{i_0}$ is such that $i_0 \geq i$ for all $w_i \in B_\ell$. This implies that  $I(v_{\ell ,1},v_{\ell ,2})=E_{\ell j}$ for all $j\in f_{S_1}(E_{\ell '})$. Let $j_1 \in f_{S_1}(E_{\ell '})$ and $j_2 \in f_{S_2}(E_\ell)$. Since $E_{\ell  j_1}=I(v_{\ell ,1},v_{\ell ,2})$,
the $B_\ell$-path $B_{\ell  j_1}$ enters $v_{\ell ,2}$, and by Lemma   \ref{lemCentralElj} (ii) $B_{\ell  j_2}$ enters $v_{\ell ,2}$. So, since $N_\ell$ is a network matrix, by Lemma \ref{lembonsaicontra} , $E_{\ell j_1} \sim_{E_\ell} E_{\ell j_2}$,
contradicting the definition of $B_\ell$. One gets a similar contradiction if $i_0 \le i$ for all $w_i \in B_\ell$. It results that $v_{\ell ',1},v_{\ell ',2} \in B_\ell$ and $v_{\ell ',k}\neq v_{\ell ,1},v_{\ell ,2}$ for $k=1$ and $2$. Since $E_\ell \cap E_{\ell '} = \emptyset$, $v_{\ell ',1}=v_{\ell ',2}$. At last, as $v_{\ell ',1}$
is an inner node on the path $p(v_{\ell ,1},v_{\ell ,2})$ and $B_{\ell '}$ is shared, the bonsais $B_\ell$ and $B_{\ell '}$ are $S_1,S_2$ linked.
\end{proof}\\

\begin{lem}\label{lemalpha}
Let $E_\ell$ be a sensitive bonsai and $E_{\ell '}$ some bonsai $S_1,S_2$ linked to $E_\ell$ such that $J_{\ell '}^2\neq \emptyset$ or $N_{\ell '}$ is not a network matrix. Then for any $j_1 \in f_{S_1}(E_\ell) \cap f_{S_1}(E_{\ell '})$, the matrix $A$ is $\{1,\rho\}$-central if and only if the matrix $A'=[A\, \, \chi^{\{1,\ldots,n \}}_{E_{\ell j_1} \cup E_{\ell 'j_1}} ]$
is $\{1,\rho\}$-central.
\end{lem}

\begin{proof}
Suppose that $A$ has a $\{1,\rho\}$-central  representation $G(A)$. Since $J_{\ell '}^2\neq \emptyset$ or $N_{\ell '}$ is not a network matrix, Lemma \ref{lemright}  implies that $B_{\ell '}$ is on the left of $\{ e_1,e_\rho\}$.
As $B_{\ell '}$ is $S_1$, $S_2$ linked to $B_\ell$, $B_{\ell '}$ is shared. So, if $v_{\ell ,1}\neq v_{\ell ,2}$, then by Lemma \ref{lem2rid} $v_{\ell ',1}=v_{\ell ',2} \in B_\ell$. 
Suppose now that $v_{\ell ,1}= v_{\ell ,2}$. Since $B_\ell$ 
and $B_{\ell '}$ are shared, by Lemma \ref{lemroot}, $v_{\ell ,1}=w_{i_0}$ for some $2\le i_0 \le \rho-1$, and $v_{\ell ',1}$ and $v_{\ell ',2}$ lie on the basic cycle. Hence, if $i_0 \geq i $ for all $w_i \in B_{\ell '}$, then $f_{S_2}(E_\ell) \cap f_{S_2}(E_{\ell '}) = \emptyset $, which contradicts the fact that $B_\ell$ and $B_{\ell '}$ are $S_1,S_2$ linked. If $i_0 \le i$ for all $w_i \in B_{\ell '}$, then $f_{S_1}(E_\ell) \cap f_{S_1}(E_{\ell '})=\emptyset$ and we get a same contradiction. Therefore $w_{i_0}$ is an inner node on the path $p(v_{\ell ',1}, v_{\ell ',2}) \subseteq B_{\ell '}$. Whatever we have $v_{\ell ',1}=v_{\ell ',2} \in B_\ell$ or $v_{\ell ,1}=v_{\ell ,2} \in B_{\ell '}$, it yields 
that $E_{\ell j_1} \cup E_{\ell 'j_1}$ is the edge index set of a path in $G(A)$ for any $j_1 \in f_{S_1}(E_\ell) \cap f_{S_1}(E_{\ell '})$. Therefore $A'$ is $\{1,\rho\}$-central.
The converse part is trivial.
\end{proof}\\

\begin{figure}[h!]
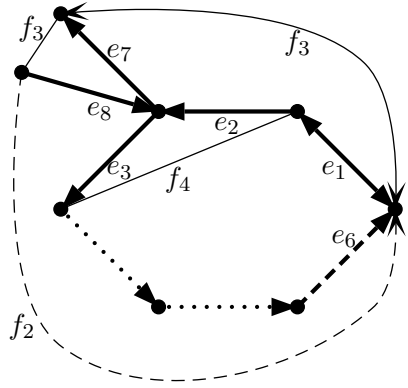


\begin{center}

\psset{xunit=1.3cm,yunit=1.3cm,linewidth=0.5pt,radius=0.1mm,arrowsize=7pt,
labelsep=1.5pt,fillcolor=black}

\pspicture(-2,-0.5)(3.7,3.5)

\pscircle[fillstyle=solid](0,1){.1}
\pscircle[fillstyle=solid](1,0){.1}
\pscircle[fillstyle=solid](1,2){.1}
\pscircle[fillstyle=solid](2.42,0){.1}
\pscircle[fillstyle=solid](2.42,2){.1}
\pscircle[fillstyle=solid](3.42,1){.1}
\pscircle[fillstyle=solid](-.4,2.4){.1}
\pscircle[fillstyle=solid](0,3){.1}

\psline[linewidth=1.6pt,arrowinset=0]{<->}(3.42,1)(2.42,2)
\rput(2.8,1.4){$e_{1}$}

\psline[linewidth=1.6pt,arrowinset=0]{->}(2.42,2)(1,2)
\rput(1.7,1.85){$e_{2}$}

\psline[linewidth=1.6pt,arrowinset=0]{->}(1,2)(0,1)
\rput(.6,1.4){$e_{3}$}

\psline[linewidth=1.6pt,arrowinset=0,linestyle=dotted]{->}(0,1)(1,0)

\psline[linewidth=1.6pt,arrowinset=0,linestyle=dotted]{->}(1,0)(2.42,0)

\psline[linewidth=1.6pt,arrowinset=0,linestyle=dashed]{->}(2.42,0)(3.42,1)
\rput(2.9,.7){$e_{6}$}

\psline[linewidth=1.6pt,arrowinset=0]{->}(1,2)(0,3)
\rput(.6,2.6){$e_{7}$}

\psline[linewidth=1.6pt,arrowinset=0]{<-}(1,2)(-.4,2.4)
\rput(.4,2){$e_{8}$}



\pscurve[arrowinset=.5,arrowlength=1.5,linestyle=dashed]{->}(-.4,2.4)(0,-.4)(3.2,0)(3.42,1)
\rput(-.4,-.2){$f_{2}$}

\pscurve[arrowinset=.5,arrowlength=1.5]{<->}(0,3)(3.1,2.5)(3.42,1)
\rput(2.42,2.7){$f_{3}$}



\psline[arrowinset=.5,arrowlength=1.5]{-}(0,3)(-.4,2.4)
\rput(-.3,2.85){$f_{3}$}


\psline[arrowinset=.5,arrowlength=1.5]{-}(0,1)(2.42,2)
\rput(1.2,1.3){$f_{4}$}




\endpspicture 

\end{center}
\vspace{.3cm}
\caption{an example of a (partial) $\{1,6\}$-central representation of some binet matrix, where $E_1=\{7,8\}$ and $E_2=\{2,3\}$ are sensitive, $v_{2,1}\neq v_{2,2}$, and $E_1$ and $E_2$ are $S_1$,$S_2$ linked.} 
\label{fig:central5}

\end{figure}

\begin{lem}\label{lembeta}
Suppose that $A$ has a $\{1,\rho\}$-central  representation $G(A)$.
Let $B_\ell$ be a sensitive bonsai and $B_{\ell '}$ some bonsai $S_1$,$S_2$ linked to $B_\ell$. Then, $B_{\ell '}$ is on the left of $\{e_1,e_\rho\}$ if and only if $J_{\ell '}^2\neq \emptyset$ or $N_{\ell '}$ is not a network matrix.
\end{lem}

\begin{proof}
The proof of the "if" part follows from Lemma \ref{lemright}.

Suppose that $B_{\ell '}$ is on the left of $\{e_1,e_\rho\}$.
Since $B_{\ell '}$ is $S_1$, $S_2$ linked to $B_\ell$, $B_{\ell '}$ is shared. So, if $v_{\ell ,1}\neq v_{\ell ,2}$, then by Lemma \ref{lem2rid} $v_{\ell ',1}= v_{\ell ',2}$.
And Lemma \ref{lemroot} implies that $B_{\ell '}$ is sensitive.
In particular, $J_{\ell '}^2 \neq \emptyset$.

Now assume $v_{\ell ,1}=v_{\ell ,2}$. By Lemma \ref{lemroot}
$v_{\ell ,1}=w_{i_0}$ for some $2\le i_0 \le \rho-1$.  One can prove in the same way as in the proof of Lemma \ref{lemalpha} that 
$w_{i_0}$ is an inner node on the path $p(v_{\ell '}^1,v_{\ell '}^2)$
and $v_{\ell '}^1 \neq v_{\ell '}^2$ (otherwise $B_\ell$ and $B_{\ell '}$ would not be $S_1,S_2$ linked). Thus by Lemma \ref{lemdefiWeight2}, for any $k\in \{1,2\}$ and  $j_k \in f_{S_k}(E_\ell)\cap f_{S_k}(E_{\ell '})$,
the $B_\ell$-paths $B_{\ell  j_1}$ and $B_{\ell  j_2}$ leave and enter $v_{\ell ,1}=v_{\ell ,2}$, respectively; then it results that 
the $B_{\ell '}$-paths $B_{\ell ' j_1}$ and $B_{\ell ' j_2}$ enter and leave $v_{\ell ,1}$, respectively. 
If $N_{\ell '}$ is a network matrix, then by Lemma \ref{lembonsaicontra} $E_{\ell ' j_1} \nsim_{E_{\ell '}} E_{\ell ' j_2}$, so  $J_{\ell '}^2\neq \emptyset$. This completes the proof.
\end{proof}\\

The advantage of assumption $\mathscr{A}$ is revealed in the next lemma.

\begin{lem}\label{lembeta2}
Suppose that $A$ has a $\{1,\rho\}$-central  representation $G(A)$ and assumption $\mathscr{A}$ is satisfied. Then, for any sensitive bonsai $B_\ell$, there is no bonsai on the left of $\{e_1,e_\rho\}$ that is $S_1$, $S_2$ linked to $B_\ell$.
\end{lem}

\begin{proof}
This directly follows from assumption $\mathscr{A}$ and Lemma \ref{lembeta}.
\end{proof}\\

In Figure \ref{fig:central5}, there is an example of a central representation of some binet matrix such that 
assumption $\mathscr{A}$  is not satisfied.
The following procedure is the first subroutine of the procedures CentralI and CentralII. It provides a matrix $A'$ of size $n\times m'$ ($m' \geq m$) such that $A'_{\bullet \{1,\ldots, m\}}=A$ and satisfying assumption $\mathscr{A}$.

\begin{tabbing}
\textbf{Procedure\,\,Initialization($A$,$\{\epsilon,\rho\}$)}\\

\textbf{Input:} A non-network matrix $A$ and two row indexes $\epsilon$ and $\rho\neq 1$.\\
\textbf{Output: }\= A matrix $A'$  such that $A'$ is $\{1,\rho\}$-central if and only if $A$ is $\{\epsilon,\rho\}$-central, \\
\> and assumption $\mathscr{A}$ is satisfied for $A'$.\\

1)\verb"  "\= let $A'=A$, permute the row of $A'$ labeled $\epsilon$ and the first one;\\
2) \> let $R^*=\{1,\rho\}$ and partition $\overline{R^*}$ into $E_1,\dots,E_b$ as described  in Section \ref{sec:DefDigraphD};\\ 
3) \> {\bf while }\= there exist  $1\le \ell,\ell' \le b$ such that  $E_\ell$ is sensitive, $E_{\ell '}$ is $S_1$,$S_2$ linked to $E_\ell$, \\
\> \> and $J_{\ell '}^2\neq \emptyset$ or $N_{\ell '}$ is not a network matrix, {\bf do}\\
\> \> choose $j_1 \in f_{S_1}(E_\ell) \cap f_{S_1}(E_{\ell '})$ and let $A'=[ A' \, \chi_{E_{\ell  j_1} \cup E_{\ell ' j_1}}^{\{1,\ldots, n\}} ]$;\\
\> \>  recompute the partition of $\overline{R^*}$ into $E_1,\dots,E_b$ with respect to $A'$;\\
\> {\bf endwhile} \\
\> output the matrix $A'$;

\end{tabbing}

\begin{lem}\label{lemCentralInit}
The output of the procedure Initialization is correct. The 
output matrix $A'$ has at most $2m$ columns.
\end{lem}

\begin{proof}
Clearly, after each passage through step 3, the number of vertices in $D$ decreases by $1$. Moreover, if some bonsai $E_\ell$ is sensitive and $E_{\ell '}$ is such that $J_{\ell '}^2\neq \emptyset$ or $N_{\ell '}$ is not a network matrix, then $|E_\ell| \geq 2$ and $|E_{\ell '}| \geq 2$, which implies $|\bar f^*(E_\ell)| \geq 1$ and $|\bar f^*(E_{\ell '})| \geq 1$. So the number of columns added to $A$ does not exceed $m$. By step 3, the assumption $\mathscr{A}$ is satisfied for $A'$, and by step 1 and Lemma \ref{lemalpha}, $A'$
is $\{1,\rho\}$-central if and only if $A$
is $\{\epsilon ,\rho\}$-central.
\end{proof}\\

Finally, we state a proposition that will be used in Sections \ref{sec:S0empty} and \ref{sec:S0notempty}.

\begin{prop}\label{propSkwatered}
Suppose that $A$ has a $\{1,\rho\}$-central representation $G(A)$. Then, for any bonsai $B_\ell$ that is $S_k$-dominated for some $k\in \{ 1,2\}$, $J_\ell^2=\emptyset$ and the bonsai matrix $N_\ell$ is a network matrix.
\end{prop}

\begin{proof}
Let $B_\ell$ be some $S_1$-dominated bonsai. (The case $S_2$ instead of $S_1$ is similar.) If $B_\ell$ is on the left of $\{e_1,e_\rho\}$, then all $B_\ell$-paths are leaving $v_{\ell ,1}$. 
By Lemma \ref{lemdigraphutile} and definition of $\sim_{E_\ell}$, $J_\ell^2=\emptyset$. Clearly,
the bonsai $B_\ell$ is a  $v_{\ell ,1}$-rooted
network representation of $N_\ell$. If $B_\ell$ is on the right of $\{e_1,e_\rho\}$, the conclusion results from Lemma \ref{lemright}.
\end{proof}\\

\section{The procedure CentralI}\label{sec:S0empty}

Througout this section we assume that $S_0 = \emptyset$, and $\epsilon=1$ (except in the procedures). We provide a proof of Theorem \ref{thmS0empty}. Suppose that $A$ has a $\{1,\rho\}$-central representation $G(A)$. If we look at the succession of bonsais intersecting the path from $w_1$ to $w_{\rho-1}$ in $G_0(A)$, we may identify a first group (maybe empty) of $S_1$-dominated  bonsais, a second with shared bonsais and a third with 
$S_2$-dominated  bonsais. This motivates the following definition. Given  a $\{1,\rho\}$-central representation $G(A)$ of $A$, a nonempty ordered set $U$ of at most two shared bonsais  is said to be \emph{left-extreme}\index{left-extreme} if the following holds.

\begin{itemize}
\item[a) ] $U=(E_u)$.\\
The bonsai $B_u$ intersects the basic cycle on the left of $\{e_1,e_\rho\}$ and for any shared bonsai $B_\ell$ intersecting the basic cycle on the left of $\{e_1,e_\rho\}$, we have $v_{\ell ,1}=v_{\ell ,2}\in B_u$.

\item[b) ] $U=(E_u,E_{u'})$. \\
The bonsais $B_u$ and $B_{u'}$ contain each one at least one edge of the basic cycle, $B_u$ is preceding $B_{u'}$ and any shared bonsai containing edges of the basic cycle is succeeding $B_{u}$ and preceding $B_{u'}$. 
\end{itemize}

In the case where $|U|=2$, $U$ is also called a \emph{left-extreme pair}. When it does not exist a shared bonsai containing at least one edge of the basic cycle, a left-extreme set may be not unique. In Figure \ref{fig:central2}, the ordered pair $(E_1,E_2)$ is left-extreme.
The notion of left-extreme set is very important and has inspired the definition of left-compatible set which will be given later. Lemma \ref{lemCellularleft-extremecomp} below shows that a left-extreme set is left-compatible. The procedure CentralI is based on a subroutine which is performed for every left-compatible set, as long as a $\{1,\rho\}$-central representation of $A$ has not been found. On the other hand, a $U$-spanning pair $(j_1,j_2)$ is such that $j_1\in S_1$, $j_2 \in S_2$ and if $|U|=2$ and $A$ has a $\{1,\rho\}$-central representation $G(A)$, then the whole basic cycle in $G(A)$
is "spanned" by the union of the fundamental circuits of $f_{j_1}$ and $f_{j_2}$. The following proposition deals with the case where there is no shared bonsai on the left of $\{e_1,e_\rho\}$ in some $\{1,\rho\}$-central representation of $A$.

\begin{prop}\label{propCellularnetwork}
Suppose that $S_0=\emptyset$ and there exists  a $\{1,\rho\}$-central  representation $G(A)$ of $A$ such that  all bonsais on the left of $\{ e_1, e_\rho\} $ are not shared. Then $A$ is a network matrix.
\end{prop}

\begin{proof}
Let $G(A)$ be a basic $\{1,\rho\}$-central  representation of $A$ such that all bonsais on the left of $\{ e_1, e_\rho\} $ are not shared. So one can cut the cycle at some node, by separating the $S_1$-dominated bonsais from the $S_2$-dominated ones. Then, replacing the bidirected edge $e_1=[w_1,w_\rho]$ by $]w_1, w_\rho]$ and reversing the orientation of all edges in the maximal subtree containing $w_1$ but not $e_1$ yields a network representation of $A$.
\end{proof}\\

Using Proposition \ref{propCellularnetwork}, since $A$ is not a network matrix, whenever $A$ is $\{1,\rho\}$-central, there always exists a left-extreme set of bonsais. Before formulating the notions of left-compatible set and $U$-spanning pair precisely, we study the following problem. Provided that $A$ has a $\{1,\rho\}$-central representation and given a shared bonsai $B_\ell$ on the left of $\{e_1,e_\rho\}$, is it possible to compute the indexes of edges belonging to the basic cycle and $B_\ell$? Surprisingly, the answer is yes, provided that the left-extreme set of bonsais is known and of cardinality $2$, and using the nonnegativity of the matrix $A$.

For any shared bonsai $E_u$ ($1\le u \le b$), we define $$I_{\cap}(E_u) = \underset{j\in S_1, j' \in S_2}{\cup} E_u\cap s(A_{\bullet j})\cap s(A_{\bullet j'})$$ 
and $$I(E_u)=  \underset{E_\ell :\, \m{sensitive} }{ \underset{
j\in f^*(E_\ell) \cap f^*(E_u)}{\cup}}  E_{uj}. $$ A vertex $\beta$ in the local connecter set of $E_u$ ($\beta \in \bar f^*(E_u)$) is said to be \emph{$S_k$-blue}\index{blue@$S_k$-blue}, for some $k\in \{ 1,2 \}$, if $s(A_{\bullet \beta}) \cap E_{u S_k} \neq \emptyset$. For any shared bonsai $E_u$ ($1\le u \le b$), under a certain condition, the following procedure computes a subset of $E_u$ denoted by $I_k(E_u)$, for some $k\in \{1,2\}$, whose interpretation appears in the next lemma.

\begin{tabbing}
\textbf{Procedure\,\,Create-Ik($E_u$,$k$)}\\

\textbf{Input:} A shared bonsai $E_u$ and an index $k\in \{1,2\}$.\\
\textbf{Output: }\= Either a row index subset $I_k(E_u)$ of $E_u$, or determines that $A$ has no\\
\>   $\{1,\rho\}$-central representation with a left-extreme pair in which $E_u$ is the  first\\
\>   bonsai for $k=1$ (respectively, second for $k=2$). \\

1)\verb"  "\= {\bf if }\= $E_u$ is  jointly shared, {\bf then}\\ 
\> \> let $I_k(E_u)= \cap_{j\in f_{S_k}(E_u)} \{ÊE_{uj} \, : \, I_\cap(E_u) \subseteq E_{uj} \}$; \\

\>{\bf otherwise } \\

2) \> \>  compute a yellow minimal path $h$ in $H_{E_{uS^*}}(A_{E_u \times\bar f^*(E_u)})$ from some $S_1$-blue  vertex $\beta_1$ \\
\> \> to a $S_2$-blue $\beta_2$; let $R_k=s(A_{\bullet \beta_k}) \cap E_{uS_k}$;\\ 

3)  \> \> compute $j_k \in S_k$ such that $R_k \subseteq E_{uj_k}$; if such a vertex ($j_k$) in $S_k$ does not exist, then \\
\> \>STOP: output that $A$ has no $\{1,\rho\}$-central representation with a left-extreme pair, \\
\> \> in which $E_u$ is the first bonsai for $k=1$ (resp., second for $k=2$); \\

4) \> \> {\bf if }\= the length of $h$ is even, {\bf then}\\ 
 \> \> \> let $I_k(E_u)= \cup_{j\in f_{S_k}(E_u)} \{ÊE_{uj} \, : \,E_{uj} \cap R_k=\emptyset\,; \, E_{uj} \subseteq E_{uj_k} \}$;\\
\> \> {\bf otherwise}  \\
 \> \> \> let $I_k(E_u)= \cap_{j\in f_{S_k}(E_u)} \{ÊE_{uj} \, : \,   R_k \subseteq E_{uj} \}$;\\
  \> \>{\bf endif} \\
\> {\bf endif} \\
\> output $I_k(E_u)$;

\end{tabbing}

\begin{lem}\label{lemCellularexcomp}
Suppose that $A$ has a $\{1,\rho\}$-central representation $G(A)$ and assumption $\mathscr{A}$ is satisfied. Let $U$ be a left-extreme set of bonsais. Then we have
\begin{itemize}

\item[(i)] If some bonsai $E_\ell \notin U$ is jointly shared and $B_\ell$ is on the left of $\{e_1,e_\rho\}$, then $I_\cap(E_\ell)$ corresponds to the index set of edges belonging to $B_\ell$ and the basic cycle.

\item[(ii)] If $U=(E_u)$ and there exists at least one sensitive bonsai distinct from $B_u$, then $I(E_u)$ corresponds to the index set of edges belonging to $B_u$ and the basic cycle.

\item[(iii)] If $U=(E_u,E_{u'})$, then $I_1(E_u)$ (respectively, $I_2(E_{u'})$) output by the procedure Create-Ik corresponds to the index set of edges belonging to $B_u$ (respectively, $B_{u'}$) and the basic cycle.

\end{itemize}
\end{lem}

\begin{proof}
\begin{itemize}
\item[(i)] Since any edge with index in $I_\cap(E_\ell)$ lies on the basic cycle, from the definition of $U$, it follows that $|U|=2$
and $I_\cap (E_\ell) \subseteq I(v_{\ell ,1},v_{\ell ,2})$. Let $U=(E_u,E_{u'})$. By definition of $U$, $B_\ell$ is succeeding  $B_u$ and preceding $B_{u'}$. This implies that for any  $j\in f_{S_1}(E_{u'})$ and $j'\in f_{S_2}(E_u)$, $I(v_{\ell ,1},v_{\ell ,2})= E_u\cap s(A_{\bullet j}) \cap s(A_{\bullet j'}) \subseteq I_\cap (E_\ell)$.  Thus $ I_\cap(E_\ell)= I(v_{\ell ,1},v_{\ell ,1})$.

\item[(ii)] Suppose that $U=(E_u)$ and there exists at least one sensitive bonsai distinct from $B_u$. For any sensitive bonsai $B_{\ell }$, $\ell\neq u$, let us prove that $v_{\ell ,1}=v_{\ell ,2}$, and  $v_{\ell ,1}=v_{u,1}$ or $v_{\ell ,1}=v_{u,2}$. 

The fact that $v_{\ell ,1}=v_{\ell ,2}$ follows from the defintion of $U$. Moreover, by Lemma \ref{lemroot}, $v_{\ell ,1}$, $v_{u,1}$ and $v_{u,2}$ lie on the basic cycle. If $v_{\ell ,1}\notin B_u$, then we may assume that $B_\ell$ is succeeding $B_u$, hence $p(v_{\ell ,1},v_{u,2})$ is contained in the fundamental circuit of $f_j$ and $f_{j'}$ for any $j\in f_{S_1}(E_\ell)$ and $j'\in f_{S_2}(E_u)$, which contradicts the definition of $U$. Thus $v_{\ell ,1}\in B_u$. Suppose, to the contrary, that $v_{\ell ,1}$ is an inner node of the path $p(v_{u,1},v_{u,2})$. Then, for any $j\in f_{S_1}(E_\ell)$ and $j'\in f_{S_2}(E_\ell)$, $B_{u j}$ (respectively, $B_{u j'}$) is a $B_u$-path entering (respectively, leaving) $v_\ell$. 
So $B_u$ is $S_1$, $S_2$ linked to $B_\ell$, and by assumption $\mathscr{A}$, we have $J_u^2=\emptyset$ and $N_u$ is  a network matrix. Hence, since $B_u$ is  a network representation of $A_{E_u\times \bar f^*(E_u)}$,
Lemma \ref{lembonsaicontra} implies that $E_{u j} \nsim_{E_u} E_{u j'}$, which contradicts $J_u^2=\emptyset$.
Thus, $v_{\ell ,1}=v_{u,1}$ or $v_{\ell ,1}=v_{u,2}$. 

Therefore, either $f^*(E_\ell)\cap f^*(E_u)=\emptyset$ and $v_{u,1}=v_{u,2}$, or $f^*(E_\ell) \cap f^*(E_u)\neq \emptyset$ and $E_{uj}=I(v_{u,1},v_{u,2})$ for all $j\in f^*(E_\ell) \cap f^*(E_u)$. This implies that $I(E_u)$ represents the index set of edges belonging to $B_u$ and the basic cycle.

\item[(iii)] Suppose that $U=(E_{u},E_{u'})$.
We prove that $I_1(E_u)=I(v_{u,1},v_{u,2})$. (By symmetry, it can be proved that $I_2(E_{u'})=I(v_{u',1},v_{u',2})$.) 
By definition of a left-extreme pair, $B_{u}$ contains at least one edge of the basic cycle and is shared. Therefore, by Lemma \ref{lemroot}, $v_{u,1}\neq v_{u,2}$ and the path $p(v_{u,1},v_{u,2})$ corresponds to the intersection of $B_{u}$ with the basic cycle in $G(A)$. Notice that
$f_{S_1}(E_{u'})\neq \emptyset$ and
$p(v_{u,1},v_{u,2})$ is in the fundamental circuit of any edge with index in $f_{S_1}(E_{u'})$. Thus 

\begin{equation}\label{eqnCellularex}
I(v_{u,1},v_{u,2})=E_{uj} \m{ for all } j\in f_{S_1}(E_{u'}).
\end{equation}

Moreover, recall that $T_{u,k}$ denotes the tree with edge index set $E_{u S_k}$ for $k=1$ and $2$. By Lemma \ref{lemCentralElS} (i) and statement (\ref{eqnCellularex}), it results that 

\begin{equation}\label{eqnCellularexT}
T_{u,1}\cap T_{u,2} \m{ is a subpath of } p(v_{u,1},v_{u,2})
\m{ with one endnode equal to } v_{u,2}.
\end{equation}


\quad \, Suppose first that $B_{u}$ is  jointly shared. 
So $I_\cap(E_{u})\neq \emptyset$. Let  $j\in f_{S_1}(E_{u})$ such that $I_\cap(E_{u}) \subseteq E_{uj}$.
Since $I_\cap(E_{u})$ is the edge index set of the subgraph $T_{u,1}\cap T_{u,2}$, by (\ref{eqnCellularexT}), it follows that the node $v_{u,2}$ is contained in the $B_u$-path $B_{uj}$,  and clearly $v_{u,1} \in B_{uj}$; hence $I(v_{u,1},v_{u,2}) \subseteq E_{uj}$. So $I(v_{u,1},v_{u,2}) \subseteq I_1(E_{u})$. Moreover,
using (\ref{eqnCellularex}), we have that 
\begin{eqnarray*}
I_1(E_{u})&  = & \cap_{j\in f_{S_1}(E_u)} \{ÊE_{uj} \, : \, I_\cap(E_u) \subseteq E_{uj} \} \\
& \subseteq  & \cap_{j\in f_{S_1}(E_{u'})} \{ÊE_{uj} \, : \, I_\cap(E_u) \subseteq E_{uj} \} \\
& =  & I(v_{u,1},v_{u,2}).
\end{eqnarray*}
We conclude that $I(v_{u,1},v_{u,2}) = I_1(E_{u})$. See Figure \ref{fig:central6} for an example.

\begin{figure}[h!]
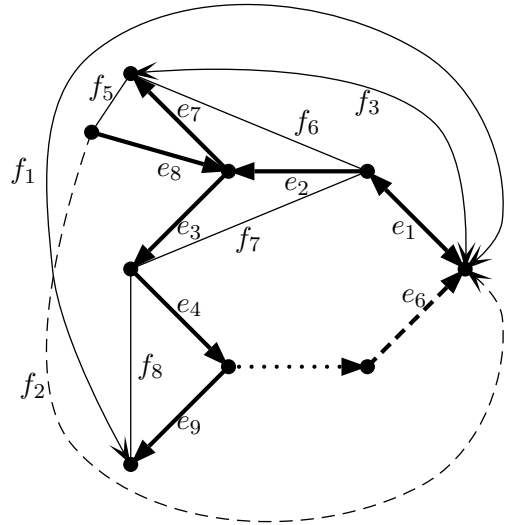

\vspace{.4cm}
$
\begin{array}{cc}

\begin{array}{c}
S_1=\{1,3\}; \,\,\, S_2=\{2\};\\
\\
E_1=\{2,3,7,8\}; \,\,\, E_2=\{4,9\};\\ 
\\
E_{11}=\{2,3\}; \,\,\, E_{13}=\{2,7\}; \,\,\, E_{12}=\{3,8\};\\
I_\cap(E_1)=\{3\};
 
\end{array}   &

\psset{xunit=1.3cm,yunit=1.3cm,linewidth=0.5pt,radius=0.1mm,arrowsize=7pt,
labelsep=1.5pt,fillcolor=black}

\pspicture(-2,1)(3.7,3.5)

\pscircle[fillstyle=solid](0,1){.1}
\pscircle[fillstyle=solid](1,0){.1}
\pscircle[fillstyle=solid](1,2){.1}
\pscircle[fillstyle=solid](2.42,0){.1}
\pscircle[fillstyle=solid](2.42,2){.1}
\pscircle[fillstyle=solid](3.42,1){.1}
\pscircle[fillstyle=solid](-.4,2.4){.1}
\pscircle[fillstyle=solid](0,3){.1}
\pscircle[fillstyle=solid](0,-1){.1}

\psline[linewidth=1.6pt,arrowinset=0]{<->}(3.42,1)(2.42,2)
\rput(2.8,1.4){$e_{1}$}

\psline[linewidth=1.6pt,arrowinset=0]{->}(2.42,2)(1,2)
\rput(1.7,1.85){$e_{2}$}

\psline[linewidth=1.6pt,arrowinset=0]{->}(1,2)(0,1)
\rput(.6,1.4){$e_{3}$}

\psline[linewidth=1.6pt,arrowinset=0]{->}(0,1)(1,0)
\rput(.6,.6){$e_{4}$}

\psline[linewidth=1.6pt,arrowinset=0,linestyle=dotted]{->}(1,0)(2.42,0)

\psline[linewidth=1.6pt,arrowinset=0,linestyle=dashed]{->}(2.42,0)(3.42,1)
\rput(2.9,.7){$e_{6}$}

\psline[linewidth=1.6pt,arrowinset=0]{->}(1,2)(0,3)
\rput(.6,2.6){$e_{7}$}

\psline[linewidth=1.6pt,arrowinset=0]{<-}(1,2)(-.4,2.4)
\rput(.4,2){$e_{8}$}

\psline[linewidth=1.6pt,arrowinset=0]{->}(1,0)(0,-1)
\rput(.6,-.6){$e_{9}$}

\pscurve[arrowinset=.5,arrowlength=1.5]{<->}(0,-1)(-.8,1)(-.4,3.2)(3,3.2)(3.8,1.5)(3.42,1)
\rput(-1.1,2){$f_{1}$}

\pscurve[arrowinset=.5,arrowlength=1.5,linestyle=dashed]{->}(-.4,2.4)(-.6,-.8)(3,-1.2)(3.8,.5)(3.42,1)
\rput(-1,-.2){$f_{2}$}

\pscurve[arrowinset=.5,arrowlength=1.5]{<->}(0,3)(3.1,2.5)(3.42,1)
\rput(2.42,2.7){$f_{3}$}



\psline[arrowinset=.5,arrowlength=1.5]{-}(0,3)(-.4,2.4)
\rput(-.3,2.85){$f_{5}$}

\psline[arrowinset=.5,arrowlength=1.5]{-}(0,3)(2.42,2)
\rput(1.8,2.5){$f_{6}$}

\psline[arrowinset=.5,arrowlength=1.5]{-}(0,1)(2.42,2)
\rput(1.2,1.3){$f_{7}$}

\psline[arrowinset=.5,arrowlength=1.5]{-}(0,1)(0,-1)
\rput(.2,0){$f_{8}$}

\endpspicture 
\end{array}
$
\vspace{2.2cm}
\caption{A (partial) $\{1,6\}$-central representation of some binet $\{0,1\}$-matrix, where the pair $(E_1,E_2)$ is left-extreme and $E_1$ is jointly shared.} 
\label{fig:central6}

\end{figure}

\quad \,Now suppose that $B_{u}$ is disjointly shared. 
Let $h$ be a yellow minimal path in 
$H_{E_{uS^*}}(A_{E_u \times\bar f^*(E_u)})$ from some $S_1$-blue  vertex $\beta_1$ to a $S_2$-blue $\beta_2$; let $R_i=s(A_{\bullet \beta_i}) \cap E_{uS_i}$ for $i=1$ and $2$. 
Denote by $p_1$ and $p_2$ the paths with edge index sets $R_1$ and $R_2$, respectively. Since $A$ is nonnegative, by Lemma \ref{lemdefiWeight2}
$p_1$ is a directed path. As $p_1$ is contained in the tree $T_{u,1}$ consisting of directed $B_u$-paths leaving the node $v_{u,1}$, $p_1$ is contained in a $B_u$-path. 
Thus there exists $j_1\in S_1$ such that $R_1 \subseteq E_{u j_1}$ and step 3 in the procedure Create-Ik is justified. Then, as $B_u$ is disjointly shared, by statement (\ref{eqnCellularex}),
$T_{u,1}\cap T_{u,2}=\{ v_{u,2}\}$. Hence, since $p_1\subseteq T_{u,1}$, $p_2\subseteq T_{u,2}$ and $h$ is a yellow path in $H_{E_{u S^*}}(A_{E_u \times \bar f^*(E_u)})$, one can prove that 
$v_{u,2}$ is a common endnode of the paths $p_1$ and $p_2$. Since $p_2$ enters $v_{u,2}$, by Lemma \ref{lemstem2} the path $h$ has an even length if and only if $p_1$ leaves $v_{u,2}$. 


\quad \, Suppose that $h$ has an even length. So $p_1$ leaves $v_{u,2}$. Since $R_1 \subseteq E_{uj_1}$, it follows that $I(v_{u,1},v_{u,2}) \cap R_1= \emptyset$ and $I(v_{u,1},v_{u,2})\subseteq E_{u j_1}$. Using
(\ref{eqnCellularex}), this implies that 
\begin{eqnarray*}
I(v_{u,1},v_{u,2}) & = & \cup_{j\in f_{S_1}(E_{u'})} \{ÊE_{uj} \, : \,E_{uj} \cap R_1=\emptyset\,; \, E_{uj} \subseteq  E_{uj_1} \} \\
& \subseteq &  I_1(E_u).
\end{eqnarray*}
Now, if $j\in f_{S_1}(E_{u})$ is such that $E_{uj} \cap R_1= \emptyset $ and $E_{uj} \subseteq E_{u j_1}$, then $E_{uj} \subseteq I(v_{u,1}, v_{u,2})$. Therefore $I_1(E_{u})= I(v_{u,1},v_{u,2})$.

\quad \,  Finally, suppose that $h$ has an odd length. So $p_1$ is entering $v_{u,2}$ and has to be contained in $p(v_{u,1},v_{u,2})$. So $R_1 \subseteq I(v_{u,1},v_{u,2})$, and using (\ref{eqnCellularex}) this implies that 
\begin{eqnarray*}
I_1(E_u) & \subseteq &  \cap_{j\in f_{S_1}(E_{u'})} \{ÊE_{uj} \, : \,   R_1 \subseteq E_{uj} \} \\
& =  & I(v_{u,1},v_{u,2}). 
\end{eqnarray*}
Now, if $j\in f_{S_1}(E_{u})$ is such that $R_1 \subseteq E_{uj}$, then $I(v_{u,1},v_{u,2})\subseteq E_{uj}$. Thus $I_1(E_{u})= I(v_{u,1},v_{u,2})$. This concludes the proof.
\end{itemize}
\end{proof}\\

For all $1\le \ell \le b$, let $E_\ell'=E_\ell \cup \{1, \rhoÊ\}$,
$$A_\ell= A_{E_\ell'  \times f(E_\ell)} \m{ and } A_\ell^\cap= [A_\ell \, \, \chi^{E_\ell'}_{I_\cap(E_\ell) \cup \{1, \rho \}} ].$$ The row indexes $1$ and $\rho$ of a matrix $A_\ell$ ($1\le \ell\le b$) are called \emph{artificial} as well as the corresponding edges\index{artificial edge} in any network representation of $A_\ell$ (if one exists).
For any shared  bonsai $E_u$ ($1\le u \le b$), let $$L_u=[A_u \, \, \chi^{E_u'}_{I(E_u)\cup \{1,\rho \}} ]$$ if there exists a sensitive bonsai distinct from $E_u$, and $L_u=A_u$ otherwise. Suppose that 
for some $k\in \{1,2\}$, the procedure Create-Ik with input $E_u$ and $k$ has output a nonempty subset $I_k(E_u)$ of $E_u$. Then we define $$L_{u,k}= [A_u \, \, \chi^{E_u'}_{I_k(E_u) \cup \{1, \rho \}} ].$$
A nonempty ordered set $U$ of at most two shared bonsais is said to be \emph{left-compatible}\index{left-compatible} if the following holds.

\begin{itemize}

\item[case: ]  $U=(E_u)$.\\ 
The matrix $L_u$
has a network representation in which $e_1$ and $e_\rho$ are nonalternating.

\item[case: ] $U=(E_u,E_{u'})$.\\
The procedure Create-Ik with input $E_u$ and index $1$ (resp., $E_{u'}$ and index $2$) outputs a nonempty set $I_1(E_u)$ (resp., $I_2(E_{u'})$). Moreover, $L_{u,1}$ and $L_{u',2}$ 
are network matrices, $f_{S_1}(E_{u'}) \subseteq \{ \beta \in f_{S_1}(E_{u}) \, : \, E_{u\beta}=I_1(E_u) \}$ and $f_{S_2}(E_u) \subseteq \{ \beta \in f_{S_2}(E_{u'}) \, : \, E_{u'\beta} =I_2(E_{u'}) \}$. 

\end{itemize}

\begin{figure}[h!]
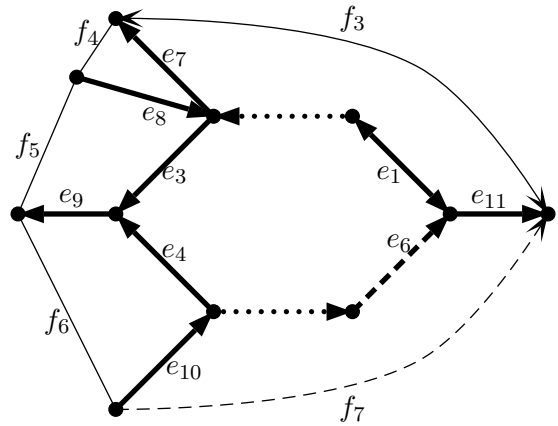

$
\begin{array}{cc}

\begin{array}{c}
E_1=\{3,4,7,8,9,10,11\};\\ 
\\

\begin{tabular}{c|c|c|c|c|c|}
   &  $f_3$ &  $f_4$ &  $f_5$ &  $f_6$ & $f_7$   \\
 \hline
$e_7$  &1&1&0&0&0 \\
 \hline
$e_8$  &0&1&1&0&0\\
\hline
$e_9$  &0&0&1&1&0\\
\hline
$e_{10}$  &0&0&0&1&1 \\
  \hline
$e_{11}$  &1&0&0&0&1\\
\hline

\end{tabular}
\end{array}   &

\psset{xunit=1.3cm,yunit=1.3cm,linewidth=0.5pt,radius=0.1mm,arrowsize=7pt,
labelsep=1.5pt,fillcolor=black}

\pspicture(-2,1)(3.7,3.5)

\pscircle[fillstyle=solid](0,1){.1}
\pscircle[fillstyle=solid](1,0){.1}
\pscircle[fillstyle=solid](1,2){.1}
\pscircle[fillstyle=solid](2.42,0){.1}
\pscircle[fillstyle=solid](2.42,2){.1}
\pscircle[fillstyle=solid](3.42,1){.1}
\pscircle[fillstyle=solid](4.42,1){.1}
\pscircle[fillstyle=solid](-.4,2.4){.1}
\pscircle[fillstyle=solid](0,3){.1}
\pscircle[fillstyle=solid](0,-1){.1}
\pscircle[fillstyle=solid](-1,1){.1}

\psline[linewidth=2pt,arrowinset=0]{<->}(3.42,1)(2.42,2)
\rput(2.8,1.4){$e_{1}$}

\psline[linewidth=2pt,arrowinset=0,linestyle=dotted]{->}(2.42,2)(1,2)

\psline[linewidth=2pt,arrowinset=0]{->}(1,2)(0,1)
\rput(.6,1.4){$e_{3}$}

\psline[linewidth=2pt,arrowinset=0]{<-}(0,1)(1,0)
\rput(.6,.6){$e_{4}$}

\psline[linewidth=2pt,arrowinset=0,linestyle=dotted]{->}(1,0)(2.42,0)

\psline[linewidth=2pt,arrowinset=0,linestyle=dashed]{->}(2.42,0)(3.42,1)
\rput(2.9,.7){$e_{6}$}

\psline[linewidth=2pt,arrowinset=0]{->}(1,2)(0,3)
\rput(.6,2.6){$e_{7}$}

\psline[linewidth=2pt,arrowinset=0]{<-}(1,2)(-.4,2.4)
\rput(.4,2){$e_{8}$}

\psline[linewidth=2pt,arrowinset=0]{->}(0,1)(-1,1)
\rput(-.45,1.16){$e_{9}$}

\psline[linewidth=2pt,arrowinset=0]{<-}(1,0)(0,-1)
\rput(.7,-.6){$e_{10}$}

\psline[linewidth=2pt,arrowinset=0]{->}(3.42,1)(4.42,1)
\rput(3.82,1.15){$e_{11}$}



\pscurve[arrowinset=.5,arrowlength=1.5]{<->}(0,3)(3.1,2.5)(4.42,1)
\rput(2.42,3){$f_{3}$}

\psline[arrowinset=.5,arrowlength=1.5]{-}(0,3)(-.4,2.4)
\rput(-.3,2.85){$f_{4}$}


\psline[arrowinset=.5,arrowlength=1.5]{-}(-1,1)(-.4,2.4)
\rput(-.9,1.7){$f_{5}$}

\psline[arrowinset=.5,arrowlength=1.5]{-}(-1,1)(0,-1)
\rput(-.6,-0.1){$f_{6}$}

\pscurve[arrowinset=.5,arrowlength=1.5,linestyle=dashed]{->}(0,-1)(3.1,-.5)(4.42,1)
\rput(2.42,-1){$f_{7}$}

\endpspicture 
\end{array}
$
\vspace{1.3cm}
\caption{A $\{1,6\}$-central representation of some binet $\{0,1\}$-matrix $A$ on the right, where $(E_1)$ is left-compatible and left-extreme, and
a non-network submatrix of $A$ on the left.} 
\label{fig:central7}

\end{figure}

\begin{figure}[h!]
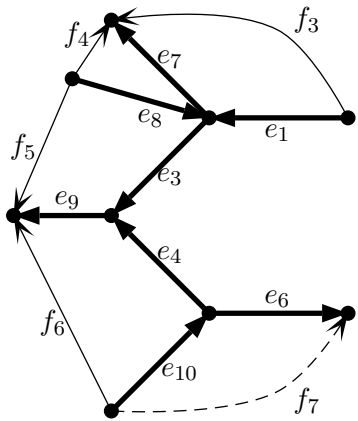

$
\begin{array}{cc}
L_1=
\begin{tabular}{c|c|c|c|c|c|}
   &  $f_3$ &  $f_4$ &  $f_5$ &  $f_6$ & $f_7$   \\
   \hline
$e_3$  &0&0&1&0&0 \\
\hline
$e_4$  &0&0&0&1&0 \\
 \hline
$e_7$  &1&1&0&0&0 \\
 \hline
$e_8$  &0&1&1&0&0\\
\hline
$e_9$  &0&0&1&1&0\\
\hline
$e_{10}$  &0&0&0&1&1 \\
  \hline
  \hline
$e_{1}$  &1&0&0&0&0\\
\hline
$e_{6}$  &0&0&0&0&1\\
\hline

\end{tabular}   &

\psset{xunit=1.3cm,yunit=1.3cm,linewidth=0.5pt,radius=0.1mm,arrowsize=7pt,
labelsep=1.5pt,fillcolor=black}

\pspicture(-2,1)(3.7,3.5)

\pscircle[fillstyle=solid](0,1){.1}
\pscircle[fillstyle=solid](1,0){.1}
\pscircle[fillstyle=solid](1,2){.1}
\pscircle[fillstyle=solid](2.42,0){.1}
\pscircle[fillstyle=solid](2.42,2){.1}
\pscircle[fillstyle=solid](-.4,2.4){.1}
\pscircle[fillstyle=solid](0,3){.1}
\pscircle[fillstyle=solid](0,-1){.1}
\pscircle[fillstyle=solid](-1,1){.1}


\psline[linewidth=2pt,arrowinset=0]{->}(2.42,2)(1,2)
\rput(1.7,1.85){$e_{1}$}

\psline[linewidth=2pt,arrowinset=0]{->}(1,2)(0,1)
\rput(.6,1.4){$e_{3}$}

\psline[linewidth=2pt,arrowinset=0]{<-}(0,1)(1,0)
\rput(.6,.6){$e_{4}$}

\psline[linewidth=2pt,arrowinset=0]{->}(1,0)(2.42,0)
\rput(1.7,.15){$e_{6}$}


\psline[linewidth=2pt,arrowinset=0]{->}(1,2)(0,3)
\rput(.6,2.6){$e_{7}$}

\psline[linewidth=2pt,arrowinset=0]{<-}(1,2)(-.4,2.4)
\rput(.4,2){$e_{8}$}

\psline[linewidth=2pt,arrowinset=0]{->}(0,1)(-1,1)
\rput(-.45,1.16){$e_{9}$}

\psline[linewidth=2pt,arrowinset=0]{<-}(1,0)(0,-1)
\rput(.7,-.6){$e_{10}$}




\pscurve[arrowinset=.5,arrowlength=1.5]{<-}(0,3)(1.7,2.9)(2.42,2)
\rput(2,3){$f_{3}$}

\psline[arrowinset=.5,arrowlength=1.5]{<-}(0,3)(-.4,2.4)
\rput(-.35,2.85){$f_{4}$}


\psline[arrowinset=.5,arrowlength=1.5]{<-}(-1,1)(-.4,2.4)
\rput(-.9,1.7){$f_{5}$}

\psline[arrowinset=.5,arrowlength=1.5]{<-}(-1,1)(0,-1)
\rput(-.6,-0.1){$f_{6}$}

\pscurve[arrowinset=.5,arrowlength=1.5,linestyle=dashed]{->}(0,-1)(1.7,-.8)(2.42,0)
\rput(2,-.9){$f_{7}$}

\endpspicture 
\end{array}
$
\vspace{1cm}
\caption{The matrix $L_1=A_1$ and a network representation of it such that $e_1$ and $e_6$ are nonalternating, with respect to the binet matrix whose a $\{1,6\}$-central representation is given in Figure \ref{fig:central7}.} 
\label{fig:central8}

\end{figure}

\noindent
The definition of left-compatible set is motivated by the next lemma. 
For an illustration, see Figures \ref{fig:central7} and \ref{fig:central8}.

\begin{lem}\label{lemCellularleft-extremecomp}
Suppose that $A$ has a $\{1,\rho\}$-central representation $G(A)$. Let $U$ be a left-extreme set of bonsais. Then $U$ is left-compatible.
\end{lem}

\begin{proof}
Suppose first that $U=(E_u)$. By definition, $B_u$ is shared.
If there is no sensitive bonsai distinct from $B_u$, then $I(E_u)=\emptyset$, otherwise $I(E_u)=I(v_{u,1}, v_{u,2})$ by Lemma \ref{lemCellularexcomp} (ii).
Consider the subgraph $B_u$ of $G(A)$ with one more basic edge $e_1$ entering $v_{u,1}$ and a basic edge $e_\rho$ leaving $v_{u,2}$. This yields
a basic network representation of $L_u$ in which $e_1$ and $e_\rho$ are nonalternating. The proof in the case $U=(E_u,E_{u'})$ is similar using Lemma \ref{lemCellularexcomp} (iii).
\end{proof}\\

Given a left-compatible set $U$, we define a \emph{$U$-spanning pair}\index{spanning@$U$-spanning pair} $(j_1,j_2)$ of column indexes as follows. If $U=(E_u)$, then $j_1 \in f_{S_1}(E_u)$, $j_2 \in f_{S_2}(E_u)$. If $U=(E_u,E_{u'})$, then $j_1 \in f_{S_1}(E_{u'})$, $j_2 \in f_{S_2}(E_{u})$. Observe that such a pair always exists.
For any $U$-spanning pair $(j_1,j_2)$, let $$V(j_1,j_2)= \{ E_\ell \in V \, :Ê\, j_1 \m{ or } j_2 \in f^*(E_\ell)\m{ or } E_\ell \m{ is shared } \}$$ and $R(j_1,j_2)= \cup_{E_\ell \in V(j_1,j_2)} E_\ell\cup \{1,\rho\}$. 

\begin{lem}\label{lemCentralRj}
Let $U$ be a left-compatible set and $(j_1,j_2)$ a $U$-spanning pair. If $A$ has a $\{1,\rho\}$-central representation $G(A)$ such that $U$ is left-extreme,
then $R(j_1,j_2)$ is the edge index set of a basic $1$-tree.
\end{lem}

\begin{proof}
Suppose that $A$ has a $\{1,\rho\}$-central representation $G(A)$ such that $U$ is left-extreme.
Then,
by definition of a $U$-spanning pair, it follows that $s(A_{\bullet j_1}) \cup s(A_{\bullet j_2})$ is the edge index set of a basic $1$-tree. Let $B_\ell$ be a shared bonsai. If $B_\ell$ is on the left of $\{e_1,e_\rho\}$, then by Lemma \ref{lemroot} $v_{\ell ,1}$ and $v_{\ell ,2}$ are on the basic cycle. If $B_\ell$ is on the right of $\{e_1,e_\rho\}$, then any bonsai $B_{\ell '}$ containing an edge of
the basic path from $v_\ell$ to $w_\rho$ has to be shared. This completes the proof.
\end{proof}\\

Given a left-compatible set $U$ and a $U$-spanning pair $(j_1,j_2)$,
the Proposition \ref{propOpposite} below informs us on the bipartiteness of the graph $F^*(\overline{V(j_1,j_2)})$, in  case where $A$ is a $\{1,\rho\}$-central matrix.
By assuming that $F^*(\overline{V(j_1,j_2)})$ is bipartite, we denote by $\mathcal{U}_1,\ldots,\mathcal{U}_\xi$ the connected components of $F^*(\overline{V(j_1,j_2)})$ and for each connected component $\mathcal{U}_{\kappa}$ ($1\le \kappa \le \xi$), let $\mathcal{U}_{\kappa}= \mathcal{U}_{\kappa}^0 \biguplus \mathcal{U}_{\kappa}^1$ be a bipartition of $\mathcal{U}_{\kappa}$ into two colour classes. For any $E_\ell\in \overline{V(j_1,j_2)}$, let $$Opposite(E_\ell)=  ( \cup_{j \in f^*(E_\ell)} s(A_{\bullet j})- \cap_{j \in f^*(E_\ell)} s(A_{\bullet j}) ) \cap R(j_1,j_2)$$, and 
for any $1\le \kappa \le \xi$ and $i\in\{0,1\}$, let 
$Opposite(\mathcal{U}_{\kappa}^i)=\cup_{E_\ell \in \mathcal{U}_{\kappa}^i} Opposite(E_\ell)$. 
The following propositions show a part of Theorem \ref{thmS0empty}.

\begin{prop}\label{propOpposite}
Let $U$ be a left-compatible set and $(j_1,j_2)$ a $U$-spanning pair. Suppose that $A$ has a $\{1,\rho\}$-central  representation $G(A)$ such that $U$ is left-extreme. Then

\begin{itemize}

\item[(i)] The graph $F^*(\overline{V(j_1,j_2)})$ is bipartite.

\item[(ii)] For any $1 \le \kappa \le \xi$, $i,i'\in \{ 0 , 1\}$, $E_\ell \in \mathcal{U}_{\kappa}^i$ and $E_{\ell '} \in \mathcal{U}_{\kappa}^{i'}$, $B_\ell$ and $B_{\ell '}$ are at different sides of $\{ e_1, e_\rho\}$ if and only if $i\neq i'$.

\item[(iii)] For any $1 \le \kappa \le \xi$, $i\in \{ 0 , 1\}$ and $E_\ell \in \mathcal{U}_{\kappa}^i$, the bonsai $B_\ell$ and the basic subgraph of $G(A)$ with edge index set $Opposite(\mathcal{U}_{\kappa}^i)$ are on both sides of $\{ e_1, e_\rho\}$.

\end{itemize}

\end{prop}

\begin{proof}
Let $E_\ell$ and $E_{\ell '}$ be two bonsais in 
$\overline{V(j_1,j_2)}$ such that $f^*(E_\ell) \cap f^*(E_{\ell '})\neq \emptyset$. Let $j\in f^*(E_\ell) \cap f^*(E_{\ell '})$. Assume $j\in S_1$ (the case $j\in S_2$ is similar). Since $V(j_1,j_2)$ contains all shared bonsais, $B_\ell$ and $B_{\ell '}$ are $S_1$-dominated and they both contain at least one edge of the fundamental circuit of $f_j$.
Following the lines of the proof of Proposition \ref{lemdigraphstem2}, one can prove that if $B_\ell$ and $B_{\ell '}$ are on the same side of $\{e_1,e_\rho\}$, then $(E_\ell,E_{\ell '})_{E_{\ell 'j}} \in D$ or $(E_{\ell '},E_\ell)_{E_{\ell j}} \in D$. 
Thus if $E_\ell$ and $E_{\ell '}$ are adjacent in $F^*(\overline{V(j_1,j_2)})$, then $B_\ell$ and $B_{\ell '}$ are at different sides of $\{e_1, e_\rho\}$. This induces a partition of $F^*(\overline{V(j_1,j_2)})$ into two colour classes, one for the bonsais on the left of $\{e_1, e_\rho\}$ and the other for those on the right of $\{e_1, e_\rho\}$. This completes the proof of (i) and (ii).

Now let $1 \le \kappa \le \xi$, $i\in \{ 0 , 1 \}$ and $E_\ell\in \mathcal{U}_{\kappa}^i$. We assume that
$B_\ell$ is $S_1$-dominated (the case $S_2$-dominated is similar). If $B_\ell$ is on the left of $\{e_1,e_\rho\}$, then the path from $v_{\ell ,1}$ to $w_1$ in $G_0(A)$ is in the fundamental circuit of every nonbasic edge with index in $f^*(E_\ell)$. Since $B_\ell$ does not contain any edge of the basic $1$-tree with edge index set $R(j_1,j_2)$
(see Lemma \ref{lemCentralRj}), we deduce that
the basic subgraph of $G(A)$ with edge index set $Opposite(E_\ell)$ is on the right of $\{e_1,e_\rho\}$.
Similarly, if $B_\ell$ is on the right of 
$\{e_1,e_\rho\}$, then the basic subgraph of $G(A)$ with edge index set $Opposite(E_\ell)$ is on the left of 
$\{e_1,e_\rho\}$. Combining this and part (ii) gives the desired result.
\end{proof}\\

Let $U$ be a left-compatible set and $(j_1,j_2)$ a $U$-spanning pair. We define relations $\prec^{j_1}$ and $\prec^{j_2}$ on $V(j_1,j_2)$. For $k=1$ and $2$ and $E_\ell,E_{\ell '} \in V(j_1,j_2)$, \index{relation!$\prec^{j_k}$}
$$E_\ell \prec^{j_k}E_{\ell '} \,\,\, \Leftrightarrow  \,\,\,
f_{S_k}(E_{\ell '}) \subseteq \{ \beta \in f_{S_k}(E_\ell) \,: \, E_{\ell \beta} = E_{\ell j_k} \}.$$ Clearly, the relations $\prec^{j_1}$ and $\prec^{j_2}$ are transitive. One can prove that if $U=(E_{u},E_{u'})$, then 
$E_{u} \prec^{j_1} E_{u'}$ and $E_{u'} \prec^{j_2} E_{u}$ (see the proof of Claim 2 at page \pageref{mycounter8}).
A bonsai $E_\ell\in V(j_1,j_2)\verb"\"U$ is said to be \emph{right-feasible}\index{right-feasible!bonsai} if 
\begin{itemize}
 
\item[$ \omega$.$0$ ]  $J_\ell^2= \emptyset$ and $N_\ell$ is a network matrix;

\end{itemize}

and \emph{left-feasible}\index{left-feasible!bonsai} if the following holds.
\begin{itemize}

\item If $E_\ell$ is disjointly shared, then 
\begin{enumerate}

\item[ $ \omega$.$1$  ] $E_\ell$ is sensitive;
\item[$ \omega$.$2$  ] for all $E_u \in U$, $E_u$ is not $S_1,S_2$ linked to $E_\ell$;
\item[$ \omega$.$3$  ] if $U=(E_u)$, then 
$E_{uj}=I(E_u)$ for all $j\in f_{S_k}(E_\ell)$ and some $k\in \{1,2\}$;
\item[$ \omega$.$4$ ] If $U=(E_{u},E_{u'})$, then either $E_{u j}=I_1(E_u)=I_2(E_{u'})$
for all $j\in f_{S_k}(E_\ell)$ and 
some $k\in \{1,2\}$, or $E_{u j}=I_1(E_u)$ and 
$E_{u' j'}=I_2(E_{u'})$ for 
all $j\in f_{S_1}(E_\ell)$ and $j'\in f_{S_2}(E_\ell)$;

\end{enumerate}
\item If $E_\ell$ is jointly shared, then 
\begin{itemize}

\item[$ \omega$.$5$ ] $A_\ell^\cap$ is a network matrix;
\item[$ \omega$.$6$ ] $|U|=2$ and writing $U=(E_u,E_{u'})$, we have
$I_\cap (E_\ell)=E_{\ell j_1}=E_{\ell j_2}$, $E_{u} \prec^{j_1} E_\ell \prec^{j_1} E_{u'}$ and $E_{u'} \prec^{j_2} E_\ell \prec^{j_2} E_{u}$.

\end{itemize}

\item If $E_\ell$ is $S_k$-dominated for some $k\in \{ 1,2 \}$,
then 
\begin{itemize}

\item[$ \omega$.$7$ ] $J_\ell^2= \emptyset$ and $N_\ell$ is a network matrix;
\item[$ \omega$.$8$ ] $E_u \prec^{j_k} E_\ell$ for all $E_u \in U\cup \{ E_{\ell '} \, :\, E_{\ell '} \m{ is sensitive }; \, j_k \in f^*(E_{\ell '}) \}$ or 
$E_\ell \prec^{j_k} E_u$ for all $E_u \in U\cup \{ 
E_{\ell '}\, : \, E_{\ell '}
\m{ is sensitive} \}$.
\end{itemize}

\end{itemize}

\begin{lem}\label{lemCellularfeasible}
Suppose that $A$ has a $\{1,\rho\}$-central  representation $G(A)$ and assumption $\mathscr{A}$ is satisfied. Let $U$ be a left-extreme set of bonsais, $(j_1,j_2)$ a $U$-spanning pair and $E_\ell \in V(j_1,j_2)\verb"\"U$. If $B_\ell$
is on the right of $\{ e_1,e_\rho\}$, then $E_\ell$ is right-feasible, otherwise left-feasible.
\end{lem}

\begin{proof}
If $B_\ell$ is on the right of $\{ e_1,e_\rho\}$, then by Lemma \ref{lemright} it is right-feasible. Now assume that $B_\ell$ is on the left of $\{ e_1,e_\rho\}$. 

Suppose first that $B_\ell$ is disjointly shared. Clearly, $B_\ell$ does not contain any edge of the basic cycle (otherwise $U=(E_{u},E_{u'})$ for some $1\le u,u' \le b$
and $B_\ell$ would be succeeding $B_{u}$ and preceding $B_{u'}$, which implies that $I_\cap(E_\ell)\neq \emptyset$). By Lemma \ref{lemroot}, $v_{\ell ,1}=v_{\ell ,2}$ and $B_\ell$ is sensitive.  For any $E_u\in U$, by Lemma 
\ref{lembeta2}, $B_u$ is not $S_1,S_2$ linked to $B_\ell$; this implies that $B_\ell$ is either preceding or succeeding $B_u$. In case where $U=(E_u)$,  using Lemma \ref{lemCellularexcomp} (ii), it follows that $E_{uj}=I(E_u)$ for all $j\in f_{S_1}(E_\ell)$ if $E_\ell$ is succeeding $E_u$, or $E_{uj}=I(E_u)$ for all $j\in f_{S_2}(E_\ell)$ otherwise.
In a similar way using  Lemma \ref{lemCellularexcomp} (iii), one can deal with the case $|U|=2$, proving that $E_\ell$ is left-feasible.

Secondly, suppose that $E_\ell$ is jointly shared. By Lemma \ref{lemCellularexcomp} (i), $B_\ell$ contains at least one edge of the basic cycle. So, by definition of $U$, $U=(E_u,E_{u'})$ for some $1\le u,u'\le b$, and
$B_\ell$ is succeeding $B_{u}$ and preceding $B_{u'}$. Then, using Lemma \ref{lemCellularexcomp} (i) and since $j_1 \in f_{S_1}(E_{u'})$ and $j_2 \in f_{S_2}(E_{u})$, we have $I_\cap(E_\ell) =E_{\ell j_1}=E_{\ell j_2}$, $E_{u} \prec^{j_1} E_\ell \prec^{j_1} E_{u'}$ and $E_{u'} \prec^{j_2} E_\ell \prec^{j_2} E_{u}$. 
Thus $E_\ell$ is left-feasible.

Finally, suppose that  $E_\ell$ is $S_1$-dominated (the case $S_2$ instead of $S_1$ is symmetric). By Proposition \ref{propSkwatered}, $J_\ell^2=\emptyset$ and $N_\ell$ is a network matrix. Moreover, since $E_\ell \in V(j_1,j_2)$, $j_1\in f^*(E_\ell)$. Then,
either $B_\ell$ contains al least one edge of
the basic cycle and so  $E_\ell \prec^{j_1} E_u$ for all $E_u \in U\cup \{ E_{\ell '} \, :\, E_{\ell '} \m{ is sensitive} \}$, or
$E_u \prec^{j_1} E_\ell$
for all $E_u \in U\cup \{ E_{\ell '} \, :\, E_{\ell '} \m{ is sensitive }; \, j_1 \in f^*(E_{\ell '}) \}$. This completes the proof.
\end{proof}\\

Let $U$ be a left-compatible set and $(j_1,j_2)$ a $U$-spanning pair. 
A pair of bonsais $E_\ell,E_{\ell '} \in V(j_1,j_2)-U$ 
such that $f^*(E_\ell) \cap f^*(E_{\ell '})\neq \emptyset$
is said to be \emph{right-feasible}\index{right-feasible!pair of bonsais} if 
\begin{itemize}
 
\item[$ \phi$.$1$ ]  $(E_\ell,E_{\ell '}) \in D$ or $(E_{\ell '},E_\ell) \in D$;

\end{itemize}
and \emph{left-feasible}\index{left-feasible!pair of bonsais} if the following holds.
\begin{enumerate}

\item[$ \phi$.$2$ ] The bonsais $E_\ell$ and $E_{\ell '}$ are not both sensitive.
\item[$ \phi$.$3$ ] If $E_\ell$ and $E_{\ell '}$ are jointly shared, then 
$E_\ell \prec^{j_k} E_{\ell '}$ and $E_{\ell '} \prec^{j_{k'}} E_\ell$ for some $k,k' \in \{ 1,2\}$, $k\neq k'$.
\item[$ \phi$.$4$ ] If $E_\ell$ and $E_{\ell '}$ are $S_k$-dominated for some $k\in \{ 1,2\}$, then $E_\ell \prec^{j_k} E_{\ell '}$ or $E_{\ell '} \prec^{j_k} E_\ell$.
\item[$ \phi$.$5$ ] If $E_\ell$ is jointly shared and $E_{\ell '}$ sensitive, then
$E_\ell$ is not $S_1,S_2$ linked to $E_{\ell '}$ and $E_\ell \prec^{j_k} E_{\ell '}$ for some $k\in \{ 1, 2 \}$.

\end{enumerate}

\begin{lem}\label{lemCellularpair}
Suppose that $A$ has a $\{1,\rho\}$-central  representation $G(A)$ and assumption $\mathscr{A}$ is satisfied. Let $U$ be a left-extreme set of bonsais  and $(j_1,j_2)$ a $U$-spanning pair. Let $E_\ell,E_{\ell '}\in V(j_1,j_2)-U$ such that $f^*(E_\ell) \cap f^*(E_{\ell '})\neq \emptyset$. If $B_\ell$ and $B_{\ell '}$ are both on the right (respectively, the left) of $\{ e_1,e_\rho\}$, then the pair $E_\ell$, $E_{\ell '}$ is right-feasible (respectively, left-feasible).
\end{lem}

\begin{proof}
Let us first remark that since $U$ is left-extreme, by Lemma \ref{lemCellularleft-extremecomp}, $U$ is left-compatible. So $(j_1,j_2)$ is well defined.
Let $E_\ell,E_{\ell '}\in V(j_1,j_2)\verb"\"U$ such that $f^*(E_\ell) \cap f^*(E_{\ell '})\neq \emptyset$. If $E_\ell$ and $E_{\ell '}$ are both on the right of $\{ e_1,e_\rho\}$, then using Proposition \ref{lemdigraphstem}, the pair is right-feasible. 

Suppose now that  
$B_\ell$ and $B_{\ell '}$ are on the left of $\{ e_1,e_\rho\}$.
The bonsais $E_\ell$ and $E_{\ell '}$ can not be both sensitive (otherwise from Lemmas  \ref{lem2rid} and \ref{lembeta2}, it follows that $v_{\ell ,1}=v_{\ell ,2}$ and $v_{\ell ',1}=v_{\ell ',2}$, hence $f^*(E_\ell)\cap f^*(E_{\ell '})=\emptyset$). If $E_\ell$ and $E_{\ell '}$ are jointly shared, by Lemma \ref{lemCellularfeasible}, we have that $I_\cap(E_\ell)=E_{\ell j_1}=E_{\ell j_2}$ and $I_\cap(E_{\ell '})=E_{\ell 'j_1}=E_{\ell 'j_2}$. Since $B_\ell$ is preceding $B_{\ell '}$ or conversely, it follows that the pair $E_\ell$, $E_{\ell '}$ is left-feasible. The proof is also straightforward, using Proposition \ref{lemdigraphstem}, in case where $E_\ell$ and $E_{\ell '}$ are both $S_k$-dominated for some $k\in \{1,2\}$.

At last, suppose that $B_\ell$ is jointly shared and $B_{\ell '}$ is sensitive. By Lemma \ref{lemCellularfeasible} and  \ref{lemCellularexcomp} (i), $I_\cap(E_\ell)=E_{\ell  j_1}=E_{\ell  j_2}$ is the index set of edges in $B_\ell$ and the basic cycle. Then, by Lemma \ref{lembeta2}, $B_{\ell '}$
is not $S_1,S_2$ linked to $B_\ell$, so that it is either preceding $B_\ell$ in which case $E_\ell \prec^{j_2} E_{\ell '}$, or succeeding $B_\ell$ and $E_\ell \prec^{j_1} E_{\ell '}$. 
\end{proof}\\

Assume that the graph $F^*(\overline{(V(j_1,j_2)})$ is bipartite. Given a left-compatible set $U$ and a $U$-spanning pair $(j_1,j_2)$,
we are now ready to construct the instance $\Omega(U,j_1,j_2)$ of the $2$-SAT problem. 
The set of variables is $\{ x_\ell \, : \, E_\ell \in V(j_1,j_2)  \}$. 
Provided that $A$ is $\{1,\rho\}$-central, the equalities $x_\ell=0$ and $x_\ell=1$ mean that the bonsai $B_\ell$ is on the left and on the right of $\{e_1,e_\rho\}$, respectively, in some $\{1,\rho\}$-central representation of $A$. 

For any $E_u\in U$, set $x_u=0$ in $\Omega(U,j_1,j_2)$.
For any $E_\ell\in V(j_1,j_2)\verb"\"U$, if $E_\ell$ is not right-feasible (respectively, left-feasible), then let $x_\ell=0$ (respectively, $x_\ell=1$).
For any pair of bonsais $E_\ell,E_{\ell '} \in V(j_1,j_2)-U$ such that $f^*(E_\ell) \cap f^*(E_{\ell '}) \neq \emptyset$, if the pair is not right-feasible, then put the clause $\bar x_\ell \vee \bar x_{\ell '} $ in $\Omega(U,j_1,j_2)$. Furthermore, if  the pair is not left-feasible, then put the clause $x_\ell \vee x_{\ell '}$ in $\Omega(U,j_1,j_2)$. Thus if the pair is not right-feasible (respectively, left-feasible), then at most one of the variables $x_\ell$ and $x_{\ell '}$ has  value $1$ (respectively, $0$). 

For $\kappa=1,\ldots,\xi$, do as follows. For any two variables $x_\ell$ and $x_{\ell '}$ such that $E_\ell \, \cap \, Opposite(\mathcal{U}_\kappa^i)$ $\neq \emptyset $ and $E_{\ell '} \cap Opposite(\mathcal{U}_\kappa^i) \neq \emptyset $ with $i=0$ or $1$, put the equality $x_\ell = x_{\ell '}$ in $\Omega(U,j_1,j_2)$. Moreover, if $Opposite(\mathcal{U}_\kappa^0) \neq \emptyset$ and $Opposite(\mathcal{U}_\kappa^{1}) \neq \emptyset$, choose some $x_\ell$ and $x_{\ell '}$ such that 
$E_\ell \cap Opposite(\mathcal{U}_\kappa^0)\neq \emptyset $ and $E_{\ell '} \cap Opposite(\mathcal{U}_\kappa^{1})\neq \emptyset $. Then put the clauses $x_\ell \vee x_{\ell '}$ and $\bar x_\ell \vee \bar x_{\ell '}$ in $\Omega(U,j_1,j_2)$. These clauses ensure that the variables $x_\ell$ and $x_{\ell '}$ have different values.

\begin{prop}\label{propCellulartruth}
Suppose that $A$ has a $\{1,\rho\}$-central representation and assumption $\mathscr{A}$ is satisfied. Let $U$ be a left-extreme set. Then, for any $U$-spanning pair $(j_1,j_2)$, the instance $\Omega(U,j_1,j_2)$ has a truth assignment.
\end{prop}

\begin{proof}
Let $(j_1,j_2)$ be a $U$-spanning pair. For each bonsai $E_\ell$, let $x_\ell=0$ if $E_\ell$ is on the left of $\{e_1,e_\rho\}$, and $x_\ell=1$ otherwise.
By Lemmas \ref{lemCellularfeasible} and \ref{lemCellularpair} and Proposition \ref{propOpposite} (ii), this is a truth assignment.
\end{proof}\\

The following procedure computes a truth assignment of 
$\Omega(U,j_1,j_2)$ for some left-compatible set $U$ and a $U$-spanning pair $(j_1,j_2)$, whenever $A$ is $\{\epsilon,\rho\}$-central.

\begin{tabbing}
\textbf{Procedure\,\,Truthassignment($A$,$\{\epsilon,\rho\}$)}\\

\textbf{Input:}\, \= A non-network matrix $A$ and two row indexes $\epsilon$ and  $\rho\neq 1$ such that $S_0= \emptyset$.\\

\textbf{Output:} \=  Either a left-compatible set $U$ and a $U$-spanning pair $(j_1,j_2)$ with a truth\\
\>  assignment of $\Omega(U,j_1,j_2)$, or determines that $A$ is  not $\{\epsilon ,\rho\}$-central.\\

1)\verb"  "\= call {\tt Initialization}($A$,$\{\epsilon,\rho\}$) outputing a matrix $A'$;\\

2) \> compute a digraph $D$ with respect to $A'$ and the row index subset $R^*=\{1,\rho\}$\\
\>  of $A'$;\\

3) \> if $g_\beta(E_\ell)=2$ for some $E_\ell\in V$, $\beta \in S^*$, then STOP: output that $A$ is not $\{\epsilon ,\rho\}$-central;\\

4) \> {\bf for } \= every ordered pair $U$ of shared bonsais, $|U|\le 2$, {\bf do }\\

\>  \> check whether $U$ is left-compatible; if it is not, then return to 4;\\

\> \> compute a $U$-spanning pair $(j_1,j_2)$, and test if $F^*(\overline{V(j_1,j_2)})$ is bipartite; \\
\> \>if it is not, then return to 4;\\

\> \> compute the instance $\Omega(U,j_1,j_2)$; compute a truth assignment of $\Omega(U,j_1,j_2)$, \\
\> \> if one exists, output it with $U$ and $(j_1,j_2)$ and STOP, otherwise return to 4;\\

 \> {\bf endfor }\\

\>  output that $A$ is not $\{\epsilon,\rho\}$-central;

\end{tabbing}

\begin{prop}\label{propCentraltruth}
The output of the procedure Truthassignment is correct.
\end{prop}

\begin{proof}
By Lemma \ref{lemCentralInit}, the subroutine Initialization outputs a matrix $A'$ of size $n\times m'$ whose $A_{\bullet \{1,\ldots,m \}}'=A$ and such that $A'$ is $\{1,\rho\}$-central if and only if $A$ is $\{\epsilon,\rho\}$-central, and assumption $\mathscr{A}$ is satisfied for $A'$. Step 3 is justified by Lemma \ref{lemCentralElj}.

Suppose that $A$ is $\{\epsilon,\rho\}$-central. Then, let $G(A')$ be a $\{1,\rho\}$-central representation of $A'$, and $U$ a left-extreme set of bonsais. Then, by Lemma \ref{lemCellularleft-extremecomp}, $U$ is left-compatible. Let $(j_1,j_2)$ be a $U$-spanning pair. By Proposition \ref{propOpposite} (i), $F^*(\overline{V(j_1,j_2)})$ is bipartite. Finally, using Proposition \ref{propCellulartruth}, the procedure Truthassignment outputs a truth assignment of $\Omega(U,j_1,j_2)$. This concludes the proof.
\end{proof}\\

Before proving Theorem \ref{thmS0empty}, we describe some advantages and properties which ensue from a truth assignment of $\Omega(U,j_1,j_2)$. Let $U$ be a left-compatible set and $(j_1,j_2)$ a $U$-spanning pair. If $|U|=2$, then we write $U=(E_{u},E_{u'})$, otherwise $U=(E_{u})$.
Suppose that each variable $x_\ell$ ($E_\ell \in V(j_1,j_2)$) has received a value $0$ or $1$, yielding a truth assignment of $\Omega(U,j_1,j_2)$. Let $$X_0(j_1,j_2)=\{ E_\ell \in V(j_1,j_2) \, :\, x_\ell=0\}$$ and $$X_1(j_1,j_2)= \{ E_\ell \in V(j_1,j_2) \,: \, x_\ell=1 \}.$$ 
By $\omega.6$, notice that if $E_\ell$ is jointly shared and $x_\ell=0$, then $|U|=2$. So,
up to a renumbering of the bonsais and a change of the value of $u$ and $u'$, we may assume that
$\{ E_\ell \, :\, x_\ell=0; \, j_1, j_2 \in f^*(E_\ell) \} = \{ E_u,E_{u+1} \ldots, E_{u'} \}$,
$\{ E_\ell \, :\, x_\ell=0; \, j_1 \in f^*(E_\ell) \}=\{E_1,E_2,\ldots, E_{\ell _1} \}$ and 
$\{ E_\ell \, :\, x_\ell=0; \, j_2 \in f^*(E_\ell) \}=\{E_{u},E_{u+1} \ldots, E_{u'}\} \cup \{E_{\ell _1+1},E_{\ell _1+2},\ldots,E_{\ell _2} \}$, where $1\le u \le u' \le \ell_1\le \ell_2 $ and $u=u'$ in case $|U|=1$. 
Now we state some useful claims.\\ 

\noindent
{\bf Claim 0.} \quad For any sensitive bonsai $E_\ell$, $x_\ell=0$ and there is no bonsai $E_{\ell '}$ ($\ell'\neq l$)  that is  $S_1$, $S_2$ linked to $E_\ell$ and such that $x_{\ell '}=0$.\\

\noindent
{\bf Proof of Claim 0.} \quad Since $J_\ell^2 \neq \emptyset$, $E_\ell$ is not right-feasible, so $x_\ell=0$.
Let $E_{\ell '}$ be a bonsai $S_1$, $S_2$ linked to $E_\ell$. Suppose, to the contrary, that $x_{\ell '}=0$. Thus, $E_{\ell '}$ is left-feasible, and the pair $E_\ell,E_{\ell '}$ is left-feasible. By $\phi.5$, $E_{\ell '}$ is not jointly shared, so $E_{\ell '}$ is disjointly shared and by $\omega.2$, $E_{\ell '}$ is sensitive, contradicting $\phi.2$.
{\hfill$\BBox{\rule{.3mm}{3mm}}$} \\

\noindent
{\bf Claim 1.} \quad For any $E_\ell \in V(j_1,j_2)$ such that 
$j_1, j_2 \notin f^*(E_\ell)$,  $x_\ell=0$ if and only if $E_\ell$ is sensitive.\\

\noindent
{\bf Proof of Claim 1.} \quad Let $E_\ell \in V(j_1,j_2)$ such that $j_1, j_2 \notin f^*(E_\ell)$. By definition of $V(j_1,j_2)$, it results that $E_\ell$ is shared. If $E_\ell$ is sensitive, then by Claim 0 $x_\ell=0$. Suppose now that
$x_\ell=0$. If $E_\ell$ is jointly shared, then, as $I_\cap(E_\ell) \neq \emptyset$, $\omega.6$ implies that $j_1,j_2 \in f^*(E_\ell)$, a contradiction. Thus $E_\ell$ is disjointly shared and by $\omega.1$, $E_\ell$ is sensitive.
{\hfill$\BBox{\rule{.3mm}{3mm}}$} \\

\noindent
{\bf Claim 2.}\label{mycounter8} \quad  Up to a renumbering of the bonsais and a change of the value of $u$ and $u'$, there exist orderings, say 

\begin{eqnarray}\label{eqnCellularj1}
E_1\prec^{j_1} E_2 \cdots \prec^{j_1}  E_{u} \prec^{j_1} E_{u+1}  \cdots  \prec^{j_1}  E_{u'} \prec^{j_1} E_{u'+1} \cdots \prec^{j_1}  E_{\ell _1} \label{eqnCellularj1} 
\end{eqnarray}

\begin{center}
and 
\end{center}
\begin{eqnarray}\label{eqnCellularj2}
E_{\ell _1+1}\prec^{j_2} E_{\ell _1 +2} \cdot \cdot \prec^{j_2} E_{\ell _2'}\prec^{j_2} E_{u'} \prec^{j_2} E_{u'-1} \cdot \cdot  \prec^{j_2} E_{u} \prec^{j_2} E_{\ell _2'+1} \prec^{j_2} E_{\ell _2'+2} \cdot \cdot \prec^{j_2} E_{\ell _2} 
\end{eqnarray}

\noindent 
of the sets $\{ E_\ell \, :\, x_\ell=0; \, j_1 \in f^*(E_\ell) \}$ and 
$\{ E_\ell \, :\, x_\ell=0; \, j_2 \in f^*(E_\ell) \}$, respectively, 
where $l_1\le l_2' \le l_2$,
$E_{u-1}\prec^{j_1} E_\ell$ (if $u\neq 1$) and $E_{\ell _2'} \prec^{j_2} E_\ell$ (if $l_2'\neq l_1$) for any sensitive bonsai $E_\ell$.\\

\noindent
{\bf Proof of Claim 2.} \quad Let us prove first that if $U=(E_u,E_{u'})$, then $E_u \prec^{j_1} E_{u'}$ (symmetrically, we can prove that $E_{u'} \prec^{j_2} E_u$). Let $j\in f_{S_1}(E_{u'})$. By definition of the $U$-spanning pair $(j_1,j_2)$, $j_1 \in f_{S_1}(E_u)$. Moreover, by definition of the left-compatible set $U$, $E_{uj}=I_1(E_u)$ and $E_{u j_1}=I_1(E_u)$. Thus, $E_{uj}=E_{u j_1}$. Therefore, $f_{S_1}(E_{u'}) \subseteq \{ \beta \in f_{S_1}(E_u) \, : \, E_{u\beta} =E_{u j_1}\}$. 

Let $E_\ell$ be a bonsai not in $U$ such that 
 $j_1,j_2\in f^*(E_\ell)$ and $x_\ell=0$ ($u \le \ell \le u'$). If $E_\ell$ is not disjointly shared, then by $\omega.1$ $E_\ell$ is sensitive and since $j_1,j_2\in f^*(E_u)$, this contradicts  $\omega.2$. Thus $E_\ell$ is jointly shared. Then, using  $\omega.6$ and $\phi.3$, the subsequences of (\ref{eqnCellularj1}) and (\ref{eqnCellularj2}) containing the bonsais  $E_u,E_{u+1},\ldots, E_{u'}$ are justified.
 
Suppose that there exists at least one sensitive bonsai $E_\ell \notin U$ such that $j_1 \in f^*(E_\ell)$ (the case $j_2 \in f^*(E_\ell)$ is similar). By $\phi.2$, $E_\ell$ is unique. Since $j_1 \in f^*(E_\ell) \cap f^*(E_u)$ for any $E_u \in U$, by Claim 0, it follows that $f_{S_2}(E_\ell)\cap f_{S_2}(E_u)=\emptyset $ for all $E_u \in U$. Therefore, by $\omega.3$ (in case $|U|=1$) or $\omega.4$, $E_{u'} \prec^{j_1} E_\ell$. Finally, using $\omega.8$ or $\phi.4$ for dealing with $S_1$ or $S_2$-dominated bonsais, this completes the proof.
{\hfill$\BBox{\rule{.3mm}{3mm}}$}\\


\noindent
{\bf Claim 3.} \quad Let $E_\ell$ be a sensitive bonsai such that $j_1,j_2 \notin f^*(E_\ell)$. If $u\neq u'$, or $u=u'$ and $I(E_u)\neq \emptyset$, then exactly one of the following cases holds.

\begin{itemize}

\item[(i)] For all $j\in f_{S_2}(E_\ell)$, we have $E_{uj}=I(E_u)$ if $u=u'$, or $E_{uj}=I_1(E_u)$ otherwise.

\item[(ii)] There exists an index $i(\ell)$ such that 
$u\le i(\ell)< u'$, $E_{i(\ell)} \prec^{j_1} E_\ell$ and $E_{i(\ell)+1} \prec^{j_2} E_\ell$.

\item[(iii)]  For all $j\in f_{S_1}(E_\ell)$, we have $E_{uj}=I(E_u)$ if $u=u'$, or $E_{uj}=I_2(E_u)$ otherwise.\\

\end{itemize}

\noindent
{\bf Proof of Claim 3.} \quad This follows from $\omega.3$, $\omega.4$,  $\phi.5$ and Claim 2.
{\hfill$\BBox{\rule{.3mm}{3mm}}$}\\

\noindent
{\bf Claim 4.} \quad  There exists a forest $T_1(j_1,j_2)$ in $D$ with vertex set $X_1(j_1,j_2)$ such that for any $E_\ell,E_{\ell '} \in X_1(j_1,j_2)$ with $f^*(E_\ell) \cap f^*(E_{\ell '}) \neq \emptyset$, $E_\ell$ and $E_{\ell '}$ are contained in a same subpath of $T_1(j_1,j_2)$.\\

\noindent
{\bf Proof of Claim 4.} \quad We first make two observations:
\begin{itemize}

\item[1)] For any pair of bonsais $E_\ell$ and $E_{\ell '}$ in $X_1(j_1,j_2)$ such that $f^*(E_\ell) \cap f^*(E_{\ell '}) \neq \emptyset$,  we have $(E_\ell,E_{\ell '})\in D$ or $(E_{\ell '},E_\ell)\in D$.

\item[2)] For any $E_\ell,E_{\ell '}\in D$, we have  $(E_\ell,E_{\ell '}) \in D \Rightarrow f^*(E_\ell) \subseteq  f^*(E_{\ell '})$.

\end{itemize}

The first observation follows from the right-feasibility of any pair of bonsais in $X_1(j_1,j_2)$, while the second  results from the definition of $D$.
Let us prove that for any $W\subseteq X_1(j_1,j_2)$, $W$ has the property claimed for $X_1(j_1,j_2)$. One proceeds by induction on the cardinality of the subsets of $X_1(j_1,j_2)$. 

Let $W\subseteq X_1(j_1,j_2)$. If  $|W|=1$, then the proof is clear. Suppose now that $|W|\geq 1$,  $W$ is a proper subset of $X_1(j_1,j_2)$ and there exists a forest $G$ with vertex set $W$ satisfying the desired property. Let $E_\ell\in X_1(j_1,j_2)\verb"\" W$. Consider the set $W'=\{E_{\ell '} \in W \,: \, f^*(E_\ell) \cap f^*(E_{\ell '})\neq \emptyset \}$. By the second observation,
we deduce that $W'$ is closed in $G$. Denote by $E_{r_1},\ldots,E_{r_k}$ the root vertices of the trees in $G(W')$.

Suppose that $G(W')$ contains at least two trees ($k\geq 2$). By induction hypothesis $f^*(E_{r_1}) \cap f^*(E_{r_2}) =\emptyset$. On the other hand, by the first observation, $(E_\ell,E_{r_1}) \in D$ or 
$(E_{r_1},E_\ell) \in D$. If $(E_\ell,E_{r_1}) \in D$, then by observation 2, $f^*(E_\ell) \subseteq f^*(E_{r_1})$ and since $f^*(E_\ell) \cap f^*(E_{r_2}) \neq \emptyset$, it follows that $f^*(E_{r_1}) \cap f^*(E_{r_2})\neq \emptyset$, a contradiction. We deduce that $(E_{r_i},E_\ell) \in D$ for $i=1,\ldots,k$. Adding these edges in $G(W')$ results in a forest of $D$ with vertex set $W\cup \{ E_\ell\}$ satisfying the claimed property.

Now suppose that $G(W')$ contains exactly one tree ($k=1$). Using $G(W')$, we construct a forest spanning $W'\cup \{ E_\ell\}$ and satisfying the claimed property.
Let $E_u$ be the bonsai in $G(W')$ of largest height such that $(E_\ell,E_u)\in D$. 
As in the previous case, we can show that if any two bonsais, say $E_{u'}$ and $E_{u''}$, in $G(W')$ are not in a same subpath of $G(W')$, then $(E_{u'},E_\ell) \in D$ and $(E_{u''},E_\ell) \in D$. This implies that  in the tree $G(W')$ the successor of $E_u$ (if it exists) has exactly one predecessor ($E_u$).
If $E_u$ exists, adding $(E_\ell,E_u)$ and $(E_{u'},E_\ell)$ in $G(W')$ and removing $(E_{u'},E_u)$, for each predecessor $E_{u'}$ of $E_u$ in $G(W')$, results in a forest satisfying the claimed property. Otherwise, adding $(E_{r_1},E_\ell)$ in $G(W')$ yields the desired forest.
{\hfill$\BBox{\rule{.3mm}{3mm}}$}\\

Below, we describe a subroutine of CentralI, called GAR, which constructs a basic $\{1,\rho\}$-central representation of $A_{R(j_1,j_2) \bullet}'$, where $A_{\bullet \times \{1,\ldots,m\}}'=A$,
provided that $A$ is $\{\epsilon,\rho\}$-central. In the following procedure, if we are dealing with a network representation of a matrix related to a bonsai $E_\ell$ (for instance $L_\ell$, $L_{\ell ,1}$, $A_\ell^\cap$, or $N_\ell$) such that $x_\ell=0$ and $j_k \in f^*(E_\ell)$ for some $k\in \{1,2\}$, then we denote by $v_{\ell ,k}$ the endnode of the path with edge index set $E_{\ell  j_k}$ incident with $e_1$ (for $k=1$) and $e_\rho$ (for $k=2$), and by $w_{\ell ,k}$ the other endnode of this  path.

For every sensitive bonsai $E_\ell$ such that $j_1,j_2 \notin f^*(E_\ell)$, we define a vertex $v_{\ell }^*$ as follows. If $u=u'$ and $I(E_u)=\emptyset$, then $v_\ell^*=v_{u,1}=v_{u,2}$; otherwise, in case (i) (resp., (ii) and (iii)) of Claim 3,
let $v_{\ell }^*=v_{u,1}$ (resp., $v_\ell^*=v_{i(\ell),2}$ and $v_\ell^*=v_{u',2}$). 
\vspace{.5cm}

\begin{tabbing}
\textbf{Procedure\,\,GAR($A$,$\{\epsilon,\rho\}$)}\\

\textbf{Input:}\, \= A non-network matrix $A$ and two row indexes $\epsilon$ and $\rho \neq 1$. \\

\textbf{Output: }\=  A basic $\{1,\rho\}$-central representation $G(A_{R(j_1,j_2)\bullet}')$ for some row index subset \\
\>  $R(j_1,j_2)$ of $A$, where $j_1\in S_1$, $j_2 \in S_2$, or determines that $A$ is not $\{\epsilon,\rho\}$-central.\\ 

1)\verb"  "\= call {\tt Truthassignment}($A$,$\{\epsilon,\rho\}$) outputing $U$, $(j_1,j_2)$ and a truth assignment of \\
\> $\Omega(U,j_1,j_2)$, or the fact that $A$ is not $\{\epsilon,\rho\}$-central;\\

2)\verb"  "\= consider basic network representations  $G(L_u)$ (in  case $U=(E_{u})$), $G(L_{u,1})$ and  \\
\> $G(L_{u',2})$ (in case $U=(E_{u},E_{u'})$ ), $G(A_\ell^\cap)$ for any jointly shared bonsai $E_\ell \notin U$ \\
\> such that $x_\ell=0$, and $G(N_\ell)$ for every $E_\ell \in V(j_1,j_2) \verb"\" \{ E_u, E_{u+1},\ldots,E_{u'} \}$;  contract all \\
\> artificial edges and  create the adjacent edges $e_1=[w_1,w_\rho]$ then $e_\rho=]w_{\rho-1},w_\rho]$;\\

3) \>  for $k=1$, $2$ and any two successive bonsais $E_\ell$ and $E_{\ell '}$ in the ordering (\ref{eqnCellularj1}) or (\ref{eqnCellularj2})\\ 
\> ($E_\ell \prec^{j_k} E_{\ell '}$) identify $v_{\ell ',k}$ with $w_{\ell ,k}$  and identify $v_{1,1}$ with $w_1$ and $v_{\ell _1+1,2}$ with $w_{\rho-1}$;\\

4) \> for every sensitive bonsai $E_\ell$ such that $j_1,j_2 \notin f^*(E_\ell)$, identify  $v_\ell$ with $v_\ell^*$;\\

5) \> compute a forest $T_1(j_1,j_2)$ as in Claim 4;\\
\> for every  root vertex $E_\ell$ of $T_1(j_1,j_2)$, identify $v_\ell$ with $w_\rho$;\\
 \> for any edge $(E_\ell,E_{\ell '})_{E_{\ell '}^k}$ in $T_1(j_1,j_2)$, identify $v_\ell$ with  the endnode ($\neq v_{\ell '}$)\\
 \>  of the $B_{\ell '}$-path with edge index set $E_{\ell '}^k$ in $G(N_{\ell '})$;\\

\>output the resulting bidirected graph $G(A_{R(j_1,j_2)\bullet}')$;

\end{tabbing}

\begin{prop}\label{propCentralGARI}
The output of GAR is correct.
\end{prop}

\begin{proof}
This results from Proposition \ref{propCentraltruth} and
Claims 0, 1, 2, 3 and 4.
\end{proof}\\

Finally, assume that we are given a left-compatible set $U$, a $U$-spanning pair $(j_1,j_2)$ such that the graph $F^*
(\overline{V(j_1,j_2)})$ is bipartite and a truth assignment of $\Omega(U,j_1,j_2)$. Let $R_i(j_1,j_2)=\cup_{E_\ell \in X_i(j_1,j_2)} E_\ell$ for $i=0$ and $1$. We state a last claim.\\

\noindent
{\bf Claim 5.} \quad The graph $F^*(\overline{V(j_1,j_2)})$
is partitionable into  two colour classes, say $Y_0$ and $Y_1$, such that

\begin{itemize}

\item for any $E_\ell \in \overline{V(j_1,j_2)}$, if $E_\ell\in Y_0$ (respectively, $E_\ell \in Y_1$), then $Opposite(E_\ell) \subseteq R_1(j_1,j_2)$ (respectively, $Opposite(E_\ell) \subseteq R_0(j_1,j_2)$);

\item for $k=0$ and $1$, there exists a forest in $D$, denoted as $T_k$, with vertex set $Y_k$; and
for any $E_\ell$ and $E_{\ell '}$ such that $f^*(E_\ell) \cap f^*(E_{\ell '}) \neq \emptyset$, either the bonsais are in two distinct colour classes, or they are contained in a same subpath of $T_0$ or $T_1$.\\

\end{itemize}

\noindent
{\bf Proof of Claim 5.} \quad Thanks to some clauses in $\Omega(U,j_1,j_2)$, up to a renumbering of the connected components of $F^*(\overline{V(j_1,j_2)})$, we may assume that for all $1\le \kappa \le \xi$, $Opposite(\mathcal{U}_\kappa^0) \subseteq R_1(j_1,j_2)$ and $Opposite(\mathcal{U}_\kappa^1) \subseteq R_0(j_1,j_2)$. Let $Y_0=\cup_{\kappa=1}^\xi \mathcal{U}_\kappa^0$ and $Y_1=\cup_{\kappa=1}^\xi \mathcal{U}_\kappa^1$. Clearly, $Y_0$ and $Y_1$ yield a bipartition of $F^*(\overline{V(j_1,j_2)})$ into two colour classes. By definition of $F^*(\overline{V(j_1,j_2)})$,
in each colour class, if two bonsais $E_\ell$ and $E_{\ell '}$ are such that $f^*(E_\ell) \cap f^*(E_{\ell '}) \neq \emptyset$, then $(E_\ell,E_{\ell '}) \in D$ or $(E_{\ell '},E_\ell) \in D$. Then, 
following the proof of claim 4, one proves the claim.
{\hfill$\BBox{\rule{.3mm}{3mm}}$}\\

Suppose that the procedure GAR has output a basic $\{1,\rho\}$-central representation \\
$G(A_{R(j_1,j_2)\bullet}')$ of the matrix $A_{R(j_1,j_2)\bullet}'$. For any $k\in\{0,1\}$ and every root vertex $E_\ell\in T_k$, we define  a node $z_\ell$ as follows. By claim 5, for all $j,j'\in f^*(E_\ell)$, 
$A_{R(j_1,j_2)\times \{ j\}}'\cap R_k(j_1,j_2)= A_{R(j_1,j_2)\times \{ j' \}}'\cap R_k(j_1,j_2)$. Let $z_\ell$ be the endnode, distinct from $w_\rho$, of the path in $G(A_{R(j_1,j_2)\bullet}')$ with edge index set $A_{R(j_1,j_2)\times \{ j\}}'\cap R_k(j_1,j_2)$ for any $j\in f^*(E_\ell)$.
We are now prepared to state the main procedure.

\begin{tabbing}
\textbf{Procedure\,\,CentralI($A$,$\{\epsilon,\rho\}$)}\\

\textbf{Input: }  \= A non-network matrix $A$ and two row indexes $\epsilon$ and $\rho \neq 1$  such that $S_0=\emptyset$. \\

\textbf{Output: } \=  A basic $\{\epsilon,\rho\}$-central representation $G(A)$ of $A$, or
determines that none exists.\\

1)\verb"  "\=call {\tt GAR}($A$,$\{\epsilon,\rho\}$) outputing 
a basic $\{1,\rho\}$-central representation $G(A_{R(j_1,j_2)\bullet}')$ \\
\> where $R(j_1,j_2)$ is  some row index subset  of some matrix $A'$, \\
\> or the fact that $A$ is not $\{\epsilon,\rho\}$-central;\\

2) \> for every $E_\ell\in  \overline{V(j_1,j_2)}$, compute a $v_\ell$-rooted network representation $B_\ell$ of $N_\ell$, \\
\> if one exists, otherwise STOP: output that $A$ is not $\{\epsilon,\rho\}$-central;\\

3) \> compute some forests $T_0$ and $T_1$ as in Claim 5; \\

4) \> for every root vertex $E_\ell$ of $T_0$ or $T_1$, identify $v_\ell$ with $z_\ell$;\\

 \> for any edge $(E_\ell,E_{\ell '})_{E_{\ell '}^k}$ in $T_0$ or $T_1$, identify $v_\ell$ with the endnode ($\neq v_{\ell '}$) of the $B_{\ell '}$-path\\
 \>  with edge index set $E_{\ell '}^k$ in $B_{\ell '}$; \\

\> up to a renumbering of the basic edges $e_1$ and $e_\epsilon$, output a basic $\{\epsilon,\rho\}$-central \\ 
\> representation of $A$;

\end{tabbing}

\noindent
{\bf Proof of Theorems \ref{thmcentralCarEmpty} and 
\ref{thmS0empty}.} \quad Let us prove the correctness of the procedure CentralI. By Proposition \ref{propCentralGARI}, the subroutine GAR in step 1 outputs  a basic $\{1,\rho\}$-central representation $G(A_{R(j_1,j_2)\bullet}')$ for some row index subset   $R(j_1,j_2)$ of $A$, where $j_1\in S_1$, $j_2 \in S_2$, or determines that $A$ is not $\{\epsilon,\rho\}$-central.
Then, in step 2, if one stops, then by Proposition \ref{propSkwatered}, $A$ is not $\{\epsilon,\rho\}$-central; otherwise, using Claim 5, in step 4, a basic $\{1,\rho\}$-central representation of $A'$ is computed.
At last, seeing step 1 of the subroutine Initialization, it is clear that the procedure CentralI outputs a basic $\{\epsilon,\rho\}$-central representation of $A$ if and only if one exists. 

The proof of Theorem \ref{thmS0empty} follows from 
Theorem \ref{thmcentralCarEmpty} and the tests and computations performed by the procedure CentralI.

Let us analyze the running time of CentralI. In the procedure Initialization, if $E_\ell$ is sensitive and $E_{\ell '}$ is $S_1$, $S_2$ linked to $E_\ell$ in step 3, then 
the bonsai matrix associated with $E_\ell\cup E_{\ell '}$ is not a network matrix. Thus we essentially have to check whether  the bonsai matrices $N_1,\ldots,N_b$ are network matrices, where $E_1,\ldots,E_b$ are given in step 2. By Theorem \ref{thmSubclassNTutteCunNet}, this takes time $O(n\alpha)$. The computation of the digraph $D$ described in Section \ref{sec:DefDigraphD} takes time $O(nm \alpha)$ by Lemma \ref{lemdigraphDtimeD}, and since $m\le 4\left( \begin{array}{c}
n\\
2 
\end{array} \right) + 2n +1 $ (see the output of the procedure Camion) the required time is bounded by $C n^3 \alpha$ for some constant $C$.

Let $U$ be a pair of shared bonsais as in step 4 of the procedure Truthassignment. By Theorem \ref{thmNetCorelatedMain}, if $U=E_u$, then computing a network representation of $L_u$ in which $e_1$ and $e_\rho$ are nonalternating takes time $O(n^2 \alpha)$. 
One interesting  observation is that for any bonsai $E_\ell$, the matrices $N_\ell$, $A_\ell$, $A_\ell^\cap$, $L_\ell$, $L_{\ell ,1}$ and $L_{\ell ,2}$  associated with $E_\ell$ are independent from $U$. Thus, it is possible to compute network representations of these matrices, if they exist, before performing the computations in step 4 of the procedure Truthassignment.
Checking whether every bonsai or pair of bonsais is left-feasible or right-feasible performs in time $O(n \alpha)$.
Computing a truth assignment of $\Omega(U,j_1,j_2)$, for any $U$-spanning pair $(j_1,j_2)$, takes time $O(n^2)$ which is bounded by $O(n\alpha)$,  as $n^2 \le n \alpha$. Finally,
 the computations in steps 2, 3 and 4 of the procedure CentralI can be executed  in time $O(n\alpha)$. Altogether, the running time of CentralI is $O(n^3 \alpha)$.
{\hfill$\BBox{\rule{.3mm}{3mm}}$}

\section{The procedure CentralII}\label{sec:S0notempty}

Througout this section, we assume that $S_0 \neq \emptyset$, and $\epsilon=1$ (except in the procedures). We provide a proof of Theorems \ref{thmcentralCarNotEmpty} and \ref{thmS0notempty}.

Suppose that $A$ has a $\{1,\rho\}$-central representation $G(A)$. A bonsai $E_u$ ($1\le u \le b$)
is called \emph{right-extreme}\index{right-extreme} if $B_u$ is $S_0$-straight and on the right of $\{ e_1,e_\rho\}$, and $E_u$ has no $S_0$-straight descendant in $T_{G_1(A)}$ (see page \pageref{mycounter9} for the definition of $T_{G_1(A)}$). The set of right-extreme bonsais is called the \emph{right-extreme set}\index{right-extreme set}. Its cardinality is clearly at most two. 

Like in the previous section, the notion of right-extreme set is related to that of right-compatibility defined as follows. We say that a set $U$ of at most two $S_0$-straight bonsais is \emph{right-compatible}\index{right-compatible} if the following holds.

\begin{itemize}

\item For any $E_u\in U$, $J_u^2=\emptyset$ and $N_u$ is a network matrix.

\item There exists a right-feasible subgraph of $D$ with vertex set $(V_0 \verb"\"V_{st})\cup U$, and in the case where $U=\{E_u,E_{u'}\}$, $E_u$ and $E_{u'}$ are not in a same subpath of this subgraph.

\end{itemize}

\noindent

\begin{lem}\label{lemCellularF*VS0}
Suppose that $A$ has a $\{1,\rho\}$-central representation. Let $U$ be the right-extreme set. Then $U$ is right-compatible.
\end{lem}

\begin{proof}
By Lemma \ref{lemright}, for any $E_u \in U$, $J_u^2= \emptyset$ and $N_u$ is a network matrix.
On the other hand, by Theorem \ref{thmdigraphrightfeasible} the forest $T_{G_1(A)}$ is a right-feasible subgraph of $D$. 
The vertex set $(V_0 \verb"\" V_{st}) \cup U$ might be not closed in $T_{G_1(A)}$. However,
using the transitivity of the relation $\prec_D$, it results that there exists a right-feasible subgraph of $D$ with vertex set $(V_0 \verb"\"V_{st})\cup U$, and in the case where $U=\{E_u,E_{u'}\}$, $E_u$ and $E_{u'}$ are not in a same subpath of this subgraph. (Otherwise, $E_u$ (or $E_{u'}$) is a descendant of $E_{u'}$ (or $E_u$, respectively) in $T_{G_1(A)}$, contradicting the definition of a right-extreme set.)
Therefore $U$ is right-compatible.
\end{proof}\\

The way of dealing with the graph $F^*(\overline{V_0'})$ is the same as with $F^*( \overline{V(j_1,j_2)})$ in Section \ref{sec:S0empty}. Let $$R=\cup_{E_\ell\in V_0' } E_\ell\cup \{1,\rho\}.$$  

\begin{lem}
If  $A$ has a $\{1,\rho\}$-central  representation $G(A)$, then the basic subgraph with edge index set $R$ is a $1$-tree.
\end{lem}

\begin{proof}
We know that the fundamental circuit of a nonbasic edge with index in $S_0$ contains the whole basic cycle, and for the rest of the proof see Lemma \ref{lemCentralRj}.
\end{proof}\\

The Proposition \ref{propOppositeII} below shows the bipartiteness of the graph $F^*(\overline{V_0'})$, provided that $A$ is $\{1,\rho\}$-central.
By assuming that $F^*(\overline{V_0'})$ is bipartite, we denote by $\mathcal{U}_1,\ldots,\mathcal{U}_\xi$ the connected components of $F^*(\overline{V_0'})$ and for each component $\mathcal{U}_{\kappa}$ ($1\le \kappa \le \xi$), let $\mathcal{U}_{\kappa}= \mathcal{U}_{\kappa}^0 \biguplus \mathcal{U}_{\kappa}^1$ be a bipartition of $\mathcal{U}_{\kappa}$ into two colour classes. For any $E_\ell\in \overline{V_0'}$, let $$Opposite(E_\ell)=  ( \cup_{j \in f^*(E_\ell)} s(A_{\bullet j})- \cap_{j \in f^*(E_\ell)} s(A_{\bullet j}) ) \cap R,$$ and 
for any $1\le \kappa \le \xi$ and $i=0$ and $1$, let 
$Opposite(\mathcal{U}_{\kappa}^i)=\cup_{E_\ell \in \mathcal{U}_{\kappa}^i} Opposite(E_\ell)$. 
The following proposition shows a part of Theorem \ref{thmS0notempty}.

\begin{prop}\label{propOppositeII}
Suppose that $A$ has a basic $\{1,\rho\}$-central  representation $G(A)$. Then

\begin{itemize}

\item[(i)] The graph $F^*(\overline{V_0'})$ is bipartite.

\item[(ii)] For any $1 \le \kappa \le \xi$, $i,i' \in \{ 0 , 1\}$, $E_\ell \in \mathcal{U}_{\kappa}^i$ and $E_{\ell '} \in \mathcal{U}_{\kappa}^{i'}$, $B_\ell$ and $B_{\ell '}$ are at different sides of $\{ e_1, e_\rho\}$ if and only if $i\neq i'$.

\item[(iii)] For any $1 \le \kappa \le \xi$, $i\in \{ 0 , 1\}$ and $E_\ell \in \mathcal{U}_{\kappa}^i$, the bonsai $B_\ell$ and the subgraph of $G(A)$ with edge index set $Opposite(\mathcal{U}_{\kappa}^i)$ are on both sides of $\{ 
e_1,e_\rho\}$.

\end{itemize}

\end{prop}

\begin{proof}
The proposition can be proved along the same lines as Proposition \ref{propOpposite}.
\end{proof}\\

Let $U$ be a right-compatible set and $j_0 \in S_0$. We define relations $\prec^1$ and $\prec^2$ on $V_{st} \cup \{ E_\ell \, : \, E_\ell \m{ is  sensitive} \}$. For any $k\in \{1,2\}$, $E_\ell \in V_{st}$ and $E_{\ell '}\in V_{st} \cup \{ E_{\ell ''} \, : \, E_{\ell ''} \m{ is sensitive} \}$, 
$$E_\ell \prec^k E_{\ell '}  \,\,\, \Leftrightarrow \,\,\, f_{S_k} (E_{\ell '}) \subseteq \{ \beta \in f_{S_k}(E_\ell) \,: \, E_{\ell  \beta}= E_{\ell  j_0} \}.$$ Clearly, the relations $\prec^1$ and $\prec^2$ are transitive. 
For all $1\le \ell \le b$, let $E_\ell'=E_\ell \cup \{1,\rho\}$.
A bonsai $E_\ell \in V_0' \verb"\" U$ is said to be \emph{right-feasible}\index{right-feasible!bonsai} if we have

\begin{itemize}

\item[$\mu.1$] $J_\ell^2= \emptyset$ and $N_\ell$ is a network matrix;
\item[$\mu.2$] if $E_\ell \in V_{st}$, then there exists some bonsai $E_u\in U$ such that $(E_\ell,E_u)\in D$ 
and for any $E_{\ell '}\in V_0 \verb"\" (V_2\cup V_{st})$, if $(E_{\ell '},E_u)\in D$ then $(E_{\ell '},E_\ell) \in D$;

\end{itemize}

\noindent
and \emph{left-feasible}\index{left-feasible!bonsai} if we have

\begin{itemize}

\item[$\mu.3$] $E_\ell$ is sensitive or $S_0$-straight;

\item[$\mu.4$] if $E_\ell$ is $S_0$-straight, then the matrix $L_\ell^0=[A_{E_\ell' \times f(E_\ell)} \, \chi_{E_{\ell j_0} \cup \{1,\rho \}  }^{E_\ell'} ]$ is a network matrix;

\end{itemize}

\begin{lem}\label{lemCellularfeasible2}
Suppose that $A$ has a $\{1,\rho\}$-central  representation $G(A)$. Let $U$ be the right-extreme set of bonsais. For any $E_\ell \in V_0'\verb"\" U$, if the bonsai $B_\ell$
is on the right (respectively, the left) of $\{ e_1,e_\rho\}$, then it is right-feasible (respectively, left-feasible).
\end{lem}

\begin{proof}
Let $E_\ell \in V_0'\verb"\" U$. Suppose that $B_\ell$ is on the right of $\{ e_1,e_\rho\}$. By Lemma \ref{lemright}, $J_\ell^2=\emptyset$ and $N_\ell$ is a network matrix. Assume $E_\ell \in V_{st}$. 
By definition of $U$, $E_\ell$ has a descendant, say $E_u\in U$, in $T_{G_1(A)}$. 
Now let $E_{\ell '} \in V_0-(V_{st}+V_2)$ such that 
$(E_{\ell '},E_u) \in D$. If $f_{S_k}(E_{\ell '})=\emptyset$ for all $k\in \{1,2\}$, then clearly $(E_{\ell '},E_\ell) \in D$. Otherwise, let $j\in f_{S_k}(E_{\ell '})$ for some $k\in \{1,2\}$. Since $E_{\ell '} \notin V_{st}\cup V_2$, $(E_u,E_{\ell '}) \notin D$. As $j\in f_{S_k}(E_{\ell '}) \cap f_{S_k}(E_u)$, $E_{\ell '}$ is an ancestor of $E_u$ in $T_{G_1(A)}$. Now, since $E_\ell$ is also an ancestor of $E_u$ in $T_{G_1(A)}$ and $g_{j'}(E_\ell)=g_{j'}(E_{\ell '})=g_{j'}(E_u)=1$ for some $j' \in S_0$,  it follows that  $E_\ell$, $E_{\ell '}$ and $E_u$ are contained in a same subpath of $T_{G_1(A)}$. Finally, the fact that $E_{\ell '} \notin V_{st}$ and $E_\ell \in V_{st}$ implies that $(E_\ell,E_{\ell '}) \notin D$. So $(E_{\ell '},E_\ell) \in D$.

Now suppose that $B_\ell$ is on the left of $\{ e_1,e_\rho\}$. If $E_\ell\in V_0$, then by Lemma \ref{lemCentralElS} (ii) $B_\ell$ contains at least one edge of the basic cycle, $E_\ell$ is $S_0$-straight and $E_{\ell j} =I(v_{\ell ,1},v_{\ell ,2})$ for any $j\in S_0$;
we deduce that the matrix  $L_\ell^0$ is a network matrix. If $E_\ell\notin V_0$, then by definition of $V_0'$ the bonsai $E_\ell$ is shared,  and it results from Lemmas \ref{lemCentralElS} (ii) and \ref{lemroot} that $E_\ell$ is sensitive.
\end{proof}\\

A pair of bonsais $E_\ell,E_{\ell '}\in V_0' \verb"\"U$ such that $f^*(E_\ell) \cap f^*(E_{\ell '}) \neq \emptyset$ is said to be \emph{right-feasible}\index{right-feasible!pair of bonsais} if we have 

\begin{itemize}
\item[$\delta.1$] if $E_\ell$ and $E_{\ell '}$ are $S_k$-linked for some $k\in \{ 1,2\}$, then $(E_\ell,E_{\ell '}) \in D$ or $(E_{\ell '},E_\ell) \in D$;
\end{itemize}

\noindent 
and \emph{left-feasible}\index{left-feasible!pair of bonsais} if the following holds.

\begin{itemize}

\item[$\delta.2$] The bonsais are not both sensitive.

\item[$\delta.3$] If they are both $S_0$-straight, then $E_\ell \prec^k E_{\ell '}$ and $E_{\ell '} \prec^{k'} E_\ell$ for some $k,k' \in \{1,2\}$, $k\neq k'$.

\item[$\delta.4$] if $E_\ell$ is $S_0$-straight and $E_{\ell '}$ sensitive, then $E_\ell \prec^k E_{\ell '}$ for some $k\in \{1,2\}$, and
$E_\ell$ and $E_{\ell '}$ are not $S_1$, $S_2$ linked.

\end{itemize}

\begin{lem}\label{lemCellularpair2}
Suppose that $A$ has a $\{1,\rho\}$-central  representation $G(A)$ and assumption $\mathscr{A}$ is satisfied. Let $U$ be the right-extreme set of bonsais and $E_\ell,E_{\ell '}\in V_0'\verb"\" U$ such that $f^*(E_\ell) \cap f^*(E_{\ell '})\neq \emptyset$. If $B_\ell$ and $B_{\ell '}$ are both on the right (respectively, the left) of $\{ e_1,e_\rho\}$, then the pair is right-feasible (respectively, left-feasible).
\end{lem}

\begin{proof}
Suppose first that $B_\ell$ and $B_{\ell '}$ are on the right  of $\{e_1,e_\rho\}$ and they are
$S_k$-linked for some $k\in \{1,2\}$. Then they belong to a same subpath of $T_{G_1(A)}$. Using the transitivity of $\prec_D$, we get that $(E_\ell,E_{\ell '}) \in D$ or $(E_{\ell '},E_\ell) \in D$.

Now suppose that $B_\ell$ and $B_{\ell '}$ are on the left of $\{e_1,e_\rho\}$. They can not be both sensitive (otherwise from Lemmas \ref{lembeta2} and \ref{lem2rid}, it follows that $v_{\ell ,1}=v_{\ell ,2}$ and $v_{\ell ',1}=v_{\ell ',2}$, hence $f^*(E_\ell)\cap f^*(E_{\ell '})=\emptyset$). 
If $E_\ell$ and $E_{\ell '}$ are both $S_0$-straight, then the proof that they are left-feasible is straightforward.
At last, assume that $E_\ell$ is $S_0$-straight and $E_{\ell '}$ sensitive. By Lemma \ref{lembeta2}, $B_{\ell '}$
is not $S_1,S_2$ linked to $B_\ell$, so that $B_{\ell '}$ is either preceding $B_\ell$ in which case $E_\ell \prec^{2} E_{\ell '}$, or succeeding $B_\ell$ and $E_\ell \prec^{1} E_{\ell '}$.
\end{proof}\\

Assume that the graph $F^*(\overline{V_0'})$ is bipartite.
Given a right-compatible set $U$, 
we are now ready to construct the instance $\Lambda(U)$ of the $2$-SAT problem. 
The set of variables is $\{ x_\ell \, : \, E_\ell \in V_0' \}$. 
Provided that $A$ is $\{1,\rho\}$-central, the equalities $x_\ell=0$ and $x_\ell=1$ mean that the bonsai $B_\ell$ is on the left and on the right of $\{e_1,e_\rho\}$, respectively, in some $\{1,\rho\}$-central representation of $A$. 

For any $E_u\in U$, set $x_u=1$ in $\Lambda(U)$.
For any $E_\ell\in V_0'$, if $E_\ell$ is not right-feasible (respectively, left-feasible), then set $x_\ell=0$ (respectively, $x_\ell=1$).
For any pair of bonsais $E_\ell,E_{\ell '} \in V_0'\verb"\"U$ such that $f^*(E_\ell) \cap f^*(E_{\ell '}) \neq \emptyset$, if the pair is not right-feasible (resp., left-feasible), then put the clause $\bar x_\ell \vee \bar x_{\ell '} $ (resp., 
$x_\ell \vee x_{\ell '}$) in $\Lambda(U)$. Thus if the pair is not right-feasible (respectively, left-feasible), then at most one of the variables $x_\ell$ and $x_{\ell '}$ has  value $1$ (respectively, $0$). 

For $\kappa=1,\ldots,\xi$, do as follows. For any $i\in \{0,1\}$ and two variables $x_\ell$ and $x_{\ell '}$ such that $E_\ell \, \cap \, Opposite(\mathcal{U}_\kappa^i)$ $\neq \emptyset $ and $E_{\ell '} \cap Opposite(\mathcal{U}_\kappa^i) \neq \emptyset $, put the equality $x_\ell = x_{\ell '}$ in $\Lambda(U)$. Moreover, if $Opposite(\mathcal{U}_\kappa^0) \neq \emptyset$ and $Opposite(\mathcal{U}_\kappa^{1}) \neq \emptyset$, choose some $x_\ell$ and $x_{\ell '}$ such that 
$E_\ell \cap Opposite(\mathcal{U}_\kappa^0)\neq \emptyset $ and $E_{\ell '} \cap Opposite(\mathcal{U}_\kappa^{1})\neq \emptyset $. Then put the clauses $x_\ell \vee x_{\ell '}$ and $\bar x_\ell \vee \bar x_{\ell '}$ in $\Lambda(U)$. These clauses ensure that the variables $x_\ell$ and $x_{\ell '}$ have different values.

\begin{prop}\label{propCellulartruth2}
Suppose that $A$ has a $\{1,\rho\}$-central representation and assumption $\mathscr{A}$ is satisfied. Then, there exists a right-compatible set $U$ such that the instance $\Lambda(U)$ has a truth assignment.
\end{prop}

\begin{proof}
Let $U$ be the right-extreme set. By Lemma \ref{lemCellularF*VS0}, $U$ is right-compatible. For each bonsai $E_\ell$, let $x_\ell=0$ if $E_\ell$ is on the left of $\{e_1,e_\rho\}$, and $x_\ell=1$ otherwise.
By Lemmas \ref{lemCellularfeasible2} and \ref{lemCellularpair2} and  Proposition \ref{propOppositeII} (ii) and (iii), we deduce that this is a truth assignment.
\end{proof}\\

The following procedure computes a truth assignment of 
$\Lambda(U)$ for some right-compatible set $U$, provided that $A$ is $\{1,\rho\}$-central.

\begin{tabbing}
\textbf{Procedure\,\,Truthassignment($A$,$\{\epsilon,\rho\}$)}\\

\textbf{Input:}\, \= A non-network matrix $A$ and two row indexes $\epsilon$ and  $\rho\neq 1$ such that $S_0\neq \emptyset$.\\

\textbf{Output:} \=  Either a right-compatible set $U$ with a truth assignment of $\Lambda(U)$, \\
\> or determines that $A$ is  not $\{\epsilon ,\rho\}$-central.\\

1)\verb"  "\= call {\tt Initialization}($A$,$\{\epsilon,\rho\}$)
outputing a matrix $A'$;\\

2) \> check whether $F^*(\overline{V_0'})$ is bipartite; if it is not,  then output that $A$ is  not $\{\epsilon ,\rho\}$-central;\\

3) \> compute a digraph $D$ with respect to $A'$ and the row index subset\\
\> $R^*=\{1,\rho\}$ of $A'$.\\

4) \> {\bf for } \= every  set $U$ of $S_0$-straight bonsais, $|U|\le 2$, {\bf do }\\

\>  \> compute whether $U$ is right-compatible; if it is not, then return to 4;\\

\> \> compute the instance $\Lambda(U)$ and a truth assignment of $\Lambda(U)$, if one exists,\\
\> \> output it with $U$ and STOP, otherwise return to 4;\\

 \> {\bf endfor }\\

\>  output that $A$ is not $\{\epsilon,\rho\}$-central;

\end{tabbing}

\begin{prop}\label{propCentraltruthII}
The output of the procedure Truthassignment is correct.
\end{prop}

\begin{proof}
By Lemma \ref{lemCentralInit}, the subroutine Initialization outputs a matrix $A'$ of size $n\times m'$ whose $A_{\bullet \{1,\ldots,m \}}'=A$ and such that $A'$ is $\{1,\rho\}$-central if and only if $A$ is $\{\epsilon,\rho\}$-central, and assumption $\mathscr{A}$ is satisfied for $A'$. 

Suppose that $A$ is $\{\epsilon,\rho\}$-central. Then, let $G(A')$ be a $\{1,\rho\}$-central representation of $A'$. In step 2, by Proposition \ref{propOppositeII} (i), the graph
$F^*(\overline{V_0'})$ is bipartite.
 Finally, using Proposition \ref{propCellulartruth2}, the procedure Truthassignment outputs a truth assignment of $\Lambda(U)$. This concludes the proof.
\end{proof}\\

Before proving Theorem \ref{thmS0notempty}, we see some properties ensuing from a truth assignment of $\Lambda(U)$. Let $U$ be a right-compatible set. If $|U|=2$, then we write $U=(E_u,E_{u'})$, otherwise $U=(E_{u})$.
We suppose that each variable $x_\ell$ corresponding to a bonsai $E_\ell \in V_0'$ has received a value $0$ or $1$, yielding a truth assignment of $\Lambda(U)$. Let $$X_k=\{ E_\ell \in V_0' \, : \, x_\ell=k \}$$ for $k=0$ and $1$. We notice that if a variable $x_\ell$ is equal to $0$, then it is left-feasible and the corresponding bonsai $E_\ell$ is either sensitive or $S_0$-straight. Up to a renumbering of the bonsais in $X_0$, we may assume that $X_0\cap V_{st}= \{ E_1,E_2,\ldots, E_t \}$. 
We now state some useful claims.\\

\noindent
{\bf Claim 0.} \quad For any sensitive bonsai $E_\ell$, $x_\ell=0$ and there is no bonsai $E_{\ell '}$  that is  $S_1$, $S_2$ linked to $E_\ell$ and such that $x_{\ell '}=0$.\\

\noindent
{\bf Claim 1.} \quad For any $E_\ell \in V_0'\verb"\" V_{st}$,   $x_\ell=0$ if and only if $E_\ell$ is sensitive.\\

\noindent
{\bf Claim 2.} \quad  Up to a renumbering of the bonsais, there exist orderings, say 

\begin{eqnarray}\label{eqnCellularI}
E_1 \prec^1 E_2 \prec^1 \cdots E_{t-1} \prec^1 E_t \label{eqnCellularI}\\
\m{ and \hspace{2cm}}  \nonumber\\
E_t \prec^2 E_{t-1} \prec^2 \cdots E_2 \prec^2 E_1 \label{eqnCellularII}
\end{eqnarray} 

\noindent
of the set $X_0\cap V_{st}$. \\

\noindent
{\bf Claim 3.} \quad For every sensitive bonsai $E_\ell$, either there exists an index $i(\ell)$ with $1\le i(\ell) \le t$ such that $E_{i(\ell)} \prec^1 E_\ell$ and $E_{i(\ell)+1} \prec^2 E_\ell$, or $E_1 \prec^2 E_\ell$.\\

\noindent
{\bf Claim 4.} \quad  There exists a right-feasible forest, say $TX_1$, in $D$ with vertex set $X_1$. \\

The proof of Claim 0 directly follows from $\mu.1$, $\mu.3$, $\delta.2$ and $\delta.4$. The proof of Claim 1 follows from Claim 0 and $\mu.3$. The proof of Claim 2 (respectively, 3) is close to the proof of 
Claim 2  (respectively, 3) in Section \ref{sec:S0empty}, using $\delta.3$ (respectively, $\delta.4$). The proof of Claim 4 is a consequence of the definition of a right-compatible set, $\mu.2$ and $\delta.1$.

Below, we describe a subroutine of CentralII, called GAR, which constructs a basic $\{1,\rho\}$-central representation of a matrix $A_{R \bullet}'$, where $A_{\bullet \times \{1,\ldots,m\}}'=A$, provided that $A$ is $\{\epsilon,\rho\}$-central. In the following procedure, for any $S_0$-straight bonsai $E_\ell$ such that $x_\ell=0$, 
if the matrix $L_u^0$ has a basic network representation $G(L_u^0)$, we denote by $v_{\ell ,1}$ (respectively,  $v_{\ell ,2}$) the initial (respectively, terminal) node of the path with edge index set $E_{\ell  j}$ for any $j\in S_0$. At last,
for every sensitive bonsai $E_\ell$, we define a vertex $v_{\ell }^*$ as follows. If there exists an index $i(\ell)$ as defined in Claim 3, then $v_{\ell }^*=v_{i(\ell),2}$, otherwise $v_{\ell }^*=v_{1,1}$.

\begin{tabbing}
\textbf{Procedure\,\,GAR($A$,$\{\epsilon,\rho\}$)}\\

\textbf{Input:}\, \= A non-network matrix $A$ and two row indexes $\epsilon$ and $\rho \neq 1$. \\

\textbf{Output:}\=  A basic $\{1,\rho\}$-central representation $G(A_{R \bullet}')$ of a matrix $A_{R \bullet}'$,  \\
\>  or determines that $A$ is not $\{\epsilon,\rho\}$-central.\\ 

1)\verb"  "\= call {\tt Truthassignment}($A$,$\{\epsilon,\rho\}$) outputing $U$ and a truth assignment of \\
\> $\Lambda(U)$, or the fact that $A$ is not $\{\epsilon,\rho\}$-central;\\

2)\verb"  "\= consider basic network representations  
$G(L_\ell^0)$ for any $E_\ell\in V_{st}\cap X_0$, and $G(N_\ell)$\\
\>  for any $E_\ell \in V_0'\verb"\" (V_{st}\cap X_0)$; contract all artificial edges and  create the edges \\
\>  $e_1=[w_1,w_\rho]$ and $e_\rho=]w_{\rho-1},w_\rho]$;\\

3) \>  for  any $1\le i \le t-1$, identify $v_{i,2}$ with $v_{i+1,1}$,\\
\>  and identify $v_{1,1}$ with $w_1$ and $v_{t,2}$ with $w_{\rho-1}$;\\

4) \> for every sensitive bonsai $E_\ell$, identify  $v_\ell$ with $v_\ell^*$;\\

5) \>  compute a forest $TX_1$ as in Claim 4;\\
\> for every  root vertex $E_\ell$ of $TX_1$, identify $v_\ell$ with $w_\rho$;\\
 \> for any edge $(E_\ell,E_{\ell '})_{E_{\ell '}^k}$ in $TX_1$, identify $v_\ell$ with  the endnode ($\neq v_{\ell '}$)\\
 \>  of the $B_{\ell '}$-path with edge index set $E_{\ell '}^k$ in $G(N_{\ell '})$;\\

\>output the resulting bidirected graph $G(A_{R\bullet}')$;

\end{tabbing}

\begin{prop}\label{propCentralGARII}
The output of GAR is correct.
\end{prop}

\begin{proof}
This results from Proposition \ref{propCentraltruthII} and
Claims 0, 1, 2, 3 and 4.
\end{proof}\\

Finally, assume that the graph $F^*(\overline{V_0'})$ is bipartite and we are given a right-compatible set $U$ and a truth assignment of $\Lambda(U)$. We define $R_i=\cup_{E_\ell \in V_i} E_\ell$ for $i=0$ and $1$.
We state a last claim.\\

\noindent
{\bf Claim 5.} \quad The graph $F^*(\overline{V_0'})$
is partitionable into  two colour classes, say $Y_0$ and $Y_1$, such that

\begin{itemize}

\item for any $E_\ell \in \overline{V_0'}$, if $E_\ell\in Y_0$ (respectively, $E_\ell \in Y_1$), then $Opposite(E_\ell) \subseteq R_1$ (respectively, $Opposite(E_\ell) \subseteq R_0$);

\item for $k=0$ and $1$, there exists a forest in $D$, denoted as $T_k$, with vertex set $Y_k$; and
for any $E_\ell$ and $E_{\ell '}$ such that $f^*(E_\ell) \cap f^*(E_{\ell '}) \neq \emptyset$, either the bonsais are in two distinct colour classes, or they are contained in a same subpath of $T_0$ or $T_1$.\\

\end{itemize}

\noindent
{\bf Proof of Claim 5.} \quad The proof is identical to the proof of Claim 5 in Section \ref{sec:S0empty}.
{\hfill$\BBox{\rule{.3mm}{3mm}}$}\\

Suppose that the procedure GAR has output a basic $\{1,\rho\}$-central representation $G(A_{R\bullet}')$ of a matrix $A_{R\bullet}'$. For any $k\in\{0,1\}$ and every root vertex $E_\ell\in T_k$, we define  a node $z_\ell$ as follows. By claim 5, for all $j,j'\in f^*(E_\ell)$, 
$A_{R\times \{ j\}}'\cap R_k= A_{R\times\{  j'\}}'\cap R_k$. Let $z_\ell$ be the endnode, distinct from $w_\rho$,
of the path in $G(A_{R\bullet}')$ with edge index set $A_{R \times \{ j\}}'\cap R_k$ for any $j\in f^*(E_\ell)$.
We are now prepared to state the main procedure.

\begin{tabbing}
\textbf{Procedure\,\,CentralII($A$,$\{\epsilon,\rho\}$)}\\

\textbf{Input: }  \= A non-network matrix $A$ and two row indexes $\epsilon$ and $\rho \neq 1$ such that $S_0\neq \emptyset$. \\

\textbf{Output: } \=  A basic $\{\epsilon,\rho\}$-central representation $G(A)$ of $A$, or
determines that none exists.\\

1)\verb"  "\=call {\tt GAR}($A$,$\{\epsilon,\rho\}$) outputing a basic $\{1,\rho\}$-central representation $G(A_{R \bullet}')$ of a matrix $A_{R \bullet}'$,  \\
\>  or the fact that $A$ is not $\{\epsilon,\rho\}$-central.\\

2) \> for every $E_\ell\in \overline{V_0'}$, compute a $v_\ell$-rooted network representation $B_\ell$ of $N_\ell$, \\
\> if one exists, otherwise STOP: output that $A$ is not $\{\epsilon,\rho\}$-central;\\

3) \> compute some forests $T_0$ and $T_1$ as in Claim 5; \\

4) \> for every root vertex $E_\ell$ of $T_0$ or $T_1$, identify $v_\ell$ with $z_\ell$;\\

 \> for any edge $(E_\ell,E_{\ell '})_{E_{\ell '}^k}$ in $T_0$ or $T_1$, identify $v_\ell$ with the endnode ($\neq v_{\ell '}$) of the $B_{\ell '}$-path\\
 \>  with edge index set $E_{\ell '}^k$ in $B_{\ell '}$; \\

\> up to a renumbering of the basic edges $e_1$ and $e_\epsilon$, output a basic $\{\epsilon,\rho\}$-central \\ 
\> representation of $A$;

\end{tabbing}

\noindent
{\bf Proof of Theorems \ref{thmcentralCarNotEmpty} and  \ref{thmS0notempty}.} \quad  Let us prove the correctness of the procedure CentralII. By Proposition \ref{propCentralGARII}, the subroutine GAR in step 1 outputs  a basic $\{1,\rho\}$-central representation $G(A_{R\bullet}')$, or determines that $A$ is not $\{\epsilon,\rho\}$-central.

If one stops in step 2, then by Proposition \ref{propSkwatered} $A$ is not $\{\epsilon,\rho\}$-central; otherwise, using Claim 5, in step 4, a basic $\{1,\rho\}$-central representation of $A'$ is computed.
At last, seeing step 1 of the subroutine Initialization, it is clear that the procedure CentralII outputs a basic $\{\epsilon,\rho\}$-central representation of $A$. 

The proof of Theorem \ref{thmS0notempty} follows from 
Theorem \ref{thmcentralCarNotEmpty} and the tests and computations performed by the procedure CentralII.
The analysis of the running time of CentralII is the same as for CentralI. Thus, the number of operations in CentralII is $O(n^3 \alpha)$.
{\hfill$\BBox{\rule{.3mm}{3mm}}$}

\section{Recognizing $\{1,\rho\}$-noncorelated network matrices}\label{sec:CentralNetcorelated}

Let $A$ be a nonnegative connected network matrix $A$ of size $n \times m$, $\alpha$ the number of nonzero entries of $A$
and $\rho> 1$ a row index. Let $G(A)$ be a basic network representation of $A$. Our present goal is to recognize whether  $A$ is $\{1,\rho\}$-noncorelated. We analyze an algorithm which takes a basic network representation $G(A)$ as input. If $A$ is $\{1,\rho\}$-noncorelated, it  outputs a basic network representation $G'(A)$ such that $e_1$ and $e_\rho$ are alternating in $G(A)$ if and only if they are nonalternating in $G'(A)$. We will prove the following.

\begin{thm}\label{thmNetCorelatedMain}
Given the matrix $A$ there exists an algorithm which either determines that $A$ is  $\{1,\rho\}$-corelated, or 
provides basic network representations $G(A)$ and $G'(A)$ such that $e_1$ and $e_\rho$ are alternating in $G(A)$ if and only if they are nonalternating in $G'(A)$. The running time of the algorithm is bounded by $C(n^2 \alpha)$ for some constant $C$.
\end{thm}

Suppose first that each column of $A$ contains at most two nonzero entries. We consider the undirected graph $L(A)$ defined as follows. The vertex set is the row index set of $A$ and two row indexes $i$ and $i'$ are adjacent if and only if $s(A_{\bullet j})= \{Êi,i'\}$ for some column index $j$. With these preliminaries, we can state the following theorem.

\begin{thm}\label{thmNetCorelatedtwo}
Suppose that each column of $A$ has at most two nonzero entries. Then the matrix $A$ is $\{1,\rho\}$-noncorelated if and only if there exists a cutvertex in $L(A)$ separating $1$ from $\rho$.
\end{thm}

\begin{proof}
Suppose that there is no cutvertex seperating $1$ from $\rho$ in $L(A)$. Then either $1$ and $\rho$ are the endnodes of a cut-edge in $L(A)$ and $A$ is clearly 
$\{1,\rho\}$-corelated, or 
the vertices $1$ and $\rho$ are contained in a  $2$-connected subgraph of $L(A)$. So assume that $1$ and $\rho$ are in a $2$-connected subgraph of $L(A)$.
By Menger's theorem (see \cite{Diestelgraph} for example), the vertices $1$ and $\rho$ belong to a cycle $C$ of $L(A)$. This implies that
the subgraph of $G(A)$ with edge index set $V(C)$ represents a star. 
So, if both paths in $C$ between $1$ and $\rho$ have an even length, then $e_1$ and $e_\rho$ are alternating in $G(A)$, otherwise nonalternating. Therefore $A$ is $\{1,\rho\}$-corelated.

Conversely, let $i^*$ be a cutvertex in $L(A)$ separating $1$ from $\rho$. Let us denote by $w_1$ and $w_2$ the endnodes of $e_{i^*}$ in $G(A)$.
Let $L_2$ be the connected component of $L(A)\verb"\"\{i^*\}$ containing the vertex $\rho$. Observe that the subgraph  of $G(A)$ with edge index set $V(L_2)$ (and no isolated vertex), call it $T_2$, is connected and contains exactly one endnode of $e_{i^*}$, say $w_2$. Now consider the following procedure.

\begin{tabbing}
\textbf{Procedure\,\,Move-T2($G(A)$,$i^*$)}\\

\textbf{Input:} A basic network representation $G(A)$ and a cutvertex $i^*$ seperating $1$ from $\rho$.\\
\textbf{Output:} \= A basic network representation $G'(A)$ such that $e_1$ and $e_\rho$ are alternating \\
\> in $G(A)$ if and only if they are nonalternating in $G'(A)$. \\

1)\verb"  "\=  make $T_2$ loose from $e_{i^*}$ by making a copy of $w_2$, say $w_2'$, in $T_2$; \\
2) \> reverse the orientation of all edges in $T_2$ and identify $w_2'$ with $w_1$;\\
 \> output the obtained basic network representation $G'(A)$ of $A$; \\
\end{tabbing}

Let $G'(A)$ be output by the procedure Move-T2.
We observe that the minimal paths linking $e_{i^*}$ and $e_1$ in $G(A)$ and $G'(A)$ are identical. Moreover, $e_2$ and $e_{i^*}$ are alternating in $G(A)$ if and only if they are alternating in $G'(A)$. Then, by construction,
$e_1$ and $e_\rho$ are alternating in $G'(A)$ if and only if $e_1$ and $e_\rho$ are nonalternating in $G(A)$.
\end{proof}\\


From now on, we assume that $A$ has a column with at least three nonzeros. Let $e_{i^*}$ be the middle edge of a basic path of length 3 in $G(A)$. (Such a path clearly exists.) Let $R^*=\{ i^*\}Ê$, $D$ be the digraph with respect to $R^*$ as constructed in Section \ref{sec:DefDigraphD} and $S^*=\{ j \, : \, i^* \in s(A_{\bullet j}) \}$. If two basic edges $e_i$ and $e_{i'}$ are at different sides of $e_{i^*}$ in $G(A)$, then they can not belong to a same bonsai. So $D$ has at least two vertices.
Two bonsais $E_\ell, E_{\ell '}\in D$ are said to be \emph{independent}\index{independent bonsais} if there exist network representations $G'(A)$ and $G''(A)$ of $A$
such that $B_\ell$ and $B_{\ell '}$ are 
at the same side of $e_{i^*}$ in $G'(A)$ and at different sides of $e_{i^*}$ in $G''(A)$; otherwise they are \emph{dependent}\index{dependent bonsais}. 

\begin{lem}\label{lemNetCorelated1=i*}
If for some $1\le \ell \le b$ either $1\in E_\ell$ and $\rho=i^*$, or $\rho \in E_\ell$ and $1=i^*$, or $1,\rho \in E_\ell$, then the following holds. The matrix $A$ is $\{1,\rho\}$-noncorelated if and only if $N_\ell$ is $\{1,\rho\}$-noncorelated.
\end{lem}

\begin{proof} 
The proof is straightforward. 
\end{proof}\\

For the remaining part of the section, we will need the following assumption.
\begin{tabbing}
{\bf assumption $\mathscr{B}$:} \= The matrix $A_{\overline{\{i \}} \times \overline{f(\{i \})}}$ is connected for $i=1$ and $\rho$, $i^*\neq 1,\rho$ and \\
\> $1$ and $\rho$ are not contained in a same bonsai of $D$.
\end{tabbing}
\noindent
Whenever the assumption $\mathscr{B}$ is satisfied, we will assume that $1\in E_1$, and $\rho \in E_2$. Let us see an auxiliary lemma then an important theorem.

\begin{lem}\label{lemNetCorelatedutile}
Suppose that the assumption $\mathscr{B}$ is satisfied. In any basic network representation of $A$, if $e_{i^*}$, $e_1$ and $e_\rho$ are contained in a same path, then $e_{i^*}$ is between $e_1$ and $e_{\rho}$.
\end{lem}

\begin{proof}
Suppose by contradiction that $e_1$ lies between $e_{i^*}$ and $e_\rho$ in some path contained in $G(A)$. (The proof is similar with $e_1$ and $e_2$ interchanged.) Then  
the matrix $A_{\overline{\{1 \}} \times \overline{f(\{1 \})}}$ is not connected, a contradiction.
\end{proof}\\

\begin{thm}\label{thmNetCorelatednoncor}
Suppose that the assumption $\mathscr{B}$ is satisfied. Then,
$A$ is $\{1,\rho\}$-noncorelated if and only if $N_1$ is $\{1,i^* \}$-noncorelated, or $N_2$ is $\{\rho,i^* \}$-noncorelated, or $E_1$ and $E_2$ are independent.
\end{thm}

\begin{proof}
Suppose first that $N_1$ and $N_2$ are $\{1,i^*\}$-corelated and $\{\rho,i^*\}$-corelated, respectively, and $E_1$ and $E_2$ are dependent. Consider the case where $e_1$ and $e_{i^*}$ as well as $e_\rho$ and $e_{i^*}$ are 
nonalternating in all network representations of $N_1$ and $N_2$, respectively; and suppose that
$B_1$ and $B_2$ are at the same side of $e_{i^*}$ in all network representations of $A$. Let us denote by $e_{i^*}=(w_1, w_2)$ in $G(A)$.
Up to a reversing of the orientation of all edges, we may assume that $v_1^*=v_2^*=w_2$ (see page \pageref{mycounter5} for the definiton of $v_1^*$ and $v_2^*$). 

Observe that $e_1$ and $e_{i^*}$ as well as $e_\rho$ and $e_{i^*}$ are nonalternating in $G(A)$. (If $e_1$ and $e_{i^*}$ are alternating in $G(A)$, then by contracting all edges of $G(A)$ with index in $\{1,\ldots,n\}\verb"\"(E_1 \cup \{i^*\})$ we obtain a basic network representation of $N_1$ such that $e_1$ and $e_{i^*}$ are alternating, a contradiction.)
From Lemma \ref{lemNetCorelatedutile}, it results that $e_1$ (respectively, $e_\rho$) does not lie on the basic path from $v_2$ to $w_2$ (respectively, from $v_1$ to $w_2$). Thus $e_1$ and $e_\rho$ are alternating in $G(A)$. In the other cases, one can also prove with similar arguments that $A$ is $\{1,\rho\}$-corelated.\\

Conversely, suppose that $N_1$ is $\{1, i^*\}$-noncorelated.
Let $T_1$ be a basic network representation of $N_1$ such that $e_1$ and $e_{i^*}$ are alternating in $T_1$ if and only if they are nonalternating in $G(A)$. Recall that $v_1$ denotes the cutvertex of $e_{i^*}$ in $T_1$ as well as the closest node of $B_1$ in $G(A)$ to $e_{i^*}$. Let $T_1'$ be the tree obtained from $T_1$ by contracting $e_{i^*}$. Observe that the subgraphs of $T_1'$ and $B_1 \subseteq G(A)$ with edge index set $\cup_{j \in f^*(E_1)} s(A_{\bullet j}) \cap E_1$ are isomorphic and rooted at $v_1$.
Moreover, for all $j\in \bar f^*(E_1)$, $s(A_{\bullet j})\cap E_1$ is the edge index set of a directed path in $T_1'$ and $B_1\subseteq G(A)$.Thus it is possible to replace $B_1$ in $G(A)$ by $T_1'$ in order to obtain a basic network representation $G'(A)$ such that $e_1$ and $e_\rho$ are alternating in $G(A)$ if and only if they are nonalternating in $G'(A)$. Similarly, if $N_2$ is $\{\rho, i^*\}$-noncorelated, one can prove that $A$ is $\{1,\rho\}$-noncorelated. \\
At last, suppose that $N_1$ is $\{1,i^*\}$-corelated, $N_2$ is $\{\rho,i^*\}$-corelated and $E_1$ and $E_2$ are independent. We may suppose that $B_1$ and $B_2$ are at different sides of $e_{i^*}$ in $G(A)$. Consider the case where $e_1$ (respectively,  $e_\rho$) and $e_{i^*}$  are nonalternating in all network representations of $N_1$ (respectively, $N_2$). Let $G'(A)$ be a basic network representation of $A$ such that $B_1$ and $B_2$ are at the same side of $e_{i^*}$ in $G'(A)$. Using Lemma \ref{lemNetCorelatedutile}, it follows that $e_1$ and $e_\rho$ are not on the basic path from $v_2$ to $w_2$ and from $v_1$ to $w_2$, respectively, in $G(A)$ and $G'(A)$. So
$e_1$ and $e_\rho$ are nonalternating in $G(A)$, and 
$e_1$ and $e_\rho$ are alternating in $G'(A)$.
The other cases  can be analyzed in a same way. 
\end{proof}\\

By Lemma \ref{lemNetCorelated1=i*} and Theorem \ref{thmNetCorelatednoncor}, our initial problem can be reduced to determining whether the bonsais $E_1$ and $E_2$ are independent. Consider the undirected graph $F^*(V)$ defined at page \pageref{mycounter4}. Denote by $\mathcal{W}_1, \ldots, \mathcal{W}_c$ the connected components of $F^*(V)$. We will prove the following proposition.

\begin{prop}\label{propNetCorelatedBs}
The bonsais $E_1$ and $E_2$ are independent if and only if $E_1$ and $E_2$ are in different connected components of $F^*(V)$.
\end{prop}

Let us see some lemmas that will be used for the proof of Proposition \ref{propNetCorelatedBs}.

\begin{lem}\label{lemNetcorelatedhelp}
Let $E_\ell, E_{\ell '} \in \mathcal{W}_{c_1}$ for some $1\le c_1 \le c$, and $E_u \in V \verb"\" \mathcal{W}_{c_1}$ such that $(E_\ell,E_u) \in D$. Then $(E_{\ell '},E_u) \in D$.
\end{lem}

\begin{proof}
Suppose first that $(E_\ell,E_{\ell '})\in F^*(V)$.
Let $j\in f^*(E_\ell) \cap f^*(E_{\ell '})$. Since $(E_\ell,E_u)\in D$, $j\in f^*(E_u)$. So $j \in f^*(E_{\ell '}) \cap f^*(E_u) $, and
as $(E_{\ell '},E_u) \notin F^*(V)$ it follows that $(E_{\ell '},E_u) \in D$ or $(E_u,E_{\ell '}) \in D$. If $(E_u,E_{\ell '}) \in D$, then by transitivity of the relation $\prec_D$ we have $(E_\ell,E_{\ell '}) \in D$, which contradicts $(E_\ell,E_{\ell '}) \in F^*(V)$. Thus $(E_{\ell '},E_u) \in D$.
If $(E_\ell,E_{\ell '}) \notin F^*(V)$, since there is a path in $F^*(V)$ between $E_\ell$ and $E_{\ell '}$, we may conclude as before.
\end{proof}\\

\begin{lem}\label{lemNetCorelatedsim}
Let $E_{\ell }, E_{\ell '} \in \mathcal{W}_{c_1}$, $E_{u}, E_{u'} \in \mathcal{W}_{c_2}$  for some $1\le c_1,c_2 \le c$, $c_1\neq c_2$, such that $(E_{\ell },E_{u}) \in D$ and $(E_{u'},E_{\ell '}) \in D$. Then $|\mathcal{W}_{c_1}|=|\mathcal{W}_{c_2}|=1$ and $E_{\ell } \sim_s E_{u}$.
\end{lem}

\begin{proof}
Since $(E_{\ell },E_{u}) \in D$, by Lemma \ref{lemNetcorelatedhelp} $(E_{\ell '}, E_{u}) \in D$. Similarly, since $(E_{u'}, E_{\ell '}) \in D$, $(E_{u}, E_{\ell '}) \in D$. Thus $E_{\ell '} \sim_s E_{u}$. 
Suppose  by contradiction that $|\mathcal{W}_{c_1}| \geq 2$. (The case $|\mathcal{W}_{c_2}| \geq 2$ is symmetric.)
Let $E_{\ell ''}\in \mathcal{W}_{c_1}$ such that $(E_{\ell '}, E_{\ell ''}) \in F^*(V)$. As $E_{\ell '} \sim_s E_{u}$, 
using Lemma \ref{lemcycle}, it follows that $(E_{u},E_{\ell ''}) \in F^*(V)$, contradicting the fact that $E_u \notin \mathcal{W}_{c_1}$. This completes the proof.
\end{proof}\\

W.l.o.g, we may assume $E_1 \in \mathcal{W}_1$. For any bonsai $E_u\in D$, let $Des_{G(A)}(E_u)$ be the set of all descendants of $E_u$ in the forest $T_{G(A)}$. 
Let $U= \{ÊE_u \in V\verb"\"\mathcal{W}_1 \, : \, (E_1, E_u) \in D \m{ and } Des_{G(A)}(E_u) \cap \mathcal{W}_1 = \emptyset\}$. 

\begin{lem}\label{lemNetCorelatedU}
The set $U$ is a closed set in $T_{G(A)}$.
\end{lem}

\begin{proof}
Using  the transitivity of the relation $\prec_D$, the proof is done.
\end{proof}\\

By Lemma \ref{lemNetCorelatedU}, it follows that there exists a path $\mathcal{P}(U)$ in $T_{G(A)}$ whose vertex set is equal to $\cup_{E_u \in U} (s(A_{\bullet j}) \cap E_u) \cup \{i^*\}$ for any $j\in f^*(E_1)$. Denote by $z_1$ and $z_2$ the endnodes of $\mathcal{P}(U)$ and let $\Gamma= \{ E_\ell \, :\, v_\ell =z_1 \m{ or } z_2 \}$.

\begin{lem}\label{lemNetCorelatedalpha}
For any $E_\gamma \in \Gamma$, the following holds.
\begin{itemize}

\item[(i)] $E_\gamma \in \mathcal{W}_1$, or
\item[(ii)] $(E_\gamma,E_\ell) \in D $ for some $E_\ell \in \mathcal{W}_1$, or 
\item[(iii)] $f^*(E_\gamma) \cap f^*(E_\ell) = \emptyset$ for all $E_\ell \in \mathcal{W}_1$.

\end{itemize}
\end{lem}

\begin{proof}
Let $E_\gamma \in \Gamma$.
Suppose that $E_\gamma \notin \mathcal{W}_1$ and $f^*(E_\gamma) \cap f^*(E_\ell) \neq \emptyset$ for some $E_\ell \in \mathcal{W}_1$.
Since $(E_\gamma,E_\ell) \notin F^*(V)$, it results that $(E_\gamma,E_\ell) \in D$ or $(E_\ell,E_\gamma) \in D$. 
Suppose  by contradiction that $(E_\ell,E_\gamma) \in D$. 
By Lemma \ref{lemNetcorelatedhelp}, $(E_1,E_\gamma)\in D$.
Observe that $Des_{G(A)}(E_\gamma) \subseteq U$ by definition of $\Gamma$. So $Des_{G(A)}(E_\gamma) \cap \mathcal{W}_1=\emptyset$, hence $E_\gamma \in U$, which is impossible since $\Gamma \cap U = \emptyset$. Therefore $(E_\gamma,E_\ell) \in D$. This concludes the proof.
\end{proof}\\

Let $\Gamma_I=\{ E_\gamma \in \Gamma \, : \, E_\gamma \in \mathcal{W}_1 \m{ or } (E_\gamma,E_\ell) \in D \m{ for some } E_\ell\in \mathcal{W}_1\} $ and $\Gamma_{II}=\{ E_\gamma \in \Gamma \, : \, f^*(E_\gamma) \cap f^*(E_\ell) = \emptyset \m{ for all } E_\ell\in \mathcal{W}_1 \} $.

\begin{lem}\label{lemNetCorelatedS1}
For all $E_\gamma \in \Gamma_I$ and $E_u \in U$,  we have $(E_\gamma,E_u) \in D$.
\end{lem}

\begin{proof}
This directly follows from Lemma \ref{lemNetcorelatedhelp}
and the transitivity of the relation $\prec_D$.
\end{proof}\\

\begin{lem}\label{lemNetCorelatedS2}
For all $E_\gamma \in \Gamma_I$ and $E_{\gamma'} \in \Gamma_{II}$, $f^*(E_\gamma) \cap f^*(E_{\gamma'}) = \emptyset$.
\end{lem}

\begin{proof}
Suppose by contradiction that $f^*(E_\gamma) \cap f^*(E_{\gamma'}) \neq \emptyset $ for some $E_\gamma \in \Gamma_I$ and $E_{\gamma'} \in \Gamma_{II}$. 
If $E_\gamma \in \mathcal{W}_1$, then this contradicts the definiton of $\Gamma_{II}$. So $E_\gamma \notin \mathcal{W}_1$ and by definition of $\Gamma_I$, $(E_\gamma,E_\ell) \in D$
for some $E_\ell \in \mathcal{W}_1$. This implies that $f^*(E_\gamma) \subseteq f^*(E_\ell)$ for some $E_l \in \mathcal{W}_1$,
and since $f^*(E_\gamma)\cap f^*(E_{\gamma'})\neq \emptyset$ it results that
$f^*(E_{\gamma'}) \cap f^*(E_\ell) \neq \emptyset $, a contradiction as before.
\end{proof}\\

\noindent
{\bf Proof of Proposition \ref{propNetCorelatedBs}.} \quad  Suppose first that the bonsais $E_1$ and $E_2$ belong to a same connected component of $F^*(V)$.
Let $E_\ell,E_{\ell '}$ be a pair of adjacent bonsais in $F^*(V)$. Since $f^*(E_\ell) \cap f^*(E_{\ell '}) \neq \emptyset$, if $B_\ell$ and $B_{\ell '}$ are at the same side of $e_{i^*}$ in $G(A)$, then $E_\ell$ and $E_{\ell '}$ belong to a same path in $T_{G(A)}$, therefore by transitivity of $\prec_D$ $(E_\ell,E_{\ell '}) \in D$ or $(E_{\ell '},E_\ell) \in D$, contradicting the fact that $(E_\ell,E_{\ell '}) \in F^*(V)$. Thus $B_\ell$ and $B_{\ell '}$ are at different sides of $e_{i^*}$ in $G(A)$. This implies that in any network representation of $A$, the bonsais $B_1$ and $B_2$ are at different sides of $e_{i^*}$ if and only if any path in $F^*(V)$ joining them has an odd length. Thus $E_1$ and $E_2$ are dependent.\\

Now suppose that $E_1$ and $E_2$ are in different connected components of $F^*(V)$. Up to a renumbering of the connected components of $F^*(V)$, we may always assume that $E_1 \in \mathcal{W}_1$ and $E_2 \in \mathcal{W}_2$. Furthermore,
up to a renumbering of the bonsais in $D$, we show that we may assume $E_2 \notin \Gamma_I$ and $des_{G(A)}(E_2) \cap \Gamma_I=\emptyset$. Before proving that, we state the following claim.\\

\noindent
{\bf Claim.} \quad Either $(E_{2},E_\ell) \notin D$ for all $E_\ell \in \mathcal{W}_1$, or $(E_{1},E_u) \notin D$ for all $E_u \in \mathcal{W}_2$, or $E_1 \sim_s E_2$.\\

\noindent
{\bf Proof of Claim.} \quad Suppose that $(E_{2},E_\ell) \in D$ for some $E_\ell \in \mathcal{W}_1$ and
$(E_{1},E_u) \in D$ for some $E_u \in \mathcal{W}_2$.
Then, by Lemma \ref{lemNetCorelatedsim}, $E_1 \sim_s E_2$.
{\hfill$\BBox{\rule{.3mm}{3mm}}$} \\

If $E_1 \sim_s E_2$, then 
up to a renumbering of the bonsais 
we may assume that $E_1$ is not a descendant of $E_2$ in $T_{G(A)}$ ($E_1 \notin Des_{G(A)}(E_2)$); therefore  
$(E_1,E_2) \in D$ and $Des_{G(A)}(E_2) \cap \mathcal{W}_1=\emptyset$ (otherwise by Lemma \ref{lemNetCorelatedsim} (with $l=1$ and $u'=u=2$) we get a contradiction),  and so $E_2 \in U$. Hence  $E_2 \notin \Gamma_I$ and $des_{G(A)}(E_2) \cap \Gamma_I=\emptyset$. Now if $E_1 \nsim_s E_2$, then by interchanging the label of $E_1$ and $E_2$ if necessary and using the Claim above we may assume that $(E_2,E_\ell) \notin D$ for all $E_\ell \in \mathcal{W}_1$, hence $E_2 \notin \Gamma_I$ and $des_{G(A)}(E_2) \cap \Gamma_I=\emptyset$. 
Consider the following procedure. 

\begin{tabbing}
\textbf{Procedure\,\,Intervert-$\Gamma_I$($G(A)$,$E_1$,$E_2$)}\\

\textbf{Input: }\= A basic network representation $G(A)$ and 
two bonsais $E_1$ and $E_2$ being in \\
\>different connected components of $F^*(V)$ and such that $E_2 \notin \Gamma_I$ and \\
\> $des_{G(A)}(E_2) \cap \Gamma_I=\emptyset$.\\
\textbf{Output:} \= A basic network representation $G'(A)$ such that $B_1$ and $B_2$ are at different\\ 
\>  sides of $e_{i^*}$ in $G(A)$ if and only if they are at the same side of $e_{i^*}$ in $G'(A)$. \\ 
 
1) \verb"  "\= for each $E_\ell \in \Gamma_I$, get the bonsai $B_\ell$ loose from $z_1$ or $z_2$ by making a copy $\tilde v_\ell$ of \\
\> $v_\ell=z_1$ or $z_2$;\\
2) \> for each $E_\ell \in \Gamma_I$, if $v_\ell$ was equal to $z_1$ (resp., $z_2$), identify $\tilde v_\ell$ with $z_2$ (resp., $z_1$);\\
\> output the obtained basic network representation $G'(A)$ of $A$;\\
\end{tabbing}

By Lemmas \ref{lemNetCorelatedS1} and \ref{lemNetCorelatedS2}, since $E_1 \in \Gamma_I$ or 
$des_{G(A)}(E_1) \cap \Gamma_I \neq \emptyset$, and 
$E_2 \notin \Gamma_I$ and
$des_{G(A)}(E_2) \cap \Gamma_I=\emptyset$, $B_1$ is moved by the procedure Intervert-$\Gamma_I$, while $B_2$ is not moved. So
the output of the procedure Intervert-$\Gamma_I$ is correct.
{\hfill$\BBox{\rule{.3mm}{3mm}}$}\\

\noindent
{\bf Proof of Theorem \ref{thmNetCorelatedMain}.} \quad
We first construct a basic network representation $G(A)$. This is known to take time $O(n \alpha)$ (see Theorem \ref{thmSubclassNTutteCunNet}).
Since $A$ is a network matrix, computing the digraph $D$ with respect to some $R^*$ ($|R^*|=1$) takes time $O(n\alpha)$. So
the following routines take time at most $C_0(n\alpha)$ for some constant $C_0$:

\begin{tabbing}
i) \verb"  "\= {\bf if } each column of $A$ has at most two nonzeros, {\bf then } \\
\> compute $L(A)$ and test if $L(A)$ has a cutvertex separating $1$ from $\rho$; if such a vertex \\
\> say $i^*$ exists, then call {\tt Move-T2}($G(A)$,$i^*$); otherwise output that $A$ is $\{1, \rho\}$-corelated;\\
\> {\bf endif } \\ 
(ii) \> {\bf if } $A$ has at least one column with three nonzeros, {\bf then } \\ 
\> choose a row index $i^*$ such that  $A_{\overline{\{i^*\}}\times \overline{f(\{i^*\})}}$ is not connected and compute the digraph\\
\> $D$ with respect to $R^*=\{i^*\}$; \\
\> test if the assumption $\mathscr{B}$ is satisfied; if it is, then test if $E_1$ and $E_2$ are in different\\
\> connected components of $F^*(V)$;  if they are, then relabel the bonsais or connected\\
\> components of $F^*(V)$ so that  $E_2 \notin \Gamma_I$  and $des_{G(A)}(E_2) \cap \Gamma_I=\emptyset$,\\
\> call {\tt  Intervert-$\Gamma_I$}($G(A)$,$E_1$,$E_2$);\\
\> {\bf endif } 
\end{tabbing}

By Theorems \ref{thmNetCorelatedtwo} and \ref{thmNetCorelatednoncor}, Lemma \ref{lemNetCorelated1=i*} and 
Proposition \ref{propNetCorelatedBs} (see also the proof of Proposition \ref{propNetCorelatedBs}), if $A$ stands each of these tests, then $A$ is $\{1,\rho\}$-noncorelated if and only if the matrix $N_\ell$ is $\{1,\rho\}$-noncorelated for some $1\le \ell \le b$ (in case where the assumption $\mathscr{B}$ is not satisfied), or $N_1$ or $N_2$ is $\{1,i^*\}$-noncorelated or $\{\rho,i^*\}$-noncorelated, respectively. This is a recursive definition of a test for being $\{1,\rho\}$-noncorelated.
It results an algorithm which determines that $A$ is  $\{1,\rho\}$-corelated, or provides a basic network representation $G'(A)$ such that $e_1$ and $e_\rho$ are alternating in $G(A)$ if and only if they are nonalternating in $G'(A)$. We show that its running time is bounded by $Cn^2 \alpha$ for some constant $C$, by induction on $n$. (The proof is inspired from \cite{ShrAlex}.)

If the algorithm stops during the step i) or ii), then the time is bounded by $C_0 n \alpha$ for some constant $C_0$.
Otherwise, in the worst case, we need to test the two bonsai matrices $N_1$ and $N_2$ for being $\{1,i^*\}$-noncorelated and $\{\rho,i^*\}$-noncorelated, respectively. This requires time at most $C n_1^2 \alpha_1+C n_2^2 \alpha_2 $ for some constant $C \geq C_0$, by the induction hypothesis, where $r_\ell$ (resp., $\alpha_\ell$) is the number of rows (resp., the number of nonzeros) of $N_\ell$, for $l=1$ and $2$. Moreover, steps i) and ii) require time at most $C_0 n \alpha$. Altogether, since $n_1 +n_2 =n+1$, $2\le n_1,n_2$ and $\alpha_1, \alpha_2\le  \alpha$, the time is bounded by:

$$ C_0 n \alpha + C n_1^2 \alpha_1 + C n_2^2 \alpha_2  \le C n^2\alpha,$$

\noindent
for some constant $C\geq C_0$ and $n\geq 5$. This terminates the proof.
{\hfill$\BBox{\rule{.3mm}{3mm}}$}\\


\clearpage
\thispagestyle{empty}
\cleardoublepage
\verb"   "
\newpage

\chapter{Conclusion}\label{ch:conclude}
   
In this last chapter we point out and discuss the main results and suggest some directions for further research.  

The main goal of this work was to recognize binet matrices.
To achieve this, given a matrix $A$,
we developed a subroutine which computes a Camion basis of the matrix $[I\, A]$ in time bounded by a polynomial function of the size of $A$, or outputs that $A$ is not binet. Then, we reduced the main problem to the recognition of nonnegative $\frac{1}{2}$-binet matrices, nonnegative bicyclic matrices and nonnegative cyclic matrices. Some characterizations of these subclasses of matrices were provided.

Furthermore, we presented a new characterization of
the Camion bases of the node-edge incidence matrix of any connected digraph, and a characterization of Camion bases of any given matrix with a polynomial-time recognition 
procedure. An algorithm which finds a
Camion basis was also described.

Finally, for recognizing cyclic matrices, we solved as a subproblem the recognition of nonnegative
$\{\epsilon,\rho\}$-noncorelated network matrices, where $\epsilon$ and $\rho$ are two given row indexes. We provided a nice characterization theorem for this class of matrices.

One can wonder whether the described method for recognizing binet matrices could be improved or whether there exists a  simpler and more elegant method. Is it really necessary to find a Camion basis first for recognizing binet matrices? We may try to answer these questions.

The way of computing the digraph $D$ in Chapter \ref{ch:multidiD} should be improvable in terms of time efficiency. This would yield a faster algorithm for recognizing nonnegative $R^*$-cyclic matrices, nonnegative bicyclic matrices and $\frac{1}{2}$-binet matrices. However, this does not modify the global complexity of the algorithm Binet. Moreover, the described method for recognizing nonnegative $\frac{1}{2}$-binet matrices does not seem efficient since one computes several network representations of different 
submatrices. There should exist another faster method.

Two other directions for recognizing binet matrices have been proposed by Kotnyek \cite{KotThesis}. The first one is to formulate the problem as a mixed integer programming problem, and then to solve it. This method does not lead to a polynomial-time algorithm so far. A second interesting approach is to convert first a given rational matrix of size $n\times m$ to an integral one by a finite number of steps; Kotnyek and Appa proved that even in the worst case we need only at most $2m$  pivoting operations to get an integral matrix, provided that the matrix we started with was binet. Then the goal is to find necessary and sufficient conditions for an integral matrix for being binet. The following theorem gives a necessary condition.

\begin{thm}(Kotnyek \cite{KotThesis})
If $A$ is an integral binet matrix, then there exist network matrices $N_1$ and $N_2$ such that (a) $A=N_1 + N_2$, and (b) both $[N_1\, N_2]$ and $\left[ \begin{array}{c}
N_1\\
N_2
\end{array} \right]$ are network matrices.
\end{thm}

As a corollary, another necessary condition can be derived, similar to the well-known characterization of totally unimodular matrices due to Ghouila-Houri \cite{Ghouila-HouriTotMod-62}. This latter claims that for each collection of columns of a totally unimodular matrix, there exists a scaling of the selected columns by $\pm 1$ such that the sum of the scaled columns is a vector of $0,\pm 1$ elements, and the same is true for rows.

\begin{thm}(Kotnyek \cite{KotThesis})
For each collection of columns or rows of an integral binet matrix, there exists a scaling of the selected columns or rows by $\pm 1$ such that the sum of the scaled columns or rows is a vector of $0$, $\pm 1$, $\pm 2$ elements.
\end{thm}

Unfortunately, these conditions are not sufficient, as matrix $A=\left[ \begin{array}{cc}
2 & 1\\
-1 & 1
\end{array} \right]$ shows. The advantage of transforming first a given matrix to an integral one follows from Lemma \ref{lemBinetintegral}: An integral matrix is binet if and only if it is cyclic or $\frac{1}{2}$-binet. From my point of view,
the main difficulty in this approach consists in recognizing cyclic $\{0,\pm 1\}$-matrices. The natural method we use for recognizing $\{\epsilon,\rho\}$-central $\{0,1\}$-matrices can not be directly adapted to $\{\epsilon, \rho\}$-central $\{0,\pm 1\}$-matrices for at least one reason:
Given an $\{\epsilon, \rho\}$-central $\{0,\pm 1\}$-matrix $A$ and an $\{\epsilon, \rho\}$-central representation $G(A)$ of $A$, it happens that a fundamental circuit contains $e_\epsilon$ and $e_\rho$, but not the whole basic cycle in $G(A)$. On the contrary, in any $\{\epsilon, \rho\}$-central representation $G(A)$ of a binet $\{0,1\}$-matrix $A$, if a fundamental circuit contains $e_\epsilon$ and $e_\rho$, then it contains the whole basic cycle in $G(A)$ (see Lemma \ref{lemdigraph12}). This small difference could be very hard to overcome.

A very nice and deep characterization of totally unimodular matrices over $GF(2)$ has been given by Tutte \cite{TutteMat-58} and proved by Gerards \cite{GerardsRegMat-89} in a simple way. So, inspired by Theorem \ref{thmMatroidConf}, we hope that there exists a simple characterization of integral binet matrices over $GF(3)$. Working on $GF(3)$ instead of $\mathbb{R}$, one can try to find a new algorithm for recognizing integral binet matrices. 

An other possibility for further research is connected to matroids. How can we translate our results about binet matrices to statements about signed-graphic matroids? In particular, is it possible to determine whether a matroid is signed-graphic using our recognition procedure for binet matrices?

We would like to mention a last open problem: the recognition of $2$-regular matrices. We believe that a combinatorial recognition algorithm should exist. Probably it would be similar to the recognition algorithm for totally unimodular matrices, originating from the decomposition theory of regular matroids, due to Seymour \cite{SeymourRegular-80}. See also \cite{SlilatyDecomp-06}, for $2$-sum and $3$-sum of signed-graphic matroids.

\nocite{*}
\bibliographystyle{plain}  
\bibliography{ThesisImpression}

\printindex


\clearpage
\thispagestyle{empty}
\cleardoublepage
\begin{center}
{\bfseries {\scshape Musitelli Antoine}}
\end{center}

\begin{tabular}{l}
Swiss nationality\\
\end{tabular}
\hfill
\begin{tabular}{r}
antoine.musitelli@epfl.ch
\end{tabular}

\vspace{2ex}

\begin{tabular}{p{12.5cm}}
{\bf Education}\quad \hrulefill \\
\end{tabular}

\vspace{1ex}

\begin{tabular}{p{0.1cm}l}
 & {\bf Doctorate of Philosophy in Mathematics} from 2004 to 2007 \\
 & \'Ecole Polytechnique F\'ed\'erale de Lausanne (EPFL), Switzerland \\
 & \\
 & {\bf Research and Teaching Assistant} from 2004 to 2007 \\
 & Operations Research chair, Prof. Liebling, EPFL, Switzerland \\
 & \\
 & {\bf Research and Teaching Assistant} from 2003 to 2004 \\
 & University of Geneva, Switzerland \\
 & \\
 & {\bf Master in Mathematics} 2003\\
 & University of Geneva, Switzerland \\
 & \\
 & {\bf Baccalaureate (Sciences)} 1999\\
 &  Coll\`ege Calvin, Geneva, Switzerland \\
\end{tabular}

\vspace{2ex}

\begin{tabular}{p{12.5cm}}
{\bf Research}\quad \hrulefill \\
\end{tabular}

\vspace{1ex}

\begin{tabular}{p{0.1cm}l}
&\begin{tabular}{cl}
 $\bullet$ & graph theory\\
 $\bullet$ & Combinatorial geometry\\
 $\bullet$ & scheduling\\
\end{tabular}
\end{tabular}

\vspace{2ex}

\begin{tabular}{p{12.5cm}}
{\bf Teaching Assistant}\quad \hrulefill \\
\end{tabular}

\vspace{1ex}

\begin{tabular}{p{0.1cm}l}
&\begin{tabular}{cl}
 $\bullet$ & Operations Research (2004 - 2006)\\
 $\bullet$ & Game Theory (2006)\\
 $\bullet$ & Combinatorics (2005)\\
 $\bullet$ & Linear Algebra (2004 - 2005)\\
 $\bullet$ & Numerical Analysis (2003-2004)\\
 $\bullet$ & Geometry (2003 - 2004)\\
 $\bullet$ & Supervision of semester and master projects\\
\end{tabular}
\end{tabular}


\begin{tabular}{p{12.5cm}}
{\bf Publications}\quad \hrulefill \\
\end{tabular}

\vspace{1ex}

\begin{tabular}{p{0.1cm}p{12cm}}
& P. de la Harpe \& A. Musitelli.Expanding graphs, {R}amanujan graphs, and 1-factor perturbations. In {\em Bulletin of the Belgian Mathematical Society. Simon Stevin}, p. 673--680, 2006.\\
&\\
& K. Fukuda \& A. Musitelli. New polynomial-time algorithms for {C}amion bases. {\em Discrete Mathematics}, vol. 306 (2006), p. 3302--3306 
\end{tabular}

\vspace{2ex}

\begin{tabular}{p{12.5cm}}
{\bf Languages}\quad \hrulefill \\
\end{tabular}

\vspace{1ex}

\begin{tabular}{p{0.1cm}l}
&\begin{tabular}{cl}
 $\bullet$ & French: native language\\
 $\bullet$ & English: fluent\\
 $\bullet$ & German \& Italian: working knowledge\\
\end{tabular}
\end{tabular}

\end{document}


\maketitle

If $A$ is $R$-cellular, the following algorithm constructs a spanning forest 
$T_F$ of $D$ such that  for all $\beta \in S$, the subgraph of $T_F$ induced 
by $\{\Gamma_l\in V\,\,:\,\, \beta \in f(R^l)\}$, denoted by 
$T_F(\{\Gamma_l\in V\,\,:\,\, \beta \in f(R^l)\})$, is a $\beta$-fork, a 
$\beta$-path or a union of two disjoint $\beta$-paths. Moreover, if $f_\beta$ is
a non basic half-edge in some $R$-cellular representation of $A$, then 
$T_F(\{\Gamma_l\in V\,\,:\,\, \beta \in f(R^l)\})$ corresponds to a $\beta$-path.

If $A$ has a $R$-cellular representation $G(A)$, the symbol $T^{*}$ denotes
the spanning forest of $D$ defined in the previous section.
We will assume that for each $\beta\in S$
and $1\le l\le r'$, $s(A_{\bullet \beta}) \cap R^l$ has been  
partitioned into at most two subsets $\tilde R_{11}(\beta,l)$ and $\tilde 
R_{12}(\beta,l)$ and $D$ has no closed directed path.

The construction of a forest $T_F$ can be reduced to a sequence of 2-SAT problems.
Assume that $T_{\bar F}=(U=\cup _{j=1}^{k-1}V_j,\bar F)$ has been 
constructed and $V_k\subseteq V- U$.
In the forest $T^{*}$, for all $\beta \in S$, 
$T^{*}(\{\Gamma_l\in V\,\,:\,\, \beta \in f(R^l)\})$ is a $\beta$-fork, a 
$\beta$-path or a union of two disjoint $\beta$-paths. So 
for all  $\Gamma_{u}\in V$, if the tree $ \Gamma_{u}\subseteq G(A)$ has a $u$-substem
with edge index set equal to some $R_h^u$, then

\begin{equation}\label{eq:tilde_g_T*}
 \begin{displaystyle}
\sum_{\substack{\Gamma_{l}\,:\, (\Gamma_{l},\Gamma_{u})_{R_h^{u}}\in 
T^{*} }}  
\tilde g(\Gamma_{l})\le g(\Gamma_{u}).
\end{displaystyle}
\end{equation}

\noindent
The following definition is a quite natural consequence of this last property.
An arc $(\Gamma_{l},\Gamma_{u})_{R_h^{u}}$ (with $\Gamma_{l} \in V_k$ and 
$\Gamma_{u}\in U$) is said to be \emph{legal}
if and only if 
$ \tilde g(\Gamma_{l}) +
\begin{displaystyle}
\sum_{\substack{\Gamma_{u'}\,:\, (\Gamma_{u'},\Gamma_{u})_{R_h^{u}}\in 
\bar F }}  
\tilde g(\Gamma_{u'})\le g(\Gamma_{u}).
\end{displaystyle}$

We construct an instance $I_k$ of 2-SAT.
First we define the variables. 
For $\Gamma_{l}\in V_k$, if there exists a legal arc $(\Gamma_{l},
\Gamma_{u})\in E$
labeled $R_h^{u}$, we define the variable 
$x^{(l,u)}_h$. The variable $x_h^{(l,u)}$ will be true if and only 
if we put $(\Gamma_{l},\Gamma_{u})$
labeled $R_h^{u}$ into $\bar F$. 
Moreover, 
if for all $\beta \in f(R^l)$, 
$T_{\bar F}(\{\Gamma_{u}\in U\,\,:\,\,\beta \in f(R^{u})\})$ is 
empty, we define the variable $x^l$. And if for all $\beta \in f(R^l)$, 
$T_{\bar F}(\{\Gamma_{u}\in U\,\,:\,\,\beta \in f(R^{u})\})$
is a $\beta$-path and $s_{\frac{1}{2}}(A_{\bullet \beta})=\emptyset$, 
we also define $x^l$. The variable $x^l$ has a true value if and only if 
we do not put any arc
leaving $\Gamma_l$ into $\bar F$ meaning that $\Gamma_l$ is a top vertex in 
$T_F$.

Now let see the clauses. For $\Gamma_{l},\Gamma_{l'}\in V_k$ ($l\neq l'$), if there 
exist $\Gamma_{u}\in U$, 
$\beta\in f(R^l)\cap f(R^{l'})$ 
such that $(\Gamma_{l},\Gamma_{u})_{R_h^{u}},
(\Gamma_{l'},\Gamma_{u})_{R_h^{u}} \in E$ are legal and 
$\tilde g_\beta (\Gamma_{l})+\tilde g_\beta (\Gamma_{l'})+
\begin{displaystyle}
\sum_{\substack{\Gamma_{u'}\,:\, (\Gamma_{u'},\Gamma_{u})_{R_h^{u}}
\in \bar F }}  
\tilde g_\beta (\Gamma_{u'})
\end{displaystyle} > g_\beta (\Gamma_{u})$, add into $I_k$ the clause 
$x_h^{(l,u)}\vee x_h^{(l',u)}$. Thus if this clause is true, at most one
of the arcs $(\Gamma_{l},\Gamma_{u})_{R_h^{u}}$ and
$(\Gamma_{l'},\Gamma_{u})_{R_h^{u}} $ will be in $T_F$.
And  if $x^l$ and $x^{l'}$
are well defined and
$T_{\bar F}(\{\Gamma_{u}\in \bar F\,\,:\,\,\beta \in f(R^{u})\})$ is a
$\beta$-path for some $\beta \in f(R^l)\cap f(R^{l'})$, we include in $I_k$ the clauses
$x^l\vee x^{l'}$ and $\bar {x^l}\vee \bar {x^{l'}}$. So in a truth assignment,
both clauses are satisfied if and only if exactly one of the variables $x^l$ or
$x^{l'}$ is true.
A variable of the type 
$x^l$ or $x_h^{(l,u)}$ is said to be \emph{associated} to the vertex 
$\Gamma_l\in D$.

If $A$ has a $R$-cellular representation, we will prove that there are at
most two variables associated to each vertex in $V_k$.
If exactly one variable, say $y$, is associated to a vertex $\Gamma_l$, let $y=1$ 
and if there are exactly two variables, say $y_1$ and $y_2$, then we 
include in $I_k$ the clauses $(y_1\vee y_2)$ and  $(\bar y_1\vee\bar  y_2)$. 
Note as before that if in a truth assignment we have $y_1=y_2$, then only one of these 
two clauses is satisfied, whereas if $y_1\neq y_2$ then both clauses are 
satisfied. Thus for all $\Gamma_l\in D$, exactly one variable associated to
$\Gamma_l$ has a true value. \\

\begin{tabbing}
\textbf{Procedure\,\,FOREST}\\

\textbf{Input:} A digraph $D=(V,E)$.\\
\textbf{Output:} A spanning forest $T_F=(V= U,F= \bar F)$.\\

1)\verb"  "\= Let $\bar F=\emptyset$; $U=\emptyset$; 
              $T_{\bar F}=(U,\bar F)$; $\bar D=D$ and $k=1$.\\
2) \> {\bf while }\= $\bar  D\neq \emptyset$ {\bf do}\\
3) \>            \> Build $V_k$ the set of top vertices of $\bar D$.\\ 
4) \>            \> Find a truth assignment of $I_k$ such that all clauses are 
                    satisfied if it exists,\\  
    \>           \> otherwise STOP.
                    For each true variable related to an arc,  
                    add in $\bar F$ the \\
    \>		 \> corresponding arc.\\    
5) \>            \> $\bar D=\bar D -V_k$; $U=U \cup V_k$ and 
                    $k=k+1$.\\
\> {\bf endwhile }\\
\end{tabbing}

\begin{thm}
Suppose $A$ has a $R$-cellular representation G(A). 
Then the algorithm FOREST does not stop at step 4) and for all $\beta \in S$,
the subgraph $T_{F}(\{\Gamma_l\in V \,\,:\,\, \beta \in f(R^l)\})$ is a 
$\beta$-fork, a $\beta$-path or a union of two disjoint $\beta$-paths. 
Moreover, if $f_\beta$ is
a non basic half-edge in $G(A)$, then 
$T_F(\{\Gamma_l\in V\,\,:\,\, \beta \in f(R^l)\})$ corresponds to a $\beta$-path.
\end{thm}

\begin{proof} 
Let $G(A)$ a $R$-cellular representation of $A$ and $T^{*}$ as defined in the
previous section. For $\beta\in S$ and $U\subseteq V$, denote $U(\beta)=
\{\Gamma_u\in U \,\,:\,\, \beta \in f(R^u)\}$. This is the set of vertices 
$\Gamma_u$ in $U$ such that $f_\beta$ generates a $u$-substem in the tree
$\Gamma_u$ of $G(A)$.
Let  $n$ the number of passages in loop 2 and 
for all $0\le k\le n$, let $F_k:=\bar F$ as in the
algorithm FOREST after $k$ passages in loop 2 and 
$T_k=( \cup _{j=1}^{k}V_j,F_k)$.

We will prove the following assertion by induction on $k$:\\

\noindent
{\bf Affirmation:} For all $\beta \in S$,
the graph $T_k(\{\Gamma_l\in \cup _{j=1}^{k}V_j \,\,:\,\, \beta \in f(R^l)\})$ is a 
$\beta$-fork, a $\beta$-path or a union of two disjoint $\beta$-paths and 
the algorithm FOREST does not stop at step 4). Moreover,
if $f_\beta$ is a non basic half-edge in $G(A)$, then 
$T_k(\{\Gamma_l\in V\,\,:\,\, \beta \in f(R^l)\})$ corresponds to a 
$\beta$-path.\\

\noindent
For $k=0$, the affirmation is clearly true.\\
Let $k\geq 1$ and suppose that the affirmation  is true for $k-1$.\\
First prove that there are at most two variables associated to each vertex in 
$V_k$. Let $\Gamma_v \in V_k$. Clearly, the validity of the affirmation for
$k-1$ implies that there are at most two legal arcs leaving $\Gamma_v$. 
Moreover, if there are 
exactly two legal arcs leaving $\Gamma_v$, then for all $\beta \in f(R^v)$,
$T_{k-1}(\{\Gamma_{u} \in U\,:\, \beta \in f(R^{u})\})$ 
is a disjoint union of two $\beta $-paths or a $\beta$-fork. 
So $x^v$ is not defined.\\
Let $U=\cup _{j=1}^{k-1}V_j$, $T_{k-1}^*=T^*(U)$ and
$T_k^*=T^*(U\cup V_k)$. We will construct a set $E_k$ of 
edges such that $E_k$ induces a 
truth assignment of $I_k$ which satisfies all clauses
with a true value. 
Let $\Gamma_{l_0} \in V_k$. \\

If there is an arc 
$(\Gamma_{l_0},\Gamma_{u_0})$ in $T_k^*$ labeled $R_h^{u_0}$, 
we distinguish two cases:

\begin{description}
\item[case a1)] Each arc $(\Gamma_{l_0},\Gamma_{u_0})$ is not legal.  
\item[case a2)] Otherwise.
\end{description}

Now, if $\Gamma_{l_0}$ has no successor in $T_k^*$, we still have two more 
cases:

\begin{tabbing}
{\bf case b1)}\,\,\,\= $T_{k-1}(U(\beta))$ is a union of
                two disjoint $\beta$-paths having the root vertices   
		$\Gamma_{u_1}$ and $\Gamma_{u_1'}$\\ 
	      \> respectively, for some $\beta \in f(R^{l_0})$.\\
{\bf case b2)}\> Otherwise.
\end{tabbing}
\vspace{0.4cm}

Let make some remarks. Thanks to step 3) of the algorithm FOREST, we have:

\begin{equation}\label{remarcorder1}
\m{ For each arc }(\Gamma_l,\Gamma_{l'})\in D,\m{ if } \Gamma_l\in
V_j \m{ and }\Gamma_{l'}\in V_{j'} \m{ for some }j \m{ and } j', \m{ then } j>j'.
\end{equation}

\noindent
And for $\beta \in S$, if the set $\{\Gamma_u\in U\,:\, \tilde 
g_\beta(\Gamma_{u})=2\}$ is not empty, then $T_{k-1}(\{\Gamma_{u} \in U\,\,:\,\,
 \tilde g_\beta(\Gamma_{u})=2 \})$ is a simple $\beta$-fork by the hypothesis of 
induction; using  remark \ref{remarcorder1}, we deduce:

\begin{equation}\label{remfork}
T_{k-1}^*(\{\Gamma_{l} \in U\,\,:\,\, \tilde g_\beta(\Gamma_{l})=2\})
\approx T_{k-1}(\{\Gamma_{l} \in U\,\,:\,\, \tilde g_\beta(\Gamma_{l})=2 \}).
\end{equation}

\noindent
In addition, we have:

\begin{equation}\label{remarcorder2}
\m{ For }\Gamma_{l},\Gamma_{l'}\in V,\, \beta \in S,\m{  if }\tilde g_\beta 
(\Gamma_{l})=1\m{  and }\tilde g_\beta(\Gamma_{l'})=2, 
\m{ then } (\Gamma_{l},\Gamma_{l'})\in D.
\end{equation}

\noindent
The remark \ref{remarcorder2} follows from the existence of 
a path in $T^*$ from $\Gamma_{l}$ to $\Gamma_{l'}$ which is 
a simple $\beta$-fork and from the transitivity of the relation
"$\Gamma_l \sim \Gamma_{l'} \Leftrightarrow (\Gamma_l, \Gamma_{l'})\in E $". \\

Consider the case a1). Since $(\Gamma_{l_0},\Gamma_{u_0})$ labeled 
$R_h^{u_0}$ is not legal, 
it follows:
 
\begin{equation}\label{illegalbeta}
\m{There exists }
\delta \in f(R^l) \m{ such that } \tilde g_\delta (\Gamma_{l_0})+
\begin{displaystyle}
\sum_{\substack{\Gamma_{u}\,:\, (\Gamma_{u},\Gamma_{u_0})_{R_h^{u_0}}\in 
F_{k-1} }}  
\tilde g_\delta (\Gamma_{u})
\end{displaystyle} > g_\delta(\Gamma_{u_0}).\end{equation}

Assume $\tilde g_\delta (\Gamma_{l_0})=2$. From the remarks \ref{remarcorder1} and 
\ref{remarcorder2}, we deduce that 
$U( \delta) =\{\Gamma_u \in U\,:\, \delta \in f(R^u)\}=
\{\Gamma_{l} \in U\,\,:\,\,
\tilde g_\delta(\Gamma_{l})=2\}$. From remark \ref{remfork}, it follows that 
$T^*_{k-1}(U( \delta))
\approx T_{k-1}(U( \delta))$. 
And since $(\Gamma_{l_0},\Gamma_{u_0})_{R_h^{u_0}}\in 
T_{k}^{*} $, the vertex $\Gamma_{u_0}$ is the root  of the $\delta$-path  
$T^*_{k-1}(U( \delta))$. So $\Gamma_{u_0}$ has no predecessor in 
$T_{k-1}(U( \delta))$.
Since $g_\delta(\Gamma_{u_0})=2$
(by definition $(\Gamma_{l_0},\Gamma_{u_0})\in E$ $\Rightarrow$ 
$\tilde g_\delta(\Gamma_{l_0})\le 
g_\delta(\Gamma_{u_0})$),
the affirmation \ref{illegalbeta} can not be true.
Thus $\tilde g_\delta(\Gamma_{l_0})=1$.\\

Assume $\tilde g_\delta (\Gamma_{u_0})=2$. See figure \ref{fig:qdelta=2}. The graph 
$T_{k-1}^*(U( \delta))$ is a simple $\delta$-fork and $g_\delta(
\Gamma_{u_0})\geq 1$.
Using remark \ref{illegalbeta}, there exists $\Gamma_{u_1}\in U$ such that 
$(\Gamma_{u_1},\Gamma_{u_0})\in F_{k-1}$ and $\delta \in 
f(R^{u_1})$. If $\Gamma_{u_0}$
has no predecessor in $T_{k-1}^*(U(\delta))$, then since 
$T_{k-1}^*(U(\delta))$ is a simple $\delta$-fork, we deduce that
 $(\Gamma_{u_0},\Gamma_{u_1})\in E$ (and 
$(\Gamma_{u_1},\Gamma_{u_0})\in E$), which contradicts the assumption that there is no
cycle in $D$.
Let $\Gamma_{v_\delta}$ the predecessor of $\Gamma_{u_0}$ in $T_{k-1}^*(U(\delta))$ 
and $\Gamma_{u_0'}$ the root of the path $T_{k-1}^*(U( \delta))$. 
Remark that $(\Gamma_{u_0'},\Gamma_{u})\in E$ for all 
$\Gamma_{u}\in U(\delta)$ ($u\neq u_0'$).
Since the set of neighbours of $\Gamma_{l_0}$ in $D$ belongs to $U( \delta)$,
we deduce that $(\Gamma_{l_0},\Gamma_{u_0'})\in E$ and $\Gamma_{u_0'} \in V_{k-1}$ (otherwise
$\Gamma_{l_0}$ would belong to some $V_j$ with $j\le k-1$). 
The vertex $\Gamma_{u_0}$ has two
predecessors in $T^{*}$ ($\Gamma_{l_0}$ and $\Gamma_{v_\delta}$). So $T^{*}
(U(\delta )\cup
\{\Gamma_{l_0}\})$ is a non simple $ \delta$-fork.
For all $\gamma\in f(R^{l_0})$,
since $(\Gamma_{l_0},\Gamma_{u_0'})\in E$, we have that $\gamma \in f(R^{u_0'})$. 
Thus for all $\gamma \in f(R^{l_0})$, the graph
$T^{*}(U( \delta)\cup \{\Gamma_{l_0} \})$ is a 
$\gamma$-fork. 
Now, since the $\gamma$-fork $T^{*}(U( \delta)\cup \{\Gamma_{l_0} \})$ is a closed 
set in $T^*$, if $ U( \gamma)-U( \delta)\neq \emptyset$ for some 
$\gamma \in f(R^{l_0})$, then there exists 
$\Gamma_{u}\in U( \gamma)-U( \delta)$ such that 
$(\Gamma_{u},\Gamma_{l_0})$ or $(\Gamma_{u}, \Gamma_{u_0'})\in D$, which contradicts 
remark \ref{remarcorder1} (since $\Gamma_{l_0}\in V_k$ and 
$\Gamma_{u_0'}\in V_{k-1}$). Thus we have:

\begin{equation}\label{remcase1gam}
\m{ For all }  \gamma\in f(R^{l_0}),
U(\gamma) =U(\delta ) \m{ and } \tilde g_\gamma (\Gamma_{l_0})=1.
\end{equation}

Suppose $T_{k-1}(U(\delta))$ is a simple $\delta$-fork. In case
$g_\delta(\Gamma_{u_0})=1$, by remark \ref{eq:tilde_g_T*}, 
the arc $(\Gamma_{v_\delta},\Gamma_{u_0}) \in T^{*}_k$ has a label
$R_{h'}^{u_0}\neq R_{h}^{u_0}$ and by transitivity of the relation 
"$\Gamma_{l}\sim\Gamma_{ l'} \Leftrightarrow
(\Gamma_{l},\Gamma_{l'})\in D $", there exists an arc $(\Gamma_{l_0},\Gamma_{u_0})$ labeled 
$R_{h'}^{u_0}$. Then using
remark \ref{remcase1gam}, we can find a legal arc $(\Gamma_{l_0},\Gamma_{u_0})$ with some
label, contradicting the assumption of case a1).
In case $g_\delta(\Gamma_{u_0})=2$, we may prove that 
$(\Gamma_{l_0},\Gamma_{u_0})$ labeled $R_{h}^{u_0}$ is legal,
which is a contradiction for the same reason. So
$T_{k-1}(U(\delta))$ has two leaves $\Gamma_{u_0'}$ and an other say 
$\Gamma_{w}$. By remark \ref{remcase1gam}, $(\Gamma_{l_0},\Gamma_{w})$ is 
a legal arc with exactly one possible label since $\tilde g_\delta (\Gamma_w)=1$.\\

Now assume $\tilde g_\delta(\Gamma_{u_0})=1$. See figure \ref{fig:qdelta=1}.
Denote by $\Gamma_{v_0'}$ the root of the greatest $\delta$-path or simple
$\delta$-fork in $T_{k-1}$ containing $\Gamma_{u_0}$. By remark \ref{illegalbeta},
we deduce that $\Gamma_{v_0'}\neq \Gamma_{u_0}$. Since there is no cycle in $D$, the
path from $\Gamma_{u_0}$ to a top vertex in $T_{k-1}^*$ does not contain 
$\Gamma_{v_0'}$. So $T_{k-1}^*(U(\delta))$ is a non simple $\delta$-forks 
or a disjoint union of two $\delta$-paths whose leaves are
$\Gamma_{u_0}$ and say $\Gamma_{u_0'}$.
Since $\Gamma_{u_0}$ has a predecessor in $T_{k-1}(U(\delta))$, 
$\Gamma_{u_0}\notin V_{k-1}$.
And as in the case "$\tilde g_\delta(\Gamma_{u_0})=2$", we deduce that $(\Gamma_{l_0},\Gamma_{u_0'})\in E$,
$\Gamma_{u_0'}\in V_{k-1}$ and the remark \ref{remcase1gam} is also true.
Let us remark that by remark \ref{remcase1gam}, the arc 
$(\Gamma_{l_0},\Gamma_{v_0'})$ is legal. If $\Gamma_{v_0'}=\Gamma_{u_0'}$, 
we may distinguish three cases:

\begin{tabbing}
{ \bf case $\alpha 1)$ :} \= $T_{k-1}(U(\delta))$ has two leaves $
\Gamma_{v_0'}=\Gamma_{u_0'}$ and say $\Gamma_{\bar v_1}$. 
The remark \ref{remcase1gam} implies that \\
\> $(\Gamma_{l_0},\Gamma_{\bar v_1})$  is a legal arc. \\
{ \bf case $\alpha 2)$ :} \> $T_{k-1}(U(\delta))$ is a simple $\delta$-fork. 
Let us call $\Gamma_{u_c}$ the  central node of $T_{k-1}(U(\delta))$. \\ 
\> Since $(\Gamma_{l_0},\Gamma_{u_0})$ and $(\Gamma_{l_0},\Gamma_{u_0'})
\in E$, by remark \ref{remcase1gam} and the transitivity of the \\
\>relation $\Gamma_{l}\sim \Gamma_{l'} \Leftrightarrow (\Gamma_{l}, \Gamma_{l'})\in D$, 
$(\Gamma_{l_0},\Gamma_{u_c})$ is a legal arc for some  label.\\
{ \bf  case $\alpha 3)$ :} $T_{k-1}(U(\delta))$ is a $\delta$-path. \\
\end{tabbing}

In case b1), let $\delta\in f(R^{l_0})$ such that $T_{k-1}(U(\delta))$ 
is a union of two disjoint $\delta$-paths with the roots  $\Gamma_{u_1}$ and 
$\Gamma_{u_1'}$ respectively.
We have that $T_{k-1}^*(U(\delta) )$ is a $\delta$-path.
We may suppose that $\Gamma_{u_1'}$ is the root of the path 
$T_{k-1}^*(U(\delta) )$.
As in the case a1), we may prove that $(\Gamma_{l_0},\Gamma_{u_1'})\in E$,
$\Gamma_{u_1'}\in V_{k-1}$ and the remark \ref{remcase1gam} is also true.
So $(\Gamma_{l_0},\Gamma_{u_1})$ is a legal arc
with exactly one possible label ($\tilde g_\delta (\Gamma_{u_1})=1$).

Consider the following procedure.

\begin{tabbing}
\textbf{Procedure\,\,TruthAssignment}\\

\textbf{Input:} The digraph $D=(V,E)$, $T_{k-1}$, $T_k^*$, 
$U$ and $V_k$.\\
\textbf{Output:} $E_k$.\\ 
 {\bf for }\= each  $\Gamma_{l_0} \in V_k$, {\bf do}\\
           \> We distinguish the following cases.\\ 
 \>            {\bf case a1):}\,\,\,\= 
 {\bf if } \=$\tilde g_\delta(\Gamma_{u_0})=2$ {\bf do }\\
 \>  \> \> Add into $E_k$ the legal arc $(\Gamma_{l_0}, \Gamma_{w})$.\\
 \>  \> {\bf otherwise}\, ($\tilde g_\delta(\Gamma_{u_0})=1$) \\
 \>  \>  \>   If $\Gamma_{v_0'}\neq \Gamma_{u_0'}$,
 add into $E_k$ the legal arc $(\Gamma_{l_0},\Gamma_{v_0'})$.\\ 
 \>   \> \> Otherwise, in case $\alpha 1)$, add into $E_k$ the legal arc 
                  $(\Gamma_{l_0},\Gamma_{\bar v_1})$\\
\>   \> \> and in case 	$\alpha 2)$, a legal arc $(\Gamma_{l_0},\Gamma_{u_c})$ 
           with some label.\\  
\>   \> {\bf endif} \\
 \>            {\bf case a2):} Add into $E_k$ a legal arc 
                  $(\Gamma_{l_0},\Gamma_{u_0})$ with label $R_h^{u_0}$ 
		   if possible. \\
 \>            {\bf case b1):} Add into $E_k$ the legal arc 
                           $(\Gamma_{l_0},\Gamma_{u_1})$.\\
{\bf endfor }\\
\end{tabbing}

Now we have to show that for all $\beta \in S$, the subgraph of 
$(U\cup V_k, F_{k-1}\cup E_k)$ induced by $\{\Gamma_l\in U\cup V_k\,\,:\,\,
\beta \in f(R^l)\}$, denoted by $D_k^{\beta}$,
is a $\beta$-fork, a $\beta$-path or a union of two disjoint
$\beta$-paths and if $f_\beta$ is a half-edge, then $D_k^{\beta}$ is a $\beta$-path.
Thus $E_k$ induces a truth assignment of $I_k$ such that all
clauses are true, which proves the affirmation for $k$.

If there are at least three vertices $\Gamma_{l_0}$,
$\Gamma_{l_1}$ and $\Gamma_{l_2}$ in $V_k$ such that $\beta \in f(R^{l_i})$ for 
$i=0,1,2$, then in the graph $T^{*}$, two of them are in a same $\beta$-path or simple
$\beta$-fork
and so adjacent in $D$ by transitivity of the relation 
"$\Gamma_{l}\sim\Gamma_{ l'} \Leftrightarrow (\Gamma_{l},\Gamma_{l'})\in D $", 
which is a contradiction, because $\{\Gamma_{l_0},\Gamma_{l_1},
\Gamma_{l_2}\}$ represents a stable set in $D$.

Let $\Gamma_{l_0},\Gamma_{l_1}\in V_k$ and suppose that that there exists 
$\gamma \in f(R^{l_0})\cap
f(R^{l_1})$. Then $D_k^{\gamma}$ is the subgraph of $(U\cup V_k, 
F_{k-1}\cup E_k)$ induced by $U(\gamma) \cup \{\Gamma_{l_0},\Gamma_{l_1}\}$.
Assume first that $\Gamma_{l_0}$ satisfies case a1)
($(\Gamma_{l_0},\Gamma_{u_0})_{R_h^{u_0}}\in
T_k^*$). From remark \ref{remcase1gam} and since $T_k^*(U(\gamma )\cup \{
\Gamma_{l_0},\Gamma_{l_1}\})$ is a $\gamma$-fork or a disjoint union of two 
$\gamma$-paths, 
it follows that $(\Gamma_{l_1},\Gamma_{u_0'})\in  T_k^*$. 
Since $\Gamma_{u_0'}\in V_{k-1}$, we deduce that 
the arc $(\Gamma_{l_1},\Gamma_{u_0'})$ is a legal arc. So $\Gamma_{l_1}$ 
satifies the case a2). Then by the procedure TruthAssignment, 
$(\Gamma_{l_1},\Gamma_{u_0'})\in  E_k$ and in all
subcases of case a1), the graph $D_k^{\gamma}$ is a $\gamma$-fork or a union
of two disjoint $\gamma$-paths.

Now assume that $\Gamma_{l_0}$ satisfies the case a2) and $\Gamma_{l_1}$ does not 
satisfy the case a1). If $\Gamma_{l_1}$ satisfies the case a2), then 
$(\Gamma_{l_0},\Gamma_{u_0}),(\Gamma_{l_1},\Gamma_{u_0})
\in T_k^*$ for some $\Gamma_{u_0}\in U$. So $\Gamma_{l_0}$ is a central node 
of the $\gamma$-fork in $T_k^*$ and for all $\Gamma_{l}\in U(\gamma)$,
$\tilde g_\gamma(\Gamma_{l})=2$. By remark \ref{remfork} and the procedure 
TruthAssignment, we deduce that 
$D_k^{\gamma}$ is a $\gamma$-fork equal to $T_k^{*}(U(\gamma)\cup 
\{\Gamma_{l_0},\Gamma_{l_1}\})$. If $\Gamma_{l_1}$ has no successor in $T_k^*$, 
then we distinguish the two cases b1) and b2) for $\Gamma_{l_1}$. Let $\Gamma_{
u_1'}$ the root of the path $T_{k-1}^{*}(U(\gamma))$. Since $(\Gamma_{l_0},
\Gamma_{u_0})\in T^{*}$, $u_0$ must be equal to $u_1'$. 
In case b1), let $\delta \in f(R^{l_1})$ such that $T_{k-1}(U(\delta))$ is a union of
two disjoint $\delta$-paths with the roots $\Gamma_{u_1}$ and $\Gamma_{u_1'}$
respectively. By remark \ref{remcase1gam}, $U(\delta)=U(\gamma)$.
In case b2), $T_{k-1}(U(\gamma))$ is a $\gamma$-path. 
By the procedure TruthAssignment, $(\Gamma_{l_0},\Gamma_{u_0})$ and 
$(\Gamma_{l_1},\Gamma_{u_1})\in D_k^{\gamma}$ in case b1) and 
$(\Gamma_{l_0},\Gamma_{u_0}) \in D_k^\gamma$ in case b2).  
Thus $D_k^{\gamma}$  is a union of two disjoint $\gamma$-paths.

At last, suppose that $\Gamma_{l_0}$ and $\Gamma_{l_1}$ satisfy the case b1) or b2). 
Then clearly $D_k^{\gamma}$  is made up of two isolated nodes $\Gamma_{l_0}$ and 
$\Gamma_{l_1}$. So $D_k^{\gamma}$  is a union of two disjoint  $\gamma$-paths.

If $f_\beta$ is a half-edge for some $\beta\in S$ and $\beta\in f(R^l)$ for some 
$\Gamma_l\in V$, then by the hypothesis of induction 
$T_{k-1}(U(\beta))$ is a $\beta$-path. Let $\Gamma_{u_0}$ the root of the path
$T_{k-1}(U(\beta))$. We may prove that $T_{k-1}(U(\beta))\approx 
T_{k-1}^{*}(U(\beta))$ and $(\Gamma_{l},\Gamma_{u_0})$ is legal for some label. Thus 
by the procedure TruthAssignment, $D_k^{\beta}$  is also a $\beta$-path.

We conclude that there exists a truth assignment of $I_k$ such that all clauses are
true and so the affirmation is true for $k$.
\end{proof}

\begin{figure}[ht]
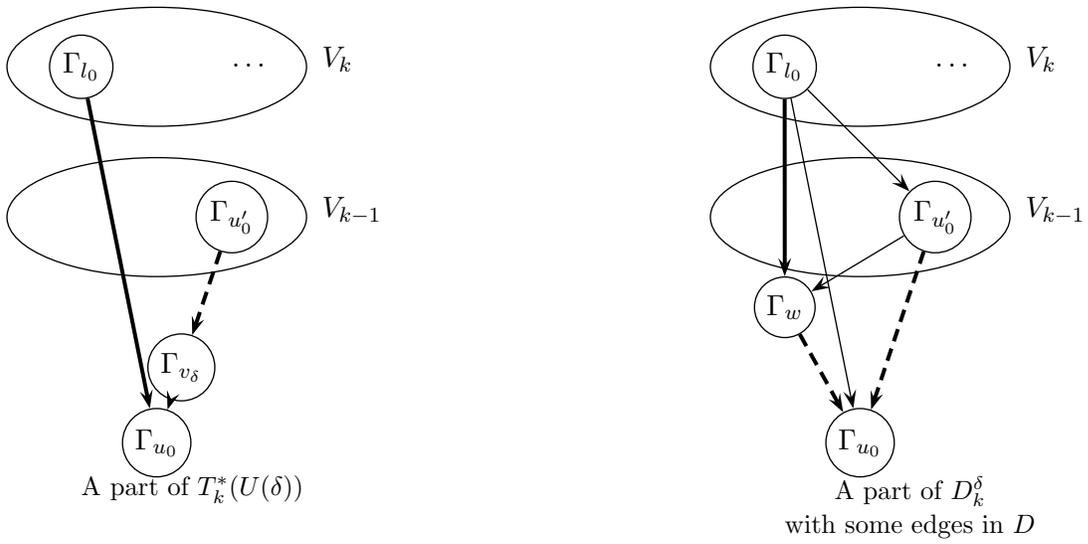

$
\begin{array}{cc}

\psset{xunit=1cm,yunit=1cm,linewidth=0.5pt,radius=0.3mm,arrowsize=5pt,
labelsep=1.5pt}

\pspicture(0,0)(9,6.5)

\cnodeput(1,5){1}{$\Gamma_{l_0}$}
\cnodeput(2,0){2}{$\Gamma_{u_0}$}
\cnodeput(2.33,1){3}{$\Gamma_{v_\delta}$}
\cnodeput(3,3){4}{$\Gamma_{u_0'}$}

\psellipse(2,5)(2,0.8)
\psellipse(2,3)(2,0.8)

\ncline[linewidth=1.5pt]{->}{1}{2}
\ncline[linewidth=1.5pt]{->}{3}{2}
\ncline[linewidth=1.5pt,linestyle=dashed]{->}{4}{3}

\put(3,5){$\ldots$}

\put(4.2,5){\shortstack{$V_k$}}
\put(4.2,3){\shortstack{$V_{k-1}$}}

\put(1,-0.7){\shortstack{\small{A part of $T_k^*(U(\delta ))$}}}
\endpspicture

&

\psset{xunit=1cm,yunit=1cm,linewidth=0.5pt,radius=0.3mm,arrowsize=5pt,
labelsep=1.5pt}

\pspicture(0,0)(4,6.5)

\cnodeput(1,5){1}{$\Gamma_{l_0}$}
\cnodeput(1,1.8){2}{$\Gamma_{w}$}
\cnodeput(2,0){3}{$\Gamma_{u_0}$}
\cnodeput(3,3){4}{$\Gamma_{u_0'}$}

\psellipse(2,5)(2,0.8)
\psellipse(2,3)(2,0.8)

\ncline[linewidth=1.5pt]{->}{1}{2}
\ncline[linewidth=1.5pt,linestyle=dashed]{->}{2}{3}
\ncline[linewidth=1.5pt,linestyle=dashed]{->}{4}{3}
\ncline{->}{1}{3}
\ncline{->}{1}{4}
\ncline{->}{4}{2}

\put(3,5){$\ldots$}

\put(4.2,5){\shortstack{$V_k$}}
\put(4.2,3){\shortstack{$V_{k-1}$}}

\put(1,-1.2){\shortstack{\small{A part of $D_k^{\delta}$}\\
\small{with some edges in $D$}}}
\endpspicture 

\end{array} 
\vspace{0.5cm}
$
\vspace{0.8cm}
\caption{An illustration of case a1) with $\tilde g_\delta(\Gamma_{u_0})=2$. Heavy edges
belong to $T_k^*(U(\delta ))$ on the left and to $D_k^{\delta}$ 
on the right.
A dashed edge corresponds to a path in $T_k^*(U(\delta ))$ on the left
and in $D_k^{\delta}$ on the right.}
\label{fig:qdelta=2}
\end{figure}

\begin{figure}[ht]
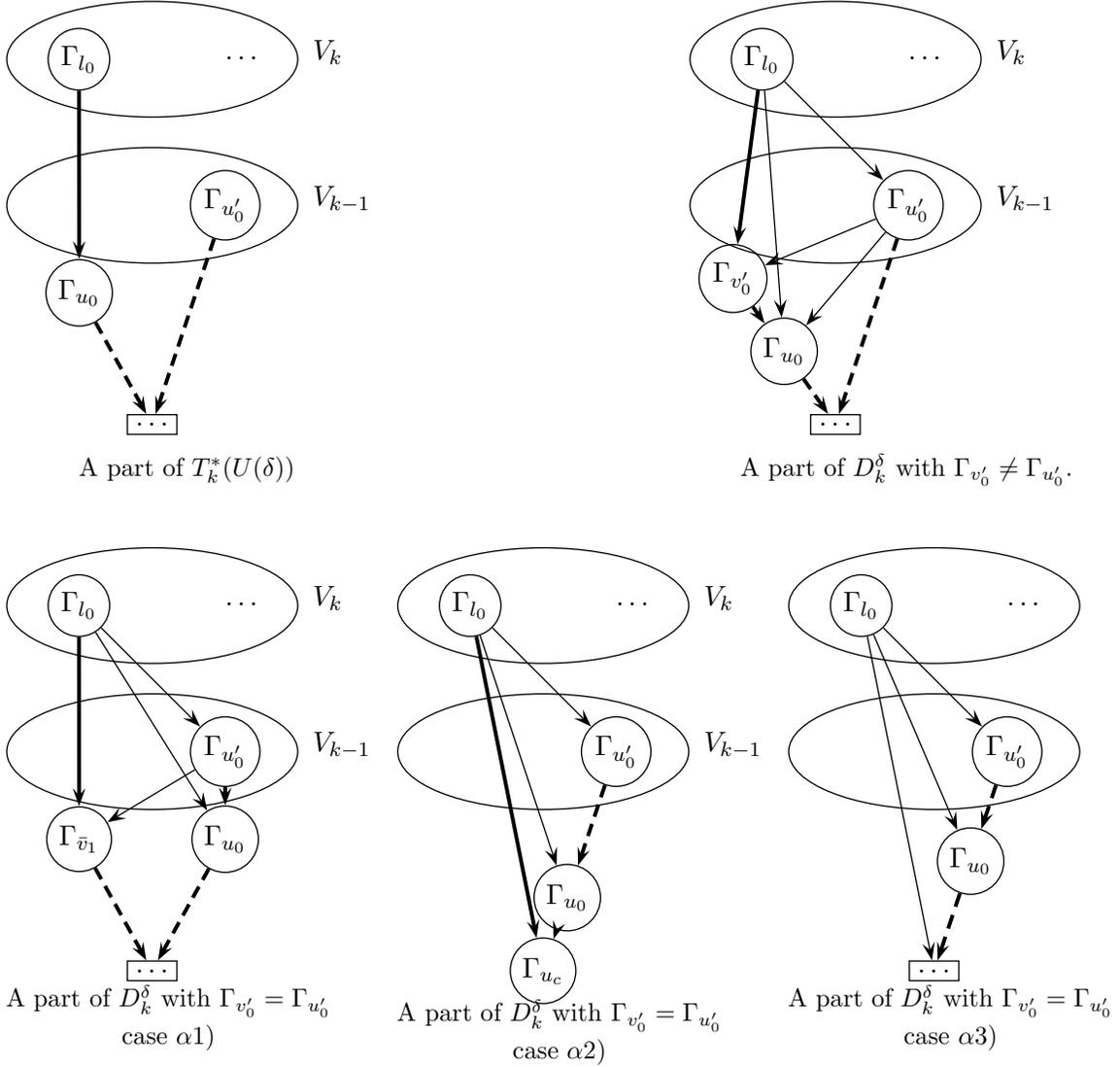

$
\begin{array}{cc}

\psset{xunit=1cm,yunit=1cm,linewidth=0.5pt,radius=0.3mm,arrowsize=5pt,
labelsep=1.5pt}

\pspicture(0,0)(9,6.5)

\cnodeput(1,5){1}{$\Gamma_{l_0}$}
\cnodeput(1,1.8){2}{$\Gamma_{u_0}$}
\cnodeput(3,3){3}{$\Gamma_{u_0'}$}
\rput(2,0){\rnode{A}{\psframebox{$\ldots$}}}

\psellipse(2,5)(2,0.8)
\psellipse(2,3)(2,0.8)

\ncline[linewidth=1.5pt]{->}{1}{2}
\ncline[linewidth=1.5pt,linestyle=dashed]{->}{3}{A}
\ncline[linewidth=1.5pt,linestyle=dashed]{->}{2}{A}

\put(3,5){$\ldots$}

\put(4.2,5){\shortstack{$V_k$}}
\put(4.2,3){\shortstack{$V_{k-1}$}}

\put(1,-0.7){\shortstack{\small{A part of $T_k^*(U(\delta ))$}}}
\endpspicture

&

\psset{xunit=1cm,yunit=1cm,linewidth=0.5pt,radius=0.3mm,arrowsize=5pt,
labelsep=1.5pt}

\pspicture(0,0)(4,6.5)

\cnodeput(1,5){1}{$\Gamma_{l_0}$}
\cnodeput(0.6,2){2}{$\Gamma_{v_0'}$}
\cnodeput(1.3,1){3}{$\Gamma_{u_0}$}
\cnodeput(3,3){4}{$\Gamma_{u_0'}$}
\rput(2,0){\rnode{A}{\psframebox{$\ldots$}}}

\psellipse(2,5)(2,0.8)
\psellipse(2,3)(2,0.8)

\ncline[linewidth=1.5pt]{->}{1}{2}
\ncline{->}{1}{3}
\ncline{->}{1}{4}
\ncline{->}{4}{2}
\ncline{->}{4}{3}
\ncline[linewidth=1.5pt,linestyle=dashed]{->}{2}{3}
\ncline[linewidth=1.5pt,linestyle=dashed]{->}{3}{A}
\ncline[linewidth=1.5pt,linestyle=dashed]{->}{4}{A}

\put(3,5){$\ldots$}

\put(4.2,5){\shortstack{$V_k$}}
\put(4.2,3){\shortstack{$V_{k-1}$}}

\put(0.7,-.7){\shortstack{\small{A part of $D_k^{\delta}$ 
with $\Gamma_{v_0'}\neq \Gamma_{u_0'}$}.}}

\endpspicture 

\end{array} 
$

\vspace{0.8cm}

$
\begin{array}{ccc}

\psset{xunit=1cm,yunit=1cm,linewidth=0.5pt,radius=0.3mm,arrowsize=5pt,
labelsep=1.5pt}

\pspicture(0,0)(5,6.5)

\cnodeput(1,5){1}{$\Gamma_{l_0}$}
\cnodeput(1,1.8){2}{$\Gamma_{\bar v_1}$}
\cnodeput(3,3){3}{$\Gamma_{u_0'}$}
\cnodeput(3,1.8){4}{$\Gamma_{u_0}$}
\rput(2,0){\rnode{A}{\psframebox{$\ldots$}}}

\psellipse(2,5)(2,0.8)
\psellipse(2,3)(2,0.8)

\ncline[linewidth=1.5pt]{->}{1}{2}
\ncline{->}{1}{3}
\ncline{->}{3}{2}
\ncline{->}{1}{4}
\ncline[linewidth=1.5pt,linestyle=dashed]{->}{3}{4}
\ncline[linewidth=1.5pt,linestyle=dashed]{->}{4}{A}
\ncline[linewidth=1.5pt,linestyle=dashed]{->}{2}{A}

\put(3,5){$\ldots$}

\put(4.2,5){\shortstack{$V_k$}}
\put(4.2,3){\shortstack{$V_{k-1}$}}

\put(0,-1){\shortstack{\small{A part of $D_k^{\delta}$ with
$\Gamma_{v_0'}=\Gamma_{u_0'}$}\\\small{case $\alpha 1)$ }}}
\endpspicture  

&

\psset{xunit=1cm,yunit=1cm,linewidth=0.5pt,radius=0.3mm,arrowsize=5pt,
labelsep=1.5pt}

\pspicture(0,0)(5,6.5)

\cnodeput(1,5){1}{$\Gamma_{l_0}$}
\cnodeput(2,0){2}{$\Gamma_{u_c}$}
\cnodeput(2.33,1){3}{$\Gamma_{u_0}$}
\cnodeput(3,3){4}{$\Gamma_{u_0'}$}

\psellipse(2,5)(2,0.8)
\psellipse(2,3)(2,0.8)

\ncline[linewidth=1.5pt]{->}{1}{2}
\ncline[linewidth=1.5pt]{->}{3}{2}
\ncline[linewidth=1.5pt,linestyle=dashed]{->}{4}{3}
\ncline{->}{1}{3}
\ncline{->}{1}{4}

\put(3,5){$\ldots$}

\put(4.2,5){\shortstack{$V_k$}}
\put(4.2,3){\shortstack{$V_{k-1}$}}

\put(0,-1.2){\shortstack{\small{A part of $D_k^{\delta}$ with
$\Gamma_{v_0'}=\Gamma_{u_0'}$}\\\small{case $\alpha 2)$ }}}
\endpspicture   

&

\psset{xunit=1cm,yunit=1cm,linewidth=0.5pt,radius=0.3mm,arrowsize=5pt,
labelsep=1.5pt}

\pspicture(0,0)(4.5,6.5)

\cnodeput(1,5){1}{$\Gamma_{l_0}$}
\cnodeput(3,3){2}{$\Gamma_{u_0'}$}
\cnodeput(2.5,1.5){3}{$\Gamma_{u_0}$}
\rput(2,0){\rnode{A}{\psframebox{$\ldots$}}}

\psellipse(2,5)(2,0.8)
\psellipse(2,3)(2,0.8)

\ncline{->}{1}{2}
\ncline{->}{1}{3}
\ncline{->}{1}{A}
\ncline[linewidth=1.5pt,linestyle=dashed]{->}{2}{3}
\ncline[linewidth=1.5pt,linestyle=dashed]{->}{3}{A}

\put(3,5){$\ldots$}

\put(0,-1){\shortstack{\small{A part of $D_k^{\delta}$ with
$\Gamma_{v_0'}=\Gamma_{u_0'}$}\\\small{case $\alpha 3)$ }}}
\endpspicture

\end{array}
$

\vspace{1cm}
\caption{An illustration of case a1) with $\tilde g_\delta(\Gamma_{u_0})=1$. Heavy edges
belong to $T_k^*(U(\delta ))$ at the top left or $D_k^{\delta}$ 
otherwise. A dashed edge corresponds to a path in $T_k^*(U(\delta ))$ 
at the top left or in $D_k^{\delta}$ otherwise. Thin edges belong to $D$.} 
\label{fig:qdelta=1}
\end{figure}